\documentclass[12pt,a4paper]{amsbook}
\usepackage{amssymb,latexsym}

\usepackage{currfile}

\usepackage[english]{babel}

\usepackage[utf8]{inputenc}
\usepackage[T1]{fontenc}
\usepackage{lmodern}

\usepackage{graphicx}

\usepackage{color}
\usepackage{colortbl}
\usepackage{subcaption, float}
\usepackage[section]{placeins}

\usepackage{enumerate}

\usepackage[margin=3cm]{geometry}
\usepackage{imakeidx}
\makeindex[title=Alphabetical Index]

\theoremstyle{plain} 
\newtheorem{theorem}{Theorem}[chapter]
\newtheorem{lemma}[theorem]{Lemma}
\newtheorem*{lemma*}{Lemma}
\newtheorem{proposition}[theorem]{Proposition}
\newtheorem{conjecture}[theorem]{Conjecture}

\newtheorem{corollary}[theorem]{Corollary}
\newtheorem{property}[theorem]{Property}
\newtheorem*{property*}{Property}
\newtheorem{properties}[theorem]{Properties}
\newtheorem{assumption}[theorem]{Assumption}
\newtheorem{assumptions}[theorem]{Assumptions}

\newtheorem{claim}[theorem]{Claim}

\newtheorem{procedure}[theorem]{Procedure}
\theoremstyle{definition}
\newtheorem{definition}[theorem]{Definition}
\newtheorem{definitions}[theorem]{Definitions}

\newtheorem{remark}[theorem]{Remark}
\newtheorem{remarks}[theorem]{Remarks}
\newtheorem{notation}[theorem]{Notation}

\newtheorem*{remark*}{Remark}
\newtheorem*{note*}{Editorial Note}

\numberwithin{section}{chapter}
\numberwithin{subsection}{section}
\numberwithin{subsubsection}{subsection}
\numberwithin{equation}{chapter}
\numberwithin{table}{chapter}
\numberwithin{figure}{chapter}

\renewcommand{\Bbb}{\mathbb} 

\newcommand{\B}{\mathbb{B}}

\newcommand{\D}{\mathbb{D}}

\newcommand{\bH}{\mathbb{H}}
\newcommand{\K}{\mathbb{K}}

\newcommand{\N}{\mathbb{N}}

\newcommand{\bP}{\mathbb{P}}
\newcommand{\R}{\mathbb{R}}
\newcommand{\bS}{\mathbb{S}}
\newcommand{\T}{\mathbb{T}}

\newcommand{\cA}{\mathcal{A}}
\newcommand{\cB}{\mathcal{B}}

\newcommand{\cD}{\mathcal{D}}
\newcommand{\cE}{\mathcal{E}}
\newcommand{\cG}{\mathcal{G}}
\newcommand{\cH}{\mathcal{H}}
\newcommand{\cI}{\mathcal{I}}

\newcommand{\cL}{\mathcal{L}}
\newcommand{\cM}{\mathcal{M}}

\newcommand{\cO}{\mathcal{O}}

\newcommand{\cR}{\mathcal{R}}
\newcommand{\cS}{\mathcal{S}}

\newcommand{\cU}{\mathcal{U}}
\newcommand{\cW}{\mathcal{W}}

\newcommand{\cZ}{\mathcal{Z}}

\newcommand{\mf}{\mathfrak}

\newcommand{\bbH}{\overline{\mathbb{H}}}
\newcommand{\bdH}{\partial{\mathbb{H}}}
\newcommand{\Omegab}{\overline{\Omega}}

\newcommand{\omegat}{\widetilde{\omega}}
\newcommand{\tauv}{\vec{\tau}}
\newcommand{\nuv}{\vec{\nu}}

\newcommand{\ww}{\cW}%

\newcommand{\pib}{\frac{\pi}{2}}
\newcommand{\pih}{\frac{\pi}{8}}

\DeclareMathOperator{\mult}{mult}
\DeclareMathOperator{\inte}{int}
\DeclareMathOperator{\sign}{sgn}
\DeclareMathOperator{\id}{Id}
\DeclareMathOperator{\ord}{ord}
\DeclareMathOperator{\spanop}{span}

\newcommand{\nf}{\normalfont}
\newcommand{\pf}{\emph{Proof.~}}
\newcommand{\qedc}{$\checkmark$}

\newcommand{\wt}[1]{\widetilde{#1}}
\newcommand{\wb}[1]{\overline{#1}}
\newcommand{\sm}{\!\setminus\!}
\newcommand{\set}[1]{\big\lbrace #1 \big\rbrace}
\newcommand{\spanv}[1]{\spanop\!\big\lbrace #1 \big\rbrace}

\newcommand{\ip}[2]{\langle #1,#2 \rangle}
\newcommand{\Cty}{C^{\infty}}
\newcommand{\floor}[1]{\big\lfloor #1 \big\rfloor}
\newcommand{\equivo}[1]{\, \overset{#1}{\equiv}\,}

\newcommand{\noi}{\noindent}
\newcommand{\noib}{\noindent $\bullet$~}
\newcommand{\noid}{\noindent $\diamond$~}
\newcommand{\noic}{\noindent $\circ$~}

\definecolor{purple}{RGB}{127,0,255}
\definecolor{lgray}{gray}{0.7}

\usepackage[pdftex,pagebackref=true,colorlinks=true,linkcolor=blue,citecolor=blue,
urlcolor=blue,pdfauthor={Name}]{hyperref}

\begin{document}

\frontmatter

\title{Upper bounds on eigenvalue multiplicities  for spheres and plane domains revisited}

\author{Pierre B\'erard}
\address{Institut Fourier, Universit\'{e} Grenoble Alpes and CNRS, CS 40700\\
  38058 Grenoble Cedex 9, France}
\email{pierrehberard@gmail.com}

\author{Bernard Helffer}
\address{Laboratoire de Math\'ematiques Jean Leray,  Nantes  Universit\'{e}  and CNRS\\ 44000 Nantes Cedex, France}
\email{Bernard.Helffer@univ-nantes.fr}



\date{\small \today}

\begin{abstract}
We revisit two papers which appeared in 1999:
\emph{M.~Hoffmann-Osten\-hof, T.~Hoffmann-Ostenhof, and N.~Nadirashvili}  [Ann. Global Anal. Geom. 17 (1999) 43--48] and \emph{T.~Hoff\-mann-Ostenhof, P.~Michor, and N.~Nadirashvili} [Geom. Funct. Anal. 9 (1999) 1169--1188]. The main result of these papers is that the multiplicity $\mathrm{mult}(\lambda_k(M))$ of the $k$th eigenvalue of the Riemannian surface $M$ is bounded from above by $(2k-3)$ provided that $k \ge 3$. In the first paper, $M$ is homeomorphic to a sphere. In the second,  $M$ is a plane domain with Dirichlet boundary condition. In both cases, the starting label of eigenvalues is $1$. The proofs given in these papers are not very detailed, and often rely on figures or special configurations of nodal sets.\\
The purpose of this monograph is to provide detailed general proofs for the above upper bounds and to extend the results to Robin boundary conditions. We also provide a survey of previous results  (Chapter~1), as well as proofs of prerequisite theorems  (Chapter~2).\\
When $M$ is homeomorphic to a sphere, we provide a complete proof of the upper bound, $\mathrm{mult}(\lambda_k) \le (2k-3)$ for any $k\ge 3$, by  introducing and carefully studying the \emph{combinatorial type} and a \emph{labeling of the nodal domains} of some particular eigenfunctions  (Chapter~3). When $M$ is a plane domain, we consider the three boundary conditions, Dirichlet, Neumann, Robin, and we also study the \emph{combinatorial types} and \emph{a labeling of the nodal domains} of some particular eigenfunctions. More precisely, we prove the inequality $\mathrm{mult}(\lambda_k) \le (2k-2)$ for general $C^{\infty}$ bounded domains and all $k \ge 3$  (Chapter~4).   We prove  the inequality $\mathrm{mult}(\lambda_k) \le (2k-3)$ for $k \ge 3$ under the additional assumption that the domain is \emph{simply connected}  (Chapter~5).  Chapter~3 serves as a warm-up for Chapters~4 and 5 which form the core of this monograph. These three chapters rely on Euler's inequality applied to the nodal graph of eigenfunctions (see Chapter~2), and a careful analysis of some eigenfunctions which optimize Euler's inequality.   Chapter~6 contains related results (nodal line conjecture; Courant-sharp eigenvalues).
\end{abstract}

\maketitle

\vspace{2cm}
\begin{center}
Final version 15/01/2026.\\
Accepted for publication in the Memoirs of the European Mathematical Society.
\end{center}%

\vspace{2cm}
[\currfilename], \today.

\cleardoublepage
\thispagestyle{empty}
\vspace*{13.5pc}
\begin{center}
This memoir is dedicated to Steve Zelditch.
\end{center}
\clearpage
\thispagestyle{empty}
\vspace*{13.5pc}
\emph{Acknowledgements.~}During the preparation of this monograph, we consulted several colleagues: A.~Berdnikov, L.~Friedlander, T.~Hoffmann-Ostenhof, P.~Jammes, M.~Kar\-pukhin, J.~Kennedy, F.~Laudenbach, Z.~Liqun, N.~Nadirashvili, and I.~Poltero\-vich. We thank them all for their comments.  Since the publication of the first version of this work in February 2022, we had several fruitful discussions with N.~Nadirashvili, which  led us to explore new approaches to fill in  the gaps in the proofs provided in \cite{HoMN1999}.   We also acknowledge the comments and further references provided by V.~Bobkov. \smallskip

Finally, we thank Steve Zelditch for his comments, a few months before his death in September 2022, and we express our deep appreciation of his contributions to spectral geometry.
\clearpage

\setcounter{page}{7}

\tableofcontents

\mainmatter


\chapter{Introduction and Survey of Previous Results}\label{Ch-ihs}

\section{Introduction}\label{S-intro}
 In this monograph, we are concerned with upper bounds for the multiplicities of the eigenvalues $\set{\lambda_k, k\ge 1}$ of a Schr\"{o}dinger operator \index{Schr\"{o}dinger operator} $-\Delta + V$ on a compact, smooth (ie. $C^{\infty}$), connected Riemannian surface. When the boundary $\partial M$ is not empty, we consider the Dirichlet, Neumann or Robin boundary conditions.  We \emph{do not} consider the Steklov eigenvalue problem for which we refer to the papers of Karpukhin, Kokarev and Polterovich \cite{KaKP2014}, Fraser and Schoen \cite{FrSc2016}, Jammes \cite{Jam2016},  Colbois, Girouard, Gordon and Sher \cite{CoGGS2024}, and their reference lists. \medskip

Our main purpose is to revisit the papers \cite{HoHN1999} (Riemannian surfaces homeomorphic to a sphere\footnote{In \cite{HoHN1999}, the authors refer to ``compact surfaces without boundary and genus $0$'', and implicitly assume that the surface is orientable. In this monograph, we have chosen a shorter terminology. We will refer to Riemannian spheres \index{Riemannian sphere} $(M,g)$ with potential $V$, where $M$ is a $\Cty$ surface homeomorphic to the sphere, equipped with a $\Cty$ Riemannian metric $g$, and with a $\Cty$ real valued potential $V$. See Chapter~\ref{Ch-rsp}.})  and \cite{HoMN1999} (planar domains with smooth boundary) whose proofs are not very detailed and often rely on figures and special configurations of nodal sets. We introduce and carefully study the \emph{combinatorial type}  (defined in Subsection~\ref{SS-h2n-s2})  of some particular eigenfunctions, as well as a \emph{labeling} of their nodal domains  (see Sections~\ref{S-llnd} and \ref{S-hmn2L}).  For domains in $\R^2$, we  provide a unified treatment for the three boundary conditions  (Dirichlet, Neumann, Robin). We also illustrate our proofs with many figures.\medskip

In the sequel $\Delta$ is the Laplace-Beltrami operator \index{Laplace-Beltrami operator} on the surface $M$ for some smooth Riemannian metric $g$ (our convention is that $\Delta$ is a nonpositive operator), and $V$ is a smooth real valued function. We list the eigenvalues in nondecreasing order, multiplicities accounted for. Our convention is that, in all cases, we label the eigenvalues \index{Labeling (eigenvalue)}\emph{starting from the label} $1$,
\begin{equation*}
\lambda_1 < \lambda_2 \le \lambda_3 \le \ldots ,
\end{equation*}
and we
denote the multiplicity of $\lambda_k$ by $\mult(\lambda_k)$.
\index{2-m@$\mult(\lambda_k)$}
We refer to Chapter~\ref{Ch-evp} for more definitions and notation.\smallskip

In Section~\ref{S-sketch} we provide a survey of the main results on the multiplicity problem, and the ideas behind their proofs. \smallskip

Chapter~\ref{Ch-evp} is devoted to definitions, notation and  prerequisites on eigenvalues and eigenfunctions of Schr\"{o}dinger operators on compact surfaces.  In Section~\ref{S-pargef}, we state Euler type formulas for nodal graphs. They will be applied extensively in the sequel. In Sections~\ref{S-lsns} and \ref{S-lsbs}, we give detailed proofs of the local structure theorem for an eigenfunction near a singular point. They will also be used extensively in the following chapters. \smallskip

In Chapter~\ref{Ch-rsp}, we revisit \cite{HoHN1999}. Introducing the \emph{combinatorial type} of some particular nodal sets (Definition~\ref{D-h2n-type}) and a \emph{labeling} of their nodal domains (Section~\ref{S-hmn2L}), we provide a complete proof of the inequality $\mult(\lambda_k) \le (2k-3)$ for all $k \ge 3$, for Riemannian spheres with potential, see Theorem~\ref{T-h2n}. This chapter  is meant as a  warm-up for the remaining chapters. \smallskip

In the next two chapters, we revisit \cite{HoMN1999} and prove the following theorem.

\begin{theorem}\label{T-hmn-bh}
Consider the eigenvalue problem for the operator  $-\Delta + V$ in a $\Cty$ bounded domain $\Omega \subset \R^2$, with the Dirichlet, the Neumann or the $h$-Robin boundary condition.
\begin{enumerate}[(i)]
  \item Without any further assumption on $\Omega$, for any $k \ge 3$, $\mult(\lambda_k) \le (2k-2)$.
  \item Assuming that $\Omega$ is \emph{simply connected}, for any $k \ge 3$, $\mult(\lambda_k) \le (2k-3)$.
\end{enumerate}
\end{theorem}%

The proof of Assertion~(i) is given in Chapter~\ref{Ch-pdwb}; the proof of Assertion~(ii)  in Chapter~\ref{Ch-scpdsb}.

\begin{remarks}\label{R-hmn-sc}~
\begin{enumerate}
  \item In \cite[Theorem~B, p.~1172]{HoMN1999}, the authors state that the bound, $\mult(\lambda_k) \le (2k-3)$ for all $k\ge 3$, holds for all smooth bounded  domains $\Omega \subset \R^2$. In \cite[Section~4]{Berd2018}, Berdnikov points out a gap in the proof when $\Omega$ is not simply connected. This is why we restrict ourselves to simply connected domains in Assertion~(ii).
  \item Theorem~\ref{T-hmn-bh} covers both Dirichlet and Robin boundary conditions, whereas \cite{HoMN1999} only deals with the Dirichlet boundary condition.  As a matter of fact, the proofs in both cases, Dirichlet and Robin, are very similar, except for a specific energy argument in the Robin case (Lemma~\ref{L-hmn3-L38-e}).
\end{enumerate}
\end{remarks}%

In Chapter~\ref{Ch-fr} we relate the problem of bounding multiplicities from above to the question of \emph{Courant-sharp eigenvalues}  (eigenvalues one of whose eigenfunctions maximizes the number of nodal domains, see Remark~\ref{R-hmn-2}), and the particular case of  the multiplicity of the second eigenvalue, $\mult(\lambda_2)$, to the \emph{Nodal Line Conjecture}.\medskip

\section{Survey of previous results}\label{S-sketch}

In the case of closed surfaces, the first upper bounds on multiplicities were obtained by  Cheng \cite{Chen1976}, Besson \cite{Bess1980}, and Nadirashvili \cite{Nadi1987}.
We denote their respective upper bounds on $\mult(\lambda_k)$ by $m_k^{*}$, with $* \in \set{B,C,N}$, where $B$ stands for ``Besson'', $C$ for ``Cheng'', and $N$ for ``Nadirashvili'', and provide a summary of their results in
Table~\ref{E-scr-2T}  (with our convention that the labeling of eigenvalues begins with $1$, not $0$). \smallskip

The upper bounds for the multiplicity of the second eigenvalue (i.e., the least positive eigenvalue of a closed surface) given in the fourth column are sharp. For the sphere the bound is achieved for the canonical (round) metric, \cite{Chen1976}; for the projective space the bound is achieved for the metric induced by the canonical metric of the sphere, \cite{Bess1980}; for the torus the bound is achieved for the equilateral torus $\T_e$ with metric induced from $\R^2$, \cite{Bess1980}; for the Klein bottle, the bound is achieved for a nontrivial pair $(g,V)$ constructed in \cite[\S\,2]{Nadi1987}, and for smooth metrics constructed in \cite[Th\'{e}or\`{e}me~4.2]{ColV1987}. An interesting feature of $\bS^2, \mathbb{RP}^2$ and $\T^2$ is that the bounds for $\mult(\lambda_2)$ are also achieved for metrics different from the ones mentioned above, see \cite{Bess1980}.\smallskip

In \cite[Th\'{e}or\`{e}me~1.5]{ColV1987}, Colin de Verdi\`{e}re shows that for a closed surface $M$,   \[\sup\set{\mult(\lambda_2(M,-\Delta_g + V)) \mid (g,V)} \ge C(M) - 1,\] where the supremum is taken over the Riemannian metrics and potentials on $M$, and where $C(M)$ is the \emph{chromatic number} \index{Chromatic number} of $M$ (the maximal  number $N$ such that the complete graph on $N$ vertices $K_N$ can be embedded into $M$). Table~\ref{E-scr-2T} shows that equality holds for $\bS^2, \mathbb{RP}^2, \T^2$ and $\K^2$; it also holds for surfaces with $\chi(M) \ge -3$, where $\chi (M)$ is the Euler characteristic \index{Euler characteristic} of $M$, see \cite{Seve2002}.  It was conjectured that equality holds for all closed surfaces, see Fortier Bourque et alii \cite{FoGPP2023} for two counter-examples.\smallskip

\begin{table}[!htb]
\centering
\resizebox{0.9\textwidth}{!}{%
\begin{tabular}[c]{|c|c|c|c|l|}%
\hline
&&&&\\[-1em]
$M$ & $\chi(M)$ & orientability & $\mult(\lambda_2) \le$
& \text{for~} $k\ge 1, ~\mult(\lambda_k) \le$ \\[0.2em]
\hline
&&&&\\[-1em]
$\bS^2$ &  $2$ & orientable & $3$ &
$
\left\{
\begin{array}{ll}
m_k^{C} & = \frac 12 k(k+1)\\[5pt]
m_k^{B} & = 2k-1\\[5pt]
m_k^{N} & = 2k-1
\end{array}
\right.
$
\\[0.2em]
\hline
&&&&\\[-1em]
$\mathbb{RP}^2$ &  $1$ & non-orientable & $5$ &
$
\left\{
\begin{array}{ll}
m_k^{C} & = \text{not considered} \\[5pt]
m_k^{B} & = 4k-1\\[5pt]
m_k^{N} & = 2k+1
\end{array}
\right.
$ \\[0.2em]
\hline
&&&&\\[-1em]
$\T^2$ &  $0$ & orientable & $6$ &
$
\left\{
\begin{array}{ll}
m_k^{C} & = \frac 12  (k+2)(k+3)\\[5pt]
m_k^{B} & = 2k+3\\[5pt]
m_k^{N} & = 2k+2
\end{array}
\right.
$ \\[0.2em]
\hline
&&&&\\[-1em]
$\K^2$ &  $0$ & non-orientable & $5$ &
$
\left\{
\begin{array}{ll}
m_k^{C} & =  \text{not considered} \\[5pt]
m_k^{B} & =  \text{not considered} \\[5pt]
m_k^{N} & = 2k+1
\end{array}
\right.
$ \\[0.2em]
\hline
&&&&\\[-1em]
$M^2$ &  $\chi(M) < 0$ & orientable & -- &
$
\left\{
\begin{array}{ll}
m_k^{C} & = \frac 12 \big( k- \chi(M) + 2\big) \big( k-\chi(M) +3 \big)\\[5pt]
m_k^{B} & = 2k - 2 \chi(M) +3\\[5pt]
m_k^{N} & = 2k- 2 \chi(M) +1 \\
\end{array}
\right.
$ \\[0.2em]
\hline
&&&&\\[-1em]
$M^2$ &  $\chi(M) < 0$ & non-orientable & -- &
$
\left\{
\begin{array}{ll}
m_k^{C} & = \text{not considered} \\[5pt]
m_k^{B} & = 4k - 4 \chi(M) + 3\\[5pt]
m_k^{N} & = 2k- 2 \chi(M) + 1\\
\end{array}
\right.
$ \\[0.2em]
\hline
\end{tabular}%
}
\caption{Closed surfaces: multiplicity upper bounds obtained by Cheng, Besson, and Nadirashvili  (labeling starting from $1$)}\label{E-scr-2T}
\end{table}

\index{1-chi@$\chi(M)$}
\index{1-gamma@$\gamma$!$\gamma(M)$}
Cheng and Besson, express their upper bounds in terms of the genus. When the surface $M$ is orientable, $\chi(M) = 2 - 2\gamma(M)$, where $\gamma(M)$ is the genus of $M$, \index{Genus} and the surface is homeomorphic to a  $2$-sphere with $\gamma(M)$ handles attached. When the surface $M$ is not orientable, $\chi(M) = 1 - \gamma(\wt{M})$, where $\gamma(\wt{M})$ is the genus of the orientable cover $\wt{M}$ of $M$, and the surface is the connected sum of $(\gamma(\wt{M})+1)$ copies of the projective plane. \smallskip

Better upper bounds were later obtained by M. and T.~Hoffmann-Ostenhof and Nadirashvili \cite{HoHN1999}  (Riemannian spheres with potential, see Chapter~\ref{Ch-rsp}),  S\'{e}vennec \cite{Seve2002}, Fortier Bourque and Petri \cite{FoBP2023} (improved bounds on the multiplicity of $\lambda_2$ when $\chi(M) < 0$), Berdnikov, Nadirashvili and Penskoi \cite{BeNP2016} (improved bounds for the multiplicities on the projective plane),  Fortier Bourque and Petri \cite{FoBP2024} (Klein quartic). \smallskip

In \cite[Theorem~2]{Nadi1987}, Nadirashvili also considers smooth bounded
domains $\Omega \subset \R^2$ and proves that the multiplicity of the $k$th eigenvalue $\lambda_k$ of an operator $-\Delta + V$ with Dirichlet or Neumann boundary condition is at most $(2k-1)$.

In the paper \cite{HoMN1999}, Hoffmann-Ostenhof, Michor and Nadirashvili, improve Nadirashvili's bound for bounded plane domains  with $C^{\infty}$ boundary and Dirichlet boundary condition. More precisely, they state that the multiplicity of $\lambda_k$ is less than or equal to $(2k-3)$. Berdnikov \cite{Berd2018} considers the case of compact surfaces with boundary, under the assumption that $\chi(M) + b_0(\partial M)$ is negative  (where $b_0$ denotes the number of connected components).
\index{2-b@$b_0$}
He points out some problem in the proof of Theorem~B in \cite{HoMN1999} when the domain is not simply connected.\medskip

The general strategy to prove upper bounds for the eigenvalue multiplicities is a combination of the following ingredients:
\begin{enumerate}[(i)]
  \item Courant's nodal domain theorem, \index{Courant nodal Theorem} Theorem~\ref{T-RC}.
  \item  Local structure of  eigenfunctions near a singular point, Theorem~\ref{T-nodinfo}.
  \item Existence of eigenfunctions with prescribed  singular points, provided the dimension of the eigenspace is large enough, Subsection~\ref{SS-evp3}.
  \item Euler's formula for the graph \index{Euler formula} associated with the nodal set of an eigenfunction, Section~\ref{S-pargef}.
  \item The \emph{rotating function argument}, which first appeared in \cite{Bess1980}, \S~\ref{SSS-h2n-s2r}.
  \item Energy arguments, Lemma~\ref{L-hmn3-L38-e}, and eigenvalue monotonicity.
\end{enumerate}

In one form or another, these arguments go back to Cheng \cite{Chen1976}, Besson \cite{Bess1980}, and Nadirashvili \cite{Nadi1987}.\smallskip

Two other papers,  respectively \cite{HeHO1999} by Helffer, M. and T. Hoffmann-Osten\-hof and Owen, and \cite{HeHN2002} by Helffer, M. and T. Hoffmann-Ostenhof and Nadirashvili, have used the same techniques
for related purposes (for example the Aharo\-nov-Bohm  operators). Similar techniques are used in the analysis of the properties of minimal partitions \cite{HeHT2009,BoHe2017}.

We refer to the papers of Burger, Colbois and/or Colin de Verdi\`{e}re \cite{BuCo1985, Colb1985, ColV1986, ColV1987, CoCo1988}  and to the paper of Letrouit and Machado \cite{LeMa2024} for results of a different flavor.

\chapter{Prerequisites on Eigenvalue Problems}\label{Ch-evp}

\section{Eigenvalue Problems}\label{S-evp}

\subsection{Definitions, notation and preliminary results}\label{SS-evp1}

In this chapter, $M$ denotes a closed surface (compact, no boundary), or a compact surface with boundary. The boundary is denoted by $\partial M$, and the interior $M\sm \partial M$ is denoted by $\inte(M)$. Unless otherwise stated, the surface is assumed to be smooth and connected. We equip $M$ with a smooth Riemannian metric $g$, and we consider a (non-magnetic) Schr\"{o}dinger operator
\index{Schr\"{o}dinger operator}
of the form $-\Delta_g + V$, where $\Delta_g$ is the Laplace-Beltrami operator \index{Laplace-Beltrami operator} for the metric $g$ and $V$ is a smooth real valued function on $M$.\medskip

The notation is  as follows. The Riemannian measure \index{Riemannian measure}  is denoted by $v_g$. When $\partial M$ is not empty, $\sigma_g$ is the Riemannian measure of $\partial M$ for the metric induced by $g$, and $\nu$ is the unit normal to $\partial M$ pointing  inward.\smallskip

When $M$ is closed ($\partial M = \emptyset$), we consider the  closed (no boundary condition) eigenvalue problem
\index{Eigenvalue problem} \index{Eigenvalue problem!closed}
\begin{equation}\label{E-evp-2c}
-\Delta u + V u = \lambda \, u \text{\quad in \quad} M,
\end{equation}
associated with the quadratic form
\begin{equation}\label{E-evp-2cq}
\int_M \big( |du|^2_g + V u^2 \big) \, dv_g, \text{\quad with domain\quad} H^1(M).
\end{equation}

When $\partial M \not = \emptyset$, we consider the boundary eigenvalue problem
\index{Eigenvalue problem} \index{Eigenvalue problem!with boundary}
\begin{equation}\label{E-evp-2bc}
\left\{
\begin{array}{rll}
-\Delta u + V u & = \lambda \, u &\text{in\quad} \inte(M),\\[5pt]
B(u) &= 0 &\text{on\quad} \partial M,
\end{array}
\right.
\end{equation}
where $B(u)$ is one the following boundary conditions:
\begin{equation}\label{E-evp-bc}
B(u) = \left\{
\begin{array}{ll}
u &\text{(Dirichlet),}\\[5pt]
\frac{\partial u}{\partial \nu} &\text{(Neumann),}\\[5pt]
 \frac{\partial u}{\partial \nu}- h\,u &\text{($h$-Robin).}
\end{array}
\right.
\end{equation}
In the Robin case, \index{Robin condition} $h$ is a  given $\Cty$ function on $\partial M$.\index{Eigenvalue problem} \index{Eigenvalue problem!$h$-Robin}

The associated quadratic forms are
\begin{equation}\label{E-evp-2dq}
\int_M \big( |du|^2_g + V u^2 \big) \, dv_g \text{\quad with domain \quad} H_0^1(M),
\end{equation}
for the Dirichlet problem\index{Dirichlet condition}, and
\begin{equation}\label{E-evp-2rq}
\int_M \big( |du|^2_g + V u^2 \big) \, dv_g +  \int_{\partial M} h \, (u_{\partial M})^2\, d\sigma_g \text{\quad with domain \quad} H^1(M)
\end{equation}
for the Neumann \index{Neumann condition}  problem (in this case $h=0$) and for the $h$-Robin problem.

For the above eigenvalue problems, the spectrum of $-\Delta + V$ is discrete, and consists of a sequence of non-negative eigenvalues with finite multiplicities,
\begin{equation}\label{E-evp-2ev}
\lambda_1 < \lambda_2 \le \lambda_3 \le \cdots \nearrow \infty ,
\end{equation}
listed in nondecreasing order, multiplicities accounted for, starting from the label $1$.
\index{Eigenvalue labeling}
In the sequel, we only consider \emph{real valued} eigenfunctions.

\begin{notation}\label{N-evp-0}  The \emph{eigenspace} \index{Eigenspace} associated with the eigenvalue $\lambda_k$ is denoted by $U(\lambda_k)$. Its dimension, the \emph{multiplicity} \index{Multiplicity} of $\lambda_k$, is  denoted by $\mult(\lambda_k)$ or $\dim U(\lambda_k)$.
\end{notation}%

\begin{remark}\label{R-evp-2}
Whenever necessary, we indicate the dependence on $M, g, V, h$ and the boundary condition. For example, $\lambda_k(M,g,V,\mf{d})$ denotes the eigenvalues of $-\Delta + V$ on $(M,g)$ with the Dirichlet boundary condition on $\partial M$.
\end{remark}%

\index{2-Z@$\cZ(u)$}
\begin{definitions}[Terminology]\label{D-evp2-00}\phantom{}
\begin{enumerate}[(i)]
\item The \emph{nodal set} \index{Nodal set} of a nontrivial eigenfunction $u$ is denoted by $\cZ(u)$, and defined by
\begin{equation}\label{E-evp-ns}
\cZ(u) = \overline{\set{x \in \inte(M) \mid u(x) = 0}}.
\end{equation}
When $\partial M \not = \emptyset$, $\cZ(u)$ is the closure in $M$ of the set of interior zeros of $u$.
\item In dimension $2$, the nodal set $\cZ(u)$ of an eigenfunction $u$ is also called the \emph{nodal line} \index{Nodal line} of $u$.
\item The \emph{nodal domains} \index{Nodal domain} of the eigenfunction $u$ are the connected components\footnote{In the sequel, unless otherwise stated, we use the word \emph{component} for the expression \emph{connected component}.} of $ \inte(M) \sm \cZ(u)$.
    \index{Component (connected)}
    We denote the number of nodal domains of $u$ by $\kappa(u)$.
    \index{1-kappa@$\kappa(u)$}
\end{enumerate}
\end{definitions}%

A key ingredient in the forthcoming proofs is the following.

\begin{theorem}[Courant's nodal domain theorem, \cite{Cour1923}]\label{T-RC}
\index{Courant nodal Theorem}
With the previous definitions, a $\lambda_k$-eigenfunction has at most $k$ nodal domains.
\end{theorem}%

For modern proofs we refer to \cite{Plej1956, BeMe1982} and \cite{SoZe2011}. See also \cite{Ales1998} which investigates the validity of Courant's theorem under low regularity assumptions on the coefficients of the operator, in dimension 2 and in higher dimension as well.\smallskip

Eigenfunctions associated with $\lambda_1$ are characterized by the fact that they have precisely one nodal domain. An eigenfunction associated with $\lambda_k, k\ge 2$, has at least two nodal domains. An eigenfunction associated with $\lambda_2$ has precisely two nodal domains.

\begin{remark}\label{R-RC1}
For any $k \ge 2$ and any $\lambda_k$-eigenfunction $u$, $\kappa(u) \ge 2$, a consequence of the fact that $u$ is $L^2$-orthogonal to a first eigenfunction. It turns out that this lower bound can in general not be improved, see \cite{Ster1925,Lewy1977}, \cite{BeBo1982}, \cite{BeHe2015s,BeHe2015r}, and the recent papers \cite{JuZe2022} and \cite{CiJLS2022}.
\end{remark}%

\begin{remark}\label{R-RC2}
As a matter of fact, $\sup\set{\kappa(u) \mid 0 \not = u \in U(\lambda_k)} < k$ when $k$ is large enough (depending on $M, g, V)$, see  Pleijel's paper \cite{Plej1956} \index{Pleijel Theorem} and Section~\ref{S-mcs} for more details and references.
\end{remark}

\subsection{Local structure of eigenfunctions near a zero}\label{SS-evp2}

\index{2-o@$\ord(u,x)$}
\begin{definitions}[Terminology]\label{D-evp2-02}
We say that a function $u$ \emph{vanishes at order} $n \ge 1$ \index{Vanishing order} at a point $x$, and we write $\ord(u,x) = n$, if (in a local coordinate system) the function and all its derivatives of order less than or equal to $(n-1)$ vanish at $x$, and at least one derivative of order $n$ does not vanish at $x$. A \emph{critical zero} \index{Critical zero} of $u$ is a point at which $u$ vanishes at order at least $2$ (i.e., $u(x)=0$ and $\nabla_x u = 0$). A critical zero $x$ of $u$ is called an \emph{interior critical zero} if $x \in \inte(M)$, and a \emph{boundary critical zero} if $x \in \partial M$.
\index{Critical zero!interior} \index{Critical zero!boundary}
\end{definitions}%

\begin{theorem}[Local structure theorem] \label{T-nodinfo}
\index{Local structure Theorem}
Let $u$ be a nontrivial eigenfunction of the Schr\"odinger operator $-\Delta +V$ on a smooth compact Riemannian surface $M$ (with or without boundary),  where $V$ is a smooth real valued  potential. Then, $u\in C^\infty(M)$,  and does not vanish at infinite order at any point of $M$. Furthermore, depending on the boundary condition on $\partial M$, $u$ has the following  properties.
\begin{enumerate}[(i)]
\item For $x_0\in M$ an interior point, if $u$ has a zero of order $\ell$ at $x_0$, then there exist local polar coordinates $(r,\omega)$ centered at $x_0$ such that
\begin{equation}\label{E-harmpoly}
u(x)=r^{\ell} \big( a \sin(\ell \omega) + b \cos(\ell \omega) \big) +\cO(r^{\ell+1}),
\end{equation}
where $a,b \in \R$, $a^2+b^2 \not = 0$.

\item For $x_0\in \partial M$, if a Dirichlet eigenfunction $u$ has a zero of order $\ell$ at $x_0$, then there exist local polar coordinates $(r,\omega)$ centered at $x_0$, such that
\begin{equation}\label{E-sinexp}
u(x)=a\, r^\ell \sin(\ell \omega) +\cO(r^{\ell +1})
\end{equation}
for some $a\in\mathbb R$, $a \not = 0$. The angle $\omega$ is chosen so that the tangent to the boundary at $x_0$ is given by the equation $\omega=0$.

\item For $x_0\in \partial M$, if a Robin eigenfunction $u$ has a zero of order $\ell$ at $x_0$, then there exist local polar coordinates $(r,\omega)$ centered at $x_0$, such that
\begin{equation}\label{E-cosexp}
u(x)=b\, r^\ell \cos(\ell \omega) +\cO(r^{\ell +1})
\end{equation}
for some $b\in\mathbb R$, $b \not = 0$. The angle $\omega$ is chosen so that the tangent to the boundary at $x_0$ is given by the equation $\omega=0$.
\end{enumerate}
\end{theorem}

We provide detailed proofs adapted to our purposes in Sections~\ref{S-lsns} and \ref{S-lsbs}. The starting point is to use the unique continuation theorem, see \cite{Aron1957} when $x_0$ is an interior point, and \cite{DoFe1990a} when $x_0$ is a boundary point.\smallskip

Classical references for the first assertion are \cite{Bers1955,HaWi1953}. For a proof of the local structure theorem under weaker regularity assumptions on the boundary,  and for references to the literature, we refer to \cite[Appendix~A]{GiHe2019}.  \medskip

From a local point of view, we have the following corollary.

\begin{corollary}\label{cor:nodloc}~
\begin{enumerate}[(i)]
\item Let $x_0\in \inte(M)$. If $u$ has a zero of order $\ell $ at $x_0$, then exactly $\ell$ nodal arcs pass through $x_0$. More precisely, in a neighborhood of $x_0 \in \inte(M)$, the nodal set $\cZ(u)$ consists of $2\ell$ semi-arcs emanating from $x_0$ tangentially to the rays $\set{\omega = \omega_j}$ where $\omega_j := \omega_{a,b} + j \frac{\pi}{\ell}, 0\le j < 2\ell$, for some angle $\omega_{a,b}$. The semi-tangents to these semi-arcs dissect the full unit circle in the tangent plane at $x_0$ into $2\ell $
    equal parts.
\item Let $x_0\in \partial M$. Let $u$ be a Dirichlet eigenfunction. If $u$ has a zero of order $\ell \ge 2$ at $x_0$, then exactly $(\ell-1)$ semi-arcs hit $\partial M$ at $x_0$, their  semi-tangents at $x_0$ dissect the half unit circle in the tangent plane at $x_0$ into $\ell$ sectors given by the equation $\sin(\ell \omega) =0$.
\item Let $x_0\in \partial M$. Let $u$ be a Robin eigenfunction. If $u$ has a zero of order $\ell \ge 1$ at $x_0$, then exactly $\ell$ semi-arcs hit $\partial M$ at $x_0$, their  semi-tangents  at $x_0$ dissect the half unit circle in the tangent plane at $x_0$ into $\ell$ sectors given by the equation $\cos(\ell \omega) =0$.
\end{enumerate}
\end{corollary}%

Assertion~(i) is proved in Section~\ref{S-lsns}.  For Assertions~(ii) and (iii), see Section~\ref{S-lsbs}, or the references in \cite[Appendix]{GiHe2019}. \medskip

Points at which nodal arcs meet in the interior $\inte(M)$, and points at which the nodal set hits the boundary $\partial M$ play an important role in the global understanding of nodal sets. The terminology in the following definitions comes from the framework of partitions. 

\index{2-S@$\cS(u)$}
\index{2-S@$\cS(u)$!$\cS_{\mathrm{i}}(u), \cS_{\mathrm{b}}(u)$}
\begin{definitions}[Terminology]\label{D-evp2-0}
Define the \emph{singular points} \index{Singular point} of an eigenfunction $u$ as follows. \index{Singular point} \index{Singular point!interior}
\index{Singular point!boundary}
\begin{enumerate}[(i)]
\item A point $x_0 \in \inte(M)$ is an \emph{interior singular point} of $u$ if and only if it is an interior critical zero; the set of interior singular points of $u$ is denoted by $\cS_{\mathrm{i}}(u)$. The \emph{index} $\nu(u,x_0)$
    \index{Index $\nu(u,x)$} \index{1-nu@$\nu(u,x)$}
    of the interior singular point $x_0$ is defined as the number of nodal semi-arcs emanating from $x_0$, $\nu(u,x_0) = 2\, \ord(u,x_0)$.
  \item A point $x_0 \in \partial M$ is a \emph{boundary singular point} of $u$ if and only if the nodal set $\cZ(u)$ hits the boundary $\partial M$ at $x_0$; the set of boundary singular points of $u$ is denoted by $\cS_{\mathrm{b}}(u)$. The \emph{index} $\rho(u,x_0)$ \index{Index $\rho(u,x)$} \index{1-rho@$\rho(u,x)$} of the boundary singular point $x_0$ is defined as the number of nodal semi-arcs hitting $\partial M$ at $x_0$.\\
      If $u$ is a Dirichlet eigenfunction, $\rho(u,x_0) = (\ord(u,x_0) - 1)$; if $u$ is a Robin eigenfunction, $\rho(u,x_0) = \ord(u,x_0)$.
\end{enumerate}
The set $\cS(u)$ of singular points of $u$ is the set $\cS(u) = \cS_{\mathrm{i}}(u) \cup \cS_{\mathrm{b}}(u)$.
\end{definitions}%

\begin{remark}\label{R-evp2-2}
The order of vanishing is \emph{semi-continuous} in the following sense. Let $\set{v_n}$ be a sequence of functions which converges to some $v$ uniformly in $C^{k}$ for some $k \ge 1$. Let $\set{x_n}$ be a sequence of points which converges to some $x$ in $M$. Assume that $\ord(v_n,x_n) \ge k$ for all $n$. Then, $\ord(v,x) \ge k$. Since they are defined in terms of order of vanishing, the indices $\nu$ and $\rho$ inherit this property.
\end{remark}%

 From the global point of view,  the set $\cS(u)$ is finite, and the components of its complement $\cZ(u) \sm \cS(u)$ are smooth $1$-dimensional submanifolds homeomorphic to either circles or open intervals whose boundaries consist of singular points.

\begin{definitions}[Terminology]\label{D-evp2-01}\phantom{}
\begin{enumerate}[(i)]
\item We call a circle-like component of $ \cZ(u)\sm \cS(u)$  \emph{a nodal circle} \index{Nodal circle}; we call an interval-like component,  \emph{a nodal interval}. \index{Nodal interval}
\item Let $I_{x,y}$ be a nodal interval with boundary $\set{x,y} \subset \cS(u)$. In this case, the closed nodal interval  $\bar{I}_{x,y} := I_{x,y}\cup \set{x,y}$  can be parametrized by arc-length, from $x$ to $y$, by $\gamma_{x,y} : [0,L_{x,y}] \to M$, with $\gamma_{x,y}(0) = x$ and $\gamma_{x,y}(L_{x,y})=y$ or, from $y$ to $x$, by $\gamma_{y,x}$ given by $\gamma_{y,x}(t) = \gamma_{x,y}(L_{x,y}-t)$. The semi-tangents to $\bar{I}_{x,y}$ at $x$ and $y$ are given by the local structure theorem. The point $x$ (resp. $y$) might be an interior singular point, or a boundary singular point. If $x=y$, we say that the component $\bar{I}_{x,x}$ is \emph{a nodal loop} \index{Nodal loop} at $x$. In this case the loop is not a smooth circle, but a continuous, piecewise $C^1$ circle.
\end{enumerate}
\end{definitions}

 From a global point of view, we have the following corollary of the structure theorem.

\begin{corollary}\label{cor:nodinfo}\phantom{}
\begin{enumerate}
  \item The nodal set of $u$ is the union of the finitely many singular points, the nodal circles in the interior of $M$, and the nodal intervals some of which may hit  $\partial M$.
  \item Each component of $\partial M$ is hit by an even number of nodal intervals: if $\Gamma$ is a component of $\partial M$, then
      \begin{equation*}
        \sum_{z \in \cS_{\mathrm{b}}(u) \cap \Gamma\,} \rho(u,z) \in 2\N.
      \end{equation*}
\end{enumerate}
\end{corollary}%

\begin{proof} The first assertion is well-known. We give the proof of the second assertion for completeness.\smallskip

\emph{\noid Dirichlet case.}~ The component $\Gamma$ is topologically a circle which meets $\cS_{\mathrm{b}}(u)$ at finitely many points $z_j$, $1 \le j \le k$, which are precisely the zeros of the normal derivative $\partial_{\nu}u(z)$. Choosing a parametrization $z$ of $\Gamma$ and taking  the local structure of $u$ at each $z_j$ into account, we see that each time $z$ passes some $z_j$, the sign of $\partial_{\nu}u$ is multiplied by $(-1)^{\rho(z_j)}$. Running through $\Gamma$ once, we must have
 \[\Pi_j (-1)^{\rho(z_j)}=(-1)^{\sum_j \rho(u,z_j)}=1.\]

\emph{\noid Robin case.}~ The proof is similar, actually simpler.
\end{proof}%

\subsection{Eigenfunctions with prescribed singular points}\label{SS-evp3}

\phantom{}  In order to bound multiplicities, we will use eigenfunctions with prescribed  singular points of sufficiently high index. Their existence is given by the following lemmas.  These lemmas appear in one form or another in the previous papers on eigenvalue multiplicity bounds, \cite[Theorem~3.4]{Chen1976}, \cite[Theorem~2.1]{Bess1980}, \cite[Lemma~4]{Nadi1987}, \cite[Proposition~2]{HoHN1999}, \cite[Lemma~2.9]{HoMN1999}.\medskip

The first lemma prescribes an interior singular point.

\begin{lemma}\label{L-zeroi} Let $M$ be a compact surface (with or without boundary), and $x$ an interior point. Let $U$ be a linear subspace of an eigenspace of $-\Delta + V$, see \eqref{E-evp-2c} or \eqref{E-evp-2bc}, with $\dim U = m \ge 2$.
\begin{enumerate}[(i)]
\item There exists a function $0 \not = u\in U$ such that $x$ is a singular point of $u$ with index $\nu(u,x) \ge 2\,\floor{\frac m2}$ (the integer part of $\frac m2$), equivalently with $\ord(u,x) \ge \floor{\frac m2}$.
\item Furthermore, if $m$ is odd, there exist at least two linearly independent such functions.
\end{enumerate}
\end{lemma}

\begin{proof} We use induction on $m$.  Recall that $\nu(u,x) = 2\ord(u,x)$. The assertion is clear when $m=2$. Assume $m=3$, and let $\set{u_1,u_2,u_3}$ be a basis of $U$. Then, we can find $0 \not = v_1 \in \spanv{u_1,u_2}$ such that $\ord(v_1,x) \ge 1$. The subspace $V_1$ of $U$ orthogonal\footnote{ Orthogonality is meant with respect to the inner product induced by the $L^2$-inner product of eigenfunctions.} to $v_1$ has dimension $2$, and hence there exists $0 \not = v_2 \in V_1$ such that $\ord(v_2,x) \ge 1$. Then $v_1$ and $v_2$ are two linearly independent functions in $U$ vanishing at order at least $1$ at $x$.\medskip

Assume that the lemma holds for $2p$ and $(2p+1)$ for some $p \ge 1$. Let $U$ be linear subspace of an eigenspace with dimension $(2p+2)$, and basis $\set{u_1,\ldots,u_{2p+2}}$. By the induction hypothesis, in the subspace $V_1 := \spanv{u_1,\ldots,u_{2p+1}}$, we can find two linearly independent functions $v_1, v_2$ such that $\ord(v_i,x) \ge p$. If one of them vanishes at order at least $(p+1)$, the assertion for $U$ is satisfied. If not, by Theorem~\ref{T-nodinfo}~(i), there exist $(a_i,b_i)$, $i=1,2$, with $a_i^2+b_i^2 \not = 0$, such that
\begin{equation*}
v_i = r^{p} \big( a_i \sin(p\, \omega) + b_i \cos(p\, \omega) \big) +\cO(r^{p+1}).
\end{equation*}

The subspace $V_2$ of $U$ orthogonal to $v_1$ and $v_2$ has dimension $2p$ and hence, there exists $0 \not = v_3 \in V_2$ such that $\ord(v_3,x) \ge p$. If $v_3$ vanishes at order at least $(p+1)$ at $x$, we are done. Otherwise, there exist $(a_3,b_3)$, with $a_3^2+b_3^2 \not = 0$ such that
\begin{equation*}
v_3 = r^{p} \big( a_3 \sin(p\, \omega) + b_3 \cos(p\, \omega) \big) +\cO(r^{p+1}).
\end{equation*}

The functions $r^p \big( a \sin(p\, \omega) + b \cos(p\, \omega) \big)$ are the homogeneous harmonic polynomials of degree $p$ in $\R^2$, a vector space of dimension $2$. The three polynomials $r^{p} \big( a_i \sin(p\, \omega) + b_i \cos(p\, \omega)\big)$, $i \in \set{1,2,3}$ must be linearly dependent, and hence there exists a nontrivial linear combination of $v_1,v_2,v_3$ which vanishes at order at least $(p+1)$ at $x$.\smallskip

Let $U$ be an eigenspace with dimension $(2p+3)$, with basis $\set{u_1,\ldots,u_{2p+3}}$. By the previous proof, in the subspace $V_1 := \spanv{u_1,\ldots,u_{2p+2}}$, there exists $0 \not = v_1$ such that $\ord(v_1,x) \ge (p+1)$. For the same reason, in the subspace $V_2$ orthogonal to $v_1$, there exists $0 \not = v_2$ such that $\ord(v_2,x) \ge (p+1)$. The functions $v_1, v_2$ are two linearly independent functions in $U$ vanishing at order at least $(p+1)$ at $x$.\\
The proof of Lemma~\ref{L-zeroi} is complete.
\end{proof}

The next lemmas prescribe respectively one or two boundary singular points.

\begin{lemma}\label{L-zero1} Let $M$ be a compact surface with
boundary, and $x \in \partial M$. Let $U$ be a linear subspace of an eigenspace of $-\Delta + V$ , see \eqref{E-evp-2bc}, with $\dim U = m \ge 2$. Then, there
exists a function $0 \not = u\in U$ such that $x$ is a boundary singular point of $u$ with index $\rho(u,x) \ge (m-1)$.\index{Index $\nu(u,x)$}
\end{lemma}

\begin{proof} We use induction on $m$.  Recall that $\rho(u,x) = (\ord(u,x) - 1)$ for Dirichlet eigenfunctions, resp. $\rho(u,x) = \ord(u,x)$ for Robin eigenfunctions. \smallskip

\emph{\noid Dirichlet boundary condition.}~ When $m=2$, the assertion is clear. Assume it is true for some $m\ge 2$. Let $U$ be a linear subspace of an eigenspace with dimension $(m+1)$, and basis $\set{u_1,\ldots,u_{m+1}}$. Consider the subspace $V_1 = \spanv{u_1,\ldots,u_m}$. By the induction hypothesis, there exists $0 \not = v_1 \in V_1$ such that $\ord(v_1,x) \ge m$. If $v_1$ vanishes at order at least $(m+1)$, we are done. Otherwise, by Theorem~\ref{T-nodinfo}, Equation~\eqref{E-sinexp}, there exists $a_1\not = 0$ such that, in local polar coordinates at $x$,
\begin{equation*}
v_1(z)=a_1\,r^m \sin(m \,\omega) +\cO(r^{m +1}).
\end{equation*}
The subspace $V_2=\set{u\in U \mid u \perp v_1}$ orthogonal to $v_1$ has dimension $m$, and hence there exists $0 \not = v_2 \in V_2$ such that $\ord(v_2,x) \ge m$. If $v_2$ vanishes at order at least $(m+1)$, we are done. Otherwise, as above we can write
\begin{equation*}
v_2(z)=a_2\,r^m \sin(m\, \omega) +\cO(r^{m +1})
\end{equation*}
for some $a_2 \not = 0$, and hence the linear combination $v = a_2 v_1 - a_1 v_2$ vanishes at order at least $(m+1)$. \quad \qedc \smallskip

\emph{\noid Robin boundary condition.}~ When $m=2$, the assertion is clear. Assume it is true for some $m\ge 2$. Let $U$ be a linear subspace of an eigenspace with dimension $(m+1)$, with basis $\set{u_1,\ldots,u_{m+1}}$. Consider the subspace $U_1 = \spanv{u_1,\ldots,u_m}$.\ By the induction hypothesis, there exists $0 \not = v_1 \in U_1$ such that $\ord(v_1,x) \ge (m-1)$. If $v_1$ vanishes at order at least $m$, we are done. Otherwise, by Theorem~\ref{T-nodinfo},  Equation~\eqref{E-cosexp}, there exists $b_1\not = 0$ such that, in local polar coordinates at $x$,
\begin{equation*}
v_1(z)=b_1\,r^m \cos((m-1) \omega) +\cO(r^{m+1}).
\end{equation*}
We can then consider the subspace $U_2$ orthogonal to $v_1$ in $U$, and conclude by arguing as above. \quad \qedc \\
The proof of Lemma~\ref{L-zero1} is complete.
\end{proof}%

\begin{lemma}\label{L-zero2} Let $M$ be a compact surface with
boundary, and $x, y \in \partial M$, with $x \not = y$. Let $U$ be a linear subspace of an eigenspace of $-\Delta + V$, see \eqref{E-evp-2bc}, with $\dim U = m \ge 3$. Then, there exists a function $0 \not = u\in U$ such that $x$ and $y$ are boundary singular points of $u$ with indices $\rho(u,x) \ge (m-2)$ and $\rho(u,y) \ge 1$.
\end{lemma}

\begin{proof} \phantom{}

\emph{\noid Dirichlet boundary condition.}~ Choose $\set{u_1,\ldots, u_m}$ a basis of $U$. Looking at a general element $u=\sum \alpha_j \phi_j$ in $U$, the condition at $y$ reads
\[
\sum_{j=1}^m  \alpha_j (\partial_\nu \phi_j)(y) =0.
\]
There are two cases.
\begin{enumerate}[$\diamond$]
\item If $\partial_\nu \phi_j(y) =0$ for all $j$, the condition at $y$ is satisfied for any $u\in U$.
\item If $\partial_\nu \phi_j(y) \neq 0$ for some $j$, then there exists a subspace $U'\subset U$ of dimension $(m-1) \ge 2$ such that the condition at $y$ is satisfied for any $u\in U'$.
\end{enumerate}
Apply Lemma~\ref{L-zero1} with $U$ in the first case and with $U'$ in the second case. \quad \qedc 

\emph{\noid Robin boundary condition.}~ The condition $\rho(u,y) \ge 1$ holds if and only if $u$ vanishes at $y$. Since $m \ge 3$ there exists a linear subspace $U' \subset U$, with $\dim U' \ge (m-1) \ge 2$ such that any $u \in U'$ satisfies $u(y) = 0$. Then, Lemma~\ref{L-zero1} implies that there exists $0 \neq u \in U'$ such that $\rho(u,x) \ge (m-2)$. \quad \qedc \\
The proof of Lemma~\ref{L-zero2} is complete. \end{proof}

\begin{lemma}\label{L-zeroc}
Let $M$ be a compact surface. Let $U$ be a linear subspace of an eigen\-space of $-\Delta + V$, see \eqref{E-evp-2c} or \eqref{E-evp-2bc}.
\begin{enumerate}[(i)]
  \item Let $x \in \inte(M)$, and let $u_1, u_2, u_3$ be three linearly independent functions in $U$, such that $ \nu(u_1,x) = \nu(u_2,x) = \nu(u_3,x) \ge 2$. Then, there exists an eigenfunction $0 \neq u \in \spanop\set{u_1,u_2,u_3}$ such that $\nu(u,x) \ge \nu(u_1,x) + 2$.
  \item Let $x \in \partial M$, and let $u_1, u_2$ be two linearly independent functions in $U$, such that $ \rho(u_1,x) = \rho(u_2,x) \ge 1$. Then, there exists $0 \neq u \in \spanop\set{u_1,u_2}$ such that $\rho(u,x) \ge \rho(u_1,x) + 1$.
\end{enumerate}
\end{lemma}%

\begin{proof} Since the index of a singular point can be expressed in terms of the vanishing order, \index{Vanishing order} the lemma follows from Theorem~\ref{T-nodinfo}. Indeed, under the assumption of Assertion~(i), we can write
\begin{equation*}
u_i(z) = p_i(z-x) + \cO(|z-x|^{k+1}),
\end{equation*}
in local coordinates centered at $x$,  where $p_i$ is a nonzero harmonic homogeneous polynomial in two variables,  of degree $k = \ord(u_1,x)$. Since the vector space of such polynomials has dimension $2$, there exist real numbers $\alpha_1, \alpha_2$ and $\alpha_3$, not all of them equal to zero, such that $\alpha_1 \, p_1 + \alpha_2 \, p_2 + \alpha_3 \, p_3 = 0$. It follows that $\alpha_1 \, u_1 + \alpha_2 \, u_2 + \alpha_3 \, u_3$ vanishes at order at least $(k+1)$ at $x$. This proves Assertion~(i). \smallskip

The proof of Assertion~(ii) is similar, using the local forms \eqref{E-sinexp} or \eqref{E-cosexp} depending on the boundary condition, Dirichlet or Robin.
\end{proof}

For later purposes, we introduce the following notation.

\index{2-u@$\breve{u}$}
\begin{notation}\label{N-evp-dr}
Let $u$ be an eigenfunction of $-\Delta + V$ of the compact surface with boundary $M$, see \eqref{E-evp-2bc}. Define the function $\breve{u}$ on $\partial M$ by
\begin{equation}\label{E-evp-dr}
\breve{u} = \left\{
\begin{array}{ll}
u|_{\partial M} & \text{in the Robin case,}\\[5pt]
\partial_{\nu}u & \text{in the Dirichlet case.}
\end{array}
\right.
\end{equation}
Then, for any $y \in \partial M$, $\rho(u,y) \ge 1$ if and only if
$\breve{u}(y) = 0$.
\end{notation}%

The following lemma will also be useful.

\begin{lemma}\label{L-breve}
Let $u$ be an eigenfunction of $-\Delta + V$, see \eqref{E-evp-2bc}.
\begin{enumerate}[(i)]
  \item If $u$ is a Dirichlet eigenfunction, and $y \in \partial M$, then $u$ vanishes at order $k$ at $y$ if and only if the function $\partial_{\nu}u$ vanishes at order $(k-1)$ at $y$ along $\partial M$.
  \item If $u$ is a Robin eigenfunction, and $y \in \partial M$, then $u$ vanishes at order $k$ at $y$ if and only if the function $u|_{\partial M}$ vanishes at order $k$ at $y$ along $\partial M$.
\end{enumerate}
Therefore, the order of vanishing \index{Vanishing order} of the function $\breve{u}$ at some boundary point $y$ is precisely the number $\rho(u,y)$ of nodal  arcs hitting $\partial M$ at $y$.
\end{lemma}%

\begin{proof} The proof is by induction on $k$. The equation $\Delta u = (V-\lambda)u$ implies relations between the derivatives of $u$ of degree $k$, evaluated at $y$, assuming that the derivatives of order less than or equal to $(k-1)$ vanish at $y$.\smallskip

More precisely, according to \cite[Section~2]{YaZh2021}, fixing some $y \in \partial M$, we can choose local boundary isothermal coordinates at $y$ transforming  the equation $(-\Delta + V)u = \lambda u$ in a neighborhood of $y$ into the equation
\begin{equation*}
\Delta v= A v \tag{e}
\end{equation*}
in some half-ball $\set{(\xi_1,\xi_2) \in \R^2 \mid \xi_1^2+\xi_2^2 < a^2 \text{~and~}\xi_2 > 0}$, where $0$ is the image of $y$. Here, $\Delta$ is the ordinary Laplacian in the variables $(\xi_1,\xi_2)$, $a$ is some given positive number, $A$ and $v$ are $C^{\infty}$ up to the boundary, and correspond to $(V-\lambda)$ and $u$ respectively.\smallskip

In the proof, we use the following conventions.
\begin{enumerate}[$\diamond$]
  \item The symbol $\equivo{t}$ indicates a trivial identity.
  \item The symbol $\equivo{e}$ indicates an identity which follows from the above identity $(e)$.
  \item The symbol $\equivo{(m)}$ indicates an identity which holds up to a linear combination of derivatives of $v$ of order less than or equal to $m$.
  \item The symbol $\partial^{(p,q)}$ stands for $\frac{\partial^{p+q}}{\partial \xi_1^p\,\partial \xi_2^q}$.
\end{enumerate}

For $k \ge 2$, we have
\begin{equation*}
\begin{array}{lcl}
\partial^{(k-2q,2q)}v & \equivo{t} & \partial^{(k-2q,2q-2)} \partial^{(0,2)} v\\[5pt]
& \equivo{e} & \partial^{(k-2q,2q-2)} \big( -\, \partial^{(2,0)}v + A v\big)\\[5pt]
& \equivo{(k-2)} & -\, \partial^{(k-2q+2,2q-2)}v.
\end{array}%
\end{equation*}

Assuming that $u$ vanishes at order larger than or equal to $(k-1)$ at $(0,0)$, we obtain that
\begin{equation*}
\partial^{(k-2q,2q)}u(0,0) = (-1)^q \partial^{(k,0)}u(0,0), \text{~for~} q \in \set{0,1,\ldots,\floor{\frac k2}}. \tag{a}
\end{equation*}

Similarly,
\begin{equation*}
\begin{array}{lcl}
\partial^{(k-2q-1,2q+1)}v & \equivo{t} & \partial^{(k-2q-1,2q-1)} \partial^{(0,2)} v\\[5pt]
& \equivo{e} & \partial^{(k-2q-1,2q-1)} \big( -\, \partial^{(2,0)}v + A v\big)\\[5pt]
& \equivo{(k-2)} & -\, \partial^{(k-2q+1,2q-1)}v.
\end{array}%
\end{equation*}

Assuming that $u$ vanishes at order larger than or equal to $(k-1)$ at $(0,0)$, we obtain that
\begin{equation*}
\partial^{(k-2q-1,2q+1)}u(0,0) = (-1)^q \partial^{(k-1,1)}u(0,0), \text{~for~} q \in \set{0,1,\ldots,\floor{\frac{k-1}{2}}}. \tag{b}
\end{equation*}

\noid \emph{Dirichlet case.}~ In this case, $\breve{v}(\xi_1) = \partial^{(0,1)}v(\xi_1,0)$. Since $v(\xi_1,0) \equiv 0$, Equation~(a) implies that
$\partial^{(k-2q,2q)}v(0,0) = 0$ for all $q \in \set{0,\ldots,\floor{\frac k2}}$. If $v$ vanishes at order greater then or equal to $(k-1)$, Equation~(b) implies that
$\partial^{(k-2q-1,2q+1)}v(0,0) = (-1)^q \partial^{(k-1)}\breve{v}(0)$ for all $q \in \set{0,\ldots,\floor{\frac{k-1}{2}}}$. It follows that if $v$ vanishes at order greater than of equal to $(k-1)$, then $v$ vanishes at order greater than or equal to $k$ if and only if $\partial^{(k-1)}\breve{v}(0) = 0$.\medskip

\noid \emph{Robin case.}~ In this case, $\breve{v}(\xi_1) = v(\xi_1,0)$. Assuming that $v$ vanishes at order at least $(k-1)$, Eq.~(a) implies that
$\partial^{(k-2q,2q)}v(0,0) = (-1)^q \partial^{k}\breve{v}(0)$. Since $\partial^{(0,1)}v(\xi_1,0) \equiv B(\xi_1) u(\xi_1,0)$ (Robin condition), Equation~(b) implies that\\ $\partial^{(k-2q-1,2q+1)}u(0,0) = 0$. Therefore, if $v$ vanishes at order at least $(k-1)$ at $(0,0)$, then $v$ vanishes at order at least $k$ if and only if $\partial^{k}\breve{v}(0)=0$.\medskip

We have proved that $v$ vanishes at order at least $k$ (resp. equal to $k$) at $(0,0)$ if and only if $\breve{v}$ vanishes at least at order $\rho(v,(0,0))$ (resp. equal to $\rho(v,(0,0))$) at $(0,0)$.
\end{proof}

\subsection{A global property of nodal sets}\label{SS-gpns}

\begin{lemma}\label{L-hdn}
Let $(M,g)$ be a compact Riemannian surface. Let $w_n, w : M \to \R$ be continuous functions with zero sets $K_n := w_n^{-1}(0)$ and $K := w^{-1}(0)$. Assume that $w_n$ converges to $w$ uniformly.
\begin{enumerate}[(i)]
  \item The limit points of the sequence $\set{K_n}$ with respect to the Hausdorff distance \index{Hausdorff distance} associated with the Riemannian distance of $(M,g)$ are compact and contained in $K$. They are connected if the sets $K_n$ are connected.
  \item If $K_n, K$ are nodal sets of eigenfunctions of $-\Delta + V$, then the sequence $\set{K_n}$ converges to $K$ in the Hausdorff distance.
\end{enumerate}
\end{lemma}%

\begin{proof} The properties that the sequence $\set{K_n}$ has limit points, and that they are compact (and connected if the sets $K_n$ are connected) are general.  Assume that a  subsequence $\set{K_{s(n)}}$ tends to $K'$ in the Hausdorff distance. Assume that $K' \not \subset K$. Then, there exists some $z' \in K'$ such that $w(z') \neq 0$. Let $\varepsilon_0>0$ be such that $|w(z')|=3\varepsilon_0$.\smallskip

(a) There exists $\eta_0$ such that $d(y_1,y_2) < \eta_0$ implies that $|w(y_1)-w(y_2)| < \varepsilon_0$.\smallskip

(b) Since $\set{K_{s(n)}}$ tends to $K'$ in the Hausdorff distance, there exists $N(\eta_0)$ such that for $n \ge N(\eta_0)$, $K' \subset \cU(K_{s(n)},\eta_0)$, the $\eta_0$-neighborhood of $K_{s(n)}$. In particular there exists some point $z_{s(n)} \in K_{s(n)}$ such that $d(z',z_{s(n)}) < \eta_0$.\smallskip

(c) There exists $N(\varepsilon_0)$ such that for $n \ge N(\varepsilon_0)$, $\| w - w_{s(n)}\| < \varepsilon_0\,$.\smallskip

Take $n \ge \max\set{N(\varepsilon_0), N(\eta_0)}$. Then,
\begin{equation*}
w(z') = w(z') - w(z_{s(n)}) + w(z_{s(n)}) - w_{s(n)}(z_{s(n)})
\end{equation*}
and we conclude that $|w(z')| < 2\varepsilon_0$, a contradiction.\medskip

The second assertion uses the nodal character. Assume that there exists $z \in K \sm K'$. Then $d(z,K') =: 2\eta > 0$, and for $n$ large enough, $d(z,K_{s(n)}) > \eta$. Since $K$ is a nodal set, there exists a small arc through $z$, from some $z_{-}$ to some $z_{+}$ such that $w(z_{-}) < 0$ and $w(z_{+}) > 0$. It follows that for $n$ large enough we also have $w_{s(n)}(z_{-}) < 0$ and $w_n(z_{+}) > 0$. This shows that $w_{s(n)}$ must vanish on this small arc, and hence that there exists some $z'_n \in K_{s(n)}$ close to $z$ contradicting the fact that $d(z,K_n) > \eta$. Since the only possible limit point of $\set{K_n}$ is $K$, the assertion follows.
 \end{proof}

\section{Euler Type Formulas for Nodal Sets}\label{S-pargef}

In this section, $M$ denotes a smooth surface homeomorphic to $\bS^2$,  or a smooth bounded domain $M \subset \R^2$ with boundary $\partial M$.

\subsection{Graphs associated with a nodal set}\label{SS-gr}
\index{Graph}
We are interested in Euler type formulas for the nodal set $\cZ(u)$ of an eigenfunction $u$ of the operator $-\Delta + V$ in $(M,g)$, where $g$ is some smooth Riemannian metric on $M$, and $V$ a smooth real valued potential. When $\partial M \neq \emptyset$, we assume a Dirichlet, Neumann or $h$-Robin boundary condition on $\partial M$, see \eqref{E-evp-2c} or \eqref{E-evp-2bc}.\smallskip

In this framework, the singular set $\cS(u) = \cS_{\mathrm{i}}(u) \cup \cS_{\mathrm{b}}(u)$ of $u$ is finite, and the set $\big( \cZ(u) \cup \partial M \big)\setminus \cS(u)$ consists of finitely many  components $\set{C_j}$ which are diffeomorphic to either circles or open intervals whose extremities are points in $\cS(u)$.\smallskip

The pair $\cG_u = \big( \cS(u), \set{C_j} \big)$
\index{2-G@$\cG_u$}
is in general not a \emph{multiple graph} in the sense of \cite[Section~1.10]{Dies2017}. Indeed, among the components $\set{C_j}$, there might be \emph{nodal circles} or  components of $\partial M$ which do not intersect $\cZ(u)$. It is not a \emph{simple graph}, as some of the $C_j$'s might form multiple edges.

\begin{notation}\label{N-gr-not}
From now on, we use the definition of graph given in \cite{Gibl2010} (i.e., ``graph = simple graph''). \index{Graph} Given a graph $G$, we denote by $\alpha_0(G)$ the number of vertices, by $\alpha_1(G)$ the number of edges, and by $c(G)$ the number of components of $G$. For a graph $G$ embedded in a surface $M$, we denote by $r(G,M)$ the number of components of $M \sm G$.
\index{1-alpha@$\alpha_0(G)$} \index{1-alpha@$\alpha_1(G)$}
\index{2-c@$c(G)$} \index{2-r@$r(G,M)$}
\end{notation}%

 With the pair $\cG_u$ we will associate a graph (not necessarily connected) to which we will apply the (Euler) formula in \cite[Theorem~1.27]{Gibl2010}. The vertices of the graph should comprise the singular points of $\cS(u)$, and the edges should comprise both sub-arcs of $\cZ(u)$ and sub-arcs of $\partial M$. Taking into account the fact that $\cG_u$ is in general not a graph, we first define a multigraph $G_0 := G_0(u,M)$, as follows.
\index{1-G@$G_0(u,M)$}

\begin{itemize}
  \item[$\diamond$]  Let $e := e(u,M)$ be the number of components of $\big( \cZ(u) \cup \partial M \big) \sm \cS(u)$ which are homeomorphic to a circle. We first choose one vertex for each such component. Call $\set{v_1,\ldots, v_{e}}$ these vertices, if any. Define the set  $V_0$ of vertices of the graph $G_0$ as $\cS(u) \cup \set{v_1,\ldots, v_{e}}$.
  \item[$\diamond$] Define the set $E_0$ of edges of $G_0$ as the set of components of $\big( \cZ(u) \cup \partial M\big) \sm V_0$ (they are all homeomorphic to intervals).
\end{itemize}

\begin{lemma}\label{L-eucl-2}
The pair $G_0 = (V_0,E_0)$ is a multigraph. \index{Multigraph} The number of vertices $\alpha_0(G_0)$, and the number of edges $\alpha_1(G_0)$ of $G_0$ are given by
\begin{equation}\label{E-eucl-2}
\left\{
\begin{array}{l}
\alpha_0(G_0) = e + |\cS_{\mathrm{i}}(u)| + |\cS_{\mathrm{b}}(u)|,\\[5pt]
\alpha_1(G_0) = e + \frac{1}{2} \, \big( \sum_{y \in \cS_{\mathrm{i}}(u)} \nu(u,y)  + \sum_{z\in \cS_{\mathrm{b}}(u)} \rho(u,z)  \big) + |\cS_{\mathrm{b}}(u)|,
\end{array}
\right.
\end{equation}
where $e$ is defined above, where $|\cS_{\mathrm{i}}(u)|$ (resp $|\cS_{\mathrm{b}}(u)|$) denotes the number of interior (resp. boundary) singular points of $u$, and where the numbers $\nu$ and $\rho$ are as in Definition~\ref{D-evp2-0}. In particular,
\begin{equation}\label{E-eucl-4}
\alpha_1(G_0) - \alpha_0(G_0) = \frac{1}{2} \, \big( \sum_{y \in \cS_{\mathrm{i}}(u)} (\nu(u,y)-2)  + \sum_{z\in \cS_{\mathrm{b}}(u)} \rho(u,z)  \big) .
\end{equation}
\end{lemma}%

\begin{proof}  Clearly $G_0$ is a multigraph in the sense of  \cite[Section~1.10]{Dies2017}. The formula for $\alpha_0$ is clear. The first term in the right-hand side for $\alpha_1$ counts both the number of points $v_j$ and the number of components of $\big( \cZ(u) \cup \partial M \big) \sm \cS(u)$ which are circles. The second term counts the number of edges between two singular points, each one being a simple curve contained in $\cZ(u) \sm \cS(u)$. The third term counts the number of edges determined by the boundary singular points on the components of $\partial M$ which intersect $\cS(u)$. \end{proof}

\begin{remark}\label{R-eucl-2L}
The second relation in \eqref{E-eucl-2} also follows from the relation
\begin{equation}\label{E-eucl-2R}
2 \, \alpha_1(\Gamma) = \sum_{x \in V(\Gamma)} \deg_{\Gamma}(x)
\end{equation}
which holds for any multi-graph $\Gamma$ (here, $\deg_{\Gamma}(x)$, the degree of the vertex $x$, \index{Degree (vertex)} is the number of edges of $\Gamma$ one of whose ends is $x$), see \cite[Section~1.2]{Dies2017}.
\end{remark}%

Note that the multigraph $G_0(u,M)$
\begin{itemize}
\item[(a)] might contain loops at some singular point or at one vertex $v_j$ ;
\item[(b)] might contain pairs of distinct vertices in $V_0$ linked by more than one edge.
\end{itemize}

We now transform the multigraph $G_0(u,M)$ into a graph $G(u,M)$ (in the sense of \cite[p.~10]{Gibl2010}), \index{2-G@$G(u,M)$}
keeping track of the number of vertices and edges. More precisely, if necessary, we introduce additional vertices and edges by performing one of the following  vertex-edge additions.

\begin{definition}\label{D-gr-2} We call \emph{vertex-edge additions} \index{Vertex-edge addition} the following modifications of the graph $G_0(u,M)$.
\begin{enumerate}
  \item[(i)] If a component $\Gamma$ of $\big( \cZ(u) \cup \partial M\big) \sm V_0$ is bounded by only one vertex $v$ (i.e., there is a loop at $v$), we add two extra vertices $v_1, v_2$ on $\Gamma$, and replace the edge $\Gamma$ by three edges, the components of $\Gamma \sm \set{v_1,v_2}$.
  \item[(ii)] If two distinct vertices of $V_0$ are the endpoints of more than one edge in $E_0$, i.e., of more than one component $\Gamma_j$ of $\big( \cZ(u) \cup \partial M\big) \sm V_0$, we add one extra vertex $w_j$ to each $\Gamma_j$, and replace $\Gamma_j$ by two edges, the components of $\Gamma_j \sm \set{w_j}$.
\end{enumerate}
\end{definition}%

Figure~\ref{F-gpart-03} illustrates the transformation of $\cZ(u) \cup \partial M$ into a graph. Boundary sub-arcs (black) are represented by thicker lines than nodal sub-arcs (red). Squares (blue) represent vertices $v_i$ initially attached to each circle component. Stars (green)  represent vertices added in the vertex-edge additions. Sub-figure~(a) illustrates the transformation of circle components of $\partial M$ or $\cZ(u)$, and loops in $\cZ(u)$ into graphs. Sub-figure~(b) illustrates the transformation of multiple edges into graphs\footnote{As indicated here, whenever useful, we have inserted symbols to identify the colors in a black and white printing setting.}.

\begin{figure}[!ht]
\centering
\begin{subfigure}[t]{.5\textwidth}
\centering
\includegraphics[width=\linewidth]{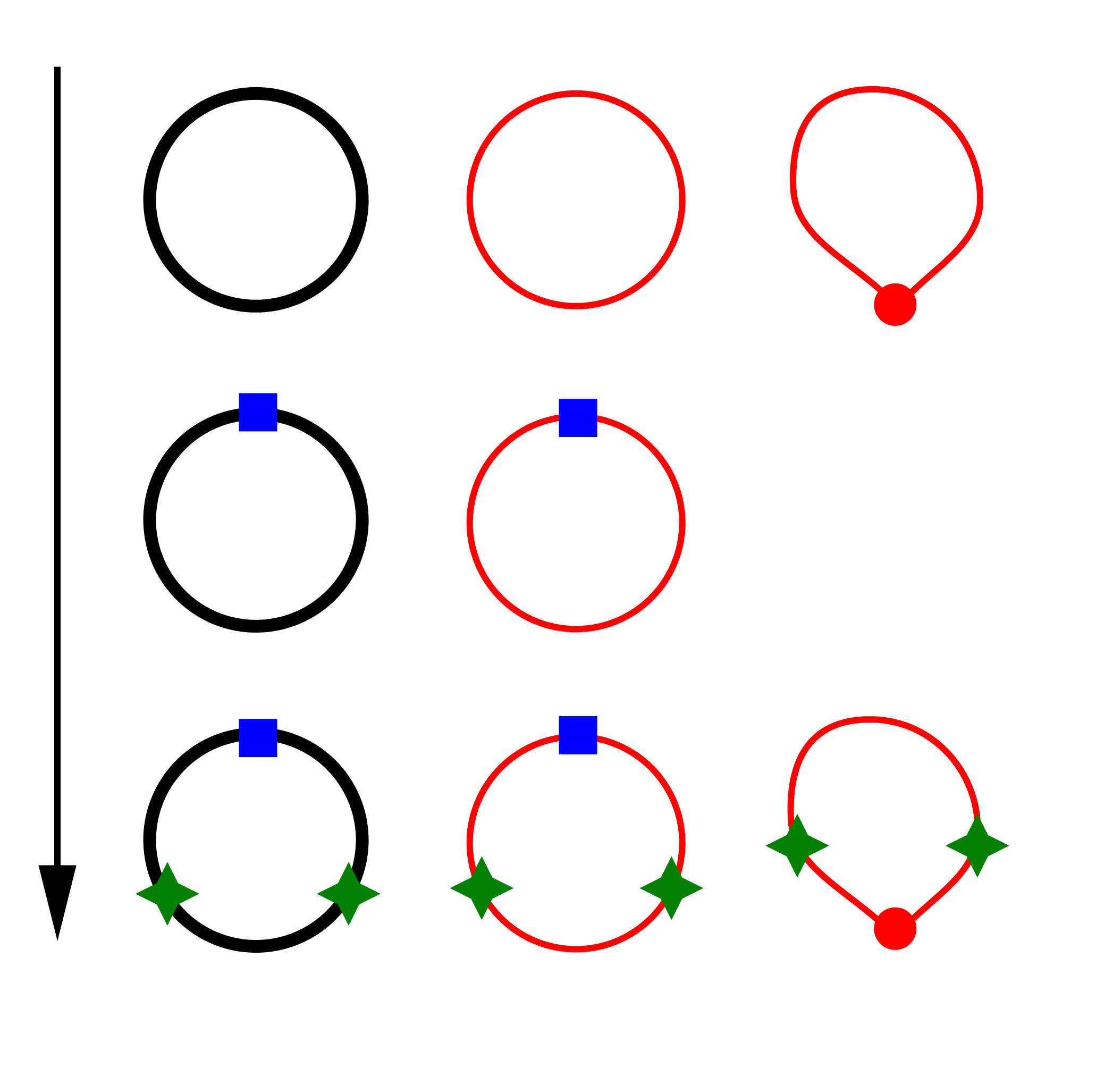}
\caption{Circle components and loops}
\end{subfigure}\qquad
\begin{subfigure}[t]{.33\textwidth}
\centering
\includegraphics[width=0.89\linewidth]{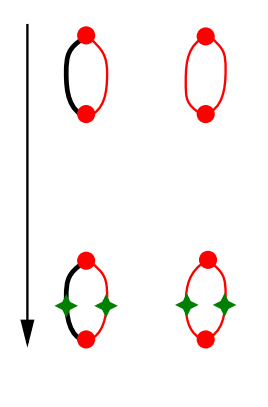}
\caption{Multiple edges}
\end{subfigure}
\caption{Vertex-edge additions}\label{F-gpart-03}
\end{figure}

The following lemma is clear.

\begin{lemma}\label{L-gr-4}
Performing finitely many  vertex-edge additions transforms the multigraph $G_0(u,M)$ into a graph $G(u,M)$.
\end{lemma}

\begin{notation}\label{N-parti-2a}
For an eigenfunction of $-\Delta + V$ in $M$, we introduce the following numbers:
\begin{itemize}
  \item[(a)] $\beta(u)$\index{1-beta@$\beta(u)$} is defined as $\beta(u) = b_0(\cZ(u) \cup \partial M) - b_0(\partial M)$, the difference between the number of components of $\cZ(u) \cup \partial M$, and the number of components of $\partial M$;
  \item[(b)] $\kappa(u)$\index{1-kappa@$\kappa(u)$} denotes the number of nodal domains of $u$;
  \item[(c)] $\sigma(u) = \sigma_{\mathrm{i}}(u)  + \sigma_{\mathrm{b}}(u) $ weighs the singular points of $u$,
      \index{1-sigma@$\sigma(u)$}
      \index{1-sigma@$\sigma(u)$!$\sigma_{\mathrm{i}}(u),\sigma_{\mathrm{b}}(u)$}
  \begin{equation*}
  \left\{
  \begin{array}{ll}
  \sigma_{\mathrm{i}}(u) = \frac{1}{2}\, \sum_{x \in \mathcal{S}_{\mathrm{i}}(u)\,}\big( \nu(u,x) - 2\big),\\[5pt]
\sigma_{\mathrm{b}}(u) = \frac{1}{2}\, \sum_{x \in \mathcal{S}_{\mathrm{b}}(u)\,} \rho(u,x).
  \end{array}%
  \right.
  \end{equation*}
\end{itemize}%
\end{notation}%

The following lemma follows from the fact that the number $\alpha_0 - \alpha_1$ remains un\-changed if we perform a vertex-edge addition.

\begin{lemma}\label{L-gr-6}
For the graph $G:=G(u,M)$ obtained by performing  vertex-edge additions on the multigraph $G_0:=G_0(u,M)$, we have
\begin{equation*}
\begin{array}{rll}
\alpha_1(G) - \alpha_0(G) & = \quad \alpha_1(G_0) - \alpha_0(G_0) & = \quad \sigma(u),\\[5pt]
c(G) & = \quad c(G_0) & = \quad b_0(\cZ(u) \cup \partial M),\\[5pt]
r(G,M) & = \quad r(G_0,M) & = \quad \kappa(u).
\end{array}%
\end{equation*}
\end{lemma}%

\subsection{Euler type formulas for nodal sets}\label{SS-etf}~
\index{Euler formula for nodal sets}
According to \cite[Theorem~1.27]{Gibl2010} a graph $\bar{G}$ in $\R^2$, divides the plane into $r(\bar{G})$ regions, with
\begin{equation}\label{E-etf-00}
r(\bar{G}) = \alpha_1(\bar{G}) - \alpha_0(\bar{G}) + c(\bar{G}) + 1.
\end{equation}
\index{Euler formula}

\begin{proposition}\label{P-etf0-2}
Let $u$ be an eigenfunction of $-\Delta + V$ on $M$ (with Dirichlet, Neumann or $h$-Robin boundary condition if $\partial M \neq \emptyset$), where $M$ is topologically a sphere or a domain in $\R^2$. The following Euler type formula holds for $\cZ(u)$,
\begin{equation}\label{E-etf0-2}
\kappa(u) = 1 + \beta(u) + \sigma(u).
\end{equation}%
\end{proposition}

\begin{proof} \phantom{}

\noid If $M=\bS^2$, we can view the graph $G := G(u,M)$ as a graph in $\R^2$. Applying \eqref{E-etf-00}, we obtain
\begin{equation*}
r(G,\R^2) = \alpha_1(G) - \alpha_0(G) + c(G) + 1,
\end{equation*}
and it suffices to apply Lemma~\ref{L-gr-6} with $\partial M = \emptyset$.\smallskip

\noid If $M \subset \R^2$, we can view $G = G(u,M)$ as a graph in $\R^2$ for which \[r(G,\R^2) = r(G,M) + b_0(\partial M),\] and apply Lemma~\ref{L-gr-6}. \end{proof}

\begin{proposition}\label{P-euler-np}
Let $u$ be an eigenfunction of $-\Delta + V$ in $M$. For any component
$\Gamma$ of $\partial M$, we have
\begin{equation}\label{E-2.80}
\sum_{z\in \cS_{\mathrm{b}}(u) \cap \Gamma}\, \rho(z) \in 2\N.
\end{equation}
Furthermore,
\begin{equation}\label{E-2.8b}
b_0(\cZ(u) \cup \partial M) -b_0(\partial M) + \frac 12 \, \sum_{y \in \cS_{\mathrm{b}}(u)} \rho(y) \geq 1.
\end{equation}
\end{proposition}

\begin{proof} The first assertion is general and contained in Corollary~\ref{cor:nodinfo}. To prove the second assertion, we divide the components of $\partial M$ into two sets: the components $\Gamma'_i, 1 \le i \le p$, which meet $\cZ(u)$, and the components $\Gamma^{''}_j, 1 \le j \le q$, which do not meet $\cZ(u)$. Let $\Gamma(u) = \cup_{i=1}^p\Gamma'_i$. \index{1-Gamma@$\Gamma$!$\Gamma(u)$}
Clearly, we have the relation
\begin{equation*}
b_0(\cZ(u) \cup \partial M) -b_0(\partial M) = b_0(\cZ(u) \cup \Gamma(u)) -b_0(\Gamma(u)).
\end{equation*}
On the other-hand, according to \eqref{E-2.80}, for each $1 \le i \le p$, we have
\begin{equation*}
\sum_{z \in \cS_{\mathrm{b}}(u) \cap \Gamma'_i} \rho(z) \ge 2.
\end{equation*}
Relation \eqref{E-2.8b} follows from the fact that $b_0(\cZ(u) \cup \Gamma(u)) \ge 1$. \end{proof}

\section[Proof: Local Structure Theorem in the Interior]{Proof of the Local Structure Theorem at an Interior Point}\label{S-lsns}

In this section, we provide a proof of the local structure theorem for the
nodal set of an eigenfunction $u$ in the neighborhood of an interior
singular point  (alias critical zero) $x$, see Theorem~\ref{T-nodinfo},
Assertion~(i).  Being quantitative, in particular when applied to
eigenfunctions depending on a parameter, it is best suited to our
purposes. This proof is probably known although we did not find  it
explicitly  in the literature. In  \cite{Bess1980}, Besson refers to
\S\,128 of the book \cite{Vali1966}. In Subsection~\ref{SS-lsns-2},
we comment on Cheng's approach \cite{Chen1976}.

\subsection{Proof of Theorem~\ref{T-nodinfo},
Assertion~(i)}\label{SS-lsns-1}

Let $(M,g)$ be a $\Cty$ compact Riemannian surface, and let $V : M \to \R$ be a real valued $\Cty$ potential. Let $u \not = 0$ be a real valued function, satisfying
\begin{equation}\label{E-lsns-1}
(-\Delta_g + V)u = \lambda  u
\end{equation}
for some real number $\lambda$. Here $\Delta_g$ denotes the Laplace-Beltrami operator of $(M,g)$. Let $x$ be a given (interior) point of $M$ (in case $M$ has a boundary).\medskip

Choose some $r_0 > 0$ such that the exponential map $\exp_x : T_xM \to M$ is a diffeomorphism from the disk $D(2r_0)$, with center $0$ and  radius $2r_0$ in $T_xM$, onto the geodesic disk $D(x,2r_0)$, with center $x$ and radius $2r_0$ in $M$. Choose an orthonormal frame in $T_xM$, call $(\xi_1,\xi_2)$ the corresponding coordinates in $T_xM$, and $(r,\omega)$ the associated polar coordinates. In the normal coordinates $(\xi_1,\xi_2)$, the Riemannian metric is given by the $2\times2$ matrix $G = (g_{ij})$, where $g_{ij} = g(\frac{\partial}{\partial \xi_i},\frac{\partial}{\partial \xi_j})$; the Riemannian measure is given by $v_g\, d\xi_1 d\xi_2$, where $v_g = \sqrt{\det G}$. Write the matrix $G^{-1}$ as $G^{-1} = (g^{ij})$. Then (see for example \cite[\S~2.89bis, p. 87]{GaHuLa2004}),
\begin{equation}\label{E-lsns-2a}
\left\{
\begin{array}{l}
G(0,0) = \id, \text{\quad i.e.,\quad} g_{ij}(0,0) = \delta_{ij},
\quad 1 \le i,j \le 2, \\[5pt]
\frac{\partial G}{\partial \xi_k}(0,0) = 0, \text{\quad i.e.,\quad} \frac{\partial g_{ij}}{\partial \xi_k}(0,0) = 0, \quad 1 \le i,j,k \le 2.
\end{array}%
\right.
\end{equation}

It follows that
\begin{equation}\label{E-lsns-2b}
\left\{
\begin{array}{l}
v_g(0,0) = 1,\\[5pt]
G^{-1}(0,0) = \id, \text{\quad i.e.,\quad} g^{ij}(0,0) = \delta_{ij},
\quad 1 \le i,j \le 2, \\[5pt]
\frac{\partial v_g}{\partial \xi_k}(0,0) = 0, \quad 1 \le k \le 2, \\[5pt]
\frac{\partial G^{-1}}{\partial \xi_k}(0,0) = 0, \text{\quad i.e.,\quad} \frac{\partial g^{ij}}{\partial \xi_k}(0,0) = 0, \quad 1 \le i,j,k \le 2.
\end{array}%
\right.
\end{equation}

Given a function $u$ on $M$, let $f = u\circ \exp_x$. In the local coordinates $(\xi_1,\xi_2)$, the Laplace-Beltrami operator \index{Laplace-Beltrami operator} $\Delta_g$ is given (see \cite[\S G.{III}, p.~126]{BeGM1971}) by
\begin{equation}\label{E-lsns-2c}
\left\{
\begin{array}{ll}
\Delta_g f & = v_g^{-1}\, \sum_{1 \le i,j \le 2\,} \frac{\partial}{\partial \xi_i}\big( v_g\, g^{ij} \frac{\partial f}{\partial \xi_j}\big),\\[5pt]
& = \sum_{1 \le i,j \le 2\,} g^{ij}\, \frac{\partial^2 f}{\partial \xi_i \partial \xi_j} + \sum_{1 \le j \le 2\,}  b_j \, \frac{\partial f}{\partial \xi_j},\\[5pt]
& \text{\quad where~}b_ j = \sum_{1 \le i \le 2\,} v_g^{-1}\, \frac{\partial}{\partial \xi_i} (v_g g^{ij}), \quad 1 \le j \le 2.
\end{array}%
\right.
\end{equation}

Letting $\Delta_0 = \sum_{1 \le j \le 2\,} \frac{\partial^2}{\partial \xi_j^2}$ denote the Laplacian in the Euclidean space $(T_xM,g_x)$, and taking the relations \eqref{E-lsns-2a} and \eqref{E-lsns-2b} into account, we obtain the following expression for the Laplace-Beltrami operator,
\begin{equation}\label{E-lsns-2d}
\left\{
\begin{array}{l}
\Delta_g = \Delta_0 + \sum_{1 \le i,j \le 2\,} a_{ij}\, \frac{\partial^2}{\partial \xi_i \partial \xi_j} + \sum_{1 \le j \le 2\,} b_{j}\, \frac{\partial}{\partial \xi_j},\\[5pt]
\text{where~~} {\ord\big( a_{ij},(0,0) \big) \ge 2 \text{~and~} \ord\big( b_j,(0,0) \big)) \ge1.}
\end{array}
\right.
\end{equation}

If $u$ satisfies \eqref{E-lsns-1} and $u(x) = 0$, the unique continuation theorem \index{Unique continuation Theorem} \cite{Aron1957,DoFe1990a} implies that $f$ does not vanish at infinite order at $0$. \\  If $\ord(u,x)= \ord(f,0)=p$, Taylor's formula at $0$, gives
\index{2-R@$R$!$R_{p+1}$}
\begin{equation}\label{E-lsns-4}
\left\{
\begin{array}{l}
f(\xi_1,\xi_2) = \sum_{|\alpha| = p} \frac{1}{\alpha!} D^{\alpha}\!f(0,0) (\xi_1,\xi_2)^{\alpha} + R_{p+1}(\xi_1,\xi_2), \text{~where}\\[8pt]
R_{p+1}(\xi_1,\xi_2) = \sum_{|\alpha| = p+1} \frac{p+1}{\alpha!} (\xi_1,\xi_2)^{\alpha} {\displaystyle \int_0^1 (1-t)^p D^{\alpha}\!f(t \xi_1,t \xi_2)\,dt}.
\end{array}%
\right.
\end{equation}
Here,  as usual,
\begin{equation}\label{E-lsns-4a}
\left\{
\begin{array}{l}
\alpha = (\alpha_1,\alpha_2), \quad|\alpha| = \alpha_1 + \alpha_2, \quad (\xi_1,\xi_2)^{\alpha} = \xi_1^{\alpha_1}\, \xi_2^{\alpha_2}, \text{and}\\[5pt]
D^{\alpha}\!f = \dfrac{\partial^{|\alpha|}\!f}{\partial \xi_1^{\alpha_1}\partial \xi_2^{\alpha_2}}.
\end{array}%
\right.
\end{equation}

Using relations \eqref{E-lsns-1} and \eqref{E-lsns-2d}, and identifying the terms with lowest order, we find that the polynomial $P_p(\xi_1,\xi_2) := \sum_{|\alpha| = p} \frac{1}{\alpha!} D^{\alpha}\!f(0,0) (\xi_1,\xi_2)^{\alpha}$ is homogenous of degree $p$, and harmonic \index{Harmonic function}  with respect to $\Delta_0$.

\begin{remark}\label{R-lsns-2}
As a matter of fact, we may write a $2$-term Taylor formula for the function $f$,
\begin{equation*}
f(\xi_1,\xi_2) = P_p(\xi_1,\xi_2) + P_{p+1}(\xi_1,\xi_2) + R_{p+2}(\xi_1,\xi_2),
\end{equation*}
where $P_p$ and $P_{p+1}$ are homogeneous polynomials of degrees $p$ and $(p+1)$ respectively, and where the remainder term $R_{p+2}$ vanishes at order at least $(p+2)$. Then, we actually have that $\Delta_0 P_p = 0$ and $\Delta_0 P_{p+1} = 0$.
\end{remark}%

Writing the condition $\Delta_0 P_p = 0$ in polar coordinates $(r,\omega)$ in $T_xM$, we find that the polynomial $P_p$ has the form
\begin{equation}\label{E-lsns-4b}
P_p(r\, \cos \omega, r\, \sin \omega) = a \, r^p \, \sin(p\,\omega - \omega_0).
\end{equation}
for some $0 \not = a \in \R$ and some $\omega_0 \in [0,2\pi]$.\medskip

Multiplying the function $f$ by some constant, and rotating the coordinates $(\xi_1,\xi_2)$ in $\R^2$ if necessary, we can assume that $a=1$ and $\omega_0 = 0$. It follows that $f$ can be written as
\begin{equation}\label{E-lsns-6}
f(r\, \cos \omega, r\, \sin \omega) = r^p \, \sin(p\,\omega) + r^{p+1}\, T_{p+1}(r\, \cos\omega, r\, \sin\omega),
\end{equation}
where $T_{p+1}$ is given by
\begin{equation}\label{E-lsns-8}
\sum_{|\alpha| = p+1} \frac{p+1}{\alpha!}\, \,(\cos \omega,\sin \omega)^{\alpha} \, {\displaystyle \int_0^1 (1-t)^p\, D^{\alpha}\!f(tr\,\cos\omega,tr\,\sin\omega)\,dt}.
\end{equation}

Define
\begin{equation}\label{E-lsns-10}
W(r,\omega) := \sin(p\, \omega) + r\, T_{p+1}(r,\omega).
\end{equation}

The function $W(0,\omega)$ vanishes precisely for the values
\begin{equation}\label{E-lsns-12}
\omega_j := j\frac{\pi}{p}\,, \quad j \in \set{0,\ldots,2p-1}.
\end{equation}

Choose $\alpha_1 \in \big( 0,\frac{\pi}{8} \big)$ and define
$\alpha_p := \frac{\alpha_1}{p}$. We have the following relations

\begin{equation}\label{E-lsns-14}
\left\{
\begin{array}{l}
\sin\big(p\,(\omega_j \pm \alpha_p)\big) = \pm (-1)^j \, \sin(\alpha_1),\\[5pt]
|\sin(p\, \omega)| \ge \sin \alpha_1, \text{~for~} \omega \not \in \bigcup_{j=0}^{2p-1}\big( \omega_j-\alpha_p,\omega_j+\alpha_p\big).
\end{array}%
\right.
\end{equation}

Define
\begin{equation}\label{E-lsns-16}
r_1 := \min\set{r_0 ,\, \frac 12 \, \sin(\alpha_1)\, \|T_{p+1}\|^{-1}_{\infty,D(r_0)}\,, \frac 12 \cos(\alpha_1)\, \|\partial_{\omega}T_{p+1}\|^{-1}_{\infty,D(r_0)}},
\end{equation}
where $\|\cdot\|_{\infty,D(r_0)}$ denotes the $L^{\infty}$ norm of functions in the disk $D(r_0)$ of radius $r_0$ in $T_xM$.

\begin{proposition}\label{P-lsns-2}
For any $0 \le r \le r_1$,
\begin{enumerate}
  \item[(i)] the function $\omega \mapsto W(r,\omega)$ does not vanish in
\[
      [0,2\pi] \sm \bigcup_{j=0}^{2p-1}\big( \omega_j-\alpha_p,\omega_j+\alpha_p\big) = \bigcup_{j=0}^{2p-1}\big[ \omega_j+\alpha_p,\omega_{j+1}-\alpha_p\big];
\]
  \item[(ii)]for each $j \in \set{0,\ldots,2p-1}$, the function $\omega \mapsto W(r,\omega)$ has exactly one zero
  $\omegat_j(r) \in \big( \omega_j - \alpha_p ,\, \omega_j +\alpha_p \big)$;
  \item[(iii)] for each $j \in \set{0,\ldots,2p-1}$, the function $r \mapsto \omegat_j(r)$ is $C^{\infty}$ in $(0, r_1)$ and tends to $\omega_j$ as $r$ tends to zero;
  \item[(iv)] for each $j \in \set{0,\ldots,2p-1}$, the curve
  \[
  \big( 0 , r_1\big) \ni r \mapsto a_j(r) = \big( r\, \cos(\omegat_j(r) ,\, r \, \sin(\omegat_j(r)  \big)
  \]
  is smooth and has semi-tangent $\omega_j$ at the origin.
\end{enumerate}
\end{proposition}%

\begin{proof} To prove (i), we observe that in each interval $\set{r} \times \big[ \omega_j+\alpha_p,\omega_{j+1}-\alpha_p\big]$, $|W(r,\omega)| \ge \frac 12 \sin(\alpha_1)$. To prove (ii), we observe that the function $W(r,\omega)$ changes sign in $\set{r}\times \big( \omega_j-\alpha_p,\omega_j+\alpha_p\big)$ and that its partial derivative with respect to $\omega$ does not vanish.  Assertion~(iii) follows from the implicit function theorem. Assertion~(iv) follows from the previous ones. \end{proof}

\begin{remark}\label{R-lsns-4}
 Assume that  there exist two eigenfunctions $v_1$ and $v_2$ of \eqref{E-lsns-1} such that the functions $f_i = v_i \circ \exp_x$ satisfy the relations
\begin{equation}
\left\{
\begin{array}{l}
f_1(r\, \cos\omega, r\, \sin\omega) = r^p \, \sin(p\,\omega) + R_{1,p+1}(r\, \cos \omega, r\, \sin\omega),\\[5pt]
f_2(r\, \cos\omega, r\, \sin\omega) = r^p \, \cos(p\,\omega) + R_{2,p+1}(r\, \cos \omega, r\, \sin\omega).
\end{array}%
\right.
\end{equation}
Defining the family of functions $w_{\theta} = \cos\theta \, v_1 - \sin\theta \, v_2$, the associated family of functions $f_{\theta} = w_{\theta}\circ \exp_x$, satisfies
\begin{equation}
\left\{
\begin{array}{l}
f_{\theta}(r\, \cos\omega, r\, \sin\omega) = r^p \, \sin(p\, \omega - \theta) + R_{\theta,p+1}(r\, \cos\omega, r\, \sin\omega), \text{~with}\\[5pt]
R_{\theta,p+1} = \cos\theta \, R_{1,p+1} - \sin\theta \, R_{2,p+1}.
\end{array}%
\right.
\end{equation}
Then, Proposition~\ref{P-lsns-2} remains valid for the family $f_{\theta}$, uniformly with respect to the variable $\theta \in [0,2\pi]$, with $\omega_j$ replaced by
$\omega_{j,\theta} +\frac{\theta}{p}$, and the corresponding functions $r \mapsto \omegat_j(r,\theta)$ and $r \mapsto a_j(r,\theta)$ are smooth in $(r,\theta)$.
\end{remark}%

\begin{remark}\label{R-lsns-6}
Proposition~\ref{P-lsns-2} tells us that, in a neighborhood of the critical zero $x$ of $u$, \index{Critical zero} the nodal set $\cZ(u)$ consists of $p$ smooth semi-arcs emanating from $x$ tangentially to the rays $\omega_j$.
\end{remark}

\subsection{About Cheng's proof of the local structure theorem} \label{SS-lsns-2}

Recall that the main purpose of Cheng's paper \cite{Chen1976} is to give general upper bounds on the multiplicities of the eigenvalues of the Laplace-Beltrami operator $\Delta_g$ on a closed (i.e., compact without boundary) Riemannian orientable \emph{surface} $(M,g)$.  Both the surface $M$ and the Riemannian metric $g$ are assumed to be $C^{\infty}$.\smallskip

The starting point is to describe the local and global structure of the nodal set $\cZ(u)$ of an eigenfunction $u$, \cite[Theorem~2.5]{Chen1976}.  In particular, Cheng states the following property in any dimension, \cite[Lemma~2.4]{Chen1976}.

\begin{property}[S.Y.~Cheng]\label{P-cheng}
Near a critical zero $x_0$ of a nontrivial eigenfunction $u$ of the closed surface $(M,g)$, the nodal set $\cZ(u)$ is $C^1$ diffeomorphic to the nodal set $\cZ(p_N)$ of some harmonic homogeneous polynomial $p_N$ of degree $N$, the vanishing order $\ord(u,x_0)$.
\end{property}%

As remarked in \cite{Paga1990}, Bemerkung~2.9.3 p.~34, the proof of Cheng's Lemma~2.4 is flawed\footnote{We thank  V.~Bobkov for pointing out this reference (Private communications, Nov. 2024 and Jan. 2025).}. More precisely Cheng's claim (page 49, line 2) ``\emph{Thus, $X(x,a) = O(|x|^{1+\varepsilon})$. This shows that $X(x,a)$ is a $C^1$ vector field.}'' is not correct as shown by the example $X(x,a) = \big( x_1^2 \sin(\frac{1}{x_1}), 0, \ldots, 0\big)$.

\begin{proof}[Proof of Property~\ref{P-cheng}]  According to unique continuation, the nontrivial eigenfunction $u$ cannot vanish at infinite order at $x_0$ so that $\ord(u,x_0) = N$ for some $N \in \N$. Under Cheng's smoothness assumptions, the eigenfunction $u$ is $C^{\infty}$. Let  $\xi$ denote exponential coordinates centered at $x_0$. Write Taylor's formula at order $N$ with integral remainder term
\begin{equation}\label{E-T1}
u(\xi) = \sum_{|\alpha| = N} \partial^{\alpha}u(0) \frac{\xi^{\alpha}}{\alpha!} + \sum_{|\beta| = N+1} R_{\beta}(\xi) \xi^{\beta},
\end{equation}
where the functions $R_{\beta}$ are given by
\begin{equation}\label{E-T2}
R_{\beta}(\xi) = \frac{|\beta|}{\beta!} \int_0^1 (1-t)^{|\beta|-1} \partial^{\beta}u(t\xi)\, dt.
\end{equation}

It is easy to check that the leading term in Taylor's formula is a nontrivial homogeneous polynomial $p_N(\xi)$ of degree $N$ which satisfies
$\Delta_0 p_N = 0$ where $\Delta_{0}$ is the  ordinary Laplacian in the $\xi$-variable. Since $\dim M = 2$, the harmonic polynomial $p_N$ can be written in polar coordinates as
\[
p_N([\rho,\theta]) = \rho^{N} \big( a\, \cos(N\theta) + b\, \sin(N\theta) \big)
\]
with $a^2+b^2 \neq 0$, and we have
\[|\nabla p_N([\rho,\theta])| = N \sqrt{a^2+b^2}\, \rho^{N-1}.\]

Then, in the coordinates $\xi$, the function $u$ satisfies the assumptions of the following lemma whose proof can be found in K.~Pagani's thesis \cite{Paga1990} (Lemma~2.9.1 for the first statement and  Korollar~2.9.4 for the second statement). The first statement and the method of proof are actually due to T-C.~Kuo \cite{Kuo1969}.

\begin{lemma}[K.~Pagani]\label{L-P1}
Let $U$ be an open neighborhood of  $0 \in \R^n$. Let $u, p \in C^{2}(U)$ satisfy the following properties for some $N \ge 1$, some positive constant $c$ and all $x \in U$.
\begin{equation}\label{E-P1}
\left\{
\begin{array}{lll}
u(\xi) & = & p(\xi) + o(|\xi|^{N})\\[5pt]
\nabla u(\xi) & = & \nabla p(\xi) + o(|\xi|^{N-1})\\[5pt]
p(0) & = & 0\\[5pt]
\left| \nabla p(\xi) \right| & \ge & c\,  |\xi|^{N-1}.\\[5pt]
\end{array}%
\right.
\end{equation}
Then, there exists a local \emph{homeomorphism} $\Phi$ in a neighborhood of $0$ in $U$ such that $\Phi(0)=0$ and $u = p \circ \Phi$.\smallskip

Assume furthermore that
\begin{equation}\label{E-P2}
\left\{
\begin{array}{lll}
\mathrm{Hess\,}u(\xi) & = & \mathrm{Hess\,}p (\xi) + o(|\xi|^{N-2})\\[5pt]
\mathrm{Hess\,}p(\xi) & = & O(|\xi|^{N-2}).
\end{array}%
\right.
\end{equation}
Then, $\Phi$ is a local $C^1$ \emph{diffeomorphism}.
\end{lemma}%

The proof of Property~\ref{P-cheng} is complete.
\end{proof}

\begin{remark} In his paper,  Cheng makes use of Bers' theorem, \cite[Theorem~2.1]{Chen1976}. The strength of Bers' theorem is that it only assumes that the coefficients are H\"{o}lder continuous. In Cheng's framework, the coefficients are $C^{\infty}$ and Bers's theorem actually boils down to writing Taylor's formula.
\end{remark}%

As pointed out by Y.~Colin de Verdi\`{e}re, see \cite[Appendix~E]{BeMe1982}, the assumption $\left| \nabla p(\xi) \right| \ge c\,  |\xi|^{N-1}$ in Equation~\eqref{E-P1} is necessary. Indeed, in $\R^3$, let
\begin{equation}\label{E-BM1}
u(\xi,\eta,\zeta) = \xi \eta - \zeta^4 \text{~and~} p(\xi,\eta,\zeta) = \xi \eta .
\end{equation}
Then $\R^3 \! \setminus \! u^{-1}(0)$ has three connected components, whereas $\R^3 \! \setminus \! p^{-1}(0)$ has four connected components. Choosing $N=2$, we have $|\nabla p|(\xi,\eta,\zeta) = \sqrt{\xi^2+\eta^2}$ which cannot be bounded from below by some multiple of $\sqrt{\xi^2+\eta^2+\zeta^2}$. This shows that Lemma~\ref{L-P1} does not hold without the last assumption in Equation~\eqref{E-P1}. \smallskip

On the other hand, defining  the operator
\begin{equation*}
L := \partial_{\xi}^{2} + \partial_{\eta}^{2} + \partial_{\zeta}^{2} + 12\, \zeta^2 \, \partial_{\xi \eta}^{2},
\end{equation*}
we may observe that it is elliptic in a neighborhood of $0$ in $\R^3$, that $Lu=0$, and that $0$ is an isolated singularity of $u$. \medskip

In \cite{Chen1976}, when working in dimension $N$ in the proof of his Theorem~2.2, Cheng misses checking  that the assumption  $\left| \nabla p(\xi) \right| \ge c\,  |\xi|^{N-1}$ is satisfied.
The previous example shows that this assumption might actually not be satisfied in dimension larger than or equal to $3$, and that local arguments are not sufficient to extend Property~\ref{P-cheng} to higher dimensions.

\section[Proof: Local structure theorem at the Boundary]{Proof of the Local Structure Theorem at a  Boundary Point}\label{S-lsbs}

\subsection{Preamble}\label{SS-lsbs1}

Let $\Omega \subset \R^2$ be a bounded domain with $C^{\infty}$ boundary $\Gamma := \partial \Omega$.
\index{1-Gamma@$\Gamma$!$\Gamma  = \partial \Omega$} We consider the eigenvalue problem
\begin{equation}\label{E-lsbs-02}
\left\{
\begin{array}{rll}
(-\Delta + V) u & = \lambda \, u &\text{~in\quad} \Omega,\\[5pt]
B(u) & = 0  &\text{~on\quad} \Gamma,
\end{array}%
\right.
\end{equation}
where $V$ is a real valued function in $C^{\infty}(\wb{\Omega})$, $\Delta$ the usual Laplacian of $\R^2$,  and $B(u)$ one of the following boundary conditions on $\Gamma$
\begin{equation}\label{E-lsbs-04}
\left\{
\begin{array}{rl}
B(u) := u|_{\Gamma} &\text{\quad the Dirichlet boundary condition},\\[5pt]
B(u) := \partial_{\nu_i}u &\text{\quad the Neumann boundary condition},\\[5pt]
B(u) := \partial_{\nu_i}u - h\, u|_{\Gamma} &\text{\quad the $h$-Robin boundary condition},
\end{array}%
\right.
\end{equation}
where $\partial_{\nu_i}u$  is the derivative of $u$ with respect to the unit normal along $\Gamma$, pointing inwards, and $h$ a $C^{\infty}$ function on $\Gamma$. \medskip

The purpose of this section is to describe the local structure of the nodal set $\cZ(u)$ of an eigenfunction  $u$ of \eqref{E-lsbs-02}--\eqref{E-lsbs-04} near a boundary singular point of $u$, i.e., a point $y \in \Gamma$ at which the nodal set hits the boundary. For the sake of simplicity, throughout this section, we assume that
\begin{assumption}\label{A-lsbs-2}
$\Omega$ is simply connected.
\end{assumption}%

We explain how to deal with non simply connected domains in Remark~\ref{R-lsbs-4} and Subsection~\ref{SS-lsbs9}.
\subsection{Notation}\label{SS-lsbs2}

We use the following notation.\smallskip

Without loss of generality, we assume that the length of $\Gamma$ is $2\pi$.  The orientation of $\R^2$ induces a natural orientation of $\Gamma$, and  we fix an arc length parametrization of $\Gamma$ compatible with this orientation:
\begin{equation}\label{E-lsbs2-1}
\left\{
\begin{array}{l}
\gamma:[0,2\pi] \to \R^2 \text{~with~} \gamma([0,2\pi]) = \Gamma\\[5pt]
\tauv_{\gamma(t)} = \dot{\gamma}(t) \text{~is the unit tangent vector}\\[5pt]
\nuv_{\gamma(t)} \text{~is the unit normal vector pointing inwards}\\[5pt]
\set{\tauv,\nuv} \text{~is a direct frame}.
\end{array}%
\right.
\end{equation}

Introduce the notation
\begin{equation}\label{E-lsbs2-2}
\left\{
\begin{array}{l}
\bH := \set{ (\xi_1,\xi_2) \in \R^2 \mid \xi_2 > 0}\\[5pt]
\bbH := \set{ (\xi_1,\xi_2) \in \R^2 \mid \xi_2 \ge 0}.
\end{array}%
\right.
\end{equation}

\begin{equation}\label{E-lsbs2-4}
\left\{
\begin{array}{l}
\D := \set{ (\xi_1,\xi_2) \in \R^2 \mid \xi_1^2+\xi_2^2 < 1}\\[5pt]
\wb{\D} := \set{ (\xi_1,\xi_2) \in \R^2 \mid \xi_1^2+\xi_2^2 \le 1}.
\end{array}%
\right.
\end{equation}%

Given $y =(y_1,y_2) \in \R^2$ and $r>0$, define the disks
\begin{equation}\label{E-lsbs2-6}
\left\{
\begin{array}{l}
D(y,r) := \set{ (z_1,z_2) \in \R^2 \mid  (y_1-z_1)^2 + (y_2-z_2)^2 < r^2}\\[5pt]
\wb{D}(y,r) := \set{ (z_1,z_2) \in \R^2 \mid  (y_1-z_1)^2 + (y_2-z_2)^2 \le r^2}.
\end{array}%
\right.
\end{equation}

Similarly, given $\eta \in \partial \bH$, define the half-disks $D_{+}(\eta,r)$ and $\wb{D}_{+}(\eta,r)$
\begin{equation}\label{E-lsbs2-6h}
\left\{
\begin{array}{l}
D_{+}(\eta,r) := \set{ (\xi_1,\xi_2) \in \bH \mid  (\eta_1-\xi_1)^2 + (\eta_2-\xi_2)^2 < r^2}\\[5pt]
\wb{D}_{+}(\eta,r) := \set{ (\xi_1,\xi_2) \in \bH \mid  (\eta_1-\xi_1)^2 + (\eta_2-\xi_2)^2 \le r^2}.
\end{array}%
\right.
\end{equation}

\subsection{Preparation}\label{SS-lsbs3}~

\subsubsection{Regularity of conformal mappings at the boundary}\label{SSS-lsbs31}

\begin{lemma}\label{L-lsbs3-2}
Let $\Omega$ be a $C^{\infty}$ simply connected domain of $\R^2$ with boundary $\Gamma$. Let $y_0, z_0$ and $y_*$ be distinct points in $\Gamma$. Then, there exists a conformal diffeomorphism $F: \Omega \to \bH$ which extends as a $C^{\infty}$ map up to the boundary, such that $F|_{\Gamma}$ sends $\Gamma \sm \set{y_*}$ diffeomorphically onto $\partial \bH$, $y_0$ to $(0,0)$ and $z_0$ to some $(\zeta_0,0)$ at finite distance in $\bdH$.
\end{lemma}%

\begin{proof}  The existence of  a conformal diffeomorphism $F_1: \Omega \to \D$ is given by the Riemann mapping theorem. The fact that this diffeomorphism is $C^{\infty}$ up to the boundary and sends $\Gamma$ diffeomorphically onto $\partial \D$ is explained in \cite{BeKr1987}.\smallskip

Since $F_1(y_*) \in \partial \D$, the map $F_2 = \overline{F_1(y_*)}\, F_1$ is a conformal map from $\Omega$ onto $\D$ which extends smoothly to the boundary, sends $\Gamma$ diffeomorphically onto $\partial \D$, and $y_*$ to the point $1$.
The map $F_3 : w \mapsto i \frac{1+w}{1-w}$ maps $\D$ onto $\bH$, is smooth up to the boundary, maps $\partial \D \sm \set{1}$ onto $\partial \bH$, and $1$ to infinity. Taking for $F_4$ a suitable horizontal translation in $\bH$, the map $F:=F_4 \circ F_3 \circ F_2$ has the required properties.
\end{proof}

\subsubsection{Eigenfunctions}\label{SSS-lsbs32}

The conformal diffeomorphism  $E:=F^{-1} : \bH \to \Omega$ is $C^{\infty}$ up to the boundary, sends $\partial \bH$ onto the boundary $\Gamma$ minus the point $y_*\,$, and the point $0$ to $y_0$.\smallskip

If $\xi = (\xi_1,\xi_2)$ and $E(\xi) = \big( E_1(\xi_1,\xi_2), E_2(\xi_1,\xi_2) \big)$, the Jacobian of $E$ is given by
\begin{equation*}
\mathrm{Jac}(E)(\xi) = \begin{pmatrix}
                        \partial_{\xi_1}E_1(\xi) & \partial_{\xi_2}E_1(\xi) \\
                        \partial_{\xi_1}E_2(\xi) & \partial_{\xi_2}E_2(\xi) \\
                      \end{pmatrix},
\end{equation*}
where $\partial_{\xi_i}$ stands for $\frac{\partial}{\partial \xi_i}$.\smallskip

Since $E$ is conformal, we have $|\nabla E_1| = |\nabla E_2|$, $\langle \nabla E_1, \nabla E_2 \rangle = 0$. The determinant of the Jacobian of $E$ is given by
\begin{equation}\label{E-lsbs3-2}
J_E  =\det(\mathrm{Jac}(E)) = |\nabla E_1|^2 = |\nabla E_2|^2.
\end{equation}

Let $u$ be a $C^{\infty}$ function in $\Omega$. The Laplacian $\Delta_{\xi}$ of the function $u \circ E$ is given by the following formula in the variables $(\xi_1,\xi_2)$ of $\bH$
\begin{equation}\label{E-lsbs3-4}
\Delta_{\xi}(u \circ E) = J_E \;\big( (\Delta_{x}u)\circ E\big),
\end{equation}
where $\Delta_x$ is the Laplacian of $u$ in the variables $(x_1,x_2)$ of $\Omega$.\smallskip

Let $u$ be a nontrivial eigenfunction of \eqref{E-lsbs-02}--\eqref{E-lsbs-04}, and let $y_0 \in \Gamma$. Choose $E$ so that $E(0,0) = y_0$.  We now work with  the function
\begin{equation}\label{E-lsbs3-6}
v = u\circ E .
\end{equation}

Define the functions
\begin{equation}\label{E-lsbs3-7}
\left\{
\begin{array}{l}
V_E := J_E \; \big( V\circ E\big), \text{~and~}\\[5pt]
h_E := \sqrt{J_E} \; \big(h\circ E\big).
\end{array}
\right.
\end{equation}

The function $v$ satisfies
\begin{equation}\label{E-lsbs3-8}
 \left\{
\begin{array}{rll}
\big( - \Delta_{\xi}v + V_E \big) v & = \lambda \, J_E \, v  &\text{~in~} \bH \\[5pt]
B_E\,v & = 0 &\text{~on~} \partial \bH,
\end{array}%
\right.
\end{equation}
where the boundary condition $B_E\,v$ is given by
\begin{equation}\label{E-lsbs3-10}
\left\{
\begin{array}{l}
B_E\,v = v|_{\partial \bH}  \text{~in the Dirichlet case}\\[5pt]
B_E\,v = \partial_{\xi_2} v|_{\partial \bH}  \text{~in the Neumann case}\\[5pt]
B_E\,v = \big( \partial_{\xi_2} v - h_E\, v\big)|_{\partial \bH}  \text{~in the Robin case}.
\end{array}%
\right.
\end{equation}

To determine the local properties of $u$ near $y_0 \in \bdH$, it is sufficient to determine the local properties of $v$ in $\wb{D}_{+}(0,r_0) \cap \bH$, for some $r_0 > 0$.\medskip

\subsection{The unique continuation property at the boundary}\label{SS-lsbs4}
\index{Unique continuation property (boundary)}
\begin{property}\label{P-lsbs4-2}
The eigenfunctions of \eqref{E-lsbs-02}--\eqref{E-lsbs-04} are in $C^{\infty}(\wb{\Omega})$.
\end{property}%

This property follows from elliptic regularity, see \cite{GiTr1977},  Sections~6.4 and 6.7, or  \cite{Mikh1978}, Chap.~4.2, page 217.

\begin{proposition}\label{P-lsbs4-4}
A nontrivial eigenfunction $u$ of \eqref{E-lsbs-02}--\eqref{E-lsbs-04} cannot vanish at infinite order at any point $y \in \Gamma$.
\end{proposition}%

\begin{proof} We follow the proof of Lemma~2.1 in Melas' paper \cite{Mela1992}. He only considers the Dirichlet boundary condition, and deals with convex domains. His proof can be adapted to the Neumann boundary condition, and the part of the proof we are interested in does actually not use the convexity assumption. For convenience, we give a complete proof here. We use the framework described in Paragraph~\ref{SSS-lsbs32}.\smallskip

Let $u$ be a nontrivial eigenfunction of \eqref{E-lsbs-02}. Then $u \in C^{\infty}(\wb{\Omega})$, and the function $v = u\circ E$ is in $C^{\infty}(\bbH)$. Furthermore, $u$ vanishes at infinite order at $y_0$ if and only if $v$ vanishes at infinite order at $0$. \smallskip

We now restrict $v$ to some neighborhood $\wb{D}_{+}(0,r_0) \cap \bbH$ of $0 \in \bbH$, and correspondingly $u$ to the image $\wb{D}_E(y_0,r_0) := E(\wb{D}_{+}(0,r_0) \cap \bbH)$. \smallskip

Since $E$ is conformal, the function $v$ satisfies the equation
\begin{equation*} 
\Delta_{\xi}v = J_E \, \big((\Delta_xu)\circ E\big) = \big( V_E - J_E\, \lambda\big) v,
\end{equation*}
where the function $V_E$ is $C^{\infty}$ and bounded in $\wb{D}_{+}(0,r_0) \cap \bH$.  Furthermore, $v$ satisfies the Dirichlet (resp. the Neumann, or a Robin) boundary condition on $\wb{D}_{+}(0,r_0) \cap \partial \bH$ when the function $u$ satisfies the Dirichlet  boundary condition on $\wb{D}_E(y_0,r_0) \cap \Gamma$ (resp. the Neumann, or a Robin boundary  condition).\medskip

We now work with the function $v$ restricted to $\wb{D}_{+}(0,r_0) \cap \bbH$, and we consider three cases separately: $v$ satisfies the Dirichlet boundary condition, $v$ satisfies the Neumann boundary condition, or  $v$ satisfies a Robin boundary condition on $\wb{D}_{+}(0,r_0) \cap \partial \bH$. \smallskip

\noid \emph{Dirichlet boundary condition.} We define the function $w$ on $D_{+}(0,r_0)$ by
\begin{equation}\label{E-lsbs4-12}
w(\xi_1,\xi_2) := \left\{
\begin{array}{rl}
v(\xi_1,\xi_2)  &\text{~if\quad} \xi_2 > 0\\[5pt]
- v(\xi_1,-\xi_2) &\text{~if\quad} \xi_2 < 0\\[5pt]
0 &\text{~if\quad} \xi_2 = 0.
\end{array}%
\right.
\end{equation}

Since $v$ extends to the boundary and vanishes on the boundary, the function $w$ is well defined and continuous in $D_{+}(0,r_0)$, with $w(\xi_1,0) = 0$ for all $\xi_1 \in (-r_0,r_0)$.\smallskip

The function $w$ has first partial derivatives given as follows:
\begin{enumerate}[a)]
\item%
\begin{equation*}
\partial_{\xi_1}  w(\xi_1,\xi_2) = \left\{
\begin{array}{rl}
\partial_{\xi_1} v(\xi_1,\xi_2)  &\text{~if\quad} \xi_2 > 0\\[5pt]
- \partial_{\xi_1}v(\xi_1,-\xi_2) &\text{~if\quad} \xi_2 < 0\\[5pt]
0  &\text{~if\quad} \xi_2 = 0,
\end{array}%
\right.
\end{equation*}
where the third line follows from the fact that $v(\xi_1,0) \equiv 0$ on $(-r_0,r_0)$.
\item%
\begin{equation*}
\partial_{\xi_2}  w(\xi_1,\xi_2) = \left\{
\begin{array}{rl}
\partial_{\xi_2} v(\xi_1,\xi_2)  &\text{~if\quad } \xi_2 > 0\\[5pt]
\partial_{\xi_2}v(\xi_1,-\xi_2) &\text{~if\quad } \xi_2 < 0\\[5pt]
\partial_{\xi_2} v(\xi_1,0)  &\text{~if\quad} \xi_2 = 0,
\end{array}%
\right.
\end{equation*}
where the third line is computed using the definition of the derivative of $w$ at the point $(\xi_1,0)$.
\end{enumerate}%

Since $v$ is $C^{\infty}$ up to the boundary, the first partial derivatives of $w$ are continuous.

The second derivatives of $w$ are given as follows:
\begin{enumerate}[a)]
\item%
\begin{equation*}
\partial^2_{\xi_1^2}  w(\xi_1,\xi_2) = \left\{
\begin{array}{rl}
\partial^2_{\xi_1^2} v(\xi_1,\xi_2)  &\text{~if\quad} \xi_2 > 0\\[5pt]
-\partial^2_{\xi_1^2}v(\xi_1,-\xi_2) &\text{~if\quad} \xi_2 < 0\\[5pt]
0 &\text{~if\quad} \xi_2 = 0,
\end{array}%
\right.
\end{equation*}
where the third line follows from the fact that $\partial_{\xi_1}w(\xi_1,0) \equiv 0$ on $(-r_0,r_0)$.  Since $v$ is $C^{\infty}$ up to the boundary, this second derivative is continuous.
\item%
\begin{equation*}
\partial^2_{\xi_2\xi_1}  w(\xi_1,\xi_2) = \left\{
\begin{array}{rl}
\partial^2_{\xi_2\xi_1} v(\xi_1,\xi_2)  &\text{~if\quad} \xi_2 > 0\\[5pt]
\partial^2_{\xi_2\xi_1}v(\xi_1,-\xi_2) &\text{~if\quad} \xi_2 < 0\\[5pt]
\partial^2_{\xi_2\xi_1} v(\xi_1,0)  &\text{~if\quad} \xi_2 = 0,
\end{array}%
\right.
\end{equation*}
where the third line follows from the definition of the derivative of $\partial_{\xi_1}w$ with respect to $\xi_2$ at some point $(\xi_1,0)$.  Since $v$ is $C^{\infty}$ up to the boundary, this second derivative is continuous.
\item From the formula for $\partial_{\xi_2}w$ we immediately obtain
\begin{equation*}
\partial^2_{\xi_1\xi_2}  w(\xi_1,\xi_2) = \left\{
\begin{array}{rl}
\partial^2_{\xi_1\xi_2} v(\xi_1,\xi_2)  &\text{~if\quad} \xi_2 > 0\\[5pt]
\partial^2_{\xi_1\xi_2}v(\xi_1,-\xi_2) &\text{~if\quad} \xi_2 < 0\\[5pt]
\partial^2_{\xi_1\xi_2} v(\xi_1,0)  &\text{~if\quad} \xi_2 = 0.
\end{array}%
\right.
\end{equation*}
Since $v$ is $C^{\infty}$ up to the boundary, this second derivative is continuous, and we have $\partial^2_{\xi_1\xi_2}w = \partial^2_{\xi_2\xi_1}w$.
\item%
\begin{equation*}
\partial^2_{\xi_2^2}  w(\xi_1,\xi_2) = \left\{
\begin{array}{rl}
\partial^2_{\xi_2^2} v(\xi_1,\xi_2)  &\text{~if\quad} \xi_2 > 0\\[5pt]
-\partial^2_{\xi_2^2}v(\xi_1,-\xi_2) &\text{~if\quad} \xi_2 < 0\\[5pt]
\pm \partial^2_{\xi_2^2} v(\xi_1,0)  &\text{~if\quad} \xi_2 = 0,
\end{array}%
\right.
\end{equation*}
where the third line follows from a direct computation of the right/left partial derivative of $\partial_{\xi_2}w$ with respect to $\xi_2$ at a point $(\xi_1,0)$. The sign indicates the value for the right/left derivatives.
\end{enumerate}%

It follows from \eqref{E-lsbs3-8} that $\Delta_{\xi}v(\xi_1,0) \equiv 0$ because $v$ satisfies the Dirichlet boundary condition. Since we already have $\partial^2_{\xi_1^2}v(\xi_1,0) \equiv 0$, we conclude that $\partial^2_{\xi_2^2}v(\xi_1,0) \equiv 0$ as well, and hence that the second derivative $\partial^2_{\xi_2^2}w$ is well defined and continuous.\medskip

We have just shown that the function $w$ is in $C^2(D_{+}(0,r_0))$.
Furthermore, we have
\begin{equation}\label{E-lsbs4-14}
\Delta_{\xi}w(\xi_1,\xi_2) = \left\{
\begin{array}{rl}
\Delta_{\xi}v(\xi_1,\xi_2) &\text{~if\quad} \xi_2 > 0\\[5pt]
-\Delta_{\xi}v(\xi_1,-\xi_2) &\text{~if\quad} \xi_2 < 0\\[5pt]
0 &\text{~if\quad} \xi_2 = 0,
\end{array}
\right.
\end{equation}
and hence, by \eqref{E-lsbs3-8},
\begin{equation}\label{E-lsbs4-16}
|\Delta_{\xi} w| \le C \,  |w| \mbox{ in } D_{+}(0,r_0).
\end{equation}

 Recall the \emph{strong unique continuation property} \index{Strong unique continuation property} for the operator $P=-\Delta+V$ in the open domain $\Omega \subset \R^2$ where $V \in L^{\infty}(\Omega)$.

\begin{theorem}\label{T-SUCP}
Assume that $u \in H^2(\Omega)$ and that $|\Delta u| \le C(|\nabla u| + |u|)$ almost everywhere in $\Omega$. If $\lim_{r\to 0} r^{-N}\int_{B(x_0,r)}|u|^2\, dx = 0$ for  some $x_0  \in \Omega$ and for all $N \ge 0$ (i.e., $u$ \emph{vanishes to infinite order at} $x_0$), then $u$ vanishes identically in $\Omega$.
\end{theorem}%
\index{Aronszajn Theorem}
For this theorem, we refer to \cite{Salo2014} (Theorem~1.2 and Remark~4 on page~4), or Aronszajn \cite{Aron1957}. \smallskip

Since $w \not \equiv 0$, in view of \eqref{E-lsbs4-16}, the strong unique continuation property implies that $w$ does not vanish at infinite order at $0$, so that $v$ does not vanish at $0$ at infinite order either. This proves that the eigenfunction $u$ cannot vanish at infinite order at the boundary point $y$.\medskip

\noid \emph{Neumann boundary condition.} We define the function $w$ on $D_{+}(0,r_0)$ by
\begin{equation}\label{E-lsbs4-18}
w(\xi_1,\xi_2) := \left\{
\begin{array}{rl}
v(\xi_1,\xi_2)  &\text{~if\quad} \xi_2 > 0\\[5pt]
v(\xi_1,-\xi_2) &\text{~if\quad} \xi_2 < 0\\[5pt]
v(\xi_1,0) &\text{~if\quad} \xi_2=0.
\end{array}%
\right.
\end{equation}

Since $v$ is $C^{\infty}$ up to the boundary, the function $w$ is well defined and continuous in $D_{+}(0,r_0)$, with $w(\xi_1,0) = v(\xi_1,0)$ for all $\xi_1 \in (-r_0,r_0)$.\smallskip

The function $w$ has first partial derivatives given as follows:
\begin{enumerate}[a)]
\item
\begin{equation*}
\partial_{\xi_1}  w(\xi_1,\xi_2) = \left\{
\begin{array}{rl}
\partial_{\xi_1} v(\xi_1,\xi_2)  &\text{~if\quad} \xi_2 > 0\\[5pt]
\partial_{\xi_1}v(\xi_1,-\xi_2) &\text{~if\quad} \xi_2 < 0\\[5pt]
\partial_{\xi_1} v(\xi_1,0)   &\text{~if\quad} \xi_2 = 0.
\end{array}%
\right.
\end{equation*}
\item
\begin{equation*}
\partial_{\xi_2}  w(\xi_1,\xi_2) = \left\{
\begin{array}{rl}
\partial_{\xi_2} v(\xi_1,\xi_2)  &\text{~if\quad} \xi_2 > 0\\[5pt]
-\partial_{\xi_2}v(\xi_1,-\xi_2) &\text{~if\quad} \xi_2 < 0\\[5pt]
0  &\text{~if\quad} \xi_2 = 0,
\end{array}%
\right.
\end{equation*}
where the third line is computed using the definition of the derivative, and the fact that the function $v$ satisfies the Neumann boundary condition $\partial_{\xi_2}v(\xi_1,0) \equiv 0$ in $(-r_0,r_0)$.
\end{enumerate}%
Since $v$ is $C^{\infty}$ up to the boundary, the first partial derivatives of $w$ are continuous.\medskip

The second derivatives of $w$ are given as follows:
\begin{enumerate}[a)]
\item
\begin{equation*}
\partial^2_{\xi_1^2}  w(\xi_1,\xi_2) = \left\{
\begin{array}{rl}
\partial^2_{\xi_1^2} v(\xi_1,\xi_2)  &\text{~if\quad} \xi_2 > 0\\[5pt]
\partial^2_{\xi_1^2}v(\xi_1,-\xi_2) &\text{~if\quad} \xi_2 < 0\\[5pt]
\partial^2_{\xi_1^2} v(\xi_1,0) &\text{~if\quad} \xi_2 = 0,
\end{array}%
\right.
\end{equation*}
where the third line follows from the fact that $\partial_{\xi_1}v(\xi_1,0)$ is $C^{\infty}$. Since $v$ is $C^{\infty}$ up to the boundary, this second derivative is continuous.
\item
\begin{equation*}
\partial^2_{\xi_2\xi_1}  w(\xi_1,\xi_2) = \left\{
\begin{array}{rl}
\partial^2_{\xi_2\xi_1} v(\xi_1,\xi_2)  &\text{~if\quad} \xi_2 > 0\\[5pt]
-\partial^2_{\xi_2\xi_1}v(\xi_1,-\xi_2) &\text{~if\quad} \xi_2 < 0\\[5pt]
\pm \partial^2_{\xi_2\xi_1} v(\xi_1,0)  &\text{~if\quad} \xi_2 = 0,
\end{array}%
\right.
\end{equation*}
where the third line follows from the definition of the right/left derivatives of $\partial_{\xi_1}w$ with respect to $\xi_2$ at the point $(\xi_1,0)$.   The term in the third line is actually zero due to the Neumann boundary conditions and the fact that $\partial^2_{\xi_2\xi_1} v(\xi_1,0) = \partial^2_{\xi_1\xi_2} v(\xi_1,0)$. Since $v$ is $C^{\infty}$ up to the boundary,  this second derivative is continuous.
\item From the formula for $\partial_{\xi_2}w$ we immediately obtain
\begin{equation*}
\partial^2_{\xi_1\xi_2}  w(\xi_1,\xi_2) = \left\{
\begin{array}{rl}
\partial^2_{\xi_1\xi_2} v(\xi_1,\xi_2)  &\text{~if\quad} \xi_2 > 0\\[5pt]
-\partial^2_{\xi_1\xi_2}v(\xi_1,-\xi_2) &\text{~if\quad} \xi_2 < 0\\[5pt]
0  &\text{~if\quad} \xi_2 = 0.
\end{array}%
\right.
\end{equation*}
Since $v$ is $C^{\infty}$ up to the boundary, this second derivative is continuous, and we have $\partial^2_{\xi_1\xi_2}w = \partial^2_{\xi_2\xi_1}w$.
\item
\begin{equation*}
\partial^2_{\xi_2^2}  w(\xi_1,\xi_2) = \left\{
\begin{array}{rl}
\partial^2_{\xi_2^2} v(\xi_1,\xi_2)  &\text{~if\quad} \xi_2 > 0\\[5pt]
\partial^2_{\xi_2^2}v(\xi_1,-\xi_2) &\text{~if\quad} \xi_2 < 0\\[5pt]
\partial^2_{\xi_2^2} v(\xi_1,0)  &\text{~if\quad} \xi_2 = 0.
\end{array}%
\right.
\end{equation*}
\end{enumerate}%
We have shown that the function $w$ is continuous, with continuous  first and second derivatives. Furthermore, we have
\begin{equation}\label{E-lsbs4-20}
\Delta_{\xi}w(\xi_1,\xi_2) = \left\{
\begin{array}{rl}
\Delta_{\xi}v(\xi_1,\xi_2) &\text{~if\quad} \xi_2 > 0\\[5pt]
\Delta_{\xi}v(\xi_1,-\xi_2) &\text{~if\quad} \xi_2 < 0\\[5pt]
\Delta_{\xi}v(\xi_1,0) &\text{~if\quad} \xi_2 = 0,
\end{array}
\right.
\end{equation}
and hence, by \eqref{E-lsbs3-8},
\begin{equation}\label{E-lsbs4-22}
|\Delta_{\xi} w| \le C\,  |w| \mbox{ in } D_{+}(0,r_0).
\end{equation}
 We can apply Aronszajn's theorem \index{Aronszajn Theorem} to $w$ and, since $w \not \equiv 0$, conclude that $w$ does not vanish at infinite order at $0$ in $L^1$-norm, so that $v$ does not vanish at infinite order in $L^1$-norm either.
This proves that $u$ cannot vanish at infinite order at the boundary point $y$.\medskip

\noid \emph{Robin boundary condition.}

According to the third line in \eqref{E-lsbs3-10}, we have $\partial_{\xi_2}v(\xi_1,0) - g(\xi_1)  v(\xi_1,0) \equiv 0$, where $ g(\xi_1) := h(\xi_1) \, \sqrt{J_E}(\xi_1,0)$. Introduce the function
\begin{equation}\label{E-lsbs4-32}
v_1(\xi_1,\xi_2) := e^{- g(\xi_1)\xi_2}\, v(\xi_1,\xi_2).
\end{equation}

Then,
\begin{equation}\label{E-lsbs4-32a}
\left\{
 \begin{array}{l}
\partial_{\xi_1}v_1(\xi_1,\xi_2) = e^{- g(\xi_1)\xi_2}\, \partial_{\xi_1}v(\xi_1,\xi_2) - g'(\xi_1)\xi_2\, v_1(\xi_1,\xi_2)\\[5pt]
\partial_{\xi_2}v_1(\xi_1,\xi_2) = e^{-g(\xi_1)\xi_2}\, \partial_{\xi_2}v(\xi_1,\xi_2) - g(\xi_1)\, v_1(\xi_1,\xi_2),
 \end{array}
 \right.
\end{equation}
and hence,
\begin{equation}\label{E-lsbs4-32b}
\left\{
 \begin{array}{l}
v_1(\xi_1,0) = v(\xi_1,0)\\[5pt]
\partial_{\xi_2}v_1(\xi_1,0) =  \partial_{\xi_2}v(\xi_1,0) - g(\xi_1)\,  v(\xi_1,0) \equiv 0,
 \end{array}
 \right.
\end{equation}
so that the function $v$ satisfies the Neumann boundary condition. Then,
\begin{equation}\label{E-lsbs4-34}
\left\{
 \begin{array}{ll}
\partial^2_{\xi_1^2}v_1(\xi_1,\xi_2) = & e^{-g(\xi_1)\xi_2} \partial^2_{\xi_1^2}v(\xi_1,\xi_2) - 2 g'(\xi_1)\xi_2 \partial_{\xi_1}v_1\\[5pt]
& \quad - [g^{''}(\xi_1)\xi_2 - (\xi_2g'(\xi_1))^2]v_1\\[5pt]
\partial^2_{\xi_2^2}v_1(\xi_1,\xi_2) = & e^{-g(\xi_1)\xi_2} \partial^2_{\xi_2^2}v(\xi_1,\xi_2) - 2 g(\xi_1) \partial_{\xi_2}v_1(\xi_1,\xi_2)\\[5pt]
& \quad - g^2(\xi_1) v_1.
 \end{array}
 \right.
\end{equation}

It follows that the function $v_1$ satisfies
\begin{equation}\label{E-lsbs-36}
\left\{
 \begin{array}{l}
\Delta_{\xi}v_1 + a_1 \partial_{\xi_1}v_1 + a_2 \partial_{\xi_2}v_1 + A v_1 = 0\\[5pt]
\partial_{\xi_2}v_1(\xi_1,0) = 0,
 \end{array}
 \right.
\end{equation}
where the functions $a_1, a_2$ and $A$ are $C^{\infty}$ and bounded in the neighborhood in which we work.\smallskip

We can now repeat the arguments of the Neumann case, using the inequality
\[
|\Delta v_1| \le C \big( |\nabla v_1| + |v_1|\big)
\]
and apply  Theorem~\ref{T-SUCP}.\smallskip

This concludes the proof of Proposition~\ref{P-lsbs4-4}.
\end{proof}

\begin{remark}\label{R-lsbs-4}
We have proved Proposition~\ref{P-lsbs4-4} under the additional assumption that $\Omega$ is simply connected. In the general case, since the property we are interested in is local, it suffices to first reduce to a simply-connected neighborhood $\Omega_1$ of the point $y_0$. Indeed, in a neighborhood of $y_0$, the boundary $\Gamma$ is a graph above the tangent to $\Gamma$ at $y_0$ which we can choose as the first coordinate axis $x_1$. Then, provided that $r$ is small enough, $\Omega_1 := \Omega \cap D_{x}(y_0,r)$ will be simply connected, and we will choose a conformal diffeomorphism from $\bH$ to $\Omega_1$,  as given in Lemma~\ref{L-lsbs3-2}.\smallskip

In the general case of a compact surface with boundary, we can use \cite[Section~2]{YaZh2021} or \cite{PiVe2020}.
\end{remark}%

We also have the following relations in the sense of distributions.

\begin{equation}\label{E-lsbs4-42f1}
\partial_{\xi_1}w(\xi_1,\xi_2) \overset{\cD'}{=} \left\{
 \begin{array}{ll}
 \partial_{\xi_1}v(\xi_1,\xi_2) & \text{if\quad} \xi_2 > 0 \\[5pt]
 \partial_{\xi_1}v(\xi_1,-\xi_2) & \text{if\quad} \xi_2 > 0.
 \end{array}
 \right.
\end{equation}

\begin{equation}\label{E-lsbs4-42f2}
\partial_{\xi_2}w(\xi_1,\xi_2) \overset{\cD'}{=} \left\{
 \begin{array}{ll}
 \partial_{\xi_2}v(\xi_1,\xi_2) & \text{if\quad} \xi_2 > 0 \\[5pt]
- \partial_{\xi_2}v(\xi_1,-\xi_2) & \text{if\quad} \xi_2 > 0.
 \end{array}
 \right.
\end{equation}

\begin{equation}\label{E-lsbs4-42s11}
\partial^2_{\xi_1^2}w(\xi_1,\xi_2) \overset{\cD'}{=} \left\{
 \begin{array}{ll}
 \partial^2_{\xi_1^2}v(\xi_1,\xi_2)  & \text{if\quad} \xi_2 > 0 \\[5pt]
 \partial^2_{\xi_1^2}v(\xi_1,-\xi_2) & \text{if\quad} \xi_2 < 0.
 \end{array}
 \right.
\end{equation}

\begin{equation}\label{E-lsbs4-42s22}
\partial^2_{\xi_2^2}w(\xi_1,\xi_2) \overset{\cD'}{=} \left\{
 \begin{array}{ll}
 \partial^2_{\xi_2^2}v(\xi_1,\xi_2)  & \text{if\quad} \xi_2 > 0 \\[5pt]
 \partial^2_{\xi_2^2}v(\xi_1,-\xi_2) & \text{if\quad} \xi_2 < 0.
 \end{array}
 \right.
\end{equation}

\begin{equation}\label{E-lsbs4-42s12}
\partial^2_{\xi_1\xi_2}w(\xi_1,\xi_2) \overset{\cD'}{=} \partial^2_{\xi_2\xi_1}w(\xi_1,\xi_2) \overset{\cD'}{=}
\left\{
 \begin{array}{ll}
 \partial^2_{\xi_1\xi_2}v(\xi_1,\xi_2)  & \text{if\quad} \xi_2 > 0 \\[5pt]
- \partial^2_{\xi_1\xi_2}v(\xi_1,-\xi_2) & \text{if\quad} \xi_2 < 0.
 \end{array}
 \right.
\end{equation}

\subsection{Proof of the Local Structure Theorem  at the Boundary}\label{SS-lsbs5}~

 Let $y_0$ be a boundary singular point of an eigenfunction $u$ of \eqref{E-lsbs-02}. To analyze $u$ in a neighborhood of $y_0$, we apply Lemma~\ref{L-lsbs3-2} and work with the function $v = u\circ E$ which
satisfies \eqref{E-lsbs3-8}. \smallskip

From Subsection~\ref{SS-lsbs4}, we know that the function $v$ does not vanish at infinite order at $0$. Let $p:=\ord(v,0)$.   Applying Taylor's formula to the function $v$ at the point $0$, in the half-disk $\wb{D}_{+}(0,r_0)$ for some $r_0 > 0$, gives
\begin{equation}\label{E-lsbs5-4}
\left\{
\begin{array}{l}
v(\xi_1,\xi_2) = \sum_{|\alpha| = p} \frac{1}{\alpha!}\, D^{\alpha}\!v(0,0)\,(\xi_1,\xi_2)^{\alpha} + R_{p+1}(\xi_1,\xi_2), \text{~where}\\[8pt]
R_{p+1}(\xi_1,\xi_2) = \sum_{|\beta| = p+1} \frac{|\beta|}{\beta!} (\xi_1,\xi_2)^{\beta}  {\displaystyle \int_0^1 (1-s)^{|\beta|-1} D^{\beta}\!v(s\xi_1,s\xi_2)\,ds}.
\end{array}%
\right.
\end{equation}
Here,  as usual,
\begin{equation}\label{E-lsbs5-4a}
\alpha = (\alpha_1,\alpha_2), \quad|\alpha| = \alpha_1 + \alpha_2, \quad (\xi_1,\xi_2)^{\alpha} = \xi_1^{\alpha_1}\, \xi_2^{\alpha_2}, \quad
D^{\alpha}\!v = \dfrac{\partial^{|\alpha|}\!v}{\partial \xi_1^{\alpha_1}\partial \xi_2^{\alpha_2}}.
\end{equation}

Using \eqref{E-lsbs3-8}, and identifying the terms with lowest order, we find that the polynomial $P_p(\xi_1,\xi_2) := \sum_{|\alpha| = p} \frac{1}{\alpha!}\, D^{\alpha}\!v(0,0)\,(\xi_1,\xi_2)^{\alpha}$ is homogenous of degree $p$, and harmonic (with respect to $\Delta_{\xi}$) in $\bH$. Writing the harmonicity condition in polar coordinates $(\rho,\omega)$ at the point $0$ in $\R^2$, we find that the polynomial $P_p$ has the form
\begin{equation}\label{E-lsbs5-4b}
P_p(\rho\, \cos \omega, \rho\, \sin \omega) = \rho^p \, Q_p(\omega),
\end{equation}
with the function $Q_p$ satisfying $Q_p^{''}(\omega) + p^2 \, Q_p(\omega) \equiv 0$ in $(0,\pi)$. It follows that $v$ can be written as
\begin{equation}\label{E-lsbs5-6}
v(\rho\, \cos \omega, \rho\, \sin \omega) = \rho^p \, Q_p(\omega) + \rho^{p+1}\,T_{p+1}(\rho , \omega),
\end{equation}
where
\begin{equation}\label{E-lsbs5-8}
\resizebox{.89 \textwidth}{!}
{$
 T_{p+1}(\rho,\omega) =
\sum_{|\beta| = p+1} \frac{|\beta|}{\beta!} (\cos \omega,\sin \omega)^{\beta} {\displaystyle \int_0^1 (1-s)^{|\beta|-1} D^{\beta}\!v(s\rho \cos\omega,s\rho \sin\omega)\,ds}.
$}
\end{equation}

Depending on the boundary condition satisfied by $u$, see \eqref{E-lsbs3-10}, the function $v$ satisfies
\begin{enumerate}
  \item the Dirichlet condition $v(\xi_1,0) = 0$,
  \item the Neumann condition $\partial_{\xi_2}v(\xi_1,0) = 0$, or
  \item the Robin condition $\partial_{\xi_2}v(\xi_1,0) - h(\xi_1) \sqrt{J_E}(\xi_1,0) \, v(\xi_1,0) = 0$.
\end{enumerate}

We now express the normal derivative $\partial_{\xi_2}$ in terms of the $\rho$ and $\omega$ derivatives,
\begin{equation}\label{E-lsbs5-8a}
\partial_{\xi_2} = \sin\omega \, \partial_{\rho} + \cos\omega \, \frac{1}{\rho}\partial_{\omega}.
\end{equation}

In polar coordinates, the relation $\xi_2=0$ is equivalent to $\omega \in \set{0,\pi}$. For $\omega_0 \in \set{0,\pi}$, we have
\begin{equation*}
v(\rho\cos\omega_0,0) = \rho^{p} Q_p(\omega_0) + \cO(\rho^{p+1})
\end{equation*}
and
\begin{equation*}
\partial_{\xi_2}v(\rho\cos\omega_0,0) = \rho^{p-1}\, \cos(\omega_0) \, Q_p'(\omega_0) + \cO(\rho^p)\,
\end{equation*}
so that
\begin{equation*}
\partial_{\xi_2}v(\rho\cos\omega_0,0) -h_E(\rho \cos\omega_0)\, v(\rho\cos\omega_0,0) = \rho^{p-1}\, \cos(\omega_0) \, Q_p'(\omega_0) + \cO(\rho^p).
\end{equation*}

From these relations, we conclude that
\begin{enumerate}
\item if $v$ satisfies the \emph{Dirichlet condition}, then  $Q_p(0) = Q_p(\pi) = 0$,
and hence
\begin{equation}\label{E-lsbs5-8d}
v(\rho\, \cos \omega, \rho\, \sin \omega) = a_v \, \rho^p \, \sin(p\, \omega) + \rho^{p+1}\,T_{p+1}(\rho , \omega),
\end{equation}
\item if $v$ satisfies the \emph{Neumann or Robin condition}, then $Q'_p(0) =
Q'_p(\pi) = 0$, and hence
\begin{equation}\label{E-lsbs5-8r}
v(\rho\, \cos \omega, \rho\, \sin \omega) = a_v \, \rho^p \, \cos(p\, \omega) + \rho^{p+1}\,T_{p+1}(\rho , \omega),
\end{equation}
\end{enumerate}
where $a_v$ is a nonzero scalar.  For an alternative proof, see Subsection~\ref{SS-lsbs6}.\smallskip

Define
\begin{equation}\label{E-lsbs5-10}
\left\{
\begin{array}{l}
V_d(\rho,\omega) := \sin(p\, \omega) + a_v^{-1}\, \rho\, T_{p+1}(\rho,\omega)\\[5pt]
V_n(\rho,\omega) := \cos(p\, \omega) + a_v^{-1}\, \rho\, T_{p+1}(\rho,\omega).
\end{array}%
\right.
\end{equation}

\noid  \emph{Dirichlet boundary condition.}~ Define the values
\begin{equation}\label{E-lsbs5-12}
\omega_j := j\frac{\pi}{p}, \quad j \in \set{0,\ldots,p},
\end{equation}
\begin{equation}\label{E-lsbs5-12a}
\alpha_1 \in \big( 0,\frac{\pi}{8} \big) \text{~and~} \alpha_p := \frac{\alpha_1}{p}.
\end{equation}

In the interval $(0,\pi)$, the function $V_d(0,\omega)$ vanishes precisely for the values $\omega_j$ with $j \in \set{1,\ldots,p-1}$.
The following relations hold.
\begin{equation}\label{E-lsbs5-14}
\left\{
\begin{array}{l}
\sin\big(p\,(\omega_j \pm \alpha_p)\big) = \pm (-1)^j \, \sin(\alpha_1)\\[5pt]
|\sin(p\, \omega)| \ge \sin \alpha_1, \text{~for~} \omega \in \bigcup_{j=0\,}^{p-1} [\omega_j+\alpha_p,\omega_{j+1}-\alpha_p].
\end{array}%
\right.
\end{equation}

On the other hand, we have
\begin{equation*}
\partial_{\omega}V_d(\rho,\omega) = p \, \cos(p \, \omega) + |a_v|^{-1}\, \rho \, \partial_{\omega}T_{p+1}(\rho,\omega), \text{~and}
\end{equation*}
\begin{equation}\label{E-lsbs5-15d}
|\cos(p \, \omega)| \ge \cos(\alpha_1) \quad \forall \omega \in [0,\alpha_p] \cup [\pi - \alpha_p, \pi]\cup\bigcup_{j=1}^{p-1}[\omega_j - \alpha_p, \omega_j + \alpha_p].
\end{equation}

Define
\begin{equation}\label{E-lsbs5-16}
r_{d,1} := \frac{1}{2} \, \min\set{r_0 ,\, |a_v|\sin(\alpha_1)\, \|T_{p+1}\|^{-1}_{\infty,\frac{r_0}{2}},\, p\,|a_v|\cos(\alpha_1)\, \|\partial_{\omega}T_{p+1}\|^{-1}_{\infty,\frac{r_0}{2}}},
\end{equation}
where $\|\cdot\|_{\infty,\frac{r_0}{2}}$ denotes the $L^{\infty}$ norm of functions in $D_{+}(0,\frac{r_0}{2})$. Then, for all $r \le r_{d,1}$,
\begin{equation}\label{E-lsbs5-16a}
\left\{
\begin{array}{l}
\pm \, (-1)^j \, V_d(r,\omega_j \pm \alpha_p) \ge \frac 12 \, \sin(\alpha_1) \text{~~for all~~} 1 \le j \le p-1 \\[5pt]
|\partial_{\omega}V_d(r,\omega)| \ge \frac p2 \, \cos(\alpha_1)\\[5pt]
\quad \text{for all~} \omega \in [0,\alpha_p] \cup [\pi - \alpha_p, \pi]\cup\bigcup_{j=1}[\omega_j - \alpha_p, \omega_j + \alpha_p],\\[5pt]
|V_d(r,\omega)| \ge \frac 12 \, \sin(\alpha_1) \text{~~for all~~}
\omega \in \bigcup_{j=0}^{p-1}[\omega_j + \alpha_p, \omega_{j+1} - \alpha_p].
\end{array}
\right.
\end{equation}

\begin{proposition}\label{P-lsbs5-2d}
Assume that $\rho \le r_{d,1}$. The following properties hold.
\begin{enumerate}
  \item[(i)] The function $\omega \mapsto V_d(\rho,\omega)$ does not vanish in
  \[
 \bigcup_{j=0}^{p-1} [\omega_j+\alpha_p,\omega_{j+1}-\alpha_p].
  \]
  \item[(ii)] For each $j \in \set{1,\ldots,p-1}$, the function $\omega \mapsto V_d(\rho,\omega)$ has exactly one zero
  $\omegat_j(\rho) \in \big( \omega_j - \alpha_p ,\, \omega_j+\alpha_p \big)$, and does not vanish in $(0,\alpha_p] \cup [\pi - \alpha_p,\pi)$.
  \item[(iii)] For each $j \in \set{1,\ldots,p-1}$, the function $\rho \mapsto \omegat_j(\rho)$ is $C^{\infty}$ in the interval $\big( 0 ,\, r_1\big)$ and tends to $\omega_j$ as $\rho$ tends to zero.
  \item[(iv)] For each $j \in \set{1,\ldots ,p-1}$, the curve
  \[
  ( 0 ,\, r_{d,1}) \ni \rho \mapsto a_j(\rho) = \big( \rho \, \cos\omegat_j(\rho) ,\, \rho \, \sin\omegat_j(\rho)  \big)
  \]
  is smooth and has semi-tangent $\omega_j$ at the origin.
\end{enumerate}
\end{proposition}%

In the Dirichlet case, recall that  $p := \ord(v,0)$ and $\rho(v,0) = p-1$.\medskip

\begin{proof} For the proof, we use \eqref{E-lsbs5-16a}. To prove (i), we observe that in each set $\set{\rho} \times \big[ \omega_j+\alpha_p , \omega_{j+1}-\alpha_p \big]$, with $0 < \rho \le r_{d,1}$ and $j = 0 \ldots (p-1)$, $|V_d(\rho,\omega)| \ge \frac 12 \sin(\alpha_1)$. To prove (ii), we observe that the function $V_d(\rho,\omega)$ changes sign in each set $\set{\rho}\times \big( \omega_j -\alpha_p,\omega_j +\alpha_p\big)$, and that its partial derivative with respect to $\omega$ does not vanish.  We use a similar reasoning in $(0,\alpha_p] \cup [\pi - \alpha_p,\pi)$, taking into account the fact that $V_d(\rho,\omega)$ vanishes for $\omega=0$ or $\pi$. The first part of Assertion~(iii) follows from the implicit function theorem; the second part from the fact that $\alpha_1$ can be chosen as small as we want. Assertion~(iv) follows from the previous ones.
\end{proof}

Introduce the following notation (``$\cG\cR\cB$-arcs'').
\begin{equation}\label{E-lsbs5-18}
\left\{
\begin{array}{l}
[r,\omega] := (r\, \cos\omega, r\, \sin\omega),\\[5pt]
C_{+}(0,r) := \set{[r,\omega] \mid \omega \in (0,\pi)},\\[5pt]
\cG_{\mathrm{d}}(r,j) := \set{[r,\omega] \mid \omega \in (\omega_j - \alpha_p,\omega_j+\alpha_p)}, \text{~for~} 1 \le j \le (p-1),\\[5pt]
\cR_{\mathrm{d}}(r,j) := \set{[r,\omega] \mid \omega \in [\omega_j + \alpha_p,\omega_{j+1}-\alpha_p]}, \text{~for~} 0 \le j \le (p-1),\\[5pt]
\cB_{\mathrm{d}}(r,0) := \set{[r,\omega] \mid \omega \in (0,\alpha_p]},\\[5pt]
\cB_{\mathrm{d}}(r,p) := \set{[r,\omega] \mid \omega \in [\pi-\alpha_p,\pi)}.
\end{array}
\right.
\end{equation}

These arcs are illustrated in Figure~\ref{F-gam0-0bis}, left image (Dirichlet case with $p=8$). The $\cG$-arcs are centered on the rays $\omega_j$; the $\cR$-arcs do not intersect the rays; the $\cB$-arcs do not intersect the rays either and touch the boundary.\smallskip

For $r \le r_{d,1}$, the nodal set $\cZ(v)$ meets each $\cG$-arc precisely once, and does not meet the $\cR\cB$-arcs. More precisely, for $0 < r \le r_{d,1}$, we have the following properties:
\begin{equation}\label{E-lsbs5-18a}
\left\{
\begin{array}{l}
 \pm \, (-1)^{j}\, \sign(a_v) \, v([r,\omega_j \pm \alpha_p]) \ge \frac 12 |a_v|\, \sin(\alpha_1)\, r^p.\\[5pt]
\text{In~} \cG_{\mathrm{d}}(r,j), 1 \le j \le (p-1), |\partial_{\omega}v(r,\omega)| \ge \frac p2 \, |a_v| \, \cos(\alpha_1) \, r^p \\[5pt]
\hspace{0.8cm} \text{and~} v(r,\omega) \text{~vanishes precisely once.} \\[5pt]
\text{In~} \cB_{\mathrm{d}}(r,0) \cup \cB_{\mathrm{d}}(r,p), |\partial_{\omega}v(r,\omega)| \ge \frac p2 \, |a_v| \, \cos(\alpha_1) \, r^p\\[5pt]
\hspace{0.8cm} \text{and~} v(r,\omega) \text{~does not vanish.}\\[5pt]
\text{In~} \cR_{\mathrm{d}}(r,j), 0 \le j \le (p-1), |v(r,\omega)| \ge \frac 12 \, |a_v|\, \sin(\alpha_1) \, r^p\\[5pt]
\hspace{0.8cm} \text{and~} v(r,\omega) \text{~does not vanish.}
\end{array}
\right.
\end{equation}

\noid \emph{Neumann or Robin boundary condition.}~  We introduce  the values
\begin{equation}\label{E-lsbs5-12n}
\omega'_j := (j-\frac 12)\,\frac{\pi}{p}, \quad j \in \set{1,\ldots,p},
\end{equation}
\begin{equation}\label{E-lsbs5-12na}
\alpha_1 \in \big( 0,\frac{\pi}{8} \big) \text{~and~} \alpha_p := \frac{\alpha_1}{p}.
\end{equation}

In the interval $(0,\pi)$, the function $V_n(0,\omega)$ vanishes precisely for the values $\omega'_j$, $1 \le j \le p$. The following relations hold.

\begin{equation}\label{E-lsbs5-14n}
\left\{
\begin{array}{l}
\cos\big(p\,(\omega'_j \pm \alpha_p)\big) = \pm (-1)^j \, \sin(\alpha_1) \text{~~for~~} 1 \le j \le p\\[5pt]
|\cos(p\, \omega)| \ge \sin \alpha_1 \text{~~for all}\\[5pt]
\hspace{0.8cm} \omega \in [0,\omega'_1-\alpha_p] \cup [\omega'_p+\alpha_p,\pi] \cup \,  \bigcup_{j=1}^{p-1} [\omega'_j+\alpha_p,\omega'_{j+1}-\alpha_p].
\end{array}%
\right.
\end{equation}

On the other hand, we have
\begin{equation*}
\partial_{\omega}V_n(\rho,\omega) = - p \, \sin(p \, \omega) + |a_v|^{-1}\, \rho \, \partial_{\omega}T_{p+1}(\rho,\omega), \text{~and}
\end{equation*}
\begin{equation}\label{E-lsbs5-15n}
|\sin(p \, \omega)| \ge \cos(\alpha_1) ,\quad \forall \omega \in \bigcup_{j=1}^p[\omega'_j - \alpha_p, \omega'_j + \alpha_p].
\end{equation}
Define
\begin{equation}\label{E-lsbs5-16n}
r_{n,1} := \frac{1}{2} \, \min\set{r_0 ,\, |a_v|\sin(\alpha_1)\, \|T_{p+1}\|^{-1}_{\infty,\frac{r_0}{2}},\, p\, |a_v|\cos(\alpha_1)\, \|\partial_{\omega}T_{p+1}\|^{-1}_{\infty,\frac{r_0}{2}}},
\end{equation}
where $\|\cdot\|_{\infty,\frac{r_0}{2}}$ denotes the $L^{\infty}$ norm of functions on the disk $D_{+}(0,\frac{r_0}{2})$.\smallskip

Then, for all $r \le r_{n,1}$,
\begin{equation}\label{E-lsbs5-16na}
\left\{
\begin{array}{l}
\pm \, (-1)^j \, V_n(r,\omega'_j \pm \alpha_p) \ge \frac 12 \, \sin(\alpha_1) \text{~~for~~} 1 \le j \le p \\[5pt]
|\partial_{\omega}V_n(r,\omega)| \ge \frac p2 \, \cos(\alpha_1)
\text{~for all~} \omega \in \bigcup_{j=1}^{p}[\omega'_j - \alpha_p, \omega'_j + \alpha_p],\\[5pt]
|V_n(r,\omega)| \ge \frac 12 \, \sin(\alpha_1) \text{~~for all}\\[5pt]
\hspace{0.8cm}
\omega \in [0,\omega'_1-\alpha_p] \cup [\omega'_p+\alpha_p,\pi] \cup \bigcup_{j=1}^{p-1} [\omega'_j+\alpha_p,\omega'_{j+1}-\alpha_p].
\end{array}
\right.
\end{equation}

\begin{proposition}\label{P-lsbs5-2n}
Assume that $\rho \le r_{n,1}$. The following properties hold.
\begin{enumerate}[(i)]
\item The function $\omega \mapsto V_n(\rho,\omega)$ does not vanish in
\[
[0,\omega'_1-\alpha_p] \cup [\omega'_p+\alpha_p,\pi] \cup \bigcup_{j=1}^{p-1} [\omega'_j+\alpha_p,\omega'_{j+1}-\alpha_p].
\]
\item For each $j \in \set{1,\ldots,p}$, the function $\omega \mapsto V_n(\rho,\omega)$ has exactly one zero $\omegat'_j(\rho) \in \big( \omega'_j - \alpha_p ,\, \omega'_j+\alpha_p \big)$.
  \item For each $j \in \set{1,\ldots,p}$, the function $\rho \mapsto \omegat'_j(\rho)$ is $C^{\infty}$ in the interval $\big( 0 ,\, r_{n,1}\big)$ and tends to $\omega'_j$ as $\rho$ tends to zero.
  \item For each $j \in \set{1,\ldots,p}$, the curve
  \[
  \big( 0 , r_{n,1}\big) \ni \rho \mapsto a_j(\rho) = \big( \rho \, \cos(\omegat_j(\rho) ,\, \rho \, \sin(\omegat_j(\rho)  \big)
  \]
  is smooth and has semi-tangent $\omega'_j$ at the origin.
\end{enumerate}
\end{proposition}%

Recall that, in the Robin case,  $p := \ord(v,0) = \rho(v,0) =: q$.\medskip

\begin{proof} The proof is similar to the previous one. \end{proof}

Introduce the following notation (``$\cG\cR$-arcs'').
\begin{equation}\label{E-lsbs5-18n}
\left\{
\begin{array}{l}
[r,\omega] := (r\, \cos\omega, r\, \sin\omega),\\[5pt]
C_{+}(0,r) := \set{[r,\omega] \mid \omega \in (0,\pi)},\\[5pt]
\cG_{\mathrm{n}}(r,j) := \set{[r,\omega] \mid \omega \in (\omega'_j - \alpha_p,\omega'_j+\alpha_p)}, \text{~for~} 1 \le j \le p,\\[5pt]
\cR_{\mathrm{n}}(r,j) := \set{[r,\omega] \mid \omega \in [\omega'_j + \alpha_p,\omega'_{j+1}-\alpha_p]}, \text{~for~} 1 \le j \le (p-1),\\[5pt]
\cR_{\mathrm{n}}(r,0) := \set{[r,\omega] \mid \omega \in (0,\omega'_1-\alpha_p]},\\[5pt]
\cR_{\mathrm{n}}(r,p) := \set{[r,\omega] \mid \omega \in [\omega'_p+\alpha_p,\pi)}.
\end{array}
\right.
\end{equation}

These arcs are illustrated in Figure~\ref{F-gam0-0bis}, right image (Robin case, with $p=7$). The $\cG$-arcs are centered on the rays $\omega'_j$\,; the $\cR$-arcs do not intersect the rays and there are no $\cB$-arc. \smallskip

For $0 < r < r_{n,1}$, the nodal set $\cZ(v)$ meets each $\cG$-arc precisely once, and does not meet the $\cR$-arcs. More precisely, for $0 < r \le r_{n,1}$, we have the following properties:
\begin{equation}\label{E-lsbs5-18na}
\left\{
\begin{array}{l}
 \mp \, (-1)^{j}\, \sign(a_v) \, v([r,\omega'_j \pm \alpha_p]) \ge \frac 12 |a_v|\, \sin(\alpha_1)\,  r^p.\\[5pt]
\text{In~} \cG_{\mathrm{n}}(r,j), 1 \le j \le p, |\partial_{\omega}v(r,\omega)| \ge \frac p2 \, |a_v| \, \cos(\alpha_1) \, r^p \\[5pt]
\hspace{0.8cm} \text{and~} v(r,\omega) \text{~vanishes precisely once.} \\[5pt]
\text{In~} \cR_{\mathrm{n}}(r,0) \cup \cR_{\mathrm{n}}(r,p), |v(r,\omega)| \ge \frac 12 \, |a_v| \, \sin(\alpha_1) \, r^p\\[5pt]
\hspace{0.8cm} \text{and~} v(r,\omega) \text{~does not vanish.}\\[5pt]
\text{In~} \cR_{\mathrm{n}}(t,j), 1 \le j \le (p-1), |v(r,\omega)| \ge \frac 12 \, |a_v|\, \sin(\alpha_1) \, r^p\\[5pt]
\hspace{0.8cm} \text{and~} v(r,\omega) \text{~does not vanish.}
\end{array}
\right.
\end{equation}

\begin{remark}\label{R-lsbs5-2}~
Translated into properties of the eigenfunction $u$, Propositions~\ref{P-lsbs5-2d}--\ref{P-lsbs5-2n}, tell us that, in a neighborhood of a singular point $y_0$ of $u$, the nodal set $\cZ(u)$ consists of $(p-1)$, resp. $p$, smooth semi-arcs emanating from $y_0$ tangentially to the rays $\omega_j$, resp. $\omega'_j$.  These arcs are contained in the image under $E$ of conical neighborhoods of the rays (controlled by the parameter $\alpha_1$).
\end{remark}%

The $\cG$-arcs appear in white on the semi-circles in Figure~\ref{F-gam0-0bis} as they represent open gates through which the nodal set exits the disk. In the digital version, the $\cR\cB$-arcs are colored respectively red, and blue. In the printed version they appear in gray.

\begin{figure}[!htb]
  \centering
  \includegraphics[width=0.9\textwidth]{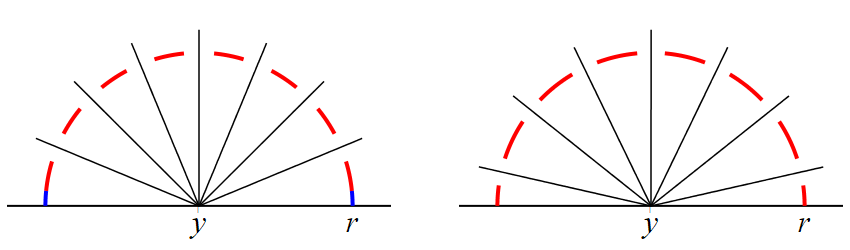}
  \caption{$\cG\cR\cB$-arcs for $v$ (left, Dirichlet -- right, Robin), $\rho(v,0)=7$}\label{F-gam0-0bis}
\end{figure}

\FloatBarrier
\subsection{A refined Taylor formula near a boundary singular point}\label{SS-lsbs6}~

Let $v = u\circ E$ as in Subsection~\ref{SS-lsbs5}, and $p = \ord(v,0)$. Applying Taylor's formula at order $(p+1)$ to the function $v$ at the point $0$ in the half-disk $\wb{D}_{+}(0,r_0)$, gives
\index{2-R@$R$!$R_{\beta}$}
\begin{equation}\label{E-lsbs6-4}
\left\{
\begin{array}{ll}
v(\xi_1,\xi_2) & = \sum_{|\alpha| = p} \frac{1}{\alpha!} D^{\alpha}\!v(0) (\xi_1,\xi_2)^{\alpha}\\[5pt]
& \quad + \sum_{|\alpha| = p+1} \frac{1}{\alpha!} D^{\alpha}\!v(0) (\xi_1,\xi_2)^{\alpha}\\[5pt]
& \quad + \sum_{|\beta|=p+2}R_{\beta}(\xi_1,\xi_2) \, (\xi_1,\xi_2)^{\beta}, \text{where} \\[8pt]
R_{\beta}(\xi_1,\xi_2) &  = \frac{|\beta|}{\beta!} {\displaystyle \int_0^1 (1-s)^{|\beta|-1} D^{\beta}\!v(s \xi_1,s \xi_2)\,ds}.
\end{array}%
\right.
\end{equation}
Then, both
\begin{equation}\label{E-lsbs6-6}
\left\{
\begin{array}{l}
p_0(\xi_1,\xi_2) := \sum_{|\alpha| = p} \frac{1}{\alpha!}\, D^{\alpha}\!v(0) (\xi_1,\xi_2)^{\alpha} \text{~and}\\[5pt]
p_1(\xi_1,\xi_2) := \sum_{|\alpha| = p+1} \frac{1}{\alpha!} D^{\alpha}\!v(0) (\xi_1,\xi_2)^{\alpha}
\end{array}
\right.
\end{equation}
 are homogeneous \emph{harmonic} polynomials, respectively of degree $p$ and $(p+1)$.\medskip

Harmonic homogenous polynomials of degree $n \ge 1$ in two variables form a two dimensional vector space $\cH_n$. Writing
\index{2-C@$C_n$} \index{2-S@$S_n$}
\begin{equation}\label{E-lsbs6-12}
(\xi_1 + i \,\xi_2)^n = C_n(\xi_1,\xi_2) + i \, S_n(\xi_1,\xi_2),
\end{equation}
where
\begin{equation}\label{E-lsbs6-14}
\left\{
\begin{array}{ll}
C_n(\xi_1,\xi_2) & = \sum^n_{k=0,\text{~even}\,}(-1)^{\frac k2} \binom{n}{k}\, \xi_1^{n-k}\, \xi_2^k\\[5pt]
& = \xi_1^n - \binom{n}{2}\xi_1^{n-2}\xi_2^2 + \binom{n}{4}\xi_1^{n-4}\xi_2^4 - \cdots \\[5pt]
S_n(\xi_1,\xi_2) & = \sum^n_{k=0,\text{~odd}\,}(-1)^{\floor{\frac k2}} \binom{n}{k}\, \xi_1^{n-k}\, \xi_2^k \\[5pt]
& = n  \xi_1^{n-1}\, \xi_2 - \binom{n}{3}\xi_1^{n-3}\xi_2^3 + \binom{n}{5}\xi_1^{n-5}\xi_2^5 - \cdots\,
\end{array}
\right.
\end{equation}
we obtain a basis $\set{C_n,S_n}$ of $\cH_n$. In polar coordinates, we have
\begin{equation}\label{E-lsbs6-14p}
\left\{
\begin{array}{l}
C_n(\rho\,\cos\omega,\rho\, \sin\omega) = \rho^n\, \cos(n\omega) \\[5pt]
S_n(\rho\,\cos\omega,\rho\, \sin\omega) = \rho^n\, \sin(n\omega) .
\end{array}
\right.
\end{equation}

We list the following relations for later use.
\begin{equation}\label{E-lsbs6-16}
\left\{
\begin{array}{ll}
\partial_{\xi_1}C_n(\xi_1,\xi_2) & = n\, C_{n-1}(\xi_1,\xi_2)\\[5pt]
& = n\, \xi_1^{n-1} - (n-2) \binom{n}{2}\, \xi_1^{n-3}\xi_2^2 +  \cdots \\[5pt]
\partial_{\xi_1}S_n(\xi_1,\xi_2) & = n\, S_{n-1}(\xi_1,\xi_2) \\[5pt]
& = n(n-1)\,  \xi_1^{n-2}\, \xi_2 - n \binom{n-1}{3}\, \xi_1^{n-4}\xi_2^3 + \cdots.
\end{array}
\right.
\end{equation}
\begin{equation}\label{E-lsbs6-18}
\left\{
\begin{array}{ll}
\partial_{\xi_2}C_n(\xi_1,\xi_2) & = -n\, S_{n-1}(\xi_1,\xi_2)\\[5pt]
& = -2\binom{n}{2}\, \xi_1^{n-2}\xi_2 + 4 \binom{n}{4}\, \xi_1^{n-4}\xi_2^3 +  \cdots \\[5pt]
\partial_{\xi_2}S_n(\xi_1,\xi_2) & = n\, C_{n-1}(\xi_1,\xi_2)\\[5pt]
& = n\,  \xi_1^{n-1} - 3 \binom{n}{3}\, \xi_1^{n-3}\xi_2^2 + \cdots.
\end{array}
\right.
\end{equation}
\begin{equation}\label{E-lsbs6-20}
\left\{
\begin{array}{rllrl}
C_n(\xi_1,0) & = \xi_1^n & \text{~and~}  & S_n(\xi_1,0) & = 0\\[5pt]
\partial_{\xi_1}C_n(\xi_1,0) & = n\, \xi_1^{n-1} &  \text{~and~} & \partial_{\xi_1}S_n(\xi_1,0) & = 0\\[5pt]
\partial_{\xi_2}C_n(\xi_1,0) & = 0 &  \text{~and~} & \partial_{\xi_2}S_n(\xi_1,0) & = n\, \xi_1^{n-1}.
\end{array}
\right.
\end{equation}
\medskip

Coming back to the Taylor formula at order $(p+1)$ for the function $v$, and using the preceding relations, we rewrite \eqref{E-lsbs6-4} as
\begin{equation}\label{E-lsbs6-26}
\begin{array}{ll}
v(\xi_1,\xi_2) & = c_{v,p}\, C_p(\xi_1,\xi_2) + s_{v,p}\, S_p(\xi_1,\xi_2)\\[5pt]
&\qquad +\, c_{v,p+1}\, C_{p+1}(\xi_1,\xi_2) + s_{v,p+1}\, S_{p+1}(\xi_1,\xi_2)\\[5pt]
&\qquad +\, \sum_{|\beta|=p+2}R_{\beta}(\xi_1,\xi_2)\, (\xi_1,\xi_2)^{\beta}.
\end{array}%
\end{equation}

\noid \emph{Dirichlet boundary condition.} In this case, we have $v(\xi_1,0) \equiv 0$. Using \eqref{E-lsbs6-20}, we obtain
\begin{equation*}
0 \equiv c_{v,p}\, \xi_1^p + c_{v,p+1} \, \xi_1^{p+1} + \cO(\xi_1^{p+2}),
\end{equation*}
which implies that $c_{v,p} = c_{v,p+1} = 0$.\medskip

\noid \emph{Neumann boundary condition.} In this case, we have $\partial_{\xi_2}v(\xi_1,0) \equiv 0$. Using \eqref{E-lsbs6-20}, we obtain
\begin{equation*}
0 \equiv p\, s_{v,p}\, \xi_1^{p-1} + (p+1)\, s_{v,p+1} \, \xi_1^{p} + \cO(\xi_1^{p+1}),
\end{equation*}
which implies that $s_{v,p} = s_{v,p+1} = 0$.\medskip

\noid \emph{Robin boundary condition.} In this case, we have $\partial_{\xi_2}v(\xi_1,0) \equiv h_E(\xi_1)\, v(\xi_1,0)$. Using \eqref{E-lsbs6-20}, we obtain

\begin{equation*}
p\, s_{v,p}\, \xi_1^{p-1} + (p+1)\, s_{v,p+1} \, \xi_1^{p} + \cO(\xi_1^{p+1}) \equiv h_E(0) \, c_{v,p}\, \xi_1^p + \cO(\xi^{p+1}),
\end{equation*}
which implies that $s_{v,p} = 0$ and $(p+1) \, s_{v,p+1} = h_E(0) \, c_{v,p}$.\medskip

We have proved the following lemma.

\begin{lemma}\label{L-lsbs6-2} Depending on the boundary condition, for the function $v$ such that $\ord(v,0)=p$, the Taylor formula at the point $0$ and at order $(p+1)$ is given as follows.
\vspace{-1em}
\begin{align*}
\intertext{Dirichlet problem}
v(\xi_1,\xi_2) & = s_{v,p} S_p(\xi_1,\xi_2) + s_{v,p+1}  S_{p+1}(\xi_1,\xi_2) + R_{p+2}(\xi_1,\xi_2),\\
\intertext{Neumann problem}
v(\xi_1,\xi_2) & = c_{v,p} C_p(\xi_1,\xi_2) + c_{v,p+1}  C_{p+1}(\xi_1,\xi_2) + R_{p+2}(\xi_1,\xi_2),\\
\intertext{Robin problem}
v(\xi_1,\xi_2) & =  c_{v,p} C_p(\xi_1,\xi_2) + c_{v,p+1}  C_{p+1}(\xi_1,\xi_2) + \frac{c_{v,p} h_E(0)}{p+1}\, S_{p+1}(\xi_1,\xi_2)\\
& \quad + R_{p+2}(\xi_1,\xi_2),\\
\intertext{where the remainder term}
R_{p+2}(\xi_1,\xi_2) & = \sum_{|\beta|=p+2}R_{\beta}(\xi_1,\xi_2) (\xi_1,\xi_2)^{\beta}
\end{align*}
vanishes at order at least $(p+2)$ at zero.
\end{lemma}%

\begin{remark}\label{R-lsbs6-2}
Note that one recovers the Neumann case from the Robin case. Note also that one recovers the formulas in polar coordinates form given in \eqref{E-lsbs5-8d} and \eqref{E-lsbs5-8r}.
\end{remark}%

\subsection{Non simply connected domains}\label{SS-lsbs9}~

Let $\Omega \subset \R^2$ be a smooth bounded domain, contained in some open disk $D(0,r)$. The boundary $\Gamma$ of $\Omega$ is a compact $1$-manifold without boundary, so that it has finitely many  components, $\Gamma_1, \ldots, \Gamma_k$, $1 \le k < \infty$, which are all diffeomorphic to $\bS^1$, and hence are Jordan curves. Each curve $\Gamma_j$ separates the plane into two  components, one bounded $B_j$, the other one unbounded $D_j$. \medskip

\textbf{Claim}.~ Relabeling the components $\Gamma_j$ is necessary,
$\Omega = B_1 \sm \bigcup_{j=2}^k B_j$.

 \emph{Sketch of the proof}.
\begin{description}
  \item[\nf Fact~1.] \emph{For all $j$, $1 \le j \le k$, $\Omega \subset B_j$ or $\Omega \subset D_j$.}\\
      Assume this is not the case. Then, for some $j$, there exists $b, d \in \Omega$ with $b \in B_j$ and $d \in D_j$. Since $\Omega$ is connected, there exists a continuous path from $b$ to $d$ entirely contained in $\Omega$. On the other hand, this path would have to intersect $\Gamma_j$, a contradiction.
  \item[\nf Fact~2.] \emph{There exists at most one $j$ such that $\Omega \subset B_j$.}\\
      Assume this is not the case, and that $\Omega \subset B_1 \cap B_2$. Since $\Gamma_2$ does not intersect $\Gamma_1$, it must be contained in $B_1$ or in $D_1$.  If $\Gamma_2 \subset D_2$ then $B_1 \cap B_2 = \emptyset$, a contradiction. If $\Gamma_2 \subset  B_1$, we reach a contradiction with the fact that we have point is $\Omega$ as close as we want from $\Gamma_1$ although $B_1$ is at positive distance from $\Gamma_1$.
  \item[\nf Fact~3.] \emph{There exists some $j$ such that $\Omega \subset B_j$.}\\
      Take some $x_0 \in \Omega$, and consider $r_0 := \sup \set{d(x_0,x) \mid x \in \Omega}$. Then $r_0$ is finite and the supremum is achieved at some $x_1 \in \Gamma$, and $\Omega \subset \wb{B(x_0,r_1)}$. If $x_1 \in \Gamma_j$ then $\Omega \subset B_j$.
  \item[\nf Fact~4.] \emph{Up to relabeling the boundary components, we have} $\Omega \subset B_1$ \emph{and} $\Omega \subset D_j$ \emph{for all} $j = 2, \ldots, k$, \emph{and}
      \[
      \Omega = B_1 \bigcap\, \cap_{j=2}^k D_j = B_1 \sm \bigcup_{j=2}^k B_j.
      \]
The inclusion $\subset$ is clear. If the inclusion $\supset$ were not true, we would find a point $x \in B_1 \bigcap \,  \cap_{j=2}^k D_j$, not in $\Omega$. The largest disk $B(x,r)$ contained in $B_1 \bigcap \cap_{j=2}^k D_j$ would touch one of the $\Gamma_j$ and this would yield a contradiction since there is a one-sided neighborhood of each $\Gamma_j$ contained in $\Omega$.
\quad \qedc
\end{description}

The domain $B_1$ is simply connected and its boundary $\Gamma_1$ is called the \emph{outer boundary} of $\Omega$. We say that $B_1$ is obtained from $\Omega$ by \emph{filling the holes.}\smallskip

Given any  component $\Gamma_j$ of $\Gamma$ there exists a diffeomorphism $\Psi$ of $\R^2$ such that the outer boundary of $\Psi(\Omega)$ is $\Psi(\Gamma_1)$. For this purpose, it suffices to use two stereographic projections from $\bS^2$ to $\R^2$, using adhoc points on $\bS^2$. As explained in \cite[p.~24]{BeKr1987} using the outer boundary and applying Lemma~\ref{L-lsbs3-2} several times, one can construct a conformal diffeomorphism from $B_1$ to the unit disc $\D$ sending the outer boundary of $\Omega$ to $\partial \D$, and the other boundary components to analytic Jordan curves in $\D$.
\begin{remark}\label{R-lsbs9-4}
As far as the local boundary behavior of an eigenfunction is concerned, we could use the following alternative argument. Given a point $m_0 \in \Gamma$, choose the coordinate system $(x_1,x_2)$ in $\R^2$ such that $m_0 = (0,0)$ and the $x_1$-axis is tangent to $\Gamma$ at $0$. Choose $a>0$ small enough so that $\Omega_a := \Omega \cap D_{x}(0,a)$ is simply connected and $\Gamma \cap D_{x}(0,a)$ is a graph above the segment $(-a,a) \times \set{0}$.

Since $\Omega_a$ is simply connected, there exists a conformal diffeomorphism $E$ from $\bH$ onto $\Omega_a$, and this map is smooth up to the boundary except around the intersection points of $\partial D_{x}(0,a)$ with $\Gamma$.  We may also assume that $E(0,0) =(0,0)$. Choose $r_0 > 0$ small enough so that $E|_{D_{+}(0,r_0)\cap \bH}$ is $C^{\infty}$ up to the boundary.
\end{remark}%

\chapter{Eigenvalue Bounds for Riemannian Spheres with Potential}\label{Ch-rsp}

\section{Revisiting the Multiplicity Bounds for Riemannian Spheres}\label{S-h2n}

\subsection{Introduction}\label{SS-h2n1}

 In the first section of this chapter, we revisit the paper \cite{HoHN1999} by M. and T.~Hoffmann-Ostenhof and N.~Nadirashvili, in which the authors consider Riemannian spheres with potential $(M,g,V)$. This means that $M$ is a $\Cty$ surface homeomorphic to the sphere, and that it is equipped with a $\Cty$ Riemannian metric $g$, and with a $\Cty$ real valued potential $V$. We denote the eigenvalues of the Schr\"{o}dinger operator $-\Delta + V$ on $M$ by $\set{\lambda_k}_{k=1}^{\infty}$, with first label $1$, see Section~\ref{S-evp}, Equation~\eqref{E-evp-2c}.\smallskip

According to Cheng \cite{Chen1976},  for $(M,g,0)$, $\mult(\lambda_2) \le 3$, and this bound is sharp, achieved for the round metric on the sphere. Nadirashvili \cite{Nadi1987}, proved that $\mult(\lambda_k)\break \le (2k-1)$, for any
$k \ge 1$, and for any such $(M,g,V)$.
In \cite{HoHN1999}, M. and T. Hoffmann-Ostenhof and Nadirashvili, prove the following result.

\begin{theorem}[\cite{HoHN1999}, Theorem~1]\label{T-h2n}
For any smooth Riemannian metric $g$, and any smooth real valued potential $V$ on the sphere $M$, the eigenvalues of the operator $-\Delta+V$ satisfy ~$\mult(\lambda_k) \le (2k-3)$ for any $k \ge 3$.
\end{theorem}%

\emph{Sketch of the proof of Theorem~\ref{T-h2n}}.
Fix some $x \in M$. Consider the eigenspace
$U(\lambda_k)$. We first prove Nadirashvili's estimate,
\begin{equation}\label{E-h2n-4a}
\text{for~}(M,g,V), ~~\mult(\lambda_k) \le (2k-1) \text{~for all~} k \ge 1.
\end{equation}
The estimate is clear for $k=1$. Assume, by contradiction,  that $\mult(\lambda_k) = \dim U(\lambda_k) \ge 2k $ for some integer $k \ge 2$. Then, according to Lemma~\ref{L-zeroi}, there exists a function $0 \not = u \in U(\lambda_k)$ such that $\nu(u,x) \ge 2k$. By Courant's theorem, Theorem~\ref{T-RC},  the number of nodal domains of $u$ satisfies $\kappa(u) \le k$. Since  $M$ is  topologically  a sphere, we can apply Euler's formula \eqref{E-etf0-2} to the nodal set $\cZ(u)$ of the function $u$,
\begin{equation}\label{E-h2n-2}
\kappa(u) = 1 + b_0(\cZ(u)) + {\frac 12} \, \sum_{z \in \cS(u)} \big( \nu(u,z) - 2 \big).
\end{equation}

Summing up the above information, we obtain
\begin{equation}\label{E-h2n-4}
0 \ge \kappa(u) - k = \left\lbrace b_0(\cZ(u)) - 1\right\rbrace + \sum_{\substack{z \in \cS(u) \\z \not = x}} \frac{\nu(u,z) - 2}{2}
+ \lbrace \frac{\nu(u,x)}{2} - k + 1 \rbrace \ge 1,
\end{equation}
a contradiction. \smallskip

Note that inequality \eqref{E-h2n-4a} is sharp for $k=1$ and $2$, see Table~\ref{E-scr-2T} in Section \ref{S-sketch}. \smallskip

In view of \eqref{E-h2n-4a}, to prove Theorem~\ref{T-h2n}, it suffices to show that the cases in which $\dim U(\lambda_k) = (2k-1)$ or $\dim U(\lambda_k) = (2k-2)$ cannot occur when $k\ge 3$. This is the purpose of the following two subsections, in which we revisit the arguments of \cite{HoHN1999}. More precisely, in Subsection~\ref{SS-h2n-s2}, we assume that $\dim U(\lambda_k) = (2k-1)$ for some $k \ge 3$, and we reach a contradiction using an argument which will be recurrent in this paper, the \emph{rotating function argument}, \index{Rotating function argument} see Paragraph~\ref{SSS-h2n-s2r}. In Subsection~\ref{SS-h2n-s3}, we assume that $\dim U(\lambda_k) = (2k-2)$ for some $k \ge 3$, and we reach a contradiction by using both the \emph{rotating function argument} and the Poincar\'{e}-Hopf theorem, \index{Poincar\'e-Hopf theorem}  see \cite[Chap.~6, p.~35]{Miln1997}

\subsection{Riemannian spheres with potential, $\dim U(\lambda_k) \le (2k-2)$ for $k\ge 3$}\label{SS-h2n-s2}\phantom{}

The proof of this upper bound is by contradiction. Taking \eqref{E-h2n-4a} into account, we assume that
\begin{equation}\label{E-h2n-s2}
\text{there exists some}\quad k \ge 3 \text{~~such that~~} \dim U(\lambda_k) = (2k-1).
\end{equation}

Fixing some $x \in M$, we introduce the subspace \index{2-W@$W_x$}
\begin{equation}\label{E-h2n-s2a}
W_x := \set{u \in U(\lambda_k) \mid \nu(u,x) \ge 2(k-1)}.
\end{equation}

Fix a direct orthonormal frame $\set{\vec{e}_1,\vec{e}_2}$ in $T_xM$, and the associated polar coordinates $(r,\omega)$, via the exponential map $\exp_x$.\smallskip

\subsubsection{Structure of the nodal set $\cZ(u)$, for $0 \neq u \in W_x$}\label{SSS-h2n-s2s}

\begin{properties}\label{propertyD2}
Assume that $\dim U(\lambda_k) = (2k-1)$ for some  $k \ge 3$. The linear subspace
\begin{equation*}
W_x = \set{u \in U(\lambda_k) \mid \nu(u,x) \ge 2(k-1)}
\end{equation*}
has the following properties. For any $0 \not = u \in W_x$,
\begin{enumerate}[(i)]
 \item $\nu(u,x) = 2(k-1)$, and $x$ is the only singular point of the function $u$;
 \item $\cZ(u)$ is connected;
 \item $\kappa(u) = k$.
\end{enumerate}
\noindent Furthermore, $\dim W_x = 2$, and there exists a basis $\set{v_1,v_2}$ of $W_x$ such that, in the local polar coordinates $(r,\omega)$ centered at $x$,
\begin{equation}\label{E-h2n-5}
\left\{
\begin{array}{l}
v_1(r,\omega) = r^{k-1} \sin ((k-1)\omega) +\mathcal O (r^k),\\
v_2(r,\omega) = r^{k-1} \cos ((k-1)\omega) +\mathcal O (r^k).
\end{array}
\right.
\end{equation}
\end{properties}

\begin{proof}  According to Lemma~\ref{L-zeroi}~(ii), there exist two linearly independent functions $u_{1}, u_{2} \in U(\lambda_k)$, with $\nu(u_{i},x) \ge 2(k-1)$, for $i = 1, 2$, so that $\dim W_x \ge 2$.  Given any $0 \not = u \in W_x$, we can apply Inequality~\eqref{E-h2n-4} to $u$,
\begin{equation}\label{E-h2n-6}
0 \ge \kappa(u) - k = \left\lbrace b_0(\cZ(u)) - 1\right\rbrace + \sum_{\substack{z \in \cS(u)\\ z \not = x}} \frac{\nu(u,z) - 2}{2} + \lbrace \frac{\nu(u,x)}{2} - k + 1 \rbrace.
\end{equation}
The three terms in the right-hand side of \eqref{E-h2n-6} are nonnegative, and their sum is nonpositive. They must all vanish. This proves Assertions~(i)--(iii).\smallskip

We already know that $W_x$ has dimension at least $2$. Assume that it has dimension at least $3$, and let $v_1, v_2, v_3$ be three linearly independent functions in $W_x$. Since $\nu(v_i,x) = 2(k-1)$,  Lemma~\ref{L-zeroc} implies the existence of a nontrivial linear combination $v$ of these functions with $\nu(v,x) \ge 2k$, contradicting Assertion~(i). Hence, $\dim W_x = 2$.\smallskip

For any $0 \neq u \in W_x$ and $\xi \in T_xM$, we have
\begin{equation*}
u\big(\exp_x(t \xi)\big) = t^{(k-1)} \, h_{x,u}(\xi) + \cO(t^k),
\end{equation*}
where $h_{x,u}$ is a harmonic homogeneous polynomial of degree $(k-1)$ in $\xi$. On $W_x$, the map $u \mapsto h_{x,u}$ is linear and injective. Since the space $\cH_{x,(k-1)}$ of harmonic homogeneous polynomials of degree $(k-1)$ has dimension $2$, this map is bijective. In the polar coordinates $(r,\omega)$,  this space  is spanned by the polynomials $r^{(k-1)}\, \sin\big( (k-1)\omega\big)$ and $r^{(k-1)}\, \cos\big( (k-1)\omega\big)$ . This proves the existence of the basis $\set{v_1,v_2}$ satisfying \eqref{E-h2n-5}. This proves the last assertion.\smallskip

The proof of Properties~\ref{propertyD2} is now complete.\end{proof}

 Let $0 \neq u \in W_x$. The local structure theorem -- Corollary~\ref{cor:nodloc}~(i) -- implies that, in a neighborhood of $x$, the nodal set $\cZ(u)$ consists of  $2(k-1)$ nodal semi-arcs which emanate from $x$, tangentially to $2(k-1)$ rays dividing the unit circle in $T_xM$ into equal parts. Since $\cS(u) = \set{x}$, if we follow a nodal semi-arc emanating  from $x$, we will eventually come back to $x$. Using the fact that $\cS(u) = \set{x}$ and the connectedness of $\cZ(u)$, we conclude that $\cZ(u)$ consists of $(k-1)$ simple loops at $x$, and that these loops only intersect each other at $x$.

\begin{definition}\label{D-h2n-rose}
A \emph{$p$-bouquet of loops at $x$}\index{Bouquet of loops}\index{Bouquet of loops!$p$-bouquet} is a collection of $p$ piecewise $C^1$ loops at $x$, which do not intersect away from $x$, and whose semi-tangents at $x$ are pairwise transverse in $T_xM$.
\end{definition}%

Therefore, for any $0 \neq u \in W_x$, the nodal set $\cZ(u)$ is a $(k-1)$-bouquet of loops at the point $x$.\smallskip

\subsubsection{Combinatorial type of the nodal set $\cZ(u)$, for $0 \neq u \in W_x$}\label{SSS-h2n-s2c}
\index{Combinatorial type}
 Using the frame $\set{\vec{e}_1, \vec{e}_2}$ in $T_xM$, we label the rays tangent to $\cZ(u)$ at $x$ counter-clockwise, according to their angle with respect to $\vec{e}_1$ (two consecutive rays making an angle $\frac{\pi}{k-1}$), so that we obtain an ordered list
\[ \set{0 \le \vartheta_0 < \ldots < \vartheta_{(2k-3)} < 2\pi}.\]
The loops in the $(k-1)$-bouquet of loops $\cZ(u)$ can now be described by a map
\begin{equation}\label{E-h2n-tau}
\tau_{x,u} : \set{0,\ldots,(2k-3)} \to \set{0,\ldots,(2k-3)},
\end{equation}
which is defined  in the following way: for $j \in   \set{0,\ldots,(2k-3)}$,   we consider the nodal arc which emanates from $x$ tangentially to the ray $\vartheta_j$, and define  $\tau_{x,u}(j)$ as the label of the ray tangent to the nodal arc when it arrives back at $x$, forming a loop at $x$, with semi-tangents the rays $\vartheta_j$ and $\vartheta_{\tau_{x,u}(j)}$.  We denote this loop by $\gamma^{x,u}_{j,\tau_{x,u}(j)}\,$.

 \index{1-tau@$\tau$!$\tau_{x,u}$}
\begin{definition}[Combinatorial type]\label{D-h2n-type}
The map $\tau_{x,u}$ is called the \emph{combinatorial type}\index{Combinatorial type} of the function $u$, or of the nodal set $\cZ(u)$, at the point $x$.
\end{definition}%

We write $\tau$ instead of $\tau_{x,u}$ when there is no ambiguity. The properties of nodal sets imply that
\begin{equation}\label{E-h2n-taup}
\left\{
\begin{array}{l}
\tau_{x,u}(j) \neq j ~~ \text{~for all~} j \in \set{0,\ldots,
(2k-3)},\\[5pt]
\tau_{x,u}^2 = \id.
\end{array}%
\right.
\end{equation}

Note that changing the frame $\set{\vec{e}_1,\vec{e}_2}$ at $x$, keeping the orientation, amounts to conjugating the map $\tau_{x,u}$ by a circular permutation of the set  $\set{0,\ldots,(2k-3)}$.\smallskip

\subsubsection{The rotating function argument}\label{SSS-h2n-s2r}
\index{Rotating function argument}
Fix the basis $\set{v_1,v_2}$ of $W_x$ provided by Properties~\ref{propertyD2}, using the direct frame $\set{\vec{e}_1, \vec{e}_2}$ in $T_xM$ and the associated  local polar coordinates $(r,\omega)$ at $x$ such that \eqref{E-h2n-5} holds.\smallskip

We now analyze the nodal sets of the one-parameter family
\begin{equation}\label{E-h2n-6a}
w_\theta = \cos((k-1) \theta) \,v_1 - \sin((k-1) \theta) \, v_2,
\end{equation}
for $\theta \in \big[ 0, \frac{\pi}{k-1}\big]$. In particular, we have
\begin{equation}\label{E-h2n-6b}
v_1 = w_0 = - w_{\frac{\pi}{k-1}} \text{~~and~~} v_2 = - w_{\frac{\pi}{2(k-1)}},
\end{equation}
and
\begin{equation}\label{E-h2n-6c}
w_\theta(r,\omega) = r^{k-1}\, \sin((k-1)(\omega - \theta)) \, +\mathcal O (r^k).
\end{equation}

Properties~\ref{propertyD2} state that $x$ is the sole singular point of the eigenfunction $w_{\theta}$, and that the nodal set $\cZ(w_{\theta})$ is connected. With respect to the frame $\set{\vec{e}_1, \vec{e}_2}$ in $T_xM$, there are $2(k-1)$ nodal semi-arcs which emanate from $x$, tangentially to the $2(k-1)$ rays $\set{\omega = \omega_j(\theta) := \omega_j + \theta}$, where $\omega_j :=j \frac{\pi}{k-1}, j \in \set{0,\ldots,(2k-3)}$. We call these rays $\omega_j(\theta)$ for short, and we view $j$ as defined modulo $2(k-1)$.\medskip

\index{1-gamma@$\gamma$!$\gamma_{j,\tau(j)}$}
The nodal set $\cZ(w_{\theta})$ is a $(k-1)$-bouquet of loops described by the map $\tau_{x,w_{\theta}}$ associated with the rays $\set{\omega_j(\theta)}$. Call this map $\tau_{\theta}$ for short,  and call $\gamma_{j,\tau_{\theta}(j)}^{\theta}$ the corresponding loops at $x$. \smallskip

\begin{property}\label{propertyD3}
Assume that  $k \ge 3$, and  $\dim U(\lambda_k) = (2k-1)$. Choose some $j \in \set{0,\ldots,(2k-3)}$. Considering $\tau_{\theta}(j)$ instead of $j$ if necessary, we can assume that $0 \le j < \tau_{\theta}(j) \le 2k-3$. The loop $\gamma_{j,\tau_{\theta}(j)}^{\theta}$ separates  (the topological sphere) $M$ into two components. The rays $\omega_k(\theta)$ such that $j < k < \tau_{\theta}(j)$ point inside one of the two components ; the rays $\omega_k(\theta)$ with $k < j$ or $k > \tau_{\theta}(j)$ point inside the other component. In particular, $\tau_{\theta}(j) - j$ is an odd integer.
\end{property}%

The proof of this property is clear.

\begin{property}\label{propertyD4}
Assume that  $k \ge 3$, and  $\dim U(\lambda_k) = (2k-1)$. The combinatorial type $\tau_{\theta}$ of the function $w_{\theta} = \cos((k-1) \theta) \,v_1 - \sin((k-1) \theta) \, v_2$ does actually not depend on $\theta$. More precisely, for any $j \in \set{0,\ldots,(2k-3)}$, and for any $\theta \in \big[ 0,\frac{\pi}{k-1}\big]$, the loop which emanates from $x$ tangentially to the ray $\omega_j(\theta)$ arrives at $x$ tangentially to the ray $\omega_{\tau_0(j)}(\theta)$.  We shall henceforth denote this map by $\tau$.
\end{property}%

\begin{proof} Since all the functions $w_{\theta}$ share the same properties, it suffices to show that $\tau_{\theta} = \tau_0$ for $\theta$ small enough. Assume the contrary. Then, there exists a sequence $\theta_n$ tending to zero, and a sequence $\set{j_n} \subset \set{0,\ldots,2k-3}$ such that $\tau_{\theta_n}(j_n) \not = \tau_0(j_n)$. Since the sequence $\{j_n\}$ takes finitely many values, we can find a constant subsequence $\set{j_{n,1}} \subset \set{j_n}$. Similarly, there exists a subsequence $\set{j_{n,2}} \subset \set{j_{n,1}}$  such that $\tau_{\theta_{n,2}}(j_{n,2})$ is constant. Hence, there exists some $\ell \in \set{0,\ldots,2k-3}$, and a sequence $\theta_n$ such that $\tau_{\theta_n}(\ell) \not = \tau_0(\ell)$. Without loss of generality, we may assume that $\ell = 0$, so that there exists a sequence $\theta_n$ tending to zero such that $\tau_{\theta_n}(0) \equiv \ell_0 \not = \tau(0)$.\smallskip

We now use a more precise version of the local structure theorem, see Section~\ref{S-lsns}. For any $\alpha > 0$ small enough, there exists $r_0 > 0$ such that, for all $\theta$, $ \cZ(w_{\theta}) \cap B(x,2r_0)$ consists of $2(k-1)$ nodal semi-arcs
\begin{equation}\label{E-h2n-8a}
A_j(r,\theta) : (0,2r_0) \ni r \mapsto \exp_x\big(r\, \omegat_j(r,\theta) \big) \in B(x,2r_0),
\end{equation}
for $j \in \set{0,\ldots,2k-3}$. Here, we assume that an orthonormal frame $\set{e_1,e_2}$ has been chosen in $T_xM$, is such a way that the vector $e_1$ directs the ray $\omega_0$. In the polar coordinates $(r,\omega)$ associated with this frame, we identify the angle $\omega$ with a point on the unit circle. Furthermore, the functions $\omegat_j$ are smooth in $(r,\theta) \in (0,2r_0) \times [0,2\pi]$, and they satisfy,
\begin{equation}\label{E-h2n-8b}
\left\{
\begin{array}{l}
\omegat_j(r,\theta) \in \big( \omega_j+\theta - \alpha , \omega_j + \theta +\alpha \big),\\[5pt]
\lim_{r \to 0} \omegat_j(r,\theta) = \omega_j + \theta,
\end{array}
\right.
\end{equation}
for all $j, 0 \le j \le 2k-3$. The semi-arc $A_j(r,\theta)$ is semi-tangent to the ray $\omega_j+\theta$ at the point $x$.\smallskip

We now reason as in the proof of Lemma~\ref{L-hdn}. In the closed ball $\wb{B}(x,r_0)$, the nodal set $\cZ(w_{\theta_n})$ consists of $2(k-1)$ nodal semi-arcs $A_j(\cdot,\theta_n)$, with end points $x$ and $\exp_x(r_0\, \omegat_j(r_0,\theta_n))$ which converge to the corresponding semi-arcs $A_j(\cdot,0)$ with end points $x$ and $\exp_x(r_0\, \omegat_j(r_0,0))$.\smallskip

In the compact set $M\sm B(x,r_0)$, the nodal set $\cZ(w_{\theta_n})$ consists of $(k-1)$ disjoint connected nodal arcs $C_j(r_0,\theta_n)$ with two end points $\exp_x(r_0\, \omegat_j(r_0,\theta_n))$ and $\exp_x(r_0\, \omegat_{\tau(j)}(r_0,\theta_n))$, which correspond to the intersections of the loops in $\cZ(w_{\theta_n})$ with $M\sm B(x,r_0)$. We look more precisely at the arcs $C_0(r_0,\theta_n)$. From this sequence of compact connected subsets of $M\sm B(x,r_0)$, we can extract a subsequence which converges in the Hausdorff distance to some compact connected set $C_0$. Since any $z \in C_0$ is the limit of a sequence $z_n \in C_0(r_0,\theta_n)$, and since $w_{\theta_n}$ tends to $w_0$ uniformly on $M$, we conclude that $w_0(z) = 0$, i.e., that $C_0 \subset \cZ(w_0)$. The set $C_0$ contains the points $\exp_x(r_0\, \omegat_0(r_0,0))$ and $\exp_x(r_0\, \omegat_{\ell_0}(r_0,0))$. Since $C_0$ is connected and contained in $\cZ(w_0)$, and in view of the structure of $\cZ(w_0)$, we must have $\ell_0 = \tau(0)$, and we reach a contradiction. The proof of Property~\ref{propertyD4} is complete.\end{proof}

\begin{figure}[!ht]
  \centering
  \includegraphics[width=0.95\textwidth]{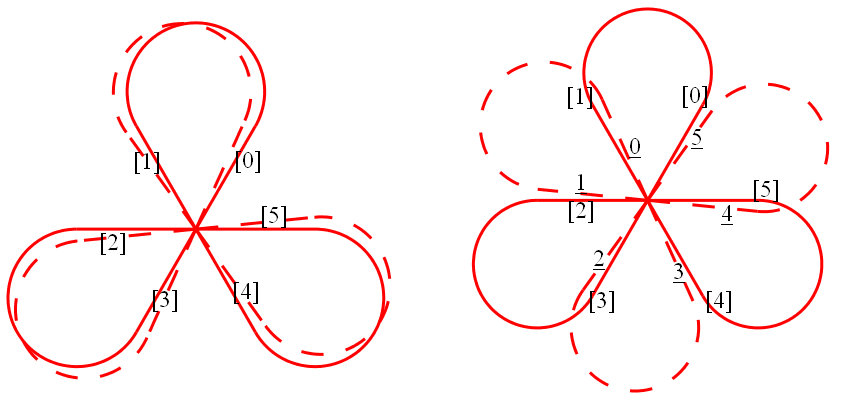}
  \caption{Example 1: $k=4$, $\theta \approx 0$ (left) and $\theta \approx \frac{\pi}{3}$ (right)}\label{F-h2n-k4-ex1}
\end{figure}

\begin{figure}[!ht]
  \centering
  \includegraphics[width=0.95\textwidth]{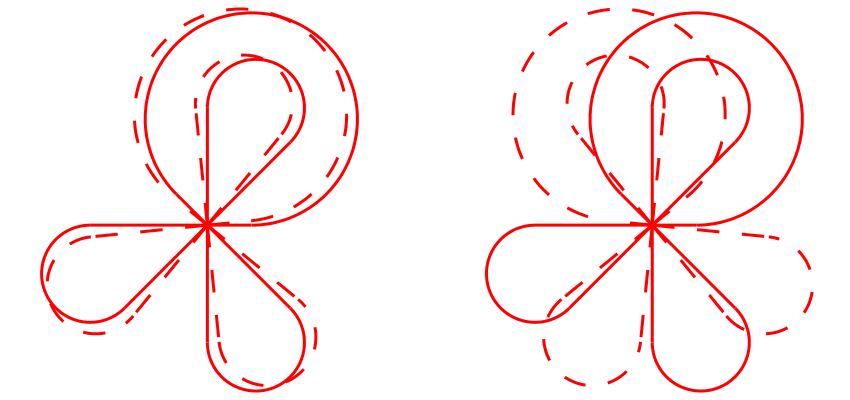}
  \caption{Example 2: $k=5$, $\theta \approx 0$ (left) and $\theta \approx \frac{\pi}{4}$ (right)}\label{F-h2n-k5-ex2}
\end{figure}

\begin{proof}[Conclusion of the rotating function argument]  Under the assumpt\-ion that $k \ge 3$, and  $\dim U(\lambda_k) = (2k-1)$, we can apply the previous construction. Since $w_{\frac{\pi}{k-1}} = - v_1$, we infer from Property~\ref{propertyD4} that $\gamma^{\frac{\pi}{k-1}}_{0,\tau(0)}= \gamma^0_{1,\tau(0)+1}$. Since there is only one nodal semi-arc tangent to a given ray at $x$, we conclude that $\tau(0) \not = 1$ and $\tau(0) \not = 2k-3$. It follows that $0 < 1 < \tau(0)$, and that $\tau(0) < \tau(0)+1 = \tau(1) \le 2k-3$  (here we have used the assumption $k\ge 3$). This contradicts Property~\ref{propertyD3}, and proves that $\dim U(\lambda_k) = (2k-1)$ cannot occur.  This is illustrated in  Figures~\ref{F-h2n-k4-ex1} and \ref{F-h2n-k5-ex2} (with $k=4$). The nodal set of $w_0$ is displayed as a solid line, the nodal set of $w_{\theta}$ is displayed as a dashed line. In the left sub-figures, $\theta$ is close to $0$, in the right sub-figures $\theta$ is close to $\frac{\pi}{k-1}$. In Figure~\ref{F-h2n-k4-ex1}, the labels $[0], \ldots, [5]$ indicate the rays for the initial function $w_0$. When $\theta \approx 0$, in the left picture, the rays for $w_{\theta}$ are close to the rays for $w_0 $ and they are not labelled. In the right picture, they are labelled $\underline{0}, \ldots, \underline{5}$  by continuity as $\theta$ varies.\\
We have proved the inequality
\begin{equation}\label{E-h2n-b2}
\mult(\lambda_k) \le (2k-2), \text{~for any~} k\ge 3.
\end{equation}
\end{proof}%

\begin{remark}\label{R-h2n-bess}
As far as we know, the idea to consider the family of functions $w_{\theta}$ was introduced by Besson \cite{Bess1980}, in the  proof of his Theorem~3.C.1 in which he improves the upper bound for the multiplicity of the second eigenvalue of a torus from $7$ to $6$. A similar idea was used by Nadirashvili  \cite{Nadi1987}, p.~231 lines 1--8, for higher eigenvalues as well. It is used in \cite{HoHN1999, HoMN1999, Berd2018} also, and will appear several times, in one form or another, in the present paper.
\end{remark}%

\begin{remark}\label{R-h2n-f}
In the second section of this chapter, we relate the \emph{combinatorial type} of a nodal bouquet to the \emph{labelling of the nodal domains} around the critical zero. In Chapters~\ref{Ch-pdwb} and \ref{Ch-scpdsb}, we will also use the notions of \emph{combinatorial type} and \emph{labelling of nodal domains} for certain nodal sets on a compact surface $M$ with boundary, and we will repeatedly use a \emph{rotating function argument}.\index{Rotating function argument}
\end{remark}%

\FloatBarrier

\newpage
\subsection{Riemannian spheres with potential, $\dim U(\lambda_k) \le (2k-3)$ for $k\ge 3$}\label{SS-h2n-s3}\phantom{}

The proof of this upper bound is by contradiction. Taking Subsection~\ref{SS-h2n-s2} into account, we assume that $\dim U(\lambda_k) = (2k-2)$. \smallskip

As in the previous subsection, fix some $x \in M$, a direct orthonormal frame $\set{\vec{e}_1,\vec{e}_2}$ in $T_xM$, and the associated polar coordinates $(r,\omega)$, via the exponential map $\exp_x$.

\begin{properties} \label{propertyD6} Assume that $\dim U(\lambda_k) = (2k-2)$ for some $k \ge 3$.  Then, the linear subspace
\begin{equation*}
W_x = \set{u \in U(\lambda_k) \mid \nu(u,x) \ge 2(k-1)}
\end{equation*}
has the following properties.
\begin{enumerate}[(i)]
 \item $\dim W_x = 1$ and, for any $0\not = u \in W_x\,$,
 \item $\nu(u,x) = 2(k-1)$, and $x$ is the only singular point of the eigenfunction $u$;
 \item $\cZ(u)$ is connected;
 \item $\kappa(u) = k$.
 \end{enumerate}
\end{properties}

\begin{proof} By Lemma~\ref{L-zeroi}, $\dim W_x \ge 1$. Given any $0 \not = u \in W_x$, Euler's formula gives
\begin{equation}\label{E-h2n-10}
0 \ge \kappa(u) - k = \left\lbrace b_0(\cZ(u)) - 1\right\rbrace + \sum_{\substack{z \in \cS(u)\\ z \not = x}} \frac{\nu(u,z) - 2}{2} + \lbrace \frac{\nu(u,x)}{2} - k + 1 \rbrace,
\end{equation}
and we conclude that (ii)--(iv) hold. To prove that $\dim W_x = 1$, assuming that $\dim W_x \ge 2$, we can repeat the arguments of Subsection~\ref{SS-h2n-s2}, and reach a contradiction. \end{proof}

\index{2-P@$\bP(U)$}
\begin{property}\label{P-h2n-laa}
 Assume that $\dim U(\lambda_k) = (2k-2)$ for some $k \ge 3$.  Let $[W_x]$ denote the line $W_x$ as a point in the projective space $\bP(U)$ of $U$. Then, the map $x \mapsto [W_x]$ from $M$ to $\bP(U)$ is $\Cty$.
\end{property}%

\begin{proof} Choose a basis $\set{\phi_j}_{j=1}^{(2k-2)}$ of $U$, and write $w_{x}$ as $w_{x} = \sum_{j=1}^{2k-2} \alpha_j(x) \phi_j$. According to Lemma~\ref{L-zeroi}, the condition $\ord(w_{x},x) \ge (k-1)$ is equivalent to $(2k-3)$ linear equations in the derivatives of $w_x$ at $x$, and hence in the  $(2k-2)$ unknowns $\set{\alpha_j(x)}_{j=1}^{(2k-2)}$. Call $\cM(x)$ the associated matrix. The system reads $\cM(x) \cA(x) = 0$ where $\cA(x)$ is the vector associated with the coefficients $\alpha_j(x)$. Since $\dim W_x = 1$ for all $x \in M$, this linear system has constant rank $(2k-3)$. Given some $x_0 \in M$, there exists a $(2k-3)\times(2k-3)$ sub-matrix $\cM'_{x_0}$ which is invertible, and the same is true for the sub-matrix $\cM'_x$ provided that $x$ is close enough to $x_0$. We can now find $w_x$, alias the coefficients $\alpha_j(x)$, by solving a linear system of the form $\cM'_x \cA'(x) = \cB(x)$ in a neighborhood of $x_0$. In view of Cramer's method, the coefficients of $\cA'(x)$ are given as quotients of determinants whose coefficients are $\Cty$ in $x$, and the determinant at the  denominator is nonzero. \end{proof}

\index{2-S@$\bS(U)$}
Since $M$ is simply connected, the map $x \mapsto [W_x]$ can be lifted to a smooth map $x \mapsto w_x$, from $M$ to $\bS(U(\lambda_k))$, the unit sphere of $U(\lambda_k)$ (for example with respect to the $L^2$ norm).\smallskip

To each $x \in M$, we associate the homogeneous polynomial of degree $(k-1)$,  $h_{x,w_x}$ on $T_xM$ defined by $p_x := h_{x,w_x} : T_xM \ni \xi \mapsto \frac{d^{k-1}}{dt^{k-1}} w_x(\exp_x(t\xi))$. It is harmonic with respect to the Riemannian metric $g_x$ in $T_xM$. The map $x \mapsto p_x$ is smooth. The restriction of the polynomial $p_x$ to the unit circle $S_xM$ in $(T_xM,g_x)$ has simple zeros. Choose some $x_0 \in M$, and some root $e_{x_0}$ of $p_{x_0}$ in $S_{x_0}M$. Given any $x_1 \in M$, and any curve $c$ from $x_0$ to $x_1$, we can follow this root by continuity along the curve $c$ so that $e_{c(t)}$ is a root of $p_{c(t)}$. Since the set of roots is discrete, and since $M$ is simply connected, the root $e_{c(1)}$ at $x_1$ does not depend on the choice of the curve $c$. \\
 It follows that we have defined a continuous unit vector-field $x \mapsto e_x$ on $M$, contradicting the Poincar\'e-Hopf theorem (as recalled below). Here, we have indeed a vector-field without zero, and $\chi(M)=2$. We have proved that the assumption $\dim U(\lambda_k) = (2k-2)$ leads to a contradiction, and hence that $\dim U(\lambda_k) \le (2k-3)$. \hfill \qed \smallskip

For the sake of completeness, we recall the statement of the Poincar\'{e}-Hopf theorem. \index{Poincar\'e-Hopf theorem}\index{Euler characteristic}

\begin{theorem}[\cite{Miln1997}, p.~35]\label{T-PH}
Let $X$ be a closed manifold. Let $w$ be a $\Cty$ vector-field on $X$ with isolated zeros. Then, the sum of the indices at the zeros of $w$ is equal to the Euler characteristic of $X$.
\end{theorem}%

The proof of Theorem~\ref{T-h2n} is now complete.

\section{Spheres: Labeling Nodal Loops and Nodal Domains}\label{S-llnd}

\subsection{Preamble}\label{SS-llnd-0}

In Section~\ref{S-h2n}, we have used the notion of \emph{combinatorial type} to study the eigenvalue multiplicity problem on Riemannian spheres with potential, $(M,g,V)$. More precisely, we have introduced the combinatorial type of very special eigenfunctions, namely eigenfunctions $u \in U(\lambda_k)$, with only one singular point $x \in M$, and whose nodal set $\cZ(u)$ is a $(k-1)$-bouquet of loops at $x$, see Properties~\ref{propertyD2} and \ref{propertyD6}. The combinatorial type $\tau_u$ of $\cZ(u)$ describes how the loops are organized in the bouquet $\cZ(u)$.\smallskip

In this section, we introduce the \emph{nodal word} \index{Nodal word} $\cW_{u}$ in order to describe how the nodal domains of $u$ are organized around $x$, by looking at their intersections with a small geodesic circle $S_x(r)$ with center $x$ and radius $r$. We prove that the nodal word $\cW_u$ determines the combinatorial type $\tau_u$ and vice-versa. Similar considerations will be applied to smooth bounded domains in $\R^2$, see Section~\ref{S-hmn2L}.\medskip

We make the following assumptions through out this section.

\begin{assumptions}\label{A-llnd-0}~
\begin{enumerate}[(i)]
  \item $(M,g,V)$ is a Riemannian sphere $M$ with $\Cty$ Riemannian metric $g$ and $\Cty$ real valued potential $V$.
  \item $u$ is an eigenfunction of the Schr\"{o}dinger operator $(-\Delta+V)$ on $M$, with a unique critical zero $x \in M$ such that $\ord(u,x) = p$ for some $p \ge 2$.
  \item $\cZ(u)$ is connected.
\end{enumerate}
\end{assumptions}%

\begin{enumerate}[]
\item For convenience, we fix
    \begin{enumerate}[~~$\diamond$]
    \item an orientation of $T_xM$,
    \item a ray  $\omega_0$ in $T_xM$, tangent to $\cZ(u)$ at $x$,
    \item the direct orthonormal frame $\cE(\omega_0) := \set{e_{0,1},e_{0,2}}$ in $T_xM$ whose first vector $e_{0,1}$ is supported by $\omega_0$,
    \end{enumerate}
\item and we denote
    \begin{enumerate}[~~$\diamond$]
    \item the rays tangent to $\cZ(u)$ at $x$ by  $\omega_0, \omega_1, \ldots, \omega_{(2p-1)}$,  counter-clockwise,
    \item the polar coordinates associated with $\cE(\omega_0)$ in $T_xM$ by $(r,\omega)$,
    \item the combinatorial type of $\cZ(u)$ with respect to the rays $\set{\omega_0,\ldots,\omega_{(2p-1)}}$, by $\tau_{u}$ (see Definition~\ref{D-h2n-type}).
    \end{enumerate}
\end{enumerate}
\medskip

Let $u$ be an eigenfunction satisfying Assumptions~\ref{A-llnd-0}.\smallskip

\emph{Fact~1.~ The nodal set $\cZ(u)$ is a $p$-bouquet of loops described by the pair $(L,\tau)$, where $L = \set{0,\ldots,(2p-1)}$ and $\tau = \tau_{u}$ is the combinatorial type of $u$. We now denote this bouquet by $\cB_L$.}\smallskip

Indeed, in a small neighborhood of $x$, the nodal set $\cZ(u)$ consists of $2p$ nodal semi-arcs emanating from $x$, tangentially to $2p$ distinct\footnote{The rays actually form an equiangular system, but this is not needed here.} rays $\omega_0,\ldots,\omega_{2p-1}$. Since nodal arcs can only meet at critical zeros, and since $\cS(u) = \set{x}$, the nodal interval emanating from $x$ tangentially to the ray $\omega_j$ must end up at $x$, arriving tangentially to some (different) ray $\omega_{\tau(j)}$, thus forming a loop $\gamma_{j,\tau(j)}$ at $x$. This defines the map $\tau$,
\begin{equation}\label{E-llnd-2}
\tau : L \to L,
\end{equation}
with the following properties,
\begin{equation}\label{E-llnd-4}
\left\{
\begin{array}{l}
\tau^2 = \id,\\[5pt]
\tau(j) \not = j, \quad \forall j \in L,\\[5pt]
\tau(j) - j \text{~is odd}, \quad \forall j \in L.
\end{array}%
\right.
\end{equation}

The first property is clear. The second and third ones follow from the local structure theorem and from the fact that $M$ is a sphere, see Property~\ref{propertyD3}.  \smallskip

\emph{Fact~2.~ The eigenfunction $u$ has $(p+1)$ nodal domains.}\smallskip

This property is a consequence of the following lemma; it is related to the fact that the eigenfunctions in Properties~\ref{propertyD2} and \ref{propertyD6} are  \emph{Courant-sharp}\index{Courant-sharp}, i.e., they have the maximum number of nodal domains allowed by Courant's nodal domain theorem, see Section~\ref{S-mcs}.

\begin{lemma}\label{L-llnd-2}
The complement of a $p$-bouquet of loops $\cB$ in the sphere $M$ has $(p+1)$ components.
\end{lemma}%

\begin{proof}  We can turn the $p$-bouquet of loops $\cB$ into a (simple) graph $\cG_{\cB}$ on the sphere by vertex-edge additions as explained in Section~\ref{S-pargef}. Then, the number of components in the complement of the bouquet $\cB$ is the same as the number of components in the complement of the graph $\cG_{\cB}$. Using the notation of Section~\ref{S-pargef}, Euler's formula for the graph $\cG_{\cB}$ reads
\[
r(\cG_{\cB}) = \alpha_1(\cG_{\cB}) - \alpha_0(\cG_{\cB}) + c(\cG_{\cB}) + 1
\]
as in Equation~\eqref{E-etf-00}.  It is easy to see that $\alpha_1(\cG_{\cB}) - \alpha_0(\cG_{\cB}) = (p-1)$. Since $c(\cG_{\cB}) = 1$, the result follows.
\end{proof}

In the polar coordinates $(r,\omega)$, and for $r$ small enough, the nodal semi-arcs  emanating from $x$ are given by equations $r \mapsto \exp_x(r\, \omegat_j(r))$, $0 \le j \le 2p-1$, see Section~\ref{S-lsns}. These arcs determine $2p$ intervals (sub-arcs)
\begin{equation}\label{E-llnd-8}
I_j(r) := \set{ \exp_x(r \,\omega) \mid  \omega \in \big( \omegat_j(r),\omegat_{j+1}(r) \big) }, \quad 0 \le j \le 2p-1,
\end{equation}
on the geodesic circle $S_x(r) = \set{\exp_x(r\, \omega) \mid \omega \in [0,2\pi]}$. The function $u$ does not vanish in these open intervals, and changes sign while crossing an end point $\exp_x(r\, \omegat_j(r))$ along the circle.  The following assertion is clear. \smallskip

\emph{Fact~3.~ Each open interval $I_j(r)$ is contained in a unique nodal domain, and two contiguous intervals are contained in different nodal domains. Each nodal domain of $u$ contains at least one interval (for $r$ small enough).}\medskip

\begin{definition}\label{D-llnd-2}
Let $\cD(u)$ be the set of nodal domains of $u$.\\[3pt] A bijection $d: \set{1,\ldots,(p+1)} \to \cD(u)$ is called a \emph{labeling of the nodal domains} \index{Labeling (nodal domains)} of $u$.  Using this labeling, we describe $\cD(u)$ as $\cD(u) = \set{\Omega_{d(1)},\ldots,\Omega_{d(p+1)}}$.
The \emph{nodal word} $\cW_{u,d}$ of $u$, associated with the labeling $d$ of $\cD(u)$, is the map $\cW_{u,d} : \set{0,\ldots,(2p-1)} \to \set{1,\ldots,(p+1)}$
\begin{equation*}
\cW_{u,d} = \begin{pmatrix}
               0 & 1 & \ldots & (2p-2) & (2p-1) \\
               a_0 & a_1 & \ldots & a_{2p-2} & a_{2p-1} \\
             \end{pmatrix},
\end{equation*}
also written as the word
\begin{equation*}
\cW_{u,d}= |a_0|a_1|\ldots |a_{(2p-1)}|,
\end{equation*}
where the letter $a_j$ is the label $d(k)$ of the nodal domain which contains $I_j(r)$,
\begin{equation*}
\cW_{u,d}(j) = d(k) \Leftrightarrow I_j(r) \subset \Omega_{d(k)}.
\end{equation*}
\end{definition}

The word $\cW(u,d)$ has length $2p$, the letters of the word are the labels of the nodal domains (they are separated by vertical bars in the second formula for notational convenience).\medskip

\emph{Fact~4.~ Given $\cW_{u,d}$ a nodal word, we recover the labeling $d$ as follows,
\begin{enumerate}
  \item $d(1) = a_0$ and $d(2) = a_1$
  \item if $m_3 := \min \set{j \mid a_j \not \in \set{d(1),d(2)}}$, then $\Omega_{d(3)}$ is the nodal domain which contains $I_{m_3}(r)$, i.e. $d(3) = a_{m_3}$.
  \item \ldots
\end{enumerate}
}
\medskip

It is convenient to describe a procedure to produce a ``standard nodal word'' $\cW_{u}$, and a ``standard labeling'' $d_s$ \index{Standard labeling} of the nodal domains of an eigenfunction $u$ satisfying Assumptions~\ref{A-llnd-0}.

\begin{definition}\label{D-llnd-4}
The \emph{standard nodal word} \index{Nodal word (standard)} $\cW_{u}$ of $u$ is the map
\[
\cW_{u} : L \to \set{1,\ldots,(p+1)}
\]
defined as follows.
\begin{enumerate}
\item Let $\cW_{u}(1)=1$. Equivalently, call $\Omega_1$ the nodal domain which contains the interval $I_0(r)$. Let  $\cW_{u}(j)=1$ whenever $I_j(r) \subset \Omega_1$.
\item Let $\cW_{u}(2)=2$. Equivalently, call $\Omega_2$ the nodal domain which contains the interval $I_1(r)$. Let  $\cW_{u}(j)=2$ whenever $I_j(r) \subset \Omega_2$.
\item Assume that $k$ nodal domains have been labeled, with the labels given in increasing order, $1, 2, \ldots,k$. Since all nodal domains intersect any neighborhood of $x$, the set  $\set{j \mid I_j(r) \not \subset  \bigcup_{j=1}^k \Omega_j}$ is not empty. Let $m_{k+1}$ be its infimum. Call
    $\Omega_{k+1}$ the nodal domain which contains $I_{m_{k+1}}(r)$. Let
    $\cW_{u}(j)=(k+1)$ whenever $I_j(r) \subset \Omega_{k+1}$.
\item After at most $p$ steps, all nodal domains will be labeled.
\end{enumerate}
\end{definition}%

\subsection{Sub-bouquets of loops}\label{SS-llnd-w}

As stated in the previous subsection, the nodal set $\cZ(u)$ is a \emph{$p$-bouquet of loops} $\cB_{L}$ at $x$, i.e., $p$ simple loops at $x$ which do not intersect away from $x$, and which meet transversally at $x$.\smallskip

Take any loop, $\gamma_{j,\tau(j)}$. Exchanging, $j$ and $\tau(j)$ if necessary, we may assume that $j < \tau(j)$. Consider the subsets
\begin{equation}\label{E-llnd-6}
\left\{
\begin{array}{l}
L_{j} := \set{j,(j+1),\ldots,(\tau(j)-1),\tau(j)} \subset L,\\[5pt]
L'_{j} := L_{j} \sm \set{j,\tau(j)}.
\end{array}
\right.
\end{equation}

Since $M$ is a sphere, $M\sm \set{\gamma_{j,\tau(j)}}$ has two components. The local structure of $\cZ(u)$ at $x$ shows that the rays $\omega_k, k \in L'_{j}$, point inside one of the components, call it $C'_{j}$. If $L'_{j}$ is empty, choose $C'_{j}$ to be the component which is a nodal domain of $u$. For $\ell \le j \le \tau(\ell)-1$, the nodal interval $I_j(r)$ is contained in $C'_{j}$, and the subsets  $L_j$, $L'_{j}$ are invariant under $\tau$, and correspond to bouquets of loops $\cB_{L_j}$ and $\cB_{L'_{j}} \subset C'_{j}$.\smallskip

Similarly, the rays $\omega_k, k \in L\sm L_{j}$, point inside the other component, call it $C''_{j}$. The corresponding nodal intervals are contained in $C''_{j}$, so that $L\sm L_{j}$ is invariant under $\tau$, and corresponds to a bouquet of loops $\cB_{L\sm L_{j}} \subset C''_{j}$. Note that if $L_j=L$, then $C''_{j}$ is a nodal domain of $u$.

\subsection{Nodal word of $u$ vs combinatorial type of $u$}\label{SS-llnd-d}

\begin{property}\label{P-llnd-2}
Once we have chosen an orientation in $T_xM$, an initial ray $\omega_0$, and labeled the other rays counter-clockwise, the combinatorial type $\tau_u : L \to L$ of $u$ determines the standard nodal word $\cW_u : L \to \set{1,\ldots,(p+1)}$ of the nodal domains of $u$. Conversely, given a standard nodal word $\cW_u$ as defined in Definition~\ref{D-llnd-4}, we can recover the nodal type $\tau_u$.
\end{property}%

\subsubsection{Proof of Property~\ref{P-llnd-2} on a simple example}
For the example, we choose $p=8$, so that $L = \set{0,1,2,3,4,5,6,7,8,9,10,11,12,13,14,15}$,
and the map $\tau$
\setcounter{MaxMatrixCols}{20}
\begin{equation}\label{E-llnd-8tau}
\tau ~=~
\begin{pmatrix}
  0 & 1 & 2 & 3 & 4 & 5 & 6 & 7 & 8 & 9 & {10} & 11 & 12 & 13 & 14 & 15\\
  3 & 2 & 1 & 0 &9  & 8 & 7 & 6 & 5 & 4 & 15 & 12 & 11 & 14 & 13 & 10\\
\end{pmatrix}
\end{equation}
written in matrix form: $\tau$ maps the first line to the second line. \smallskip

We view the pair $(L,\tau)$ as describing an abstract bouquet of loops $\cB_L$ which satisfies the properties explained in  Subsection~\ref{SS-llnd-w}, see Figure~\ref{F-w-llnd-2-bouquet} for a representation in $\R^2$. The numbers between brackets are the labels of the rays or nodal arcs emanating from $x$.
Since we may think of $\cB_L$ as the zero set of some function, we still call the components $\Omega_j$ of $M\sm \cB_L$ ``nodal domains of $\cB_L$''. \medskip

\begin{figure}[!ht]
  \centering
  \includegraphics[width=0.4\textwidth]{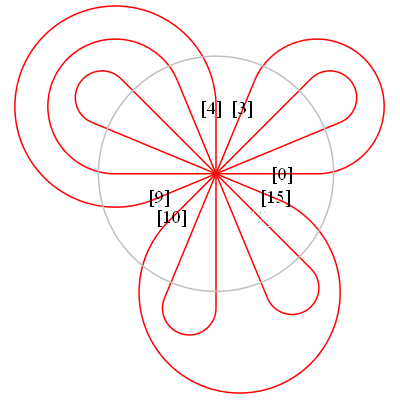}
  \caption{The bouquet $\cB_L$ with $\tau$ given by \eqref{E-llnd-8tau}}\label{F-w-llnd-2-bouquet}
\end{figure}

We first partition the set $L$ into $\tau$-invariant subsets with the same notation as in Equation~\eqref{E-llnd-6}.\smallskip

\begin{enumerate}[$\diamond$]
  \item Since $\tau(0) = 3$, we have the $\tau$-invariant subset $L_0 = \set{0,1,2,3}$, the ``big'' loop $\gamma_{0,3}$ and $C'_0$ the connected component of $M\sm \gamma_{0,3}$ which contains $\cB_{L'_0}$.
  \item Since $\tau(4) = 9$, we have the $\tau$-invariant subset $L_4 = \set{4,5,6,7,8,9}$, the ``big'' loop $\gamma_{4,9}$ and $C'_4$ the corresponding component of $M\sm \gamma_{4,9}\,$.
  \item Since $\tau(10) = 15$, we have the $\tau$-invariant subset $L_{10} = \set{10,11,12,13,14,15}$, the ``big'' loop $\gamma_{10,15}$ and $C'_{10}$ the corresponding component of $M\sm \gamma_{10,15}\,$.
\end{enumerate}

The set $L$ has been partitioned into three $\tau$-invariant subsets $L_0 = \set{0,1,2,3}$, $L_4 = \set{4,5,6,7,8,9}$ and $L_{10} = \set{10, 11, 12, 13, 14, 15}$. Accordingly, the matrix representing $\tau$ can be decomposed into three blocks,
\begin{equation*}
\tau ~=~\big(
\begin{vmatrix}
  0 & 1 & 2 & 3 \\
  3 & 2 & 1 & 0 \\
\end{vmatrix}
\begin{vmatrix}
  4 & 5 & 6 & 7 & 8 & 9 \\
  9 & 8 & 7 & 6 & 5 & 4 \\
\end{vmatrix}
\begin{vmatrix}
  10 & 11 & 12 & 13 & 14 & 15 \\
  15 & 12 & 11 & 14 & 13 & 10 \\
\end{vmatrix}
\big).
\end{equation*}

\noib \emph{Determining $\cW$ from  $\tau$}. The set $M\sm (C'_0 \cup C'_4 \cup C'_{10})$ is a nodal domain of $\cB_L$ which we call the \emph{exterior} of $\cB_L$. Any other nodal domain of $\cB_L$ is contained in either $C'_0,  C'_4$ or $C'_{10}$. It now suffices to define $\cW_u$ on each of the invariant subsets or, equivalently, to find the labels of the exterior of $\cB_L$ and of the nodal domains contained in  $C'_0,  C'_4$ or $C'_{10}$. \smallskip

We first label the nodal domains inside $C'_0$. We have $I_j(r) \subset C'_0$ for $j \in \set{0,1,2}$. According to Definition~\ref{D-llnd-4}, $\cW(0) = 1$, and $\cW(1)=2$. Following $S_x(r)$ from the point $\exp_x(r\omegat_0(r))$ to the point $\exp_x(r\omegat_3(r))$ and exiting $C'_0$, we find that $\cW(2)=1$ and $\cW(3)=3$. This implies that the exterior of $\cB_L$ is the nodal domain $\Omega_3$ and this also implies that $\cW(9)=\cW(15)=3$ because
$I_9(r)$ and $I_{15}(r)$ are both contained in the exterior of $\cB_L$.\smallskip

We now label the nodal domains inside $C'_4$. We have $I_j(r) \subset C'_4$ for $4 \le j \le 8$. According to Definition~\ref{D-llnd-4}, we set $\cW(4) = 4$. Following $S_x(r)$ from $\exp_x(r\omegat_4(r))$ to $\exp_x(r\omegat_9(r))$ and exiting $C'_4$, we conclude that $\cW(8)=4$ and $\cW(10)=7$ because $C'_4$ contains $3$ nodal domains. It remains to label the nodal domains in $L'_4$. This is similar to the first step (with different labels though) and we find that $\cW(5)=\cW(7)=5$ and $\cW(6)=6$ because we have one loop $\gamma_{6,7}$ inside the loop $\gamma_{5,8}$.\smallskip

It remains to label the nodal domains inside $C'_{10}$, starting from $\cW(10)=7$. Because there are two loops $\gamma_{11,12}$ and $\gamma_{13,15}$ inside the loop $\gamma_{10,15}$, we find that $\cW(10)=\cW(12)=\cW(14)=7$, $\cW(11)=8$ and $\cW(13)=9$.\smallskip

Finally, the map $\cW$ is given by the matrix
\begin{equation}\label{E-llnd-9d}
\cW ~=~
\begin{pmatrix}
  0 & 1 & 2 & 3 & 4 & 5 & 6 & 7 & 8 & 9 & 10 & 11 & 12 & 13 & 14 & 15 \\
  1 & 2 & 1 & 3 & 4 & 5 & 6 & 5 & 4 & 3 & 7 & 8 & 7 & 9 & 7 & 3  \\
\end{pmatrix},
\end{equation}
where $\cW$ sends the entry $j \in L$ in the first line, the label of the interval $I_j(r)$, to the entry $\cW(j)$ in the second line, the label of the nodal domain which contains $I_j(r)$. \medskip

\noib \emph{Recovering $\tau$ from $\cW$.} We use the fact that the nodal domains come in disjoint families separated by loops.\smallskip

The boundary of the domain $\Omega_1$ contains the loop $\gamma_{0,\tau(0)}$. To determine $\tau(0)$, we look at the largest integer $\ell$ such that $\cW(\ell) =1$: this is $2$ and we conclude that $\tau(0)=3$. Looking at the second row of the matrix of $\cW$ in \eqref{E-llnd-9d}, we infer that $\Omega_2$ is bounded by a single loop, so that $\tau(1)=2$.  We have determined a $\tau$-invariant subset $L_0 = \set{0,1,2,3}$, and that $\tau$ satisfies
\begin{equation*}
\tau|_{L_0} ~=~
\begin{pmatrix}
  0 & 1 & 2 & 3 \\
  3 & 2 & 1 & 0 \\
\end{pmatrix}.
\end{equation*}

The next nodal domain is $\Omega_3$ and we have $\cW(3)=\cW(9)=\cW(15)=3$. Taking into account how we constructed $\cW$ from $\tau$, it follows that $\Omega_3$ is the ``exterior'' of $\cB_L$.\smallskip

The next nodal domain which appears is $\Omega_4$. To determine $\tau(4)$, we look at the largest integer $\ell$ such that $\cW(\ell)=4$: this is $8$. This means that $\tau(4) = 9$, and we have a loop $\gamma_{4,9}$ whose complement in $M$ has two connected components $C'_4$ and $C''_4$, with the rays labeled $5$ to $8$ pointing inside $C'_4$ . We have the $\tau$-invariant subset  $L_4 = \set{4,5,6,7,8,9}$ . In $C'_4$, the nodal domain label $4$ occurs twice, the label $5$ occurs twice as well, and the label $6$ once. We conclude that
\begin{equation*}
\tau|_{L_4} ~=~
\begin{pmatrix}
  4 & 5 & 6 & 7 & 8 & 9 \\
  9 & 8 & 7 & 6 & 5 & 4 \\
\end{pmatrix}.
\end{equation*}

The two next nodal domain to appear are $\Omega_3$ again and $\Omega_7$. Reasoning as above, we find that $\tau(10) = 15$ and we have the $\tau$-invariant subset $L_{10} = \set{10,11,12,13,14,15}$. Finally,
\begin{equation*}
\tau|_{L_{10}} ~=~
\begin{pmatrix}
  10 & 11 & 12 & 13 & 14 & 15 \\
  15 & 12 & 11 & 14 & 13 & 10 \\
\end{pmatrix}.
\end{equation*}

We have recovered the map $\tau$ from the map $\cW$ in the example at hand. \hfill \qed

\subsubsection{Proof of Property~\ref{P-llnd-2} in general}~\smallskip

\noib \emph{Proof that $\tau$ determines $\cW$.}\smallskip

\noid We begin by partitioning $L$ into invariant subsets.\smallskip

First we define $\ell_1 := 0$. Then $\ell_1 < \tau(\ell_1) \le (2p-1)$ and $L_{\ell_1} = \set{\ell_1,\ldots, \tau(\ell_1)}$ is a $\tau$-invariant subset. If $\tau(\ell_1) = (2p-1)$, $M \sm \gamma_{\ell_1,\tau(\ell_1)}$ has two components, one of them $C''_{\ell_1}$ is a nodal domain of $\cB_L$, and we define $\cW(2p-1) = (p+1)$. The other component $C'_{\ell_1}$ of $M \sm \gamma_{\ell_1,\tau(\ell_1)}$ contains all the nodal domains  of $\cB_L$, except the nodal domain $\Omega_{p+1}$. If $\tau(\ell_1) \neq (2p-1)$,  $\tau(\ell_1) \le (2p-3)$, and we introduce $\ell_2 :=\tau(\ell_1)+1$. The subset $L_{\ell_2} = \set{\ell_2,\ldots,\tau(\ell_2)}$ is $\tau$-invariant, and we can repeat the procedure. After at most $|L|/2$ steps, we obtain a sequence $\ell_1 < \ldots < \ell_m$ such that $\ell_{j+1} = \tau(\ell_j)+1$ and $\tau(\ell_m)=(2p-1)$. Then, we have a partition, $L = \bigsqcup_{j=1}^m L_{\ell_j}$, of $L$. A nodal domain of $\cB_L$ is either contained
in one of the components $C'_{\ell_1}$, or is equal to the \emph{exterior nodal domain} of $\cB_L$, the set $M\sm \bigcup_{j=1}^m C'_{\ell_j}$.\smallskip

\noid Consider a loop $\gamma_{\ell_j,\tau(\ell_j)}$, with
\begin{equation}\label{E-llnd-9}
\left\{
\begin{array}{l}
L_{\ell_j} = \set{\ell_j,\ldots,\tau(\ell_j)} \text{~and~} L'_{\ell_j} = L_{\ell_j}\sm \set{\ell_j,\tau(\ell_j)},\\[5pt]
k_{\ell_j} = \frac{1}{2}|L_{\ell_j}| .
\end{array}%
\right.
\end{equation}

Taking into account Subsection~\ref{SS-llnd-w} and Lemma~\ref{L-llnd-2}, with the set $L_{\ell_j}$ we associate a $k_{\ell_j}$-bouquet of loops $\cB_{L_{\ell_j}}$, whose complement in $M$ has $(k_{\ell_j}+1)$ components. Similarly,  with the set $L'_{\ell_j}$ we associate a $(k_{\ell_j}-1)$-bouquet $\cB_{L'_{\ell_j}}$ which is contained in $C'_{\ell_j}$, and whose complement in $C'_{\ell_j}$ has $k_{\ell_j}$ components which are actually nodal domains of $\cB_L$. \smallskip

\noid The $k_{\ell_1}$ nodal domains contained in $C'_{\ell_1}$ are labeled from $1$ to $k_{\ell_1}$.  The intervals $I_j(r), \ell_1 \le j \le \tau(\ell_1)-1$ are contained in $C'_{\ell_1}$; the intervals $I_j(r), \tau(\ell_1) \le j$ are contained in $C''_{\ell_1}$. From these facts, we infer that
\begin{equation}\label{E-llnd-10}
\left\{
\begin{array}{l}
\cW(0) = \cW(\tau(\ell_1)-1) = 1,\\[5pt]
\cW(1) = 2,\\[5pt]
\cW(\tau(\ell_1)) = k_{\ell_1}+1,\\[5pt]
\cW(\tau(\ell_1)+1) = k_{\ell_1}+2.
\end{array}%
\right.
\end{equation}

The label $\cW(\tau(\ell_1))$ plays a special role. Indeed, this is the label of the exterior of $\cB_L$,  $M\sm \bigcup_{j=1}^m C'_{\ell_j}$. It follows that the nodal domains of $\cB_L$ will be labeled as follows:
\begin{enumerate}[1)]
  \item The $k_{\ell_1}$ nodal domains contained in $C'_{\ell_1}$ are labeled from $1$ to $k_{\ell_1}$.
  \item The exterior nodal domain of $\cB_L$ is labeled $(k_{\ell_1}+1)$.
  \item The $k_{\ell_2}$ nodal domains contained in $C'_{\ell_2}$ are labeled from $(k_{\ell_1}+2)$ to $(k_{\ell_1} + k_{\ell_2} + 1)$.
  \item The $k_{\ell_3}$ nodal domains contained in $C'_{\ell_3}$ are then labeled from $(k_{\ell_1}+k_{\ell_2}+2)$ to $(k_{\ell_1} + k_{\ell_2} + k_{\ell_3}+1)$.
  \item \ldots ~and so on.
\end{enumerate}

Once we know the label of the exterior domain, and which label sets to use for the domains contained in the sets $C'_{\ell_j}$, it suffices to determine $\cW$ on each set $L_{\ell_j}$ independently, and we can reason by induction on the size of $|L|$.\smallskip

\noid  Given $\ell \in \set{\ell_1,\ldots,\ell_m}$, the  connected component $C'_{\ell_1}$ of $M \sm \gamma_{\ell,\tau(\ell)}$  is simply connected, with boundary $\gamma_{\ell,\tau(\ell)}$. The nodal domains inside $C'_{\ell}$ are numbered from $K_{\ell}$ to $K_{\ell} + k_{\ell}-1$, according to the above list. The interval $I_{\ell}(r)$ is the first interval contained in $C'_{\ell}$ to be labeled, $\cW(\ell) = K_{\ell}$, and we must have $\cW(\ell+1) = (K_{\ell}+1)$. Since $|L_{\ell}| < |L|$ we can know apply the induction hypothesis, and $\cW|_{L_{\ell}}$ is well defined.\medskip

\noib \emph{Proof that one can recover $\tau$ from $\cW$ in general.}\smallskip

\noid Assume that we are given some $L := \set{0,\ldots,(2p-1)}$ and some map $\cW : L \to \set{1,\ldots,(p+1)}$ associated with a $p$-bouquet of loops $\cB_L$ as in Definition~\ref{D-llnd-4}. Let $\tau : L \to L$ be the combinatorial type of the $p$-bouquet $\cB_L$.\smallskip

In order to recover $\tau$ from $\cW$, the idea is to recover the $\tau$-invariant subsets $L_{\ell_j}$ introduced above, and to reason by induction on the size of the bouquets, i.e., on $|L|$.\medskip

\noid We first look at the nodal domain $\Omega_1$, with label $\cW(0)=1$, and at the set $\cW^{-1}(1)$.\smallskip

If $\cW^{-1}(1) = \set{0}$, then we must have $\tau(0)=1$, the loop $\gamma_{0,1}$ bounds $\Omega_1$, $\tau(0)=1$, and $\Omega_2$ is the exterior domain of $\cB_L$.\smallskip

If $|\cW^{-1}(1)| > 1$, we look at $m_1 := \max \cW^{-1}(1)$. Then, the only possibility is that $\tau(0) = m_1+1$. According to Subsection~\ref{SS-llnd-w}, the subsets $L_0 := \set{0,\ldots, (m_1+1)}$, $L'_0:=\set{1,\ldots,m_1}$, and $L\sm L_0$ are $\tau$ invariant. Defining $C'_0$ and $C''_0$ as above, there are exactly $k_0$ nodal domains inside $C'_0$, where $2k_0 = |L_0|$, $\cW(m_1+1) = (k_0+1)$, and the domain $\Omega_{\cW(m_1+1)}$ is the exterior domain of $\cB_L$. Looking at $\cW^{-1}(k_0+1)$, we obtain the partition of $L$ into $\tau$-invariant subsets which we used to deduce the word $\cW$ from the combinatorial type $\tau$.\smallskip

Example: If we look back at the example given by Equation~\eqref{E-llnd-8tau} and at the corresponding map $d$ given by Equation~\eqref{E-llnd-9d}, we find that $m_1 = 2$, and that $k_1+1 = 3$, and we recover the fact that
\[L = \set{0,\ldots,3} \sqcup \set{4,\ldots,9} \sqcup \set{10,\ldots,15}\]
and the fact that $\Omega_3$ is the exterior domain of $\cB_L$.\medskip

\noid To conclude in the general case, it suffices to determine $\tau$ in each of the invariant subsets, so that we can now use an induction argument on $|L|$.

\chapter{Plane Domains: the Estimate $\mult(\lambda_k) \le (2k-2)$ for $k \ge 3$}\label{Ch-pdwb}

Let $\Omega$ be a regular  (i.e.,  $C^{\infty}$) bounded domain\footnote{By ``domain'', we mean a connected open subset.} in $\R^2$. We are interested in the eigenvalue problem for the Laplacian or for a Schr\"{o}dinger operator of the form $-\Delta + V$ in $\Omega$, with Dirichlet, Neumann or $h$-Robin boundary condition, see \eqref{E-evp-2bc}.   As indicated in the introduction, we do not consider the Steklov problem. In this chapter, we prove the following result.

\begin{theorem}\label{T-hmn-bh1}
The multiplicities of the eigenvalues of the operator  $-\Delta + V$ in $\Omega$, with the Dirichlet, Neumann or $h$-Robin boundary condition, satisfy $\mult(\lambda_k) \le (2k-2)$ for any $k \ge 3$.
\end{theorem}%

In the next chapter, we will discuss the proof of  the sharper estimate $\mult(\lambda_k) \le (2k-3)$ for any $k \ge 3$, under the additional assumption that $\Omega$ is \emph{simply connected}, and relate this estimate to \cite[Theorem~A]{HoMN1999}.

\section{Bounding $\mult(\lambda_k)$ from Above}\label{S-hmn}

\subsection{Introduction}\label{SS-hmn-i}

As in \cite{HoMN1999}, our proof of Theorem~\ref{T-hmn-bh1} consists of three steps.
 \begin{itemize}
 \item The first step is to prove the upper bound $\mult(\lambda_k) \le (2k-1)$ for all $k \ge 1$. This upper bound follows easily from Courant's nodal domain theorem, Theorem~\ref{T-RC}, and Euler's formula for nodal sets, see Subsection~\ref{SS-hmn-0B}. It is sharp for $k=1$ since $\mult(\lambda_1) = 1$  for any domain. The inequality $\mult(\lambda_2) \le 3$ turns out to be sharp either. This is related to the \emph{nodal line conjecture},
   see Section~\ref{S-nlc}.
\item In a  second step, we prove that $\mult(\lambda_k)$ cannot be equal to $(2k-1)$ for $k\ge 3$. This is done in Section~\ref{S-hmn2}, under the simplifying assumption that $\Omega$ is simply-connected,  and in Section~\ref{S-hmn2N} in the general case.
\end{itemize}

\begin{remark}\label{R-hmn-0}
As far as we know, for the Neumann and $h$-Robin boundary condition, the upper bound on the eigenvalue multiplicities given in Theorem~\ref{T-hmn-bh1}  is new.
\end{remark}%

\subsection{Notation}\label{SS-hmn-0N}

Let us fix some notation for   Sections~\ref{S-hmn}--\ref{S-hmn9}.\smallskip

Let $U$ denote a linear subspace of an eigenspace of the eigenvalue problem for $-\Delta + V$ in $\Omega$, see \eqref{E-evp-2bc}. We assume that $U$ satisfies the inequality
\begin{equation}\label{E-hmn0-2}
\max \set{\kappa(u) \mid 0 \neq u \in U} \le \ell ,
\text{~~for some integer~} \ell \ge 2.
\end{equation}

 Denote $\partial \Omega$ by $\Gamma$, with $q = b_0(\Gamma)$. Write $\Gamma$ as the union of its components
\[
\Gamma = \bigcup_{j=1}^q \Gamma_j,\mbox{ with } q \ge 1.
\]

Given $0 \neq u \in U$, define the sets \index{2-J@$J(u)$} \index{1-Gamma@$\Gamma$!$\Gamma(u)$}
\begin{equation}\label{E-hmn0-4}
\left\{
\begin{array}{ll}
J(u) & := \set{j \mid \Gamma_j \cap \cZ(u) \neq \emptyset},\\[5pt]
\Gamma(u) & := \cup_{j \in J(u)} \Gamma_j.
\end{array}
\right.
\end{equation}

Given a function $0 \neq u \in U$, $[u]$ denotes the line \index{3-[@$[u]$}
\begin{equation}\label{E-hmn0-5}
[u] := \set{a u \mid a \in \R\sm \set{0}}
\end{equation}
in the projective space $\bP(U)$. We say that $u$ is a generator of the line $[u]$. If a function $u$ is uniquely determined by some condition, up to multiplication by a nonzero scalar, we say that $u$ is uniquely determined \emph{up to scaling} or, equivalently, that $[u]$ is uniquely determined.

\subsection{The initial inequalities}\label{SS-hmn-0B}

We shall make an extensive use of Euler's formula for the nodal set $\cZ(u)$ of an eigenfunction $u$, see Subsection~\ref{SS-etf}. Taking into account the assumption \eqref{E-hmn0-2} on $U$, we have
\begin{equation}\label{E-euler-or}
\ell \ge \kappa(u) = 1 + \beta(u) + \sigma_{\mathrm{i}}(u) + \sigma_{\mathrm{b}}(u),
\end{equation}
where,
\begin{equation}\label{E-hmn0-8}
\left\{
\begin{array}{ll}
\beta(u) & := b_0(\cZ(u) \cup \Gamma) - b_0(\Gamma) \\[5pt]
& ~= b_0\big(\cZ(u) \bigcup \Gamma(u)\big) - b_0(\Gamma(u)),
\end{array}
\right.
\end{equation}
\begin{equation}\label{E-hmn0-10}
\left\{
\begin{array}{ll}
\sigma_{\mathrm{i}}(u) &= \frac 12\, \sum_{z \in \cS_{\mathrm{i}}(u)} \big( \nu(u,z) - 2 \big),\\[5pt]
\sigma_{\mathrm{b}}(u) &= \frac 12\, \sum_{z \in \cS_{\mathrm{b}}(u)} \rho(u,z) = \sum_{j \in J(u)} \, \frac 12 \, \sum_{z\in \cS_{\mathrm{b}}(u) \cap \Gamma_j\,} \rho(u,z).\\[5pt]
\end{array}
\right.
\end{equation}

From the nodal character, using Proposition~\ref{P-euler-np} and the definition of $J(u)$, we also have
\begin{equation}\label{E-hmn0-12}
\forall j \in J(u), \quad \sum_{z \in \cS_{\mathrm{b}}(u)\cap \Gamma_j}\rho(u,z) \text{~~is even and~} \ge 2.
\end{equation}

We now rewrite  the Euler inequality \eqref{E-euler-or} in the form,
\begin{equation}\label{E-euler-or2}
\left\{
\begin{array}{ll}
0 \geq  \kappa(u) - \ell  \,  = &\big[ b_0(\cZ(u) \cup \Gamma(u))-1 \big]+ \frac 12 \sum_{z \in \cS_{\mathrm{i}}(u)} (\nu(u,z)-2)\\[5pt]
& + \sum_{j \in J(u)\,} \frac 12 \, \big( \sum_{z\in \cS_{\mathrm{b}}(u) \cap
\Gamma_j\,} \rho(u,z) \, - 2 \big) - (\ell - 2).
\end{array}
\right.
\end{equation}

We will apply this inequality to  eigenfunctions with prescribed singular points.\medskip

Fix some $x \in \Gamma_1$, and let $m := \dim U$. By Lemma~\ref{L-zero1}, there exists $0 \neq u \in U$ such that $\rho(u,x) \ge (m-1)$. Rewrite \eqref{E-euler-or2} as,
\begin{equation}\label{E-euler-or2a}
\left\{
\begin{array}{ll}
0 \geq  \kappa(u) - \ell  \,  & = \big[ b_0(\cZ(u) \cup \Gamma(u))-1 \big] + \frac 12 \sum_{z \in \cS_{\mathrm{i}}(u)} (\nu(z)-2)\\[5pt]
& ~~ + \sum_{j \in J(u), j\not = 1\,}  \frac 12 \, \big( \sum_{z\in \cS_{\mathrm{b}}(u) \cap
\Gamma_j\,} \rho(u,z) \, - 2 \big)\\[5pt]
& ~~ + \frac 12 \, \sum_{z\in \cS_{\mathrm{b}}(u) \cap \Gamma_1\,} \rho(u,z) \, - \ell + 1. \end{array}
\right.
\end{equation}
The first three terms in the right-hand side of the equality are nonnegative. It follows that
\begin{equation*}
2\ell - 2 \ge \sum_{z\in \cS_{\mathrm{b}}(u) \cap \Gamma_1\,} \rho(u,z) \ge m-1,
\end{equation*}
so that
\begin{equation}\label{E-hmn0-14}
\dim U = m \le (2\ell - 1).
\end{equation}

Courant's nodal domain theorem states that $\max \set{\kappa(u) \mid 0 \neq u \in U(\lambda_k)} \le k$. Cho\-osing $U = U(\lambda_k)$, inequality \eqref{E-hmn0-14} yields the following estimate.

\begin{proposition}[\cite{Nadi1987}, Theorem~2]\label{P-hmn-s1}
Let $\Omega \subset \R^2$ be a bounded domain with smooth boundary. Let $\set{\lambda_k, k\ge 1}$ be the eigenvalues of the operator $-\Delta + V$ in $\Omega$, with Dirichlet or Robin boundary condition. Then, for any $k \ge 1$,
\begin{equation*}
\mult(\lambda_k) \le (2k-1).
\end{equation*}
\end{proposition}%

In view of Proposition~\ref{P-hmn-s1}, in order to prove Theorem~\ref{T-hmn-bh1}, it suffices to show that the equality $\dim U(\lambda_k) = (2k-1)$ cannot occur for $k\ge 3$. This is the purpose of  Sections~\ref{S-hmn2} and \ref{S-hmn2N} in which we revisit and extend the arguments of \cite{HoMN1999} for the three boundary conditions \eqref{E-evp-bc}.

\begin{remark}\label{R-hmn-2}
According to Pleijel \cite{Plej1956}, \index{Pleijel Theorem} Courant's Theorem~\ref{T-RC} \index{Courant nodal Theorem} is sharp for finitely many Dirichlet eigenvalues only.  The eigenvalue $\lambda_k$ is called  \emph{Courant-sharp} \index{Courant-sharp} whenever the associated eigenspace $U(\lambda_k)$ contains an eigenfunction with $k$ nodal domains, the maximum number allowed by Courant's nodal domain theorem. If $\lambda_k$ is not a Courant-sharp eigenvalue, we have $\max \set{\kappa(u) \mid 0 \neq u \in U(\lambda_k)} \le (k-1)$ and hence $\mult(\lambda_k) \le (2k-3)$, improving the inequality in Proposition~\ref{P-hmn-s1} by $2$.  We refer to Section~\ref{S-mcs} for more details and references on Courant-sharp eigenvalues, and results \`{a} la Pleijel.
\end{remark}%

\begin{remark}\label{R-hmn-4}
In the forthcoming sections, under the assumptions that $\ell = k\ge 3$ and $\mult(\lambda_k) = (2k-1)$ or $\mult(\lambda_k) = (2k-2)$,  we prescribe eigenfunctions $u$ with a singular set $\cS(u)$ such that equality holds in \eqref{E-euler-or2a}, implying that $\kappa(u) = k$,  and hence that $\lambda_k$ is Courant-sharp. We actually do not use this information in the proofs, but rather carefully analyze the nodal sets $\cZ(u)$ to reach a topological contradiction.
\end{remark}%

\section{$\Omega$ Simply Connected: the Estimate $\mult(\lambda_k) \le (2k-2)$ for $k\geq 3$}\label{S-hmn2}

\subsection{Introduction}\label{SS-hmn20}

In this section, we provide detailed proofs of the statements in \cite[Section~2]{HoMN1999}. The general idea is to prove that an a priori upper bound on the number of nodal domains of eigenfunctions in a given subspace $U$ implies an upper bound on $\dim U$. Indeed, the bigger the dimension of $U$, the easier to construct eigenfunctions with prescribed high order singular points  and, by Euler's formula, with more nodal domains.\smallskip

The inequality in the title is valid for \emph{any} smooth bounded domain $\Omega \subset \R^2$. Note that it is not true for $k=1$ and $k=2$. In order to simplify the presentation, we shall however give the proof under \emph{the additional assumption that $\Omega$ is simply connected}, see Proposition~\ref{P-hmn-s2}.  In Remark~\ref{R-hmn-final}, we explain how to deal with the general case,  referring to Section~\ref{S-hmn2N} for complete proofs. \smallskip

The proof of the inequality is by contradiction. Taking Proposition~\ref{P-hmn-s1} into account, we assume that $\dim U(\lambda_k) = (2k-1)$ for some $k \ge 3$, and reach a contradiction.  In this section, we make the following assumptions.

\begin{assumptions}\label{A-hmn2-0}~
\begin{enumerate}[i)]
  \item $\Omega$ is simply connected.
  \item $U$ is a linear subspace of an eigenspace $U(\lambda)$ of $-\Delta + V$ in $\Omega$, with Dirichlet or Robin boundary condition, see \eqref{E-evp-2bc}.
  \item For some $\ell \ge 2$,
\begin{equation*}
\left\{
\begin{array}{l}
\sup\set{\kappa(u) \mid 0 \neq u \in U} \le \ell \text{~~and}\\[5pt]
\dim U = (2\ell - 1).
\end{array}%
\right.
\end{equation*}
\end{enumerate}
\end{assumptions}%

\subsection{Eigenfunctions with two  prescribed  boundary singular points}\label{SS-hmn-21}

We use the notation of Subsection~\ref{SS-hmn-0N}, and work  under Assumptions~\ref{A-hmn2-0}.\smallskip

For $y \neq z \in \Gamma$, we introduce the subspace\index{2-V@$V_{y,z}$}
\begin{equation*}
V_{y,z} := \set{u \in U \mid \rho(u,y) \ge (2\ell-3) \text{~and~} \rho(u,z) \ge 1}.
\end{equation*}

According to Lemma~\ref{L-zero2},  $V_{y,z} \neq \set{0}$. The purpose of this subsection is to investigate the properties of the functions $u \in V_{y,z}$, precise order of vanishing, and structure of their nodal sets.\smallskip

\subsubsection{Properties of $V_{y,z}$}\label{SSS-hmn-21a}

\begin{lemma}\label{L-Uxy0}
Assume that $\Omega$ is simply connected. Let $U$ be a linear subspace of an eigenspace of $-\Delta + V$ in $\Omega$,  such that $\sup\set{\kappa(u) \mid 0 \not = u \in U}\le \ell$ for some $\ell \ge 2$, and $\dim U = (2\ell-1)$. Let $y \neq z \in \Gamma$. The subspace
\[
V_{y,z} := \set{u \in U \mid \rho(u,y) \ge (2\ell - 3) \text{~and~} \rho(u,z) \ge 1}.
\]
has the following properties.
\begin{enumerate}[(i)]
  \item $\dim V_{y,z} = 1$ and, for any $0 \neq u \in V_{y,z}\,$;
  \item 
            $\cS_{\mathrm{i}}(u) = \emptyset$ and $\cS_{\mathrm{b}}(u) = \set{y,z}$;
  \item $\rho(u,y)=(2\ell-3)$ and $\rho(u,z) = 1$;
  \item $\kappa(u) = \ell$;
  \item the set $\cZ(u) \cup \Gamma$ is connected.
\end{enumerate}
A generator of $V_{y,z}$ will be denoted by $v_{y,z}$ (defined up to scaling).
\end{lemma}%

\begin{proof} For simplicity, in the proof, we write $\nu(z)$ for $\nu(u,z)$, \ldots.\smallskip

The fact that $\dim V_{y,z} \ge 1$ follows from Lemma~\ref{L-zero2}.
In view of our assumptions, for any $0 \neq u \in U$, Euler's formula \eqref{E-euler-or2a} gives,
\begin{equation}\label{E-LA5}
\begin{array}{ll}
0 \geq  \kappa(u) - \ell  \,  = &\big( b_0(\cZ(u) \cup \Gamma)-1\big)+ \frac 12 \sum_{x \in \cS_{\mathrm{i}}(u)} (\nu(x)-2)\\[5pt]
& \quad +
\frac 12 \sum_{x\in \cS_{\mathrm{b}}(u),\, x\neq y,z\,}\rho (x)  + \frac 12 \big( \rho(y) + \rho(z) -2\ell +2 \big).
 \end{array}
\end{equation}
Each term in the right-hand side of the equality being nonnegative, the inequality implies that each term is zero, thus proving Assertions~(ii)--(v). \smallskip

To prove the first assertion, assume that there exist two linearly independent functions $u_1$ and $u_2$ in $U$. By Assertion~(iii) they both satisfy $\rho(u_i,y) = (2\ell - 3)$ and $\rho(u_i,z) = 1$. Applying Lemma~\ref{L-zeroc} at the point $z$, we find a nontrivial linear combination $\tilde u$ of $u_1$ and $u_2$ such that $\rho(\tilde u,z) \ge 2$ and  $\rho(\tilde u,y) \ge (2\ell - 3)$, contradicting Assertion~(iii).
\end{proof}

\begin{figure}[!ht]
\centering
\begin{subfigure}[t]{.30\textwidth}
\centering
\includegraphics[width=\linewidth]{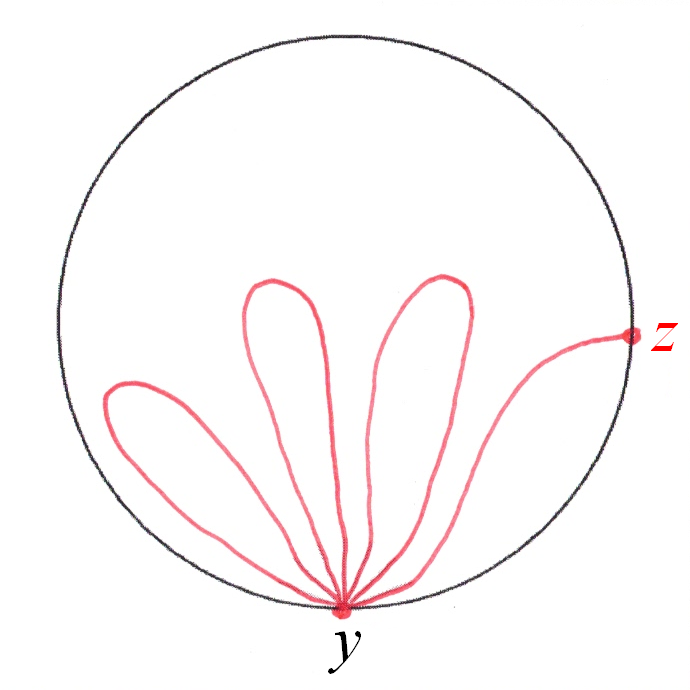}
\caption{}
\end{subfigure}
\begin{subfigure}[t]{.30\textwidth}
\centering
\includegraphics[width=\linewidth]{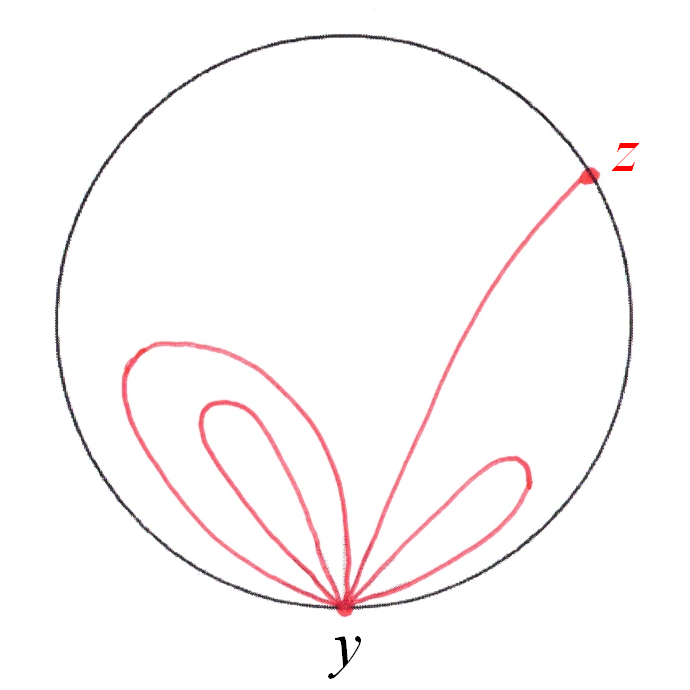}
\caption{}
\end{subfigure}
\begin{subfigure}[t]{.30\textwidth}
\centering
\includegraphics[width=\linewidth]{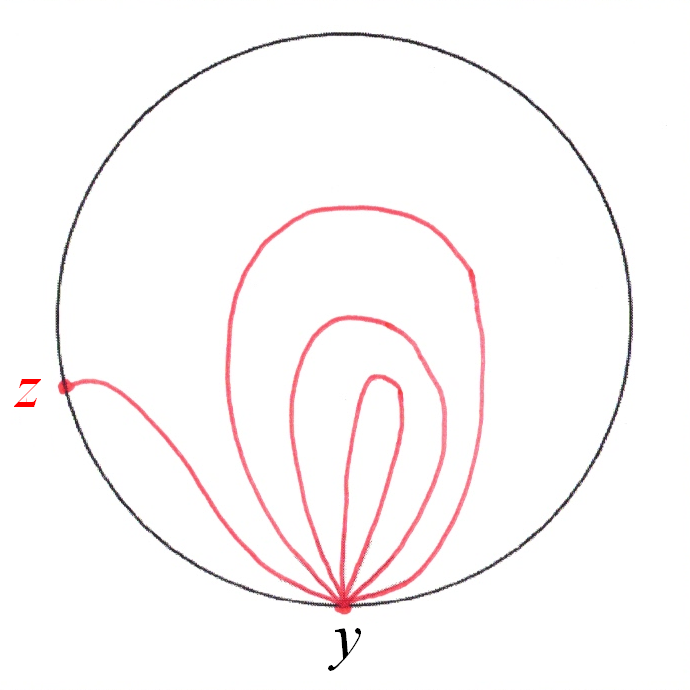}
\caption{}
\end{subfigure}
\caption{$\Omega$ simply connected, $\ell = 5$,  some possible nodal patterns for $v_{y,z}$}\label{F-hmn2-Uxy0}
\end{figure}

\begin{definition}\label{D-hmn-np}
By a \emph{nodal pattern}, \index{Nodal pattern}  we mean a nodal set, up to  continuous deformations under which singular points may move, but neither appear nor disappear (singular points  occur when  the nodal set has self-intersections, or when it hits the boundary).
\end{definition}%

\begin{remark}\label{R-Fgene}
The \emph{nodal patterns} displayed in Figure~\ref{F-hmn2-Uxy0} are valid for both the Dirichlet and Robin boundary conditions. Unless otherwise stated this remark applies to all figures of this section.
\end{remark}

\subsubsection{Structure and combinatorial type of nodal sets in $V_{y,z}$}\label{SSS-hmn-21as}

Fix $y \neq z \in \Gamma$.  Under the assumptions of Lemma~\ref{L-Uxy0}, the nodal set of an eigenfunction $u \in V_{y,z}$ can be described as follows, using the notation of Section~\ref{S-lsbs}. \smallskip

For $r_0$ small enough, in the neighborhood $D_{+}(y,r_0)$ of $y$, the nodal set $\cZ(u)$  consists of $(2\ell - 3)$ nodal semi-arcs $\delta_j$ emanating from $y$, tangentially to the rays $\omega_j, 1 \le j \le (2\ell - 3)$ (they actually depend on $y,z$ as well). Choosing any $j$, we follow the nodal semi-arc $\delta_j$ along $Z(u)$, until we reach a singular point of $u$. Otherwise stated, we consider the  component of $\cZ(u) \sm \cS(u)$ which contains the semi-arc $\delta_j$ (by abuse of notation we also denote this  component by $\delta_j$). This is a nodal interval one of whose end points is $y$. Since $\cS(u) = \cS_{\mathrm{b}}(u) = \set{y,z}$, the other end point is either $y$ or $z$.
More precisely, in view of the general properties of nodal sets, we define a map
\begin{equation}\label{E-hmn-nt}
\tau_{y,z} : \set{\downarrow} \cup L_{(2k-3)} \to \set{\downarrow} \cup L_{(2k-3)}
\end{equation}
as follows (recall that $L_{m} = \set{1,\ldots,m}$).
\begin{enumerate}[(i)]
  \item There exists a unique element $a^z_y \in L_{(2k-3)}$ such that starting from $y$ along $\delta_{a^z_y}$, we reach the boundary $\Gamma$ at $z$. We let $\tau_{y,z}(\downarrow) = a^z_y$ and $\tau_{y,z}(a^z_y) = \, \downarrow\,$.
  \item For $j \in L_{(2k-3)}\sm \set{a^z_y}$, following $\delta_j$, we arrive back at $y$, along another nodal semi-arc, which we denote by $\delta_{\tau_{y,z}(j)}$; this semi-arc emanates from $y$ tangentially to the ray $\omega_{\tau_{y,z}(j)}$. This defines $\tau_{y,z}$ on $L_{(2k-3)}\sm \set{a^z_y}$. The local structure theorem implies that for $j \in L_{(2k-3)}\sm \set{a^z_y}$, $\tau_{y,z}(j) \in L_{(2k-3)}\sm \set{a^z_y}$, and $\tau_{y,z}(j)\neq j$.
\end{enumerate}

Doing so, we obtain a uniquely defined map $\tau_{y,z}$ from $\set{\downarrow} \cup L_{(2\ell-3)}$ to itself, such that $(\tau_{y,z})^2 = \id$, and  $(\tau_{y,z})(j) \neq j$.\smallskip

The pair $\set{\downarrow,a^z_y}$ corresponds to the nodal interval $\delta_{a^z_y}$ from $y$ to $z$. For an index $j \in L_{(2k-3)}\sm \set{a^z_y}$, the pair $\set{j,\tau_{y,z}(j)}$ corresponds to a loop $\gamma^{y,z}_{j,\tau_{y,z}(j)}$ at $y$.  There are $(\ell - 2)$ such loops.  Since $\cS(u) = \set{y,z}$ and $\rho(u,z) = 1$, these loops and arc do not intersect away from $y$. Since $\cZ(u) \cup \Gamma$ is connected, the nodal set $\cZ(u)$ is actually the union of these $(\ell - 2)$ loops and arc. Otherwise stated, $\cZ(u)$ is the wedge sum $\cB_{y,(\ell-2)}^{z}$ of the simple arc $\delta_{\tau_{y,z}(\downarrow)}$ from $y$ to $z$ with an $(\ell-2)$-bouquet of loops at $y$. By analogy with Paragraph~\ref{SSS-h2n-s2c}, we give the following definition.

 \index{1-tau@$\tau$!$\tau_{y,z}$}
\begin{definition}\label{D-hmn-nt}
The map $\tau_{y,z}$  is called the \emph{combinatorial type} \index{Combinatorial type} of the eigenfunction $u \in V_{y,z}$ (or of the nodal set $\cZ(u)$) with respect to the points $y$ and $z$.
\end{definition}%

We describe the map $\tau_{y,z}$ in matrix form as
\begin{equation}\label{E-hmn-21b2}
\resizebox{.88 \textwidth}{!}
{$
\tau_{y,z} = \begin{pmatrix}
                   \downarrow & 1 & \ldots & (a^z_y-1) & a^z_y & (a^z_y+1) & \ldots & (2\ell -3) \\
                   a^z_y & \tau_{y,z}(1) & \ldots & \tau_{y,z}(a^z_y-1) & \downarrow & \tau_{y,z}(a^z_y+1) & \ldots & \tau_{y,z}(2\ell - 3) \\
                 \end{pmatrix}.
$}
\end{equation}

In the sequel, we  skip the sub- or super-scripts whenever the context is clear.  Figure~\ref{F-hmn2-Uxy0} displays some possible nodal patterns (for $\ell = 5$, $\rho(y)=7$, and $\rho(z) =1$).  The corresponding combinatorial types are given respectively by
\begin{equation*}
\resizebox{.97 \textwidth}{!}
{$
\tau_A = \begin{pmatrix}
           \downarrow & 1 & 2 & 3 & 4 & 5 & 6 & 7 \\
           1 & \downarrow & 3 & 2 & 5 & 4 & 7 & 6 \\
               \end{pmatrix},\quad
\tau_B = \begin{pmatrix}
           \downarrow & 1 & 2 & 3 & 4 & 5 & 6 & 7 \\
           3 & 2 & 1 & \downarrow & 7 & 6 & 5 & 4 \\
               \end{pmatrix},\quad
\tau_C = \begin{pmatrix}
           \downarrow & 1 & 2 & 3 & 4 & 5 & 6 & 7 \\
           7 & 6 & 5 & 4 & 3 & 2 & 1 & \downarrow \\
              \end{pmatrix}.
$}
\end{equation*}

\subsection{Eigenfunctions with one prescribed   boundary singular point}\label{SS-hmn-22}

We use the notation of Subsection~\ref{SS-hmn-0N}, and work under Assumptions~\ref{A-hmn2-0}.\smallskip

For $y \in \Gamma$, we introduce the subspaces
\index{2-U@$U_y$!$U^1_y, U^1_y$}
\begin{equation}\label{E-bha4-2}
\left\{
\begin{array}{l}
U^1_y = \set{u \in U \mid \rho(u,y) \ge (2\ell -2)},\\[5pt]
U^2_y = \set{u \in U \mid \rho(u,y) \ge (2\ell -3)}.\\[5pt]
\end{array}%
\right.
\end{equation}

According to Lemma~\ref{L-zero2},  $U^1_y \neq \set{0}$. The purpose of this subsection is to investigate the properties of the functions $u \in U^1_y$ or $U^2_y$ -- their precise order of vanishing, the structure of their nodal sets -- under Assumptions~\ref{A-hmn2-0}.\smallskip

\subsubsection{Properties of $U^1_y$ and $U^2_y$}\label{SSS-hmn-22a}

\begin{lemma}\label{L-Ux1}
Assume that $\Omega$ is simply connected. Let $U$ be a linear subspace of an eigenspace of $-\Delta + V$ in $\Omega$, such that $\sup\set{\kappa(u) \mid 0 \not = u \in U} \le \ell$ for some $\ell \ge 2$,  and $\dim U = (2\ell - 1)$.
Fix some $ y \in \Gamma$.  The subspaces
\begin{equation*}
\left\{
\begin{array}{l}
U^1_y = \set{u \in U \mid \rho(u,y) \ge (2\ell -2)}\\[5pt]
U^2_y = \set{u \in U \mid \rho(u,y) \ge (2\ell -3)}\\[5pt]
\end{array}%
\right.
\end{equation*}
have the following properties.
\begin{enumerate}[(i)]
  \item $\dim U^1_y = 1$, \, $\dim U^2_y = 2$ and,
  \item  for any $0 \neq u \in U^2_y$,
  \begin{equation}\label{E-bha4-2a}
\left\{
\begin{array}{l}
\kappa(u) = \ell \text{~~and~~} \cZ(u) \cup \Gamma \text{~is connected},\\[5pt]
\cS_{\mathrm{i}}(u) = \emptyset,\\[5pt]
\sum_{z \in \cS_{\mathrm{b}}(u)} \rho(u,z) = (2\ell - 2) \text{~and, more precisely,}\\[5pt]
\begin{array}{ll}
\text{(a) either} & \rho(u,y) = (2\ell - 2) \text{~and~} \cS_{\mathrm{b}}(u) = \set{y},\\[5pt]
\text{(b) or} & \rho(u,y)  = (2\ell-3),\, \exists\, z_u \in \Gamma \sm \set{y}, \\[5pt]
& \quad \text{~with~} \rho(u,z_u) = 1 \text{~and~}  \cS_{\mathrm{b}}(u)= \set{y,z_u}.
\end{array}%
\end{array}%
\right.
\end{equation}
\end{enumerate}
\end{lemma}%

\begin{proof} Clearly, $\set{0} \neq U_y^1 \subset U_y^2$. Take any $0 \not = u \in U^2_y$. Euler's formula \eqref{E-euler-or2a} can be rewritten as
\begin{equation}\label{E-bha4-2e}
\begin{array}{ll}
0 \ge \kappa(u) - \ell =& \big( b_0(\cZ(u) \cup \Gamma) - 1 \big)
+ \frac 12\, \sum_{x \in \cS_{\mathrm{i}}(u)} (\nu(x)-2)\\[5pt]
& +\, \frac 12\, \big( \sum_{x\in \cS _b(u)\,}\rho (x)  - 2\ell +2\big).
\end{array}%
\end{equation}
The first two  terms in the right-hand side of the equality are nonnegative. Since the last term $\sum_{x\in \cS _b(u)\,}\rho (x)$ is even, and larger than or equal to $(2\ell -3)$, the last term is nonnegative too. In view of the first inequality, the three terms must vanish. This proves the relations \eqref{E-bha4-2a}.
\medskip

\noid \emph{Proof that $\dim U^1_y = 1$.~}  Assume that this is not the case. Then, there exist two linearly independent functions $u_1, u_2$ in $U^1_y$ such that $\rho(u_i,y) = (2\ell - 2)$. By Lemma~\ref{L-zeroc}, there would exist a nontrivial linear combination $u$ such that $u \in U^1_y$ and $\rho(u,y) \ge (2\ell - 1)$, a contradiction with \eqref{E-bha4-2a}. \medskip

\noid \emph{Proof that $\dim U^2_y = 2$.~} Choose some $0 \not = v_1 \in U^1_y$. Clearly $v_1 \in U^2_y$. On the other-hand, given any $z \in \Gamma\sm\set{y}$, Lemma~\ref{L-Uxy0} provides a function $v_{y,z}$ belonging to $U^2_y$, not to $U^1_y$, and hence $\dim U^2_y \ge 2$. Choose $0 \not = v_2 \in U^2_y$ orthogonal to $v_1$. Since $\dim U^1_y = 1$, the function $v_2$ satisfies $\rho(v_2,y) = (2\ell - 3)$, and by Proposition~\ref{P-euler-np}, there must exist some $z_2 \in \Gamma$ such that $\rho(v_2,z_2) \ge 1$.  By Lemma~\ref{L-Uxy0}, $\rho(v_2,z_2) = 1$ and $v_2 \in V_{y,z_2}$. The subspace $U^{1,\perp}_y := \set{u \in U^2_y \mid u \perp u_1}$ has dimension at least one. Assume that $\dim U^2_y \ge 3$. Then $\dim U^{1,\perp}_y \ge 2$, and we can find two linearly independent functions $u_1, u_2 \in U^{1,\perp}_y$ such that $\rho(u_i,y) = (2\ell - 3)$. By Lemma~\ref{L-zeroc}, there exists a linear combination $u \in U^{1,\perp}_y$ such that $\rho(u,y) \ge (2\ell - 2)$, a contradiction.  \end{proof}

\begin{remark}\label{R-Ux1-2}
Up to scaling, there is a uniquely defined orthogonal basis $\set{v_1,v_2}$ of $U^2_y$, with $v_1 \in U^1_y$, $v_2 \in  U^{1,\perp}_y$, and a uniquely defined $z_2 \in \Gamma \sm \set{y}$ such that $\rho(v_2,z_2) = 1$. In view of Lemma~\ref{L-breve}, we can choose $v_1$ such that $\breve{v}_1 > 0$ on $\Gamma \sm \set{y}$, and $v_2$ such that $\breve{v_2} > 0$ on the arc $\cA(y,z_2)$ from $y$ to $z_2$ moving counter-clockwise on $\Gamma$.
\end{remark}%

\begin{figure}[!ht]
\centering
\begin{subfigure}[t]{.30\textwidth}
\centering
\includegraphics[width=\linewidth]{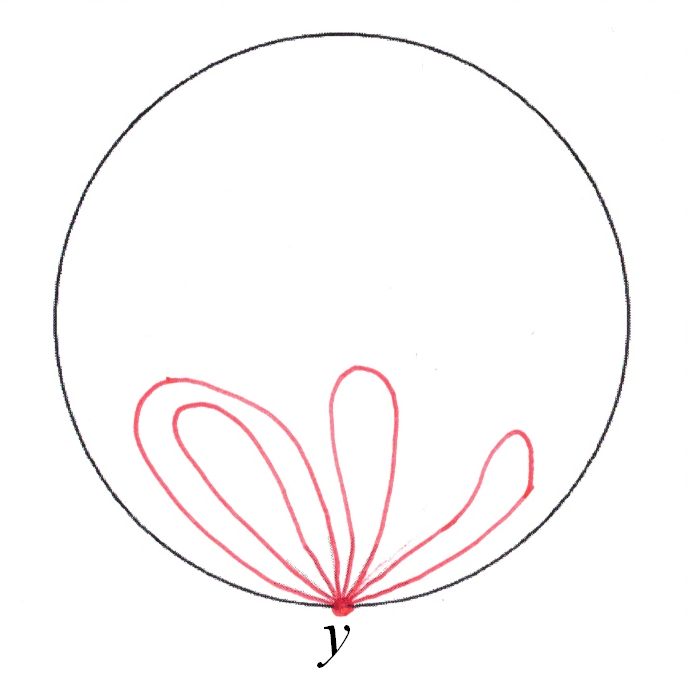}
\caption{}
\end{subfigure}
\begin{subfigure}[t]{.30\textwidth}
\centering
\includegraphics[width=\linewidth]{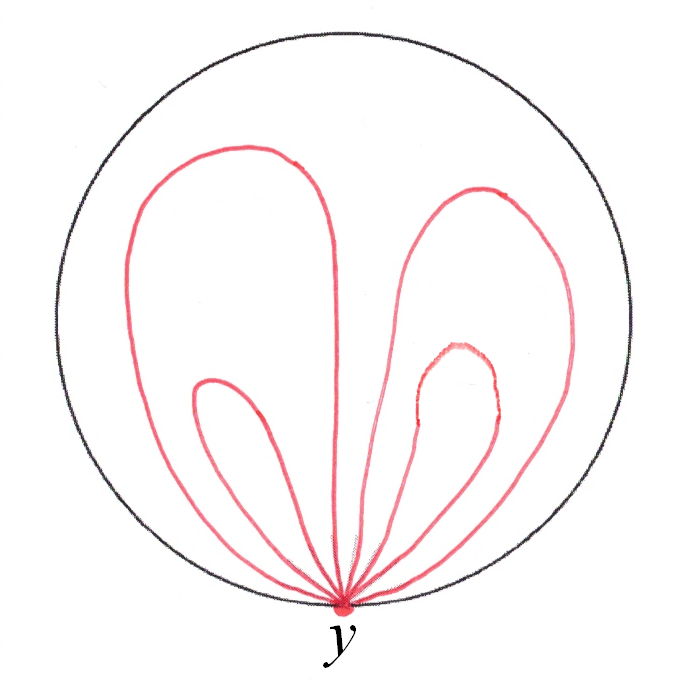}
\caption{}
\end{subfigure}
\begin{subfigure}[t]{.30\textwidth}
\centering
\includegraphics[width=\linewidth]{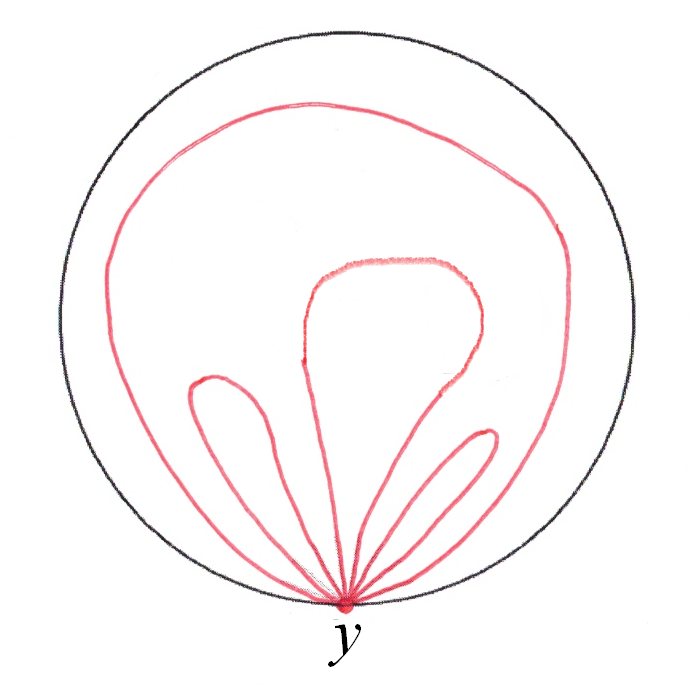}
\caption{}
\end{subfigure}
\caption{Some possible nodal patterns for $0 \not = u \in U^1_y$ ($\ell = 5$)}\label{F-hmn2-Ux100}
\end{figure}
\FloatBarrier

\subsubsection{Structure and combinatorial type of nodal sets in $U^1_y$ and $U^2_y$}\label{SSS-hmn-22as}\phantom{}~

\noid Relations~\eqref{E-bha4-2a} and an analysis as in Subsection~\ref{SS-hmn-21}, show that the nodal set of any $0\neq u \in U^1_y$ consists of $(\ell -1)$ nodal loops at the point $y$, and that these loops do not intersect away from $y$. The set $\cZ(u)$ is an $(\ell - 1)$-bouquet of loops $\cB_{y,(\ell -1)}$ at $y$. Adapting the description given in Paragraph~\ref{SSS-hmn-21as},  for $0 \neq u \in U_y^1$, we define the \emph{combinatorial type} $\tau_{y}$ of the nodal set $\cZ(u)$ with respect to $y$ for $u \in U_y^1$.   This is a map from $L_{(2\ell - 2)}$ to itself. \smallskip

Some possible nodal patterns  for $u \in U_y^1$ are displayed in Figure~\ref{F-hmn2-Ux100},  where $\ell = 5$, and $\rho(y) = 8$.  The corresponding combinatorial types, labeled  according to the figures,  are
\begin{equation*}
\begin{array}{l}
 \tau^{\ref{F-hmn2-Ux100}}_A  = \begin{pmatrix}
                          1 & 2 & 3 & 4 & 5 & 6 & 7 & 8\\
                          2 & 1 & 4 & 3 & 8 & 7 & 6 & 5\\
                        \end{pmatrix},\\[12pt]
\tau^{\ref{F-hmn2-Ux100}}_B = \begin{pmatrix}
                          1 & 2 & 3 & 4 & 5 & 6 & 7 & 8\\
                          4 & 3 & 2 & 1 & 8 & 7 & 6 & 5 \\
                        \end{pmatrix},\\[12pt]
\tau^{\ref{F-hmn2-Ux100}}_C = \begin{pmatrix}
                          1 & 2 & 3 & 4 & 5 & 6 & 7 & 8\\
                          8 & 3 & 2 & 5 & 4 & 7 & 6 & 1 \\
                        \end{pmatrix}.
\end{array}%
\end{equation*}

\noid  If $u \in U^2_y$ and $u \not \in U^1_y$, there exists a unique $z_u \in \Gamma$, such that $\cS_{\mathrm{b}}(u) \cap \Gamma = \set{y,z_u}$, with $z_u \neq y$, $\rho(u,y) = (2\ell - 3)$, and $\rho(u,z_u)=1$. Furthermore, $V_{y,z_u}=[u]$. The nodal set $\cZ(u)$  and its \emph{combinatorial type} $\tau_{y,z_u}$  are described in Paragraph~\ref{SSS-hmn-21as}. Then, $\cZ(u)$ is the wedge sum $\cB^{z_u}_{y,(\ell -2)}$ of a simple arc from $y$ to $z_u$ with an $(\ell-2)$-bouquet of loops at $y$, see Figure~\ref{F-hmn2-Uxy0}.

\subsection{Application of the previous analysis}\label{SS-hmn-24}~\\
Fix some $y \in \Gamma$. We now apply the analysis of Subsections~\ref{SS-hmn-21} and \ref{SS-hmn-22} to investigate the limits of $v_{y,z} \in U^2_y\sm U^1_y$, when $z$ tends to $y$ on $\Gamma$, \emph{clockwise} or \emph{counter-clockwise}. The notation are the same as in  Subsection~\ref{SS-hmn-21}.\smallskip

We choose a basis $\set{v_1,v_2}$ of $U^2_y$ as described in Remark~\ref{R-Ux1-2}. In particular, $\rho(v_1,y)=(2\ell -2)$,  $v_1 \perp v_2$ in $L^2(\Omega)$,  $\rho(v_2,y)=(2\ell-3)$, there exists $z_2 \in \Gamma\sm \set{y}$ such that $\rho(v_2,z_2)= 1$, and $\cS_{\mathrm{b}}(v_2) = \set{y,z_2}$.
Recall the definition of the functions $\breve{v}_i$ on $\Gamma$,
\begin{equation}\label{E-bha4-3d}
\breve{v}_i := \left\{
\begin{array}{ll}
\partial_{\nu}v_i &\quad \text{in the Dirichlet case,}\\[5pt]
v_i|_{\Gamma} &\quad \text{in the Robin case.}\\[5pt]
\end{array}
\right.
\end{equation}

According to Lemma~\ref{L-breve}, the function $\breve{v}_1$ vanishes only at $y$ and does not change sign on $\Gamma$. The function $\breve{v}_2$ does not vanish on $\Gamma \sm \set{y,z_2}$, and changes sign when crossing $z_2$  and $y$ along $\Gamma$.\smallskip

Let $\gamma : [0,2\pi] \to \Gamma$ be a parametrization such that $\gamma(0)=\gamma(2\pi) = y$.  Given any $z\in \Gamma \sm \set{y}$, there exists a function $v_{y,z}$ which satisfies \eqref{E-bha4-2a}, and this function is uniquely defined up to scaling. In the Dirichlet case, this function is characterized by the fact that $\breve{v}_{y,z} = \partial_{\nu}v_{y,z}|_{\Gamma}$ only vanishes at $y$ and $z$. In the Robin case, it is characterized by the fact that $\breve{v}_{y,z} =v_{y,z}|_{\Gamma}$ only vanishes at $y$ and $z$. Up to scaling, we may choose
\begin{equation}\label{E-bha4-4ab}
v_{y,z} = a(z) \, v_1 + b(z) \, v_2 ,
\end{equation}
with
\begin{equation}\label{E-bha4-4c}
\left\{
\begin{array}{l}
a(z) = -\, \breve{v}_2(z)\, \big( \breve{v}_1^2(z) + \breve{v}_2^2(z) \big)^{- \frac 12} ,\\[5pt]
b(z) = \breve{v}_1(z) \, \big( \breve{v}_1^2(z) + \breve{v}_2^2(z) \big)^{- \frac 12},\\[5pt]
\end{array}
\right.
\end{equation}
where $\breve{v}_1, \breve{v}_2$ are defined in \eqref{E-bha4-3d}. \\
Then, there exists a unique $\theta(z) \in (0,\pi)$ such that $\cos(\theta(z)) = a(z)$ and $\sin(\theta(z)) = b(z)$ (this is because $\breve{v}_1$ is positive on $\Gamma \sm \set{y}$). Defining the family of functions
\begin{equation}\label{E-bha4-4d}
w_{\theta} = \cos\theta \, v_1 + \sin\theta \, v_2,
\end{equation}
we have $v_{y,z} = w_{\theta(z)}$. Conversely, according to the proof of Lemma~\ref{L-Ux1}, any function $w_{\theta}$ has exactly two singular points on $\Gamma$, the point $y$ and some other point $z_{\theta} \not = y$. Note that the point $z$ determines the eigenfunction $v_{y,z}$ uniquely (up to scaling) and vice versa. It follows that we have a continuous, bijective map $(0,2\pi) \ni t \mapsto \theta(\gamma(t)) \in (0,\pi)$. This map is strictly monotone, and provides a diffeomorphism from $(0,2\pi)$ to $(0,\pi)$, with limits $0$ and $\pi$ respectively. Otherwise stated, the function $v_{y,\gamma(t)}$ defined in \eqref{E-bha4-4ab} tends to $v_1$ when $t$ tends to $0$ and to $-v_1$ when $t$ tends to $2\pi$. There exists $t_2$ such that $\gamma(t_2) = z_2$, and hence $\theta(z_2) = \frac \pi 2$.  We have proved the following property.

\begin{property}\label{P-hmn}
The function $v_{y,z}$ defined in \eqref{E-bha4-4ab} tends to $v_1$ when $z \neq y$ tends to $y$  counter-clockwise,  and to $-v_1$ when $z \neq y$ tends to $y$ clockwise.
\end{property}%

\subsection{$\Omega$ simply connected, proof that $\mult(\lambda_k) \le (2k -2)$  for all $k\ge 3$}\label{SS-hmn-25}~\\
In this subsection, we work with the family of functions $\set{w_{\theta} \mid \theta \in [0,\pi]}$ introduced in \eqref{E-bha4-4d}.\smallskip

\subsubsection{Preparation}\label{SSS-hmn-25A}  In view of Proposition~\ref{P-hmn-s1}, and reasoning by contradiction, we assume that $\dim U(\lambda_k) = (2k -1)$. By Courant's theorem, we have \[\sup\set{\kappa(u) \mid 0 \not = u \in U} \le k.\] We can apply Lemma~\ref{L-Ux1} with $\ell = k$ and $U := U(\lambda_k)$. \smallskip

In the arguments below we keep the notation of Lemma~\ref{L-Ux1} and its proof (with $\ell = k$). We fix a basis $\set{v_1,v_2}$ of $U^2_y$ as described at the beginning of  Subsection~\ref{SS-hmn-24}, and the direct frame $\set{\vec{e_1},\vec{e}_2}$ such that $\vec{e_1}$ is tangent to $\Gamma$ at $y$, and $\vec{e}_2$ is normal to $\Gamma$, pointing inwards.\smallskip

\subsubsection{Structure and combinatorial types for $v_1$ and $v_2$}\label{SSS-hmn-25B}

 Making a conformal chan\-ge of coordinates as in Section~\ref{S-lsbs}, we may assume that $\Gamma$ is a line segment in some neighborhood of $y$. Taking $r_1$ small enough, in the half-disk $D_{+}(y,r_1)$, the nodal set $\cZ(v_1)$ consists of $(2k-2)$ nodal semi-arcs $\delta_{1,j}$ emanating from $y$ tangentially to rays $\omega_{1,j}, j \in L_{(2k-2)}$; the nodal set $\cZ(v_2)$ consists of $(2k-3)$ nodal semi-arcs $\delta_{2,j}$ emanating from $y$ tangentially to rays $\omega_{2,j}, j \in L_{(2k-3)}$.\smallskip

The combinatorial type of the function $v_1 \in U^1_y$ with respect to $y$ is defined in Subsection~\ref{SS-hmn-22}. This is a map
\begin{equation}\label{E-hmn-206a}
\begin{array}{l}
\tau_{y} : L_{(2k-2)} \to L_{(2k-2)} \text{~such that}\\[5pt]
\tau_{y}(j) \not = j \text{~and~} (\tau_{y})^2(j) = j,
\text{~~for all~} j \in L_{(2k-2)}.
\end{array}%
\end{equation}
The nodal set $\cZ(v_1)$ is a $(k-1)$-bouquet of loops at $y$ described by the map $\tau_{y}$.\medskip

The combinatorial type $\tau_{y,z_2}$ of the function $v_2 \in U^{1,\perp}_y$, with respect to the points $y$ and $z_2\,$, is defined in Subsection~\ref{SS-hmn-22}. Recall that it is described as a map
\begin{equation}\label{E-hmn- 206b}
\left\{
\begin{array}{l}
\tau_{y,z_2} : \set{\downarrow} \cup L_{(2k-3)} \to \set{\downarrow} \cup  L_{(2k-3)} \text{~such that}\\[5pt]
\tau_{y,z_2}(\downarrow) =: a \in L_{(2k-3)} \text{~and~}
\tau_{y,z_2}(a) = \downarrow\,,\\[5pt]
\tau_{y,z_2}(j) \not = j \text{~and~} (\tau_{y,z_2})^2(j) = j,
\text{~~for all~} j \in L_{(2k-3)}\sm \set{a}.
\end{array}%
\right.
\end{equation}
Here, $\tau_{y,z_2}(\downarrow)$ is the element $a \in L_{(2k-3)}$ such that the semi-arc $\delta_a$ of $\cZ(v_2)$ which emanates from $y$ tangentially to $\omega_{2,a}$ eventually hits $\Gamma$ at the point $z_2$. For $a \neq j \in L_{(2k-3)}$,  the pairs $\big( j, \tau_{y,z_2}^{v_2}(j)\big)$ describe the loops of $\cZ(v_2)$ at the point $y$, so that $\cZ(v_2)$ is the wedge sum of the nodal interval $\delta_{a}$, where $a := \tau_{y,z_2}(\downarrow)$, with a $(k-2)$-bouquet of loops at $y$, described by the map $\tau_{y,z_2}$. \smallskip

Since $\Omega$ is simply connected, the arc $\delta_a$ separates $\Omega$ into
two  components $\Omega_{a,R}$ (on the right side of $\delta_a$),
and $\Omega_{a,L}$ (on the left side of $\delta_a$). There are three cases to consider, $a = 1$, $1 < a < (2k-3)$, and $a = (2k-3)$. The following properties follow easily from looking at the local structure of $\cZ(v_2)$ at $y$.

\begin{properties}\label{P-hmn-v2}\phantom{}
\begin{enumerate}[(i)]
  \item If $a = 1$, the component $\Omega_{a,R}$ does not contain any nodal arc, and the rays $\omega_{2,j}$, $2 \le j \le (2k-3)$ point inside $\Omega_{a,L}$.
  \item If $1 < a < (2k-3)$, the rays $\omega_{2,j}$, $1 \le j \le (a-1)$ point inside $\Omega_{a,R}$ ; the rays $\omega_{2,j}$, $(a+1) \le j \le (2k-3)$ point inside $\Omega_{a,L}$.
  \item If $a = (2k-3)$, the rays $\omega_{2,j}$, $1 \le j \le (2k-4)$ point inside $\Omega_{a,R}$ ; the component $\Omega_{a,L}$ does not contain any nodal arc.
\end{enumerate}
\end{properties}%

If the ray $\omega_{2,j}$ points inside $\Omega_{a,R}$, the whole nodal interval $\delta_{j}$ of $\cZ(v_2)$ is contained in $\Omega_{a,R}$, and so does the corresponding loop $\gamma_{j,\tau_{y,z_2}(j)}$. There is an analogous statement for $\Omega_{a,L}$. \smallskip

This means that
\begin{equation}\label{E-hmn-206c}
\left\{
\begin{array}{l}
a = \tau_{y,z_2}(\downarrow) \in L_{(2k-3)} \text{~~ is odd,}\\[5pt]
\tau_{y,z_2}\big(\set{1,\ldots,(a-1)}\big) \subset \set{1,\ldots,(a-1)}, \\[5pt]
\tau_{y,z_2}\big(\set{(a+1),\ldots,(2k-3)}\big) \subset \set{(a+1),\ldots,(2k-3)}.
\end{array}
\right.
\end{equation}
\begin{figure}[!ht]
  \centering
  \includegraphics[width=0.5\textwidth]{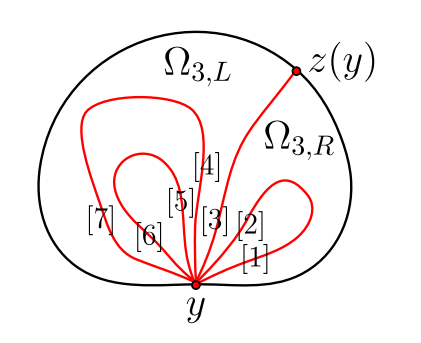}
  \caption{$k=5$, $a=3$}\label{F-hmn2-Uzz}
\end{figure}

More precisely, define
\begin{equation}\label{E-hmn-208}
\left\{
\begin{array}{l}
R := \set{1,\ldots,(a-1)}, \text{~with~~} R = \emptyset \text{~~if~~} a=1,\\[5pt]
L := \set{(a+1),\ldots,(2k-3)}, \text{~with~~} L = \emptyset \text{~~if~~} a=(2k-3),\\[5pt]
n_{R} := \frac{a-1}{2}.
\end{array}
\right.
\end{equation}

Then, the set $R$ corresponds to $n_{R}$ loops in the component $\Omega_{a,R}$ of $\Omega\sm \delta_a$, and these loops divide $\Omega_{a,R}$ into $(n_{R}+1)$ nodal domains of $v_2$. The set $L$ corresponds to $n_{L} := (k-2-n_{R})$ loops in the component $\Omega_{a,L}$ of $\Omega\sm \delta_a$, and these loops divide $\Omega_{a,L}$ into $n_{L}+ 1 = (k-1-n_{R})$ nodal domains of $v_2$, so that we recover the fact that $v_2$ has $k$ nodal domains.\smallskip

Otherwise stated the nodal set $\cZ(v_2)$ consists of the wedge sum of the nodal interval $\delta_a$ with two bouquets of loops, one $\cB_{R}$ contained in $\Omega_{a,R}$, corresponding to $ \tau_{y,z_2}|_{R}$,  another $\cB_{L}$ contained in $\Omega_{a,L}$, corresponding to $ \tau_{y,z_2}|_{L}$. One of these bouquets may be empty (when $a=1$ or $a=(2k-3)$).\medskip

For $\theta \in (0,\pi)$,  let $w_{\theta} := \cos\theta \, v_1 + \sin\theta \, v_2$.  Then $\cS_{\mathrm{b}}(w_{\theta}) = \set{y,z_{\theta}}$. The nodal set $\cZ(w_{\theta})$ has a structure similar to the structure of $\cZ(v_2)$:
\begin{itemize}
\item[$\diamond$] one nodal interval $\delta_{a_{\theta}}^{w_{\theta}}$ emanating from $y$ tangentially to some ray $\omega_{2,a_{\theta}}$, and hitting the boundary $\Gamma$ at the point $z_{\theta} \not = y$; $a_{\theta}$ is odd, and $a_{\theta} = \tau_{y}^{w_{\theta}}(\downarrow)$;
\item[$\diamond$]  loops at $y$, on either side of $\delta_{a_{\theta}}^{w_{\theta}}$, described by the restriction of the combinatorial type $\tau_{y}^{w_{\theta}}$ to $L_{(2k-3)} \sm \set{a_{\theta}}$.
\end{itemize}

Note:  In the above description, the nodal interval $\delta_{a_{\theta}}^{w_{\theta}}$ depends on $w_{\theta}$ and $a_{\theta}$. This point will be needed later on.

\begin{lemma}\label{L-hmn-202} Recall the notation $a = \tau_{y,z_{2}}(\downarrow)$ and $a_{\theta} = \tau_{y,z_{\theta}}(\downarrow)$.
For all $\theta \in (0,\pi)$, we have $a_{\theta} = a$, and $\tau_{y,z_{\theta}} = \tau_{y,z_2}$\,, i.e., the  combinatorial type of the nodal set $\cZ(w_{\theta})$ is the same as the combinatorial type of $\cZ(v_2)$.
\end{lemma}%

\begin{proof}[Proof of Lemma \ref{L-hmn-202}] We consider local conformal coordinates  as in Section~\ref{S-lsbs}. The proof of the local structure theorem shows that one can choose the radius $r_0$ uniformly with respect to $\theta$. We use polar coordinates $(r,\omega)$ associated with $(\xi_1,\xi_2)$ in $\R^2$, and we write $E(r,\omega)$ for $E(r\cos\omega,r\sin\omega)$.\smallskip

By connectivity, to prove that $a_{\theta}$ is constant it suffices to prove that it is locally constant:\smallskip

\emph{For all $\theta \in (0,\pi)$ there exists $\varepsilon_{\theta} > 0$ such that $a_{\theta'} = a_{\theta}$ for $|\theta' - \theta| < \varepsilon_{\theta}$\,.} \smallskip

Assume, by contradiction, that this is not the case. Then, there exists $\theta_0 \in (0,\pi)$ and a sequence $\theta_n$ with $|\theta_n - \theta_0| < \frac 1n$ and $a_{\theta_n} \not = a_{\theta_0} =: a_0$. Since $a_{\theta}$ can only take finitely many values in $ L_{(2k-3)}$,  passing to a subsequence if necessary, we can assume that $a_{\theta_n} \equiv a_1 \not = a_0$. By the local structure theorem, there exists a uniform $r_0 > 0$ (depending on $\theta_0$) such that the nodal arc $ \delta_{a_1,\theta_n}$ intersects the   set $C_{+}(y,r_0)$ at the point $z_n := E\big( r_0, \omegat_{a_1}(r_0,\theta_n) \big)$, where the function $\omegat_{a_1}(r,\theta)$ is smooth in a neighborhood of $(r_0,\theta_0)$ (with the notation of Section~\ref{S-lsbs}).
The arcs $\delta_{a_1,\theta_n} \cap \Omega\sm B(y,r_0)$ are compact and connected, and we can find a subsequence which converges in the Hausdorff distance to some compact connected set $\bar{\delta}$ which contains the point $z_0 = E\big( r_0, \omegat_{a_1}(r_0,\theta_0) \big)$ and the point $z_{\theta_0} = \lim z_{\theta_n}$ at the boundary. The set $\bar{\delta}$ is also contained in $\cZ(w_{\theta_0})$ because $w_{\theta_n}$ tends to $w_{\theta_0}$ uniformly. Since $a_1 \not = a_0$, we have a contradiction.\smallskip

Since $a_{\theta} \equiv a$ in $(0,\pi)$, in order to prove that $\tau_{\theta} := \tau_{y,z_{\theta}}$ does not depend on $\theta$, it suffices to show that its restrictions to the sets $R$ and $L$ are locally constant in $\theta$. We give the proof for $R$ in the case $a > 1$. The other cases are similar. Reasoning by contradiction, we assume that there exists $\theta_0 \in (0,\pi)$ and a sequence $\theta_n$ such that $|\theta_n-\theta_0| < \frac 1n$, and $j_n \in R$, such that $\tau_{\theta_n}(j_n) \not = \tau_{\theta_0}(j_n)$. Since $R$ is finite, passing to subsequences if necessary, we may assume that $j_n \equiv b$ and $\tau_{\theta_n}(b) \equiv c$ for some $b,c \in R$ with $c\not = \tau_{\theta_0}(b)$. Since $\theta_n$ is close to $\theta_0 \in (0,\pi)$ we have a uniform structure theorem, and we can reason as in the proof of Property~\ref{propertyD4} to conclude.
\end{proof}

Given the basis $ \set{v_1,v_2}$, we have the associated odd integer  $a := \tau_2(\downarrow) \in L_{(2k-3)}$, where $\tau_2 := \tau_{y,z_2}$ is the combinatorial type of $v_2$. For all $\theta \in (0,\pi)$, the nodal set $\cZ(w_{\theta})$ has the same combinatorial type $\tau_2$. In particular, it contains a single simple nodal interval $\delta_{a,\theta}$, emanating from $y$ tangentially to the ray $\omega_{2,a}$, and hitting the boundary $\Gamma$ at the point $z_{\theta}$. \smallskip

Call $\Omega_{\theta,R}$ the component of $\Omega \sm \delta_{a,\theta}$ with semi-tangent at $y$ the vector $\vec{e}_1$, and  $\Omega_{\theta,L}$ the other component, with semi-tangent at $y$ the vector $-\vec{e}_1$. The component $\Omega_{\theta,R}$ contains the $n_R$ loops corresponding to the set $R$. These loops bound $(n_R+1)$ nodal domains of $w_{\theta}$ which can be labeled from $1$ to $(n_R+1)$. The component $\Omega_{\theta,L}$ contains the $n_L = (k - n_R -2)$ loops corresponding to the set $L$. These loops bound $(k - n_R -1)$ nodal domains of $w_{\theta}$ which can be labeled from $(n_R+2)$ to $k$.\smallskip

\subsubsection{The rotating function argument}\label{SSS-hmn-25D}
\index{Rotating function argument}
Look at the simple examples, displayed in Figures~\ref{F-hmn2-Uxyw-L1}
and \ref{F-hmn2-Uxyw-L2}.

\begin{figure}[!hbt]
\centering
\includegraphics[width=0.98\linewidth]{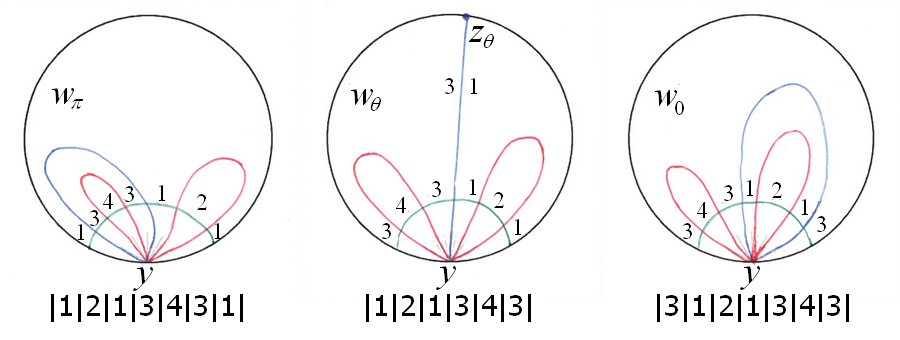}
\caption{Example with $k=4$, $a=3$}\label{F-hmn2-Uxyw-L1}
\end{figure}

\begin{figure}[!hbt]
\centering
\includegraphics[width=0.98\linewidth]{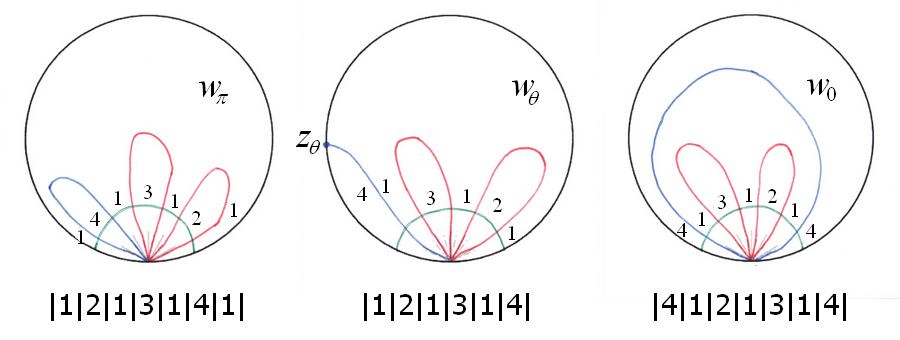}
\caption{Example with $k=4$, $a=5$}\label{F-hmn2-Uxyw-L2}
\end{figure}
\FloatBarrier

When $\theta$ tends to zero, $z_{\theta}$ tends to $y$ clockwise and the nodal interval $\delta_{a,\theta}$, from $y$ to $z_{\theta}$, tends to a loop in the nodal set of $w_0 = \lim_{\theta \to 0}w_{\theta}$, whose nodal pattern is displayed in the right sub-figure. When $\theta$ tends to $\pi$, $z_{\theta}$ tends to $y$ counter-clockwise and the nodal interval $\delta_{a,\theta}$ tends to a loop in the nodal set of $w_{\pi} = \lim_{\theta \to \pi}w_{\theta}$, whose nodal pattern is displayed in the left sub-figure. \smallskip

The combinatorial types of the functions $w_0$ and $w_{\pi}$ are as follows.
\begin{equation*}
\resizebox{1\hsize}{!}{%
$\begin{array}{c}
\text{Nodal patterns in Figure~\ref{F-hmn2-Uxyw-L1}:}\\[5pt]
\tau_{\pi} = \begin{pmatrix}
               1 & 2 & 3 & 4 & 5 & 6 \\
               2 & 1 & 6 & 5 & 4 & 3 \\
             \end{pmatrix}~
\tau_{\theta} = \begin{pmatrix}
               \downarrow & 1 & 2 & 3 & 4 & 5 \\
               3 & 2 & 1 & \downarrow & 5 & 4 \\
             \end{pmatrix}~
\tau_{0} = \begin{pmatrix}
               0 & 1 & 2 & 3 & 4 & 5 \\
               3 & 2 & 1 & 0 & 5 & 4 \\
             \end{pmatrix}.
\end{array}$%
}%
\end{equation*}
\begin{equation*}
\resizebox{1\hsize}{!}{%
$\begin{array}{c}
\text{Nodal patterns in Figure~\ref{F-hmn2-Uxyw-L2}:}\\[5pt]
\tau_{\pi} = \begin{pmatrix}
               1 & 2 & 3 & 4 & 5 & 6 \\
               2 & 1 & 4 & 3 & 6 & 5 \\
             \end{pmatrix}~
\tau_{\theta} = \begin{pmatrix}
               \downarrow & 1 & 2 & 3 & 4 & 5 \\
               5 & 2 & 1 & 4 & 3 & \downarrow \\
             \end{pmatrix}~
\tau_{0} = \begin{pmatrix}
               0 & 1 & 2 & 3 & 4 & 5 \\
               5 & 2 & 1 & 4 & 3 & 0 \\
             \end{pmatrix}.
\end{array}$%
}%
\end{equation*}
\smallskip

The above assertions are a consequence of the local structure theorem applied to $w_0$ or to $w_{\pi}$ in a disk  $D_{+}(y,r)$ with $r$ small enough. A loop in $\cZ(w_{\theta})$ always intersects $C_{+}(y,r)$ at two distinct points. For $\theta$ away from $0$ and $\pi$, the nodal interval from $y$ to $z_{\theta}$ only intersects $C_{+}(y,r)$ at the point $A_a$. When $\theta$ tends to $0$, the point $z_{\theta}$ enters the disc $D_{+}(y,r)$ and the nodal interval from $y$ to $z_{\theta}$ intersects $C_{+}(y,r)$ at two points $A_a$ and $A_0$, which lies below the first $\cR$-arc, see Figure~\ref{F-hmn2-w-theta0}. This is why $\tau_{0}$ is defined on the set $\set{0,\ldots,5}$. In the figure, the points $A_j$ are the intersection points of the nodal set $\cZ(w_{\theta})$ with $C_{+}(y,r)$, for $r$ small enough. Similarly, when $\theta$ tends to $\pi$, the nodal arc from $y$ to $z_{\theta}$ intersects $C_{+}(y,r)$ at two points, one of them below the last $\cR$-arc. This is why $\tau_{\pi}$ is defined on the set $\set{1,\ldots,6}$. For these arguments, we also refer to Section~\ref{S-hmn3}, proof of Lemma~\ref{L-L38}.\smallskip

\begin{figure}[!ht]
  \centering
  \includegraphics[width=0.95\textwidth]{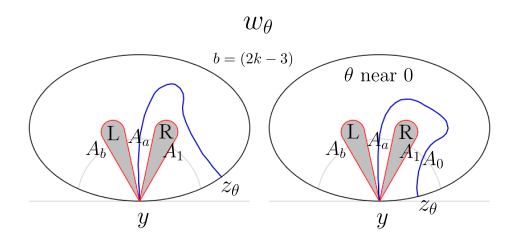}
  \caption{The behavior of $w_{\theta}$ for $\theta$ close to $0$}\label{F-hmn2-w-theta0}
\end{figure}

We label the nodal domains of $w_{\theta}$ in the central sub-figures according to the order in which they appear when one moves on $C_{+}(y,r)$ counter-clockwise. This order is independent of $\theta$ because the combinatorial type of $w_{\theta}$ is independent of $\theta$, see Lemma~\ref{L-hmn-202}. We then follow the deformation of the nodal sets of $w_{\theta}$, when $\theta$ tends to $0$, resp. to $\pi$. The order in which the nodal domains appear is encoded in the words $\ww_{\theta}$ which appear below the images  (labels are separated by vertical bars). In Figure~\ref{F-hmn2-Uxyw-L1}, $\ww_{\theta} = |1|2|1|3|4|3|$, $\ww_{\pi} = |1|2|1|3|4|3|1|$, $\ww_{0} = |3|1|2|1|3|4|3|$. In Figure~\ref{F-hmn2-Uxyw-L2}, $\ww_{\theta} = |1|2|1|3|1|4|$, $\ww_{\pi} = |1|2|1|3|1|4|1|$, $\ww_{0} = |4|1|2|1|3|1|4|$. \smallskip

The general procedure for labeling the nodal domains and producing the words is described in Section~\ref{S-hmn2L}.\smallskip

To see that the nodal patterns are different, and hence that the functions $w_0$ and $w_1$ are linearly independent, we look at the invariant $\sigma(\ww)$ defined in \eqref{D-hmn2L-4}: $\sigma(\ww)$ is the position at which the first letter of the word $\ww$ reappears. For the nodal patterns in Figure~\ref{F-hmn2-Uxyw-L1}, we have
$\sigma(\ww_{0})=5$  and $\sigma(\ww_{\pi})=3$. For the nodal patterns in Figure~\ref{F-hmn2-Uxyw-L2}, we have $\sigma(\ww_{0})=7$  and $\sigma(\ww_{\pi})=3$. \smallskip

At least for the above examples, the invariants being different, the functions $w_0$ and $w_{\pi}$ cannot be equal up to scaling, contradicting the fact that $w_0 = - w_{\pi} = v_1$.\smallskip

The proof in the general case is given in Section~\ref{S-hmn2L}, Section~\ref{SS-hmn2L-P1}. \hfill \qed \smallskip

 This completes the proof that the equality $\dim U(\lambda_k) = (2k-1)$ for some $k \ge 3$ leads to a contradiction.  Hence we have proved Theorem~\ref{T-hmn-bh1} under the additional assumption that $\Omega$ is \emph{simply connected}.

\begin{proposition}\label{P-hmn-s2}
Let $\Omega \subset \R^2$ be a \emph{simply connected} bounded domain with smooth boundary. Let $\set{\lambda_k, k\ge 1}$ be the eigenvalues of the operator $-\Delta + V$ in $\Omega$, with Dirichlet or Robin boundary condition  (where $V$ is a real valued $C^{\infty}$ function). Then,
\begin{equation*}
\mult(\lambda_k) \le (2k-2)~~  \text{for all~}k \ge 3  .
\end{equation*}
\end{proposition}%

\begin{remark}\label{R-hmn-final}
 The general idea to get rid of the assumption that $\Omega$ is simply connected is as follows.  We decompose the boundary $\Gamma := \partial \Omega$ into its  components, $\Gamma = \bigcup_{j=1}^q \Gamma_j$.  Then we repeat the arguments in Subsections~\ref{SS-hmn-21} and \ref{SS-hmn-22} with the prescribed singular points $y$ and $z$ chosen to belong to $\Gamma_1$. The main difference with the simply connected case is that an eigenfunction $u$ with prescribed singular points $y$ and $z$ on $\Gamma_1$ may also have singular points on some other components of $\Gamma$. More precisely, using the notation \eqref{E-hmn0-4}, the following properties hold: $\cZ(u) \cup \Gamma(u)$ is connected and, for any $j \in J(u)\sm \set{1}$, $\sum_{z \in \cS_{\mathrm{b}}(u) \cap \Gamma_j\,}\rho(u,z) = 2$. It then suffices to work with the projection $\check{\cZ}(u)$ of the nodal set $\cZ(u)$ to the set $\check{\Omega}$ obtained from $\Omega$ by identifying each component $\Gamma_j, j\ge 2$, to a point, so that the boundary of $\check{\Omega}$ is $\Gamma_1$.
The complete proof is given in Section~\ref{S-hmn2N}.
\end{remark}%

\section{General Case: the Estimate $\mult(\lambda_k) \le (2k-2)$ for $k \ge 3$}\label{S-hmn2N}

\subsection{An abstract setting}\label{SS-hmn2-10}

When the domain $\Omega \subset \R^2$ is not simply connected, its boundary $\Gamma := \partial \Omega$ has  $(q+1)$ components, with $q\geq 1$.  One of them $\Gamma_1$, the ``outer boundary'', bounds the unbounded component of $\R^2\sm \Gamma$. The other components, $\Gamma_j, ~j\neq 1$, are contained in the bounded component of $\R^2\sm \Gamma_1$. We consider the following equivalence relation in $\wb{\Omega}$:
\begin{equation}\label{E-hmn210-equiv}
x \sim y \text{~~if and only if~~} x, y \in \Gamma_j
\text{~for some~} j \in \set{2,\ldots,(q+1)}.
\end{equation}

\begin{notation}\label{N-hmn210-2}
Let $\check{\Omega}$ denote the quotient space $\wb{\Omega}/\!\sim$, where each $j\neq 1$, $\Gamma_j$  is identified to one point $\xi_j$ in $\check{\Omega}$ (see \cite{Bona2009}). Define
\[\Xi := \set{\xi_2,\ldots,\xi_{q+1}}.\]
Generally speaking, $\check{A}$ will denote the image of the subset $A$  of  $\wb{\Omega}$ under the projection map from $\wb{\Omega}$ to $\check{\Omega}$.\\
We also introduce $\bS_{\Omega}$, the quotient space in which each $\Gamma_j, j\ge 1$, is identified to a point.
\end{notation}%

\subsubsection{$\Omega$ with one hole. Properties of $V_{y,z}$}\label{SSS-hmn-21c}

\begin{lemma}\label{L-Uxy1}
Assume that $\Omega$ has one hole, and $\Gamma =\Gamma_1 \cup \Gamma_2$,  where $\Gamma_1$ is  the  outer boundary.  Let $U$ be a linear subspace of an eigenspace of \eqref{E-evp-2bc} in $\Omega$,  such that  $\dim U = (2\ell-1)$ and $\sup\set{\kappa(u) \mid 0 \not = u \in U}\le \ell$ for some $\ell \ge 2$. Let $y \neq z\in \Gamma_1$, and define
\[V_{y,z} := \set{u \in U \mid \rho(u,y) \ge (2\ell - 3) \text{~and~} \rho(u,z) \ge 1}.\] Then,
\begin{enumerate}[(i)]
  \item $\dim V_{y,z} = 1$.
  \item  For $0 \neq u \in V_{y,z}$, the following alternative holds,
  \begin{itemize}
    \item[$\diamond$] either $b_0(\cZ(u) \cup \Gamma) = 1$, in which case $\cZ(u)$ hits $\Gamma_2$, $u$ has precisely two singular points on $\Gamma_2$ (counting multiplicities), $\sum_{x \in \cS_{\mathrm{b}}(u) \cap \Gamma_2} \rho (u,x) = 2$,
$\cS_{\mathrm{b}}(u) \cap \Gamma_1 = \set{y,z}$ and $\cS_{\mathrm{i}}(u) = \emptyset$;
    \item[$\diamond$] or $b_0(\cZ(u) \cup \Gamma) = 2$, in which case $\cZ(u) \cap \Gamma_2 = \emptyset$,
$\cS_{\mathrm{b}}(u) \cap \Gamma_1 = \set{y,z}$, and $\cS_{\mathrm{i}}(u) = \emptyset$;
\end{itemize}
\item $\rho(u,y)=(2\ell-3)$ and $\rho(u,z) = 1$.
\item $\kappa(u) = \ell$.
\end{enumerate}
A generator of $V_{y,z}$ will be denoted by  $v_{y,z}$ (defined up to scaling).
\end{lemma}%

\begin{proof} The fact that $\dim V_{y,z}\ge 1$ follows from Lemma~\ref{L-zero2}. Since we now have $b_0(\Gamma)=2$, Euler's formula \eqref{E-euler-or2a} applied to $u$ gives
\begin{equation}\label{E-A6}
\begin{array}{ll}
0 \geq \kappa(u) - \ell  \,  = & \big( b_0(\cZ(u)\cup \Gamma) -2\big) + \frac 12 \sum_{x\in \cS_{\mathrm{i}}(u)} (\nu(x)-2)\\[5pt]
& \hspace{3mm}+ \frac 12 \sum_{x \in \cS_{\mathrm{b}}(u) \cap \Gamma_2} \rho (x)\\[5pt]
& \hspace{3mm}+ \frac 12 \sum_{\substack{x\in \cS_{\mathrm{b}}(u) \cap \Gamma_1,\\ x \neq y,z}\,} \rho (x) + \frac 12 \big( \rho(y) + \rho(z) -2 \ell +2\big).
 \end{array}
\end{equation}

Except for the term $\big( b_0(\cZ(u) \cup \Gamma) -2\big)$, all the terms in the right-hand side of the equality are nonnegative. This implies that
\begin{equation*}
2 \ge b_0(\cZ(u) \cup \Gamma) \ge 1,
\end{equation*}
and we have to examine two cases.\smallskip

\noid If $b_0(\cZ(u) \cup \Gamma) = 1$, then the nodal set $\cZ(u)$ must hit $\Gamma_2$. According to Proposition~\ref{P-euler-np}, the sum $\sum_{x \in \cS_{\mathrm{b}}(u)\cap \Gamma_2} \rho (x)$ is an even integer, and we deduce from \eqref{E-A6} that $\sum_{x \in \cS_{\mathrm{b}}(u)\cap \Gamma_2} \rho (x) = 2$. This equality now implies that the other terms are zero, and hence that $\kappa(u) = \ell$.\smallskip

\noid If $b_0(\cZ(u) \cup \Gamma) = 2$, then all the terms in the right-hand side of \eqref{E-A6} vanish, and $\kappa(u) = \ell$. \smallskip

This proves Assertions~(ii)--(iv).  As in the proof of Lemma~\ref{L-Uxy0}, assuming that there are at least two linearly independent functions $u_1$ and $u_2$ in $V_{y,z}$, the first assertion follows from Assertion~(iii)  and Lemma~\ref{L-zeroc}. \end{proof}

\subsubsection{$\Omega$ with one hole. Structure and combinatorial type of nodal sets in $V_{y,z}$}\label{SSS-hmn-21cs}

From a geometric point of view and under the assumptions of Lemma~\ref{L-Uxy1}, once we have fixed $y \not = z$ in $\Gamma_1$, either $\cZ(u) \cap \Gamma_2 = \emptyset$\quad or $\cZ(u) \cap \Gamma_2 = \set{y_1,y_2}$, possibly with $y_1=y_2$.  In the first case, we simply reproduce the description given in Paragraph~\ref{SSS-hmn-21as}. In the second case, the connectivity of $\cZ(u) \cup \Gamma$ implies that one of the nodal arcs hitting $\Gamma_2$ also contains $y$.  The other one contains either $y$ or $z$. More precisely, we choose any $j \in L_{(2\ell - 3)}$ and follow the nodal semi-arc $\delta_j$ emanating from $y$ along $\cZ(u)$, until we meet a singular point $x$ as described in Paragraph~\ref{SSS-hmn-21as}. Since $x \in \cS(u) = \set{y,z,y_1,y_2}$, there are two possibilities. If $x \in \set{y,z}$ the description is similar to the one in Paragraph~\ref{SSS-hmn-21as}. If $x \in \set{y_1,y_2}$, say $y_1$, we continue our path from $y_1$ to $y_2$ along $\Gamma_2$, and leave $\Gamma_2$ along the second nodal arc hitting $\Gamma_2$ at $y_2$ until we meet a singular point. We then either reach the point $y$ again or the point $z$. It then follows that the nodal set of $u \in V_{y,z}$ consists of $ (\ell - 2)$ ``generalized'' nodal loops at $y$ (one of the loops may comprise some part of $\Gamma_2$), and a ``generalized'' simple arc from $y$ to $z$ (this arc may comprise a sub-arc from $y_1$ to $y_2$ on $\Gamma_2$).  In the preceding description, the points $y_1$ and $y_2$ may coincide. These ``generalized'' loops and arc do not intersect away from $y$. Then, $\cZ(u)$ is the wedge sum $\cB^{z}_{y,(\ell-2)}$ of an $(\ell - 2)$-bouquet of ``generalized'' loops at $y$, with a simple ``generalized'' arc from $y$ to $z$. We can then define the \emph{combinatorial type} $\tau_{y,z}^{u}$ of $u$ with respect to the points $y$ and $z$ as we did in Paragraph~\ref{SSS-hmn-21as}, somehow ignoring $\Gamma_2$.\smallskip

Projecting $\cZ(u)$ to $\check{\Omega}$, we obtain a set $\check{\cZ}(u) \subset \check{\Omega}$ which is the wedge sum $\cB_{\check{y}, (\ell-2)}^{\check{z}}$ of an $(\ell-2)$-bouquet of loops at $\check{y}$ with a simple arc from $\check{y}$ to $\check{z}$. One of the loops or the arc may contain the point $\xi_2$, the image of $\Gamma_2$ in $\check{\Omega}$.   Since there are only two semi-arcs at $\xi_2$, this point is a regular point of the projected nodal partition $\check{\cD}_{u}$. The general  picture is then similar to the picture in the simply connected case.

Figures~\ref{F-hmn2-Uxy1-1} and \ref{F-hmn2-Uxy1-2} display some possible nodal patterns for $0 \neq u \in V_{y,z}$ when $\Omega$ has one hole ($\ell=5$, $\rho(y)=7$, and $\rho(z)=1$).

\begin{figure}[!ht]
\centering
\begin{subfigure}[t]{.30\textwidth}
\centering
\includegraphics[width=\linewidth]{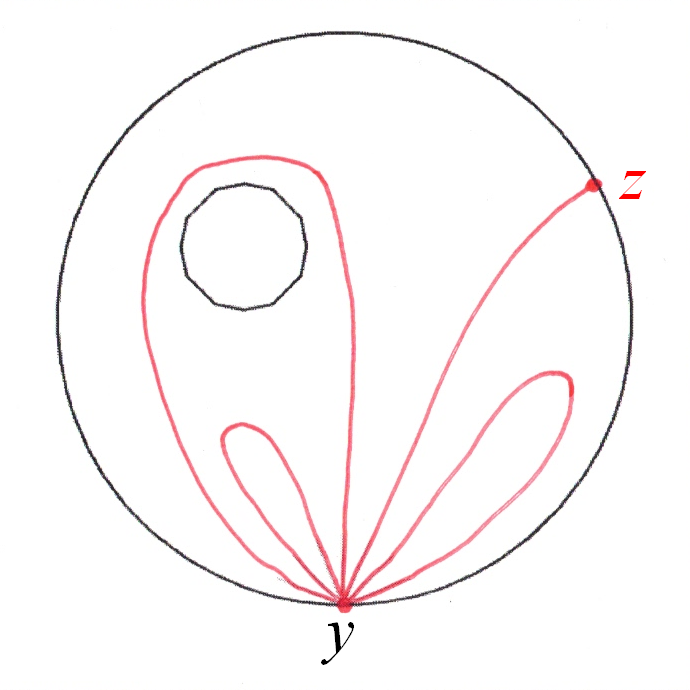}
\caption{}
\end{subfigure}
\begin{subfigure}[t]{.30\textwidth}
\centering
\includegraphics[width=\linewidth]{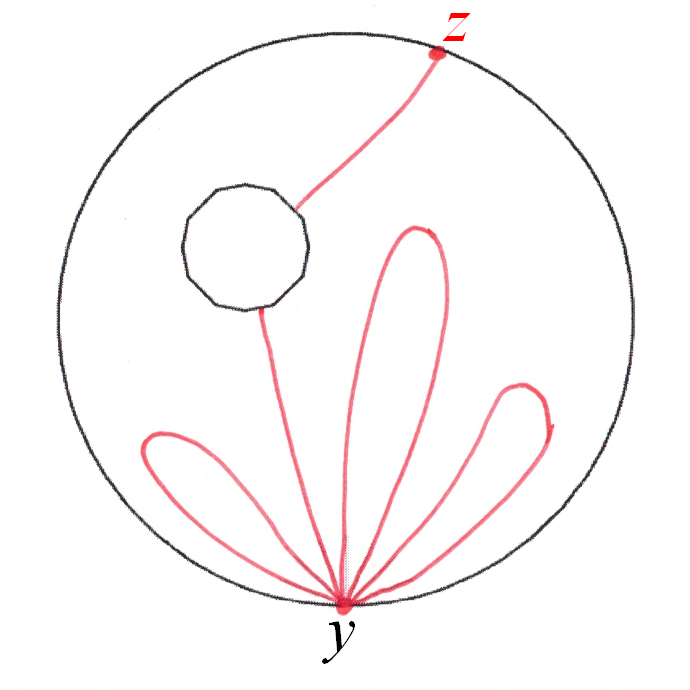}
\caption{}
\end{subfigure}
\begin{subfigure}[t]{.30\textwidth}
\centering
\includegraphics[width=\linewidth]{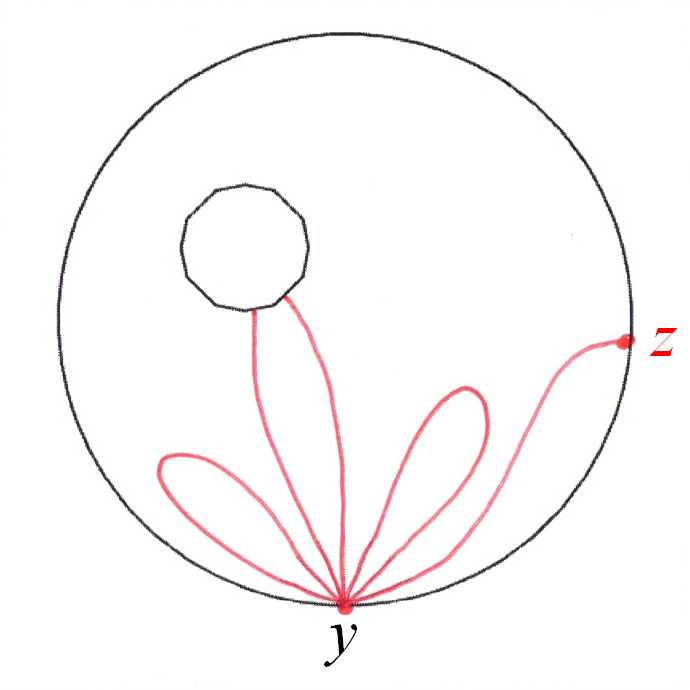}
\caption{}
\end{subfigure}
\caption{$\Omega$ with one hole: some possible nodal patterns for $v_{y,z}$}\label{F-hmn2-Uxy1-1}
\end{figure}

\begin{figure}[!ht]
\centering
\begin{subfigure}[t]{.30\textwidth}
\centering
\includegraphics[width=\linewidth]{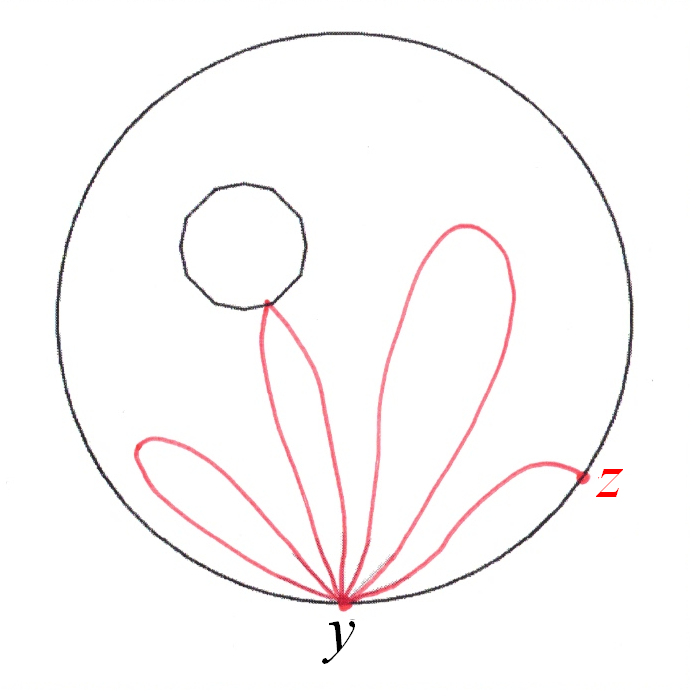}
\caption{}
\end{subfigure}
\begin{subfigure}[t]{.30\textwidth}
\centering
\includegraphics[width=\linewidth]{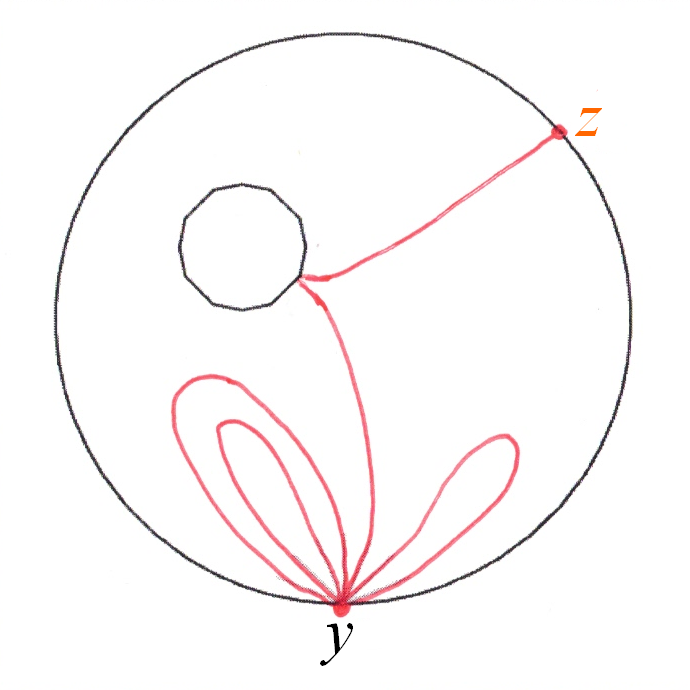}
\caption{}
\end{subfigure}
\begin{subfigure}[t]{.30\textwidth}
\centering
\includegraphics[width=\linewidth]{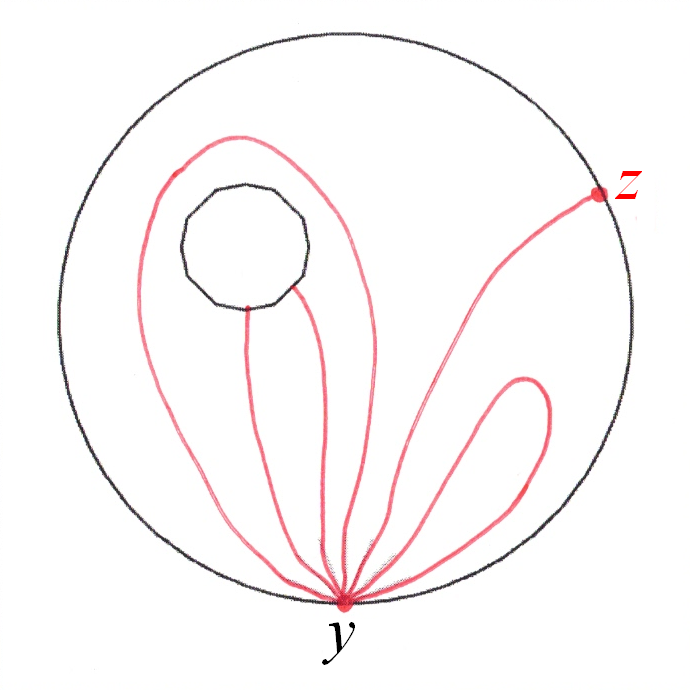}
\caption{}
\end{subfigure}
\caption{$\Omega$ with one hole: some possible nodal patterns for $v_{y,z}$}\label{F-hmn2-Uxy1-2}
\end{figure}
\FloatBarrier

For the nodal patterns in Figures~\ref{F-hmn2-Uxy1-1} and  \ref{F-hmn2-Uxy1-2}, we have the  combinatorial types
\begin{equation*}
\begin{array}{ll}
\tau^{\ref{F-hmn2-Uxy1-1}}_A = \tau^{\ref{F-hmn2-Uxy1-2}}_B = \tau^{\ref{F-hmn2-Uxy1-2}}_C &= \begin{pmatrix}
                         \downarrow & 1 & 2 & 3 & 4 & 5 & 6 & 7 \\
                          3 & 2 & 1 & \downarrow & 7 & 6 & 5 & 4 \\
                        \end{pmatrix},\\[12pt]
\tau^{\ref{F-hmn2-Uxy1-1}}_C = \tau^{\ref{F-hmn2-Uxy1-2}}_A &= \begin{pmatrix}
                         \downarrow & 1 & 2 & 3 & 4 & 5 & 6 & 7 \\
                          1 & \downarrow  & 3 & 2 & 5 & 4 & 7 & 6 \\
\end{pmatrix},\\[12pt]
                        \tau^{\ref{F-hmn2-Uxy1-1}}_B &= \begin{pmatrix}
                        \downarrow & 1 & 2 & 3 & 4 & 5 & 6 & 7 \\
                        5 & 2 & 1 & 4 & 3 & \downarrow & 7 & 6 \\
                        \end{pmatrix}.
\end{array}%
\end{equation*}

\subsubsection{ $\Omega$ with $k$ holes. Properties of $V_{y,z}$}\label{SSS-hmn-21d}

In this case, $\Gamma$ has $(k+1)$ components,  $\Gamma = \bigcup_{j=1}^{k+1} \Gamma_j$. Fix $y \not = z \in \Gamma_1$.
With the notation of Subsection~\ref{SS-hmn-0N}, we have the following lemma.

\begin{lemma}\label{L-Uxyk} Assume that $\Omega$ has $k$ holes, with $\Gamma = \bigcup_{j=1}^{k+1} \Gamma_j$.  Let $U$ be an eigenspace of \eqref{E-evp-2bc} in $\Omega$, such that, for some $\ell \ge 2$,  $\sup\set{\kappa(u) \mid 0 \not = u \in U}\le \ell$, and  $\dim U = (2\ell-1)$. Let $y \neq z \in \Gamma_1$, and define the subspace \[V_{y,z} := \set{u \in U \mid \rho(u,y) \ge (2\ell - 3) \text{~and~} \rho(u,x) \ge 1}.\]  Then, $\dim V_{y,z} = 1$. Furthermore, for all $0 \neq u \in V_{y,z}$:
\begin{enumerate}[(i)]
  \item The set $\cZ(u) \cup \Gamma(u)$ is connected.
  \item If $J(u)=\set{1}$, the only singular points of $u$ are the points $y$ and $z$, with $\rho(u,y) = (2\ell - 3)$ and $\rho(u,z)=1$.
  \item If $J(u) \not = \set{1}$, each component $\Gamma_j, j \in J(u)$, is hit by exactly two nodal arcs, the function $u$ has no interior singular point, and its only singular points on $\Gamma_1$ are $y$ and $z$, with $\rho(u,y) = (2\ell - 3)$ and $\rho(u,z)=1$.
  \item In all cases, $\kappa(u) = \ell$.
  \item In all cases, the nodal set of $u$ consists of $(\ell - 2)$ simple non-intersecting ``generalized'' nodal loops at $y$ (loops comprising nodal arcs, and possibly arcs contained in some boundary components $\Gamma_j, j \in J(u)\sm \set{1}$), a simple nodal arc from $y$ to either $z$ (when $J(u)=\set{1}$) or to some inner component of $\Gamma$, a simple nodal arc from $y$ to some component $\Gamma_j, j\in J(u)\sm \set{1}$, and possibly some nodal arcs joining components which meet $\cZ(u)$. These nodal arcs can only intersect at $y$ or possibly on the components $\Gamma_j, j \in J(u)\sm \set{1}$. In all cases, the point $y$ is joined to the point $z$ by a simple arc comprising nodal arcs and possibly sub-arcs of the $\Gamma_j, j \in J(u)\sm \set{1}$.
\end{enumerate}
A generator of $V_{y,z}$ will be denoted by $v_{y,z}$ (defined up to scaling).
\end{lemma}%

\begin{proof}[Proof of Lemma~\ref{L-Uxyk}]  With the assumptions of the lemma, Euler's formula \eqref{E-euler-or2a}
can be rewritten as,
\begin{equation}\label{E-bha-k10}
\begin{array}{ll}
0 \ge \kappa(u) - \ell =& \big( b_0(\cZ(u) \cup \Gamma(u)) - 1) \big)
+ \frac 12\, \sum_{x \in \cS_{\mathrm{i}}(u)} (\nu(x)-2)\\[5pt]
& +\, \sum_{j \in J(u), j\neq 1\,} \frac 12  \, \big( \sum_{x\in \cS _b(u) \cap \Gamma_j}\rho (x) - 2 \big)\\[5pt]
& +\, \frac 12 \sum_{x\in \cS _b(u) \cap \Gamma_1,\, x \not = y,z\,}\rho (x) + \frac 12 \big( \rho(y) + \rho(z) - 2\ell +2\big).
\end{array}%
\end{equation}

In view of our assumptions, and Proposition~\ref{P-euler-np}, the terms in the right-hand side of \eqref{E-bha-k10} are all non-negative. In view of the left hand side of the inequality, they must all be zero. We now examine two cases.\smallskip

\noid If $J(u)=\set{1}$, the second line is the right hand side disappears, the nodal set $\cZ(u)$ only meets $\Gamma_1$, $b_0(\cZ(u) \cup \Gamma_1) =1$, the only  singular points  of the function $u$ are the points $y$ and $z$, and $ \rho(y) = (2\ell-3)$, $\rho(z) = 1$.\smallskip

\noid If $J(u) \not = \set{1}$, all the terms in the right hand side must be zero: $b_0(\cZ(u) \cup \Gamma(u)) =1$, each component $\Gamma_j, j\in J(u)\sm \set{1}$, is hit by precisely two nodal arcs of $\cZ(u)$, $ \rho(y) = (2\ell-3)$, and $ \rho(z) = 1$, and the function $u$ has no other  singular point whether in the interior of $\Omega$ or on $\Gamma$.  Furthermore, there is a simple nodal arc from $y$ to one of the components $\Gamma_j, j \in J(u)$, a simple nodal arc from $z$ to one of the components $\Gamma_j, j \in J(u)$, and there is a simple nodal arc, possibly comprising arcs contained in $\Gamma(u)$ joining $y$ to $z$. Finally, $\kappa(u) = \ell$. This proves assertions (i)--(v).\medskip

To prove the first assertion, assuming there are at least two linearly independent functions $u_1$ and $u_2$ in $V_{y,z}$, we can apply Lemma~\ref{L-zeroc} as in the previous proofs, and construct yet another function $0 \not = \tilde{u}$ such that $\rho(\tilde{u},y) \ge (2\ell -3)$ and $\rho(\tilde{u},z)\ge 2\,$, contradicting Assertions~(ii).
\end{proof}

\begin{figure}[!ht]
\centering
\begin{subfigure}[t]{.30\textwidth}
\centering
\includegraphics[width=\linewidth]{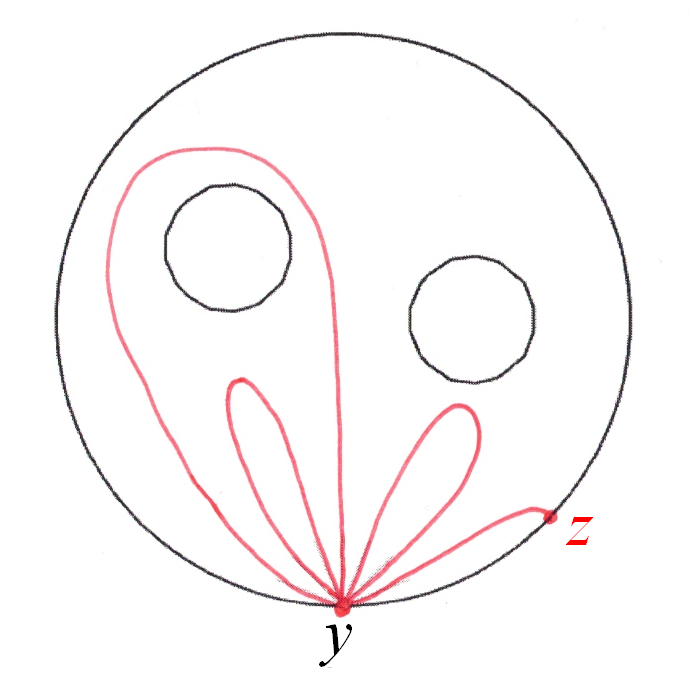}
\caption{}
\end{subfigure}
\begin{subfigure}[t]{.30\textwidth}
\centering
\includegraphics[width=\linewidth]{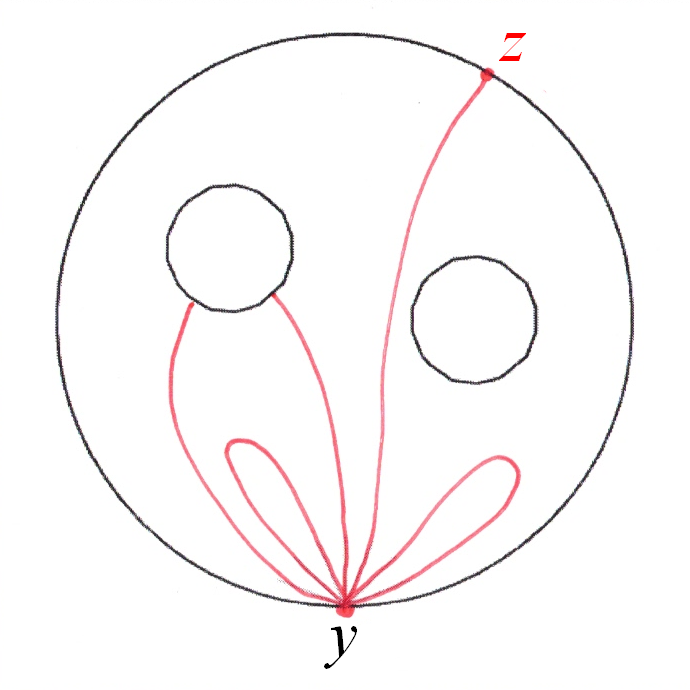}
\caption{}
\end{subfigure}
\begin{subfigure}[t]{.30\textwidth}
\centering
\includegraphics[width=\linewidth]{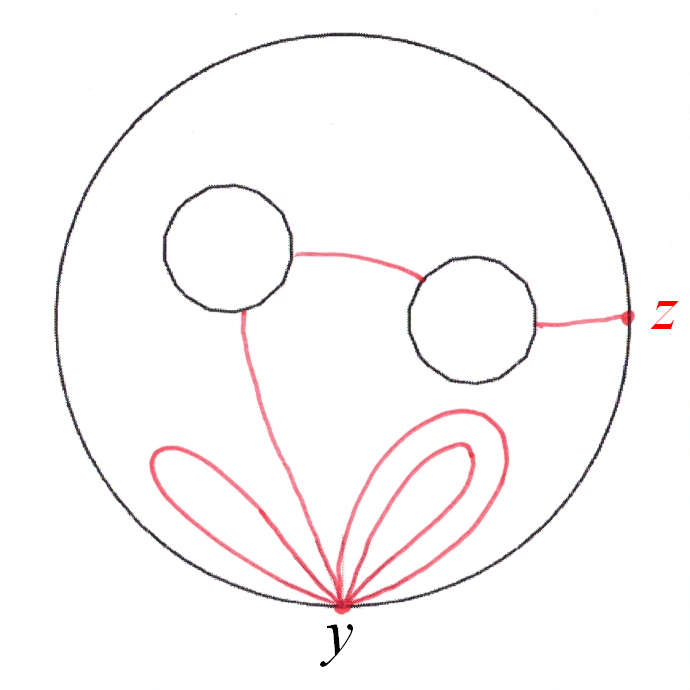}
\caption{}
\end{subfigure}
\caption{$\Omega$ with two holes: some possible nodal patterns for $v_{y,z}$}\label{F-hmn2-Uxy2-1}
\end{figure}

\begin{figure}[!ht]
\centering
\begin{subfigure}[t]{.30\textwidth}
\centering
\includegraphics[width=\linewidth]{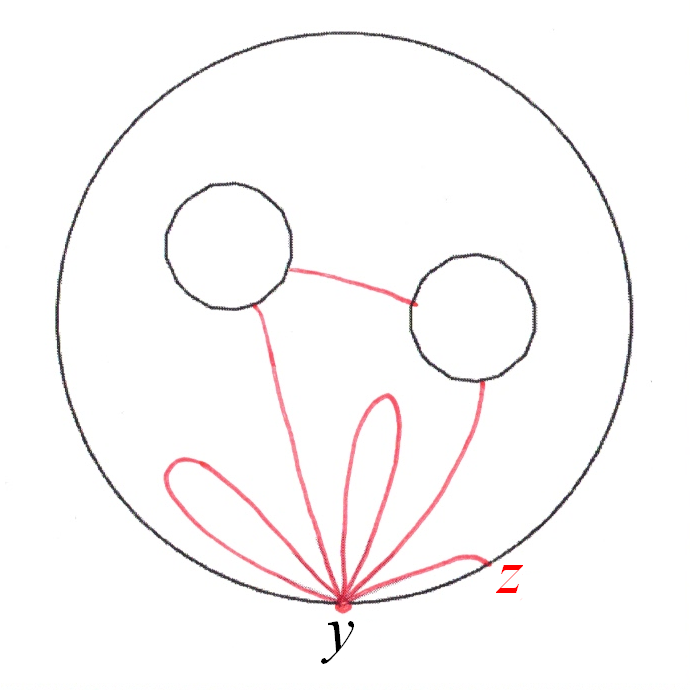}
\caption{}
\end{subfigure}
\begin{subfigure}[t]{.30\textwidth}
\centering
\includegraphics[width=\linewidth]{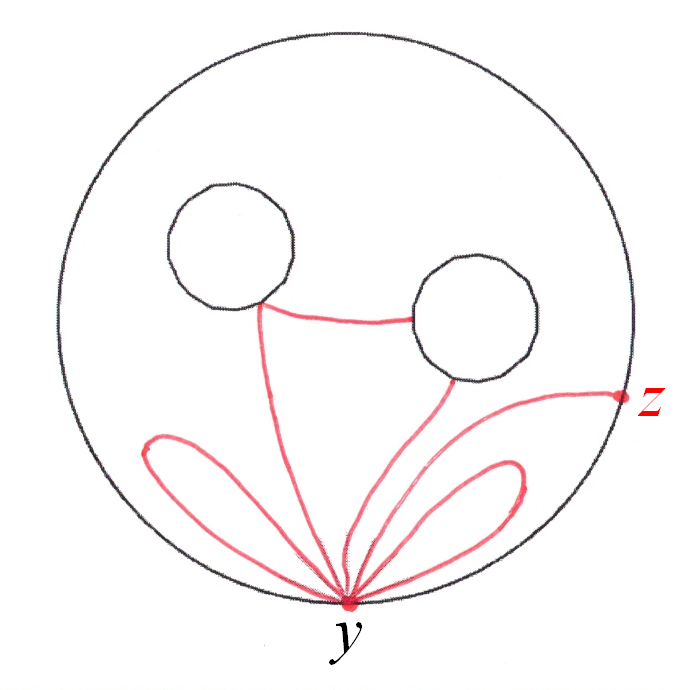}
\caption{}
\end{subfigure}
\begin{subfigure}[t]{.30\textwidth}
\centering
\includegraphics[width=\linewidth]{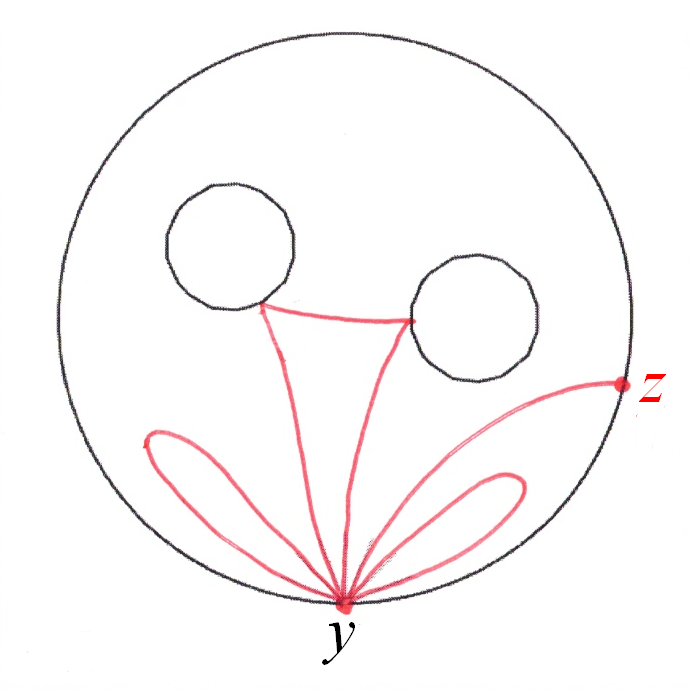}
\caption{}
\end{subfigure}
\caption{$\Omega$ with two holes: some possible nodal patterns for $v_{y,z}$}\label{F-hmn2-Uxy2-2}
\end{figure}

\subsubsection{$\Omega$ with $k$ holes. Structure and combinatorial type of nodal sets in $V_{y,z}$}\label{SSS-hmn-21ds}

 We can adapt the description of the nodal set $\cZ(u)$, $u \in V_{y,z}$ given in  Paragraph~\ref{SSS-hmn-21cs} to the present case (multiple components of $\Gamma$). The ``generalized'' loops or arc will then hit one or several components $\Gamma_j, j \in J(u)\sm \set{1}$. We can also define the \emph{combinatorial type}  $\tau_{y,z}^{u}$ of $u$ with respect to the points $y$ and $z$. \smallskip

Figures~\ref{F-hmn2-Uxy2-1} and \ref{F-hmn2-Uxy2-2} display possible nodal patterns for $ 0 \neq u \in V_{y,z}$  (in these examples, $\ell = 5$, $\rho(y) = 7$, and there are $3$ loops). For these nodal patterns, we have  the combinatorial types

\begin{equation*}
\begin{array}{ll}
\tau^{\ref{F-hmn2-Uxy2-1}}_A  &= \begin{pmatrix}
                       \downarrow &  1 & 2 & 3 & 4 & 5 & 6 & 7 \\
                        1 &  \downarrow & 3 & 2 & 7 & 6 & 5 & 4 \\
                        \end{pmatrix},\\[12pt]
\tau^{\ref{F-hmn2-Uxy2-1}}_B  &= \begin{pmatrix}
                         \downarrow & 1 & 2 & 3 & 4 & 5 & 6 & 7 \\
                          3 & 2 & 1 & \downarrow & 7 & 6 & 5 & 4 \\
                        \end{pmatrix},\\[12pt]
\tau^{\ref{F-hmn2-Uxy2-1}}_C  &= \begin{pmatrix}
                         \downarrow & 1 & 2 & 3 & 4 & 5 & 6 & 7 \\
                          5 & 4 & 3 & 2 & 1 & \downarrow & 7 & 6 \\
                        \end{pmatrix},\\[12pt]
\end{array}%
\end{equation*}
and
\begin{equation*}
\begin{array}{ll}
\tau^{\ref{F-hmn2-Uxy2-2}}_A  &= \begin{pmatrix}
                        \downarrow &  1 & 2 & 3 & 4 & 5 & 6 & 7 \\
                         1 & \downarrow & 5 & 4 & 3 & 2 & 7 & 6 \\
                        \end{pmatrix},\\[12pt]
\tau^{\ref{F-hmn2-Uxy2-2}}_B = \tau^{\ref{F-hmn2-Uxy2-2}}_C &= \begin{pmatrix}
                         \downarrow & 1 & 2 & 3 & 4 & 5 & 6 & 7 \\
                         3 & 2  & 1 & \downarrow & 5 & 4 & 7 & 6 \\
                        \end{pmatrix}.
\end{array}%
\end{equation*}

Lemma~\ref{L-Uxyk} can be reformulated in the abstract setting of Subsection~\ref{SS-hmn2-10} as follows.
For any $0 \neq u \in V_{y,z}$, the projection $\check{\cZ}(u)$ of the nodal set $\cZ(u)$ consists of $(\ell-2)$ continuous simple loops at $\check{y}$ and a continuous simple curve from $\check{y}$ to $\check{z}$. The loops and curve only intersect at $\check{y}$ and may contain points in $\Xi$. If $\xi_j \in \check{\cZ}(u)$, there are exactly two projected nodal semi-arcs at this point, and the point $\xi_j$ is a regular point of $\check{\cD}_u$.  The set $\check{\cZ}(u) \subset \check{\Omega}$ is therefore the wedge sum $\cB_{\check{y},(\ell-2)}^{\check{z}}$ of an $(\ell-2)$-bouquet of loops at $\check{y}$, with a simple arc from $\check{y}$ to $\check{z}$. The loops or the arc may contain points in $\Xi$.  This is illustrated in Figure~\ref{F-hmn2-Uxyk-p}.\smallskip

\begin{figure}[!ht]
\centering
\includegraphics[width=0.6\linewidth]{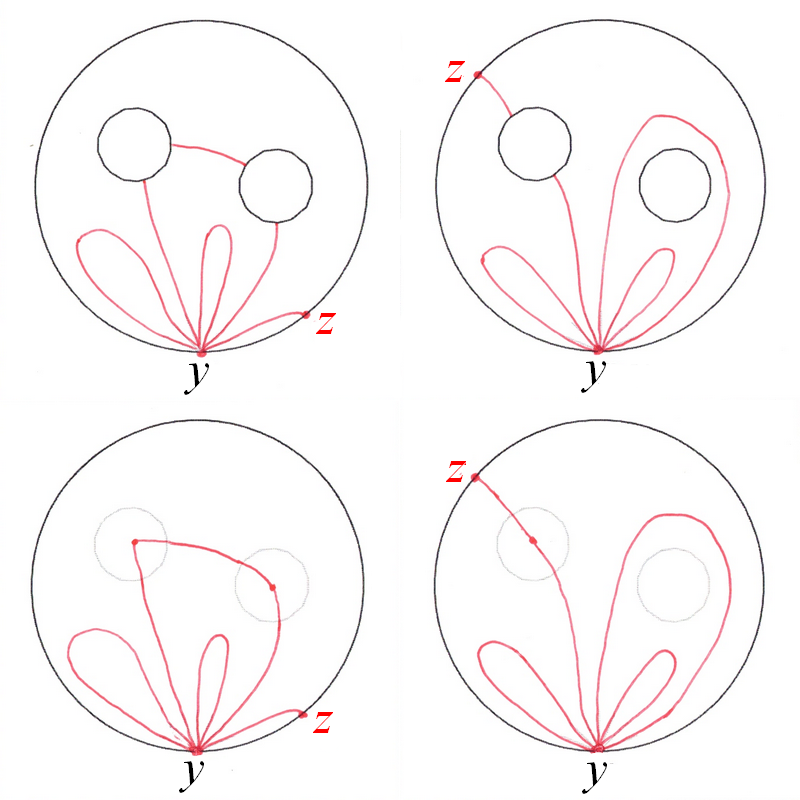}
\caption{Nodal patterns in $\Omega$ and their projections in $\check{\Omega}$}
\label{F-hmn2-Uxyk-p}
\end{figure}
\FloatBarrier

\begin{remark}\label{R-hmn23-2}
From the point of view of \emph{partitions}, see \cite{BoHe2017}, the points in $\Xi$ are not singular points of $\check{\cD}_u$, the projection of the nodal partition $\cD_u$ of $u$.
\end{remark}

\subsection{Analysis of eigenfunctions with one prescribed   boundary singular point}\label{SS-hmn-2N2}

We use the notation of Subsection~\ref{SS-hmn-0N}.  In this subsection, we assume that $U$ is a linear subspace of an eigenspace $U(\lambda)$ of \eqref{E-evp-2bc}, and that for some $\ell \ge 2$,
\begin{equation*}
\left\{
\begin{array}{l}
\sup\set{\kappa(u) \mid 0 \neq u \in U} \le \ell \text{~~and}\\[5pt]
\dim U = (2\ell - 1).
\end{array}%
\right.
\end{equation*}

For $x \in \Gamma_1$, we introduce the subspaces
\begin{equation}\label{E-bha4-2N}
\left\{
\begin{array}{l}
U^1_y = \set{u \in U \mid \rho(u,y) \ge (2\ell -2)},\\[5pt]
U^2_y = \set{u \in U \mid \rho(u,y) \ge (2\ell -3)}.\\[5pt]
\end{array}%
\right.
\end{equation}

According to Lemma~\ref{L-zero2},  $U^1_y \neq \set{0}$. The purpose of this subsection is to investigate the properties of the functions $u \in U^1_y$ or $U^2_y$ --their precise order of vanishing, the structure of their nodal sets-- under the above assumptions on $U$.

\subsubsection{Properties of $U^1_y$ and $U^2_y$}\label{SSS-hmn-2N2a}

\begin{lemma}\label{L-Ux1N}
Let $U$ be a linear subspace of an eigenspace of \eqref{E-evp-2bc} in $\Omega$, with \[ \sup\set{\kappa(u) \mid 0 \not = u \in U} \le \ell,  \text{~~and~~} \dim U = (2\ell - 1).\] Fix some $y \in \Gamma_1$.  For the spaces $U^1_y$ and $U^2_y$ defined in \eqref{E-bha4-2N}, we have
\begin{enumerate}[(i)]
  \item $\dim U^1_y = 1$, \, $\dim U^2_y = 2$ and,
  \item  for any $0 \neq u \in U^2_y$,
  \begin{equation}\label{E-bha4-2aN}
\left\{
\begin{array}{l}
\kappa(u) = \ell \text{~~and~~} \cZ(u) \cup \Gamma(u) \text{~is connected},\\[5pt]
\cS_{\mathrm{i}}(u) = \emptyset,\\[5pt]
\sum_{z \in \cS_{\mathrm{b}}(u) \cap \Gamma_j} \rho(u,z) = 2 \text{~for all~} j \in J(u)\sm \set{1},\\[5pt]
\sum_{z \in \cS_{\mathrm{b}}(u) \cap \Gamma_1} \rho(u,z) = (2\ell - 2) \text{~and, more precisely,}\\[5pt]
\begin{array}{ll}
\text{either} & \rho(u,y) = (2\ell - 2) \text{~and~} \cS_{\mathrm{b}}(u) \cap \Gamma_1 = \set{y},\\[5pt]
\text{or} & \rho(u,y)  = (2\ell-3) \text{~and~} \exists\, z_u \in \Gamma_1 \sm \set{y} \text{~such that~}\\[5pt]
& \cS_{\mathrm{b}}(u) \cap \Gamma_1 = \set{y,z_u}, \text{~with~} \rho(u,z_u) = 1.
\end{array}%
\end{array}%
\right.
\end{equation}

\end{enumerate}
\end{lemma}%

\begin{proof} Assume that $\Gamma$ has $(q+1)$ components, $\Gamma_1, \ldots, \Gamma_{q+1}$, with $x \in \Gamma_1$.\smallskip

Clearly, $\set{0} \neq U_y^1 \subset U_y^2$. Take any $0 \not = u \in U^2_y$. Euler's formula \eqref{E-euler-or2a} can be rewritten as
\begin{equation}\label{E-bha4-2eN}
\begin{array}{ll}
0 \ge \kappa(u) - \ell =& \big( b_0(\cZ(u) \cup \Gamma(u)) - 1 \big)
+ \frac 12\, \sum_{z \in \cS_{\mathrm{i}}(u)} (\nu(z)-2)\\[5pt]
& +\, \sum_{i \in J(u), i\neq 1\,} \frac 12 \, \big( \sum_{z\in \cS _b(u) \cap \Gamma_i}\rho (z) - 2 \big)\\[5pt]
& +\, \frac 12\, \big( \sum_{z\in \cS _b(u) \cap \Gamma_1\,}\rho (z)  - 2\ell +2\big).
\end{array}%
\end{equation}
The first $ |J(u)|+2$  terms in the right-hand side of the equality are nonnegative. Since $\sum_{z\in \cS _b(u) \cap \Gamma_1\,}\rho (z)$ is even, and larger than or equal to $(2\ell -3)$, the last term is nonnegative also. In view of the first inequality, the four terms must vanish. This proves the relations \eqref{E-bha4-2aN}.
\medskip

\noid \emph{Proof that $\dim U^1_y = 1$.~}  Assume that this is not the case. Then, there exist two linearly independent functions $u_1, u_2$ in $U^1_y$ such that $\rho(u_i,y) = (2\ell - 2)$. By Lemma~\ref{L-zeroc}, there would exist a nontrivial linear combination $u$ such that $u \in U^1_u$ and $\rho(u,y) \ge (2\ell - 1)$, a contradiction with \eqref{E-bha4-2aN}. \medskip

\noid \emph{Proof that $\dim U^2_y = 2$.~} Choose some $0 \not = v_1 \in U^1_y$. Clearly $v_1 \in U^2_y$. On the other-hand, given any $y \in \Gamma_1\sm\set{y}$, Lemma~\ref{L-Uxy0} provides a function $u_{y,z}$ belonging to $U^2_y$, not to $U^1_y$, and hence $\dim U^2_y \ge 2$. Choose $0 \not = v_2 \in U^2_y$ orthogonal to $v_1$. Since $\dim U^1_y = 1$, the function $v_2$ satisfies $\rho(v_2,y) = (2\ell - 3)$, and by Proposition~\ref{P-euler-np}, there must exist some $y_2 \in \Gamma$ such that $\rho(v_2,y_2) \ge 1$.  By Lemma~\ref{L-Uxyk}, $\rho(v_2,y_2) = 1$ and $v_2 \in V_{y,y_2}$. The subspace $U^{1,\perp}_y := \set{u \in U^2_y \mid u \perp u_1}$ has dimension at least one. Assume that $\dim U^2_y \ge 3$. Then $\dim U^{1,\perp}_y \ge 2$, and we can find two linearly independent functions $u_1, u_2 \in U^{1,\perp}_y$ such that $\rho(u_i,y) = (2\ell - 3)$. By Lemma~\ref{L-zeroc}, there exists a linear combination $u \in U^{1,\perp}_y$ such that $\rho(u,y) \ge (2\ell - 2)$, a contradiction.
\end{proof}

\begin{remark}\label{R-Ux1-2N}
Up to scaling, there is a uniquely defined orthogonal basis $\set{v_1,v_2}$ of $U^2_y$, with $v_1 \in U^1_y$, $v_2 \in  U^{1,\perp}_y$, and a uniquely defined $y_2 \in \Gamma_1 \sm \set{y}$ such that $\rho(v_2,y_2) = 1$. In view of Lemma~\ref{L-breve}, we can choose $v_1$ such that $\breve{v}_1 > 0$ on $\Gamma_1 \sm \set{y}$, and $v_2$ such that $\breve{v_2} > 0$ on the arc from $y$ to $y_2$ moving counter-clockwise on $\Gamma_1$.
\end{remark}%

\begin{figure}[!ht]
\centering
\begin{subfigure}[t]{.30\textwidth}
\centering
\includegraphics[width=\linewidth]{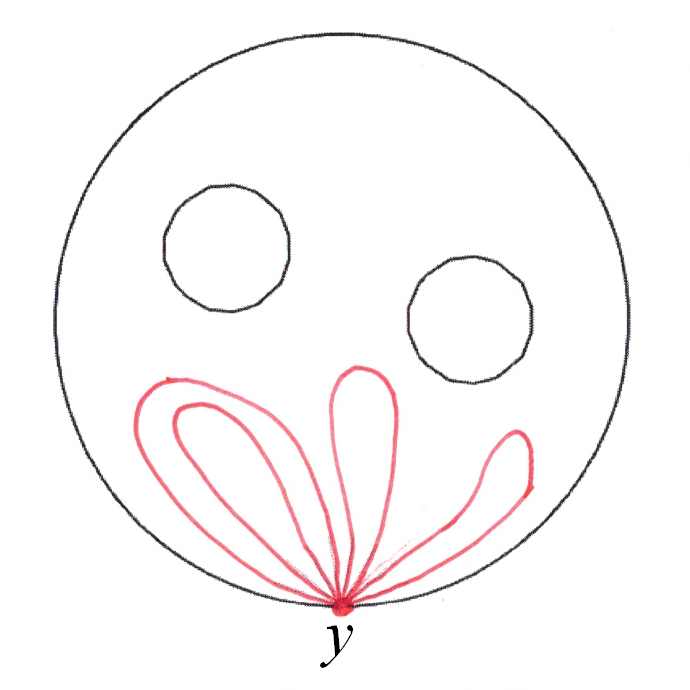}
\caption{}
\end{subfigure}
\begin{subfigure}[t]{.30\textwidth}
\centering
\includegraphics[width=\linewidth]{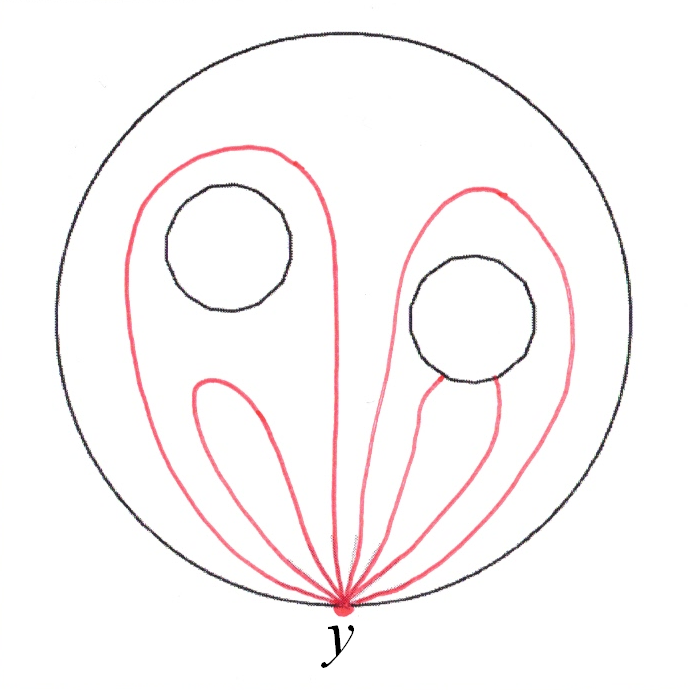}
\caption{}
\end{subfigure}
\begin{subfigure}[t]{.30\textwidth}
\centering
\includegraphics[width=\linewidth]{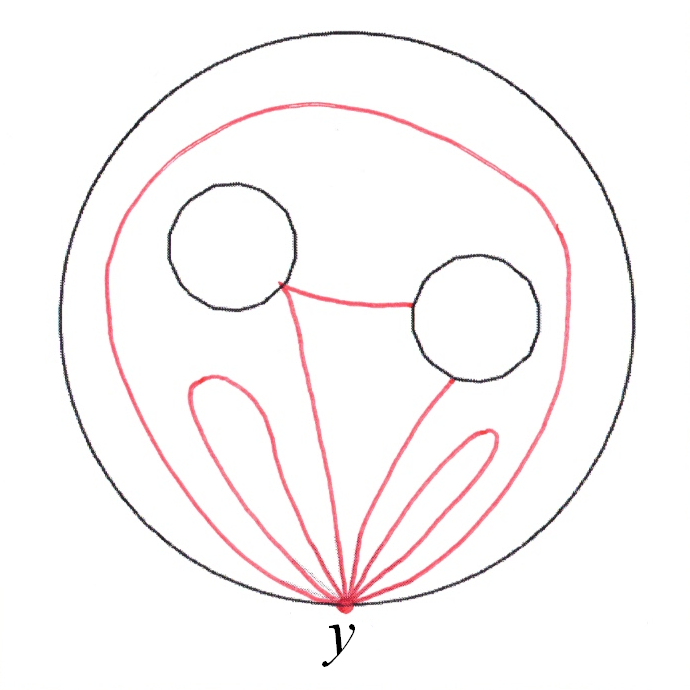}
\caption{}
\end{subfigure}
\caption{Some possible nodal patterns for $0 \not = u \in U^1_y$}\label{F-hmn2-Ux1}
\end{figure}
\FloatBarrier

\subsubsection{Structure and combinatorial type of nodal sets in $U^1_y$ and $U^2_y$}\label{SSS-hmn-2N2as}\phantom{}~

\noid Relations~\eqref{E-bha4-2aN} and an analysis as in Subsection~\ref{SS-hmn-21}, show that the nodal set of any $0\neq u \in U^1_y$ consists of $(\ell -1)$ ``generalized'' nodal loops at the point $y$, and that these loops do not intersect away from $y$.  In the abstract setting of Subsection~\ref{SS-hmn2-10},  for any $0 \neq u \in U^1_y$, the projection $\check{\cZ}(u)$ of the nodal set $\cZ(u)$ consists of $(\ell-1)$ continuous loops at $\check{y}$. The loops only intersect at $\check{y}$, and may contain points in $\Xi$. If $\xi_j \in \check{\cZ}(u)$, there are exactly two projected nodal semi-arcs at this point. It follows that $\xi_j$ is a regular point of $\check{\cD}_u$.  The set $\check{\cZ}(u)$ is an $(\ell - 1)$-bouquet of loops $\cB_{\check{y},(2\ell -2)}$ at $\check{y}$. \smallskip

Adapting the description given in Paragraph~\ref{SSS-hmn-21as},  for $0 \neq u \in U_y^1$, we define the \emph{combinatorial type} $\tau_{y}^{u}$ of the nodal set $\cZ(u)$ with respect to $y$.   This is a map from $L_{(2\ell - 2)}$ to itself.
\smallskip

Some possible nodal patterns  for $u \in U_y^1$ are displayed in Figure~\ref{F-hmn2-Ux1},  where $\ell = 5$, and $\rho(y) = 8$.  The corresponding combinatorial types  are
\begin{equation*}
\begin{array}{l}
 \tau^{\ref{F-hmn2-Ux1}}_A  = \begin{pmatrix}
                          1 & 2 & 3 & 4 & 5 & 6 & 7 & 8\\
                          2 & 1 & 4 & 3 & 8 & 7 & 6 & 5\\
                        \end{pmatrix},\\[12pt]
\tau^{\ref{F-hmn2-Ux1}}_B = \begin{pmatrix}
                          1 & 2 & 3 & 4 & 5 & 6 & 7 & 8\\
                          4 & 3 & 2 & 1 & 8 & 7 & 6 & 5 \\
                        \end{pmatrix},\\[12pt]
\tau^{\ref{F-hmn2-Ux1}}_C = \begin{pmatrix}
                          1 & 2 & 3 & 4 & 5 & 6 & 7 & 8\\
                          8 & 3 & 2 & 5 & 4 & 7 & 6 & 1 \\
                        \end{pmatrix}.
\end{array}%
\end{equation*}

\noid  If $u \in U^2_y$ and $u \not \in U^1_y$, there exists a unique $z_u \in \Gamma_1$, such that $z_u \neq y$ and $\cS_{\mathrm{b}}(u) \cap \Gamma_1 = \set{y,z_u}$, with $\rho(u,y) = (2\ell - 3)$, $\rho(u,z_u)=1$. Furthermore, $V_{y,z_u}=[u]$. The nodal set $\cZ(u)$  and its \emph{combinatorial type} $\tau_{y,z_u}^{u}$  are described in Paragraph~\ref{SSS-hmn-21d}. Projecting  $\cZ(u)$ to $\check{\Omega}$, $\check{\cZ}(u)$ is the wedge sum $\cB^{\check{z}_u}_{\check{y},(\ell -2)}$ of a simple arc from $\check{y}$ to $\check{z}_u$ with an $(\ell-2)$-bouquet of loops at $\check{y}$.

\subsection{Application of the previous analysis}\label{SS-hmn-2N4}
Fix some $y \in \Gamma_1$. We now apply the analysis of Subsections~\ref{SS-hmn-21} and \ref{SS-hmn-22} to investigate the limits $v_{y,z} \in U^2_y\sm U^1_y$, when $z$ tends to $y$ on $\Gamma_1$, clockwise or anti-clockwise. The notation are the same as in  Subsection~\ref{SS-hmn-21}.\smallskip

We choose a basis $\set{v_1,v_2}$ of $U^2_x$ as described in Remark~\ref{R-Ux1-2N}. In particular, $\rho(v_1,y)=(2\ell -2)$,  $v_1 \perp v_2$ in $L^2(\Omega)$,  $\rho(v_2,y)=(2\ell-3)$, there exists $y_2 \in \Gamma_1\sm \set{y}$ such that $\rho(v_2,y_2)= 1$, and $\cS(v_2) \cap \Gamma_1 = \set{y,y_2}$.
Recall the definition of the functions $\breve{v}_i$ on $\Gamma$,
\begin{equation}\label{E-bha4-3dN}
\breve{v}_i := \left\{
\begin{array}{ll}
\partial_{\nu}v_i &\quad \text{in the Dirichlet case,}\\[5pt]
v_i|_{\Gamma} &\quad \text{in the Robin case.}\\[5pt]
\end{array}
\right.
\end{equation}

According to Lemma~\ref{L-breve}, the function $\breve{v}_1$ vanishes only at $y$ and does not change sign on $\Gamma_1$. The function $\breve{v}_2$ does not vanish on $\Gamma_1 \sm \set{y,y_2}$, and changes sign when crossing $y_2$  and $y$ along $\Gamma_1$.\smallskip

Let $\gamma : [0,2\pi] \to \Gamma_1$ be a parametrization such that $\gamma(0)=\gamma(2\pi) = x$.  Given any $y\in \Gamma_1 \sm \set{y}$, there exists a function $u_{y,z}$ which satisfies \eqref{E-bha4-2aN}(ii), and this function is uniquely defined up to multiplication by a nonzero scalar. In the Dirichlet case, this function is characterized by the fact that $\breve{u}_{y,z} = \partial_{\nu}u_{y,z}|_{\Gamma_1}$ only vanishes at $y$ and $z$. In the Robin case, it is characterized by the fact that $\breve{u}_{y,z} =u_{y,z}|_{\Gamma_1}$ only vanishes at $y$ and $z$. Up to a constant factor, we may choose
\begin{equation}\label{E-bha4-4abN}
u_{y,z} = a(z) \, v_1 + b(z) \, v_2 ,
\end{equation}
with
\begin{equation}\label{E-bha4-4cN}
\left\{
\begin{array}{l}
a(z) = -\, \breve{v}_2(z)\, \big( \breve{v}_1^2(z) + \breve{v}_2^2(z) \big)^{- \frac 12} ,\\[5pt]
b(z) = \breve{v}_1(z) \, \big( \breve{v}_1^2(z) + \breve{v}_2^2(z) \big)^{- \frac 12},\\[5pt]
\end{array}
\right.
\end{equation}
where $\breve{v}_1, \breve{v}_2$ are defined in \eqref{E-bha4-3dN}. Then, since $\breve{v}_1$ is positive on $\Gamma_1 \sm \set{y}$, there exists a unique $\theta(z) \in (0,\pi)$ such that $\cos(\theta(z)) = a(z)$ and $\sin(\theta(z)) = b(z)$. Defining
\begin{equation}\label{E-bha4-4dN}
w_{\theta} = \cos\theta \, v_1 + \sin\theta \, v_2,
\end{equation}
we have $u_{y,z} = w_{\theta(z)}$. Conversely, according to the proof of Lemma~\ref{L-Ux1N}, any function $w_{\theta}$ has exactly two singular points on $\Gamma_1$, the point $y$ and some other point $z_{\theta} \not = y$. Note that the point $z$ determines the eigenfunction $u_{y,z}$ uniquely (up to scaling) and vice versa. It follows that we have a continuous, bijective map $(0,2\pi) \ni t \mapsto \theta(\gamma(t)) \in (0,\pi)$. This map is strictly monotone (we can assume that it is increasing), and provides a diffeomorphism from $(0,2\pi)$ to $(0,\pi)$, with limits $0$ and $\pi$ respectively. Otherwise stated, the function $u_{y,\gamma(t)}$ defined in \eqref{E-bha4-4abN} tends to $v_1$ when $t$ tends to $0$ and to $-v_1$ when $t$ tends to $2\pi$. There exists $t_2$ such that $\gamma(t_2) = y_2$, and hence $\theta(y_2) = \frac \pi 2$.  We have proved the following property.

\begin{property}\label{P-hmnN}
The function $u_{y,z}$ defined in \eqref{E-bha4-4abN} tends to $v_1$ when $z \neq y$ tends to $y$ clockwise,  and to $-v_1$ when when $z \neq y$ tends to $y$ counter-clockwise.
\end{property}%

\subsubsection{Proof that $\mult(\lambda_k) \le (2k -2)$ for $k \ge 3$, general case}\label{SS-hmn-26}\phantom{}

Under the assumption that $\dim U(\lambda_k) = (2k-1)$, the arguments in the simply connected case only use Euler's formula applied to nodal partitions $\cD_u$, Jordan's separation theorem, the structure of the nodal sets $\cZ(v_1)$ (a bouquet of loops at $y$) and $\cZ(v_2)$ (the wedge sum of an arc from $y$ to the boundary $\Gamma_1$ with one or two bouquets  of loops at $y$), and the fact that the combinatorial type  $\tau_{y,a_{\theta},w_{\theta}}$ of the nodal sets $\cZ(w_{\theta})$ is constant for $\theta \in (0,\pi)$.\smallskip

In the general case, Euler's formula leads to a similar structure for the nodal sets $\cZ(v_1)$, $\cZ(v_2)$, and $\cZ(w_{\theta})$, with ``generalized'' loops and arcs. We can now look at the projection of these sets to $\check{\Omega}$ as in Section~\ref{SS-hmn2-10}. As observed in Remark~\ref{R-hmn23-2}, the only singular points of the projected sets $\check{\cZ}(u)$ (or partitions $\check{\cD}_u$), $u \in \set{v_2, w_{\theta}}$, are the points $\check{y}$ and $ \check{z}_{\theta}$, and the combinatorial type of $\check{\cZ}(w_{\theta})$ is constant for $\theta \in (0,\pi)$. Since $\check{\Omega}$ is simply connected, we can now apply the same arguments as in the simply connected case.\medskip

This completes the proof of Theorem~\ref{T-hmn-bh1} for general $\Cty$ bounded domains.

\chapter[Simply Connected Domains: $\mult(\lambda_k) \le (2k-3)$]{Simply Connected Plane Domains: the Estimate\\ $\mult(\lambda_k) \le (2k-3)$ for $k \ge 3$}\label{Ch-scpdsb}

\section{Introduction}\label{S-hmn3-pre}

In Section~\ref{S-hmn2}, we have established that the estimate, $\mult(\lambda_k) \le (2k-2)$ for all $k \ge 3$, is valid for any $\Cty$ bounded domain $\Omega$, see Theorem~\ref{T-hmn-bh1}. In \cite[Theorem~B, p.~1172]{HoMN1999}, the authors state that  this estimate can be improved to $\mult(\lambda_k) \le (2k-3)$ for all $k\ge 3$.\smallskip

The purpose of this chapter is  to prove Theorem~\ref{T-hmn-bh}, Assertion~(ii), namely,

\begin{theorem}\label{T-hmn-bh2}
Let  $\Omega$ be a \emph{simply connected} $\Cty$ bounded domain in $\R^2$. The multiplicities of the eigenvalues of the operator  $-\Delta + V$ in $\Omega$, with the Dirichlet, Neumann or $h$-Robin boundary condition, satisfy the estimate $\mult(\lambda_k) \le (2k-3)$ for any $k \ge 3$.
\end{theorem}%

The proof  of the theorem is by contradiction. Introducing the following assumptions,  to hold throughout this chapter, we shall reach a contradiction.

\begin{assumptions}\label{A-hmn3-0}\phantom{}
\begin{enumerate}[$\diamond$]
\item $\Omega$ is a simply connected, $\Cty$, bounded domain in $\R^2$, and we let $\Gamma := \partial \Omega$.
\item For some $k \ge 3$, the $k$-th eigenvalue $\lambda_k$ of the eigenvalue problem \eqref{E-evp-2bc}--\eqref{E-evp-bc} has multiplicity $(2k-2)$, and we let $U := U(\lambda_k)$ be the corresponding eigenspace.
\end{enumerate}
\end{assumptions}%

More precisely, the proof of Theorem~\ref{T-hmn-bh2} is organized as follows.\smallskip

\noid In Section~\ref{S-hmn3}, we prove the existence of certain functions with prescribed singularities. More precisely, given any $x \in \Omega$ and $y \in \Gamma$, we introduce the linear subspaces
\index{2-W@$W_x$} \index{2-U@$U_y$}
\begin{equation*}
\left\{
\begin{array}{ll}
W_x & := \set{u \in U \mid \nu(u,x) \ge 2k - 2}\\[5pt]
U_y & := \set{u \in U \mid \rho(u,y) \ge 2k - 3}.
\end{array}
\right.
\end{equation*}
As it turns out, they have dimension $1$. Furthermore, for any $y \in \Gamma$ and $u \in U_y$, either $\rho(u,y) = (2k-2)$ and $\cS(u) = \set{y}$, or $\rho(u,y) = (2k-3)$ and $\cS(u) = \set{y,z(y)}$ for some $z(y) \neq y$ on $\Gamma$ \index{2-zy@$z(y)$} (Lemmas~\ref{L-L37} and \ref{L-L32}).
We introduce the following subsets of $\Gamma$,
\index{1-Gamma@$\Gamma$!$\Gamma_{(2k-3)},\Gamma_{(2k-2)}$}
\begin{equation*}
\left\{
\begin{array}{l}
\Gamma_{(2k -3)} := \set{y \in \Gamma \mid \forall ~0 \neq u \in U_y,~~\rho(u,y) = 2k - 3}\\[5pt]
\Gamma_{(2k -2)} := \set{y \in \Gamma \mid \forall ~0 \neq u \in U_y,~~\rho(u,y) = 2k - 2}.
\end{array}
\right.
\end{equation*}

We carefully study the properties of the maps $\Omega \ni x \mapsto [W_x] \in \bP(U)$ and $\Gamma \ni y \mapsto [U_y]$ (the one-dimensional linear subspaces viewed as points in the projective space of $U$), and of the sets $\Gamma_{(2k -3)}$ and $\Gamma_{(2k -2)}$, proving in particular that $\Gamma_{(2k -3)}$ is open and $\Gamma_{(2k -2)}$ finite  (Lemma~\ref{L-L33}).  This subsection contains three key lemmas. In Lemma~\ref{L-L33b}, we investigate the global behavior of $[U_y]$ and its associated singular point when $y \in \Gamma_{(2k-3)}$. In Lemma~\ref{L-L38} we prove that $[W_x]$ tends to $[U_y]$ when $x \in \Omega$ tends to $y \in \Gamma$. In Lemma~\ref{L-L33c}, we describe the global behavior of the combinatorial types of the $[U_y]$, $y \in \Gamma$. As a consequence, we obtain that $\#\big( \Gamma_{(2k -2)} \big)$ must be an even integer.\smallskip

\noid In Section~\ref{S-gam0}, analyzing the behavior of $[W_x]$ when $x$ is close  to some $y \in \Gamma$,  we conclude that Assumptions~\ref{A-hmn3-0} lead to a contradiction in the case $\Gamma_{(2k -2)} = \emptyset$ (Lemma~\ref{L-gam0-5}).\smallskip

\noid In Section~\ref{S-hmn9}, the analysis of the global behavior of $[W_x]$ in a neighborhood of $\Gamma$ shows that the Assumptions~\ref{A-hmn3-0} lead to a contradiction when $\Gamma_{(2k -2)} \neq \emptyset$, see Lemma~\ref{L-gam0-7}.\smallskip

\noid We can finally conclude that Assumptions~\ref{A-hmn3-0} lead to a contradiction, and hence that $\mult(\lambda_k) \le (2k-3)$ for $k\ge 3$.\smallskip

\noid Section~\ref{S-hmn2L} describes the \emph{labeling of nodal domains} of certain eigenfunctions. This notion is closely related to the notion of combinatorial type and used to prove that the combinatorial types of two eigenfunctions are different. This notion also appears in Section~\ref{S-llnd}.\smallskip

\noid  In Section~\ref{S-hmn32b}, we give a detailed proof of Lemma~\ref{L-L34a}. We also prove other statements which appear in \cite{HoMN1999} but which we do not actually use for our proof of Theorem \ref{T-hmn-bh2}.

\section[Properties of $\lambda_k$-Eigenfunctions]{Properties of $\lambda_k$-Eigenfunctions under Assumptions~5.2}\label{S-hmn3}

\subsection{Preamble}\label{SS-hmn30}

This section is devoted to establishing properties of $\lambda_k$-eigen\-functions under Assumptions~\ref{A-hmn3-0} (which will systematically be repeated in the lemmas).  These properties will be used in Sections~\ref{S-gam0} and \ref{S-hmn9}, leading to a contradiction, and showing that $\mult(\lambda_k)$ cannot be equal to $(2k-2)$ for $k \ge 3$.\smallskip

The assumption that the domain $\Omega$ is simply connected is actually not necessary in this Section~\ref{S-hmn3}, and only meant to simplify the proofs.  For the proofs in the general case, use arguments similar to those given in  Section~\ref{S-hmn2N}.\medskip

We use the notation of Subsection~\ref{SS-hmn-0N}. For later purposes, we also introduce the following notation.

\begin{notation}\label{N-hmn3-arcs}
Given two points $y_1 \neq y_2 \in \Gamma $, we denote by $\cA(y_1,y_2)$\index{2-A @$\cA(y_1,y_2)$} the open arc from $y_1$ to $y_2$, moving counter-clockwise. Given $y \in \Gamma $ and some number $a$ smaller than half the length of $\Gamma $, $\cA(y;a)$ denotes the arc centered at $y$, with length $2a$, taken counter-clockwise. In both cases, we use  the mathematical  symbols  $[$ and $]$ to denote the closed or semi-closed arcs.
\end{notation}%

According to Courant's nodal domain theorem, $\sup\set{\kappa(u) \mid u \in U} \le k$, so that the number $\ell$ in \eqref{E-euler-or} is now $k$. \smallskip

For any $0 \neq u \in U$,  Euler's formula \eqref{E-euler-or2} becomes,
\begin{equation}\label{E-hmn-3a12}
\left\{
\begin{array}{ll}
0 \ge \kappa(u) - k & = \big[ b_0\big( \cZ(u) \cup \Gamma\big) -1 \big]  + \frac 12 \, \sum_{z\in \cS_{\mathrm{i}}(u)} \big(  \nu(u,z) \, -2\big)\\[5pt]
& \quad + \frac 12 \,   \sum_{z\in \cS_{\mathrm{b}}(u)\,}\rho(u,z) \, - (k -1).
\end{array}
\right.
\end{equation}

In the next subsections, we analyze eigenfunctions with prescribed singular points, under Assumptions~\ref{A-hmn3-0}.

\subsection{Eigenfunctions with one prescribed interior singular point}\label{SS-hmn-31i}~\\
For $x \in \Omega$, define the subspace \index{2-W@$W_x$}
\begin{equation}\label{E-hmn3-n6a}
W_x := \set{u \in U \mid \nu(u,x) \ge 2k - 2}.
\end{equation}

In view of Assumptions~\ref{A-hmn3-0}, Lemma~\ref{L-zeroi} implies that $W_x \neq \set{0}$. \smallskip

\subsubsection{Properties of $W_x$}\label{SSS-hmn-31ia}

\begin{lemma}\label{L-L37}
Assume that $\Omega$ is simply connected. Let $U := U(\lambda_k)$ for some $k \ge 3$. Assume that $\dim U = (2k-2)$. [Assumptions~\ref{A-hmn3-0}].  Let $x \in \Omega$. Then, the subspace
\[W_x = \set{u \in U \mid \nu(u,x) \ge 2k - 2}\]
has the following properties.
\begin{enumerate}[(i)]
  \item  The dimension of $W_x$ is $1$.
  \item For all $0 \neq u \in W_x$,
  \begin{equation}\label{E-hmn3-n6b}
  \left\{
\begin{array}{l}
\kappa(u) = k,\\[5pt]
\cZ(u) \text{~~is connected,}\\[5pt]
\cS_{\mathrm{i}}(u) = \set{x} \quad\text{and~~} \nu(u,x) = 2(k -1),\\[5pt]
\sum_{z \in \cS_{\mathrm{b}}(u)} \rho(u,z) \in \set{0,2}.\\[5pt]
\end{array}
\right.
  \end{equation}
\item If $w_x$ is a generator of $W_x$, the map $\Omega \ni x \mapsto [w_x] \in \bP(U)$ is $C^{\infty}$.
\end{enumerate}%
\end{lemma}

\begin{proof} We already know that $\dim W_x \ge 1$. \smallskip

\emph{Assertion~(ii).} The assumptions of the lemma and \eqref{E-hmn-3a12} imply that
\begin{equation}\label{E-hmn-312}
\begin{split}
0 \geq  \kappa(u) - k & = \big( b_0(\cZ(u) \cup \Gamma) -2\big) + \frac 12\, \sum_{z \in \cS_{\mathrm{i}}(u), z\neq x } (\nu(u,z)-2) \\
&\hspace{5mm} + \frac 12 \, \big( \nu(u,x) - 2k +2\big)  + \frac 12 \sum_{z\in \cS_{\mathrm{b}}(u) }\rho (u,z) .
\end{split}
\end{equation}
The terms in the right-hand side are nonnegative, except possibly the first one. The inequality implies that $b_0(\cZ(u) \cup \Gamma) \le 2$. We now consider two cases.\smallskip

\noid If $b_0(\cZ(u)\cup \Gamma) = 2$, the terms in the right-hand side are nonnegative, with a nonpositive sum. They must all vanish so that $\kappa(u) = k$, $\cS_{\mathrm{i}}(u) = \set{x}$, $\cS_{\mathrm{b}}(u) = \emptyset$, and $\nu(u,x) = 2(k - 1)$. In this case, $\cZ(u) \cap \Gamma = \emptyset$.  It follows that  $\cZ(u)$ is connected. Indeed, since $\cZ(u)$ does not hit $\Gamma$, the nodal arcs emanating from $x$ must form loops at $x$. These loops can only intersect each other at $x$ because $\cS_{\mathrm{i}}(u) = \set{x}$.
\index{2-Z@$\cZ(u)$!$\cZ_x(u)$}
The  component $\cZ_x(u)$ of $x$ in $\cZ(u)$ is a $(k-1)$-bouquet of  loops at $x$ whose complement has $k$ components. Since $k$ is the maximal possible number of nodal domains, this implies that $\cZ_x(u) = \cZ(u)$.\smallskip

\noid If $b_0(\cZ(u)\cup \Gamma) = 1$, the nodal set $\cZ(u)$ must hit $\Gamma$, which implies the inequality $\sum_{z\in \cS_{\mathrm{b}}(u) }\rho (z) \ge 2$ (use Proposition~\ref{P-euler-np}). Re-arranging the inequality \eqref{E-hmn-312}, we conclude that $\kappa(u) = k$, $\cS_{\mathrm{i}}(u) = \set{x}$, $\nu(u,x) = 2(k - 1)$, $\sum_{z\in \cS_{\mathrm{b}}(u) }\rho (z) = 2$,  and that $Z(u) \cup \Gamma$ is connected. The component $\cZ_x(u)$ of $x$ in $\cZ(u)$ is either a $(k-1)$-bouquet of loops at $x$, or consists of a $(k-2)$-bouquet of loops and two simple arcs from $x$ to the boundary. Away from $x$, the arcs do not intersect the loops and do not intersect each other except possibly on $\Gamma$. In the first case, the complement of the bouquet of loops has $k$  components, and the two points at which $\cZ(u)$ hits $\Gamma$ would be linked by a simple arc (possibly a loop if these points coincide). We would have too many nodal domains. This means that the first case does not occur. In the second case, the complement of $\cZ_x(u)$  has $k$ components, the maximal possible number. As above, this implies that $Z(u)$  is connected. We have proved Assertion~(ii).\medskip

\emph{Assertion~(i).~} Lemma~\ref{L-zeroc} and \eqref{E-hmn3-n6b} imply that $\dim W_x \le 2$. Assume by contradiction that $\dim W_x = 2\,$.  We use a \emph{rotating function argument} similar to the one used in \S~\ref{SSS-h2n-s2r}.\smallskip
\index{Rotating function argument}
As in Proposition~\ref{propertyD2}, we can choose a basis $\set{v_1,v_2}$ of $W_x$ such that, in local polar coordinates centered at $x$,
\begin{equation*}
\left\{
\begin{array}{l}
v_1 = r^{k -1} \sin ((k -1)\omega) +\mathcal O (r^{k}),\\
v_2 = r^{k -1} \cos ((k -1)\omega) +\mathcal O (r^{k}).
\end{array}
\right.
\end{equation*}

Introducing the family of functions
\begin{equation*}
w_\theta = \cos((k -1) \theta) \,v_1 - \sin((k -1) \theta) \, v_2,
\end{equation*}
and letting $\theta$ tend to $0$ or $\frac{\pi}{(k - 1)}$, we can follow the arguments given in the proofs of  Properties~\ref{propertyD3} and \ref{propertyD4} to reach a contradiction.\medskip

\emph{Assertion~(iii).~}  Same proof as for Property~\ref{P-h2n-laa}.
\end{proof}

\subsubsection{Structure and combinatorial type of nodal sets in $W_x$}\label{SSS-hmn-31ias}


 In view of Assertion~(ii), one can describe the possible nodal patterns for a generator $w_x$ of $W_x$.  There are two cases as displayed in Figure~\ref{F-hmn3-h0-L37}.
\begin{enumerate}
   \item Either $\cZ(w_x)$ consists of $(k - 1)$ loops at $x$ which do not intersect away from $x$, and do not hit $\Gamma$.
   \item Or $\cZ(w_x)$ consists of
   \begin{itemize}
     \item[$\diamond$] $(k - 2)$ loops at $x$ which do not hit the boundary, and
     \item[$\diamond$] two arcs emanating from $x$ and hitting $\Gamma$ at points $y_1 \neq y_2$, such that $\rho(w_x,y_i) = 1$ or, possibly, at one point $y$, with $\rho(w_x,y) = 2$.
   \end{itemize}
Furthermore, the loops at $x$ and the arcs from $x$ to the boundary are pairwise disjoint away from $x$,  except possibly at the boundary.
We then have a ``generalized'' nodal loop at $x$ which consists of the two arcs, and a portion of the boundary.
\end{enumerate}

\begin{figure}[!ht]
\centering
\begin{subfigure}[t]{0.25\textwidth}
\centering
\includegraphics[width=\linewidth]{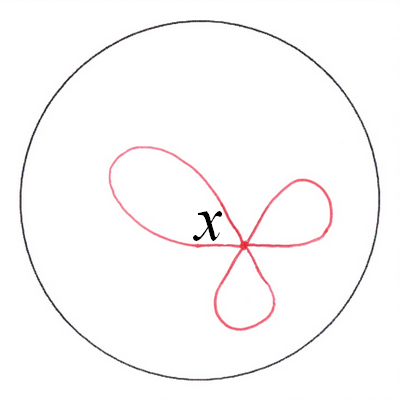}
\caption{}
\end{subfigure}
\begin{subfigure}[t]{.25\textwidth}
\centering
\includegraphics[width=\linewidth]{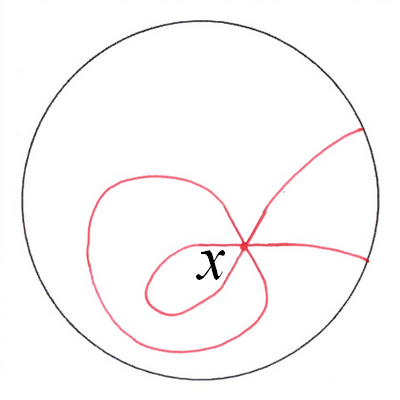}
\caption{}
\end{subfigure}
\begin{subfigure}[t]{.25\textwidth}
\centering
\includegraphics[width=\linewidth]{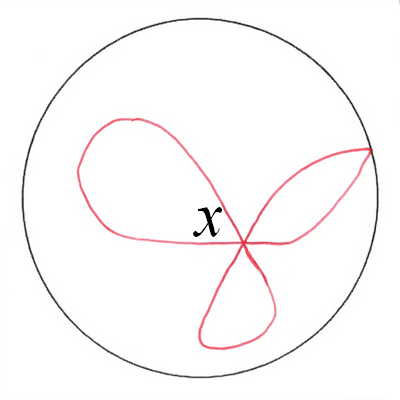}
\caption{}
\end{subfigure}
\caption{$\Omega$ simply connected, nodal patterns for $w_x \in W_x$ ($k=4$)}\label{F-hmn3-h0-L37}
\end{figure}

\FloatBarrier

\begin{remark}\label{R-Fgene2}
The \emph{nodal patterns} displayed in Figure~\ref{F-hmn3-h0-L37} are valid for both the Dirichlet and Robin boundary conditions. Unless otherwise stated, this remark applies to all figures of this section.
\end{remark}

\subsection{Eigenfunctions with one prescribed boundary singular point}\label{SS-hmn-31b}~\\
For $y \in \Gamma$, we introduce the subspace \index{2-U@$U_y$}
\begin{equation}\label{E-hmn3-n10a}
U_y := \set{u \in U \mid \rho(u,y) \ge 2k - 3}.
\end{equation}

In view of Assumptions~\ref{A-hmn3-0},  Lemma~\ref{L-zero1} implies that $U_y \neq \set{0}$.

\subsubsection{Properties of $U_y$}\label{SSS-hmn-31ba}

\begin{lemma}\label{L-L32}
Assume that $\Omega$ is simply connected. Let $U := U(\lambda_k)$ for some $k \ge 3$. Assume that $\dim U = (2k-2)$. [Assumptions~\ref{A-hmn3-0}]. Let $y \in \Gamma$. Then, the subspace
 \[U_y = \set{u \in U \mid \rho(u,y) \ge 2k - 3}\]
 has the following properties.
\begin{enumerate}[(i)]
  \item  The dimension of $U_y$ is $1$.
  \item For all $0 \neq u \in U_y$,
  \begin{equation}\label{E-hmn3-n10b}
  \left\{
\begin{array}{l}
\kappa(u) = k \text{~~and~~} \cZ(u), ~\cZ(u) \cup \Gamma \text{~are connected,}\\[5pt]
\cS_{\mathrm{i}}(u) = \emptyset,\\[5pt]
\sum_{z \in \cS_{\mathrm{b}}(u)\,} \rho(u,z) = 2k - 2.\\[5pt]
\end{array}
\right.
  \end{equation}
Furthermore,
  \begin{equation}\label{E-hmn3-n10d}
\left\{
\begin{array}{ll}
\text{either} & \rho(u,y) = 2k -2 \text{~~and~~} \cS_{\mathrm{b}}(u) = \set{y},\\[5pt]
\text{or} & \rho(u,y) = 2k -3 \text{~~and~~} \cS_{\mathrm{b}}(u) = \set{y,z(y)},\\[5pt]
&\quad \text{for some~} z(y) \in \Gamma, z(y) \neq y, \text{~with~} \rho(u,z(y)) = 1.
\end{array}
\right.
\end{equation}
  \item If $u_y$ denotes a generator of $U_y$, then the map $\Gamma \ni y \mapsto [u_y] \in \bP(U)$ is $C^{\infty}$.
\end{enumerate}
\end{lemma}

\begin{proof} We already know that $\dim U_y \ge 1$. \smallskip

\noi \emph{Assertion~(ii)}.  Choose a function $0 \neq u \in U_y$, and apply the inequality \eqref{E-hmn-3a12} to obtain,
\begin{equation}\label{E-hmn-314}
\begin{split}
0 \geq  \kappa(u) - k & = \big( b_0(\cZ(u) \cup \Gamma) -1 \big) + \frac 12\, \sum_{z \in \cS_{\mathrm{i}}(u)} (\nu(u,z)-2) \\
&\hspace{5mm} + \frac 12 \big( \sum_{z\in \cS_{\mathrm{b}}(u)}\rho (u,z) - 2k +2\big) .
\end{split}
\end{equation}
Since $\rho(u,y) \ge 2k -3$, Proposition~\ref{P-euler-np} implies that the last term in \eqref{E-hmn-314} is nonnegative; all the terms in the right-hand side are nonnegative, with nonpositive sum, and hence they must all vanish. This proves \eqref{E-hmn3-n10b}. Looking at the two possible cases, $\rho(u,y) = (2k - 2)$ or $(2k -3)$, we obtain \eqref{E-hmn3-n10d}. Assertion~(ii) is proved.\quad \qedc \medskip

\noi \emph{Assertion~(i).~} We already know that $\dim U_y  \ge 1$. Assume that there exist at least two linearly independent functions $w_1$, $w_2$ in $U_y$. By \eqref{E-hmn3-n10b}, we have $2k- 3 \le \rho(w_i,y) \le 2k - 2$.  If $\rho(w_1,y) = \rho(w_2,y) = 2k - 3$, by Lemma~\ref{L-zeroc} there exists some linear combination $w$ of $w_1$ and $w_2$ such that $\rho(w,y) \ge 2k -2$. This function $w$ must satisfy \eqref{E-hmn3-n10d} and hence, is uniquely defined (up to scaling). If $\rho(w_1,y) = (2k - 2)$, then we must have $\rho(w_2,y) = (2k -3)$ since $w_1$ is uniquely defined. Any other function in $W_y$ must be a linear combination of $w_1$ and $w_2$. It follows that $\dim U_y \le 2$.\smallskip

 Assume that $\dim U_y = 2$. Choose a basis $\set{w_1,w_2}$ of $U_y$, with $\rho(w_1,y) = (2k -2)$,  $\rho(w_2,y) = (2k -3)$, and let $y_2 = z(y)$ be the unique other singular point of $w_2$ on $\Gamma$.  We encountered a similar framework in Subsection~\ref{SS-hmn-24} (Proposition~\ref{P-hmn}) and in Subsection~\ref{SS-hmn-25}, and we can use a \emph{rotating function argument} \index{Rotating function argument} to conclude that $\dim U_y = 1$. The claim is proved. This completes the proof of Assertion~(i). \quad \qedc \medskip

\emph{Assertion~(iii).}~   The proof of this assertion is similar to the proof of Property~\ref{P-h2n-laa}.
\end{proof}

Let $\set{\phi_j, 1 \le j \le (2k-2)}$ be an orthonormal basis of the eigenspace $U$. Let $y_0$ be a point on $ \Gamma$. By Lemma~\ref{L-L32}\,(iii), there exists some positive $\sigma_0$, and a $\Cty$ map $\cA(y_0;\sigma_0) \ni y \mapsto \big( a_{y_0,1}(y),\ldots,a_{y_0,(2k-2)}(y) \big) \in \bS^{2k-3}$ such that, for all $y \in \cA(y_0;\sigma_0)$, the eigenfunction
\begin{equation}\label{E-hmn3-n10f}
u_y:= \sum_{j=1}^{2k-2} a_{y_0,j}(y) \, \phi_j \in \bS(U)
\end{equation}
is a generator of $U_y$   and lies in the unit sphere $\bS(U)$ of the eigenspace $U$. (For the notation $\cA(y_0;\sigma_0)$ see Notation~\ref{N-hmn3-arcs}.)

\index{1-Gamma@$\Gamma$!$\Gamma_{(2k-3)},\Gamma_{(2k-2)}$}
\begin{notation}\label{N-hmn3-gam}
In view of the lemma, define the following subsets of $\Gamma$\,:
\begin{equation}\label{E-hmn3-n12}
\left\{
\begin{array}{l}
\Gamma_{(2k -3)} := \set{y \in \Gamma \mid \rho(u_y,y) = 2k - 3}\\[5pt]
\Gamma_{(2k -2)} := \set{y \in \Gamma \mid \rho(u_y,y) = 2k - 2}.
\end{array}
\right.
\end{equation}
\end{notation}%

\subsubsection{Structure and combinatorial type of nodal sets in $U_y$}\label{SSS-hmn-31bas}~

Using Lemma~\ref{L-L32}\,(ii) one can describe the possible nodal patterns of a generator $u_y$ of $U_y$, as we did in Paragraph~\ref{SSS-hmn-31ias}, see also Subsections~\ref{SS-hmn-24} and \ref{SS-hmn-25}.
If $\rho(u_y,y) = (2k-3)$, the nodal set $\cZ(u_y)$ consists of $(k-2)$ simple loops at $y$, and a simple arc from $y$ to some $z(y) \in \Gamma$, $z(y) \neq y$;  if $\rho(u_y,y) = (2k-2)$, the nodal set $\cZ(u_y)$ consists of $(k-1)$ simple loops at $y$. The loops and the arc do not intersect away from $y$. Figure~\ref{F-hmn3-h0-L33} displays some possible nodal patterns. \medskip

\begin{figure}[!ht]
\centering
\begin{subfigure}[t]{.30\textwidth}
\centering
\includegraphics[width=\linewidth]{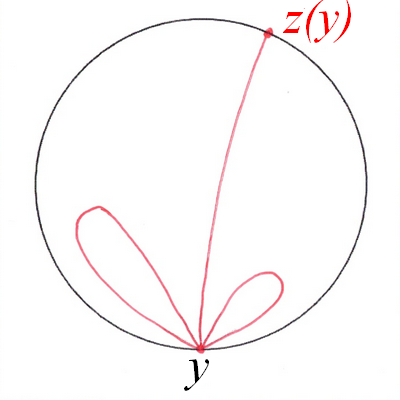}
\caption{}
\end{subfigure}
\begin{subfigure}[t]{.30\textwidth}
\centering
\includegraphics[width=\linewidth]{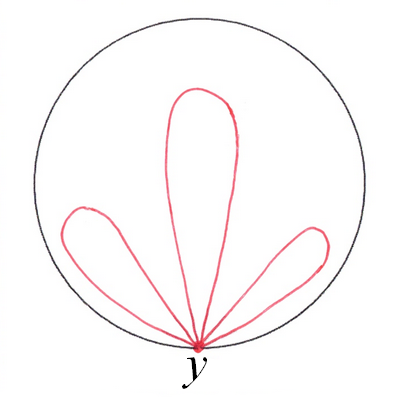}
\caption{}
\end{subfigure}
\caption{$\Omega$ simply connected, nodal patterns for $u_y \in U_y$ ($k=4$)}\label{F-hmn3-h0-L33}
\end{figure}

For a given $y \in \Gamma$, we apply  Section~\ref{S-lsbs}  to a generator $u_y$ of $U_y$. For both Dirichlet or Robin boundary condition, in a neighborhood of $y$, the nodal set $\cZ(u_y)$ consists of $\rho(u_y,y)$ nodal semi-arcs emanating from $y$. As in Paragraph~\ref{SSS-hmn-25B}, for a given $ j \in L_{\rho(u_y,y)} = \set{1,\ldots,\rho(u_y,y)}$,
we follow the nodal semi-arc emanating from $y$ tangentially to the ray $\omega_j$ along $\cZ(u_y)$.  There are two cases.
\begin{enumerate}[$\diamond$]
\index{1-tau@$\tau$!$\tau_x^U$}
\item When $y \in \Gamma_{(2k -2)}$, according to Lemma~\ref{L-L32}, we eventually arrive back at $y$ along a nodal semi-arc emanating from another ray which we denote by $\omega_{\tau_y^{U}(j)}\,$. This uniquely defines a map $\tau_y^{U} : L_{(2k-2)} \to L_{(2k-2)}$, such that $\tau_y^{U}(j) \neq j$ and $(\tau_y^{U})^2 = \id$. In this case, the nodal set $\cZ(u_y)$ consists of a $(k-1)$-bouquet of loops $\gamma^{U}_{j,\tau_y^{U}(j)}$ at $y$, for $j \in L_{(2k-2)}$.
  \item When $y \in \Gamma_{(2k -3)}$, we define a map $\tau_y^{U}$ from $\set{\downarrow} \cup L_{(2k-3)}$ to itself as follows. According to Lemma~\ref{L-L32}, there exists a unique $a_y \in L_{(2k-3)}$ (depending on $y$) such that the nodal semi-arc $\delta_{a_y}$ emanating from $y$ tangentially to the ray $\omega_{a_y}$ eventually hits $\Gamma$ at some $z(y) \neq y$. We let $\tau_y^{U}(a_y) = \,\downarrow$\, and $\tau_y^{U}(\downarrow)=a_y$. For $j \neq a_y$, following the nodal semi-arc $\delta_j$ emanating from $y$ tangentially to $\omega_j$ along $\cZ(u_y)$, we will eventually  reach $y$ again, along another ray denoted by $\omega_{\tau_{y}^{U}(j)}$. This uniquely defines a map $\tau_{y}^{U}$ from $\set{\downarrow} \cup L_{(2k-3)}$ to itself such that $\tau_{y}^{U}(j) \neq j$ and $(\tau_{y}^{U})^2 = \id$. In this case, the nodal set $\cZ(u_y)$ is the wedge sum of the arc $\delta_{a_y}$ from $y$ to $z(y)$ with a $(k-2)$-bouquet of loops $\gamma^U_{j,\tau_{y}^{U}(j)}$ at $y$.
\end{enumerate}
When $\Omega$ is simply connected and $y \in \Gamma_{(2k -3)}$, the $(k-2)$-bouquet of loops actually consists of two bouquets of loops, one in each  component of $\Omega \sm \delta_{a_y}$.  When $a_y = 1$ or $a_y = (2k-3)$, one of these bouquets of loops is actually  empty.\smallskip

As in Paragraph~\ref{SSS-hmn-25B}, we define the map $\tau_y^{U}$ as the \emph{combinatorial type} of the nodal set $\cZ(u_y)$ at $y$, when $y \in \Gamma_{(2k-2)}$, resp. $y \in \Gamma_{(2k-3)}$. The source-set of $\tau_{y}^{U}$ is $L_{(2k-2)}$, resp. $\set{\downarrow} \cup L_{(2k-3)}$.

\subsection{Eigenfunctions with Two Prescribed Boundary Singular Points}\label{SS-hmn-31b2}

Given $(y,s) \in \Gamma_{(2k -3)} \times \Gamma $, with $ y \neq s$, we define the subspace
\begin{equation}\label{E-L34-2}
V_{y,s} := \set{u \in U \mid \rho(u,y) \ge 2k-4 \text{~and~} \rho(u,s) \ge 1}.
\end{equation}

In view of Assumptions~\ref{A-hmn3-0},  Lemma~\ref{L-zero2} implies that $V_{y,s} \neq \set{0}$.

\begin{lemma}\label{L-L34a}%
Assume that $\Omega$ is simply connected. Let $U:=U(\lambda_k)$ with $k \ge 3$, and assume that $\dim U = (2k-2)$. Given $(y,s) \in \Gamma_{(2k -3)} \times \Gamma $, with $ y \neq s$,  the subspace
\[
V_{y,s} = \set{u \in U \mid \rho(u,y) \ge 2k-4 \text{~and~} \rho(u,s) \ge 1}
\]
has the following properties.
\begin{enumerate}[(i)]
  \item The subspace $V_{y,s}$ has dimension $1$.
  \item Any $0 \neq u \in V_{y,s}$ satisfies
  \begin{equation}\label{E-L34-4}
  \left\{
  \begin{array}{l}
  \kappa(u) = k,\\[5pt]
  \cZ(u) \cup \Gamma  \text{~is connected,}\\[5pt]
  \cS_{\mathrm{i}}(u) = \emptyset,\\[5pt]
  \sum_{z \in \cS_{\mathrm{b}}(u)} \rho(u,z) = 2k-2, \text{~and}\\[5pt]
  \quad 2k-4 \le \rho(u,y) \le 2k-3,\\[5pt]
  \quad 1 \le \rho(u,s) \le 2.
  \end{array}
  \right.
  \end{equation}
  More precisely, there are three distinct possibilities.
  \begin{description}
    \item[Case~A] $\rho(u,y) = (2k-3)$ and $\rho(u,s) = 1$.  In that case, $u \in U_y$, with $\cS_{\mathrm{b}}(u) = \set{y,s}$, and hence $s = z(y)$.
    \item[Case~B] $\rho(u,y) = (2k-4)$ and $\rho(u,s) = 2$.  In that case, $\cS_{\mathrm{b}}(u) = \set{y,s}$.
    \item[Case~C] $\rho(u,y) = (2k-4)$, $\rho(u,s) = 1$, and there exists some $s' \in \Gamma  \sm \set{y,s}$ such that $\cS_{\mathrm{b}}(u) = \set{y,s,s'}$, with $\rho(u,s') = 1$.
  \end{description}
  \item If $s = z(y)$, then $V_{y,z(y)} = U_y$.
  \item The map $\set{(y,s) \mid (y,s) \in \Gamma_{(2k-3)}\times \Gamma , s\neq y} \ni (y,s) \mapsto [V_{y,s}] \in \bP(U)$ is $C^{\infty}$.
  \end{enumerate}
\end{lemma}%

\begin{proof}  The proof of Assertion~(ii) follows from Euler's formula. Using \eqref{E-L34-4}, it is easy to prove that $\dim V_{y,s} \le 2$. The proof that $\dim V_{y,s} = 1$ is easy when $s \neq z(y)$. The proof that $\dim V_{y,z(y)} = 1$ is more involved and follows from a \emph{rotating function argument}. Assertion~(iv) is a consequence of Assertion~(i), as in the proofs of Lemmas~\ref{L-L37} and \ref{L-L32}.
\end{proof}

In Section~\ref{S-hmn32b} we provide a detailed proof of Lemma~\ref{L-L34a}, as well as further properties of the functions in $V_{y,s}$. These results, partially revisit Lemmas~3.4, 3.5 and 3.6 of \cite[pp.~1180-1183]{HoMN1999}. In this monograph, we actually only use Lemma~\ref{L-L34a}, see the proof of Lemma~\ref{L-L33b}.

\subsection{ Local properties of the map $\Gamma \ni y \mapsto [U_y] \in \bP(U)$}\label{SSS-hmn-31bb}

Let $y_0 \in \Gamma$. For $y \in \cA(y_0;\sigma_0)$ with $\sigma_0$ small enough (see Notation~\ref{N-hmn3-arcs}), we represent a generator of $U_y$ as in \eqref{E-hmn3-n10f},
\begin{equation*}
u_y = \sum_{j=1}^{2k-2} a_{y_0,j}(y) \, \phi_j\,.
\end{equation*}

Applying Lemma~\ref{L-lsbs3-2}, we have a conformal mapping $E_0 : \bH \to \Omega$ such that $E_0$ extends smoothly to $\bbH$, $E_0(0) = y_0$ and, when $y_0 \in \Gamma_{(2k-3)}$, such that $E_0(\zeta_0) = z(y_0)$ for some $\zeta_0 \in \bdH$. Since $E_0|_{\bdH}$ is a diffeomorphism from $\bdH$ onto $\Gamma\sm\set{y_*}$, we can choose some $r_0 > 0$ such that $E_0\big( (-r_0,r_0) \times \set{0} \big) \subset \cA(y_0;\sigma_0)$. We now work in $\wb{D}_{+}(0,r_0) \subset \bbH$, and consider the $t$-family of functions $\xi \mapsto v_t(\xi)$
\index{2-v@$v_t$!$v_t$}
\begin{equation*}
v_t(\xi_1,\xi_2) = \sum_{j=1}^{2k-2} a_{y_0,j}(E_0(t,0)) \, \phi_j \circ E_0(\xi_1,\xi_2)
\end{equation*}
which we rewrite as
\index{2-v@$v_t$!$v_t(\xi_1,\xi_2)$}
\begin{equation}\label{E-L33a-2}
v_t(\xi_1,\xi_2) = \sum_{j=1}^{2k-2} a_j(t) \, \psi_j(\xi_1,\xi_2),
\end{equation}
with the obvious notation.\smallskip

The functions $t \mapsto a_j(t)$ are $\Cty$ in $(-r_0,r_0)$ and the functions $\psi_j$ satisfy \eqref{E-lsbs3-8}. Furthermore, for all $t \in (-r_0,r_0)$, we have $\rho(v_t,(t,0)) = (2k-2)$ if $E_0((t,0))$ belongs to $ \Gamma_{(2k-2)}$, and $\rho(v_t,(t,0)) = (2k-3)$ if $E_0(t,0)$ belongs to $ \Gamma_{(2k-3)}$. Restricting $r_0$ if necessary, we may also assume that $\cS_{\mathrm{b}}(v_t) = \set{(t,0),(z(t),0)}$  for any $t \in (-r_0,r_0)$ such that $E_0(t,0)$ belongs to $\Gamma_{(2k-3)}$,  and for some $z(t) \neq t$. \smallskip

For convenience, we introduce the notation
\index{1-Gamma@$\Gamma$!$\Gamma_{0,(2k-2)},\Gamma_{0,(2k-3)}$}
\begin{equation}\label{E-L33a-2a}
\left\{
\begin{array}{l}
\Gamma_{0,(2k-2)} := \set{(t,0) \mid t \in (-r_0,r_0) \text{~and~} E_0(t,0) \in \Gamma_{(2k-2)}} \subset \bdH \\[5pt]
\Gamma_{0,(2k-3)} := \set{(t,0) \mid t \in (-r_0,r_0) \text{~and~} E_0(t,0) \in \Gamma_{(2k-3)}} \subset \bdH,
\end{array}
\right.
\end{equation}
and
\begin{equation}\label{E-L33a-4}
p := \left\{
        \begin{array}{l}
          (2k-2) \text{~in the Dirichlet case}\\[5pt]
          (2k-3) \text{~in the Robin case.}
        \end{array}
      \right.
\end{equation}

Then,
\begin{equation*}
\left\{
\begin{array}{ll}
\ord(v_t,(t,0)) = p &\text{if~} (t,0) \in \Gamma_{0,(2k-3)}\\[5pt]
\ord(v_t,(t,0)) = (p+1) &\text{if~} (t,0) \in \Gamma_{0,(2k-2)}.
\end{array}
\right.
\end{equation*}
\smallskip

For each $t$, write Taylor's formula at order $(p+1)$ for the function $\xi \mapsto v_t(\xi)$ in the coordinates $\xi=(\xi_1,\xi_2)$ of $\bH$,  at the point $\xi = (t,0)$:
\begin{equation}\label{E-L33a-6t}
\left\{
\begin{array}{ll}
v_t(\xi_1,\xi_2) & = \sum_{|\alpha| = p} \frac{1}{\alpha!} \, D^{\alpha}_{\xi}v_t(t,0) \, \big( \xi_1 - t, \xi_2\big)^{\alpha}\\[5pt]
&\quad  + \sum_{|\alpha| = p+1} \frac{1}{\alpha!} \, D^{\alpha}_{\xi}v_t(t,0) \, \big( \xi_1 - t, \xi_2\big)^{\alpha}\\[5pt]
&\quad  + \sum_{|\beta| = p+2} R_{\beta}(t;\xi_1,\xi_2)\, \big( \xi_1 - t, \xi_2\big)^{\beta},
\end{array}
\right.
\end{equation}
where
\begin{equation}\label{E-L33a-6r}
R_{\beta}(t;\xi_1,\xi_2) = \frac{|\beta|}{\beta!} \,\int_0^1 (1-s)^{|\beta|-1} D^{\beta}_{\xi}v_t\big( t + s(\xi_1-t),s\xi_2\big) \, ds.
\end{equation}

The first two terms in the Taylor formula \eqref{E-L33a-6t} are \emph{harmonic} homogeneous polynomials of degrees $p$ and $(p+1)$ respectively, see Subsection~\ref{SS-lsbs6}. When $(t,0)$ belongs to $\Gamma_{0,(2k-3)}$,  $\ord(v_t,(t,0)) = p$, and the first term is nonzero. When $(t,0)$ belongs to $\Gamma_{0,(2k-2)}$, $\ord(v_t,(t,0)) = (p+1)$, the first term is identically zero, and the second term is nonzero. In view of \eqref{E-L33a-2}, we can express the coefficients in Taylor's formula as
\begin{equation}
D^{\alpha}_{\xi}v_t(t,0) = \sum_{j=1}^{2k-2}a_j(t) \, D^{\alpha}_{\xi}\psi_j(t,0).
\end{equation}

Applying the proof of Lemma~\ref{L-lsbs6-2} to each function $\xi \mapsto v_t(\xi)$, for $t \in (-r_0,r_0)$, with a Taylor formula at the boundary point $(t,0)$ rather than at $(0,0)$, we obtain the following lemma.

\begin{lemma}\label{L-L33a}
The Taylor formula for the function $\xi \mapsto v_t(\xi)$, at the point $\xi=(t,0)$ and at order $(p+1)$, is given by the following identities, depending on the boundary condition (with the notation of Subsection~\ref{SS-lsbs6}).
\begin{equation}\label{E-L33a-6td}
\begin{array}{ll}
 \text{Dirichlet case}&\\[5pt]
v_t(\xi_1,\xi_2) & = s_{p}(t)\, S_p(\xi_1-t,\xi_2) + s_{p+1}(t) \, S_{p+1}(\xi_1-t,\xi_2)\\[5pt]
&\quad +\, R_{p+2}(t;\xi_1-t,\xi_2).
\end{array}
\end{equation}
\begin{equation}\label{E-L33a-6tn}
\begin{array}{ll}
\text{Neumann case}&\\[5pt]
  v_t(\xi_1,\xi_2) &= c_{p}(t)\, C_p(\xi_1-t,\xi_2) + c_{p+1}(t) \, C_{p+1}(\xi_1-t,\xi_2)\\[5pt]
&\quad +\, R_{p+2}(t;\xi_1-t,\xi_2).
\end{array}
  \end{equation}
\begin{equation}\label{E-L33a-6tr}
\begin{array}{ll}
\text{Robin case}&\\[5pt]
v_t(\xi_1,\xi_2) & = c_{p}(t)\, C_p(\xi_1-t,\xi_2) + c_{p+1}(t) \, C_{p+1}(\xi_1-t,\xi_2)\\[5pt]
&\quad +\, \frac{1}{p+1}\, c_{p}(t)\, h_E(t) \, S_{p+1}(\xi_1-t,\xi_2)\\[5pt]
& \quad +\, R_{p+2}(t;\xi_1-t,\xi_2).
\end{array}
\end{equation}
The remainder term, $R_{p+2}(t;\xi_1-t,\xi_2) = \sum_{|\beta|=p+2}R_{\beta}(t;\xi_1,\xi_2)\, (\xi_1-t,\xi_2)^{\beta}$, vanishes at order at least $(p+2)$ at $\xi = (t,0)$, with $R_{\beta}$ as in \eqref{E-L33a-6r}.
\end{lemma}%

We will write these Taylor identities as
\begin{equation}\label{E-L33a-8v}
\begin{array}{ll}
v_t(\xi_1,\xi_2) &= A_0(t) P_p(\xi_1-t,\xi_2) + A_1(t) P_{p+1}(\xi_1-t,\xi_2)\\[5pt]
&\qquad+\, \frac{1}{p+1}\, A_0(t)\, h_E(t) \, Q_{p+1}(\xi_1-t,\xi_2) + R_{p+2}(t;\xi_1-t,\xi_2),
\end{array}%
\end{equation}
where, using the notation \eqref{E-lsbs6-14},
\index{2-P@$P_p$} \index{2-Q@$Q_p$}
\begin{enumerate}[$\diamond$]
  \item $P_p = S_p$, $P_{p+1} = S_{p+1}$ and $Q_{p+1} = 0$ in the Dirichlet case,
  \item $P_p = C_p$, $P_{p+1} = C_{p+1}$ and $Q_{p+1} = S_{p+1}$ in the Robin case.
\end{enumerate}

Note that the third term in the right hand side of  \eqref{E-L33a-8v}  disappears in the Dirichlet and Neumann cases. \medskip

The family of functions $\xi \mapsto v_t(\xi)$ given by \eqref{E-L33a-2} is $\Cty$ with respect to the para\-meter $t \in (-r_0,r_0)$. Its $t$-derivative is given by
\index{2-w@$w_t$!$w_t$}
\begin{equation}\label{E-L33a-9}
w_t := \frac{d}{dt} v_t = \sum_{j=1}^{2k-2} a'_j(t) \psi_j\, .
\end{equation}

Let $w_t^{\Omega}$ denote the related family
\begin{equation}\label{E-L33a-9O}
w_t^{\Omega} :=  \sum_{j=1}^{2k-2} a'_j(t) \phi_j\,,
\end{equation}
which is a $\Cty$ family of eigenfunctions in the eigenspace $U$.

Taking the derivative of the identity \eqref{E-L33a-8v} with respect to $t$, at $t=0$, we obtain
\begin{equation}\label{E-L33a-12}
\left\{
\begin{array}{ll}
\partial_t v_t|_{t=0}(\xi_1,\xi_2)& = A'_{0}(0) P_p (\xi_1,\xi_2)  - A_{0}(0) \, \partial_{\xi_1}P_p(\xi_1,\xi_2)\\[5pt]
&\qquad +\, A'_{1}(0) \, P_{p+1}(\xi_1,\xi_2) -  A_{1}(0) \, \partial_{\xi_1}P_{p+1}(\xi_1,\xi_2)\\[5pt]
&\qquad +\, \frac{1}{p+1}\, \big( A_0(t) h_E(t)\big)'_{t=0} \, Q_{p+1}(\xi_1,\xi_2)\\[5pt]
&\qquad -\, \frac{1}{p+1}\, A_0(0) h_E(0) \partial_{\xi_1}Q_{p+1}(\xi_1,\xi_2)\\[5pt]
&\qquad +\, \sum_{|\beta|=p+2} \partial_t R_{\beta}(0;\xi_1,\xi_2)\, (\xi_1,\xi_2)^{\beta}\\[5pt]
&\qquad -\, \sum_{|\beta|=p+2} \beta_1 R_{\beta}(0;\xi_1,\xi_2)\,
(\xi_1,\xi_2)^{\beta - (1,0)}.
\end{array}
\right.
\end{equation}

In view of the relations \eqref{E-lsbs6-16} and the definitions of $P_n$ and $Q_n$ (depending on the boundary condition, Dirichlet or Robin), we have the relations \[\partial_{\xi_1}P_{n} = n\, P_{n-1} \mbox{ and } \partial_{\xi_1}Q_{n} = n\, Q_{n-1}.\]
It follows that \eqref{E-L33a-12} reduces to
\begin{equation}\label{E-L33a-12w}
\left\{
\begin{array}{ll}
w_0(\xi_1,\xi_2)& =  - p\, A_{0}(0) \, P_{p-1}(\xi_1,\xi_2)\\[5pt]
&\quad + \big[ A'_{0}(0) - (p+1) A_1(0) \big]\, P_p (\xi_1,\xi_2) \\[5pt]
&\quad - A_{0}(0) \, h_E(0) \, Q_{p}(\xi_1,\xi_2) + \cO\big( (\xi_1^2+\xi_2^2)^{\frac{p+1}{2}})\big).
\end{array}
\right.
\end{equation}
\smallskip

We also consider the function $\xi_1 \mapsto \breve{v}_t(\xi_1)$ as defined in \eqref{E-evp-dr}. In the Dirichlet case, this function is given by $\breve{v}_t(\xi_1) = \partial_{\xi_2}v_t(\xi_1,0)$. In the Robin case, it is given by $\breve{v}_t(\xi_1) = v_t(\xi_1,0)$. From the definition of $\breve{v}_t(\xi_1)$, the identity \eqref{E-L33a-8v}, the relations in \eqref{E-lsbs6-20} and \eqref{E-L33a-4}, we obtain the following relations. 
\begin{subequations}\label{E-L33a-14dr}
\begin{align}
\begin{split}\label{E-L33a-14d}
& \text{Dirichlet case:} \\
& \breve{v}_t(\xi_1) = (2k-2) \, A_0(t)\, (\xi_1-t)^{2k-3} + (2k-1) \, A_1(t)\, (\xi_1-t)^{2k-2}\\
& \hspace{16mm}  +\, \cO\big( (\xi_1-t)^{2k-1}\big).
\end{split}\\[5pt]
\begin{split}\label{E-L33a-14r}
&\text{Robin case:} \\
& \breve{v}_t(\xi_1) = A_0(t)\, (\xi_1-t)^{2k-3} + \, A_1(t)\, (\xi_1-t)^{2k-2}
+\, \cO\big( (\xi_1-t)^{2k-1}\big).
\end{split}
\end{align}
\end{subequations}

\subsubsection{Properties of $\Gamma_{(2k-3)}$ and $\Gamma_{(2k-2)}$}\label{SSS-hmn-31bc}

\begin{lemma}\label{L-L33}
Assume that $\Omega$ is simply connected. Let $U := U(\lambda_k)$ for some $k \ge 3$. Assume that $\dim U = (2k-2)$. [Assumptions~\ref{A-hmn3-0}]. Then, the following properties hold.
\begin{enumerate}[(i)]
  \item The sets $\Gamma_{(2k-3)}$ and $\Gamma_{(2k-2)}$ are disjoint and
      \[\Gamma = \Gamma_{(2k-3)} \bigsqcup \Gamma_{(2k-2)}.\]
  \item The set $\Gamma_{(2k-3)}$ is open in $\Gamma$ and the set $\Gamma_{(2k-2)}$ is finite.
\end{enumerate}
\end{lemma}%

\begin{proof} Assertion~(i) follows from Lemma~\ref{L-L32}, Assertion~(i),  and \eqref{E-hmn3-n10d}. \quad \qedc \medskip

\emph{Assertion~(ii).~} Let $y_0 \in \Gamma_{(2k -3)}$. The generator $u_{y_0}$ of $U_{y_0}$ given by \eqref{E-hmn3-n10f} satisfies $\cS_{\mathrm{b}}(u_{y_0}) = \set{y_0, z_0}$ for some $z_0 \in \Gamma$, $z_0 \neq y_0$, with $\rho(u_{y_0},y_0) = (2k-3)$, and $\rho(u_{y_0},z_0) = 1$. This means that the function $\breve{u}_{y_0}$ has precisely two zeros on $\Gamma$, $y_0$ and $z_0$, and changes sign at these points. Given two points $z_1$ and $z_2$ on either sides of $z_0$, close enough to $z_0$ and away from $y_0$, we have $\breve{u}_{y_0}(z_1) \, \breve{u}_{y_0}(z_2) < 0$. When $y \in \cA(y_0;\sigma_0)$ is close enough to $y_0$, the function $u_y$ is $C^1$-close to the function $u_{y_0}$ and hence $\breve{u}_y$ is uniformly close to $\breve{u}_{y_0}$, and satisfies $\breve{u}_{y}(z_1) \, \breve{u}_{y}(z_2) < 0$. In view of Lemma~\ref{L-L32}, this implies that for $y$ close enough to $y_0$ in $\Gamma$, $y \in \Gamma_{(2k-3)}$, so that $\Gamma_{(2k-3)}$ is open in $\Gamma$.\smallskip

To prove that the set $\Gamma_{(2k -2)}$ is finite, it suffices to prove that it is discrete.  We work in the setup of Paragraph~\ref{SSS-hmn-31bb} with the function $v_{t}$ given by \eqref{E-L33a-2}. Assume, by contradiction, that the point $y_0$ is not isolated in $\Gamma_{(2k-2)}$. Then the point $(0,0) = E_0^{-1}(y_0)$ is not isolated in $\Gamma_{0,(2k-2)}$, and  there exists a sequence $\set{t_n}$ tending to zero, such that  $v_{t_n}$ satisfies $\rho(v_{t_n},(t_n,0)) = (2k-2)$ for all $n$.\smallskip

Writing
\begin{equation*}
\begin{array}{ll}
v_t(\xi_1,\xi_2) &= A_0(t) P_p(\xi_1-t,\xi_2) + A_1(t) P_{p+1}(\xi_1-t,\xi_2)\\[5pt]
&\qquad+\, \frac{1}{p+1}\, A_0(t)\, h_E(t) \, Q_{p+1}(\xi_1-t,\xi_2) + R_{p+2}(t;\xi_1-t,\xi_2),
\end{array}%
\end{equation*}
as in \eqref{E-L33a-8v}, we have $A_0(0) = 0$, $A_1(0) \neq 0$, and since $A_0(t_n) = 0$ for all $n$, we also have $A'_0(0) = 0$. Equation \eqref{E-L33a-12w} then reduces to
\begin{equation*}
w_0(\xi_1,\xi_2) = - (p+1) \, A_1(0) P_p(\xi_1,\xi_2) +  \cO\big( (\xi_1^2+\xi_2^2)^{\frac{p+1}{2}}\big).
\end{equation*}
This means that $\ord(w_0^{\Omega},y_0) = \ord(w_0,(0,0)) = p$, and hence, using \eqref{E-L33a-4}, that \linebreak $\rho(w_0^{\Omega},y_0)=(2k-3)$, and $w_0^{\Omega} \in U_{y_0}$. On the other hand, $u_{y_0} \in U_{y_0}$, with $\rho(u_{y_0},y_0) = (2k-2)$. We would have two linearly independent functions in $U_{y_0}$, a contradiction with $\dim U_{y_0} = 1$. This proves that $y_0$ is isolated in $\Gamma$. It follows that $\Gamma_{(2k-2)}$ is discrete and hence finite. Assertion~(ii) is proved.  \quad \qedc \smallskip

The proof of Lemma~\ref{L-L33} is complete.
\end{proof}

\begin{remark}\label{R-L3237}
Lemmas~\ref{L-L37}, \ref{L-L32}, and \ref{L-L33} are actually valid when $\Omega$ is not simply connected: same arguments as in Section~\ref{S-hmn2N}.
\end{remark}%

\subsection{Global properties of the map $\Gamma \ni y \mapsto [U_y] \in \bP(U)$}\label{SSS-hmn-31bd}

\begin{lemma}\label{L-L33b}
Assume that $\Omega$ is simply connected. Let $U := U(\lambda_k)$ for some $k \ge 3$. Assume that $\dim U = (2k-2)$. [Assumptions~\ref{A-hmn3-0}]. Then, the following properties hold.
\begin{enumerate}[(i)]
  \item The map $\Gamma_{(2k -3)} \ni y \mapsto z(y) \in \Gamma$, where $z(y)$ is defined in \eqref{E-hmn3-n10d} is continuous in $\Gamma_{(2k -3)}$. Moreover, if $\eta \in \Gamma_{(2k -2)}$, then $\lim_{y \to \eta, y \in \Gamma_{(2k-3)}}z(y) = \eta$, i.e.,  the map $y \mapsto z(y)$ extends continuously to $\Gamma$, with $z(\eta) = \eta$ for all $\eta \in \Gamma_{(2k-2)}$.
  \item Let $C$ be any connected component of $\Gamma_{(2k-3)}$. The function $C \ni y \mapsto z(y) \in \Gamma$ is $\Cty$ and monotonic  in $C$ (more precisely, the derivative of $z$ does not vanish).
  \item  \emph{Assume that} $\Gamma_{(2k-2)}$ \emph{is not empty}, and let $\eta \in \Gamma_{(2k-2)}$. When $y$ is close to $\eta$, the points $y$ and $z(y)$ lie on either sides of $\eta$. More precisely, let $E:\bbH \to \Omegab$ is a conformal map such that $E(0,0)=\eta$. For $t\neq 0$ small enough, let $\cS_{\mathrm{b}}(u_{E(t,0)})= \set{E(t,0), z(E(t,0))}$ and define $z(t) := z(E(t,0))$. Then,
      \[z(t) = -(2k-3) t + o(t).\]
  \item \emph{Assume that} $\Gamma_{(2k-2)}$ \emph{is not empty}. When the point $y$ moves clockwise in a connected component $C$ of $\Gamma_{(2k-3)}$ the point $z(y)$ moves counter-clockwise in $\Gamma$.
  \item \emph{Assume that} $\Gamma_{(2k-2)}$ \emph{is not empty}. The set $\Gamma_{(2k-2)}$ cannot be reduced to one point. Each component $C$ of $\Gamma_{(2k-3)}$ has two distinct boundary points $\eta_1, \eta_2 \in \Gamma_{(2k-2)}$, and its image $z(C)$ is equal to $\Gamma \sm \wb{C}$. In particular, if $y \in C$, $z(y) \not \in C$.
\end{enumerate}
\end{lemma}%

\begin{proof}[Proof of Lemma~\ref{L-L33b}]

\emph{Assertion~(i).~} Consider a component $C$ of $\Gamma_{(2k -3)}$. Recall that
\begin{equation}\label{E-bha4-3e}
\breve{u}_y := \left\{
\begin{array}{ll}
\partial_{\nu}u_y &\quad \text{in the Dirichlet case}\\[5pt]
u_y|_{\Gamma} &\quad \text{in the Robin case.}\\[5pt]
\end{array}
\right.
\end{equation}

Let $y \in C$, and let $\set{y_n} \subset C$ be a sequence such that $y_n$ converges to $y$, so that $u_n := u_{y_n}$ converges to $u:=u_y$ (uniformly in the $C^m$ topology for any fixed $m$, see Lemma~\ref{L-L32}).
 Recall that $\breve{u}$ and $\breve{u}_n$ have precisely two distinct zeros $y, z(y)$ and $y_n, z(y_n)$ respectively. Since $\breve{u}_n$ converges uniformly to $\breve{u}$, and since $\breve{u}$ changes sign at $z(y)$, it follows that $z(y_n)$ belongs to some neighborhood of $z(y)$, and that $z(y_n)$ tends to $z(y)$.
 This proves that $y \to z(y)$ is continuous in $C$. We now investigate the behavior of $z(y)$ when $y$ tends to $\partial C$ (assuming that $C \neq \Gamma$). Assume that $\set{y_n} \subset C$, with $y_n$ tending to some $\eta \in \partial C \subset \Gamma_{(2k-2)}$.  Choose a subsequence of $\set{z(y_n)}$ which converges to some $z$. Since $u_n$ tends to $u_{\eta}$, we conclude that $\breve{u}_{\eta}(z) = 0$, and hence that $z=\eta$ since $\eta$ is the unique zero of $\breve{u}_{\eta}$ in $\Gamma$.\quad \qedc\medskip

\emph{Assertion~(ii).} The properties to be established are local. We work in a neighborhood of some $y_0 \in \Gamma_{(2k-3)}$. Taking a suitable conformal mapping as in Subsection~\ref{SSS-hmn-31bc}, we consider the family of functions $v_t$ defined  in \eqref{E-L33a-2}, near the point $0 \in \bdH$ corresponding to $y_0$. Since $\Gamma_{(2k-3)}$ is open in $\Gamma$, so does $\Gamma_{0,(2k-3)}$ in $\bdH$. Hence, there exists $r_0$ such that $(-r_0,r_0) \times \set{0} \subset \Gamma_{0,(2k-3)}$,  i.e., for all $t \in (-r_0,r_0)$, $\rho(v_t, (t,0)) = (2k-3)$. As a consequence, the following properties hold. \smallskip

\noid For all $t \in (-r_0,r_0)$, the first term $A_0(t)$ in the Taylor expansion \eqref{E-L33a-8v} is nonzero, so that
\begin{equation}
v_t(\xi_1,\xi_2) = A_0(t) \, P_p(\xi_1-t,\xi_2) + R_{p+1}(t;\xi_1-t,\xi_2),
\end{equation}
where $P_p = S_p$ in the Dirichlet case, $P_p = C_p$ in the Robin case, and the remainder term is given by
\begin{equation*}
R_{p+1}(t;\xi_1-t,\xi_2) =\sum_{|\beta| = p+1} R_{\beta}(t;\xi_1,\xi_2)\, \big( \xi_1 - t, \xi_2\big)^{\beta},
\end{equation*}
with $R_{\beta}$ as in \eqref{E-L33a-6r}.\\ Taking the derivative of $v_t$ with respect to $t$, and using \eqref{E-lsbs6-16}, we infer that
\begin{equation}\label{E-L33b-4}
w_t(\xi_1,\xi_2) = - p\, A_0(t) P_{p-1}(\xi_1-t,\xi_2) + R_{w,p}(t;\xi_1-t,\xi_2)\,
\end{equation}
where the remainder term $R_{w,p}(t;\xi_1,\xi_2)$ vanishes at order at least $p$ at $\xi = (t,0)$. This implies that
\begin{equation}\label{E-L33b-5z}
\rho(w_t^{\Omega},E_0(t,0)) = \rho(w_t, (t,0)) = (2k-4).
\end{equation}
\medskip

\noid We now look at the associated family of maps $\breve{v}_t$ on the boundary $\bdH$.
\begin{equation*}
\breve{v}_t(\xi_1) = \sum a_j(t) \breve{\psi}_j(\xi_1).
\end{equation*}

For convenience, we write the families $v_t$ and $\breve{v}_t$ as
\index{2-v@$v_t$!$v(t,\xi_1,\xi_2)$}
\index{2-v@$v_t$!$\breve{v}(t,\xi_1)$}
\[
v(t;\xi_1,\xi_2) := \sum a_j(t) \psi_j(\xi_1,\xi_2) \text{~for~} (\xi_1,\xi_2) \in  \bH, ~~ t \in (-r_0,r_0),
\]
and
\[
\breve{v}(t;\xi_1) := \sum a_j(t) \breve{\psi}_j(\xi_1) \text{~for~} (\xi_1,0) \in \partial \bH, ~~t \in  (-r_0,r_0).
\]
\smallskip

\index{2-w@$w_t$!$w(t,\xi_1,\xi_2)$}
\index{2-w@$w_t$!$\breve{w}(t,\xi_1)$}
Similarly, for the derivatives with respect to  the parameter $t$, we write
\[
w(t;\xi_1,\xi_2) := \partial_t v(t;\xi_1,\xi_2) = \sum a'_j(t) \psi_j(\xi_1,\xi_2) \text{~for~} (\xi_1,\xi_2) \in  \bH, ~~ t \in (-r_0,r_0),
\]
and
\begin{equation}\label{E-L33b-4z}
\breve{w}(t;\xi_1) := \sum a'_j(t) \breve{\psi}_j(\xi_1) \text{~for~} (\xi_1,0) \in \partial \bH, ~~t \in  (-r_0,r_0).
\end{equation}

\noid The map $(-r_0,r_0) \ni t \mapsto z(t)$ is such that $\cS_{\mathrm{b}}(v(t;\cdot)) = \set{t,z(t)}$,  where $z(t)\neq t$. According to Lemmas~\ref{L-L32} and \ref{L-breve}, for all $t$, the function $\breve{v}(t;\cdot)$ has a zero of order~$1$ at the point $z(t)$, i.e. $\breve{v}(t;z(t)) = 0$  and $\partial_{\xi_1}\breve{v}(t;z(t)) \neq 0$. The implicit function theorem implies that $t \mapsto z(t)$ is $\Cty$. This proves the first half of Assertion~(ii).\medskip

\noid Since $\breve{v}(t;z(t)) \equiv 0$, taking the derivative with
respect to $t$, we obtain
\[
\partial_t \breve{v}(t;z(t)) + z'(t) \, \partial_{\xi_1} \breve{v}(t;z(t)) \equiv 0.
\]
\noi \emph{Assuming by contradiction that} $z'(t_0)=0$ for some $t_0 \in (-r_0,r_0)$, we conclude that $\partial_t \breve{v}(t_0;z(t_0))=0$, i.e., $\breve{w}(t_0;z(t_0))=0$, and hence
\begin{equation}\label{E-L33b-10z}
\rho(w(t_0;z(t_0)) \ge 1.
\end{equation}
From Equations~\eqref{E-L33b-5z} and \eqref{E-L33b-10z} we obtain
\begin{equation*}
\rho(w^{\Omega}_{t_0}, E_0(t_0,0)) = (2k-4) \text{~and~} \rho(w^{\Omega}_{t_0},z(t_0))\ge 1,
\end{equation*}
which implies that  the function $w^{\Omega}_{t_0}$ belongs to the subspace $V_{E_0(t_0,0),z(t_0)}$. According to Lemma~\ref{L-L34a}, this subspace is $U_{E_0(t_0,0)}$ so that $\rho(w^{\Omega}_{t_0},E_0(t_0,0)) = (2k-3)$,
contradicting Equation~\eqref{E-L33b-5z}. We have proved that the assumption $z'(t_0) = 0$  yields a contradiction. Hence $z'(t) \neq 0$ for all $t \in (-r_0,r_0)$ and Assertion~(ii) follows. \quad \qedc \medskip

\emph{Assertion~(iii).}  As above, we work in the framework described in Paragraph~\ref{SSS-hmn-31bb}, with $t$ in an interval $(-r_0,r_0)$ such that $\rho(v_t,(t,0)) = (2k-3)$ for $t \neq 0$, and $\rho(v_0,0)= (2k-2)$.
According to \eqref{E-L33a-8v}, we have
\begin{equation*}
\begin{array}{ll}
v_t(\xi_1,\xi_2) &= A_0(t) P_p(\xi_1-t,\xi_2) + A_1(t) P_{p+1}(\xi_1-t,\xi_2)\\[5pt]
&\qquad+\, A_0(t) \frac{1}{p+1}\, h_E(t) \, Q_{p+1}(\xi_1-t,\xi_2) + R_{p+2}(t;\xi_1-t,\xi_2),
\end{array}%
\end{equation*}
with $A_0(0) = 0$, $A_1(0) \neq 0$, and $A_0(t) \neq 0$ for $t \neq 0$.
(Recall that  $P_p = S_p$, $P_{p+1} = S_{p+1}$ and $Q_{p+1} = 0$ in the Dirichlet case;  $P_p = C_p$, $P_{p+1} = C_{p+1}$ and $Q_{p+1} = S_{p+1}$ in the Robin case.) \smallskip

Using \eqref{E-L33a-12w}, we obtain
\begin{equation*}
w_0(\xi_1,\xi_2) =  \big[ A'_{0}(0) - (p+1) A_1(0) \big]\, P_p (\xi_1,\xi_2) + \cO\big( (\xi_1^2+\xi_2)^{\frac{p+1}{2}})\big).
\end{equation*}

\begin{claim}\label{C-L33b-2} Assume that $A_0(0) = 0$, $A_1(0) \neq 0$, and $A_0(t) \neq 0$ when $t \neq 0$. Then, $A_{0}'(0) = (p+1) \, A_1(0)$.
\end{claim}%

\noi \pf Otherwise, we would have $\ord(w_0,0) = p$,  and hence $\rho(w_0,0) = (2k-3)$ so that $w^{\Omega}_0$ belongs to $U_{y_0}$, contradicting the fact that $\dim U_{y_0} = 1$, since $u_{y_0} \in U_{y_0}$ with $\rho(u_{y_0},y_0)=(2k-2)$. The claim is proved. \quad \qedc\medskip

We now use the relations \eqref{E-L33a-14d} and \eqref{E-L33a-14r}, respectively in the Dirichlet and the Robin cases:\\[5pt]
\noi \emph{Dirichlet case,}%
\begin{equation*} 
\breve{v}_t(\xi_1) = (2k-2) A_0(t) (\xi_1-t)^{2k-3} + (2k-1) A_1(t) (\xi_1-t)^{2k-2} + \cO\big( (\xi_1-t)^{2k-1}\big).
\end{equation*}
\noi \emph{Robin case,}
\begin{equation*}
\breve{v}_t(\xi_1) = A_0(t) (\xi_1-t)^{2k-3} +  A_1(t) (\xi_1-t)^{2k-2}
 + \cO\big( (\xi_1-t)^{2k-1}\big).
\end{equation*}
Choosing $\xi_1 = z(t)$, and recalling that $z(t)$ tends to $0$ as $t$ tends to zero (Assertion~(i)), we obtain respectively:\\[5pt]
\noi \emph{Dirichlet case,}%
\begin{equation*} 
0 \equiv (2k-2)  A_0(t) (z(t)-t)^{2k-3} + (2k-1)  A_1(t)\, (z(t)-t)^{2k-2}
 + \cO\big( (z(t)-t)^{2k-1}\big).
\end{equation*}
\noi \emph{Robin case,}%
\begin{equation*} 
0 \equiv A_0(t) (z(t)-t)^{2k-3} + A_1(t) (z(t)-t)^{2k-2}
 + \cO\big( (z(t)-t)^{2k-1}\big).
\end{equation*}
Writing $A_0(t) = A_0'(0) \, t + o(t)$, $A_1(t) = A_1(0) + O(t)$ with $A_1(0) \neq 0$, and taking into account the fact that $A_0'(0) = (p+1) A_1(0)$,  we conclude that
\begin{equation}\label{E-L33b-0}
z(t) = - (2k-3) \, t + o(t) \text{~as~} t \text{~tends to~} 0,
\end{equation}
in both the Dirichlet and the Robin cases. \medskip

For $t \neq 0$ small enough, Equation~\eqref{E-L33b-0} implies that $t$ and $z(t)$ are on either sides of $0$. This means that for $y$ close enough to $\eta \in \Gamma_{(2k-2)}$, the points $y$ and $z(y)$ are located on either sides of $\eta$. Assertion~(iii) is proved. \quad \qedc \medskip

\emph{Assertion~(iv).}   Assume that $0 \in \Gamma_{0,(2k-2)}$ and $t \in \Gamma_{0,(2k-3)}$ for $t \neq 0$, small enough.  Then, $t^{-1}[z(2t)-z(t)] = z'(\theta_t) = - (2k-3) + o(1)$ for some $\theta_t$ between $t$ and $2t$. This implies that $z'(\theta_t) < 0$ for $t$ small enough. According to Assertion~(ii), $z'(t) < 0$ in the connected components of $\Gamma_{0,(2k-3)}$ which have $0$ as boundary point. \smallskip

  Note that Equation~\eqref{E-L33b-0} also implies that $z$ is differentiable everywhere on $\Gamma$, with derivative equal to $-(2k-3)$ at the points in $\Gamma_{(2k-2)}$. It is not clear though that $z'$ is continuous everywhere. \medskip

\emph{Assertion~(v).}  Assume that $\Gamma_{(2k-2)} = \set{\eta}$. For $y_0$ close to $\eta$ and on the right of $\eta$, the point $z(y_0)$ is close to $\eta$ and on the left of $\eta$. When $y$ moves  counter-clockwise from $y_0$, the point $z(y)$ moves  clockwise from $z(y_0)$ and  we would eventually find some $y_1$ with $z(y_1) = y_1$, a contradiction. If $\Gamma_{(2k-2)}$ is not empty, then $\# \big( \Gamma_{(2k-2)} \big) \ge 2$, and the boundary of a  connected component $C$ of $\Gamma_{(2k-3)}$ consists of two distinct points $\eta_1, \eta_2$ belonging to $\Gamma_{(2k-2)}$. Since $z(y)$ tends to $z(\eta_i)$ when $y$ tends to $\eta_i$, the last assertion follows.\quad \qedc \medskip

Lemma~\ref{L-L33b} is now proved.
\end{proof}

\begin{remark}\label{R-L33-2}
Under Assumptions~\ref{A-hmn3-0},  the Taylor identity \eqref{E-L33a-8v} for the function $v_t$ yields the Taylor identity \eqref{E-L33a-12w} for its derivative $w_0$ at $t=0$.  Since $0$ is an isolated point in $\Gamma_{0,(2k-2)}$,  we have $\rho(v_0,0)=(2k-2)$, i.e.,  $A_0(0) = 0$, and Claim~\ref{C-L33b-2} tells us that $A'_0(0) = (p+1) A_1(0) \neq 0$. From \eqref{E-L33a-12w}, we deduce that $\rho(w_0,0) \ge (2k-2)$. Since $w_0$ is orthogonal\footnote{Recall that orthogonality is meant with respect to the inner product induced by the $L^2$-inner product of eigenfunctions.} to $v_0$, this implies that $w_0 = 0$ because $\dim U_0 = 1$. The second derivative of $v_t$ at $t=0$ does not vanish, more precisely,
\begin{equation*}
\frac{d^2v_t}{dt^2}|_{t=0}(\xi_1,\xi_2) = - p(p+1) A_1(0) P_{p-1}(\xi_1,\xi_2) + \cO\big( (\xi_1^2+\xi_2^2)^{\frac p2}\big).
\end{equation*}
Since $v_t$ has norm $1$, $w_t$ is orthogonal to $v_t$, and since $w_0=0$, it follows that $\frac{d^2v_t}{dt^2}|_{t=0}$ is orthogonal to $v_0$.
\end{remark}%

\begin{remarks}\label{R-L33-4}\phantom{}~\\
1) In Lemma~\ref{L-L33c}, we shall prove that $\#\big(\Gamma_{(2k-2)}\big)$ is an even integer.\\
2) In Section~\ref{S-gam0}, using a global argument, we shall prove that $\Gamma_{(2k-2)} \neq \emptyset$.\\
3) For the time being, note that if  $\Gamma_{(2k-2)}$ were empty, we would have $z'(t) >0$. The function $z'$ has indeed a constant sign and if  $z'(t)$ were negative, we would reach a contradiction by finding a point $y_1$ such that $z(y_1)=y_1$ as in the proof of Assertion~(iv).
\end{remarks}

\subsection{Boundary behavior of the map $\Omega \ni x \mapsto [w_x] \in \bP(U)$}\label{SS-hmn-33}~\\
The assumption that $\Omega$ is simply connected in this subsection might be necessary. It is motivated by Remark~\ref{R-hmn-sc} and also makes the proofs of the following lemmas simpler. It would be worthwhile determining where the assumption that $\Omega$ is simply connected is actually necessary.\smallskip

The proof of the next lemma relies very much on Section~\ref{S-lsbs}, in particular Subsection~\ref{SS-lsbs5}.

\begin{lemma}\label{L-L38}
Assume that $\Omega$ is simply connected. Let $U := U(\lambda_k)$ for some $k \ge 3$. Assume that $\dim U = (2k-2)$. [Assumptions~\ref{A-hmn3-0}].
Let $\set{x_n} \subset \Omega$ be a sequence converging to some $y \in \Gamma$.  Let $\set{w_n}$ be a corresponding sequence of eigenfunctions, with $w_n  := w_{x_n} \in W_{x_n}\cap \bS(U)$.
\begin{enumerate}[(i)]
   \item If $w$ is a limit point of $\set{w_n}$, then $w \in U_y$. In particular, the continuous maps
\[
\Omega \ni x \mapsto [w_x] \text{~of Lemma~\ref{L-L37}},
\]
and
\[
\Gamma  \ni y \mapsto [u_y] \text{~of Lemma~\ref{L-L32},}
\]
give rise to a continuous map $x\mapsto [\bar{w}_x]$ from $\overline{\Omega}$ into $\bP(U)$.
  \item  The point $y$ belongs to $\Gamma_{(2k -3)}$ if and only if,
for $n$ large enough, \linebreak  $\cS_{\mathrm{b}}(w_{n}) = \set{y_{n},z_{n}}$ with $y_{n} \to y$, and $z_{n} \to z(y) \not = y$, and $\cS_{\mathrm{b}}(u_y) = \set{y,z(y)}$.
  \item  The point $y$ belongs to $\Gamma_{(2k -2)}$ if and only if there exists an infinite subsequence $\set{w_{s(n)}}$  such that $\cS_{\mathrm{b}}(w_{s(n)})= \emptyset$,  or  an infinite subsequence  $\set{w_{s(n)}}$  such that $\cS_{\mathrm{b}}(w_{s(n)}) \neq \emptyset$, and  the points in $\cS_{\mathrm{b}}(w_{s(n)})$ converge to $y$.
  \item  There exists a continuous map $x \mapsto \bar{w}_x$, from $\wb{\Omega}$ to $\bS(U)$, whose restrictions to $\Omega$ and to $\Gamma$ are $\Cty$, and such that $\bar{w}_x \in W_x$ whenever $x \in \Omega$ and $\bar{w}_x \in U_x$ whenever $x \in \Gamma$.
\end{enumerate}
\end{lemma}%

 For the proof of Lemma~\ref{L-L38},  it suffices to reason locally near a point $y \in \Gamma$. Using a conformal map $E : \bH \to \Omega$ as in Section~\ref{S-lsbs}, we work with the functions $v_n := w_n \circ E$ and $v := w\circ E$ in $D_{+}(0,r_0)$. Let $\xi_n \in \bH$ be the points such that $E(\xi_n)=x_n$. In the sequel, we use the notation of Paragraph~\ref{SSS-lsbs32}: $J_E$ is the Jacobian of the conformal map, and
\begin{equation*}
V_E := J_E\, (V\circ E), \quad h_E := \sqrt{J_E}\, (h\circ E).
\end{equation*}

At some point in the proof of Lemma~\ref{L-L38}, we will  need the following \emph{energy argument}.

\begin{lemma}[Energy argument]\label{L-hmn3-L38-e}
\index{Energy argument} Working with the eigenvalue problem \eqref{E-lsbs3-8}--\eqref{E-lsbs3-10},
\begin{equation*}
\left\{
\begin{array}{rll}
(-\Delta + J_E V_E)v &= \lambda \, J_E v & \text{~in~} \bH\\[5pt]
B_E(v) &= 0 & \text{~on~} \partial \bH,
\end{array}%
\right.
\end{equation*}
there exists $r_E > 0$  such that
\begin{equation}\label{E-hmn3-L38-mu}
\mu_1(D_{+}(0,r)) > \lambda_k
\end{equation}
for any $r \le r_E$. Here $\lambda_k$ is the eigenvalue associated with $U$, and $\mu_1(D_{+}(0,r))$ denotes the  lowest eigenvalue of
$(-\Delta + J_E V_E) v = \lambda \, J_E v$
in the domain $D_{+}(0,r)$  with the following mixed boundary conditions:  Dirichlet on
the subset  $C_{+}(0,r) = \partial D_{+}(0,r)\cap \bH$, and the current boundary condition  $B_E(u) = 0$  (Dirichlet or Robin) on the subset $(-r,r) \times \set{0} = \partial D_{+}(0,r) \cap \partial \bH$.
\end{lemma}%

\begin{proof}[Proof of Lemma~\ref{L-hmn3-L38-e}] We give the proof of the lemma when the boundary condition \eqref{E-lsbs3-10} on $\partial \bH$ is  the $h$-Robin condition (the Dirichlet or Neumann cases are simpler to deal with). We have to consider the Rayleigh quotient $R(u)$ for $u \in \cH_r$ where
\begin{equation*}
\cH_r := \set{u \in H^1(D_{+}(0,r)) \mid u=0 \text{~on~} C_{+}(0,r)},
\end{equation*}
and
\begin{equation*}
R(u) := \Big( \int_{D_+(0,r)}  (|du|^2 + V_E\, u^2)\, d\xi + \int_{-r}^{r} h_E\, u^2(t,0) \, dt \Big) \, \Big( \int_{D_+(0,r)} J_E\, u^2 \, d\xi\Big)^{-1}.
\end{equation*}

On $\wb{D}_{+}(0,r_0)$, the functions $V_E$ and $h_E$ are bounded from below, and $J_E$ is bounded from above and below by positive constants. Since $\int_{-r}^{r} u^2(t,0) \, dt \le r \, \int_{D_{+}(0,r)} |du|^2 \, d\xi$, it follows that
\begin{equation*}
R(u) \ge (1 - c_1 \,r) \, \big( \sup_{D_{+}(0,r_0)} J_E \big)^{-1}\, R_0(u) - c_2\,
\end{equation*}
where $c_1$ and $c_2$ are positive constants depending only on $(V,h,E,r_0)$, and where $R_0(u) := \big( \int_{D_+(0,r)}  |du|^2 \, d\xi \big) \big( \int_{D_+(0,r)}  u^2 \, d\xi\big)^{-1}$. The quotient $R_0(u)$ is bounded from below by the least eigenvalue of the Laplacian with mixed boundary conditions, Dirichlet on $C_{+}(0,r)$ and Neumann on $(-r,r)$. Hence, $R_0(u) \ge \frac{\pi j_{0,1}^2}{r^2}$, the least Dirichlet eigenvalue of $D(0,r)$, the disk of center $0$ and radius $r$. The lemma follows.
\end{proof}

\begin{remarks}\label{R-hmn3-L38-d} The proof of Lemma~\ref{L-hmn3-L38-e} shows that a similar result holds if we fix the current boundary condition \eqref{E-lsbs3-10} on a given interval $(a,b) \subset (-r,r)$, and the Dirichlet boundary condition on $\partial D_{+}(0,r) \sm \big( (a,b)\times \set{0}\big)$.
\end{remarks}%
\smallskip

\begin{proof}[Proof of Lemma~\ref{L-L38}] We divide the proof of Lemma~\ref{L-L38} into several steps labeled {\textbf (A), (B),} \ldots.\smallskip

\noi \textbf{(A)} To the sequence of interior points $\set{x_n} \subset \Omega$ we associate a sequence $\set{w_n := w_{x_n}}$ in the sphere $\bS(U)$ (Lemma~\ref{L-L37}). Taking a subsequence if necessary, we may assume that $\set{w_n}$ converges to some $w \in \bS(U)$. Then, the convergence is uniform in $C^m$ for any fixed $m \ge 0$. Since $\nu(w_n,x_n) = 2(k-1)$, or equivalently $\ord(w_n,x_n) = (k-1)$, with $k \ge 3$, and since the convergence is uniform, we have $\ord(w,y) \ge (k-1) \ge 2$,  so that $y$ is a boundary singular point of the $\lambda_k$-eigenfunction $w$. \smallskip

 Define $p := \ord(w,y)$, $q:=\rho(w,y)$. Recall that $p = (q+1)$ in the Dirichlet case, and $p=q$ in the Robin case.\smallskip

By Lemma~\ref{L-hdn}, the (sub)sequence $\set{\cZ(w_n)}$ converges to $\cZ(w)$ in the Hausdorff distance. This in particular implies that the set $\cZ(w)$ is connected.\medskip

\noi \textbf{(B)} The singular points of the nodal set  $\cZ(w)$ are isolated. We can choose some point $y_0 \in \Gamma$ with $y_0 \not \in \cS_{\mathrm{b}}(w)$. According to Section \ref{S-lsbs}, there is a conformal map $E : \bH \to \Omega$ which extends smoothly to $\bbH$, sends $0$ to $y$, and the point at infinity on $\partial \bH$  to $y_0$.  Fix the map $E$, and choose some
$r_0 > 0$ such that the nodal set $\cZ(w\circ E)$  is contained in $D_{+}(0,r_0) \cup (-r_0,r_0)\times \set{0}$. For $n$ large enough, the nodal sets $\cZ(w_n\circ E)$ will also be contained in $D_{+}(0,r_0) \cup (-r_0,r_0)\times \set{0}$. \smallskip

To prove Assertion~(i) in Lemma~\ref{L-L38},  we apply Subsection~\ref{SS-lsbs5}, to the function $v$. Let $(\rho,\omega)$ be the polar coordinates at $0 \in \bbH$, $\xi = (\rho \, \cos\omega, \rho \, \sin\omega)$. We use the following notation,

\begin{equation*}\label{E-lsbs-18n}
\left\{
\begin{array}{l}
[r,\omega] := (r\, \cos\omega, r\, \sin\omega),\\[5pt]
C_{+}(0,r) := \set{[r,\omega] \mid \omega \in (0,\pi)},\\[5pt]
\text{and we fix~} \alpha_1 \in (0, \frac{\pi}{8}), \alpha_p := \frac{\alpha_1}{p}.
\end{array}
\right.
\end{equation*}
We consider the Dirichlet and  Robin boundary conditions separately.\medskip

\noid In the Dirichlet case, $v([\rho,\omega]) = a_v \, \rho^p \, \sin(p\, \omega) + \cO(\rho^{p+1})$, for some $a_v \neq 0$. Define the rays
\[\set{\omega = \omega_{j}\mid 1 \le j \le p-1},\]
where $\omega_{j} := j \frac{\pi}{p}$.  As in \eqref{E-lsbs5-18}, consider the ``$\cG\cR\cB$-arcs''
\begin{equation}\label{E-hmn3-L38-vda}
\left\{
\begin{array}{l}
\cG_{\mathrm{d}}(r,j) := \set{[r,\omega] \mid \omega \in (\omega_{j} - \alpha_p,\omega_{j}+\alpha_p)}, \text{~for~} 1 \le j \le (p-1),\\[5pt]
\cR_{\mathrm{d}}(r,j) := \set{[r,\omega] \mid \omega \in [\omega_{j} + \alpha_p,\omega_{j+1}-\alpha_p]}, \text{~for~} 0 \le j \le (p-1),\\[5pt]
\cB_{\mathrm{d}}(r,0) := \set{[r,\omega] \mid \omega \in (0,\alpha_p]},\\[5pt]
\cB_{\mathrm{d}}(r,p) := \set{[r,\omega] \mid \omega \in [\pi-\alpha_p,\pi)}.
\end{array}
\right.
\end{equation}

These arcs are displayed in Figure~\ref{F-gam0-L38} (left picture, here $p = \ord(v,0) = 8$ and $q = \rho(v,0) = 7$).

\begin{figure}[!ht]
  \centering
  \includegraphics[width=0.9\textwidth]{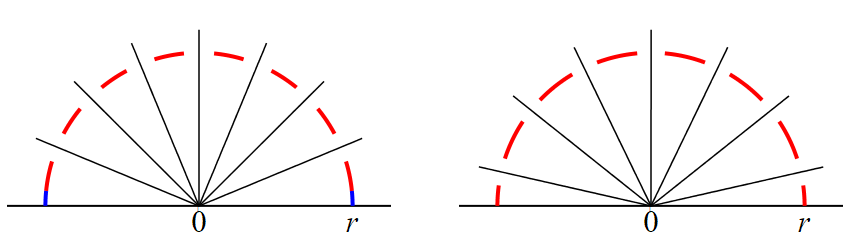}
  \caption{``$\cG\cR\cB$-arcs'' for $v$ with $\rho(v,0)=7$ (Dirichlet/Robin)}\label{F-gam0-L38}
\end{figure}

According to Proposition~\ref{P-lsbs5-2d},  there exists  some $r_1$, $0 < r_1 \le \frac{r_0}{2}$, such that the following properties hold for any $r \le r_1$, see Equation~\eqref{E-lsbs5-18a}.
\begin{equation}\label{E-hmn3-L38-vd}
\left\{
\begin{array}{l}
\pm \, (-1)^{j}\, \sign(a_v) \, v([r,\omega_j \pm \alpha_p]) \ge \frac 12 |a_v| \, \sin(\alpha_1)\, r^p.\\[5pt]
|\partial_{\omega}v([r,\omega])| \ge \frac p2 \, |a_v| \, \cos(\alpha_1) \, r^p \text{~~in each~} \cG_{\mathrm{d}}(r,j), 1 \le j \le (p-1),\\[5pt]
\hspace{0.8cm} \text{and~} v([r,\omega]) \text{~vanishes precisely once in each arc.} \\[5pt]
|\partial_{\omega}v([r,\omega])| \ge \frac p2 \, |a_v| \, \cos(\alpha_1) \, r^p \text{~~in~} \cB_r(r,0) \cup \cB_{\mathrm{d}}(r,p), \\[5pt]
\hspace{0.8cm} \text{and~} v(r,\omega) \text{~does not vanish in these arcs.}\\[5pt]
|v([r,\omega])| \ge \frac 12 \, |a_v|\, \sin(\alpha_1) \, r^p \text{~~in~} \cR_{\mathrm{d}}(r,j), 0 \le j \le (p-1), \\[5pt]
\hspace{0.8cm} \text{and~} v(r,\omega) \text{~does not vanish in these arcs.}
\end{array}
\right.
\end{equation} \smallskip

\noid In the Robin case, $v([\rho,\omega]) = a_v \, \rho^p \, \cos(p\, \omega) + \cO(\rho^{p+1})$. Define the rays
\[\set{\omega = \omega'_j\mid 1 \le j \le p},\]
where $\omega'_j := (j - \frac 12) \frac{\pi}{p}$. As in \eqref{E-lsbs5-18n}, consider the ``$\cG\cR\cB$-arcs''
\begin{equation}\label{E-hmn3-L38-vra}
\left\{
\begin{array}{l}
\cG_{\mathrm{n}}(r,j) := \set{[r,\omega] \mid \omega \in (\omega'_j - \alpha_p,\omega'_j+\alpha_p)}, \text{~for~} 1 \le j \le p,\\[5pt]
\cR_{\mathrm{n}}(r,j) := \set{[r,\omega] \mid \omega \in [\omega'_j + \alpha_p,\omega'_{j+1}-\alpha_p]}, \text{~for~} 1 \le j \le (p-1),\\[5pt]
\cR_{\mathrm{n}}(r,0) := \set{[r,\omega] \mid \omega \in (0,\omega'_1-\alpha_p]},\\[5pt]
\cR_{\mathrm{n}}(r,p) := \set{[r,\omega] \mid \omega \in [\omega'_p+\alpha_p,\pi)}.
\end{array}
\right.
\end{equation}
These arcs  are displayed in Figure~\ref{F-gam0-L38} (right picture, with $p = \ord(v,0) = 7$ and $q = \rho(v,0) = 7$).  According to Proposition~\ref{P-lsbs5-2n}, there exists some $r_1$, $0 < r_1 \le \frac{r_0}{2}$, such that the following properties hold for any $r \le r_1$, see Equation~\eqref{E-lsbs5-18na}.
\begin{equation}\label{E-hmn3-L38-vr}
\left\{
\begin{array}{l}
\mp \, (-1)^{j}\, \sign(a_v) \, v([r,\omega'_j \pm \alpha_p]) \ge \frac 12 |a_v| \, \sin(\alpha_1) \, r^p.\\[5pt]
|\partial_{\omega}v([r,\omega])| \ge \frac p2 \, |a_v| \, \cos(\alpha_1) \, r^p \text{~~in each~} \cG_{\mathrm{n}}(r,j), 1 \le j \le p, \\[5pt]
\hspace{0.8cm} \text{and~} v([r,\omega]) \text{~vanishes precisely once in each arc.} \\[5pt]
|v([r,\omega])| \ge \frac 12 \, |a_v| \, \sin(\alpha_1) \, r^p \text{~~in~} \cR_{\mathrm{n}}(r,0) \cup \cR_{\mathrm{n}}(r,p),\\[5pt]
\hspace{0.8cm} \text{and hence~} v([r,\omega]) \text{~does not vanish in these arcs.}\\[5pt]
|v([r,\omega])| \ge \frac 12 \, |a_v|\, \sin(\alpha_1) \, r^p \text{~~in~} \cR_{\mathrm{n}}(r,j), 1 \le j \le (p-1), \\[5pt]
\hspace{0.8cm} \text{and~} v(r,\omega) \text{~does not vanish in these arcs.}
\end{array}
\right.
\end{equation}
\smallskip

The arcs $\cG(r,j)$ appear in white on the semi-circles in  Figure~\ref{F-gam0-L38}. In both cases, $\cZ(v) \cap D_{+}(0,r_1)$ consists of $q$ nodal arcs  emanating from $0$, $r \mapsto \delta_j(r) := \big[ r, \omegat_j(r) \big]$, $1 \le j \le q$, where the functions $\omegat_j(r)$ are smooth for $0 < r < r_1$. These nodal arcs are transverse to the half circles $C_{+}(0,r)$.
\medskip

 Fix $r_2$ such that $0 < r_2 < r_1$. When $x_n$ tends to $y$ in $\wb{\Omega}$, the sequence $\set{w_n}$ tends to $w$ uniformly in the $C^m$ topology for any fixed $m$. Similarly, when the sequence $\set{\xi_n}$ tends to $0$ in $\bbH$, the sequence $\set{v_n}$ tends to $v$ uniformly in the $C^m$ topology in $D_{+}(0,r_0)$. This implies that the properties \eqref{E-hmn3-L38-vd} (Dirichlet case) and \eqref{E-hmn3-L38-vr} (Robin case) are satisfied by the functions $v_n$  provided that $r_2 < r < r_1$,  provided we relax the constant $|a_v|$ to $\frac{|a_v|}{2}$ in the right hand sides of the above inequalities, and provided we choose $n$ large enough implying that $v_n$ is $C^1$-close to $v$. This proves the following claim.

\begin{claim}\label{C-hmn3-L38-b}
Fix any $r_2$ such that $0 < r_2 < r_1$. Then, for $n$ large enough, depending on $r_2$, and for any $r$ such that $r_2 \le r \le r_1$,
\begin{equation}\label{E-hmn3-L38n}
\left\{
\begin{array}{l}
\cZ(v_n) \cap C_{+}(0,r) \subset \bigcup_{j=1}^{q} \cG(r,j),\\[5pt]
\text{and~}\\ \cZ(v_n) \text{~crosses each~} \cG(r,j) \text{~precisely once.}
\end{array}%
\right.
\end{equation}
\end{claim}%

From now on, we assume that:

\begin{enumerate}[$\diamond$]
  \item  $r_1 < r_E$ and  we fix $r_2$, $0 < r_2 < r_1$,
  \item   $n$ is large enough so that both \eqref{E-hmn3-L38-mu}  and  \eqref{E-hmn3-L38n}  are satisfied for any $r_2 < r < r_1$, with $\set{\xi_n} \subset D_{+}(0,r_2/2)$ close enough to $0$.
\end{enumerate}

\noi \textbf{(C)}~ Recall that $v_n = w_n\circ E$, $v = w \circ E$, and $x_n = E(\xi_n)$. We now study the sequence $\set{v_n}$, with $\set{\xi_n} \subset D_{+}(0,r_2/2)$. Taking Lemma~\ref{L-L37} into account,  there are two possible cases.\medskip

\noi \textbf{Case~C1}.~ \emph{There exists an infinite subsequence $\set{v_{s(n)}}$ of the sequence $\set{v_n}$ such that, for all $n$, $\cS_{\mathrm{b}} (v_{s(n)}) = \emptyset$.}\smallskip

In this case, according to the proof of Lemma~\ref{L-L37}, the nodal set $\cZ(v_{s(n)})$ consists of $(k -1)$ simple loops at $\xi_{s(n)}$. These loops do not intersect away from $\xi_{s(n)}$, and do not hit $\partial \bH$. Choose $\gamma$, any of these loops. Since $\xi_{s(n)} \in D_{+}( 0,r_2/2)$, either the loop $\gamma$ crosses $C_{+}(0,r_2)$ at (at least) two distinct points $z_{\gamma,1}$ and $z_{\gamma,2}$, or it is entirely contained in $D_{+}(0,r_2)$. In the latter case, the function $v_{s(n)}$ would have a nodal domain contained in $D_{+}(0,r_2)$. Taking into account our choice for $r_2$ and Lemma~\ref{L-hmn3-L38-e} (Energy argument), this would contradict \eqref{E-hmn3-L38-mu}, see Figure~\ref{F-hmn3-L38-C1} (right)\footnote{We draw the following pictures in a domain $\Omega$ whose boundary $\Gamma$ is a segment around $y$.}.  For each $n$, the set $\cZ(v_{s(n)})$ consists of $(k-1)$ loops which do not intersect away from $\xi_{s(n)}$. It follows that we have at least $(2k-2)$ distinct points $z_{s(n),j}$ in $C_{+}(0,r_2) \cap \cZ(v_{s(n)})$, for $1 \le j \le (2k-2)$.\smallskip

By Claim~\ref{C-hmn3-L38-b}, $\cZ(v_{s(n)})$ can only intersect $C_{+}(0,r_2)$ in the arcs $\cG(r_2,j)$, and at only one point for each $j$. This means that $q \ge (2k-2)$,  i.e., $\rho(w,y) = \rho(v,0) \ge (2k-2)$. From Lemma~\ref{L-L32} we infer that $w \in U_y$, and $\rho(w,y) = (2k-2)$, so that $y \in \Gamma_{(2k-2)}$.

\begin{figure}[t!h]
\centering
\includegraphics[width=0.6\linewidth,angle=0.3]{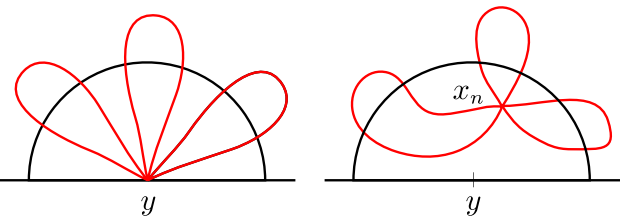}
\caption{Lemma~\ref{L-L38}, Case C1: nodal patterns for $w$ and $w_{s(n)}$} \label{F-hmn3-L38-C1}
\end{figure}

 Figure~\ref{F-hmn3-L38-FC1} displays forbidden configurations for the nodal sets $\cZ(w_{s(n)})$.
\begin{figure}[t!h]
\centering
\includegraphics[width=0.6\linewidth,angle=0.3]{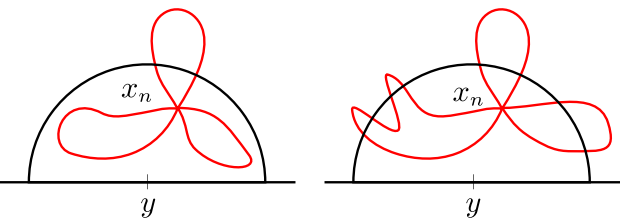}
\caption{Lemma~\ref{L-L38}, Case C1: forbidden configurations for $w_{s(n)}$} \label{F-hmn3-L38-FC1}
\end{figure}

\noi \textbf{Case~C2}.~ \emph{There exists an infinite subsequence $\set{v_{s(n)}}$ such that, for each $n$, we have $\cS_{\mathrm{b}} (v_{s(n)}) \neq \emptyset$.}\smallskip

In this case, according to the proof of Lemma~\ref{L-L32}, the nodal set $\cZ(v_{s(n)})$ consists of $(k -2)$ simple loops at $\xi_{s(n)}$, and two simple nodal intervals from $\xi_{s(n)}$ to the boundary $\partial \bH$. The loops do not intersect away from $\xi_{s(n)}$, and do not hit $\partial \bH$. The nodal intervals do not intersect each other, except at $\xi_{s(n)}$, and possibly on $\partial \bH$ if they hit $\partial \bH$  at the same point; they do not intersect the loops away from $\xi_{s(n)}$. The energy argument (Lemma~\ref{L-hmn3-L38-e}) used in Case~C1, shows that each loop in $\cZ(v_{s(n)})$ intersects $C_{+}(0,r_2)$ at (at least) two distinct points. A similar energy argument for mixed boundary conditions (Dirichlet on the nodal intervals, and the given boundary condition, Dirichlet or Robin, on $\partial \bH$) shows that the nodal intervals cannot both be contained in $D_{+}(0,r_2)$,   and at least one of them crosses $C_{+}(0,r_2)$, see Figure~\ref{F-hmn3-L38-C2}.  Indeed,  there would otherwise exist  a nodal  domain $\Omega_1$ of $v_{s(n)}$  with $\Omega_1 \subset D_{+}(0,r)$ as in Figure~\ref{F-hmn3-P} (such a domain appears in Figure~\ref{F-hmn3-L38-FC2a} (right), with $y_{s(n),2} = y_{s(n),1}$, and in Figure~\ref{F-hmn3-L38-FC2b} (left)). The function $v_{s(n)}|_{\Omega_1}$  would be a first eigenfunction of  $(-\Delta + J_E V\circ E) \, u = \lambda \, J_E \,u$ in $\Omega_1$, with the Dirichlet boundary condition on the nodal arcs from $\xi_{s(n)}$ to the boundary, and the given boundary condition \eqref{E-lsbs3-10} on the interval between $y_{s(n),1}$ and $y_{s(n),2}$, with associated first eigenvalue $\lambda_k$. Consider the function defined by
\begin{equation*}
f(x) := \left\{
\begin{array}{ll}
v_{s(n)}(x) & \text{~if\quad} x \in \Omega_1\\[5pt]
0 & \text{~if\quad} x \in D_{+}(0,r) \sm \Omega_1.
\end{array}%
\right.
\end{equation*}

The function $f$ satisfies the $B_E$ boundary condition (Dirichlet or Robin) on the segment $\big( y_{s(n),1}, y_{s(n),2} \big)$, vanishes on the nodal arcs from $x_{s(n)}$ to $y_{s(n),1}$ and $y_{s(n),2}$, and is $0$ outside $\Omega_1$.

\begin{figure}[t!h]
\centering
\includegraphics[width=0.5\linewidth]{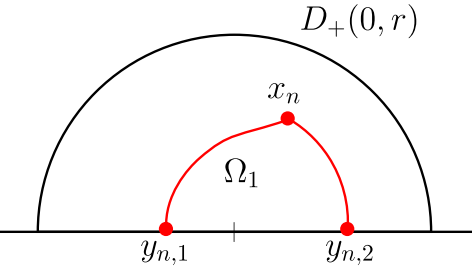}
\caption{Proof of Lemma~\ref{L-L38}: the domain $\Omega_1$} \label{F-hmn3-P}
\end{figure}

Then, $f \in H^1(D_{+}(0,r))$ and vanishes on $\partial D_{+}(0,r)$, except possibly on the segment $\big( y_{s(n),1} , y_{s(n),2} \big)$.
Looking at the quadratic forms,  and using Lemma~\ref{L-hmn3-L38-e} and Remark~\ref{R-hmn3-L38-d}, we conclude that
\[
\lambda_k = \mu_1(\Omega_1) > \mu_1(D_{+}(0,r)) > \lambda_k\,,
\]
a contradiction. \smallskip

Finally, for each $n$, we have at least $(2k-3)$   distinct points in $C_{+}(0,r_2) \cap \cZ(v_{s(n)})$, for $1 \le j \le (2k-3)$. As in Case~C1, we conclude that $q \ge (2k-3)$, i.e., that $\rho(w,y) = \rho(v,0) \ge (2k-3)$ so that $w \in U_y$, and we have two possible cases, either $\rho(w,y) = (2k-3)$ and $y \in \Gamma_{(2k-3)}$, or
$\rho(w,y) = (2k-2)$ and $y \in \Gamma_{(2k-2)}$. \smallskip

Claim~\ref{C-hmn3-L38-b} also implies that, for $n$ large enough, $Z(v_{s(n)})$ meets $C_{+}(0,r_2)$ at precisely $(2k-3)$ or $(2k-2)$ points. \smallskip

At this stage we have proved that the only possible limit points of a sequence $\set{w_n}$ are in $U_y$,  see Figure~\ref{F-hmn3-L38-C1} (left).  Since $\dim U_y = 1$, this proves Assertion~(i) of the lemma. \quad \qedc \medskip

According to  Lemma~\ref{L-L32}, we have $\cS_{\mathrm{b}}(v_{s(n)}) = \set{\eta_{s(n),1},\eta_{s(n),2}}$ possibly with $\eta_{s(n),1} = \eta_{s(n),2}$. Recalling that $y=E(0)$, the only possible limit points of these sequen\-ces are
\begin{equation*}
\left\{
  \begin{array}{ll}
    0 & \text{~if\quad} \rho(v,0) = (2k-2),\\[5pt]
    0\text{~and~} \zeta & \text{~if\quad} \rho(v,0) = (2k-3),
  \end{array}
\right.
\end{equation*}
 where $v = u_y\circ E$, $\cS_{\mathrm{b}}(v) = \set{0,\zeta}$, with $\zeta \neq 0$. When $\rho(v,0) = (2k-3)$, $\rho(v,\zeta)=1$, and the function $\breve{v}$ vanishes and changes sign at $0$ and $\zeta$. Since $\set{v_{s(n)}}$ converges to $v$ $C^{1}$-uniformly, this implies that, for $n$ large enough, the function $\breve{v}_{s(n)}$ changes sign near $0$ and near $\zeta$, and hence that one sequence, say $\set{\eta_{s(n),2}}$ tends to $\zeta$, and the other $\set{\eta_{s(n),1}}$ tends to $0$. Note that they cannot both tend to $\zeta$ since $\rho (v,\zeta)=1$.\smallskip

When $\rho(v,0) = (2k-2)$, the sequences $\set{\eta_{s(n),1}}$ and $\set{\eta_{s(n),2}}$ must both converge to $0$.\smallskip

Applying Claim~\ref{C-hmn3-L38-b}, we find that there are three sub-cases.
\begin{description}
  \item[C2\,(i)] There exists a subsequence $\set{v_{s(n)}}$ such that $\set{\eta_{s(n),1}}$ and  $\set{\eta_{s(n),2}}$ coincide and tend to $0$. For energy reasons  ($r_1 < r_E$), the arcs from $\xi_{s(n)}$ to $\eta_{s(n),1}$ and $\eta_{s(n),2}$ cannot both be contained in $D_{+}(0,r_2)$. One of these arcs, and actually only one for $n$ large enough, has to meet $C_{+}(0,r_2)$ at two distinct points,  see  Figure~\ref{F-hmn3-L38-C2} (left).
  \item[C2\,(ii)] There exists a subsequence $\set{v_{s(n)}}$ such that $\eta_{s(n),1} \not = \eta_{s(n),2}$, and both tend to $0$. For energy reasons ($r_1 < r_E$), the arcs from $\xi_{s(n)}$ to $\eta_{s(n),1}$ and $\eta_{s(n),2}$ cannot both be contained in $D_{+}(0,r_2)$. One of these arcs, and actually only one for $n$ large enough, has to meet $C_{+}(0,r_2)$ at two distinct points.  See Figure~\ref{F-hmn3-L38-C2} (center).
  \item[C2\,(iii)] There exists a subsequence $\set{v_{s(n)}}$ such that $\eta_{s(n),1} \not = \eta_{s(n),2}$, with  $\eta_{s(n),1}$ tending to $0$ and $\eta_{s(n),2}$ tending to some $\zeta \not = 0$. For $n$ large enough, the arc from $\xi_{s(n)}$ to $\eta_{s(n),2}$ intersects  $C_{+}(0,r_2)$ at one point, and the arc from $\xi_{s(n)}$ to $\eta_{s(n),1}$ stays inside $D_{+}(0,r_2)$.  See Figure~\ref{F-hmn3-L38-C2} (right).
\end{description}

In subcases C2\,(i) and C2\,(ii), we have $\rho(w,y) = \rho(v,0) = (2k-2)$, so that $y$ belongs to $\Gamma_{(2k-2)}$. In subcase C2\,(iii), we have $\rho(w,y) = \rho(v,0) = (2k-3)$, so that $y$ belongs to $\Gamma_{(2k-3)}$, with $z(y) = E(\zeta)$, the limit of $\set{E(\eta_{s(n),2})}$.

\begin{figure}[t!h]
\centering
\includegraphics[width=0.9\linewidth,angle=0.3]{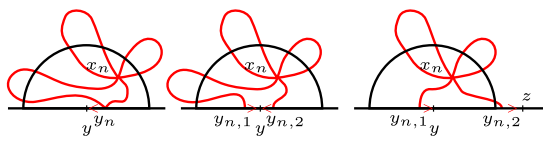}
\caption{Lemma~\ref{L-L38}, Case C2: nodal patterns for $w_{s(n)}$}\label{F-hmn3-L38-C2}
\end{figure}

 Figures~\ref{F-hmn3-L38-FC2a} and \ref{F-hmn3-L38-FC2b} display forbidden configurations for the nodal sets $\cZ(w_{s(n)})$ when $\cS_{\mathrm{b}}(w_{s(n)}) \neq \emptyset\,$.
\begin{figure}[t!h]
\centering
\includegraphics[width=0.6\linewidth,angle=0.3]{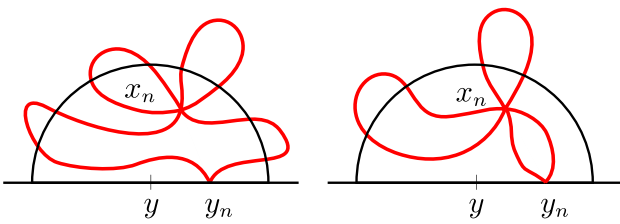}
\caption{Lemma~\ref{L-L38}, Case C2\,(i): forbidden configurations for $w_{s(n)}$}\label{F-hmn3-L38-FC2a}
\end{figure}

\begin{figure}[t!h]
\centering
\includegraphics[width=0.9\linewidth,angle=0.3]{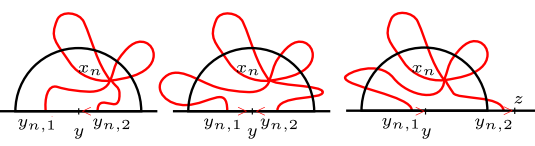}
\caption{Lemma~\ref{L-L38}, Case C2(ii)/(iii): forbidden configurations  for $w_{s(n)}$} \label{F-hmn3-L38-FC2b}
\end{figure}
\FloatBarrier

This proves Assertions~(ii) and (iii). \quad \qedc \medskip

 \emph{Assertion~(iv).}~ Since $\wb{\Omega}$ and $\bS(U)$ are simply connected, the map from $\wb{\Omega}$ to $\bP(U)$ given by Assertion~(i) can be lifted to the $\bS(U)$ with the desired properties. \quad \qedc \medskip

The proof of Lemma~\ref{L-L38} is complete.
\end{proof}

\begin{remark}\label{R-hmn3-L38-r1}  As a byproduct of the proof of Lemma~\ref{L-L38}, we obtain the configurations of the nodal sets $\cZ(w_x)$ when $x \in \Omega$ is close to  some $y \in \Gamma $. When $y \in \Gamma_{(2k-3)}$, $\cS_{\mathrm{b}}(u_y) = \set{y, z(y)}$ with $z(y) \neq y$. For $x$ close enough to $y$, $\cS_{\mathrm{b}}(w_x) = \set{y(x), z(x)}$ with $y(x) \neq z(x)$, $y(x)$ close to $y$ and $z(x)$ close to $z(y)$.  When $y \in \Gamma_{(2k-2)}$, $\cS_{\mathrm{b}}(u_y) = \set{y}$. For $x$ close enough to $y$, $\cS_{\mathrm{b}}(w_x)$ might be empty or consist of one or two points close to $y$.
\end{remark}%

Figures~\ref{F-hmn3-L38-C1} and \ref{F-hmn3-L38-C2} display typical configurations for $\cZ(w_{n})$ when $r_1 < r_E$ is small enough and $n$ large: the loops intersect $E(C_{+}(0,r_2))$ at two distinct points; one arc exits $E(D_{+}(0,r_0))$.\smallskip

Figures~\ref{F-hmn3-L38-FC1}, \ref{F-hmn3-L38-FC2a} and \ref{F-hmn3-L38-FC2b} display forbidden configurations for $\cZ(w_{n})$ when $r_2$ is small enough and $n$ large: a loop cannot be contained in  $E(D_{+}(0,r_2))$; the arcs cannot both be contained in $E(D_{+}(0,r_2))$;  the arcs cannot both meet $E(C_{+}(0,r_2))$.

\begin{remark}\label{R-Fgene3}
The \emph{nodal patterns} displayed in the above figures hold for both the Dirichlet and Robin boundary conditions. Unless otherwise stated this remark applies to all figures of this section.
\end{remark}

\begin{remark}\label{R-L38-2}
For $x \in \Omega$, let $h_{x,(k-1)}(\bar{w}_x)$ be the first nonzero term in the Taylor expansion of $\bar{w}_x$ at the point $x$ (this is a harmonic polynomial of degree $(k-1)$). Then, the map $x \to h_{x,(k-1)}(\bar{w}_x)$ is continuous, and  extends continuously to $\overline{\Omega}$. Unfortunately, this extension is not so interesting because $\lim_{x \to y \in \Gamma} h_{x,(k-1)}(\bar{w}_x) = 0$ since $\bar{w}_x$ tends to $\bar{w}_y \in U_y$, so that $h_{y,(k-1)}(\bar{u}_y) = 0$.  See also the final comment in \cite{BeNP2016}. We will mainly use Assertions~(ii) and (iii).
\end{remark}%

\begin{remark}\label{R-L38-4}
In this Subsection~\ref{SS-hmn-33}, we considered the behavior of $\cZ(w_x)$  when $x$ tends to some fixed $y \in \Gamma$. The radii $r_0, r_E, r_1$ which appear in the proofs depend on $y$ and $E$. In the next sections we will need to take care of these constants for varying values of $y$.
\end{remark}%

\subsection{Behavior of the combinatorial types of the functions $u_y$}\label{SS-hmn-34}

\begin{lemma}\label{L-L33c}
Assume that $\Omega$ is simply connected. Let $U := U(\lambda_k)$ for some $k \ge 3$. Assume that $\dim U = (2k-2)$. [Assumptions~\ref{A-hmn3-0}].  Then, the following properties hold.
\begin{enumerate}[(i)]
  \item  The  combinatorial type of a generator $u_y$ of $U_y$ is constant in any component of $\Gamma_{(2k-3)}$ in the sense that the maps $y \mapsto \tau_{y}^{U}(\downarrow)$ and $y \mapsto \tau_{y}^{U}$ are constant in each  component of $\Gamma_{(2k-3)}$.
  \item \emph{Assume that} $\Gamma_{(2k-2)}$ \emph{is not empty}, and let $\eta \in \Gamma_{(2k-2)}$. The eigenfunctions $u_y$ have different combinatorial types on either sides of $\eta$.  Then, $\#\big( \Gamma_{(2k-2)} \big) \ge 2$.
  \item \emph{Assume that} $\Gamma_{(2k-2)}$ \emph{is not empty}. Let $\eta_1 \neq \eta_2$ be two points of $\Gamma_{(2k-2)}$ such that the open arc $\cA(\eta_1,\eta_2)$ is contained in $\Gamma_{(2k-3)}$. The combinatorial types $\tau_{\eta_1}$ and $\tau_{\eta_2}$ are different.
  \item \emph{Assume that} $\Gamma_{(2k-2)}$ \emph{is not empty},   then $\#\big( \Gamma_{(2k-2)} \big)$ is an even positive integer.
\end{enumerate}
\end{lemma}%

\begin{proof}[Proof of Lemma~\ref{L-L33c}]

\emph{Assertion~(i).} Let $C$ be a component of $\Gamma_{(2k-3)}$.  For $y \in C$, define the number $\alpha(y) = \tau_{y}^{U}(\downarrow) \in L_{(2k-3)}$.  Assume that the map $y \mapsto \alpha(y)$ is not locally constant. Then, there exists $y \in C$ and a sequence $y_n$ tending to $y$ in $C$ such that $\alpha(y_n) \neq \alpha(y):=a$. Since the map $\alpha$ takes finitely many values, after taking a subsequence if necessary, we may assume that $\alpha(y_n) = b \neq a$. Let $u_n := u_{y_n}$. The nodal interval of $\cZ(u_n)$ which emanates from $y_n$ tangentially to the ray $\omega_b$ hits the boundary at the point $z_n := z(y_n)$. Since the sequence $\set{u_n}$ converges to $u_y$ in the $C^m$ topology for any fixed $m$, taking subsequences if necessary, we may assume that $\cZ(u_n)$ converges to $\cZ(u_y)$ in the Hausdorff distance, and that $\set{z_n}$ converges to $z(y)$. On the other hand, we can apply the local structure theorem to the functions $u_n$ in a neighborhood of $y$: the arcs emanating from $y_n$ intersect a circle of radius $\varepsilon$
 (with $\varepsilon$ independent of $n$), at points $x_{n,j}, 1 \le j \le 2k -3$ and these points converge to the corresponding points $x_j, 1 \le j \le 2k -3$, for the function $u_y$. To prove that $\varepsilon$ can be taken independent of $n$ we use the fact that, for any fixed $m$, the derivatives of $u_n$ of order less than or equal to $m$ converge uniformly to the corresponding derivatives of $u_y$ so that the remainder term in Taylor's formula can be controlled independently of $n$, see the proof of the local structure theorem in   Section~\ref{S-lsbs}. The arc in $\cZ(u_n)$ between $x_{n,b}$ and $z_n$ must tend in the Hausdorff distance to the arc in $\cZ(u_y)$ between $x_b$ and $y$, and we get a contradiction since $b\neq a$. It follows that the map $\alpha$ is locally constant, hence constant, on the component $C$.  Since the map $y \mapsto \tau_{y}^{U}(\downarrow)$ is constant on $C$, the set $L_{(2k-3)} \sm \set{\tau_{y}^{U}(\downarrow)}$ is constant, and it suffices to look at the restriction of $\tau_{y}^{U}$ to this set. To prove that the map $y \mapsto \tau_{y}^{U}$ is locally constant on $C$ we can reproduce the arguments in the proof of  Property~\ref{propertyD4} or
Lemma~\ref{L-hmn-202}. This proves Assertion~(i).\quad \qedc \smallskip

\emph{Assertion~(ii).} Since $\eta$ is isolated in $\Gamma_{(2k-2)}$, it suffices to work locally near $\eta$. \emph{Assume by contradiction}, that the combinatorial type of $u_y$ is the same on either sides of $\eta$ for $y$ close to $\eta$.\smallskip

\noid In the framework of Paragraph~\ref{SSS-hmn-31bb}, let $E_{\eta}$ be a conformal mapping from $\bH$ to $\Omega$ whose extension to the boundary sends $0$ to $\eta$. Let $r_0$ be small enough so that $E_{\eta}((-r_0,r_0)\times \set{0})$ is a neighborhood of $\eta$ in $\Gamma$ whose intersection with $\Gamma_{(2k-2)}$ is reduced to $\set{\eta}$. Let $\set{v_t}$ be the associated family of functions given by \eqref{L-L32}. \smallskip

According to our assumption, for $t \neq 0$ small enough, the combinatorial type of $v_t$ is constant. For $t \neq 0$, let $\cS_{\mathrm{b}}(v_t) = \set{t,z(t)}$. According to Lemma~\ref{L-L33b}, when $t$ tends to $0$, the point $z(t)$ tends to $0$, with $z(t) < 0$ when $t > 0$, and  $z(t) > 0$ when $t < 0$.\medskip

\noid We first consider the simple case $k=4$ and $a=3$ displayed in Figure~\ref{F-hmn3-L33c-2}.  The numbers between brackets are the labels of the rays. The numbers between braces are the labels of the nodal domains according to their order of appearance along a small half-circle centered at $t$, moving counter-clockwise. The labeling word for the nodal domains of $v_t$ is $\ww_t = |1|2|1|3|4|3|$. From our assumption, it is constant for $t$ small. The nodal interval from $t$ to $z(t)$ separates the domain $\Omega$ into two connected components and bounds the nodal domains labeled $\{1\}$ and $\{3\}$.

\begin{figure}[!ht]
  \centering
  \includegraphics[width=0.99\textwidth]{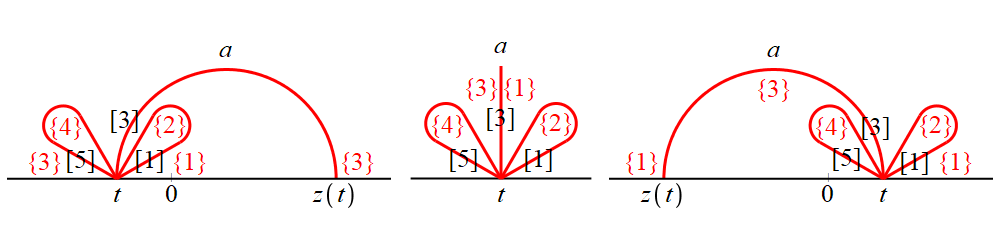}
  \caption{Example with $k=4$ and $a=3$}\label{F-hmn3-L33c-2}
\end{figure}

The combinatorial type of $v_t$ is given by
\begin{equation*}
\tau = \begin{pmatrix}
         \downarrow & 1 & 2 & 3 & 4 & 5 \\
         3 & 2 & 1 & \downarrow & 5 & 4 \\
       \end{pmatrix}.
\end{equation*}

The arguments in the proof of Lemma~\ref{L-L38} show that the combinatorial type of $v_{0,L}$, the limit of $v_t$ when $t$  tends to $0$ from the left,  is determined by the combinatorial type of $v_t$. When $t$ tends to $0$ from the left, $z(t)$ tends to $0$ from the right, and the nodal interval from $t$ to $z(t)$ closes up to form a loop $\gamma_{0,a}$ in the nodal set of $v_{0,L}$.  We write the combinatorial type of $v_{0,L}$ as
\begin{equation*}
\tau_{L} = \begin{pmatrix}
             0 & 1 & 2 & 3 & 4 & 5\\
             3 & 2 & 1 & 0 & 5 & 4 \\
           \end{pmatrix},
\end{equation*}
 so that the ray previously labeled $a$ (here $3$) keeps the same label.  Following the deformation of the nodal domains, when $t$ tends to zero from the left, we see that the nodal word $\ww_t$ yields the word $\ww_L = |3|1|2|1|3|4|3|$ for $v_{0,L}$.

\begin{figure}[!ht]
  \centering
  \includegraphics[width=0.99\textwidth]{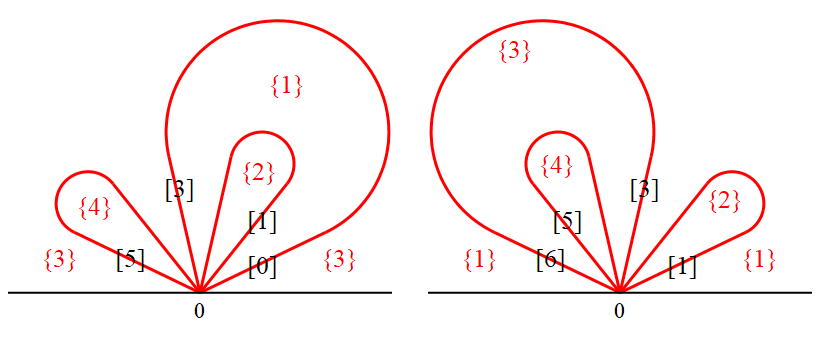}
  \caption{$k=4$ and $a=3$: $\cZ(v_{0,L})$ and $\cZ(v_{0,R})$}\label{F-hmn3-L33c-4}
\end{figure}

Similarly, when $t$ tends to $0$ from the right, we obtain a function $v_{0,R}$
whose combinatorial type is given by
\begin{equation*}
\tau_{R} = \begin{pmatrix}
             1 & 2 & 3 & 4 & 5 & 6 \\
             2 & 1 & 6 & 5 & 4 & 3 \\
           \end{pmatrix}
\end{equation*}
 and whose nodal set contains a loop $\gamma_{a,6}$ (the ray previously labeled $a$, here $a=3$, keeps the same label). Following the deformation of the nodal domains, when $t$ tends to zero from the right, we see that the nodal word $\ww_t$ yields the word $\ww_R = |1|2|1|3|4|3|1|$ for $v_{0,R}$. This is illustrated in Figure~\ref{F-hmn3-L33c-4}. \smallskip

The \emph{signature} $\sigma(\ww)$ of a word $\ww$ \index{Signature of a word} is the least rank, greater than or equal to $2$ at which the first letter of the word reappears (see Section~\ref{S-hmn2L} for more details). We obtain
\begin{equation*}
\ww_{L} = |3|1|2|1|3|4|3|, ~~\sigma(\ww_L)=5  \text{~~~~and~~~~} \ww_{R} = |1|2|1|3|4|3|1|, ~~\sigma(\ww_R)=3
\end{equation*}
respectively. This shows that the combinatorial types $\tau_{L}$ and $\tau_{R}$ are different, contradicting the fact that $v_{0,L} = v_{0,R}$ according to Lemma~\ref{L-L32}.\medskip

The general case follows similar lines.  It is detailed in Subsection~\ref{SS-hmn2L-P2ii}. In particular, this excludes the case $\#(\Gamma_{(2k-2)})=1$. We have proved Assertion~(ii).\quad \qedc \medskip

\emph{Assertion~(iii).} We give the proof on an example. \smallskip

Assuming that $\Gamma_{(2k-2)}$ is not empty, we have $\#(\Gamma_{(2k-2)}) \ge 2$. Consider  two conse\-cutive points $\eta_1$ and $\eta_2$ in $\Gamma_{(2k-2)}$, i.e.,  $\cA(\eta_1,\eta_2)$ is a connected component of $\Gamma_{(2k-3)}$. From Assertion~(i), we know that the combinatorial type of $u_y$ is constant in $\cA(\eta_1,\eta_2)$.  For $y \in \cA(\eta_1,\eta_2)$, Lemma~\ref{L-L33b} and Lemma~\ref{L-L38} imply that the combinatorial types of $u_{\eta_1}$, $u_y$ and  $u_{\eta_2}$ are determined once one of them is known. \smallskip

We work out the following example: $k=10$ and the combinatorial type $\tau_y$ of $u_y$ is given by
\begin{equation*}
\tau_y = \begin{pmatrix}
\downarrow & 1 & 2 & 3 & 4 & 5 & 6 & 7 & 8 & 9 & 10 & 11 & 12 & 13 & 14 & 15 & 16 & 17 \\
         11&  6 & 3 & 2 & 5 & 4 & 1 & 10 & 9 & 8 & 7 & \downarrow & 15 & 14 & 13 & 12 & 17& 16\\
         \end{pmatrix}.
\end{equation*}

The nodal set $\cZ(u_y)$ is displayed in Figure~\ref{F-hmn3-L33c-iv-1x3-rays}, middle sub-figure. The labels between brackets are the labels of the rays at $y$. Recall that $\cS_{\mathrm{b}}(u_y) = \set{y,z(y)}$ with $z(y) \in \Gamma$ and $z(y) \neq y$. \smallskip

We have chosen $k=10$, i.e., $\rho(u_y,y) = 17$, so that the example looks as general as possible, see Section~\ref{S-hmn2L}. To draw a readable figure and accommodate the seventeen rays tangent to $\cZ(u_y)$ at $y$ and the various labels,  we have opened the half-plane like a fan.\smallskip

\begin{figure}[!ht]
  \centering
  \includegraphics[width=0.99\textwidth]{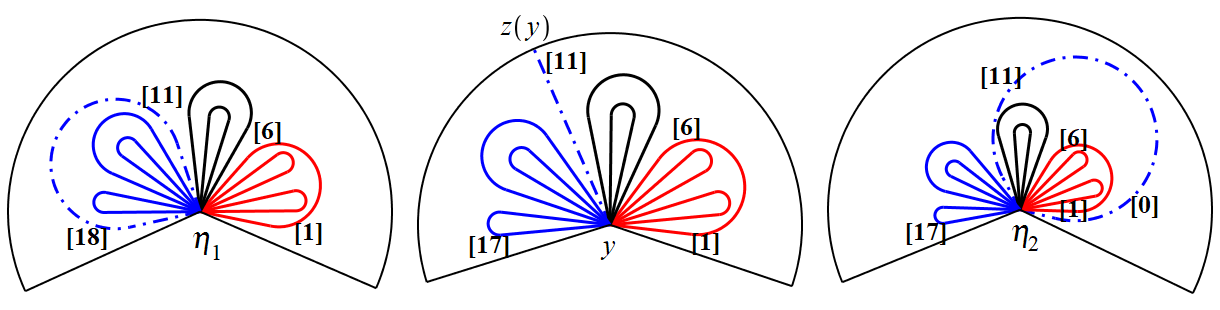}
  \caption{Transition from $u_{\eta_1}$ to $u_{\eta_2}$ (here $k=10$)}\label{F-hmn3-L33c-iv-1x3-rays}
\end{figure}

According to Lemma~\ref{L-L33b}, when $y \in \cA(\eta_1,\eta_2)$ moves clockwise to $\eta_1$, the point $z(y)$ moves counter-clockwise to $\eta_1$, and $u_y$ tends to some $u_{\eta_1} \in U_{\eta_1}$. The proof of Lemma~\ref{L-L38} shows that the loops in $\cZ(u_y)$ move continuously as $y$ moves, that the nodal interval $\delta_y^{z(y)} \subset \cZ(u_y)$ from $y$ to $z(y)$ tends to a loop $\gamma^{u_{\eta_1}}_{11,18}\subset \cZ(u_{\eta_1})$, and that the combinatorial type of  $u_{\eta_1}$ is
\begin{equation*}
\tau_{\eta_1} = \begin{pmatrix}
1 & 2 & 3 & 4 & 5 & 6 & 7 & 8 & 9 & 10 & 11 & 12 & 13 & 14 & 15 & 16 & 17 & 18 \\
6 & 3 & 2 & 5 & 4 & 1 & 10 & 9 & 8 & 7 & 18 & 15 & 14 & 13 & 12 & 17& 16 & 11 \\
         \end{pmatrix}.
\end{equation*}

This is illustrated in the left sub-figure of Figure~\ref{F-hmn3-L33c-iv-1x3-rays}. The loops in $\cZ(u_y)\sm \delta_y^{z(y)}$ and in $\cZ(u_{\eta_1})\sm \gamma^{u_{\eta_1}}_{11,18}$ have the same combinatorics and look very much alike.\smallskip

When $y$ moves counter-clockwise to $\eta_2$, $z(u)$ moves clockwise to $\eta_2$, and the proof of Lemma~\ref{L-L38} shows that the nodal interval $\delta_y^{z(y)}$ in $\cZ(u_y)$ closes up as a loop $\gamma^{u_{\eta_2}}_{0,11}$ in $\cZ(u_{\eta_2})$ tangent to a ray labeled $[0]$. This is illustrated in the right sub-figure  (upon arriving at $\eta_2$ a new ray pops up and we label it $\omega_{0}$ so that the labels of the other rays do no change, and we retain the counter-clockwise labeling of rays). This is a consequence of the local structure theorem applied to $u_{\eta_2}$, the limit of $u_y$ when $y$ tends to $\eta_2$ from the left. Then,
\begin{equation*}
\tau_{\eta_2} = \begin{pmatrix}
0 & 1 & 2 & 3 & 4 & 5 & 6 & 7 & 8 & 9 & 10 & 11 & 12 & 13 & 14 & 15 & 16 & 17\\
11 & 6 & 3 & 2 & 5 & 4 & 1 & 10 & 9 & 8 & 7 & 18 & 15 & 14 & 13 & 12 & 17& 16\\
         \end{pmatrix}.
\end{equation*}

The combinatorial types $\tau_{\eta_1}$ and $\tau_{\eta_2}$ look different. In order to prove that they are indeed different, we look at how the nodal domains of $u_y$ deform when $y$ moves in $\cA(\eta_1,\eta_2)$. We first choose $r$ small enough so that the local structure theorem holds (for the eigenfunctions we are interested in). Then, we label the nodal domains of $u_y$ according to their order of appearance while moving counter-clockwise along $C_{+}(y,r)$, taking into account that the intersection of a given nodal domain with $C_{+}(y,r)$ may consist of several disjoint intervals.  This labeling of the nodal domains of $u_y$ is displayed in the top sub-figure of Figure~\ref{F-hmn3-L33c-iv-2nd}. One can view the labeling as a map from the set of intervals determined by $\cZ(u_y)$ on $C_{+}(y,r)$, $\set{1, \ldots, 18}$, to the set of nodal domains of $u_y$ which has $10$ elements. Equivalently, one can view the labeling as a word of length $18$ in the $10$ letters $1, \ldots , 10$, the labels of the nodal domains, separated by a vertical bar $|$.  The labeling word $\ww_y$ corresponding to $u_y$ in the example appears at the bottom of the top sub-figure in Figure~\ref{F-hmn3-L33c-iv-2nd}
\begin{equation*}
\cW_y = |1|2|3|2|4|2|1|5|6|5|1|7|8|9|8|7|10|7|.
\end{equation*}

The nodal interval $\delta_y^{z(y)}$ divides $\Omega$ into two components, one component on the right of the nodal interval which contains $5$ loops and hence $6$ nodal domains, labeled from $1$ to $6$; and another component on the left of the nodal interval which contains $3$ loops and hence $4$ nodal domains labeled from $7$ to $10$. The nodal interval $\delta_{y}^{z(y)}$ is the common boundary of the nodal domains labeled $1$ and $7$. \smallskip

\begin{figure}[!ht]
  \centering
  \includegraphics[width=0.99\textwidth]{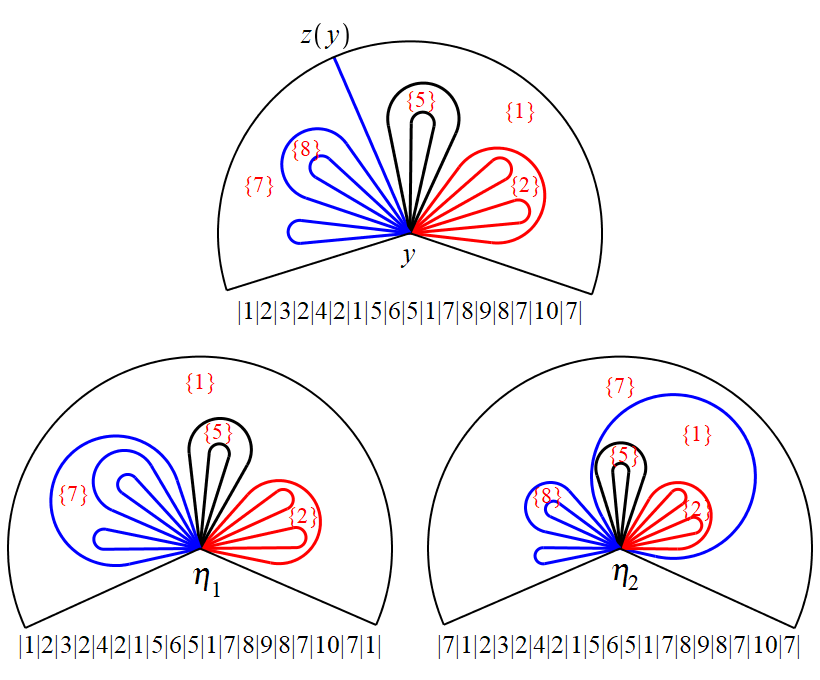}
  \caption{Words for $u_{\eta_1}$ and $u_{\eta_2}$ deduced
  from the word for $u_y$}\label{F-hmn3-L33c-iv-2nd}
\end{figure}

When $y$ moves, the nodal domains of $u_y$ deform but the corresponding labeling word does not change. When $y$ reaches $\eta_1$, the nodal interval $\delta_y^{z(y)}$ closes up to form the loop $\gamma^{u_{\eta_1}}_{11,18}$ in $\cZ(u_{\eta_1})$ and the nodal domain labeled $1$ contains $\Gamma$ in its boundary. When $y$ reaches $\eta_2$, the nodal interval $\delta_y^{z(y)}$ closes up to form the loop $\gamma^{u_{\eta_2}}_{0,11}$ in $\cZ(u_{\eta_2})$ and the nodal domain labeled $7$ contains $\Gamma$ in its boundary. The corresponding words appear at the bottom of the bottom sub-figures in Figure~\ref{F-hmn3-L33c-iv-2nd},
\begin{equation*}
\begin{array}{l}
\cW_{\eta_1} = |1|2|3|2|4|2|1|5|6|5|1|7|8|9|8|7|10|7|1| = \cW_y|1|,\\[5pt]
\cW_{\eta_2} = |7|1|2|3|2|4|2|1|5|6|5|1|7|8|9|8|7|10|7| = |7|\cW_y\,.
\end{array}%
\end{equation*}

The word $\cW_{\eta_1}$ is obtained from the word $\cW_y$ by adding the letter $1$ at the end; the word $\cW_{\eta_2}$ by adding the letter $7$ at the beginning. In $\cW_{\eta_1}$, the first letter of the word is $1$ and it first reappears as the  7th letter. In $\cW_{\eta_2}$, the first letter is $7$ and it first reappears is as the 13th letter, so that for the signatures, $\sigma(\cW_{\eta_1}) \neq \sigma(\cW_{\eta_2}) $. This shows that the functions $u_{\eta_1}$ and $u_{\eta_2}$ have different combinatorial types. This proves Assertion~(iv) for the above example. \quad \qedc \smallskip

The general case follows similar lines, see Subsection~\ref{SS-hmn2L-P2iv}.  \medskip

 \emph{Assertion~(iv).}~ Look at the example in Figure~\ref{F-hmn3-L33c-iv-2nd}. The function $u_y$ changes sign across the nodal interval $\delta_y^{z(y)}$, i.e., it has different signs in the nodal domains labeled $1$ and $7$. For $i=1,2$,  the function $\breve{u}_{\eta_i}$ only vanishes at $\eta_i$ and has a constant sign on $\Gamma\sm\set{\eta_i}$. Letting $y$ tend to $\eta_1$, resp. $\eta_2$, in the above argument, we infer that $\breve{u}_{\eta_1} \cdot \breve{u}_{\eta_2} < 0$  on $\Gamma \sm \set{\eta_1,\eta_2}$. This property is general and does not depend on the particular example.  According to Lemma~\ref{L-L38}\,(iv), we have a globally defined continuous function  $\Gamma \ni y \mapsto u_y \in \bS(U)$. In view of the previous property, this implies that the number of point in $\Gamma_{(2k-2)}$ is even. \quad \qedc\smallskip

The proof of Lemma~\ref{L-L33c} is complete.
\end{proof}

\FloatBarrier

\section[Nodal Sets of $\lambda_k$-Eigenfunctions]{Nodal Sets of $\lambda_k$-Eigenfunctions under  Assumptions~5.2}\label{S-gam0}

In this section, we continue to work under Assumptions~\ref{A-hmn3-0}.

\subsection{$\cZ(w_x)$ for $x$ close to $y \in \Gamma_{(2k-3)}$, local picture}
\label{SS-gam01}

In order to describe the combinatorial type of the eigenfunction $w_x$ when $x \in \Omega$ is close to some boundary point $y \in \Gamma_{(2k-3)}$, we first review the description of $\cZ(u_y)$. \smallskip

\fbox{1} Fix some $y \in \Gamma_{(2k-3)}$. The nodal set $\cZ(u_y)$ only hits $\Gamma$ at $y$ and some $z(y)\neq y$. Fix a conformal mapping $E$ from $\bH$ to $\Omega$ such that $E(0)=y$, $E(\zeta) = z(y)$.  According to the proof of Lemma~\ref{L-L38} and Section~\ref{S-lsbs}, $\cZ(u_y) \subset E\big( D_{+}(0,r_0) \big)$ for some  $r_0 > 0$.  We now work locally in $\bH$ rather than in $\Omega$, using the conformal mapping $E$. For the sake of simplicity, we identify the function $u_y\circ E$ with the function $u_y$ and, for $x \in \Omega$, the function $w_x\circ E$ with $w_x$.  We use the same notation $D_{+}(y,r)$, resp. $C_{+}(y,r)$, to denote $D_{+}(0,r)$, resp. $C_{+}(0,r)$, and their images under $E$. With these identifications, we write $\cZ(u_y) \subset D_{+}(y,r_0)$. Then, for $x$ close enough to $y$, we also have $\cZ(w_x) \subset D_{+}(y,r_0)$. \smallskip

\fbox{2} Recall from Subsection~\ref{SS-hmn-31b}, that the nodal set $\cZ(u_y)$ can be described in terms of the combinatorial type of the eigenfunction $u_y$,
\begin{equation}\label{E-gam0-2u}
\tau := \tau_{u_y} = \begin{pmatrix}
          1 & \ldots & (a-1) & a & (a+1) & \ldots & (2k-3) \\
          \tau(1) & \ldots & \tau(a-1) & \downarrow & \tau(a+1) & \ldots
          & \tau(2k-3) \\
       \end{pmatrix}.
\end{equation}
 We can add to $\tau$ a last, resp. first, column
$\begin{pmatrix} \downarrow \\  a \\ \end{pmatrix}$ to take into account the fact that
\begin{equation*}\label{E-gam0-2uL}
\tau_{\eta_1} = \begin{pmatrix}
          1 & \ldots & (a-1) & a & (a+1) & \ldots & (2k-3) & (2k-2) \\
          \tau(1) & \ldots & \tau(a-1) & (2k-2) & \tau(a+1) & \ldots
          & \tau(2k-3) & a \\
       \end{pmatrix}
\end{equation*}
and
\begin{equation*}\label{E-gam0-2uR}
\tau_{\eta_2}  = \begin{pmatrix}
          0 &1 & \ldots & (a-1) & a & (a+1) & \ldots & (2k-3) \\
          a & \tau(1) & \ldots & \tau(a-1) & 0 & \tau(a+1) & \ldots
          & \tau(2k-3) \\
       \end{pmatrix},
\end{equation*}
when $y \in \cA(\eta_1,\eta_2)$, with $\eta_1 \neq \eta_2$ two consecutive points in $\Gamma_{(2k-2)}$,  so that the initial rays keep their labels while
\begin{equation*}
\tau_{\eta_1}(a) = (2k-2) \text{~and~} \tau_{\eta_1}(2k-2) = a, \text{~resp.~} \tau_{\eta_2}(a)=0 \text{~and~} \tau_{\eta_2}(0)=a.
\end{equation*}
The nodal set $\cZ(u_y)$ contains a nodal interval $\delta_{y,a}$ which emanates from $y$ tangentially to the ray $\omega_{y,a}$ at $y$, for some $a \in J:= \set{1,\ldots,(2k-3)}$, and hits the boundary $\Gamma$ at the point $z(y) \neq y$. The nodal interval $\delta_{y,a}$ separates $\Omega$ into two components $\Omega_{y,-}$ on the right of $\delta_{y,a}$, and $\Omega_{y,+}$ on the left. For $1 \le j \le (a-1)$, the rays $\omega_{y,j}$  point inside $\Omega_{y,-}\,$. For $(a+1) \le j \le (2k-3)$, the rays $\omega_{y,j}$ point inside $\Omega_{y,+}$. The map $\tau$ leaves the subsets $J_{a,-} := \set{1,\ldots,(a-1)}$ and $J_{a,+} := \set{(a+1),\ldots,(2k-3)}$ globally invariant. Its restrictions $\tau_{\pm}$ to $J_{a,\pm}$ describe two bouquets of nodal loops $\cB_{y,\pm} := \set{\gamma^{y}_{j,\tau(j)}}_{j \in J_{a,\pm}}$ at the point $y$, contained respectively in $\Omega_{y,\pm}$. The nodal set $\cZ(u_y)$ is the wedge sum of the bouquets of loops $\cB_{y,\pm}$ with the nodal interval $\delta_{y,a}$.  There are $n_{y,-}:=(a+1)/2$ nodal domains of $u_y$ contained in $\Omega_{y,-}$ and $n_{y,+}:=[k-(a+1)/2]$ nodal domains contained in $\Omega_{y,+}$.  The nodal interval $\delta_{y,a}$ is a partial common boundary to the nodal domains $D_1$ and $D_{1+n_{y,-}}\,$. \smallskip

 We now apply the structure theorem of Section~\ref{S-lsbs} to the function $u_y$ at the point $y$. More precisely, we apply Propositions~\ref{P-lsbs5-2d} (Dirichlet case) and \ref{P-lsbs5-2n} (Robin case). We also retain the definitions and notation of this section. In particular, we use the notation $\cG_{y,j}(r)$ for the ``$\cG\cR\cB$-arcs'', defined by Equations~\eqref{E-lsbs5-18} (Dirichlet case) and \eqref{E-lsbs5-18n} (Robin case). Similar arguments can be found in the proof of Lemma~\ref{L-L38}.\smallskip

We are given some angle $\alpha_1 \in (0, \pih)$, and we have a radius  $r_{1,d}$ or $r_{1,n}$, given by \eqref{E-lsbs5-16} or \eqref{E-lsbs5-16n}. We now work in $D_{+}(y,r_1)$, with the radius $r_1$ satisfying
\begin{equation}\label{E-gam0-3}
\left\{
\begin{array}{l}
r_1 \le r_{1,d} \text{~or~} r_{1,n} \text{~(Dirichlet or Robin)}\\[5pt]
r_1 \le r_E \text{~(the energy radius given by Lemma~\ref{L-hmn3-L38-e}) }\\[5pt]
z(y) \not \in D_{+}(y,2r_1).
\end{array}%
\right.
\end{equation}

This choice of $r_1$ implies that the half disk $D_{+}(y,r)$ satisfies the energy inequality \eqref{E-hmn3-L38-mu} of Lemma~\ref{L-hmn3-L38-e}, for any $r \le r_1$. Fix some $r_2$, such that $0 < r_2 < r_1$. \smallskip

\begin{enumerate}[$\diamond$]
  \item In the Dirichlet case, we have the ``$\cG\cR\cB$-arcs'' \eqref{E-lsbs5-18} associated with the rays \eqref{E-lsbs5-12}, and the inequalities  \eqref{E-lsbs5-18a} satisfied by $u_y$ (alias $u_y\circ E = v$).
  \item In the Robin case, we have the ``$\cG\cR$-arcs'' \eqref{E-lsbs5-18n} associated with the rays \eqref{E-lsbs5-12n}, and the inequalities \eqref{E-lsbs5-18na} satisfied by $u_y$.
\end{enumerate}

The ``$\cG\cR\cB$-arcs'' appear in  Figure~\ref{F-gam0-0}, left image for the Dirichlet boundary condition, right image for the Robin boundary condition\footnote{The rays which appear in the other figures of this section are for the Dirichlet boundary condition. Except for the rays, the figures for the Neumann or Robin boundary condition are similar to the figures for the Dirichlet boundary condition.}.\smallskip

\begin{figure}[!ht]
  \centering
  \includegraphics[width=0.9\textwidth]{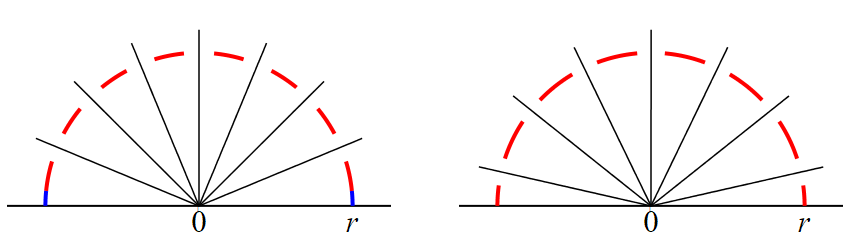}
  \caption{$\cG\cR\cB$-arcs for $u_y$, here $k=5$, $\rho(u_y,y) = 7$}\label{F-gam0-0}
\end{figure}

According to Propositions~\ref{P-lsbs5-2d} and \ref{P-lsbs5-2n}, for $0 < r < r_1$, the nodal arc $\delta_{y,j}$ emanating from $y$  {tangentially to the ray $\omega_{y,j}$} intersects the curve $C_{+}(y,r)$ at a unique point $A_{y,j}(r) = [r,\omegat_{y,j}(r)] \in \cG_{y,j}(r)$  in polar coordinates,  with $\omega_j - \alpha < \omegat_{y,j}(r) < \omega_j + \alpha$, and the nodal set $\cZ(u_y)$ does not intersect $C_{+}(y,r)$ elsewhere.  Here, $ \alpha = \alpha_1/ \ord(u_y,y)$ with $\alpha_1 \in (0, \frac{\pi}{8})$, see Subsection~\ref{SS-lsbs5}.\smallskip

\noi \emph{Sketch of the proof of the above properties}.~
Using the Taylor expansion of $u_y$, the fact that the arcs
$\delta_{y,j}$ meet the arcs $\cG_{y,j}(r)$ (white arcs on the
half-circles in the figure) follows from the intermediate value theorem and the
choice of $r_1$. The fact
that there is precisely one crossing point in each arc $\cG_{y,j}(r)$ follows
from the fact that the derivative of $u_y$ along the circle is  nonzero
in these arcs. The fact that the nodal set $\cZ(u_y)$ does not
meet the arcs $\cR_{y,j}(r)$ and $\cB_{y,k}(r)$ contained in $C_{+}(y,r) \sm
\big( \bigcup\, \cG_{y,j}(r) \big)$ follows
from the choice of $r_1$ and the fact that either  $u_y$ or its
derivative along the circle are controlled away from $0$ in these
arcs (the details are given in Section~\ref{S-lsbs} and in the
proof of Lemma~\ref{L-L38}). \quad \qedc \smallskip

The nodal set $\cZ(u_y)$ (viewed in $\bH$) is displayed in Figure~\ref{F-gam0-2}, with
\begin{equation*}
k = 5, a = 3, \text{~and~} \tau = \begin{pmatrix}
                                    1 & 2 & 3 & 4 & 5 & 6 & 7 \\
                                    2 & 1 & \downarrow & 5 & 4 & 7 & 6 \\
                                  \end{pmatrix}.
\end{equation*}
\begin{figure}[!ht]
  \centering
  \includegraphics[width=0.5\textwidth]{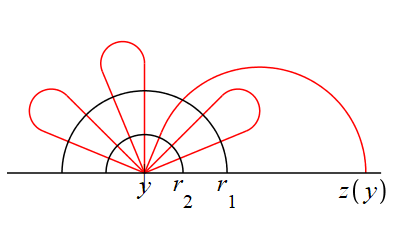}
  \caption{The nodal sets $\cZ(u_y)$ (here $k=5, a=3$)}\label{F-gam0-2}
\end{figure}
Look at $D_{+}^{c}(y,r_1)$, the complement of $D_{+}(y,r_1)$. The nodal arcs in $\cZ(u_y) \cap D_{+}^{c}(y,r_1)$ are pairwise disjoint and compact, so that they have disjoint neighborhoods of size $\varepsilon$,  for some $\varepsilon > 0$, $\cU^{\varepsilon}_{y,a}$, $\cU^{\varepsilon}(\tau_{-}) := \cup_{j=1}^{a-1} \cU^{\varepsilon}_{y,j}$, and $\cU^{\varepsilon}(\tau_{+}) := \cup_{j=a+1}^{2k-3} \cU^{\varepsilon}_{y,j}$. This is illustrated in Figure~\ref{F-gam0-4}.
\begin{figure}[!ht]
  \centering
  \includegraphics[width=0.5\textwidth]{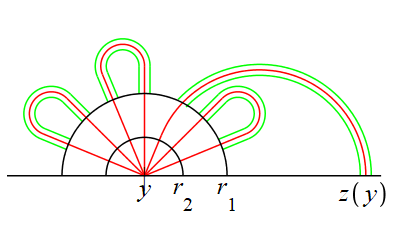}
  \caption{Neighborhoods of $\cZ(u_y)$ outside $D_{+}(y,r_1)$}\label{F-gam0-4}
\end{figure}
In the annulus $D_{+}(y,r_1) \sm D_{+}(y,r_2)$, we consider sectors containing the rays, as illustrated in Figure~\ref{F-gam0-6}. According to the local structure theorem for $u_y$, the intersection $\cZ(u_y) \cap \big( D_{+}(y,r_1) \sm D_{+}(y,r_2) \big)$ is contained in these sectors.
\begin{figure}[!ht]
  \centering
  \includegraphics[width=0.4\textwidth]{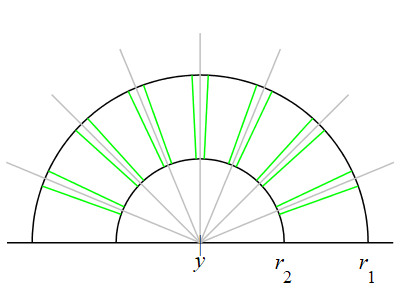}
  \caption{Sectors in $D_{+}(y,r_1) \sm D_{+}(y,r_2)$}\label{F-gam0-6}
\end{figure}
At the boundary, we consider the function $\breve{u}_y$ which vanishes precisely at the points $y$ and $z(y)$, and changes sign at both points. Fix some $\beta$, $0 < \beta < r_2$, such that $\cA(y;\beta) \subset \cA(y;r_2)$ and $\cA(z(y);\beta) \cap \cA(y;r_1) = \emptyset$, see Figure~\ref{F-gam0-7a}.
\begin{figure}[!ht]
  \centering
  \includegraphics[width=0.7\textwidth]{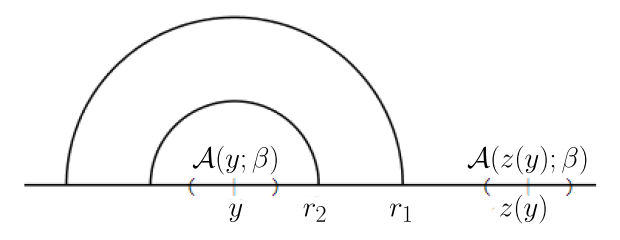}
  \caption{The setting at $y$}\label{F-gam0-7a}
\end{figure}

\fbox{3} We now look at the nodal sets $\cZ(w_x)$ when $x$ is close to $y$,  making use of the description of $\cZ(u_y)$ in  \fbox{2}.

Since $w_x$ tends to $u_y$ $C^1$-uniformly when $x$ tends to $y$, for $x$ close enough to $y$, the function $\breve{w}_x$ vanishes precisely once in both arcs $\cA(y;\beta)$ and $\cA(z(y);\beta)$, and only there, so that $\cS_{\mathrm{b}}(w_x)= \set{y(x),z(x)}$.  According to Lemma~\ref{L-L37}, the nodal set $\cZ(w_x)$ is the wedge sum of two nodal intervals, $\delta_x^{y(x)}$ from $x$ to $y(x)$ and $\delta_x^{z(x)}$ from $x$ to $z(x)$, and a $(k-2)$-bouquet of loops at $x$. \smallskip

Since $w_x$ tends to $u_y$ $C^1$-uniformly when $x$ tends to $y$, for $r_2 < r < r_1$, and for $x$ close enough to $y$, the function $w_x$ satisfies inequalities similar to the inequalities \eqref{E-hmn3-L38-vd} or \eqref{E-hmn3-L38-vr} satisfied by $u_y$ (see the proof of Lemma~\ref{L-L38}). It follows that $\cZ(w_x) \cap C_{+}(y,r) \subset \bigcup_{j=1}^q \, \cG_{y,j}(r)$ and $\cZ(w_x)$ crosses each $\cG_{y,j}(r)$ precisely once,  at a point denoted by $A_{x,j}(r)$. Since $z(x)$ lies outside $D_{+}(y,r_1)$, the nodal interval $\delta_x^{z(x)}$ must cross $C_{+}(y,r)$ precisely once. For energy reasons (choice of $r_1$), each nodal loop in $\cZ(w_x)$ intersects $C_{+}(y,r)$ precisely twice. Counting the points in $\cZ(w_x) \cap C_{+}(y,r)$, it  follows that the nodal interval $\delta_x^{y(x)}$, from $x$ to $y(x)$, does not cross $C_{+}(y,r)$ for $r_2 < r < r_1$, and hence is contained in $D_{+}(y,r_2)$. \smallskip

Given any $x \in \Omega$, there is a priori no natural labeling of the star\footnote{Recall that the star at $x$ is the collection of rays tangent to $\cZ(w_x)$ at the point $x$.} of $w_x$ at the point $x$. Assuming that $x$ is close enough to some $y \in \Gamma_{(2k-3)}$, the situation is different.  We have indeed  identified the nodal interval $\delta_x^{z(x)}$ which crosses $C_{+}(y,r)$ in $\cG_{y,a}(r)$, and we label $\omega_{x,a}$ the corresponding ray at $x$ accordingly. This gives us a natural labeling of the rays at $x$: we label the rays $\omega_{x,(a-1)}$ to $\omega_{x,1}$ counter-clockwise starting from $\omega_{x,a}$, and $\omega_{x,(a+1)}$ to $\omega_{x,(2k-2)}$ clockwise starting from $\omega_{x,a}$. In this labeling, we denote by $\omega_{x,b}$ the ray tangent to $\delta_x^{y(x)}$, for some $b \in \set{1,\ldots,(2k-2)} \sm \set{a}$. \smallskip

Since $w_x$ tends to $u_y$ in the $C^1$ topology, for $x$ close enough to $y$, $\cZ(w_x) \cap D_{+}^{c}(y,r_1)$ is contained in the neighborhood $\cU^{\varepsilon}(u_y) = \cup_{j=1}^{(2k-3)} \cU^{\varepsilon}_{y,j}$.\smallskip

Properties~\ref{P-gam0-2} summarize the above statements.

\begin{properties}\label{P-gam0-2}
Under Assumptions~\ref{A-hmn3-0},  let $y \in \Gamma_{(2k-3)}$ and let $\omega_{y,a}, 1 \le a \le (2k-3)$ be the ray at $y$ tangent to the nodal interval $\delta_y^{z(y)}$ in $\cZ(u_y)$ going from $y$ to $z(y)$. For $x$ close enough to $y$, the nodal set $\cZ(w_x)$ has the following properties.
\begin{enumerate}[(i)]
\item  There is one nodal interval, denoted by $\delta_x^{z(x)}$, going from $x$ to a point $z(x) \in \Gamma$ close to $z(y)$. We denote the ray at $x$ tangent to this nodal interval by  $\omega_{x,a}$, so that $\delta_x^{z(x)} = \delta_{x,a}$, the nodal interval in $\cZ(w_x)$ emanating from $x$  tangentially to $\omega_{x,a}$. For each $r_2 \le r \le r_1$, $\delta_{x,a}$ crosses $C_{+}(y,r)$ at exactly one point $A_{x,a}(r)$ in $\cG_{y,a}(r)$. In $D_{+}^{c}(y,r_2)$, $\delta_{x,a}$ is ``close'' to $\delta_{y,a}$, in the sense that it is contained in $\cU^{\varepsilon}_{y,a}$.
 \item  The ray $\omega_{x,a}$ induces a natural labeling of the rays at $x$.
  \item There is one nodal interval, denoted by $\delta_x^{y(x)}=\delta_{x,b}$, from $x$ to  a point $y(x) \in \Gamma$ close to $y$. This nodal interval emanates from $x$ tangentially to a ray, denoted by $\omega_{x,b}$, for some $b \neq a$. This nodal interval is entirely contained in $D_{+}(y,r_2)$.
  \item There are $(k-2)$ loops $\gamma^{x}_{j,\tau_x(j)}$ for some map $\tau_x$ on the set of rays at $x$, minus the pair $\set{\omega_{x,a},\omega_{x,b}}$.  Each loop crosses $C_{+}(y,r)$ at exactly two points, $A_{x,j}(r)$ in $\cG_{y,j}(r)$ and $A_{x,\tau(j)}(r)$ in $\cG_{y,\tau_x(j)}(r)$.
\end{enumerate}
\end{properties}

\begin{figure}[!ht]
  \centering
  \includegraphics[width=0.9\textwidth]{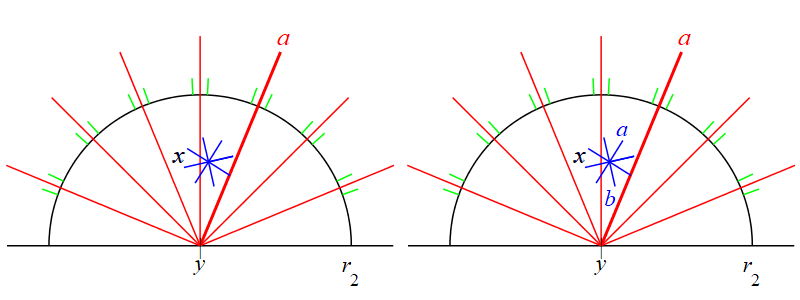}
  \caption{$k=5$: the star at $x$ and the rays $\omega_{x,a}$ and $\omega_{x,b}$}\label{F-gam0-8a}
\end{figure}

\begin{lemma}\label{L-gam0-2}
Let $y \in \Gamma_{(2k-3)}$. Under Assumptions~\ref{A-hmn3-0}, the combinatorial type of $u_y$ determines the combinatorial type of $w_x$ when $x$ is close enough to $y$.
\end{lemma}%

\begin{proof} We use Properties~\ref{P-gam0-2}.  For $x$ close enough to $y$, we follow  the nodal interval $\delta^{-1}_{x,b} \subset \cZ(w_x)$ from $y(x)$ to $x$, and then the nodal interval $\delta_{x,a} \subset \cZ(w_x)$ from $x$ to $z(x)$. The corresponding arc $\delta_{x,a}\circ \delta^{-1}_{x,b}$ from $y(x)$ to $z(x)$, through $x$, divides the simply connected domain $\wb{D}_{+}(y,r)$ into two connected components, say $\Omega_{x,\mp}$ respectively on the right/left of the arc $\delta_{x,a}\circ \delta^{-1}_{x,b}$.  Since the nodal interval $\delta_{x,b}$ is contained in $D_{+}(y,r_2)$, the domain $\Omega_{x,-}$ contains the points $A_{x,j}(r)$ for $1 \le j \le (a-1)$, and the domain $\Omega_{x,+}$ contains the points $A_{x,j}(r)$ for $(a+1) \le j \le (2k-3)$.

\begin{claim}\label{C-gam0-2C}
In the natural labeling of the rays at $x$, $b = (2k-2)$ and, for all $j \in \set{1,\ldots, (2k-3)} \sm \set{a}$, the nodal arc $\delta_{x,j}  \subset \cZ(w_x)$ intersects $C_{+}(y,r)$ at the point $A_{x,j}(r)$.
\end{claim}%

\pf If $b \neq (2k-2)$ there would exist some $j \in \set{1,\ldots, (2k-2)}\sm \set{a,b}$ such that the nodal arc $\delta_{x,j}$ intersects $ \delta_{x,a}\circ \delta^{-1}_{x,b}$ inside $D_{+}(y,r)$,  away from $x$, a contradiction with the fact that $\cS_{\mathrm{i}}(w_x) = \set{x}$. \smallskip

Assume that $\delta_{x,(a-1)}$  intersects $C_{+}(y,r)$ at the point $A_{x,j}(r)$, with $j \le (a-2)$. Then the point $A_{x,(a-1)}(r)$ would be on $\delta_{x,k}$ for some $k \le (a-2)$. The arcs $\delta_{x,(a-1)}$ and $\delta_{x,k}$ would therefore intersect which is not possible because $x$ is the only interior singular point of $w_x$, see Figure~\ref{F-gam0-7b}. We can then reason recursively with $(a-2)$, $(a-3) \ldots 1$. The proof is similar for $(a+1) \le j \le (2k-3)$. \quad \qedc \smallskip

\begin{figure}[!ht]
  \centering
  \includegraphics[width=0.7\textwidth]{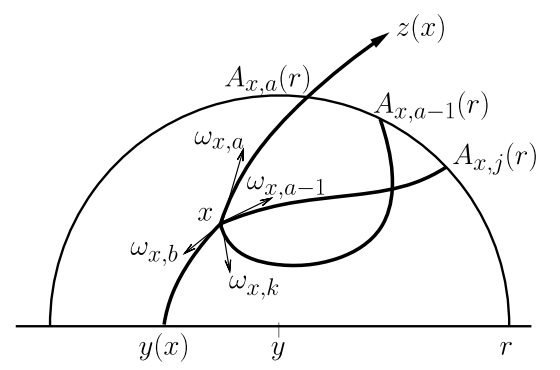}
  \caption{ Claim~\ref{C-gam0-2C}: prohibited situation}\label{F-gam0-7b}
\end{figure}

\begin{figure}[!ht]
  \centering
  \includegraphics[width=0.9\textwidth]{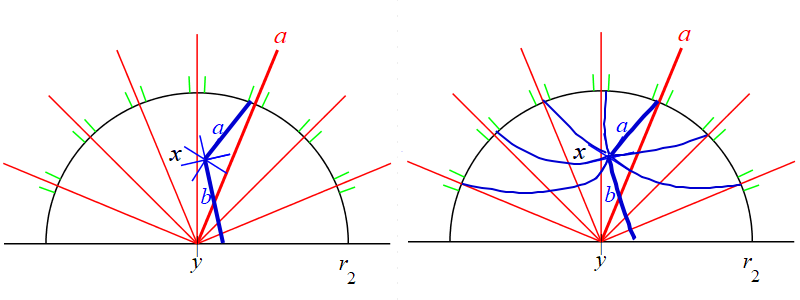}
  \caption{$k=5$: star at $x$ with $\delta_{x,b}^{-1} \circ \delta_{x,a}\,$, and $\cZ(w_x)$ near $x$}\label{F-gam0-8b}
\end{figure}

Outside $D_{+}(y,r_1)$, $\cZ(w_x) \cap D^c_{+}(y,r_1)\subset \cU^{\varepsilon}(u_y)$, and we conclude that the combinatorial type of $w_x$, in the above labeling of rays at $x$, is given by
\begin{equation}\label{E-gam0-2w}
\tau_{w_x} = \begin{pmatrix}
1 & \ldots & (a-1) & a & (a+1) & \ldots & (2k-3) & b  \\
\tau(1) & \ldots & \tau(a-1) & \downarrow_{z(x)} & \tau(a+1) & \ldots & \tau(2k-3)  & \downarrow_{y(x)} \\
             \end{pmatrix},
\end{equation}
where $b = (2k-2)$,
and
\begin{equation*}
\left\{
\begin{array}{l}
\tau_{w_x} = \tau_{u_y}=\tau \text{~in~} \set{1,\ldots,(a-1)} \cup \set{(a+1),\ldots,(2k-3)}\\[5pt]
\tau_{w_x}(a) = \, \downarrow_{z(x)} \text{~and~}
\tau_{w_x}(b) = \, \downarrow_{y(x)}.
\end{array}%
\right.
\end{equation*}

This is illustrated in Figure~\ref{F-gam0-14}.

\begin{figure}[!ht]
  \centering
  \includegraphics[width=0.9\textwidth]{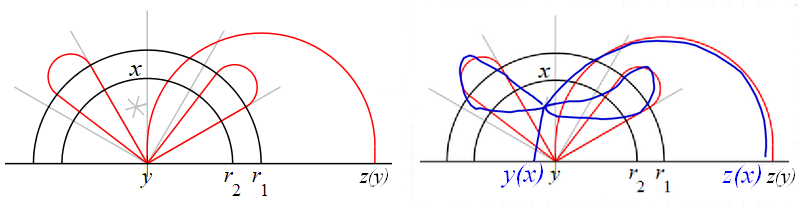}
  \caption{$k=4, a=3$: nodal sets $\cZ(u_y)$ and $\cZ(w_x)$}\label{F-gam0-14}
\end{figure}

The proof of Lemma~\ref{L-gam0-2} is complete.
\end{proof}

\begin{remark}\label{R-gam0-1}
Note that for $x$ close to $y \in \Gamma_{(2k-3)}$, the map $x \mapsto \omega_{x,a}$ is continuous. This is because there are finitely many rays.
\end{remark}%

\begin{remark}\label{R-gam0-2}  If $y \in \cA_0 \subset \Gamma_{(2k-3)}$ where $\cA_0$ is some compact arc, the numbers $r_1, r_2, \ldots$ in the proof of Lemma~\ref{L-gam0-2} can be chosen uniformly with respect to $y \in \cA_0$. This is in particular the case when we assume $ \Gamma_{(2k-3)} = \Gamma$  and consider  $\cA_0 =  \Gamma$.
\end{remark}%

\subsection{ $\cZ(w_x)$ for $x$ close to $y \in \Gamma_{(2k-3)}$, global picture when  $\Gamma_{(2k-2)}=\emptyset$}
\label{SS-gam02}

We now describe the global picture for $\cZ(u_y)$ and $\cZ(w_x)$ when $x$ is close to $y$, under the additional assumption that $\Gamma_{(2k-2)} = \emptyset$. Let $\gamma : [0,L] \to \Gamma$ be a parametrization by arc length. Let $\gamma(s,t) = \gamma(t) + s \nu(t)$ for $s$ small enough,  with $\nu$ the unit normal pointing inwards.

\begin{lemma}\label{L-gam0-4}
Under  Assumptions~\ref{A-hmn3-0} and  the further assumption that $\Gamma_{(2k-2)}$ is empty, the following properties hold.
\begin{enumerate}[(i)]
  \item The infimum $\delta := \inf\set{d(y,z(y)) \mid y \in \Gamma}$ is positive. ($d$ is the distance on $\Gamma$.)
  \item For all $\beta \le \frac{\delta}{4}$, there exists some positive $\varepsilon$ such that for all $s$, $0 < s \le \varepsilon$, for all $t \in [0,L]$,
\[
\cS_{\mathrm{b}}(w_{\gamma(s,t)}) \cap \cA(\gamma(t);\beta) \neq \emptyset \text{~and~} \cS_{\mathrm{b}}(w_{\gamma(s,t)}) \cap \cA(z(t);\beta) \neq \emptyset,
\]
where $z(t) := z(\gamma(t))$ is such that $\cS_{\mathrm{b}}(u_{\gamma(t)} )= \set{\gamma(t),z(t)}$.
\end{enumerate}
\end{lemma}%

\begin{proof} \smallskip

\emph{Assertion~(i).~} Assume that $\delta=0$. Then there exists a sequence $\set{y_n} \subset \Gamma$ such that $\delta(y_n,z(y_n)) \le \frac 1n$. We may assume that $y_n$ tends to some $y \in \Gamma$.  Then, $z(y_n)$ tends to $y$ as well. Choose $u_y, u_{y_n} \in \bS(U)$ given by Lemma~\ref{L-L32}. We may assume that $u_{y_n}$ tends to $u_y$ $C^1$-uniformly, implying that $\breve{u}_{y_n}$ tends to $\breve{u}$ uniformly. Since $\Gamma_{(2k-3)} = \Gamma$, $\cS_{\mathrm{b}}(u) = \set{y,z(y)}$ with $z(y) \neq y$. The properties of $\breve{u}_y$ imply that $\breve{u}_{y_n}$ vanishes near $y$ and $z(y)$, a contradiction with the fact that $z(y_n)$ tends to $y \neq z(y)$. \phantom{xxxx} \quad \qedc \medskip

\emph{Assertion~(ii).~} Assume that the assertion is false. Then, there exists some $\beta \le \frac{\delta}{4}$, a sequence $\set{s_n}$ tending to $0$, a sequence $\set{t_n}$ tending to some $\bar{t}$, such that the sequence $\set{w_n:=w_{\gamma(s_n,t_n)}}$ tends to $u_{\gamma(\bar{t})}$,   with the property that $\cS_{\mathrm{b}}(w_n) \cap \cA(\gamma(t_n);\beta) = \emptyset$ or
$\cS_{\mathrm{b}}(w_n) \cap \cA(z(t_n);\beta) = \emptyset$. Since $\bar{t} \in \Gamma_{(2k-3)}$, $ \cS_{\mathrm{b}}(u_{\gamma(\bar{t})}) = \set{\gamma(\bar{t}),z(\bar{t})}$.  We may also assume that $\mp \breve{u}_{\bar{t}}\big(\gamma(\bar{t})\pm \frac{\beta}{4}\big) > 0$ and $\pm \breve{u}_{\bar{t}}\big(z(\bar{t})\pm \frac{\beta}{4}\big) > 0$. For $n$ large enough, we will have that $|\gamma(t_n) - \gamma(\bar{t})| < \frac{\beta}{4}$, $|z(t_n) - z(\bar{t})| < \frac{\beta}{4}$, $\mp \breve{w}_{n}\big(\gamma(\bar{t})\pm \frac{\beta}{4}\big) > 0$, and $\pm \breve{w}_{n}\big(z(\bar{t})\pm \frac{\beta}{4}\big) > 0$. This implies that $\breve{w}_n$ vanishes in both $\big( \gamma(t_n) - \frac{\beta}{2},\gamma(t_n) + \frac{\beta}{2} \big)$ and $\big( z(t_n) - \frac{\beta}{2},z(t_n) + \frac{\beta}{2} \big)$, a contradiction with our assumption. \quad \qedc \smallskip

The lemma is proved.
\end{proof}

\begin{lemma}\label{L-gam0-5}
 Under  Assumptions~\ref{A-hmn3-0},  the set $\Gamma_{(2k-2)}$ is not empty.
\end{lemma}%

\begin{proof} Assuming that $\Gamma_{(2k-2)}=\emptyset$, the arguments in the proof of Lemma~\ref{L-gam0-2} can be globalized  because the radii $r_1$ and $r_2$ can be chosen uniformly with respect to $y \in \Gamma = \Gamma_{(2k-3)}$. There exists $\varepsilon > 0$ such that, for all $(s,t) \in (0,\varepsilon] \times [0,2\pi]$, the combinatorial type of $w_{\gamma(s,t)}$ is determined by the combinatorial type of $u_{\gamma(t)}$ which is constant on $\Gamma$ (see Lemma~\ref{L-L33c}).  Taking $\varepsilon$ small enough in Lemma~\ref{L-gam0-4}, we can apply Lemma~\ref{L-gam0-2} to the pair $y = \gamma(t)$ and $x(s,t) = \gamma(s,t)$ for all $(s,t) \in (0,\varepsilon] \times  [0,2\pi]$. Fix some $s_0, s_1, 0 < s_1 < s_0 \le \varepsilon$.  According to Lemma~\ref{L-gam0-2}, for each $t \in [0,2\pi]$  the star at $\gamma(s_0,t)$  inherits a natural labeling from the labeling of the star at $\gamma(t)$, with the same index $a$ corresponding to the nodal interval emanating from $\gamma(s_0,t)$ and hitting $\Gamma$ at $z(\gamma(s_0,t))$ close to $z(\gamma(t))$. Since the curve $t \mapsto \gamma(s_0,t)$ bounds a simply connected domain $\Omega_{s_0}$, using the continuity of $x \mapsto \omega_{x,a}$, we can extend this labeling from the curve $\gamma(s_0,\cdot)$ continuously into $\Omega_{s_0}$. \smallskip

Fix $s_1, 0 < s_1 < s_0$. Along $s \mapsto \gamma(s,t)$, we can deform the labeled  star at $\gamma(s_0,t)$ continuously into the labeled ``star'' $\set{\omega_{y,1}, \ldots, \omega_{y,(2k-3)},\omega_{y,\nu}}$ at  $\gamma(s_1,t)$, so that $\omega_{\gamma(s_0,t),(2k-2)}$ deforms to $\omega_{y,\nu}$. Here, $\omega_{y,\nu}$ is the direction of the normal to $\Gamma$ at $y$, pointing outwards (this can be visualized on Figure~\ref{F-gam0-8b}, right image).\smallskip

We have constructed a continuous nonzero vector field in $\Omega_{s_1}$ which is transverse to the boundary $\Gamma_{s_1}$, pointing outwards.  This is impossible by the Poincar\'{e}-Hopf theorem for manifold with boundary,  see  \cite{Miln1997}, Chap. 6, p. 35. The assumption that $\Gamma_{(2k-2)}$ is empty yields a contradiction, therefore, $\Gamma_{(2k-2)}$ cannot be empty.
\end{proof}

For later purposes (see Remark~\ref{R-hmn93-PH}), we introduce the following generalization of the Poincar\'{e}-Hopf theorem \index{Poincar\'e-Hopf theorem}, see \cite{Gott1990, GoSa1995} and the recent paper \cite{BaPP2024}.

\begin{theorem}\label{ThPHbd}
Let $X$ be a compact manifold with boundary. Let $V$ be a $\Cty$ vector-field on $X$ with isolated zeros  $x_i$  in $X$ and no zero on $\partial X$.  Then,
\begin{equation}\label{eq:ThPHbd}
\sum_i {\textrm ind}_{x_i} V = \chi (X) - \sum_{\xi \in \partial X, V^{tg} (\xi)=0,\ip{V}{\nu} < 0} {\textrm ind}_\xi V^{tg},
\end{equation}
where $V^{tg}= V - \ip{V}{\nu}\, \nu$ is the tangential component of $V$ at the boundary and $\nu$ is the outward pointing normal.
\end{theorem}

When $V$ points outwards,  the sum on the right hand side is empty.

\section[Behavior of $\lambda_k$-Eigenfunctions near $\Gamma$]{Behavior of $\lambda_k$-Eigenfunctions near $\Gamma$ under  Assumptions~5.2}\label{S-hmn9}

In this section, we work under Assumptions~\ref{A-hmn3-0}.\smallskip

Let $\gamma : [0,L] \to \Gamma$ be an arc-length parametrization of $\Gamma$, compatible with the orientation, and such that the unit normal vector $\nu(t)$ points inwards. Assume that $\gamma(0) = \gamma(L) \not \in \Gamma_{(2k-2)}$. Let $\gamma(s,t) = \gamma(t) + s\, \nu(t)$ where $0 < s < s_0$, with $s_0$ small enough so that the map $(0,s_0) \times [0,L] \to \Omega$ is a diffeomorphism onto a neighborhood of $\Gamma$ in $\Omega$. We also use the notation $\gamma_s(t)$ for $\gamma(s,t)$.

\subsection{Analysis near $\eta \in \Gamma_{(2k-2)}$}\label{SS-hmn91}~\\
 Under Assumptions~\ref{A-hmn3-0}, according to Lemmas~\ref{L-L33} and \ref{L-L33c}, the subset $\Gamma_{(2k-2)}$ is finite,  with an even number of points.  According to Section~\ref{S-gam0}, this set  has at least two points.
Fix a radius $r_1$ such that for all $\eta \in \Gamma_{(2k-2)}$ the local structure theorem (see Section~\ref{S-lsbs}) and the energy argument (Lemma~\ref{L-hmn3-L38-e}) apply to the function $u_{\eta}$ in the disk $D_{+}(\eta,2r_1)$. This is possible because the set $\Gamma_{(2k-2)}$ is finite. Fix $\beta = {r_1}/{10}$.

\begin{lemma}\label{L-hmn91-2}
There exists $r_2$, $ 0 < 2 r_2 < \beta$, such that for all $\eta \in \Gamma_{(2k-2)}$, and for all $x \in D_{+}(\eta,2r_2)$, $\cS_{\mathrm{b}}(w_x) \subset \cA(\eta;\beta)$, including the possibility that $\cS_{\mathrm{b}}(w_x) = \emptyset$.
\end{lemma}

\begin{proof} Assume that this is not the case. Then, there exists a sequence $\set{x_n}$ tending to some $\eta \in \Gamma_{(2k-2)}$ such that $\cS_{\mathrm{b}}(w_n) \not \subset \cA(\eta;\beta)$, with $w_n := w_{x_n} \in \bS(U) \cap W_x$, i.e., there exists $z_n \in \cS_{\mathrm{b}}(w_n)$, $z_n \not \in \cA(\eta;\beta)$. We may assume that the sequence $\{z_n\}$ converges to some $z \neq \eta$. Since $w_n$ tends to $u_{\eta}$ $C^1$-uniformly and since $\breve{w}_n(z_n) = 0$, we have $\breve{u}_{\eta}(z)=0$, a contradiction since $\breve{u}_{\eta}$ vanishes only at $\eta$.
\end{proof}

\noid A general combinatorial type $\tau_{\eta}$ for $u_{\eta}$, $\eta \in \Gamma_{(2k-2)}$ is given by
\begin{equation}\label{E-hmn9-}
\tau_{\eta} =
\begin{pmatrix}
  1 & R & a & C & b & L & f \\
  a & \tau_{\eta}(R) & 1 & \tau_{\eta}(C) & f & \tau_{\eta}(L) & b \\
\end{pmatrix}
\end{equation}
where $f := (2k-2)$, $g:=(2k-3)$, $R := \set{2,\ldots, (a-1)}$, $C:=\set{(a+1),\ldots,(b-1)}$, and $L = \set{(b+1),\ldots,g}$, with $R, C$ and $L$ globally invariant under $\tau_{\eta}$.  Here, we only consider the case $2 \le a < b \le (2k-3)$. The case in which $\tau_{\eta}(1) = (2k-2)$ can be treated similarly, see Section~\ref{S-hmn2L}.

\newcommand{\twa}{0.45}
\begin{figure}[!ht]
\centering
\begin{subfigure}[t]{\twa\textwidth}
\centering
\raisebox{-6pt}{\includegraphics[width=\linewidth]{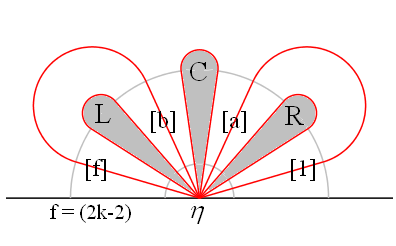}}
\end{subfigure}
\begin{subfigure}[t]{\twa\textwidth}
\centering
\includegraphics[width=\linewidth]{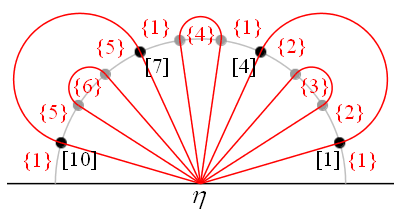}
\end{subfigure}
\caption{Rays and nodal domains for $u_{\eta}$}\label{F-hmn91-k6-ueta-labels}
\end{figure}

The nodal set $\cZ(u_{\eta})$ is a $(k-1)$-bouquet of loops $\gamma^{\eta}_{j,\tau_{\eta}(j)}$. Call $A_{\eta,j}(r_1)$ the intersection points of $\cZ(u_{\eta})$ with $C_{+}(\eta,r_1)$,  labeled counter-clockwise along $C_{+}(\eta,r_1)$. By the local structure theorem and the energy argument, there are precisely $(2k-2)$ such points, and there exists some $\alpha_0 > 0$ such that the $(2k-2)$ sub-arcs $\cA \big( A_{\eta,j}(r_1) ; \alpha_0 \big)$ on $C_{+}(\eta,r_1)$ are pairwise disjoint. Let $D_{+}^c(\eta,r_1) = \Omega \sm D_{+}(\eta,r_1)$ denote the complement of $D_{+}(\eta,r_1)$. The set $\cZ(u_{\eta}) \cap D_{+}^c(\eta,r_1)$ consists of $(k-1)$ pairwise disjoint nodal arcs $\gamma^{\eta}_{j,\tau_{\eta}(j)}\cap D_{+}^c(\eta,r_1)$. By compactness, these arcs have pairwise disjoint $\varepsilon_0$-tubular neighborhoods $\cU^{\varepsilon}_{j,r_1,\varepsilon_0}$, for some $\varepsilon_0 < \alpha_0$ small enough. Fix the values  $\alpha_0$ and $\varepsilon_0$ for the rest of  this subsection.\smallskip

We label the nodal domains of $u_{\eta}$ according to their order of appearance when moving along $C_{+}(\eta,r_1)$ counter-clockwise, encountering the point $A_{\eta,1}(r_1)$, $A_{\eta,2}(r_1)$, etc. In the right part of Figure~\ref{F-hmn91-k6-ueta-labels}, $k=6$, $a=4$ and $b=7$. The numbers between brackets are the labels of the rays at $\eta$. The numbers between braces are the labels of the nodal domains. The big dots stand for the intervals around the points $A_{\eta,j}(r_1)$  (in grey, the dots corresponding to the bouquets of loops associated with the subsets $R, C$ and $L$) which might occur in the general case.\medskip

\noid To study $\cZ(w_x)$, we now choose $r_2$ such that:
\begin{enumerate}[(i)]
  \item $r_2$ satisfies Lemma~\ref{L-hmn91-2}
  \item for all $\eta \in \Gamma_{(2k-2)}$, for all $x \in D_+(\eta, 2r_2)$, for all $j, 1 \le j \le (2k-2)$, the set $\cZ(w_x) \cap \cA\big( A_{\eta,j}(r_1),\alpha_0 \big)$ contains exactly one point $A_{x,j}(r_1)$
  \item for all $\eta \in \Gamma_{(2k-2)}$, for all $x \in D_+(\eta, 2r_2), \cZ(w_x) \cap D_{+}^{c}(\eta,r_1) \subset \bigcup \cU^{\varepsilon_0}_{\eta,j}$.
\end{enumerate}

For $x \in D_{+}(\eta,r_2)$, there are two possibilities:
\begin{enumerate}[a)]
\item either $\cS_{\mathrm{b}}(w_x) = \emptyset$ and $\cZ(w_x)$ is a $(k-1)$-bouquet of loops
\item or $\cS_{\mathrm{b}}(w_x) \neq \emptyset$ and $\cZ(w_x)$ is the wedge sum at $x$ of a $(k-2)$-bouquet of loops with two nodal intervals from $x$ to the boundary points in $\cS_{\mathrm{b}}(w_x) \subset \cA(\eta;\beta)$.
\end{enumerate}%

The choice of $r_1$ and the energy argument imply that any nodal loop in $\cZ(w_x)$ intersects $C_{+}(\eta,r_1)$ at precisely two points located in different intervals $\cA\big( A_{\eta,j}(r_1);\alpha_0 \big)$. Furthermore, when $\cS_{\mathrm{b}}(w_x) \neq \emptyset$, the nodal intervals from $x$ to the boundary cannot both be contained in $D_{+}(\eta,r_1)$. Since $\cS_{\mathrm{b}}(w_x) \subset \cA(\eta;\beta)$, one of the nodal intervals has to exit $D_{+}(\eta,r_1)$ and re-enter. Call $\delta_{x}^{z(x)}$ this nodal interval and $z(x) \in \cS_{\mathrm{b}}(w_x)$ its end point. The interval $\delta_x^{z(x)}$ actually intersects $C_{+}(\eta,r_1)$ at precisely two points. Counting the points in $\cZ(w_x) \cap C_{+}(\eta,r_1)$, we infer that the other nodal interval does not exit $D_{+}(\eta,r_1)$. Call $\delta_{x}^{y(x)}$ this nodal interval and $y(x) \in \cS_{\mathrm{b}}(w_x)$ its end point. Note that it may happen that $y(x) = z(x)$.\smallskip
The points $A_{\eta,j}(r_1), 1 \le j \le (2k-2)$, are in natural bijection with the rays $\omega_{\eta,j}$ at $\eta$. We will now show that, for $x \in D_{+}(\eta,r_2)$, the points $A_{x,j}(r_1) \in \cA\big( A_{\eta,j},\alpha_0 \big)$ define a labeling of the rays of the star at $x$.\smallskip \\

\emph{Case a)~~$\cS_{\mathrm{b}}(w_x) = \emptyset$.}~ Call $\omega_{x,j}$ the unique ray at $x$ such that the nodal arc $\delta_{x,\omega_{x,j}}$ emanating from $x$ tangentially to $\omega_{x,j}$ exits $C_{+}(\eta,r_1)$ at $A_{x,j}(r_1)$. Because these nodal arcs are pairwise disjoint away from $x$, Jordan's theorem implies that the $\omega_{x,j}$ are ordered counter-clockwise as the points $A_{x,j}(r_1)$, i.e. $\cR_{\frac{\pi}{k-1}}(\omega_{x,j}) = \omega_{x,j+1}$, where $\cR_{\frac{\pi}{k-1}}$ is the rotation with center $x$ and angle $\frac{\pi}{k-1}$. Looking at $\cZ(w_x) \cap D_{+}^{c}(\eta,r_1)$, we infer that the combinatorial types satisfy $\tau_{w_x} = \tau_{\eta}$ with the above labeling of the star at $x$.

\begin{figure}[!htb]
  \centering
  \includegraphics[width=0.80\textwidth]{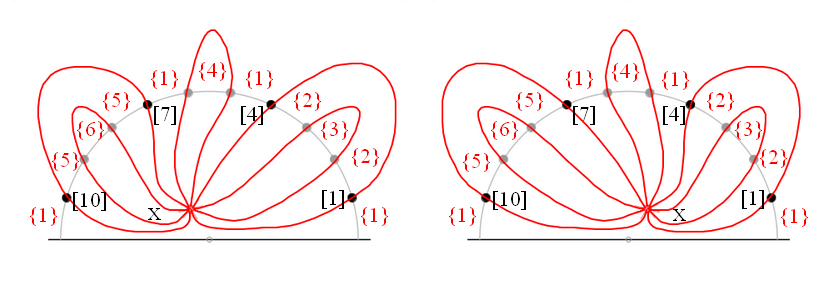}
  \caption{$k=6$, $\cZ(w_x)$ with no boundary singular point}\label{F-hmn91-k6-nobdry-LR}
\end{figure}

\emph{Case b)~~$\cS_{\mathrm{b}}(w_x) \neq \emptyset$.}~ The $(2k-2)$ points $A_{x,j}(r_1)$ can be partitioned into two subsets, the subset $\cL(x)$ which consists of the $(2k-4)$ points which belong to a nodal loop in $\cZ(w_x)$ and the subset $\cL'(x)$ which consists of the two points in the set $\delta_{x}^{z(x)} \cap C_{+}(\eta,r_1)$; call $A_{x,e}(r_1)$ the point at which $\delta_{x}^{z(x)}$ exits  $D_{+}(\eta,r_1)$ and $A_{x,e_z}(r_1)$ the point at which $\delta_{x}^{z(x)}$ re-enters $D_{+}(\eta,r_1)$.

\begin{claim}\label{C-ss-hmn91-2}
With the above notation, $e \neq e_z$ and $e_z \in \set{1,(2k-2)}$.
\end{claim}%

\begin{proof} The first assertion is obvious.\smallskip

Assume that $e = 1$ and that $e_z \neq (2k-2)$. Follow the nodal arc from $A_{x,(2k-2)}(r_1)$ to $x$ and then to $A_{x,1}(r_1)$. In $D_{+}(\eta,r_1)$, the point $A_{x,e_z}(r_1)$ lies above this arc and the point $z(x)$ below this arc, a contradiction  since nodal arcs cannot intersect away from $x$.  This is illustrated in Figure~\ref{F-hmn91-proofs-claims}, left picture. The proof in the other cases, $e = (2k-2)$ and $e, e_z \not \in \set{1,(2k-2)}$, is similar.
\end{proof}

\begin{figure}[!ht]
\centering
\begin{subfigure}[b]{.45\textwidth}
\centering
\includegraphics[width=\linewidth]{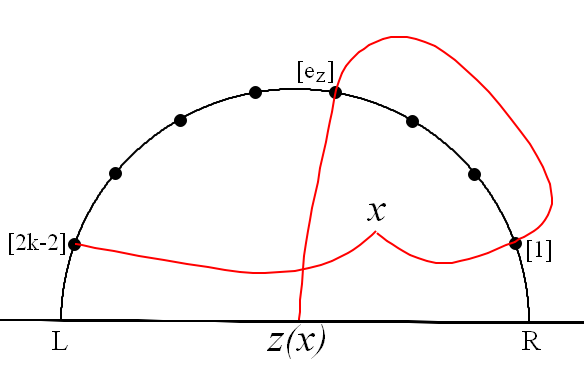}
\caption{Proof of Claim~\ref{C-ss-hmn91-2}}
\end{subfigure}
\begin{subfigure}[b]{.45\textwidth}
\centering
\includegraphics[width=\linewidth]{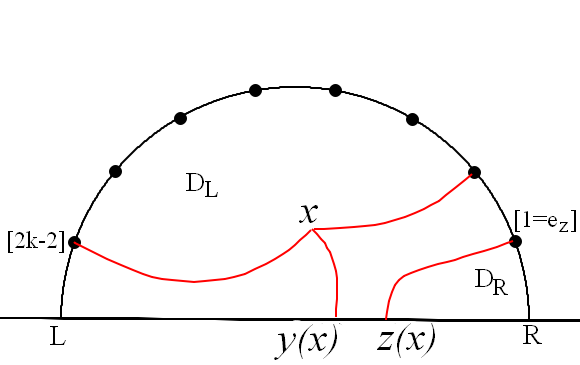}
\caption{Proof of Claim~\ref{C-ss-hmn91-4}}
\end{subfigure}
\caption{Proofs of Claims}\label{F-hmn91-proofs-claims}
\end{figure}

Look at the rays at $x$. Call $\omega_{x,j}$ the ray such that the nodal arc $\delta_{x,\omega_{x,j}}$, emanating from $x$ tangentially to the ray $\omega_{x,j}$,  first exits $C_{+}(\eta,r_1)$ at the point $A_{x,j}(r_1)$. Doing so, we label the $(2k-4)$ rays corresponding to the $(k-2)$ loops in $\cZ(w_x)$, as well as the ray corresponding to the nodal interval $\delta_{x}^{z(x)}$. One ray has not yet been labeled, namely the ray which corresponds to the nodal interval $\delta_{x}^{y(x)}$. This must be the  ray $\omega_{x,e_z}$, so that $\delta_{x}^{y(x)} = \delta_{x,\omega_{x,1}}$ or $\delta_{x}^{y(x)} = \delta_{x,\omega_{x,(2k-2)}}$ according to Claim~\ref{C-ss-hmn91-2}.

\begin{claim}\label{C-ss-hmn91-4}
 In Case b), with the above notation,
\begin{equation}\label{E-hmn91-a}
\left\{
\begin{array}{lll}
e_z = 1 \Rightarrow y(x) \le  z(x) & \text{and} & y(x) <  z(x) \Rightarrow e_z = 1 \\[5pt]
e_z = (2k-2) \Rightarrow y(x) \ge z(x)  & \text{and}  & y(x) >  z(x) \Rightarrow e_z = (2k-2).
\end{array}%
\right.
\end{equation}
\end{claim}%

\begin{proof}   Assume that $e_z = 1$. Consider the nodal arc inside $D_{+}(\eta,r_1)$, between $A_{x,1}(r_1)$ and $z(x)$. This arc divides $D_{+}(\eta,r_1)$ into two connected components, $D_R$ and $D_{L}$, with $D_R$ containing the arc of $C_{+}(\eta,r_1)$ from $A_{x,1}(r_1)$ to the boundary point $R$ on the right of $z(x)$ and $D_{L}$ containing the arc of $C_{+}(\eta,r_1)$ from $A_{x,(2k-2)}(r_1)$ to the boundary point $L$ on the left of $z(x)$. Because the nodal arcs cannot intersect away from $x$, the point $x$ must belong to $D_{L}$ and $y(x) \le z(x)$. Similarly, if $e_z = (2k-2)$, then $z(x) \le y(x)$. These statements imply the remaining statements.  The proof is illustrated in Figure~\ref{F-hmn91-proofs-claims}, right picture.
\end{proof}

\begin{remark}
Let $x \in \Omega$ be such that  $y(x)=z(x)$. Then, there exists a neighborhood $\cU_x$ of $x$ such that there do not exist $x_1, x_2 \in \cU_x$ with $y(x_1) < z(x_1)$ and $y(x_2) > z(x_2)$. Otherwise stated, for any $x_1 \in \cU_x$, with $\cS_{\mathrm{b}}(w_{x_1}) \neq \emptyset$, either $y(x_1) \le z(x_1)$ or $y(x_1) \ge z(x_1)$. This is a consequence of Claim~\ref{C-ss-hmn91-4}.
\end{remark}%

The possible combinatorial types of $w_x$, depending on whether $y(x) < z(x)$ or $z(x) < y(x)$  are as follows,
\begin{equation}\label{E-hmn91-2L}
\tau_{y(x) < z(x)} =
\begin{pmatrix}
  1 & R & a & C & b & L & f \\
  \downarrow_{y(x)} & \tau_{\eta}(R)& \downarrow_{z(x)} & \tau_{\eta}(C)
  & f & \tau_{\eta}(L) & b \\
\end{pmatrix}
\end{equation}
\begin{equation}\label{E-hmn91-2R}
\tau_{z(x) < y(x)} =
\begin{pmatrix}
  1 & R & a & C & b & L & f \\
  a & \tau_{\eta}(R)& 1 & \tau_{\eta}(C) & \downarrow_{z(x)} & \tau_{\eta}(L) & \downarrow_{y(x)} \\
\end{pmatrix}.
\end{equation}

\begin{figure}[!th]
  \centering
  \includegraphics[width=0.9\textwidth]{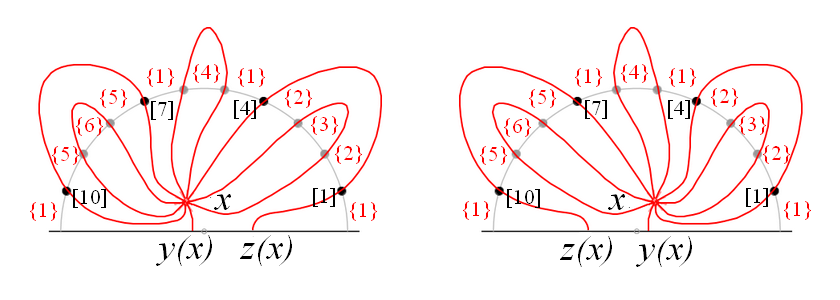}
  \caption{$k=6$, $\cZ(w_x)$ with two boundary singular points}\label{F-hmn91-k6-bdry-LR}
\end{figure}

Joining the extremities $y(x)$ and $z(x)$ of the nodal intervals $\delta_{x}^{y(x)}$ and $\delta_{x}^{z(x)}$ with the interval $[y(x),z(x)]$, we obtain a loop at $x$. The combinatorial types $\tau_{z(x) < y(x)}$ and $\tau_{y(x) < z(x)}$ then correspond to $\tau_{\eta}$.  Labeling the nodal domains of $u_{\eta}$ as usual, and following the nodal domains of $w_x$ by deformation, the words describing the nodal domains on $C_{+}(\eta,r_1)$ are the same.\medskip

Note that both cases $y(x) < z(x)$ and $y(x) > z(x)$ actually occur. Indeed, choose some $y \in \Gamma_{(2k-3)}$ close to $\eta$, on its left. Then $\cS_{\mathrm{b}}(u_y) = \set{y, z(y)}$ with $z(y)$ on the right of $\eta$. Choosing $x$ above $y$ and close enough, we will have $\cS_{\mathrm{b}}(w_x) = \set{y(x),z(x)}$ with $y(x) < z(x)$. If we choose $y$ to the right of $\eta$, we will similarly obtain $z(x) < y(x)$. \smallskip

Figure~\ref{F-hmn91-k6-bdry-LR} illustrates the two possible cases.  In this figure, the numbers in brackets represent the labels of the points $A_{x,j}(r_1)$. Label the rays at $x$ so that the nodal arc emanating from $x$ tangentially to $\omega_{x,a}$ exits $C_{+}(\eta,r_1)$ at $A_{x,a}(r_1)$ and hits the boundary at $z(x)$. Then, when $y(x) < z(x)$, the ray tangent to $\delta_{x}^{y(x)}$ at $x$ is $\omega_{x,1}$; when $z(x) < y(x)$, the ray tangent to $\delta_{x}^{y(x)}$ at $x$ is $\omega_{x,(2k-2)}$. \medskip

Figure~\ref{F-hmn91-k6-transition-labels} displays a simple transition. When $x$ moves on a line parallel to the boundary $\Gamma$, the rays corresponding to $\delta_x^{z(x)}$ are labeled either $[4]$, when $y(x) < z(x)$, or $[7]$, when $z(x) < y(x)$.

\begin{figure}[!ht]
  \centering
  \includegraphics[width=0.98\textwidth]{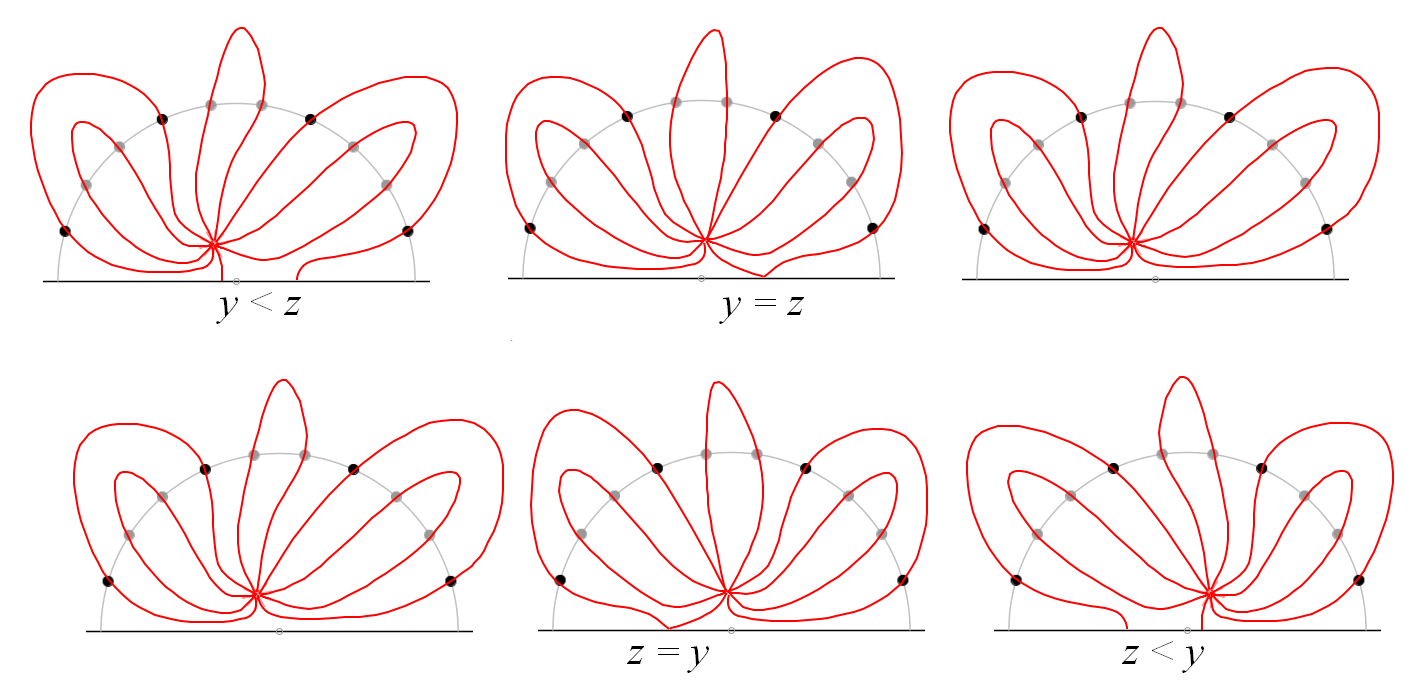}
  \caption{Simple transition in $D_{+}(\eta,r_2)$}\label{F-hmn91-k6-transition-labels}
\end{figure}

Move from left to right on a parallel to $\Gamma$. Start from  $x_1$ above $y_1$ on the left of $\eta$ with $y(x_1) < z(x_1)$. Arrive at $x_2$ such that $y(x_2)=z(x_2)$. When $x$ moves from $x_2$ to the right, the broken loop $\gamma^{w_{x_2}}_{1,4}$ lifts off, and $\cZ(w_x)$ does not hit $\Gamma$ (middle picture). When  $x$ reaches some $x_3$ such that $y(x_3)=z(x_3)$, a loop in $\cZ(w_x)$ touches down as the broken loop $\gamma^{w_{x_3}}_{7,10}$. When $x$ moves on from $x_3$ to $x_4$ on the right, we have $z(x) < y(x)$. During this transition, we can glue the interval between $y(x)$ and $z(x)$ on the boundary with the nodal intervals $\delta_{x}^{y(x)}$ and $\delta_{x}^{z(x)}$ in order to make a loop. Following the nodal domains continuously, their labeling does not change and the combinatorial type of $w_x$ does not change either. There has been some shift from the ray $\omega_{x_1,4}$ to the ray $\omega_{x_4,7}$.\smallskip

\begin{remark}\label{R-hmn91-4}
The simple transition displayed in Figure~\ref{F-hmn91-k6-transition-labels} might not be what happens in general. When $x$ moves on $\gamma_{s_0}$, from above $\eta_1$ to above $\eta_2$, there might be several points $x$ for which $y(x)=z(x)$. Although we do not need this information in the reasoning below, it would be interesting to investigate what actually occurs.
\end{remark}%

We can continuously deform the nodal arcs inside $D_{+}(\eta,r_1)$ into the rays from $x$ to the points $A_{\eta,j}(r_1)$. For $r_2$ small enough and for $x \in D_{+}(\eta,r_2)$, the star at $x$ can be continuously deformed to the constant set $\set{A_{\eta,1}(r_1), \ldots, A_{\eta,(2k-2)}(r_1)}$.\medskip

\FloatBarrier

\subsection{Analysis inside the arc $\cA(\eta_1,\eta_2)$}\label{SS-hmn92}

 Let $\eta_1, \eta_2$ be two points in $\Gamma_{(2k-2)}$, such that $\cA(\eta_1,\eta_2) \subset \Gamma_{(2k-3)}$. Call $t'_1, t'_2$ the parameters  such that $\gamma(t'_i) = \eta_i$. \smallskip

For the sake of simplicity, we assume that the combinatorial type of a generator $u_y$ of $U_y$ for $y \in \cA(\eta_1,\eta_2)$ is given by Figure~\ref{F-hmn92-2}. More precisely, we assume that $1 < a < (2k-3)$,
\begin{equation*}
R = \set{2,\ldots,(a-1)} \text{~~and~~} L = \set{(a+1),\ldots,(2k-3)},
\end{equation*}
and that, for $y \in \cA(\eta_1,\eta_2)$, the combinatorial type $\tau$ of $u_y$  is given by
\begin{equation}\label{E-hmn92-2}
\tau_{+} = \begin{pmatrix}
             \downarrow &  R & a & L\\
             a & \tau(R) & \downarrow & \tau(L) \\
           \end{pmatrix}
\text{~~or~~}
\tau_{-} = \begin{pmatrix}
             R & a & L & \downarrow\\
             \tau(R) & \downarrow & \tau(L) & a\\
           \end{pmatrix}.
\end{equation}
The cases in which $a \in \set{1,(2k-3)}$ can be dealt with similarly.\medskip

\begin{figure}[!ht]
  \centering
  \includegraphics[width=0.9\textwidth]{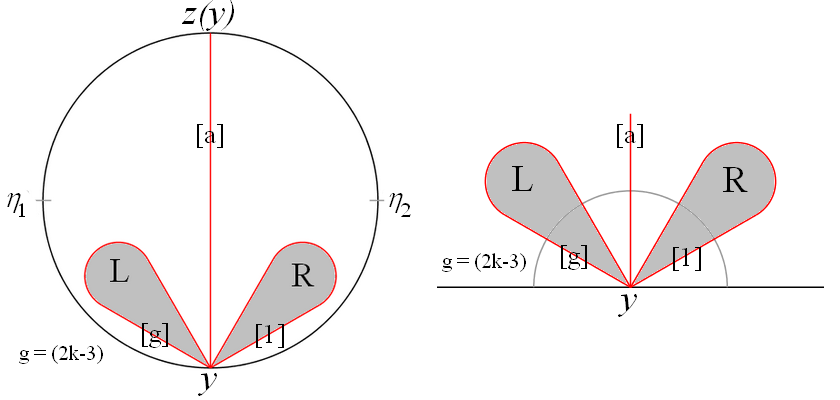}
  \caption{The global and local pictures for $\cZ(u_y)$}\label{F-hmn92-2}
\end{figure}

The nodal interval $\delta:=\delta_{y,a} = \delta_{y}^{z(y)}$ from $y$ to $z(y)$ separates the domain $\Omega$ into two connected components $\Omega_{a,R}$ and $\Omega_{a,L}$. To label nodal domains, we use Procedure~\ref{P-hmn2L-la}.\smallskip

The domain $\Omega_{a,R}$ contains the bouquet of loops $\cB_R$ and $(n_R+1)$ nodal domains of $u_y$, where $n_R = (a-1)/2$ is the number of loops in $\cB_R$. We call $D_1$ the nodal domain exterior to $\cB_R$ in $\Omega_{a,R}$, its boundary contains the nodal interval $\delta$. The interior nodal domains of $\cB_R$ are labeled $D_2$ to $D_{n_R+1}$. The word which describes the nodal domains of $u_y$ inside $\Omega_{a,R}$ is
\begin{equation}\label{E-hmn92-4R}
\ww_{a,R}= |1|\ww_R|1| \text{~~with~~} \|\ww_{a,R}\|=2 + \|\ww_R\| = (2n_R+1).
\end{equation}
The ``letters'' of the words are labels of the nodal domains, and separated by vertical bars, as in Paragraph~\ref{SSS-hmn-25D}.\smallskip

Here, the word $\ww_R$ describes the nodal domains related to $\cB_R$. This is a word in the letters $2, \ldots, (n_R+1)$, and possibly the letter $1$ (this occurs if $\cB_R$ contains consecutive loops as in Figure~\ref{F-hmn2L-L33c-sb}, left sub-figure).  \medskip

The domain $\Omega_{a,L}$ contains the bouquet of loops $\cB_L$ and $(n_L+1)$ nodal domains of $u_y$, where $n_L = (2k-a-3)/2$ is the number of loops in $\cB_L$. We call $D_{n_R+2}$ the nodal domain exterior to $\cB_L$ in $\Omega_{a,L}$, its boundary contains the nodal interval $\delta$. The interior nodal domains of $\cB_L$ are labeled $D_{n_R+3}$ to $D_{k}$, and we have $k = n_R+n_L+2$. Letting $m:=(n_R+2)$, the word which describes the nodal domains of $u_y$ inside $\Omega_{a,L}$ is
\begin{equation}\label{E-hmn92-4L}
\ww_{a,L}= |m|\ww_L|m| \text{~~with~~} \|\ww_{a,L}\|=2 + \|\ww_L\| = 2n_L+1.
\end{equation}
Here, $\ww_L$ describes the nodal domains related to $\cB_L$. This is a word in the letters $(n_R+3), \ldots, k$, and possibly $m=(n_R+2)$.\medskip

Finally, the word which describes how the nodal domains of $u_y$ hit $C_{+}(y,r)$ for $r$ small enough is given by
\begin{equation}\label{E-hmn92-4u}
\ww_y = |1|\ww_R|1|m|\ww_L|m|\text{~~with~~} \|\ww_y\|= (2n_R+2n_L+2) = (2k-2).
\end{equation}

The important fact is that the nodal domains $D_1$ and $D_{m}$ (with $m = (n_R+2)$) share a common boundary line, the nodal interval $\delta$.\medskip

In the following figures, the letters or numbers between brackets are the labels of the rays; the numbers between braces are the labels of the nodal domains.  The labeling of the nodal domains in the central sub-figure of Figure~\ref{F-hmn92-4} follows Procedure~\ref{P-hmn2L-la}.

\begin{figure}[!ht]
  \centering
  \includegraphics[width=0.95\textwidth]{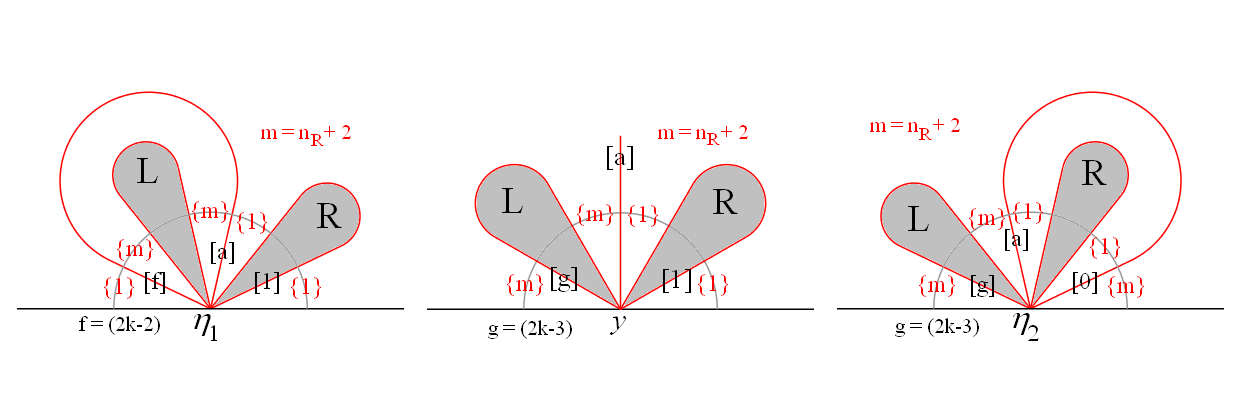}
  \caption{Limits when $y$ tends to $\eta_2$ from the left or to $\eta_1$ from the right}\label{F-hmn92-4}
\end{figure}

According to Lemma~\ref{L-L33b}, when $y \in \cA(\eta_1,\eta_2)$ moves monotonically counter-clock\-wise from $\eta_1$ to $\eta_2$, the point $z(y)$ moves monotonically clockwise from $\eta_1$ to $\eta_2$ in $\Gamma\sm \cA(\eta_1,\eta_2)$. Then, the nodal interval $\delta$ pushes the nodal domain $D_1$ which deforms into a nodal domain, still denoted by  $D_1$, contained in the interior of the loop $\gamma^{\eta_2}_{0,a}$ in $\cZ(u_{\eta_2})$. The nodal domains in  Figure~\ref{F-hmn92-4}, right sub-figure, are labeled by continuity from the labeling in the central sub-figure. For the function $u_{\eta_2}$ we obtain the following nodal type and word describing the nodal domain.

\begin{equation}\label{E-hmn92-6L}
\left\{
\begin{array}{ll}
\tau_{\eta_2} & = \begin{pmatrix}
                  0 & R & a & L \\
                  a & \tau(R) & 0 & \tau(L) \\
                \end{pmatrix}\\[12pt]
\ww_{\eta_2} &= |m|\ww_y.
\end{array}
\right.
\end{equation}

When $y \in \cA(\eta_1,\eta_2)$ moves monotonically clockwise from $\eta_2$ to $\eta_1$, the point $z(y)$ moves monotonically counter-clockwise from $\eta_2$ to $\eta_1$ in $\Gamma\sm \cA(\eta_1,\eta_2)$. Then, the nodal interval pushes the nodal domain $D_m$ (with $m=(n_R+2))$ which deforms into a nodal domain, still denoted $D_m$, contained in the interior of the loop $\gamma^{\eta_1}_{a,f}$ in $\cZ(u_{\eta_1})$, with $f=(2k-2)$. The nodal domains in  Figure~\ref{F-hmn92-4}, left  sub-figure, are labeled by continuity from the labeling in the central sub-figure. For the function $u_{\eta_1}$ we obtain the following nodal type and word describing the nodal domain.\smallskip

For $u_{\eta_1}$, we obtain
\begin{equation}\label{E-hmn92-6R}
\left\{
\begin{array}{ll}
\tau_{\eta_1} &= \begin{pmatrix}
                  R & a & L & f\\
                  \tau(R) & f & \tau(L) & a\\
                \end{pmatrix}\\[12pt]
\ww_{\eta_1} &= \ww_y|1|.
\end{array}
\right.
\end{equation}

We have $\cW_{\eta_1}=|1|\cW_R|1|m|\cW_L|m|1|$. Since $1$ may appear as a letter in the word $\cW_R$, the signature of the word $\cW_{\eta_1}$ given by \eqref{E-hmn92-4R} satisfies  \[\sigma(\cW_{\eta_1}) \le \|\cW_{a,R}\| = 2n_R+1=a.\] On the other hand,
$\cW_{\eta_2}=|m|1|\cW_R|1|m|\cW_L|m|$, and the letter $m$ does not appear in the word $\cW_R$, so that $\sigma(\cW_{\eta_2})=\|\cW_{a,R}\| + 2 = a+2$. We recover the fact that $u_{\eta_1}$ and $u_{\eta_2}$ have different combinatorial types. \medskip

Fix $y_1, y_2$ such that $\cA(y_1,y_2) \subset \cA(\eta_1,\eta_2)$, with $y_1$ close to $\eta_1$ (on its right), and $y_2$ close to $\eta_2$ on its left. The nodal pattern of $u_y$ for $y$ close to $\eta_1$, resp. to $\eta_2$, is displayed in Figure~\ref{F-hmn92-6-yeta12}.
For a point $x$ in $\Omega$, above $y$ and close enough to $y$, the nodal pattern of $w_x$ is displayed in Figure~\ref{F-hmn92-6-xeta12}. \smallskip

\begin{figure}[!ht]
  \centering
  \includegraphics[width=0.9\textwidth]{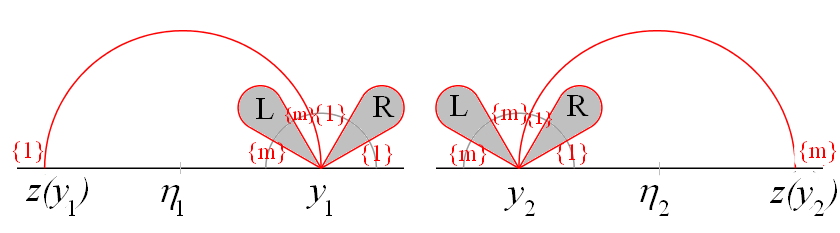}
  \caption{Nodal patterns for $y$ close to $\eta_1$ or $\eta_2$}\label{F-hmn92-6-yeta12}
\end{figure}

\begin{figure}[!ht]
  \centering
  \includegraphics[width=0.9\textwidth]{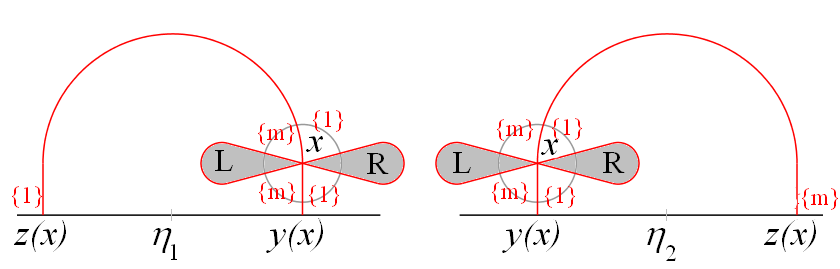}
  \caption{Nodal patterns for $x$ above $y$, and close to $y$}\label{F-hmn92-6-xeta12}
\end{figure}

\begin{figure}[!ht]
  \centering
  \includegraphics[width=0.95\textwidth]{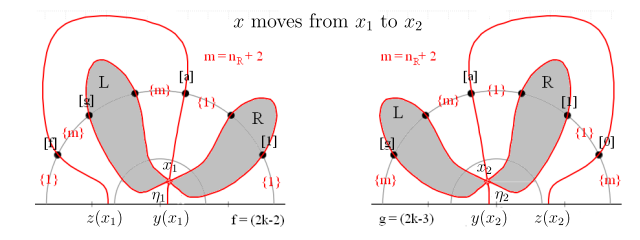}
  \caption{The transition from above $y_1$ to above $y_2$}\label{F-hmn92-k4-x-right-eta1-left-eta2}
\end{figure}

More precisely, let $t_1$ and $t_2$ be such that $\gamma(t_i) = y_i$. As in Lemma~\ref{L-gam0-4}, there exists some $s_1 > 0$ such that for all $t \in [t_1,t_2]$, and $0 < s \le s_1$, the function $w_{\gamma(s,t)}$ satisfies $\cS_{\mathrm{b}}(w_{\gamma(s,t)}) = \set{y(s,t),z(s,t)}$ with $y(s,t)$ close to $\gamma(t)$, and $z(s,t)$ close to $z(\gamma(t))$.\smallskip

\subsection{ Under Assumption~\ref{A-hmn3-0} and the further assumption $\Gamma_{(2k-2)} \neq \emptyset$}\label{SS-hmn93}~

Let $\gamma_0 : [0,L] \to \Gamma$ be an arc-length `counter-clockwise' parametrization of $\Gamma$, such that $\gamma_0(0) = \gamma_0(L) \not \in \Gamma_{(2k-2)}$. Let $m=\#(\Gamma_{(2k-2)}) \ge 2$, an even integer. Define $0 < t_1 < \cdots < t_m < L$ the points such that $\Gamma_{(2k-2)} = \set{\eta_1, \ldots, \eta_m}$, with $\eta_i = \gamma_0(t_i)$.  Let  $\gamma(s,t) := \gamma_0(t) + s \, \nu(t)$, where $\nu(t)$ is the unit normal to $\gamma_0$ at $\gamma_0(t)$, pointing inwards. For $s$ small enough, we have a diffeomorphism from $[0,s)\times [0,L]$ onto a neighborhood of $\Gamma$. We also write $\gamma_s(t)$ for $\gamma(s,t)$ and view this map as $L$-periodic in $t$. Denote by $\Omega_{s_0}$ the simply connected domain contained in $\Omega$, and bounded by $\Gamma_{s_0}\,$, with $\Gamma_{s_0} := \partial \Omega_{s_0} = \gamma_{s_0}([0,L])$. \smallskip

We now choose $r_1, r_2, r_3$, and $s_0$ small enough so that the following properties hold.
\begin{enumerate}[i)]
  \item The number $r_1, r_2$ are chosen according to Subsection~\ref{SS-hmn91}, which describes the behaviour of $\cZ(w_x)$ in  $D_{+}(\eta,r_1)$ for any $x$ in $D_{+}(\eta,r_2)$, and any $\eta$ in $\Gamma_{(2k-2)}$.
  \item For $j \in \set{1,\ldots,m}$, define the points $t_j^{\pm}:= t_j\pm \frac 12 r_2$, and first choose $s_0$ so that $\gamma(s_0,t_j^{\pm}) \in D_{+}(\eta_j,r_2)$. According to Remark~\ref{R-gam0-2}, there exists some $r_3 > 0$ such that, choosing $s_0$ small enough, Subsection~\ref{SS-hmn92} applies to the behavior of $\cZ(w_{\gamma(s_0,t)})$ in $D_{+}(\gamma(t),r_3)$, for $t \in [t_j^{+} ,t_{j+1}^{-}]$, $y(\gamma(s_0,t)) \neq z(\gamma(s_0,t))$, and we can follow these points by continuity.
\end{enumerate}
\smallskip

Figure~\ref{F-hmn97-eta1R-y-eta2LR-ex1} displays the nodal sets $\cZ(w_{\gamma(s_0,t)})$ for $t = t_1^{+}, t \in (t_1^{+},t_2^{-}), t= t_2^{-}, t=t_2^{+}$ (here $k=6$).

\begin{figure}[!ht]
  \centering
  \includegraphics[width=0.98\textwidth]{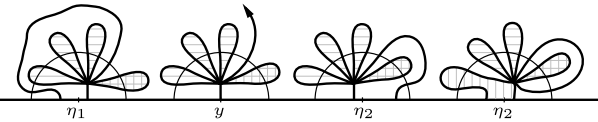}
  \caption{$\cZ(w_{\gamma(s_0,t)})$ for $t = t_1^{+}\,, ~t \in (t_1^{+}\,, t_2^{-}), ~t= t_2^{-}\,, ~t=t_2^{+}$}\label{F-hmn97-eta1R-y-eta2LR-ex1}
\end{figure}

The $\Cty$ family $\Omega \ni x \mapsto w_x \in U(\lambda_k)$ has the property that the function $w_x$ vanishes precisely at order $(k-1)$ at the point $x$, and that the leading term $h_x$ in the Taylor expansion of $w_x$ at the point $x$ is a nonzero harmonic homogeneous polynomial of degree $(k-1)$ in $T_x\Omega$.\smallskip

We fix a reference orthonormal direct frame in $\R^2$, with coordinates $(\xi_1,\xi_2) = [\rho,\theta]$. We use ${C,S}$ as the basis for the two dimensional space of harmonic homogenous polynomials of degree $(k-1)$ in $\R^2$, where
\[
C(\rho\cos\theta,\rho\sin\theta)=\rho^{k-1}\cos((k-1)\theta); \quad
S(\rho\cos\theta,\rho\sin\theta)=\rho^{k-1}\sin((k-1)\theta).
\]
We represent $h_x$ in this basis as
\[
h_x = a_x \, C + b_x \, S
\]
with $a_x^2+b_x^2 > 0$.  Since $x \mapsto h_x$ is $\Cty$, it follows that we have a $\Cty$ map
\begin{equation}\label{E-hmn93-2}
\tilde{h}: \Omega_{s_0} \ni x \mapsto \tilde{h}_x :=\big( a_x \, (a_x^2+b_x^2)^{-\frac 12}, b_x \, (a_x^2+b_x^2)^{-\frac 12}   \big) \in \bS^{1}.
\end{equation}
Since $\Omega_{s_0}$ is simply connected, the restriction $\tilde{h}|_{\Gamma_{s_0}}$ of this map to $\Gamma_{s_0}$ has degree $0$.

\begin{claim}\label{C-hmn93-2}
The map $\tilde{h}|_{\Gamma_{s_0}}$ has nonzero degree.
\end{claim}%

\noi \pf Define the  angle $\phi_{\gamma(s_0,t)}$ by continuity along $\gamma_{s_0}(\R)$ so that
\[
\tilde{h}_{\gamma(s_0,t)} = \big( \cos (\phi_{\gamma(s_0,t)}), \sin (\phi_{\gamma(s_0,t)})\big).
\]
Consider the map $x \mapsto h_x$. The zero set of the polynomial $h_{x}$ consists of the $(2k-2)$ equi-angular rays tangent to the nodal set of $w_{x}$ at the point $x$, the so-called \emph{star} $\Sigma_{x}$ at this point.
These rays $\omega_{x,j}$  are the zeros of the equation $\cos\big( (k-1)\theta - \phi_x\big) = 0$, so that
\[
\omega_{x,j} = \frac{1}{k-1}\big(\frac{\pi}{2} + \phi_x \big) + j \frac{\pi}{k-1}, \quad j\in \set{0,\ldots, (2k-3)}.
\]
When $t$ varies, we can follow one ray in $\Sigma_{\gamma(s_0,t)}$ by continuity. If this ray turns by an angle $\omega$, then  the previous equation shows that $\tilde{h}_{\gamma(s_0,t)}$ turns by the angle $(k-1)\omega$.  To compute the degree of $\tilde{h}|_{\Gamma_{s_0}}$ it suffices to follow one ray of $\Sigma_{\gamma(s_0,t)}$.\smallskip

Fix some $j, 1 \le j \le m$ (with the convention that $t_{m+1} = t_1$). Call $e^{(j)}(s_0)$ the unit vector at $\gamma(s_0,t_j^{+})$ which is tangent to the nodal arc emanating from $\gamma(s_0,t_j^{+})$ which intersects $C_{+}(\eta_j,r_1)$ at $A_{\gamma(s_0,t_j^{+}),2}(r_1)$ near the point $A_{\eta_j,2}(r_1)$. Viewing $\gamma$ as an $L$-periodic function in $t$, call $e^{(j)}(\gamma(s_0,t))$ the continuous unit vector field along $\gamma(s_0,[t_j, t_j+L))$ which takes the value $e^{(j)}(s_0)$ at $t_j^{+}$ and such that $e^{(j)}(\gamma(s_0,t))$ belongs to the star $\Sigma_{\gamma(s_0,t)}$.

According to Subsection~\ref{SS-hmn92}, for $t \in [t_j^{+},t_{j+1}^{-}]$, the vector $e^{(j)}(\gamma(s_0,t)$ is tangent at the point $\gamma(s_0,t)$ to the nodal arc from $\gamma(s_0,t)$ to $A_{\gamma(s_0,t),2}(r_1)$. When $t$ varies from $t_j^{+}$ to $t_{j+1}^{-}$, the nodal interval $\delta_{\gamma(s_0,t)}^{z(\gamma(s_0,t))}$ changes continuously and we have the following phenomenon:
\begin{enumerate}[$\diamond$]
  \item For $t:= t_j^{+}$, $\delta_{\gamma(s_0,t)}^{z(\gamma(s_0,t))}$ exits $D_{+}(\eta_j,r_1)$ near the point $A_{\eta_j,a_j}(r_1)$ for some integer $a_j$, and re-enters $D_{+}(\eta_j,r_1)$ near the point $A_{\eta_j,(2k-2)}(r_1)$.
  \item For $t:= t_{j+1}^{-}$, $\delta_{\gamma(s_0,t)}^{z(\gamma(s_0,t))}$ exits $D_{+}(\eta_{j+1},r_1)$ near the point $A_{\eta_{j+1},a_j+1}(r_1)$, and re-enters $D_{+}(\eta_{j+1},r_1)$ near the point $A_{\eta_{j+1},1}(r_1)$. This is illustrated in Figure~\ref{F-hmn92-k4-x-right-eta1-left-eta2}, right picture, in which the points $A_{\eta_{j+1},i}(r_1)$ are labeled $i=0,1, \ldots, (2k-3)$ where as we use the labeling $i=1,\ldots, (2k-2)$ in the previous statement. The important fact is the shift from $a_j$ to $a_j+1$.
  \item Furthermore, for $t \in (t_j^{+},t_{j+1}^{-})$ the nodal interval $\delta_{\gamma(s_0,t),e^{(j)}(\gamma(s_0,t))}$ exits the set $D_{+}(\eta_{j+1},r_1)$ near the point $A_{\eta_{j+1},a_j+1}(r_1)$, and re-enters $D_{+}(\eta_{j+1},r_1)$ near the point $A_{\eta_{j+1},1}(r_1)$.
\end{enumerate}

This behavior is illustrated in Figure~\ref{F-hmn97-nu6-eta-e}. Figure~\ref{F-hmn97-k4-eta-1-2-case2} gives a more global view.

\begin{figure}[!ht]
  \centering
  \includegraphics[width=0.98\textwidth]{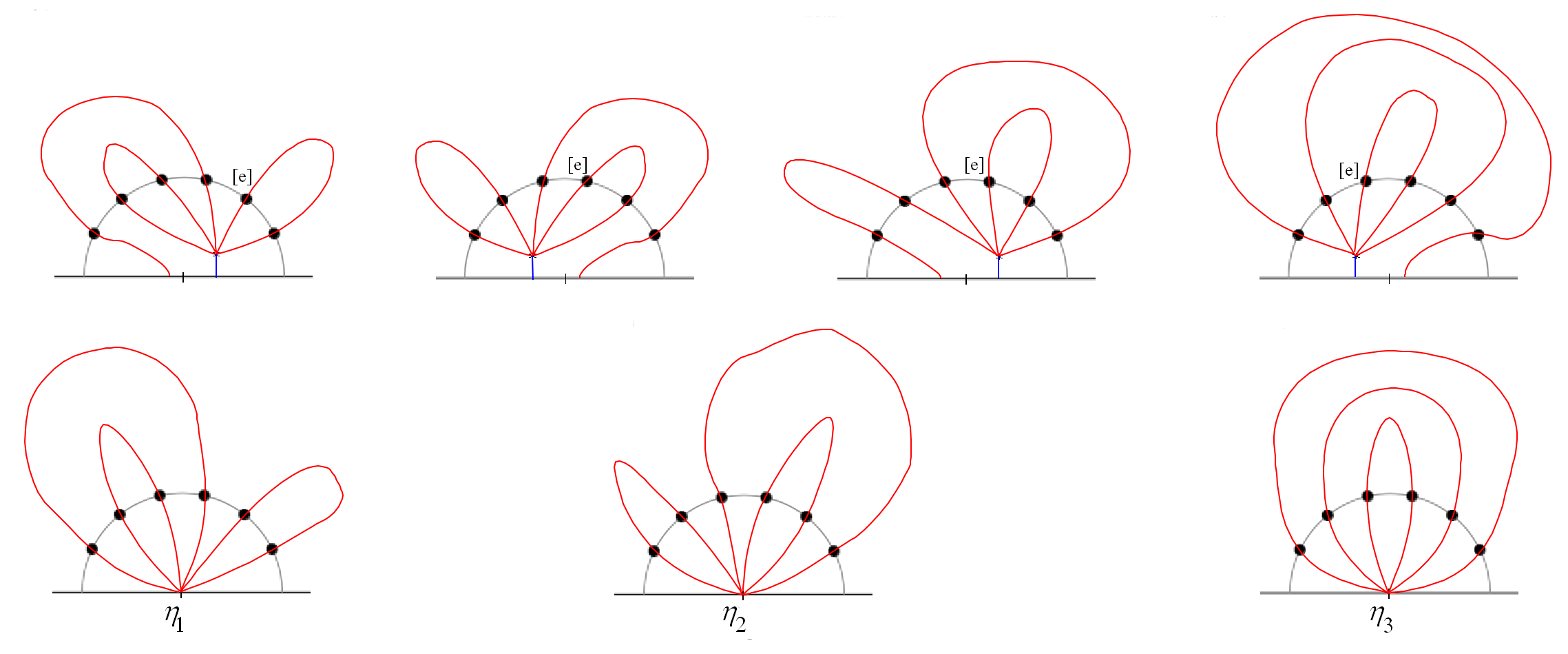}
  \caption{$e^{(1)}$ is in  positions 2-3-3-4 on $C_{+}$}\label{F-hmn97-nu6-eta-e}
\end{figure}\smallskip

\begin{figure}[!t]
  \centering
  \includegraphics[width=0.95\textwidth]{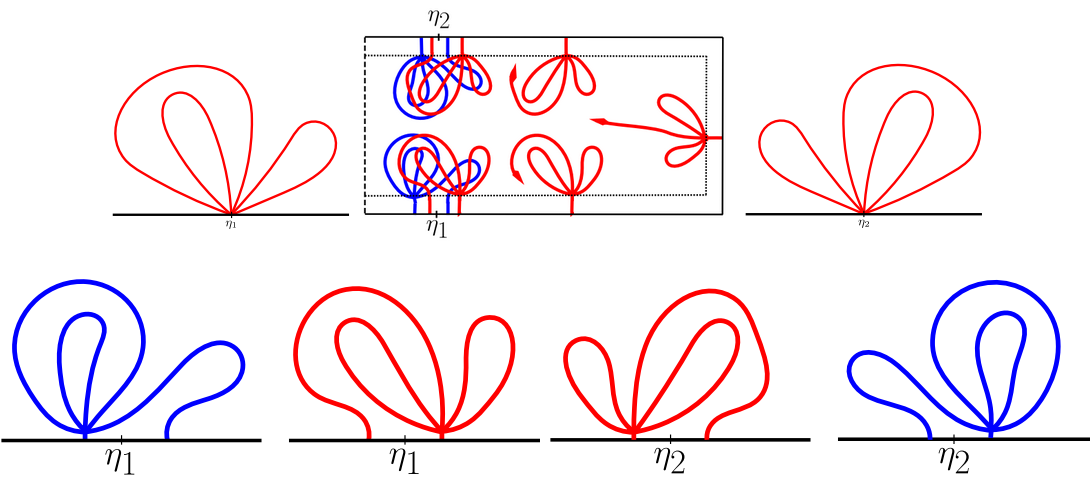}
  \caption{From $t_1^{+}$ to $t_2^{+}\,$, local and global views} \label{F-hmn97-k4-eta-1-2-case2}
\end{figure}

We interpret this phenomenon by saying that the angle of the vector $e^{(j)}(\gamma(s_0,t))$ with respect to the $t$-derivative $\gamma'(s_0,t)$ has increased by $\frac{\pi}{k-1}$ when passing from $t_j^{+}$ to $t_{j+1}^{+}$.
Otherwise stated, in the reference frame, the vector $\tilde{h}_{\gamma(s_0,t)}$ has turned by $\pi + (k-1)\measuredangle(\gamma'(s_0,t_j^{+}),\gamma'(s_0,t_{j+1}^{+}))$, where $\measuredangle$ denotes the angle between two vectors. Moving along $\Gamma_{s_0}$, the vector $\tilde{h}$ has turned by $m\, \pi + 2(k-1)\pi$ in the reference frame. It follows that the degree of  $\tilde{h}$ is $\frac m2 + (k-1)$, a positive integer since $m$ is even and positive. Claim~\ref{C-hmn93-2} is proved. \quad \qedc

\begin{remark}
One could also give a ``degree proof'' of Lemma~\ref{L-gam0-5}. In the framework of this lemma, the star along $\gamma_s$, turning counter-clockwise, follows the moving frame. When $\Gamma_{(2k-2)} \neq \emptyset$, the same occurs with extra counter-clockwise rotations due to the presence of the points $\eta_j \in \Gamma_{(2k-2)}$.
\end{remark}%

\begin{remark}\label{R-hmn93-PH}
An alternative approach to reach a contradiction is to construct a non-zero continuous vector-field, and to apply the Poincar\'{e}-Hopf theorem with boundary, Theorem~\ref{ThPHbd}.
\end{remark}%

 We can summarize the previous analysis in  the following lemma.
\begin{lemma}\label{L-gam0-7}
 Under  Assumptions~\ref{A-hmn3-0},  the set $\Gamma_{(2k-2)}$ cannot be non-empty.
\end{lemma}

\subsection{Overall conclusion}\label{SS-hmn93a}

Putting Subsection~\ref{SS-hmn92} (Lemma~\ref{L-gam0-5}) and Subsection~\ref{SS-hmn93} (Lemma~\ref{L-gam0-7}) together, we see that  Assumptions~\ref{A-hmn3-0},  lead to a contradiction. Therefore, $\mult(\lambda_k) \le (2k-3)$ for all $k\ge 3$.

\subsection{Examples}\label{SS-hmn94}

 In this subsection, we look at simple examples which shed some more light on the approach in Subsection~\ref{SS-hmn93}. Indeed moving along $\Gamma$ counter-clockwise starting from $\eta_1$,  the combinatorial type of $u_y$ changes on crossing a point $\eta_j$ (Lemma~\ref{L-L33c}). In some cases, arriving back at $\eta_1$, the  combinatorial type is different from the original one, a contradiction since $\dim U_y = 1$ for all $y \in \Gamma$.\smallskip

Let $m := \#\big( \Gamma_{(2k-2)}\big)$. As we already know,  $m$ is positive (see Lemma \ref{L-gam0-5})  and even (see Lemma \ref{L-L33c}). \smallskip

\noi Let us consider the simple case in which $m=2$, with  $\Gamma_{(2k-2)} = \set{\eta_1,\eta_2}$. Figure~\ref{F-hmn95-k4-rho6-m2-i}, left picture, exhibits an impossible configuration. Indeed, when the base point $y$ moves away from $\eta_1$ towards $\eta_2$, and continues moving to reach $\eta_1$ again, the combinatorial type of $u_y$ changes according to the figure. The corresponding functions give rise to the words $\cW_{\eta_1} = |1|2|1|3|4|3|1|$, $\cW_{\eta_2} = |1|2|3|2|1|4|1|$, and then $\cW_{\eta_1} = |1|2|3|4|3|2|1|$. The first and third words have different signatures, a contradiction.

\begin{figure}[!ht]
  \centering
  \includegraphics[scale=0.85]{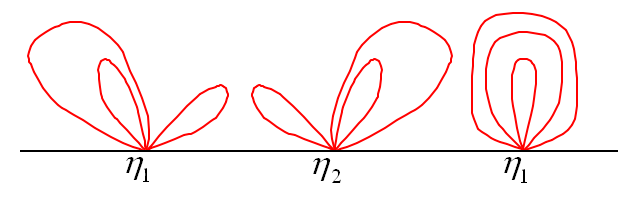}
  \caption{Case $m=2$, $k=4$, impossible configuration}\label{F-hmn95-k4-rho6-m2-i}
\end{figure}

Figure~\ref{F-hmn95-k4-rho6-m2-q} displays a case in which no such contradiction is reached  by this argument.  In this example, $y$ moves counter-clockwise (i.e. from left to right) from $\eta_1$ to $\eta_2$, and then to $\eta_1$ again.  To visualize the changes better, we change the color\footnote{\label{FN-color}To identify colors in black and white printing, blue arcs are labeled by a bullet, red arcs by a square. No label for black arcs.} of the loop which opens up,  first from black to red at $\eta_1$, then from black to blue at $\eta_2$.  The red loop $\gamma^{\eta_1}_{1,6}$ in the left figure opens up on the left of $\eta_1$,  the end point $z(y)$ moves from $\eta_1$ to $\eta_2$ clockwise, and the loop closes up again on the right of  $\eta_2$ to become the red loop $\gamma^{\eta_2}_{1,2}$ (middle figure). When $y$ continues moving counter-clockwise, from $\eta_2$ to $\eta_1$, the blue loop at $\gamma^{\eta_2}_{5,6}$ opens up on the left of $\eta_2$, its end point $z(y)$ moves clockwise from $\eta_2$ to $\eta_1$, and closes up again on the right of $\eta_1$ to become the blue loop $\gamma^{\eta_1}_{1,6}$  (right figure). The colors show that although the combinatorial types of $u_{\eta_1}$ and $u_{\eta_2}$ are identical, there is some change in the colored loop, hinting at some kind of rotation.
\begin{figure}[!ht]
  \centering
  \includegraphics[scale=0.9]{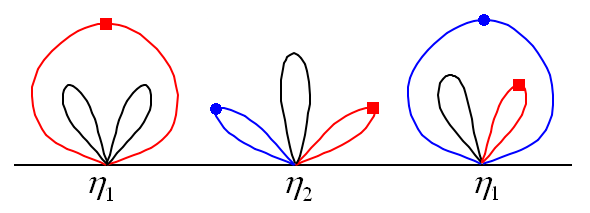}
  \caption{Case $m=2$, $k=4$}\label{F-hmn95-k4-rho6-m2-q}
\end{figure}
 This rotation can be visualized as follows. Label the nodal domains in the left figure. When $y$ moves counter-clockwise from $\eta_1$ to $\eta_2$ and then to $\eta_1$ again, follow the nodal domains under deformation. The labels are indicated by the numbers between braces in Figure~\ref{F-hmn95-k4-rho6-m2-qnd}. The corresponding words are respectively given by
\begin{equation*}
|1|2|3|2|4|2|1|, \qquad  |2|1|2|3|2|4|2|, \qquad |4|2|1|2|3|2|4| .
\end{equation*}
\begin{figure}[!ht]
  \centering
  \includegraphics[scale=0.9]{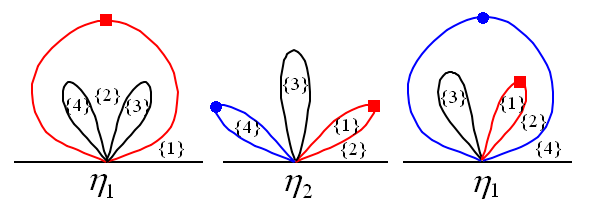}
  \caption{Case $m=2$, $k=4$, with nodal labeling}\label{F-hmn95-k4-rho6-m2-qnd}
\end{figure}
Since the first and last letter are always the same, it turns out to be more appropriate to look at these words as written on a circle (so that the last letter is suppressed). This is indicated by the symbol $\circlearrowleft$ at the end of the words. With this convention, the words in Figure~\ref{F-hmn95-k4-rho6-m2-qnd} are now given by
\begin{equation*}
|1|2|3|2|4|2|\!\circlearrowleft \,, \qquad  |2|1|2|3|2|4|\!\circlearrowleft \,, \qquad |4|2|1|2|3|2|\!\circlearrowleft \,,
\end{equation*}
as illustrated in Figure~\ref{F-hmn95-m2-k4-words}. This indicates a positive rotation by $\frac{2\pi}{3}$.
\begin{figure}[!ht]
  \centering
  \includegraphics[scale=0.2]{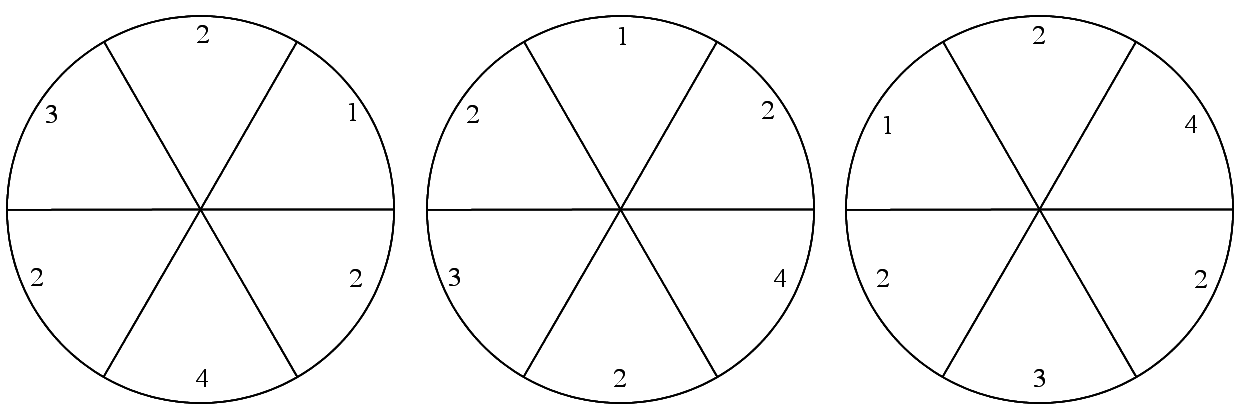}
  \caption{Nodal words seen on the circle}\label{F-hmn95-m2-k4-words}
\end{figure}

\begin{figure}[!ht]
  \centering
  \includegraphics[scale=0.9]{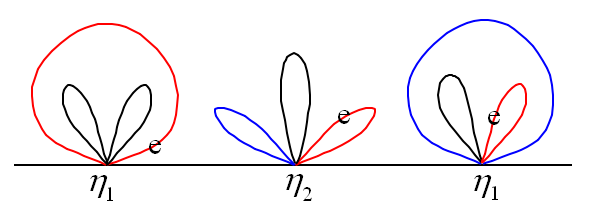}
  \caption{Case $m=2$, $k=4$, with the vector $e^{(1)}$}\label{F-hmn95-k4-rho6-m2-qe}
\end{figure}

The previous situation always occurs when $m=2$ and $k=3$,  as we now show. \medskip

\emph{Case $m := \#\big( \Gamma_{(2k-2)}\big) = 2$, $k=3$.} In this case, there are two possible combinatorial types at $\eta \in \Gamma_{(2k-2)}$,  namely
\begin{equation*}
\tau_{1} = \begin{pmatrix}
                  1 & 2 & 3 & 4 \\
                  4 & 3 & 2 & 1 \\
                \end{pmatrix}
\text{~~and~~}
\tau_{2} = \begin{pmatrix}
                  1 & 2 & 3 & 4 \\
                  2 & 1 & 4 & 3 \\
                \end{pmatrix}.
\end{equation*}%

\begin{figure}[!ht]
  \centering
  \includegraphics[scale=0.3]{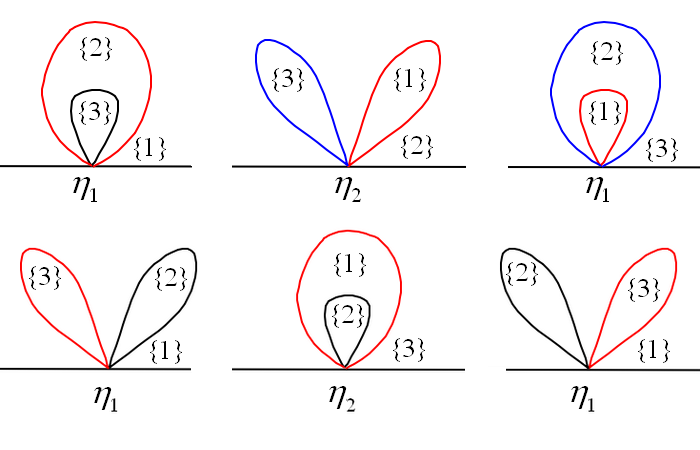}
  \caption{Evolution of the nodal domains ($m=2, k=3$)}\label{F-hmn95-m2-k3-linex2-nd}
\end{figure}

\begin{figure}[!ht]
  \centering
  \includegraphics[scale=0.35]{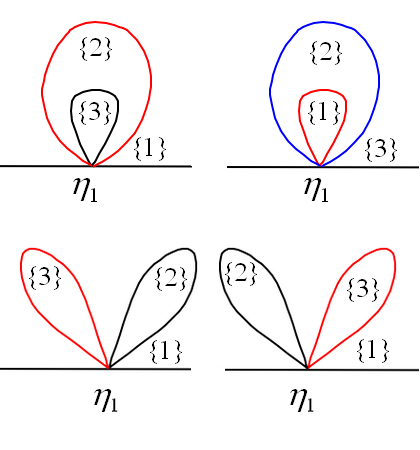}
  \caption{$\cZ(u_{\eta_1})$: initial (left) -- after returning to $\eta_1$ (right)}\label{F-hmn95-m2-k3-linex2-nde1}
\end{figure}

\newpage

Figure~\ref{F-hmn95-m2-k3-linex2-nd} describes the evolution of $\cZ(u_y)$ when $y$ moves counter-clockwise from $\eta_1$ to $\eta_2$, and then to $\eta_1$. The nodal domains of $u_{\eta_1}$ are indicated in the pictures on the left; the  other pictures then indicate how the nodal domains deform. The initial and final words (seen on the circle as above) are then given by
\begin{equation*}
\begin{array}{lll}
\text{Row 1:}& |1|2|3|2|\!\circlearrowleft  & \quad |3|2|1|2|\!\circlearrowleft \\[5pt]
\text{Row 2:}& |1|2|1|3|\!\circlearrowleft  & \quad |1|3|1|2|\!\circlearrowleft
\end{array}%
\end{equation*}
and they differ by a circular permutation which indicates a rotation by $\pi$. The initial and final patterns are best compared in Figure~\ref{F-hmn95-m2-k3-linex2-nde1}. \medskip

\begin{figure}[!ht]
  \centering
  \includegraphics[width=0.9\textwidth]{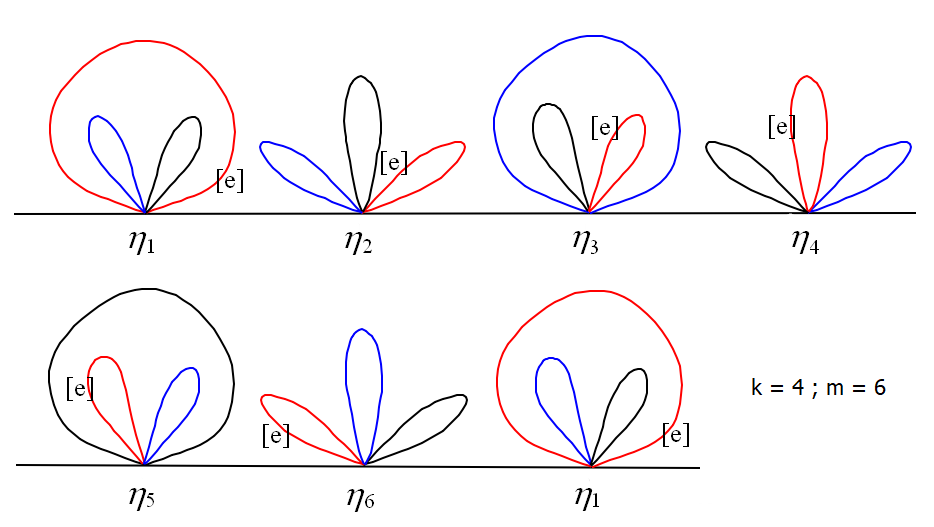}
  \caption{The case $k=4, m=6$, with the vector $e$} \label{F-hmn95-k4-rho6-m6}
\end{figure}

\begin{figure}[!ht]
  \centering
  \includegraphics[width=0.9\textwidth]{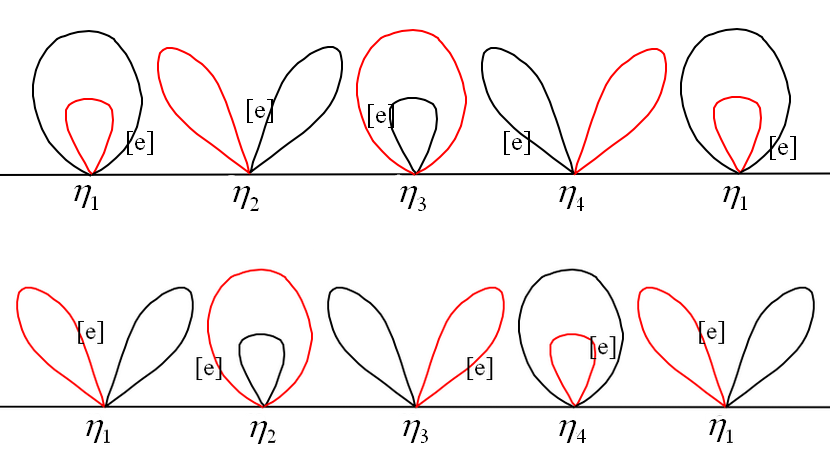}
  \caption{The case $k=3, m=4$, with the vector $e$} \label{F-hmn95-k3-rho4-m4}
\end{figure}
\medskip

Figures~\ref{F-hmn95-k3-nu4-eta1} and \ref{F-hmn95-k3-nu4-eta2} describe the behavior of $\cZ(w_x)$ when $x$ is near $\eta_1$ or $\eta_2$, and moves on a parallel curve close enough to $\Gamma$, on the left, resp. on the right. The end point $y(x)$ of the blue arc and the end point $z(x)$ of the red arc satisfy $y(x) < z(x)$, resp. $z(x) < y(x)$. When $x$ moves, the points may coincide or disappear (the nodal set $\cZ(w_x)$ does not touch the boundary). In any case, the intersection points of $\cZ(w_x)$ with the curve $C_{+}(\eta_i,r)$, $i=1,2$ remain in small pairwise disjoint intervals (the black dots) around the points in $\cZ(u_{\eta_i}) \cap C_{+}(\eta_i,r)$, $i=1,2$.

\begin{figure}[!ht]
  \centering
  \includegraphics[width=0.60\textwidth]{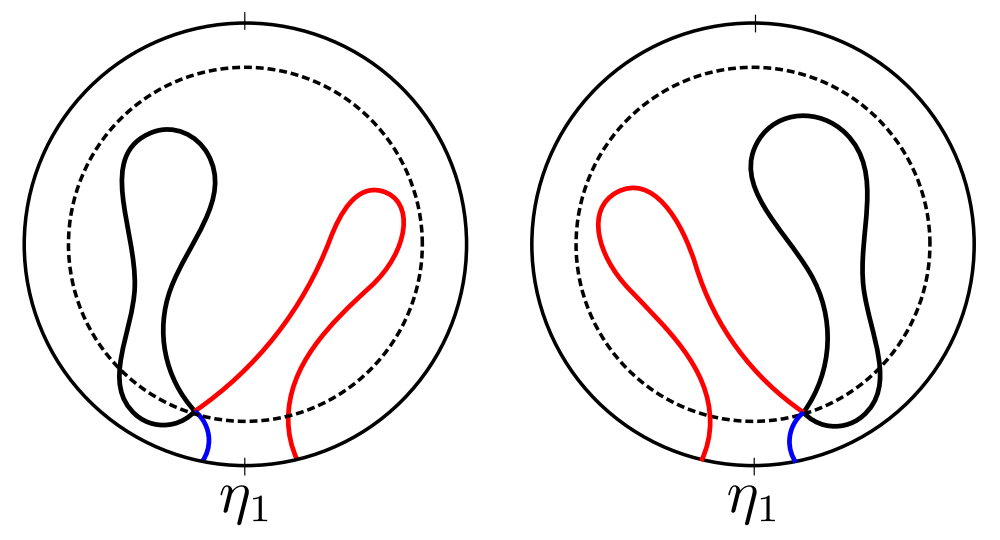}
  \caption{$m=2$, $k=3$, $x$ near $\eta_1$}\label{F-hmn95-k3-nu4-eta1}
\end{figure}

\begin{figure}[!ht]
  \centering
  \includegraphics[width=0.60\textwidth]{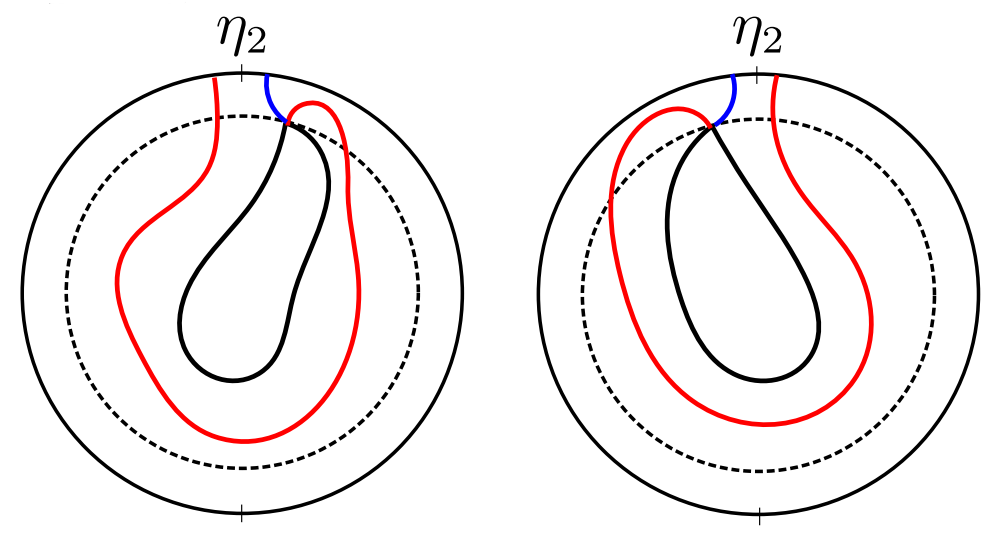}
  \caption{$m=2$, $k=3$, $x$ near $\eta_2$}\label{F-hmn95-k3-nu4-eta2}
\end{figure}

\begin{figure}[!ht]
  \centering
  \includegraphics[width=0.99\textwidth]{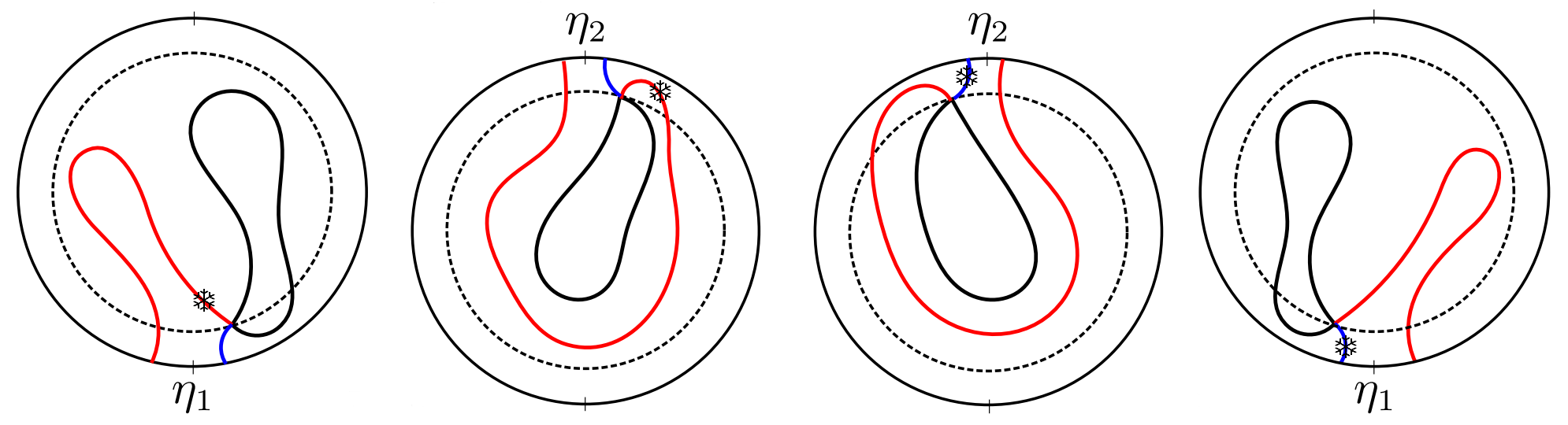}
  \caption{$m=2$, $e^{(1)}$}\label{F-hmn95-k3-nu4-m2}
\end{figure}

\begin{figure}[!ht]
  \centering       
  \includegraphics[width=0.9\textwidth]{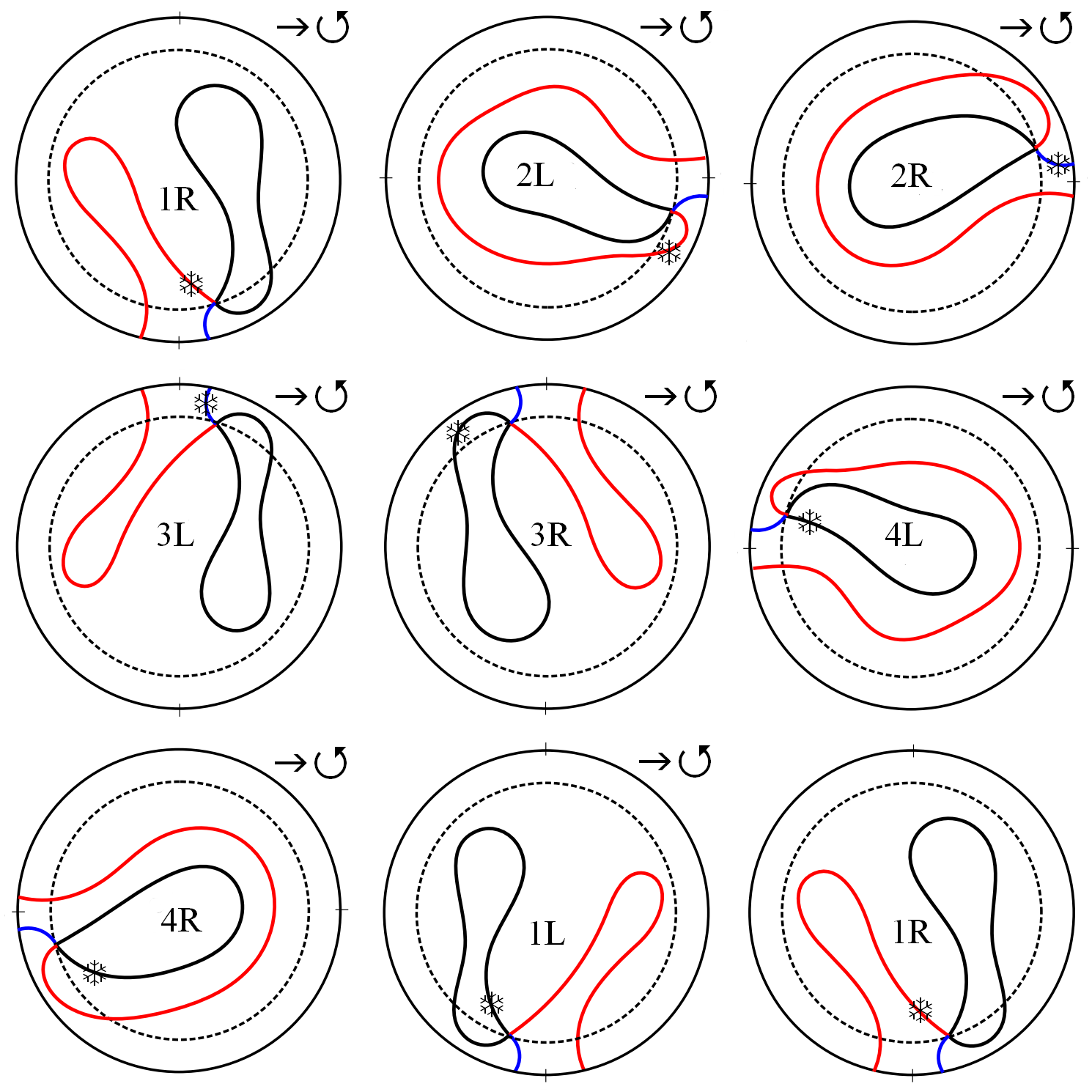}
  \caption{$m=4$, $e^{(1)}$}\label{F-hmn95-k3-nu4-m4}
\end{figure}

\FloatBarrier

\section{Labeling Nodal Domains}\label{S-hmn2L}

\subsection{Preliminaries}\label{SS-hmn2L-pre}

Let $y \in \Gamma$ and let $u$ be an eigenfunction such that $\rho(u,y) = q$.\smallskip

We will either work in $\Omega$ (global picture), or in a neighborhood of $y$ (local picture) of the form $E\big( D_{+}(0,r)\big)$ where $E$ is a conformal map as in Section~\ref{S-lsbs}. By abuse of notation, we shall denote such a neighborhood by $D_{+}(y,r)$ without mentioning $E$, and work there as if we were actually working in $\bH$.\smallskip

Let $\set{\omega_1,\ldots,\omega_q}$ be the rays tangent to the nodal set $\cZ(u)$ at the point $y$. In the Dirichlet case, the rays are given by  $\omega_{j} = j \frac{\pi}{q}$, for $1\le j\le q$. In the Robin case, they are given by $\omega_{j} = (j-\frac 12) \frac{\pi}{q}$.\smallskip

Let $r > 0$ be small enough so that the local structure theorem applies to $u$ in $D_{+}(y,2r)$. For $1 \le j \le q$, the nodal arc of $u$ emanating from $y$ tangentially to the ray $\omega_j$ intersects $C_{+}(y,r)$ at a unique point $A^{u}_j(r)$ close to the intersection point $\omega_j \cap C_{+}(y,r)$. Call $A_{+}(r)$, resp. $A_{-}(r)$, the intersection points of $C_{+}(y,r)$ with $\Gamma$ (these points do not belong to $\cZ(u)$). The points $A^{u}_j(r)$ determine $(q+1)$ intervals on $C_{+}(y,r)$, denoted $I^{u}_j(r)$ for $1 \le j \le (q+1)$: $I^{u}_1(r)$ is the arc of $C_{+}(y,r)$ from $A_{+}(r)$ to $A^{u}_1(r)$ (moving on $\Gamma$ counter-clockwise), \ldots, $I^{u}_{q+1}(r)$ is the arc from $A^{u}_q(r)$ to $A_{-}(r)$. We shall skip the superscripts when the context is clear.\medskip

Throughout this section, we fix some $k \ge 3$, and we consider eigenfunctions $u$ satisfying the following assumptions.

\begin{assumptions}\label{A-hmn2L-2}~
\begin{enumerate}[(i)]
  \item The function $u$ satisfies $\rho(u,y) = q$, for some $y \in \Gamma$, $q \in \set{(2k-3),(2k-2)}$.
  \item When $q=(2k-2)$, the point $y$ is the only singular point of $u$ and $\cZ(u)$ is a $(k-1)$-bouquet of loops at $y$.
  \item When $q=(2k-3)$, $u$  has two singular points $y$ and $z\neq y$,
  both in $\Gamma$, and there is a nodal interval $\delta = \delta^{u}_{b,y,z}$ which emanates from $y$ tangentially to the ray $\omega_b$ and hits $\Gamma$ at $z$. The nodal set $\cZ(u)$ is the wedge sum of the nodal interval $\delta$ with a $(k-2)$-bouquet of loops at $y$.
  \item The function $u$ has $k$ nodal domains, i.e., $\kappa(u)=k$.
  \item For $r$ small enough, all the nodal domains of $u$ intersect $C_{+}(y,r)$.
\end{enumerate}
\end{assumptions}%

 Eigenfunctions satisfying the above assumptions occur in Section~\ref{S-hmn3}, Lemma~\ref{L-L32} and Figure~\ref{F-hmn3-h0-L33}.  They also occurred in  Section~\ref{S-hmn2}, Lemmas~\ref{L-Uxy0} and \ref{L-Ux1}, Figures~\ref{F-hmn2-Uxy0} and \ref{F-hmn2-Ux100}). \smallskip

Examples are displayed in Figures~\ref{F-hmn2L-2} and \ref{F-hmn2L-4}. To differentiate the two cases, we will denote by $u_0$ or $v_0$ a function for which $q = (2k-2)$, and $u_1$ a function for which $q = (2k-3)$.\smallskip

\begin{figure}[!ht]
\centering
\begin{subfigure}[t]{.40\textwidth}
\centering
\includegraphics[width=\linewidth]{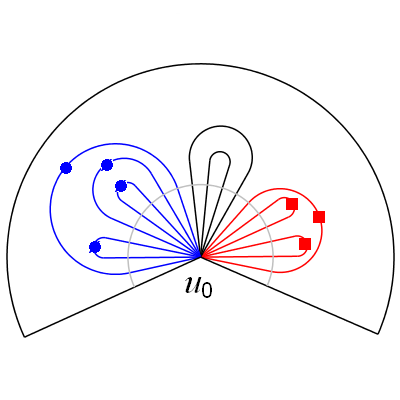}
\end{subfigure}
\begin{subfigure}[t]{.40\textwidth}
\centering
\includegraphics[width=\linewidth]{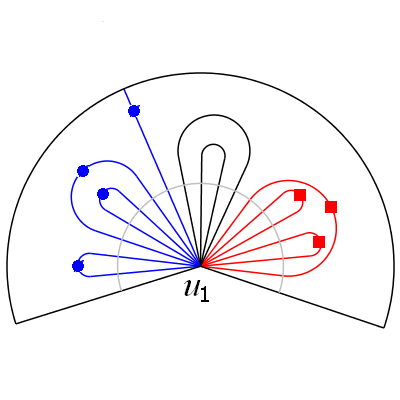}
\end{subfigure}
\caption{$\rho(u_0,y) = (2k-2)$, $\rho(u_1,y) =(2k-3)$ [here $k=10$]}\label{F-hmn2L-2}
\end{figure}

\begin{figure}[!ht]
\centering
\begin{subfigure}[t]{.40\textwidth}
\centering
\includegraphics[width=\linewidth]{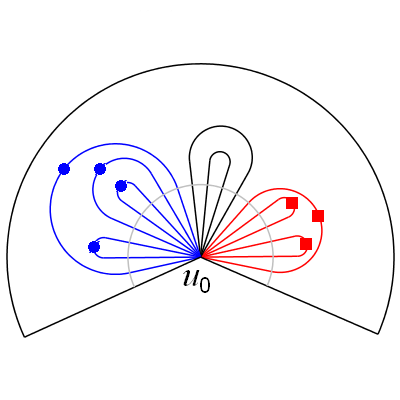}
\end{subfigure}
\begin{subfigure}[t]{.40\textwidth}
\centering
\includegraphics[width=\linewidth]{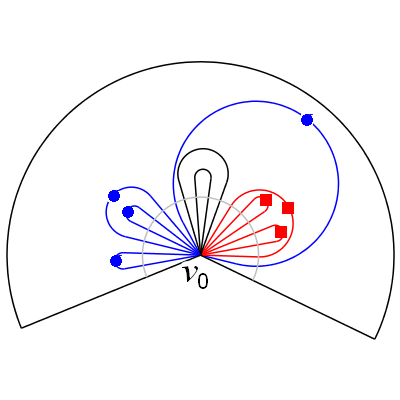}
\end{subfigure}
\caption{$\rho(u_{0},y) = \rho(v_{0},y) = (2k-2)$ [here $k=10$]}\label{F-hmn2L-4}
\end{figure}

The purpose of this section is to give a simple criterion to establish that the functions $u_{0}$ and $v_{0}$, whose nodal patterns are displayed in Figure~\ref{F-hmn2L-4}, are different in $\bP(U)$, i.e., $[u_0] \neq [v_0]$.  This criterion is used in Paragraph~\ref{SSS-hmn-25D},  in Section~\ref{S-hmn3}, and in Section~\ref{S-hmn32b}.

\FloatBarrier

\subsection{Labeling nodal domains and the signature of nodal patterns}\label{SS-hmn2L-lnds}
\index{Signature (nodal patterns)}
\begin{definition}\label{D-hmn2L-2}
A \emph{labeling of the nodal domains} \index{Labeling (nodal domains)} of an eigenfunction $u$ with $\kappa(u)=k$ is a set of  pairwise distinct labels $\cD := \set{d_1, \ldots , d_k}$  in bijection with the set of nodal domains of $u$, so that they can be listed as $D_{d_1},\ldots,D_{d_k}$. \end{definition}%

Let $\cD = \set{d_1, \ldots , d_k}$ be a labeling of the nodal domains of $u$. Given $r$ small enough, let $\set{I^{u}_{i}(r)}_{i=1}^{q+1}$ be the intervals determined by the points $\cZ(u) \cap C_{+}(y,r)$ on $C_{+}(y,r)$. We attach labels to these intervals as follows: for $1 \le j \le k$,  the label $d_j$ is attached to the intervals $I^{u}_i(r)$ which are contained in the nodal domain $D_{d_j}$.\\
 Note that the same label $d_j$ may be given to several intervals.\\ We now encode this information into a word $\ww_{u,\cD}$, of length $\|\ww_{u,\cD}\| = (q+1)$,
\begin{equation*}
\ww_{u,\cD} =\ell_{u,\cD,1} \cdots \ell_{u,\cD,(q+1)},
\end{equation*}
whose letters $\ell_{u,\cD,i}$, $1 \le i \le (q+1)$, belong to the labeling set $\cD$.\\
Note that the word does not depend on $r$ provided that $r$ is small enough (this is a consequence of the local structure theorem, Section~\ref{S-lsbs}). Since an eigenfunction changes sign across a nodal line, two consecutive letters in the word $\ww_{u,\cD}$ are different.  Different label sets $\cD_1$ and $\cD_2$ give rise to a priori different words.

\begin{definition}\label{D-hmn2L-4}
Let $\cI_{u,\cD} := \set{i,\, 1 \le i \le (q+1)  \mid \ell_{u,\cD,i} = \ell_{u,\cD,1}}$.
The \emph{signature} \linebreak $\sigma(\ww_{u,\cD})$ of the word $\ww_{u,\cD}$
is defined as
\begin{equation}\label{E-hmn2L-0}
\sigma(\ww_{u,\cD}) :=
          \left\{
          \begin{array}{ll}
          1 &\text{if~~} \cI_{u,\cD} = \set{1},\\[5pt]
          \min(\big( \cI_{u,\cD}\sm \set{1}\big) &\text{if~~} \cI_{u,\cD} \neq \set{1}.
           \end{array}
           \right.
\end{equation}
\end{definition}%

\begin{properties}\label{P-hmn2L-2}
The signature $\sigma$ \index{Signature} is well defined for eigenfunctions satisfying Assumptions~\ref{A-hmn2L-2}, and does not depend on the labeling set $\cD$.
\end{properties}%

 Indeed, let $D$ be the nodal domain of u which contains the interval $I_1^{u}(r)$ and let  $\cI_{u} := \set{j \mid I_j^{u}(r) \subset D}$. Then,
$\cI_{u,\cD} = \cI_{u}$, and
\begin{equation*}
\sigma(\ww_{u,\cD}) = \left\{
\begin{array}{ll}
1 & \text{if~~} \cI_u = \set{1}\\[5pt]
\min \big( \cI_u \sm \set{1}\big)& \text{if~~} \cI_u \neq \set{1},
\end{array}
\right.
\end{equation*}
and the right hand side is clearly independent of the choice of the  labeling set.
Figure~\ref{F-hmn2L-6a} illustrates this  fact (here on the coloring scheme). Figure~\ref{F-hmn2L-6b} illustrates the fact that the signature provides a criterion to distinguish different nodal patterns.\smallskip

 As we shall see below, a labeling of the nodal domains of an eigenfunction satisfying Assumptions~\ref{A-hmn2L-2} can be deduced from the combinatorial type $\tau_u$ of $u$. \medskip

\begin{figure}[!ht]
  \centering
  \includegraphics[width=0.8\textwidth]{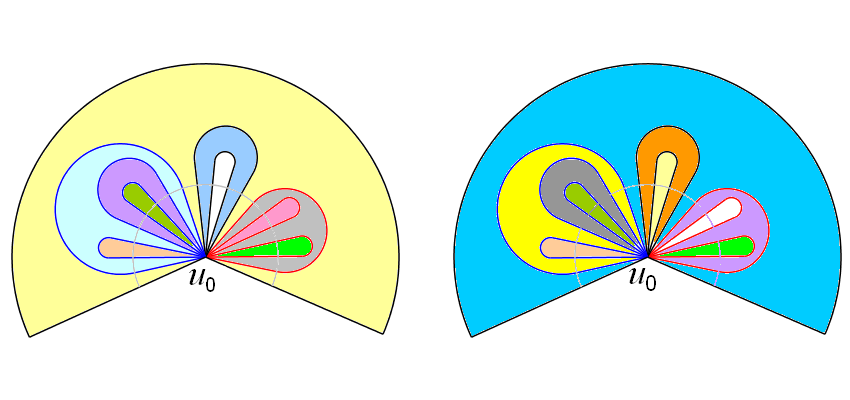}
  \caption{Same nodal pattern, different labeling sets, $\sigma = 7$}\label{F-hmn2L-6a}
\end{figure}

\begin{figure}[!ht]
  \centering
  \includegraphics[width=0.8\textwidth]{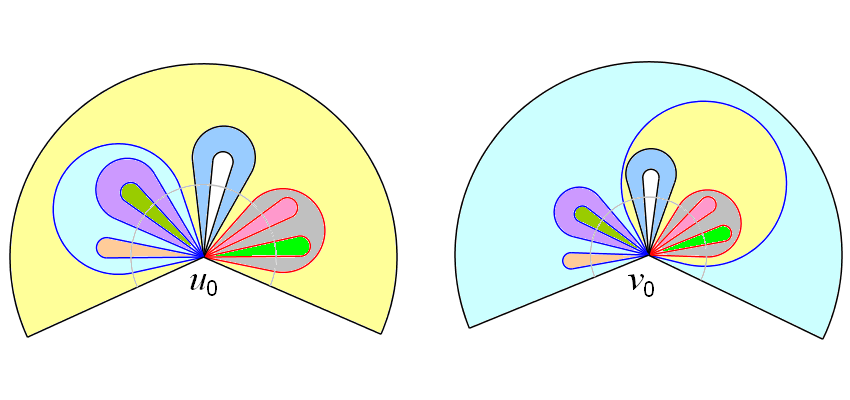}
  \caption{$\sigma(\ww_{u_0,\cD_1})=7$ (left), $\sigma(\ww_{v_0,\cD_2})=13$ (right)}\label{F-hmn2L-6b}
\end{figure}
\FloatBarrier

\subsection{A general description of nodal patterns}\label{SS-hmn2L-gdnp}~\\
 Let $u$ be an eigenfunction satisfying Assumptions~\ref{A-hmn2L-2}. Let $\tau$ be its combinatorial type.\smallskip

\fbox{A}~\emph{Sub-bouquets.}~ Given a subset $F \subset \set{1,\ldots,q}$, define $b_F := \min F$ and $e_F := \max F$. In the sequel, we only consider subsets $F$ with the following property:
\begin{equation}\label{E-hmn2L-2}
F = \set{j \mid b_F \le j \le e_F} \text{~and~} \tau(F)=F,
\end{equation}
i.e., $F$ is an interval in $\set{1,\ldots,q}$, and $\tau$ leaves $F$ globally invariant.  If $q = (2k-3)$ and $\tau(a) =\, \downarrow$ for some $a \in \set{1,\ldots, q}$, we also assume that $a \not \in F$.\smallskip

With such a subset $F$ we associate the bouquet of loops $\cB^{u}_F$,
\begin{equation*}
\cB^{u}_F = \bigcup_{j \in F} \gamma^{u}_{j,\tau(j)},
\end{equation*}
more precisely, the wedge sum at $y$ of the loops in $\cZ(u)$ associated with $F$.
\smallskip

\begin{definitions}\label{D-hmn2L-6}\phantom{}~
\begin{enumerate}[(i)]
  \item A loop $\gamma$ in $\cZ(u)$, at the point $y$, taken individually, divides $\Omega$ into two connected components. The \emph{interior} of $\gamma$ \index{Interior of a loop} is the component which only touches the boundary $\Gamma$ at $y$. The other component is called the \emph{exterior} of $\gamma$.\index{Exterior of a loop}
  \item Given $\cB^{u}_F$, the bouquet of loops associated with $u$ and $F$, we call \emph{interior} domain of $\cB^{u}_F$ a nodal domain of $u$ contained in the interior of some loop $\gamma_{j,\tau(j)}, j \in  F$. We call \emph{exterior} of $\cB^{u}_F$ the set of points of $\Omega$ which belong neither to $\cB^{u}_F$, nor to an interior domain of $\cB^{u}_F$.
  \item We denote by $n^{u}_F$ the number of loops in $\cB^{u}_F$.
\end{enumerate}
\end{definitions}%

\begin{figure}[!ht]
  \centering
  \includegraphics[width=0.9\textwidth]{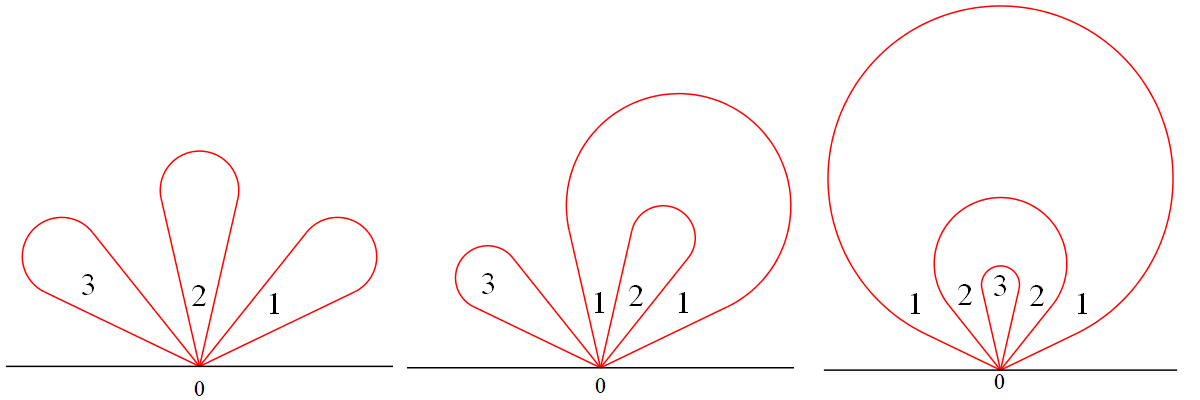}
  \caption{Examples of bouquets}\label{F-hmn2L-L33c-sb}
\end{figure}

Figure~\ref{F-hmn2L-L33c-sb} displays bouquets with $3$ loops:
\begin{itemize}
\item The numbers are labels for the interior nodal domains.
\item  The unlabeled domain is the exterior of the bouquet.
\end{itemize}

The following relations hold.
\begin{equation}\label{E-hmn2L-2a}
\left\{
\begin{array}{l}
\#(F) = e_F - b_F + 1.\\[5pt]
\|F\| := e_F - b_F = \#\set{i \mid I_i(r) \subset \cA\big( A^{u}_{b_F}(r),A^{u}_{e_F}(r)
\big)}.\\[5pt]
n_F = \# \set{\text{~interior nodal domains of~} \cB^{u}_F~}.\\[5pt]
2 \, n_F = \#(F) = \|F\|+1.
\end{array}%
\right.
\end{equation}
Here, $\cA\big( A^{u}_{b_F}(r),A^{u}_{e_F}(r)\big)$ denotes the arc from $A^{u}_{b_F}(r)$ to $A^{u}_{e_F}(r)$, moving counter-clockwise along  $C_{+}(y,r)$. The intervals contained in this arc are $I^{u}_{b_F+1}, \ldots, I^{u}_{e_F}$. Since $\tau(F)=F$, the number $\#(F)$ is even.  Note that $n_F$ depends on $\tau$, not on $u$ itself. \medskip

\fbox{B} \emph{The case $q = (2k-3)$.~} Let $u_1$ be an eigenfunction satisfying Assumptions~\ref{A-hmn2L-2}, with $\rho(u_1,y) = (2k-3)$. Let $\tau_1$  denote its combinatorial type. Let $b := \tau_1(\downarrow)$. Let $\delta = \delta^{u_1}_{b,y,z}$ be the nodal interval which emanates from $y$ tangentially to the ray $\omega_b$ and hits $\Gamma$ at a point $z \neq y$. It divides $\Omega$ into two connected (and simply-connected) components $\Omega_{b,R}$ and $\Omega_{b,L}$, resp. on the right and on the left of $\delta$. A nodal domain $D$ of $u_1$ must be contained in either $\Omega_{b,R}$ or $\Omega_{b,L}$. Define the subsets $R := \set{1,\ldots,(b-1)}$ and $L:=\set{(b+1),\ldots,(2k-3)}$. Assumptions~\ref{A-hmn2L-2} and Jordan's theorem imply that $\tau_1(R) = R$ and $\tau_1(L) = L$. It follows that $\#(R)$ and $\#(L)$ are even, and that $b$ is odd. The corresponding bouquets of loops $\cB^{u_1}_{R}$ and  $\cB^{u_1}_{L}$ are contained respectively in $\set{y}\cup \Omega_{b,R}$, resp. $\set{y}\cup \Omega_{b,L}$. It follows that $\Omega_{b,R}$ contains $k_R := n_R + 1$ nodal domains of $u_1$, and that  $\Omega_{b,L}$ contains $k_L := n_L + 1$ nodal domains of $u_1$. Here, $n_R := \frac{b-1}{2}$ and $n_L = \frac{2k-b-3}{2}$, so that $k_R+k_L = k$. Figure~\ref{F-hmn2L-2}\,(right) displays an example with $k=10, b=11$, $R = \set{1,\ldots,10}$ and $L=\set{12,\ldots,17}$.  The bouquet $\cB^{u_1}_R$ consists of the two black and three red loops; the bouquet $\cB^{u_1}_L$ of the three blue loops (the colors of the arcs are labeled according to Footnote~\ref{FN-color} in this chapter). \smallskip

To label the $k$ nodal domains of $u_1$, it suffices to first label the nodal domains contained in $\Omega_{b,R}$, from $d_1$ to $d_{k_R}$, and then the nodal domains contained in $\Omega_{b,L}$, from $d_{k_R+1}$ to $d_{k}$.\smallskip

Since both $\Omega_{b,R}$ and $\Omega_{b,L}$ are simply connected and only contain loops, we are reduced to labeling nodal domains for functions such that $\rho(u,y)$ is even. \medskip

\fbox{C} \emph{The case $q=(2k-2)$. An example with $k=10$, $q=18$.~} Decompose the nodal set $\cZ(u_0)$ displayed in Figure~\ref{F-hmn2L-2}\, (left),  into a large (red) loop $\gamma^{u_0}_{1,6}$ which contains the (red) bouquet $\cB^{u_0}_R$, a large (blue) loop $\gamma^{u_0}_{11,18}$ which contains the (blue) bouquet $\cB^{u_0}_L$, and a (black) bouquet $\cB^{u_0}_C$. More precisely, for this example,
\begin{equation}\label{E-hmn2L-4h}
\left\{
\begin{array}{lll}
 a=6, ~\tau_{u_0}(1)=6, & &
b=11, ~\tau_{u_0}(11)=18,\\[5pt]
R = \set{2,3,4,5}, & C = \set{7,8,9,10}, & L = \set{12,13,14,15,16,17}.
\end{array}%
\right.
\end{equation}
Similarly, in Figure~\ref{F-hmn2L-4}\, (right), $k=10$ and $q=18$, and we decompose the nodal set into a large (blue) loop  $\gamma^{v_0}_{1,12}$ whose interior contains the (red) bouquet $\cB^{v_0}_{R'}$ and the (black) bouquet $\cB^{v_0}_{N'}$, and the (blue) bouquet  $\cB^{v_0}_{L'}$ contained in the exterior of $\gamma^{v_0}_{1,12}$. More precisely, for this example,
\begin{equation}\label{E-hmn2L-4y}
\left\{
\begin{array}{lll}
& \tau_{v_0}(1)=12, &\\[5pt]
R' = \set{2,3,\ldots,6,7}, & N' = \set{8,9,10,11}, & L' = \set{13,14,\ldots,17,18}.
\end{array}
\right.
\end{equation}

\begin{remark}\label{R-hmn2L-2}
In the example $u_0$, we pay a special attention to the loops one of whose semi-tangents is $\omega_1$ or $\omega_{(2k-2)}$. The reason is  explained
in Section~\ref{S-hmn3}.
\end{remark}%

\fbox{D} \emph{The general case $k\ge 3$, $q=(2k-2)$.~} Let $u_0$ be an eigenfunction satisfying Assumptions~\ref{A-hmn2L-2}, with $\rho(u_0,y) = (2k-2)$ and combinatorial type $\tau_0$. Define $a=\tau_0(1)$ and $b=\tau_0(2k-2)$. We observe that $a$ is even and $b$ odd.\medskip

\noid When $a=(2k-2)$, $b=1$. The nodal set $\cZ(u_0)$ decomposes into the loop $\gamma_{1,(2k-2)}$ and the bouquet $\cB^{u_0}_R$ with $R = \set{2,\ldots,(2k-3)}$. The exterior of $\gamma_{1,(2k-2)}$ is a nodal domain of $u_0$. The bouquet $\cB^{u_0}_R$ is contained in the interior of $\gamma_{1,(2k-2)}$ which contains $(n_R + 1)$ nodal domains, with $n_R = (k-2)$.\medskip

\noid If $a \neq (2k-2)$, then $2 \le a \le (2k-4)$ and $(a+1) \le b \le (2k-3)$.
In this case, $\cZ(u_0)$ consists of two loops $\gamma^{u_0}_{1,a}$, $\gamma^{u_0}_{b,(2k-2)}$, and three bouquets of loops $\cB^{u_0}_R, \cB^{u_0}_C$ and $\cB^{u_0}_L$. The subsets $R, C$ and $L$ are given by
\begin{equation}\label{E-hmn2L-4z}
\left\{
\begin{array}{l}
R = \set{2,\ldots,(a-1)}, \quad C = \set{(a+1),\ldots,(b-1)},\\[5pt]
L = \set{(b+1),\ldots,(2k-3)}.
\end{array}%
\right.
\end{equation}
The combinatorial type $\tau_{u_0}$ is given by
\begin{equation}\label{E-hmn2L-6}
\tau_{u_0} = \begin{pmatrix}
                1 & R & a & C & b & L & q \\
                a & \tau_{u_0}(R) & 1 &\tau_{u_0}(C) &  q & \tau_{u_0}(L) & b \\
              \end{pmatrix}.
\end{equation}

\begin{remark}\label{R-hmn2L-4}
With the same subsets $R, C, L$, the combinatorial type of the function $u_1$, whose nodal pattern is displayed in Figure~\ref{F-hmn2L-2}, is
\begin{equation}\label{E-hmn2L-6u1}
\tau_{u_1} = \begin{pmatrix}
                1 & R & a & C & b & L & \downarrow\\
                a & \tau_{u_1}(R) & 1 &\tau_{u_1}(C) & \downarrow & \tau_{u_1}(L) & b \\
              \end{pmatrix},
\end{equation}
where the symbol $\downarrow$ indicates that the nodal interval $\delta^{u_1}_{b,y,z}$ emanating from $y$ tangentially to the ray $\omega_b$ hits the boundary at $z \neq y$.  In Figure~\ref{F-hmn2L-2}, the maps $\tau_{u_0}$ and $\tau_{u_1}$ coincide on the sets $R, C$ and $L$.
\end{remark}%

\subsection{A labeling procedure and the word associated with it}\label{SS-hmn2L-lpw}~\\
We now explain a procedure to label the nodal domains of a function $u$ satisfying Assumptions~\ref{A-hmn2L-2} in the case $q = (2k-2)$, see Figure~\ref{F-hmn2L-2}.\smallskip

Let $\cD = \set{d_1,\ldots,d_k}$ be a set of labels with $k=\kappa(u)$. Since every nodal domain of $u$, intersects $C_{+}(y,r)$, we label the nodal domains from $d_1$ to $d_k$ according to their order of appearance in the intervals $I^{u}_i(r)$, $1 \le i \le (2k-1)$, working counter-clockwise along $C_{+}(y,r)$.

\begin{procedure}\label{P-hmn2L-la}
The following procedure attributes a unique label $d_j, 1 \le j \le \kappa(u)=k$ to each nodal domain of the eigenfunction $u_0$, and a well defined label to each interval $I^{u}_i(r)$ determined by $\cZ(u) \cap C_{+}(y,r)$, $1 \le i \le (q+1)$ on $C_{+}(y,r)$. The labeling is independent of $r$ provided that $r$ is small enough for the local structure theorem to apply to the function $u$ at $y$.
\begin{enumerate}[{\textrm Step~1}]
  \item Let $D_{d_1}$ be the nodal domain which contains the interval $I^{u}_1(r)$. Attach the label $d_1$ to all the intervals $I^{u}_i(r)$ contained in $D_{d_1}$.
  \item Because an eigenfunction changes sign across a nodal arc, the interval $I^{u}_2(r)$ is not contained in $D_{d_1}$. Call $D_{d_2}$ the nodal domain which contains $I^{u}_2(r)$. Attach the label $d_2$ to all the intervals contained in $D_{d_2}$.
  \item For the same reason as in the previous item, the label $d_2$ is not attached to the interval $I^{u}_3(r)$. If $I^{u}_3(r) \subset D_{d_1}$, then the label $d_1$ is already attached to $I^{u}_3(r)$ by step 1. If $I^{u}_3(r) \not \subset D_{d_1}$, let $D_{d_3}$ be the nodal domain which contains $I^{u}_3(r)$, and attach the label $d_3$ to all the intervals $I^{u}_i(r)$ contained in $D_{d_3}$.
  \item Assume that the intervals $I^{u}_i(r)$, $1 \le i \le (j-1)$, have been labeled, using the labels $d_1, \ldots, d_p$ for some $p \ge 2$.
  \begin{enumerate}[$\diamond$]
    \item  If $p=k$, all the nodal domains have been labeled, and all the intervals have received a label as well.
    \item If $p < k$ then,
        \begin{enumerate}[$\centerdot$]
    \item either $I^{u}_j(r)$ has already been labeled because it is contained in an already labeled nodal domain and we proceed to $I^{u}_{j+1}(r)$,
    \item or $I^{u}_j(r)$ has no label attach to it yet. Then, call  $D_{d_{p+1}}$ the nodal domain which contains $I^{u}_j(r)$, and attach the label $d_{p+1}$ to all the intervals $I^{u}_i(r)$ which are contained in $D_{d_{p+1}}$.
        \end{enumerate}
  \end{enumerate}
\end{enumerate}
After at most $q$ steps all nodal domains and all intervals will be labeled.\smallskip

The labeling of the $(q+1)$ intervals $I^{u}_1(r), \ldots I^{u}_{(q+1)}(r)$ produces a word $\ww_{u}$ of length $\|\ww_u\|=(q+1)$ (the number of intervals), in the $\kappa(u)$ letters $d_1, \ldots, d_k$.  Equivalently, we can view the labeling as a map $\Lambda_{u}: \set{1,\ldots, (q+1)} \to \cD$.
\end{procedure}%

\begin{remark}\label{R-hmn2L-6}
The above procedure applies to both types of functions, $u_0$ with  $q = (2k-2)$, or $u_1$ with $q=(2k-3)$. There are two differences in the output words. The word $\ww_{u_0}$ has length $(2k-1)$, it begins and ends with the letter $d_1$.  The word $\ww_{u_1}$ has length $(2k-2)$, and consists of two words with no common letter. This is due to the fact that the nodal interval $\delta^{u_1}_{b,y,z}$ divides the domain $\Omega$ into two components, $\Omega_{b,R}$, $\Omega_{b,L}$, so that the nodal domains of $u$ are divided into two distinct families. The Procedure~\ref{P-hmn2L-la} will first label the $(n_R+1)$ nodal domains contained in $\Omega_{b,R}$, from $1$ to $(n_R+1)$, and then label the $(n_L+1)$ nodal domains contained in $\Omega_{b,L}$, from $(n_R+2)$ to $k$.
\end{remark}

\newpage
\subsection{Examples}\label{SS-hmn2L-ex}\phantom{xx}
\smallskip

\subsubsection{Examples with $k=4$, and $a=3$ or $5$, in Paragraph~\ref{SSS-hmn-25D}}\label{SSS-hmn2L-25D}

The labeling of nodal domains in Figures~\ref{F-hmn2-Uxyw-L1} and \ref{F-hmn2-Uxyw-L2}, left and middle sub-figures, follows Procedure~\ref{P-hmn2L-la}. The labeling in the right sub-figures follows the deformation of the nodal domains when $\theta$ tends to $0$. In the first figure, the signatures are $\sigma(\text{left}) = 3$, $\sigma(\text{right}) = 5$; in the second figure, $\sigma(\text{left}) = 3$, $\sigma(\text{right}) = 7$.
\smallskip

\subsubsection{An example with $k=4$ and $a=3$ in the proof of Lemma~\ref{L-L33c}\,(ii)}\label{SSS-hmn2L-exii}

The labeling of nodal domains in Figure~\ref{F-hmn3-L33c-4} is deduced by deformation from the labeling of the nodal domains in the middle subfigure of Figure~\ref{F-hmn3-L33c-2} which follows Procedure~\ref{P-hmn2L-la}. We have $\sigma(v_{0,R})=3$ and $\sigma(v_{0,L})=5$.
\smallskip

\subsubsection{Examples with $k=10$ in the proof of Lemma~\ref{L-L33c}\,(iv)}\label{SS-hmn2L-exk10}\phantom{}

\vspace{2pt}~\\
\fbox{A}~Example \eqref{E-hmn2L-4h}, $k=10, q = 18$, Figure~\ref{F-hmn2L-2}, function $u_0$.

\begin{enumerate}
  \item The interval $I^{u_0}_1(r)$ is contained in the exterior domain of $\cZ(u_0)$. We call $D_{d_1}$ the exterior domain of $\cZ(u_0)$, and we label $d_1$ the intervals $I^{u_0}_{j}(r)$ with $j \in \set{1,7,11,19}$.
  \item The interval $I^{u_0}_2(r)$ is contained in the interior of the loop $\gamma^{u_0}_{1,6}$. We call $D_{d_2}$, the nodal domain in the interior $\gamma^{u_0}_{1,6}$, which contains points close to the loop, and we label $d_2$ the intervals $I^{u_0}_{j}(r)$, $j \in \set{2,4,6}$.
  \item The interval $I^{u_0}_3(r)$ is contained in the interior of the loop $\gamma^{u_0}_{2,3}$. We label both $d_3$.
  \item The interval $I^{u_0}_4(r)$ is already labeled $d_2$. The interval $I^{u_0}_5(r)$ is contained in the interior of the loop $\gamma^{u_0}_{4,5}$. We label both $d_4$.
  \item The intervals $I^{u_0}_6(r)$ and $I^{u_0}_7(r)$ are already labeled, respectively $d_2$ and $d_1$. The interval $I^{u_0}_8(r)$ is contained in the interior of the loop $\gamma^{u_0}_{7,10}$. We label both $d_5$, as well as $I^{u_0}_{10}(r)$.
  \item The interval $I^{u_0}_9(r)$ is contained in the interior of the loop $\gamma^{u_0}_{8,9}$. We label both $d_6$.
  \item Continuing the procedure, we obtain
\begin{equation*}
\ww_{u_0} = d_1d_2d_3d_2d_4d_2d_1d_5d_6d_5
d_1d_7d_8d_9d_8d_7d_{10}d_7d_1.
\end{equation*}
\end{enumerate}
For simplicity, we use the alternative notation
\begin{equation*}
\ww_{u_0} = |1|2|3|2|4|2|1|5|6|5|1|7|8|9|8|7|10|7|1|
\end{equation*}
in which we have only written the indices of the labels, separated by a vertical bar.\medskip

For this example, we have
\begin{equation}\label{E-hmn2L-10A}
\ww_{u_0} = |1|2|3|2|4|2|1|5|6|5|1|7|8|9|8|7|10|7|1|
\text{~~with~~} \sigma(\ww_{u_0}) = 7.
\end{equation}

\begin{figure}[!ht]
  \centering
  \includegraphics[width=0.9\textwidth]{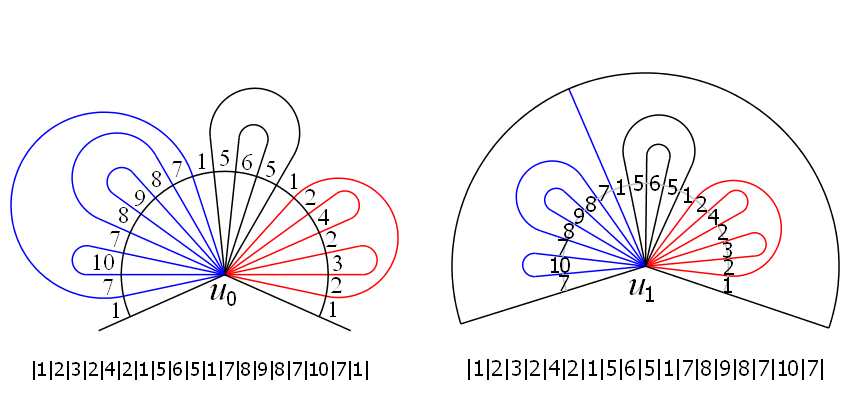}
  \caption{The words $\ww_{u_0}$ and $\ww_{u_1}$ [here $k=10$]}\label{F-hmn2L-10a}
\end{figure}

\fbox{B}~Example \eqref{E-hmn2L-4h}, $k=10, q = 17$, Figure~\ref{F-hmn2L-2}, function $u_1$. Applying Procedure~\ref{P-hmn2L-la}, we obtain
\begin{equation}\label{E-hmn2L-10B}
\ww_{u_1} = |1|2|3|2|4|2|1|5|6|5|1|7|8|9|8|7|10|7|.
\end{equation}
Note that the word $\ww_{u_1} = \ww_{1,R}\ww_{1,L}$ is the juxtaposition of two disjoint words, namely $\ww_{1,R}=|1|2|3|2|4|2|1|5|6|5|1|$ and $\ww_{1,L}=|7|8|9|8|7|10|7|$ which correspond respectively to the nodal sets contained in $\Omega_{11,R}$ and $\Omega_{11,L}$.\medskip

\fbox{C}~Example \eqref{E-hmn2L-4y}, $k=10, q = 18$, Figure~\ref{F-hmn2L-4}, function $v_0$. Applying Procedure~\ref{P-hmn2L-la}, we obtain
\begin{equation}\label{E-hmn2L-10C}
\ww_{v_0} = |1|2|3|4|3|5|3|2|6|7|6|2|1|8|9|8|1|10|1|  \text{~~with~~}
\sigma(\ww_{v_0}) = 13.
\end{equation}

\begin{figure}[!ht]
  \centering
  \includegraphics[width=0.9\textwidth]{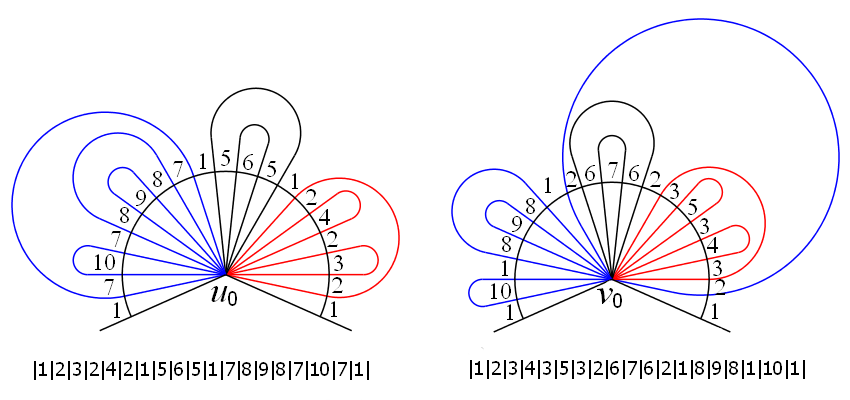}
  \caption{Words $\ww_{u_0}$ and $\ww_{v_0}$, with $7$ and $13$ [$k=10$]}\label{F-hmn2L-10b}
\end{figure}

\FloatBarrier

\subsection{Applying the labeling procedure in the general case}\label{SS-hmn2L-gen}

The nodal set $\cZ(u_0)$ of an eigenfunction $u_0$ such that $\rho(u_0,y) = (2k-2)$ is a $(k-1)$ bouquet of loops at $y$. Let $\tau_0$ be the combinatorial type of $u_0$. \smallskip

\noid If $\tau_0(1) = (2k-2)$, we decompose $\cZ(u_0)$ into the loop $\gamma^{u_0}_{1,(2k-2)}$ and the bouquet of loops $\cB_R$ which is contained in the interior of $\gamma^{u_0}_{1,(2k-2)}$, where $R:=\set{2,\ldots,(2k-3)}$, see paragraph \fbox{A} below.
\begin{figure}[!ht]
  \centering
  \includegraphics[width=0.5\textwidth]{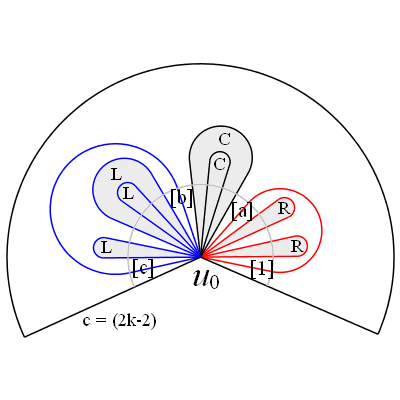}
  \caption{The decomposition of $\cZ(u_0)$ when $\tau_0(1) \neq (2k-2)$}\label{F-hmn2L-14}
\end{figure}

\noid If $a:=\tau_0(1) \neq (2k-2)$, then we define $b:=\tau_0(2k-2)$, and we decompose $\cZ(u_0)$ into the loops $\gamma^{u_0}_{1,a}$, $\gamma^{u_0}_{b,(2k-2)}$, and the bouquets of loops associated with the subsets $R,C,L$ given in \eqref{E-hmn2L-4z}. This decomposition is illustrated in Figure~\ref{F-hmn2L-14}, with the bouquets $R,C$ and $L$ represented in grey shade. Recall that this decomposition implies that $a$ is even and $b$ odd. Using Definitions~\ref{D-hmn2L-6} and Equation~\eqref{E-hmn2L-2a}, we have that $(n_R+n_C+n_L+3) = k$, the number of nodal domains of $u_{0}$. The labeling of nodal domains for this decomposition is given in paragraph \fbox{B} below. \smallskip

\fbox{A}~ \emph{The case $\tau_0(1) := a=(2k-2)$}.~ Call $w_0$ a corresponding eigenfunction. The exterior of the loop $\gamma_{1,(2k-2)}$ is a nodal domain of $w_0$, which we call $D_{d_1}$ and we label $d_1$ the intervals $I^{w_0}_1(r)$ and $I^{w_0}_{(2k-1)}(r)$. No other interval has the label $d_1$. The nodal domain $D_{d_2}$ contains the intervals $I^{w_0}_2(r)$ and $I^{w_0}_{(2k-2)}(r)$ and, possibly, other intervals. This nodal domain is the intersection of the interior of $\gamma_{1,(2k-2)}$ with the exterior of the bouquet $\cB^{w_0}_R$ associated with $R = \set{2,\ldots,(2k-3)}$. There are $(k-2)$ interior domains of $\cB^{w_0}_R$, labeled $D_{d_3}, \ldots, D_{d_k}$. The associated word is $\ww_{w_0} = d_1d_2\ww_Rd_2d_1$, where $\ww_R$ is the word associated with $R$, with length $\|\ww_R\| = (2k-5)$, in the letters $d_3,\ldots,d_k$, and possibly the letter $d_2$. The word $\ww_R$ does not contain the letter $d_1$. It follows that the signature of the nodal pattern of $w_0$ is $(2k-1)$.\medskip

\fbox{B}~ \emph{The case ~$2 \le \tau_0(1) := a \le (2k-4)$}.~ Call $u_0$ a corresponding function, with combinatorial type $\tau_0$. In this case, $(a+1) \le b \le (2k-3)$. \smallskip

\begin{enumerate}[1)]
\item The  nodal domains of $u_{0}$ are split into four disjoint families. The following description takes Procedure~\ref{P-hmn2L-la} and the relation $k = (n_R+n_C+n_L+3)$ into account.
\begin{enumerate}[$\diamond$]
  \item The exterior of $\cZ(u_0)$. It is called $D_{d_1}$.
  \item The $(n_R+1)$ nodal domains contained in the interior of $\gamma^{u_0}_{1,a}$. This family consists of the $n_R$ interior domains of $\cB^{u_0}_R$, $D_{d_3}$ to $D_{d_{(n_R+2)}}$, and the domain $D_{d_2}$ which is the complement of $D_{d_3}\cup \ldots \cup D_{d_{(n_R+2)}}$ in the interior of $\gamma^{u_0}_{1,a}$. The domain $D_{d_2}$ contains the interval $I^{u_0}_2(r)$, and its boundary contains the loop $\gamma^{u_0}_{1,a}$ itself.
  \item The $n_C$ interior domains of $\cB^{u_0}_C$. They are labeled from $d_{(n_R+3)}$ to $d_{(n_R+n_C+2)}$.
  \item The $(n_L+1)$ nodal domains contained in the interior of $\gamma^{u_0}_{b,(2k-2)}$. This family consists of the $n_L$ interior domains of $\cB^{u_0}_L$, labeled from $d_{(n_R+n_C+4)}$ to $d_k$, and the domain $D_{d_{(n_R+n_C+3)}}$ which is the complement of $D_{d_{n_R+n_C+4}}\cup \ldots \cup D_{d_k}$ in the interior of $\gamma^{u_0}_{b,(2k-2)}$. The domain $D_{d_{(n_R+n_C+3)}}$ contains the interval $I^{u_0}_{b+1}(r)$, and its boundary contains the loop $\gamma^{u_0}_{b,(2k-2)}$ itself.\medskip
\end{enumerate}
\item The $(2k-1)$ intervals $I^{u_0}_j(r)$ determined by $\cZ(u_{0}) \cap C_{+}(\eta,r)$ are as follows.
\begin{enumerate}[$\diamond$]
  \item The exterior domain $D_{d_1}$ contains four intervals $I^{u_0}_j(r)$, with $j$ in the set $\set{1,(a+1),b,(2k-1)}$.
  \item The interior of the loop $\gamma^{u_0}_{1,a}$ contains $(a-1)$ intervals $I^{u_0}_j(r)$, $j \in \set{2,\ldots,a}$. The bouquet $\cB^{u_0}_R$ determines $(a-3)$ intervals, and there are two intervals touching the loop $\gamma^{u_0}_{1,a}$.
  \item The bouquet $\cB^{u_0}_C$ determines $(b-a-2)$ intervals $I^{u_0}_j(r)$, with $j$ in the set $ \set{(a+2),\ldots,(b-1)}$.
  \item The interior of the loop $\gamma^{u_0}_{b,(2k-2)}$ contains $(2k-b-2)$ intervals $I^{u_0}_j(r)$, with the index $j$ in the set $\set{(b+1),\ldots,(2k-2)}$. The bouquet $\cB^{u_0}_L$ determines $(2k-b-4)$ intervals, and there are two intervals touching the loop $\gamma^{u_0}_{b,(2k-2)}$.\medskip
\end{enumerate}
\item  Applying Procedure~\ref{P-hmn2L-la}, the nodal domains of $u_{0}$ are described by the word
\begin{equation}\label{E-hmn2L-14}
\ww_{u_0} = |1|2|\ww^{u_0}_R|2|1|\ww^{u_0}_C|1|\hat{b}|\ww^{u_0}_L|\hat{b}|1|.
\end{equation}
Here $\hat{b} := (n_R+n_C+3)$, and $d_{\hat{b}}$ is the label of the nodal domain contained in the interior of the loop $\gamma^{u_0}_{b,(2k-2)}$ and whose boundary contains the loop itself; the words $\ww^{u_0}_R$, resp. $\ww^{u_0}_C$ and $\ww^{u_0}_L$, describe the nodal domains containing the intervals determined by $\cB^{u_0}_R$, resp. $\cB^{u_0}_C$ and $\cB^{u_0}_L$. More precisely,
\begin{enumerate}[$\diamond$]
  \item $\ww^{u_0}_R$ is a word of length $\|\ww^{u_0}_R\| = (a-3)$, in the letters $d_3, \ldots, d_{(n_R+2)}$ and, possibly the letter $d_2$
  \item $\ww^{u_0}_C$ is a word of length $\|\ww^{u_0}_C\| = (b-a-2)$, in the letters $d_{(n_R+3)}$ to $ d_{(n_R+n_C+2)}$ and, possibly the letter $d_1$
  \item $\ww^{u_0}_L$ is a word of length $\|\ww^{u_0}_L\| = (2k-b-4)$, in the letters $d_{(n_R+n_C+4)}$ to $ d_{k}$ and, possibly the letter $d_{\hat{b}}$.\smallskip
\end{enumerate}

\noi We recover the fact that $\|\ww^{u_0}\| = (2k-1) = 8 + \|\ww^{u_0}_R\| + \|\ww^{u_0}_C\| + \|\ww^{u_0}_L\|$.\smallskip

\noi In view of \eqref{E-hmn2L-14}, and the description of the words $\ww^{u_0}_R, \ww^{u_0}_C$ and $\ww^{u_0}_L$, we have
\begin{equation}\label{E-hmn2L-14sa}
\sigma(\ww_{u_0}) = 4 + \|\ww^{u_0}_R\| = (a+1).
\end{equation}
\end{enumerate}
\medskip

\fbox{C}~ ~ \emph{Another example with $q = (2k-2)$}.~ Call $v_0$ an eigenfunction such that such that $\rho(v_0,y) = (2k-2)$ with combinatorial type $\tau_{v_0}$ given by
\begin{equation*}
\tau_{v_0} = \begin{pmatrix}
               0& 1 & R & a & C & b & L \\
               b& a & \tau_{v_0}(R) & 1 &\tau_{v_0}(C) & 0 & \tau_{v_0}(L) \\
              \end{pmatrix},
\end{equation*}
where the sets $R, C, L$ are as in \eqref{E-hmn2L-4z} (the same subsets as in the previous example) and the corresponding bouquets are shaded in grey in the picture. The nodal set $\cZ(v_0)$ is partitioned as follows, see Figure~\ref{F-hmn2L-16}.
\begin{enumerate}[$\diamond$]
  \item The loop $\gamma_{0,b}$.
  \item The bouquet $\cB_{R'}:= \gamma_{1,a}\cup \cB_R \cup \cB_C$ contained in the interior of the loop $\gamma_{0,b}$.
  \item The bouquet $\cB_L$ contained in the exterior of $\gamma_{0,b}$.
\end{enumerate}

\begin{figure}[!ht]
  \centering
  \includegraphics[width=0.6\textwidth]{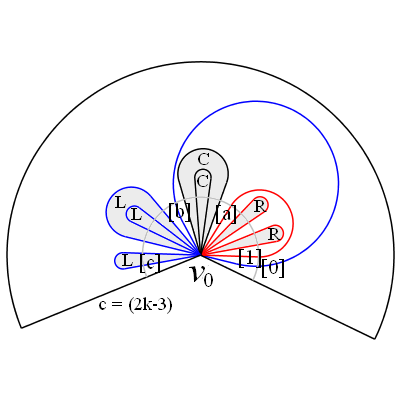}
  \caption{The decomposition of the nodal set $\cZ(v_0)$}\label{F-hmn2L-16}
\end{figure}

The intervals are ordered from $I_1(r)$ to $I_{(2k-2)}(r)$ as usual. The exterior of $\cZ(v_0)$ is called $D_{d_1}$. The nodal domain $D_{d_2}$ contains the interval $I_2(r)$. It also contains an inner neighborhood of $\gamma_{0,b}$ and the interval $I_{(b+1)}(r)$. The interior of $\gamma_{0,b}$ contains the $1+n_R+ n_C$ interior domains of the bouquet $\cB_{R'}$. The word $\ww_{v_0}$ is given by
\begin{equation*}
\ww_{v_0} = |1|2|\ww_{R'}|2|1|\ww_L,
\end{equation*}
where the word $\ww_{R'}$ is associated with $\cB_{R'}$. By \eqref{E-hmn2L-2a}, $\|\ww_{R'}=(b-2)\|$ and $\ww_{R'}$ is a word in the letters $d_3, \ldots, d_{n_R+n_C+2}$ and, possibly, $d_2$. It follows that
\begin{equation}\label{E-hmn2L-14sb}
\sigma(\ww_{v_0}) = (b+2).
\end{equation}
Note that $(b+2) \ge (a+3) > (a+1)$, so that $[v_0] \neq [u_0]$.\medskip

\FloatBarrier
\subsection{The rotating function argument of \S~\ref{SSS-hmn-25D} in the general case}\label{SS-hmn2L-P1}

In Section~\ref{S-hmn2}, under Assumption~\ref{A-hmn2-0}, we studied functions $u_{y,z}$ with $\rho(u_{y,z},y) = (2k-3)$ and $\rho(u_{y,z},z)=1$, and we looked at the limits when $z$ tends to $y$ clockwise or counter-clockwise.\smallskip

The general combinatorial type $\tau$ is
\begin{equation*}
\tau = \begin{pmatrix}
           \downarrow & R & a & L \\
           a & \tau(A) & \downarrow & \tau(B) \\
         \end{pmatrix},
\end{equation*}
where $R = \set{1,\ldots,(a-1)}$ and $L = \set{(a+1),\ldots,(2k-3)}$. Letting $z$ tend to $y$ clockwise or counter-clockwise, we obtain the combinatorial types $\tau_R$ and $\tau_L$ given by
\begin{equation*}
\tau_R = \begin{pmatrix}
           0 & R & a & L \\
           a & \tau(R) & 0 & \tau(L) \\
         \end{pmatrix}
\end{equation*}
and
\begin{equation*}
\tau_L = \begin{pmatrix}
           R & a & L & (2k-2)\\
           \tau(R) & (2k-2) & \tau(L) & a \\
         \end{pmatrix}.
\end{equation*}
Using Procedure~\ref{P-hmn2L-la}, it is easy to show that
\begin{equation*}
\sigma_L = 4 + \|\ww_R\| \text{~~and~~} \sigma_R \le 2 + \|\ww_R\|,
\end{equation*}
where $\ww_R$ is the word describing $\cB_R$. It follows that the combinatorial types $\tau_R$ and $\tau_L$ are different.\smallskip

This proves that the rotating function argument in Paragraph~\ref{SSS-hmn-25D} yields a contradiction in the general case. This finishes the proof of Proposition~\ref{P-hmn-s2}.

\subsection{Proof of Lemma~\ref{L-L33c}, Assertion~(ii), general case}\label{SS-hmn2L-P2ii}

Let $\eta \in \Gamma_{(2k-2)}$ (assuming this set in not empty). Assume, by contradiction, that the functions $u_y$ have the same combinatorial type $\tau$ for $y$ close enough to $\eta$, on either sides of $\eta$.\smallskip

Working in $\bH$, according to Lemma~\ref{L-L33b}, the general combinatorial type is given by
\begin{equation*}
\tau = \begin{pmatrix}
           \downarrow & R & a & L \\
           a & \tau(R) & \downarrow & \tau(L) \\
         \end{pmatrix},
\end{equation*}
where $R = \set{1,\ldots,(a-1)}$ and $L = \set{(a+1),\ldots,(2k-3)}$, and the nodal patterns are as follows depending on whether $t$ is on the right or on the left of $0$.
\begin{figure}[!ht]
  \centering
  \includegraphics[width=0.99\textwidth]{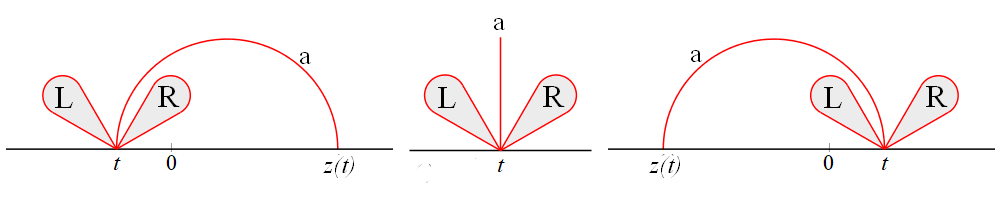}
  \caption{Nodal patterns with the same $\tau$}\label{F-hmn2L-L33c-6-LR}
\end{figure}

When $t$ tends to zero, the limit functions have the following combinatorial types
\begin{equation*}
\tau_L = \begin{pmatrix}
           0 & R & a & L \\
           a & \tau(R) & 0 & \tau(L) \\
         \end{pmatrix}
\end{equation*}
and
\begin{equation*}
\tau_R = \begin{pmatrix}
           R & a & L & (2k-2)\\
           \tau(R) & (2k-2) & \tau(L) & a \\
         \end{pmatrix}.
\end{equation*}
and nodal patterns as shown in Figure~\ref{F-hmn2L-L33c-8-LR}.
\begin{figure}[!ht]
  \centering
  \includegraphics[width=0.98\textwidth]{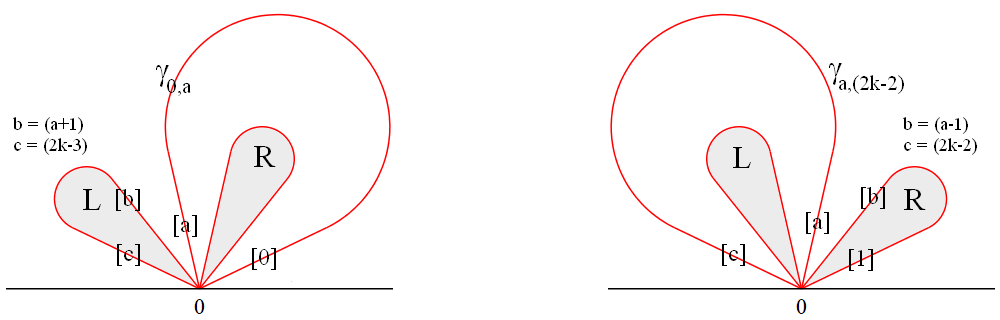}
  \caption{Nodal patterns for $u_L$ and $u_R$}\label{F-hmn2L-L33c-8-LR}
\end{figure}

Using Procedure~\ref{P-hmn2L-la}, it is easy to show that
\begin{equation*}
\sigma_L = 4 + \|\ww_R\| \text{~~and~~} \sigma_R \le 2 + \|\ww_R\|,
\end{equation*}
where $\ww_R$ is the word describing $\cB_R$. It follows that the combinatorial types $\tau_R$ and $\tau_L$ belong to different eigenfunctions.\hfill \qed

\FloatBarrier
\subsection{Proof of Lemma~\ref{L-L33c}, Assertion~(iv), general case}\label{SS-hmn2L-P2iv}

Assuming it is not empty, let $\eta_1, \eta_2$ be two successive points of $\Gamma_{(2k-2)}$, i.e., points such that the arc $\cA(\eta_1,\eta_2)$ is contained in $\Gamma_{(2k-3)}$. We want to prove that the combinatorial types of $u_{\eta_1}$ and $u_{\eta_2}$ are different. For this purpose, we choose the general pattern described in Subsection~\ref{SS-hmn2L-gen}, part \fbox{B}, see Figure~\ref{F-hmn2L-14}. In view of Lemma~\ref{L-L33b}, when the point $y$ moves off $\eta_1$ to the right (i.e., counter-clockwise on $\Gamma$), the loop $\gamma^{u_{\eta_1}}_{b,(2k-2)}$ opens up to become a nodal interval from $y$ to $z(y)$, emanating from $y$ tangentially to the ray $\omega_b$. When $y$ approaches $\eta_2$ from the left, this nodal interval closes in into the loop $\gamma^{u_{\eta_2}}_{0,b}$ and $u_{\eta_2}$ has the nodal pattern of $v_0$ (as in Figure~\ref{F-hmn2-w-theta0}), see Subsection~\ref{SS-hmn2L-gen}, part \fbox{C}. The transition from $\cZ(u_{\eta_1})$ to $\cZ(u_{\eta_2})$ is illustrated in Figure~\ref{F-hmn2L-L33c-iv}. According to \eqref{E-hmn2L-14sa}
and \eqref{E-hmn2L-14sb}, we have $\sigma(\ww_{\eta_1}) \neq \sigma(\ww_{\eta_2})$, showing that the nodal patterns are different. This proves Lemma~\ref{L-L33c}, Assertion~(iv), in the general case. \hfill \qed \smallskip

\begin{figure}[!ht]
  \centering
  \includegraphics[width=0.98\textwidth]{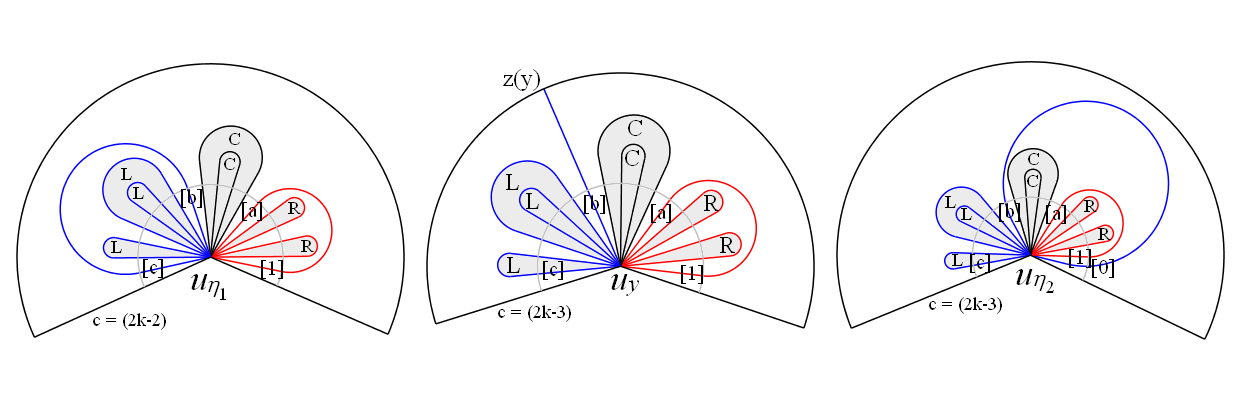}
  \caption{Nodal patterns for $u_{\eta_1}, u_y$ and $u_{\eta_2}$}\label{F-hmn2L-L33c-iv}
\end{figure}
\FloatBarrier

\section[Eigenfunctions with Two Prescribed Boundary Singular Points]{Eigenfunctions with Two Prescribed Boundary Singular Points: proof of Lemma~\ref{L-L34a} and further results}\label{S-hmn32b}

\subsection{Introduction}
The main purpose of this section is provide a detailed proof of Lemma~\ref{L-L34a}.  We also derive further properties of $V_{y,s}$ which appear in  \cite[pp.~1180-1183]{HoMN1999}, see Table~\ref{T-comp}.
We do not use these properties for our proof of Theorem~\ref{T-hmn-bh2}.\smallskip

We retain the notation  of  Section~\ref{S-hmn3},  in particular Notation~\ref{N-hmn3-arcs}.

We work under Assumptions~\ref{A-hmn3-0}, i.e. assuming that
\begin{equation}\label{E-hmn-bs}
\left\{
\begin{array}{l}
\Omega \text{~is simply connected,}~~ \Gamma := \partial \Omega, \\[5pt]
k \ge 3 \text{~and~} \dim U(\lambda_k) = (2k-2).
\end{array}%
\right.
\end{equation}
The sets $\Gamma_{(2k-3)}$ and $\Gamma_{(2k-2)}$ are defined in \eqref{E-hmn3-n12}.

\begin{remark}\label{R-simply}
The assumption $\Omega$ simply connected is motivated by Remark~\ref{R-hmn-sc}, and makes the proofs of the following lemmas simpler. It would be interesting to know whether it is actually necessary.
\end{remark}%

Recall that for  $(y,s) \in \Gamma_{(2k -3)} \times \Gamma $, with $ y \neq s$,
\begin{equation*}
V_{y,s} := \set{u \in U \mid \rho(u,y) \ge 2k-4 \text{~and~} \rho(u,s) \ge 1}.
\end{equation*}

 In view of the Equation~\eqref{E-hmn-bs}, Lemma~\ref{L-zero2} implies that $V_{y,s} \neq \set{0}$. \smallskip

\subsection{Proof of Lemma~\ref{L-L34a}}\label{SSS-hmn32ba}


For convenience, we reproduce the lemma.

\begin{lemma}[Restatement of Lemma~\ref{L-L34a}]\label{RL-L34a}%
Assume that $\Omega$ is simply connected. Let $U:=U(\lambda_k)$ with $k \ge 3$, and assume that $\dim U = (2k-2)$. Given $(y,s) \in \Gamma_{(2k -3)} \times \Gamma $, with $ y \neq s$,  the subspace
\[
V_{y,s} = \set{u \in U \mid \rho(u,y) \ge 2k-4 \text{~and~} \rho(u,s) \ge 1}
\]
has the following properties.
\begin{enumerate}[(i)]
  \item The subspace $V_{y,s}$ has dimension $1$.
  \item Any $0 \neq u \in V_{y,s}$ satisfies
  \begin{equation}\label{RE-L34-4}
  \left\{
  \begin{array}{l}
  \kappa(u) = k,\\[5pt]
  \cZ(u) \cup \Gamma  \text{~is connected,}\\[5pt]
  \cS_{\mathrm{i}}(u) = \emptyset,\\[5pt]
  \sum_{z \in \cS_{\mathrm{b}}(u)} \rho(u,z) = 2k-2, \text{~and}\\[5pt]
  \quad 2k-4 \le \rho(u,y) \le 2k-3,\\[5pt]
  \quad 1 \le \rho(u,s) \le 2.
  \end{array}
  \right.
  \end{equation}
  More precisely, there are three distinct possibilities.
  \begin{description}
    \item[Case~A] $\rho(u,y) = (2k-3)$ and $\rho(u,s) = 1$. In that case, $u \in U_y$, with $\cS_{\mathrm{b}}(u) = \set{y,s}$, and hence $s = z(y)$.
    \item[Case~B] $\rho(u,y) = (2k-4)$ and $\rho(u,s) = 2$. In that case, $\cS_{\mathrm{b}}(u) = \set{y,s}$.
    \item[Case~C] $\rho(u,y) = (2k-4)$, $\rho(u,s) = 1$, and there exists some $s' \in \Gamma  \sm \set{y,s}$ such that $\cS_{\mathrm{b}}(u) = \set{y,s,s'}$, with $\rho(u,s') = 1$.
  \end{description}
  \item If $s = z(y)$, then $V_{y,z(y)} = U_y$.
  \item The map $\set{(y,s) \mid (y,s) \in \Gamma_{(2k-3)}\times \Gamma , s\neq y} \ni (y,s) \mapsto [V_{y,s}] \in \bP(U)$ is $C^{\infty}$.
  \end{enumerate}
\end{lemma}%

\begin{proof}  We already know that $\dim V_{y,s} \ge 1$.\smallskip

\noi We retain the notation of Lemma~\ref{L-L32}. In particular, for $y \in \Gamma_{(2k -3)}$, we have $U_y = [u_y]$ with $0 \neq u_y \in U$ satisfying $\cS_{\mathrm{b}}(u_y) = \set{y,z(y)}$ with $z(y) \neq y$, and $\rho(u_y,y) = (2k-3)$, $\rho(u_y,z(y)) = 1$. \smallskip

\noid \emph{Assertion~(ii).~} Let $u \in V_{y,s}$. From Euler's formula \eqref{E-hmn-3a12} we obtain,
\begin{equation}\label{RE-L34-6}
\begin{split}
0 \geq  \kappa(u) - k & = \big( b_0(\cZ(u) \cup \Gamma ) -1 \big) + \frac 12\, \sum_{z \in \cS_{\mathrm{i}}(u)} (\nu(u,z)-2) \\
&\hspace{5mm} + \frac 12 \big( \sum_{z\in \cS_{\mathrm{b}}(u)}\rho (u,z) - 2k +2\big) .
\end{split}
\end{equation}
If $0 \neq u \in V_{y,s}$, we have $\sum_{z\in \cS_{\mathrm{b}}(u)}\rho (u,z) \ge 2k-3$, and hence $\sum_{z\in \cS_{\mathrm{b}}(u)}\rho (u,z) \ge 2k-2$ since the sum is an even integer,  by Corollary~\ref{cor:nodinfo}. All the terms in the right hand side of \eqref{RE-L34-6} must vanish; this proves \eqref{RE-L34-4}.  Assertion~(ii) then follows from \eqref{RE-L34-4} and the assumptions that $\rho(u,y) \ge (2k-4)$ and $\rho(u,s) \ge 1$.\medskip

\noid \emph{Assertion~(i).}~

\noi \fbox{1}~ We first assume that $s \neq z(y)$. Assume that there are at least two linearly independent functions $u_1,u_2 \in V_{y,s}$, then $\rho(u_i,y) = (2k-4)$ and  $\rho(u_i,s) \ge 1$. According to Lemma~\ref{L-zeroc}, there exists a nontrivial linear combination $u$ of $u_1$ and $u_2$ such that $\rho(u,y) \ge (2k-3)$ and $\rho(u,s) \ge 1$. Euler's formula implies that $u$ pertains to Assertion~(ii), Case~(1), contradicting the fact that $s \neq z(y)$.\smallskip

\noi \fbox{2}~ We now assume that $s = z(y)$. In this case, a generator $u_y$ of $U_y$ belongs to $V_{y,z(y)}$. Assume that $\dim V_{y,z(y)} \ge 2$. Define $V'_{y,z(y)} = V_{y,z(y)} \ominus U_y$, which has dimension at least $1$. If $\dim V_{y,z(y)} \ge 3$, we can find two linearly independent $u_1,u_2 \in V'_{y,z(y)}$, such that $\rho(u_i,y) = (2k-4)$, and $\rho(u_i,s) \ge 1$. By Lemma~\ref{L-zeroc}, there exists a nontrivial linear combination $u \in V'_{y,z(y)}$ such that $\rho(u,y) \ge (2k-3)$ and $\rho(u,s) \ge 1$.  Hence, $u \in U_y$, a contradiction. Assuming that $\dim V_{y,z(y)} = 2$, we can choose a basis $\set{u_y,v_y}$ such that $v_y \not \in U_y$. Then, $\rho(v_y,y) = (2k-4)$, and there are two cases,\smallskip
\begin{description}
  \item[Case~(a):] $\rho(v_y,z(y))=2$,
  \item[Case~(b):] $\rho(v_y,z(y)) = 1$, and there exists some $z_1(y) \in \Gamma $, $z_1(y) \neq z(y)$,  such that $\cS_{\mathrm{b}}(v_y) = \set{y,z(y),z_1(y)}$ and $\rho(v_y,z_1(y)) = 1$.
\end{description}\smallskip

Without loss of generality,  making use of Lemma~\ref{L-breve}, we may choose the functions $u_y$ and $v_y$ as follows  (we consider open arcs).  First we choose $u_y$ so that $\breve{u}_y > 0$ on the arc $\cA(y,z(y))$, and $\breve{u}_y < 0$ on the arc $\cA(z(y),y)$. \smallskip

\noic In Case (a), we choose $v_y$ such that $\breve{v}_y > 0$ on $\cA(y,z(y)) \cup \cA(z(y),y)$.\smallskip

\noic In Case (b), assuming that $z_1(y) \in \cA(y,z(y))$, we choose $v_y$ such that $\breve{v}_y > 0$ on $\cA(z(y),y) \cup \cA(y,z_1(y))$, and $\breve{v}_y < 0$ on $\cA(z_1(y),z(y))$.\smallskip

Figure~\ref{F-hmn3-L34-4} displays the signs of  $\breve{u}_y$ and of $\breve{v}_y$ in both cases. \smallskip

\begin{figure}[!ht]
  \centering
  \includegraphics[scale=0.80]{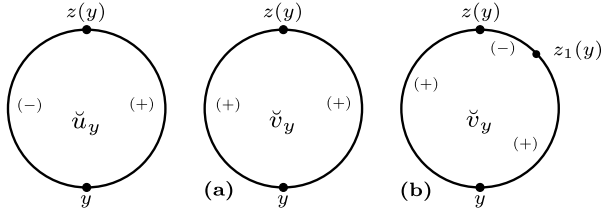}
  \caption{Signs of $\breve{u}_y$ and of $\breve{v}_y$, Cases a and b}\label{F-hmn3-L34-4}
\end{figure}

\emph{Claim~1.~ Under the assumption that $\dim V_{y,z(y)} = 2$, there exists some function $v \in V_{y,z(y)}$ such that $\rho(v,y)=(2k-4)$ and $\rho(v,z(y))=2$.} \smallskip

\emph{Proof of Claim~1.}~ If $v_y$ satisfies Claim~1, there is nothing to prove. If not, $v_y$ falls into Case~(b) above.\smallskip

Given $t \in \Gamma  \sm \set{y, z(y)}$, $\breve{u}_y(t) \neq 0$, and we can define the function
\begin{equation}\label{E-L34-10}
\xi_t := a(t)\, u_y - b(t)\, v_y \in V_{y,z(y)},
\end{equation}
where
\begin{equation}\label{E-L34-10a}
\left\{
\begin{array}{ll}
a(t) & =\breve{v}_y(t)\, \big( \breve{v}^2_y(t) + \breve{u}^2_y(t) \big)^{- \frac 12},\\[5pt]
b(t) & =\breve{u}_y(t) \, \big( \breve{v}^2_y(t) + \breve{u}^2_y(t) \big)^{- \frac 12}.
\end{array}
\right.
\end{equation}

For $t \not \in \set{y, z(y)}$, $b(t) \neq 0$, and hence $\rho(\xi_t,y) = 2k-4$, $\rho(\xi_t,z(y)) \ge 1$, and $\rho(\xi_t,t) \ge 1$. Euler's formula applied to $\xi_t$ implies that  $\rho(\xi_t,z(y)) = \rho(\xi_t,t) = 1$,  and $\cS_{\mathrm{b}}(\xi_t) = \set{y,z(y),t}$. According to Lemma~\ref{L-breve}, the function $\breve{\xi}_t$ has precisely three zeros at $y, z(y)$ and $t$, changes sign at $z(y)$ and $t$, and does not change sign at $y$.  For $t \in \cA(z_1(y),z(y))$, $\breve{\xi}_t(z_1(y)) < 0$, and we conclude that
\begin{equation}\label{E-L34-10b}
\text{for~} t \in \cA(z_1(y),z(y)), \qquad\left\{
\begin{array}{l}
\breve{\xi}_t > 0 \text{~~in~~} \cA(t,z(y)), \text{~and}\\[5pt]
\breve{\xi}_t < 0 \text{~~in~~} \cA(z(y),y) \bigcup \cA(y,t).
\end{array}%
\right.
\end{equation}
where $y, z(y)$ and $z_1(y)$ are as in Figure~\ref{F-hmn3-L34-4} (b).

Choose a sequence $\set{t_n} \subset \cA(z_1(y),z(y))$, with $t_n \to z(y)$. Taking a subsequence if necessary, we may assume that the sequence $\set{\big( a(t_n),b(t_n)\big)}$ converges to some $(a,b) \in \bS^1$, so that the sequence $\set{\xi_{t_n}}$ converges uniformly to the function $\xi$ given by $\xi = a\ u_y - b v_y$. From \eqref{E-L34-10b}, we conclude that $\breve{\xi} \le 0$ on $\Gamma $. Since $\xi \in V_{y,z(y)}$ we have three possibilities,
\begin{enumerate}[(i)]
  \item $\rho(\xi,y) = (2k-3)$ and $\rho(\xi,z(y)) =1$,
  \item $\rho(\xi,y) = (2k-4)$, $\rho(\xi,z(y)) =1$, and $\rho(\xi,z_2)$ for some $z_2 \neq y, z(y)$,
  \item $\rho(\xi,y) = (2k-4)$ and $\rho(\xi,z(y)) =2$.
\end{enumerate}
Since (i) and (ii) are incompatible with $\breve{\xi} \le 0$ on $\Gamma $, we conclude that $\rho(\xi,y) = (2k-4)$ and $\rho(\xi,z(y)) =2$. This proves Claim~1.\quad \qedc \medskip

We now continue  with part \fbox{2} in  the proof of Assertion~(i). In view of Claim~1,  assuming that $\dim V_{y,z(y)}= 2$, we may choose a basis $\set{u_y,v_y}$ of $V_{y,z(y)}$ such that
\begin{equation}\label{E-L34-10d}
\begin{array}{ll}
\left\{
\begin{array}{l}
\rho(u_y,y) = 2k-3,\\[5pt]
\rho(u_y,z(y)) = 1,\\[5pt]
\breve{u}_y|_{\cA(y,z(y))} > 0 \text{~and~} \breve{u}_y|_{\cA(z(y),y)} < 0,\\[5pt]
\end{array}
\right.
\text{~and~} & \left\{
\begin{array}{l}
\rho(v_y,y) = 2k-4,\\[5pt]
\rho(v_y,z(y)) = 2,\\[5pt]
\breve{v}_y|_{\Gamma  \sm \set{y,z(y)}} > 0.
\end{array}
\right.
\end{array}
\end{equation}
Examples of nodal sets of these functions are displayed in Figure~\ref{F-hmn3-L34-8}: on the left $\cZ(u_y)$, on the right $\cZ(v_y)$, with two possible cases.\smallskip

\begin{figure}[!ht]
  \centering
  \includegraphics[scale=0.8]{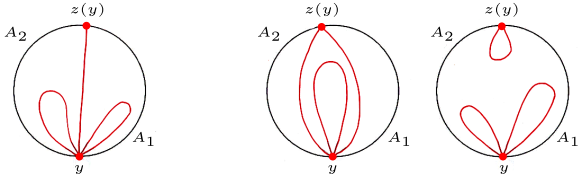}\\
  \caption{Nodal sets of $u_y$ (left) and $v_y$ (right), with $k=4$.}\label{F-hmn3-L34-8}
\end{figure}

From now on, to simplify the notation in the proof,  we denote the arc $\mathcal A(y,z(y))$  by $A_1$, and the arc $\mathcal A(z(y),y)$ by $A_2$.  We now use another ``rotating function argument''. \smallskip

For $s \not \in \set{y,z(y)}$, we consider the function
\begin{equation}\label{E-L34-12}
\xi_s = a(s) \, u_y - b(s) \, v_y,
\end{equation}
where $u_y$ and $v_y$ satisfy \eqref{E-L34-10d}, and $a(s)$, $b(s)$ are given by
\begin{equation}\label{E-L34-14}
\left\{
\begin{array}{l}
a(s) =\breve{v}_y(s)\, \big( \breve{v}^2_y(s) + \breve{u}^2_y(s) \big)^{- \frac 12},\\[5pt]
b(s) = \breve{u}_y(s) \, \big( \breve{v}^2_y(s) + \breve{u}^2_y(s) \big)^{- \frac 12}.
\end{array}
\right.
\end{equation}
In particular, $a(s) > 0$ in $A_1 \cup A_2$, $b(s) > 0$ in $A_1$ and $b(s) < 0$ on $A_2$. \smallskip
Since $a(s)$ and $b(s)$ are different from $0$, $\rho(\xi_s,y) = (2k-4)$ and $\rho(\xi_s,z(y))=1$. Since $\breve{\xi}_s(s) =0$, $\rho(\xi_s,s) \ge 1$. Since $\xi_s$ belongs to $V_{y,z(y)}$,  Equation~\eqref{E-L34-4} implies that $\rho(\xi_s,s) = 1$, $\cS_{\mathrm{b}}(\xi_s) = \set{y,z(y),s}$, and $\breve{\xi_s}$ changes sign at $z(y)$ and $s$  (use Lemma~\ref{L-breve} again).\smallskip

Taking $s_2 \in A_2$ and $s \in A_1$, we find that $\breve{\xi}_s(s_2) < 0$ and hence,
\begin{equation}\label{E-L34-14a}
\text{for~} s \in A_1, \quad \left\{
\begin{array}{ll}
\breve{\xi_s} > 0 & \text{in~~} \cA(s, z(y)),\\[5pt]
\breve{\xi_s} < 0 & \text{in~~} A_2 \bigcup \cA(y,s).
\end{array}
\right.
\end{equation}
Similarly, taking $s_1 \in A_1$ and $s \in A_2$, we find that $\breve{\xi}_s(s_1) > 0$ and hence,
\begin{equation}\label{E-L34-14b}
\text{for~} s \in A_2, \quad \left\{
\begin{array}{ll}
\breve{\xi_s} < 0 & \text{in~~} \cA(z(y), s),\\[5pt]
\breve{\xi_s} > 0 & \text{in~~} \cA(s,y) \bigcup A_1 .
\end{array}
\right.
\end{equation}

\begin{figure}[!ht]
\centering
\includegraphics[width=0.8\linewidth]{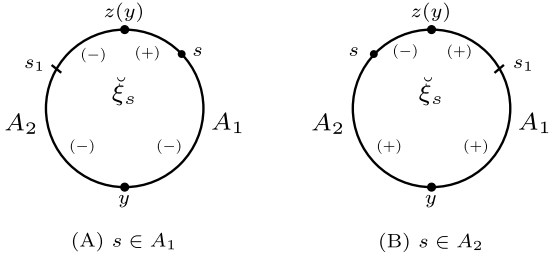}
\caption{Signs of the function $\breve{\xi}_s$}\label{F-hmn3-L34-20}
\end{figure}
There exists a unique $\theta(s) \in (-\pib , \pib) $ such that $a(s) = \cos(\theta(s))$ and $b(s) = \sin(\theta(s))$, so that $\xi_s = \cos(\theta(s)) \, u_y - \sin(\theta(s))\, v_y$. Then, $\theta(s) \in (0,\pib)$ when $b(s) > 0$ or equivalently when $s \in A_1$;  $\theta(s) \in (-\pib,0)$ when $b(s) < 0$ or equivalently when $s \in A_2$. Furthermore, the map $s \mapsto \theta(s)$ is injective because $\cS_{\mathrm{b}}(\xi_s) = \set{y,z(y),s}$. \smallskip

For $s_1, s_2 \in A_1$ or $A_2$, we have
\begin{equation}\label{E-L34-14c}
\left\{
\begin{array}{ll}
\xi_{s_1} - \xi_{s_2} & = \big( \cos(\theta(s_1) - \cos(\theta(s_2)\big) u_y - \big( \sin(\theta(s_1) - \sin(\theta(s_2)\big) v_y \\[5pt]
& = - 2 \sin\frac{\theta(s_1) - \theta(s_2)}{2}\big( \sin\frac{\theta(s_1) + \theta(s_2)}{2}\, u_y + \cos\frac{\theta(s_1) + \theta(s_2)}{2}\, v_y\big).
\end{array}
\right.
\end{equation}
If $s_1, s_2 \in A_1$, $\theta(s_1), \theta(s_2) \in (0,\pib)$, and the second factor in the second line of Equation~\eqref{E-L34-14c} is positive in $A_1$. If $s_1, s_2 \in A_2$, $\theta(s_1), \theta(s_2) \in (-\pib,0)$, and the second factor in the second line of Equation~\eqref{E-L34-14c} is positive in $A_2$. \\  Since $\breve{\xi}_{s_1}(s_2) - \breve{\xi}_{s_2}(s_2) = \breve{\xi}_{s_1}(s_2) $, using Equations~\eqref{E-L34-14a} and \eqref{E-L34-14b} we conclude that,
\begin{equation}\label{E-L34-14e}
\left\{
\begin{array}{l}
s_1 \in A_1 \text{~and~} s_2 \in \cA(s_1,z(y)) \Rightarrow \breve{\xi}_{s_1}(s_2) > 0\\[5pt]
 \quad \Rightarrow  \sin\frac{\theta(s_1)-\theta(s_2)}{2}
< 0  \quad \Rightarrow \theta(s_2) > \theta(s_1) \text{~in~} (0,\pib),\\[5pt]
s_1 \in A_2  \text{~and~} s_2 \in \cA(s_1,y) \Rightarrow \breve{\xi}_{s_1}(s_2) > 0\\[5pt]
 \quad \Rightarrow  \sin\frac{\theta(s_1)-\theta(s_2)}{2}
< 0  \quad \Rightarrow \theta(s_2) > \theta(s_1)  \text{~in~} (-\pib,0),\\[5pt]
\end{array}%
\right.
\end{equation}
otherwise stated, when $s$ moves counter-clockwise on $A_1$, resp. $A_2$,  the function $\theta(s)$ increases from $0$ to $\pib$, resp. from $-\pib$ to $0$, see Figure~\ref{F-hmn3-L34-22}.\medskip

\begin{figure}[!ht]
\centering
\includegraphics[width=0.8\linewidth]{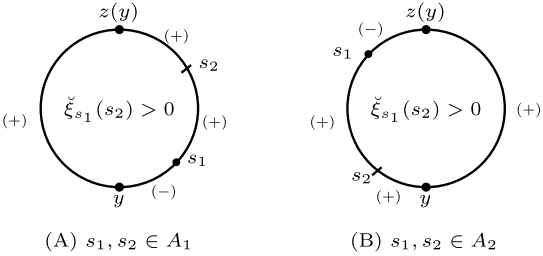}
\caption{Sign of $\breve{\xi}_{s_1}(s_2)$}\label{F-hmn3-L34-22}
\end{figure}

 Under the assumption that $\dim V_{y,z(y)}=2$, we are in a framework similar to that of Subsection~\ref{SS-hmn-25}, with $V_{y,z(y)}$ replacing $U^2_x$. We use a rotating function argument similar to the one used in Paragraph~\ref{SSS-hmn-25D}. For this purpose, we first investigate the limits of $\xi_s$ and $\theta(s)$ when $s$ tends to $z(y)$ or to $y$.\smallskip

Let $\gamma_z$ (resp. $\gamma_y$) denote a local parametrization of $\Gamma $ in a neighborhood of $z(y)$ (resp. $y$), such that $\gamma_z(0) = z(y)$, $\gamma_z(-\varepsilon) \in A_2$, and $\gamma_z(\varepsilon) \in A_1$ (resp. $\gamma_y(0) = y$, $\gamma_y(-\varepsilon) \in A_2$, and $\gamma_y(\varepsilon) \in A_1$). Using our choice of sign for  $\breve{u}_y$ and $\breve{v}_y$, the vanishing properties of these functions, and  Lemma~\ref{L-breve}, we find that, in a  pointed neighborhood of $z(y)$,
\begin{equation*}
\left\{
\begin{array}{l}
\breve{u}_y(\gamma_z(t)) = \alpha_u\, t + o(t), {\text{~with~}}\alpha_u > 0,\\[5pt]
\breve{v}_y(\gamma_z(t)) = \alpha_v\, t^2 + o(t^2), {\text{~with~}}\alpha_v > 0,\\[5pt]
a(\gamma_z(t)) = \frac{\alpha_v}{\alpha_u} |t| + o(t),\\[5pt]
b(\gamma_z(t)) = \sign(t) + o(1).\\[5pt]
\end{array}%
\right.
\end{equation*}
Similarly, in a neighborhood of $y$,

\begin{equation*}
\left\{
\begin{array}{l}
\breve{u}_y(\gamma_y(t)) = \beta_u\, t^{2k-3} + o(t^{2k-3}), {\text{~with~}}\beta_u > 0,\\[5pt]
\breve{v}_y(\gamma_y(t)) = \beta_v\, t^{2k-4} + o(t^{2k-4}), {\text{~with~}}\beta_v > 0,\\[5pt]
a(\gamma_y(t)) =  1 + o(1),\\[5pt]
b(\gamma_y(t)) = \frac{\beta_u}{\beta_v}\, t + o(t).\\[5pt]
\end{array}%
\right.
\end{equation*}

This gives us the limits of $\xi_s$ and $\theta(s)$ when $s$ tends to $z(y)$ in $A_1$ or $A_2$ (resp. when $s$ tends to $y$ in $A_1$ and $A_2$),

\begin{equation}\label{E-L34-14g}
\left\{
\begin{array}{llll}
\lim_{\substack{s \to z(y)\\s \in A_1}}\xi_s &= - v_y,
& \qquad  \lim_{\substack{s \to z(y)\\s \in A_1}}\theta(s) &= \pib,\\[10pt]
\lim_{\substack{s \to y\\s \in A_1}}\xi_s &= u_y,
& \qquad  \lim_{\substack{s \to y\\s \in A_1}}\theta(s) &= 0,\\[10pt]
\lim_{\substack{s \to z(y)\\s \in A_2}}\xi_s &=  v_y,
& \qquad  \lim_{\substack{s \to z(y)\\s \in A_2}}\theta(s) &= 0,\\[10pt]
\lim_{\substack{s \to y\\s \in A_2}}\xi_s &= u_y,
& \qquad  \lim_{\substack{s \to y\\s \in A_2}}\theta(s) &=  -\pib .
\end{array}
\right.
\end{equation}
\medskip

\begin{figure}[!ht]
  \centering
  \includegraphics[scale=0.7]{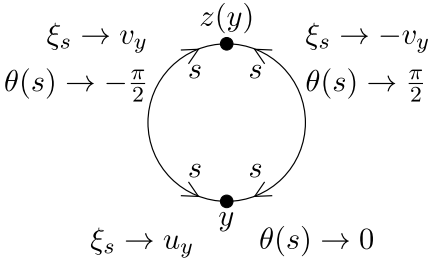}
  \caption{Lemma~\ref{RL-L34a}: limits of $\xi_s$ when $s$ tends to $y$ or $z(y)$}\label{F-hmn3-L34-22nu}
\end{figure}

When $s$ moves counter-clockwise from $y$ to $z(y)$ on $A_1$, $\theta(s)$ increases from $0$ to $\pib$; when $s$ moves counter-clockwise from $z(y)$ to $y$ on $A_2$, $\theta(s)$ increases from $-\pib$ to $0$.\medskip

\noi Let $ [-\pib,\pib] \ni \sigma \mapsto \Gamma (\sigma)$ be a parametrization of $\Gamma $ such that  $\gamma_1(-\pib) = \gamma_1(\pib) = z(y)$, $\gamma_1(0)=y$,  $\gamma_1\big( (-\pib,0) \big) = A_2$, and $\gamma_1\big( (0,\pib) \big) = A_1 \,$.\smallskip

\noi Consider the map
\begin{equation*}
\theta_1 : (-\pib,0) \cup (0,\pib) \ni \sigma \mapsto \theta(\gamma_1(\sigma)) \in (-\pib,\pib).
\end{equation*}
Then, $\theta_1$ extends to a continuous, monotone increasing map from $(-\pib,\pib)$ to $(-\pib,\pib)$ such that  $\lim_{\sigma \to \pm \pib}\theta_1(\sigma) = \pm \pib$ and $\lim_{\sigma \to 0}\theta_1(\sigma) = 0$. \medskip

\noi For $t \in [-\pib,\pib]$, introduce the functions
\begin{equation}\label{E-L34-15a}
\zeta_{t} := (\cos t)  u_y - (\sin t)  v_y,
\end{equation}
so that $\zeta_{-\pib} = v_y$, $\zeta_0 = u_y$, and $\zeta_{\pib} = - v_y$. If $t \not \in \set{-\pib, 0, \pib}$ there exists a unique $s(t) \in \Gamma \sm\set{y,z(y)}$ such that
\begin{equation}\label{E-L34-15c}
\left\{
\begin{array}{l}
\rho(\zeta_{t},y) = (2k-4), ~~\rho(\zeta_{t},z(y))=1,~~\rho(\zeta_{t},s(t))=1,\\[5pt]
\cS_{\mathrm{b}}(\zeta_{t}) = \set{y,z(y),s(t)},\\[5pt]
s(t) \in A_1 \text{~if\quad} t \in (0,\pib), \text{\quad and\quad} s(t) \in A_2 \text{~if\quad} t \in (-\pib,0).
\end{array}
\right.
\end{equation}
Indeed, near $z(y)$, $\cos t > 0$ implies that $\zeta_t$ has the sign of $u_y$. In a small pointed arc $J_y$ around $y$, $\zeta_t \sim - \sin t \, v_y$ which has the sign of $(-t)$. \medskip

We now apply the ``rotating function  argument'', see Paragraph~\ref{SSS-hmn-25D}, to the family of nodal sets $\cZ(\zeta_t)$. We have $\cZ(\zeta_{-\pib}) = \cZ(\zeta_{\pib}) = \cZ(v_y)$; when $t \in (0,\pib)$, $\cS_{\mathrm{b}}(\zeta_t) = \set{y, z(y),s(t)}$ with $s(t) \in A_1$; when $t \in (-\pib,0)$, $s(t) \in A_2$. \smallskip

\noi Figure~\ref{F-hmn3-L34-24n} (resp. Figure~\ref{F-hmn3-L34-26n}) illustrates the deformation of the nodal set $\cZ(v_y)$ given in Subfigure (A) in a particular case with $k=4$.\medskip

\noi In Figure~\ref{F-hmn3-L34-24n} the nodal set $\cZ(v_y)$ is connected. When $t$ decreases from $\pib$ to $0$ (top line),
 the nodal set $\cZ(\zeta_t)$ deforms from  $\cZ(v_y)$ in Subfigure~(A) to $\cZ(u_y)$ in Subfigure~(L).  When $t$ increases from $-\pib$ to $0$ (bottom line), the nodal set $\cZ(\zeta_t)$ deforms from $\cZ(v_y)$ in Subfigure~(A) to $\cZ(u_y)$ in Subfigure~(R). In Figure~\ref{F-hmn3-L34-26n} the nodal set $\cZ(\zeta_t)$ has two components and deforms to (R) or (L).\medskip

\begin{figure}[!ht]
  \centering
  \includegraphics[width=0.99\textwidth]{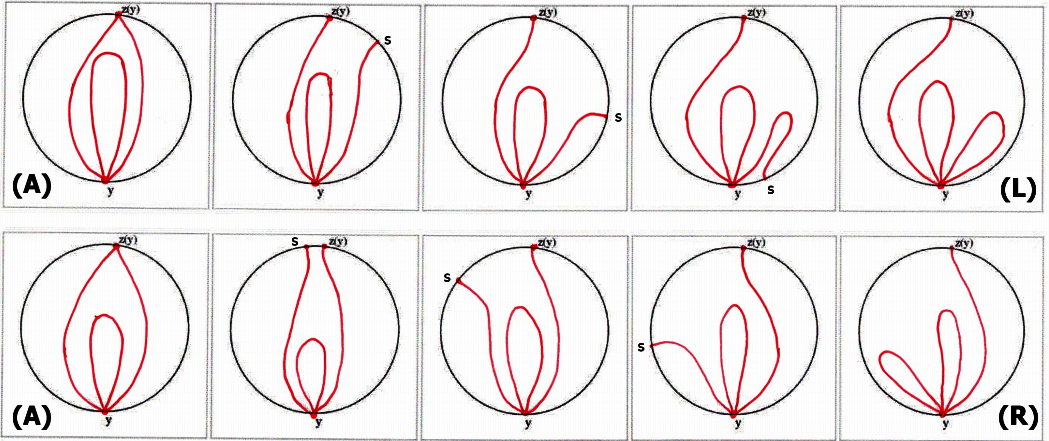}
  \caption{The point $s$  tends to $y$ clockwise (top) or counter-clockwise (bottom), here $k=4$}\label{F-hmn3-L34-24n}
\end{figure}

\begin{figure}[!ht]
  \centering
  \includegraphics[width=0.99\textwidth]{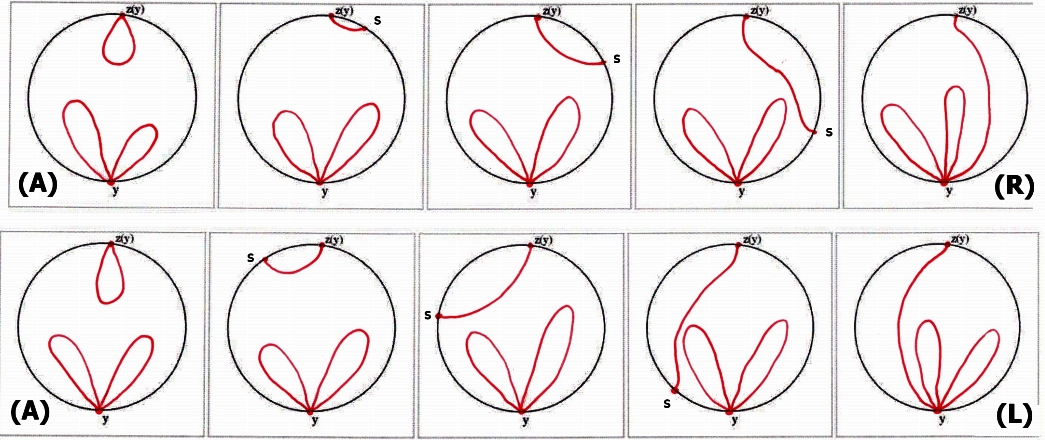}
  \caption{The point $s$  tends to $y$ clockwise (top) or counter-clockwise (bottom), here $k=4$}\label{F-hmn3-L34-26n}
\end{figure}

 The nodal patterns $(L)$ and $(R)$ belong to a function $u_y$. We claim that they are different. For this purpose, we label the loops as in Paragraph~\ref{SSS-hmn-25B}, and we use the combinatorial type of the function $u_y$, see Paragraph~\ref{SSS-hmn-31bas}.

\begin{figure}[!ht]
  \centering
  \includegraphics[scale=0.25]{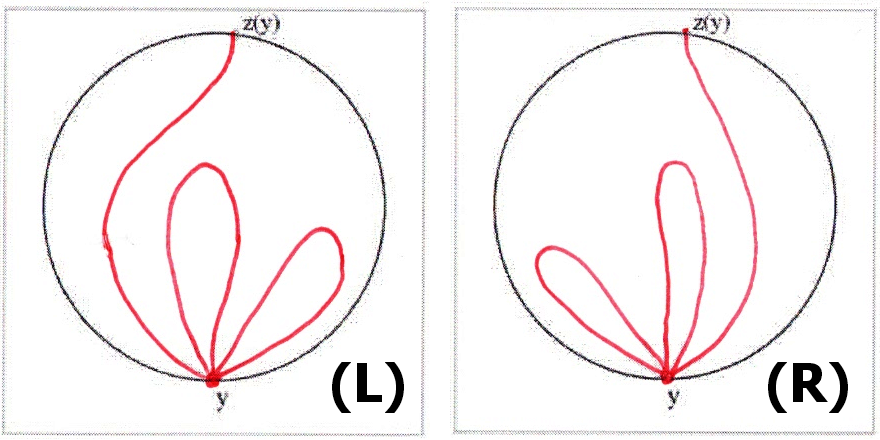}
  \caption{The nodal patterns $(L)$ and $(R)$ are different}\label{F-hmn3-L34-26LR}
\end{figure}

\FloatBarrier

The maps $\tau$ describing the combinatorial types of the nodal patterns $(L)$ and $(R)$ of Figure~\ref{F-hmn3-L34-26LR} are given by
\begin{equation*}
\tau_L = \begin{pmatrix}
           \downarrow & 1 & 2 & 3 & 4 & 5 \\
          5 & 2 & 1 & 4 & 3 & \downarrow \\
         \end{pmatrix}
\text{\quad and\quad}
\tau_R = \begin{pmatrix}
           \downarrow & 1 & 2 & 3 & 4 & 5 \\
           1 & \downarrow & 3 & 2 & 5 & 4 \\
              \end{pmatrix},
\end{equation*}
where $\downarrow$ corresponds to the arc hitting the boundary. Correspondingly, we label the nodal domains as in Section~\ref{S-hmn2L}, and we find the words,
\begin{equation*}
\ww_L = |1|2|1|3|1|4| \text{\quad and\quad} \ww_R = |1|2|3|2|4|2|.
\end{equation*}

 The nodal patterns  $(L)$ and $(R)$ having different signatures (see Definition~\ref{D-hmn2L-4}), they are different, although they should both be the nodal pattern of $u_y$, a contradiction. Recall that we already proved that $\dim V_{y,s} \le 2$. Since the assumption $\dim V_{y,s} = 2$ leads to a contradiction, at least in the example at hand, we conclude that  $\dim V_{y,s} = 1$. The proof in the general case follows the same lines, as in  Subsection~\ref{SS-hmn-25}. This proves Assertion~(i). \quad \qedc \medskip

\begin{remark}\label{R-hmn3-L34-4}
When $t$ varies in $(-\pib,\pib)$, the nodal sets $\cZ(\zeta_{t})$
vary continuously with respect to the Hausdorff distance, see Lemma \ref{L-hdn}.
It follows that they are either all connected, or that they all have two components, so that their type in Figure~\ref{F-hmn3-L34-2} is either (b) \& (c)
or (d) \& (e).
\end{remark}%

\noid \emph{Assertion~(iii).~} This is a consequence of  Assertion~(i) and its proof. \quad \qedc \medskip

\noid \emph{Assertion~(iv).~}   Let $v_{y,s}$ denote a generator of $V_{y,s}$. Then, the linear system which defines $v_{y,s}$ up to scaling has constant rank, so that it has a solution which depends smoothly on the parameters $y, s$ locally. \quad \qedc \medskip

Lemma~\ref{RL-L34a} (aka Lemma~\ref{L-L34a}) is proved.
\end{proof}

\subsection{Structure and combinatorial type of nodal sets in $V_{y,s}$}\label{SSS-hmn32bas}

The last two lines in \eqref{E-L34-4} give rise to three cases for $0 \neq u \in V_{y,s}$.\smallskip

\noi \textbf{Case~1}.~ $\rho(u,y) = (2k-3)$, $\rho(u,s) = 1$,  and $\cS_{\mathrm{b}}(u) = \set{y,s}$.  This means that $u \in U_y$, and this case only occurs when $s = z(y)$, see Figure~\ref{F-hmn3-L34-2}\,(a).\smallskip

\noi \textbf{Case~2}.~ $\rho(u,y) = (2k-4)$, $\rho(u,s) = 2$, and $\cS_{\mathrm{b}}(u) = \set{y,s}$, with two possibilities for $\cZ(u)$,
\begin{itemize}
  \item[$\diamond$] \emph{either} $\cZ(u)$ consists of $(k-2)$ loops at $y$ which do not intersect nor meet $\Gamma $ away from $y$, and one loop at $s$ which does not hit $\Gamma $ away from $s$, and does not meet the loops at $y$,
  \item[$\diamond$] \emph{or} $\cZ(u)$ consists of $(k-3)$ loops at $y$  which do not intersect nor meet $\Gamma $ away from $y$, and two simple arcs from $y$ to $s$ which do not meet except at $y$ and $s$, and do not meet the loops except at $y$; in this case we have a ``generalized loop'' which hits $\Gamma $ at $s$,
\end{itemize}
see Figures~\ref{F-hmn3-L34-2}\,(b) and \ref{F-hmn3-L34-2}\,(d).\smallskip

 \noi \textbf{Case~3}.~ $\rho(u,y) = (2k-4)$, $\rho(u,s) = 1$, and there exists another point $s_1 \in \Gamma $, $s_1 \neq s,y$ such that  $\cS_{\mathrm{b}}(u) = \set{y,s,s_1}$ and $\rho(u,s_1)=1$. In this case there are two possibilities for $\cZ(u)$,
\begin{itemize}
  \item[$\diamond$] \emph{either} $\cZ(u)$ consists of $(k-2)$ loops at $y$ which do not intersect nor meet $\Gamma $ away from $y$, and one arc from $s$ to $s_1$ which does not hit $\Gamma $ away from $s,s_1$, and does not meet the loops at $y$,
  \item[$\diamond$] \emph{or} $\cZ(u)$ consists of $(k-3)$ loops at $y$  which do not intersect nor meet $\Gamma $ away from $y$, and two simple arcs, one from $y$ to $s$ and one from $y$ to $s_1$ which do not meet except at $y$, and do not meet the loops except at $y$; in this case we have a ``generalized loop'' which contains a sub-arc of $\Gamma $ from $s$ to $s_1$,
\end{itemize}
see Figures~\ref{F-hmn3-L34-2}\,(c) and  \ref{F-hmn3-L34-2}\,(e).

\begin{figure}[!ht]
  \centering
  \includegraphics[width=0.5\textwidth]{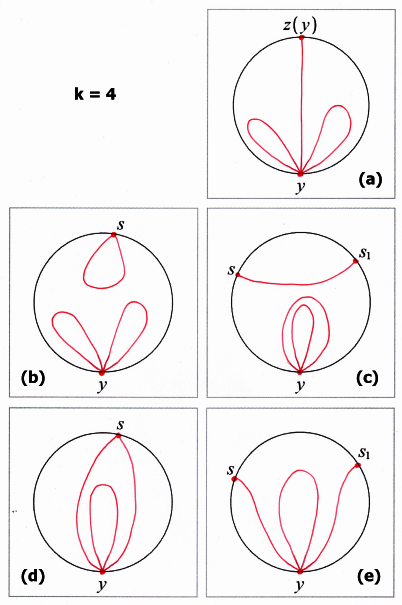}
\caption{Nodal patterns for $u \in V_{y,s}$ (k=4)}\label{F-hmn3-L34-2}
\end{figure}

\begin{remarks}\label{R-hmn3-L34-2}\phantom{}
\begin{enumerate}[(i)]
  \item The subcases in Cases~2 and 3 are distinguished by the fact that $b_0(\cZ(u)) = 1$ (as in Figures~\ref{F-hmn3-L34-2}\,(d) \&\,(e)) or $b_0(\cZ(u)) = 2$, as in Figures~\ref{F-hmn3-L34-2} (b) \& (c), if $s \neq z(y)$. Since $\dim V_{y,s} = 1$, the subcases cannot occur simultaneously. Indeed, by Lemma~\ref{L-zeroc} we would otherwise find a function $u$ such that $\rho(u,y) \ge (2k-3)$ and $\rho(y,s) \ge 1$, with $s \neq z(y)$, contradicting Case~1.
  \item  At this stage of the discussion, the location of $z(y)$ with respect to $y$, $s$, and $s_1$ in Sub-figures~\ref{F-hmn3-L34-2}\,(b)--(e) is not clear. This will be explained in Lemma~\ref{L-L35a}.
\end{enumerate}
\end{remarks}%

For a pair $(y,s) \in \Gamma_{(2k -3)} \times \Gamma $ with $y \neq s$, since $\dim V_{y,s} = 1$, we can define the \emph{combinatorial type} of a generator $v_{y,s}$ at the point $y$, as we did for the generator $u_y$ of $U_y$ in Paragraph~\ref{SSS-hmn-31bas}, taking the above cases into account.

When $(y,s) = (y,z(y))$, the combinatorial type is that of $u_y$, and we denote it by $\tau_{u_y}$. For example, the combinatorial type $\tau_{u_y}$ of the function $u_y$ whose nodal pattern appears in Figure~\ref{F-hmn3-L34-2}\,(a) is given by
\begin{equation*}
\tau_{a}^{\ref{F-hmn3-L34-2}} = \begin{pmatrix}
             \downarrow &  1 & 2 & 3 & 4 & 5 \\
               3 & 2 & 1 & \downarrow & 5 & 4 \\
             \end{pmatrix},
\end{equation*}

When $s \neq y$, the combinatorial type $\tau := \tau_{v_{y,s}}$ of a function $v_{y,s}$ is described as follows. When the nodal set is connected,  as in Figure~\ref{F-hmn3-L34-2} (d) and (e), we write $\tau(j) = s$ (resp. $s_1$) to indicate that the nodal semi-arc emanating from $y$ tangentially to the ray labeled $j$ ends up at $s$ (resp. $s_1$). When the nodal set has two  components, as in Figure~\ref{F-hmn3-L34-2}\,(b) and\,(c), we write $\tau(s)=s$ to indicate that there is a loop at $s$, and $\tau(s) = s_1$ to indicate that there is a nodal interval from $s$ to $s_1$. We describe the maps $\tau$ by $2\times (2k-2)$ matrices. The first row enumerates the rays at $y$ and the rays at $s$ and $s_1$ (counter-clockwise). With this convention, the combinatorial type $\tau_{v_{y,s}}$ of a function $v_{y,s}$ whose nodal pattern appears in Figure~\ref{F-hmn3-L34-2}\,(b)--(e), is given by one of the following formulas.
\begin{equation*}
\tau_{2,b}^{\ref{F-hmn3-L34-2}} = \begin{pmatrix}
               1 & 2 & 3 & 4 & s & s\\
               2 & 1 & 4 & 3 & s & s\\
             \end{pmatrix},\quad
\tau_{2,c}^{\ref{F-hmn3-L34-2}} = \begin{pmatrix}
               1 & 2 & 3 & 4 & s_1 & s\\
               2 & 1 & 4 & 3 & s & s_1\\
             \end{pmatrix},
\end{equation*}
and
\begin{equation*}
\tau_{1,d}^{\ref{F-hmn3-L34-2}} = \begin{pmatrix}
               1 & 2 & 3 & 4 & s & s\\
               s & 3 & 2 & s  & 1 & 4\\
             \end{pmatrix},\quad
\tau_{1,e}^{\ref{F-hmn3-L34-2}} = \begin{pmatrix}
               1 & 2 & 3 & 4 & s_1 & s\\
               s_1 & 3  & 2 & s & 1 & 4\\
             \end{pmatrix}.
\end{equation*}

\subsection{Precise description of $V_{y,s}$}\label{SSS-hmn32bf}

 In this subsection  we analyze the behavior of a generator $v_{y,s}$ of $V_{y,s}$ when $y$ and $s$ vary. More precisely, Lemma~\ref{L-L34b} describes the behavior of $v_{y,s}$ when $y \in \Gamma_{(2k -3)}$ is fixed and $s$ tends to $y$ or to $z(y)$, and the behavior when $s \in \Gamma_{(2k -3)}$ is fixed and $y$ tends to $s$. Lemma~\ref{L-L35a} describes the global behavior of $v_{y,s}$ for a given $y \in \Gamma_{(2k -3)}$.

\begin{lemma}\label{L-L34b}
Assume that  $\Omega$ is simply connected. Let $U:=U(\lambda_k)$ with $k \ge 3$.  Assume that $\dim U = (2k-2)$. Given $(y,s) \in \Gamma_{(2k -3)} \times \Gamma $, with $ y \neq s$, recall that
\begin{equation*}
V_{y,s} := \set{u \in U \mid \rho(u,y) \ge 2k-4 \text{~and~} \rho(u,s) \ge 1}.
\end{equation*}
 Let $v_{y,s}$ be a generator of $V_{y,s}$. The function $v_{y,s}$ has the following properties.
  \begin{enumerate}[(i)]
  \item  When $s=z(y)$, the function $ \breve{v}_{y,z(y)}=\breve{u}_y$ vanishes on $\Gamma$ precisely at the points $y$ and $z(y)$, and changes sign at these points. When $s \neq z(y)$ and $\cS_{\mathrm{b}}(v_{y,s}) = \set{y,s,s'}$ with $s \neq s'$, $\rho(v_{y,s},s)=\rho(v_{y,s},s')=1$, the function $ \breve{v}_{y,s}$ vanishes on $\Gamma$ precisely at the points $y, s$ and $s'$, does not change  sign at $y$, and changes sign at $s$ and $s'$. When $s \neq z(y)$ and $\cS_{\mathrm{b}}(v_{y,s}) = \set{y,s}$ with  $\rho(v_{y,s},s)=2$, the function $ \breve{v}_{y,s}$ vanishes on $\Gamma$ precisely at the points $y$ and $s$, and does not change sign.
   \item For fixed $y \in \Gamma_{(2k -3)}$, and $s$ close enough to $z(y)$, $\cS_{\mathrm{b}}(v_{y,s}) = \set{y,s,s'}$, with $s' \neq s, y$,  and $\rho(v_{y,s},s) = 1$, $\rho(v_{y,s},s') = 1$. Furthermore, when $s$ tends to $z(y)$, $[v_{y,s}]$ tends to $[u_{y}]$, and $s'$ tends to $y$.
  \item For fixed $y \in \Gamma_{(2k -3)}$, and $s$ close enough to $y$, $\cS_{\mathrm{b}}(v_{y,s}) = \set{y,s,s'}$, with $s' \neq s, y$,  and $\rho(v_{y,s},s) = 1$, $\rho(v_{y,s},s') = 1$. Furthermore, when $s$ tends to $y$, $[v_{y,s}]$ tends to $[u_{y}]$, and $s'$ tends to $z(y)$.
  \item  For fixed $s \in \Gamma_{(2k -3)}$, and $y\in \Gamma_{(2k -3)}$ close enough to $s$, $\cS_{\mathrm{b}}(v_{y,s}) = \set{y,s,s'}$, with $s' \neq y,s$,  and $\rho(v_{y,s},s) = 1$, $\rho(v_{y,s},s') = 1$. Furthermore, when $y$ tends to $s$, $[v_{y,s}]$ tends to $[u_s]$, and $s'$ tends to $z(s)$.
\end{enumerate}
\end{lemma}%

\begin{proof}\phantom{}

\noid \emph{Assertion~(i).~} Use Lemma~\ref{L-breve}, and the description of the possible nodal patterns for $V_{y,s}$ which follows from Lemma~\ref{RL-L34a}.\medskip

\noid \emph{Assertion~(ii).~} Assume that the first statement is not true. Then, we can find a sequence $\set{s_n}$ tending to $z(y)$, and a corresponding sequence $\set{u_n := v_{y,s_n}} \subset \bS(U)$ such that $\rho(u_n,y) = (2k-4)$,  $\rho(u_n,s_n) = 2$, and $u_n$ tends to some $u \in \bS(U)$. Since the convergence is uniform in $C^m$ for fixed $m \ge 0$, it follows that $\rho(u,y) \ge (2k-4)$ and $\rho(u,z(y)) \ge 2$, see Remark~\ref{R-evp2-2}  (lower semi-continuity of $\rho$),  and hence $u \in V_{y,z(y)}$. By Lemma~\ref{RL-L34a}~(iii),  we must have $[u]=[u_y]$, and we reach a contradiction since $\rho(u_y,z(y))=1$. This proves the first statement. \smallskip

We now prove the second statement. Considering $\set{u_n} \subset \bS(U)$ such that $\cS_{\mathrm{b}}(u_n) = \set{y,s_n,s'_n}$ with $s_n \neq s'_n$ and $s_n$ tends to $z(y)$, we may also assume that $s'_n$ tends to some $s'$, and $u_n$ tends to $u$. Then, $\rho(u,y) \ge (2k-4)$, and $\rho(u,z(y)) \ge 1$,  $\rho(u,s') \ge 1$. Lemma~\ref{RL-L34a}~(iii) implies that $[u] =[u_y]$, and $\breve{u}(s') = 0$ so that $s' \in \set{y,z(y)}$ (where $\breve{u}$ is defined by \eqref{E-evp-dr}). Assuming that $s' = z(y)$, we would have $\rho(u,z(y)) = 2$, leading to a contradiction. Indeed, in that case,  $s_n$ and $s'_n$ would both tend to $z(y)$, and the derivative $\partial_b \breve{u}_n$ of the function $\breve{u}_n$ along the boundary would vanish at some $t_n$ on the smallest arc between $s_n$ and $s'_n$, with $t_n$ tending to $z(y)$. Passing to the limit, we would have $\partial_b \breve{u}_y(z(y)) = 0$, implying that $\rho(u_y,z(y)) \ge 2$. Assertion~(ii) is proved.\medskip

\noid \emph{Assertion~(iii).~} Assume that the first statement is false. Then, we can find a sequence $\set{s_n}$ tending to $y$,  and a corresponding sequence $\set{u_n := v_{y,s_n}} \subset \bS(U)$ such that $\rho(u_n,y) = (2k-4)$,  $\rho(u_n,s_n) = 2$, and $u_n$ tends to some $u \in \bS(U)$. Then, $\rho(u,y) \ge (2k-4)$. Applying the local structure theorem to $u$, and following the arguments in the proof of Lemma~\ref{L-L38},  Part~(C), we may choose $r_1 > 0$ such that the neighborhood $D_{+}(r_1)$ of $y$ satisfies Properties~(B-a)--(B-d) in the proof of Lemma~\ref{L-L38}. We may also assume that the sequence $\set{s_n}$ is contained in $D_{+}(r_1)$. The nodal set of $u_n$ consists of either $(k-2)$ loops at $y$ and a loop at $s_n$, or $(k-3)$ loops at $y$ and two arcs from $y$ to $s_n$, see Figure~\ref{F-hmn3-L34-2}\,(b-d).
 As in the proof of Lemma~\ref{L-L38},  Part~(C), for $r_0 < r_1$ small enough, each loop at $y$ must cross $C_{+}(r_0)$ at (at least) two distinct points; so does each loop at $s_n$, and each arc from $y$ to $s_n$. We then conclude that $\rho(u,y) \ge (2k-2)$. Euler's formula then implies that actually $\rho(u,y) = (2k-2)$. This is not possible since $y \in \Gamma_{(2k -3)}$.   This proves the first statement in Assertion~(iii).\smallskip

We now prove the second statement. If $s_n$ tends to $y$, we have $\cS_{\mathrm{b}}(u_n) = \set{y,s_n,s'_n}$, with $s_n \neq s'_n$. The preceding argument also shows that no subsequence of $s'_n$ can tend to $y$. We may then assume that $s_n$ tends to $y$ and $s'_n$ tends to some $s' \neq y$, with $u_n$ tending to some $u$. Then, $\rho(u,y) \ge (2k-4)$, and the previous argument shows that $\rho(u,y) \ge (2k-3)$, and $\rho(u,s') \ge 1$. Lemma~\ref{L-L32} then shows that $s' = z(y)$, and that $[u]=[u_y]$. Assertion~(iii) is proved. \medskip

\noid \emph{Assertion~(iv).~} See Figure~\ref{F-hmn3-L34-8i}.\smallskip

\noi Assume that the first statement in false. Then, there exists a sequence $\set{y_n} \subset \Gamma_{(2k -3)}$ which tends to $s$, with a corresponding sequence $\ \set{u_n:=v_{y_n,s}} \subset \bS(U)$ tending to some $u \in \bS(U)$, and such that $\rho(v_{y_n,s},s) = 2$. The convergence of $u_n$ to $u$ being uniform in $C^{2k}$, we have $\rho(u,s) \ge (2k-4)$. Lemma~\ref{L-hdn} and Lemma~\ref{RL-L34a}(ii) applied to $u_n$  imply that $ \cZ(u) \cup \Gamma $ is connected.  Applying Euler's formula to $u$, we conclude that $(2k-4) \le \sum_{z \in \cS_{\mathrm{b}}(u)} \rho(u,z) \le (2k-2)$. Since the functions $\breve{u}_n$ do not change sign on $\Gamma $, $\breve{u}$ does not change sign either, so that $\rho(u,s) \neq (2k-3)$. Applying the local structure theorem to $u$ at $s$, and using the same proof as in Lemma~\ref{L-L38},  Part~(C),  (with the disc $D_{+}(s,r_0)$ and circle $C_{+}(s,r_0)$ centered at $s$), we infer that $\rho(u,s) = (2k-2)$.   Indeed, the nodal sets $\cZ(u_n)$ consist either is $(k-2)$ loops at $y_n$ (including a special loop touching $s$), or $(k-3)$ loops at $y_n$ and a loop at $s$. Since $y_n$ tends to $s$, for $r_0$ small enough, these loops must intersect $C_{+}(s;r_0)$ at $(2k-2)$ distinct points, and we can conclude as in the proof of Lemma~\ref{L-L38}.\smallskip

\noi To prove the second statement, we choose a sequence $\set{y_n}$ such that $\rho(u_n,s) = 1$ for $n$ large enough, so that $\cS_{\mathrm{b}}(u_n) = \set{y_n,s,s'_n}$, with $s'_n \neq s$. An argument similar to the previous one, shows that no subsequence of $\set{s'_n}$ can tend to $s$. Since $s \in \Gamma_{(2k -3)}$, the only remaining possibility is that $\rho(u,s) = (2k-3)$, and hence that $u \in U_s$. Assertion~(iv) is proved. The proof of Lemma~\ref{L-L34b} is complete.
\end{proof}

\begin{figure}[!ht]
  \centering
  \includegraphics[width=0.8\textwidth]{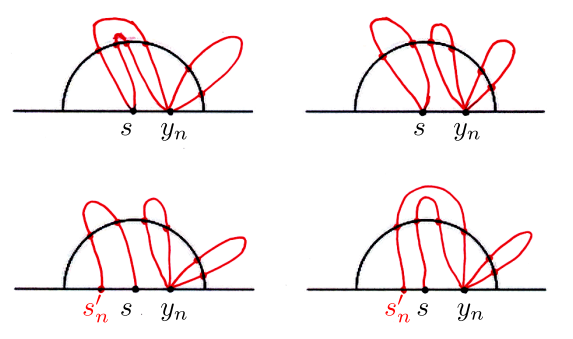}
  \caption{Proof of Lemma~\ref{L-L34b}(iv)}\label{F-hmn3-L34-8i}
\end{figure}

We can enhanced the previous lemma by the following properties.

\begin{properties}\label{P-hmn3-P38}
Assume that $\Omega$ is simply connected, and that $\dim U = (2k-2)$. Given $(y,s) \in \Gamma_{(2k-3)} \times \Gamma $ with $y\neq s$, there exists $\varepsilon_0$ such that
\begin{enumerate}[(i)]
  \item for all $s \in \cA(z(y);\varepsilon_0)\cup \cA(y;\varepsilon_0) \sm \set{y,z(y)}$, $\rho(v_{y,s},s)=1$;
  \item for all $\varepsilon < \varepsilon_0$, there exists $\eta > 0$ such that for all $s \in \cA(z(y);\eta) \sm \set{z(y)}$, $\cS_{\mathrm{b}}(v_{y,s}) = \set{y,s,s'}$ with $s' \in \cA(y;\varepsilon)\sm \set{y}$;
  \item for all $\varepsilon < \varepsilon_0$, there exists $\eta > 0$ such that for all $s \in \cA(y;\eta) \sm \set{y}$, $\cS_{\mathrm{b}}(v_{y,s}) = \set{y,s,s'}$ with $s' \in \cA(z(y);\varepsilon)\sm \set{z(y)}$.
\end{enumerate}%

Let $s_1 \not \in \set{y,z(y)}$. Assume that the function $v_{y,s_1}$ satisfies $\rho(v_{y,s_1},s_1)=1$, i.e., $\cS_{\mathrm{b}}(v_{y,s_1}) = \set{y,s_1,s'_1}$, with $s'_1 \neq s_1$. Then,
\begin{enumerate}[(a)]
  \item there exists $\varepsilon_1 > 0$, such that for all $s \in \cA(s_1;\varepsilon)$, $\rho(v_{y,s},s) = 1$;
  \item for all $\varepsilon > 0$, there exists $\eta < \varepsilon_1$ such that for all $s \in \cA(s_1;\eta)$, $s' \in \cA(s'_1;\varepsilon)$, i.e., the map $s \mapsto s'$ is continuous.
\end{enumerate}
\end{properties}%

\begin{lemma}\label{L-L35a}
 Assume that $\Omega$ is simply connected, and that $\dim U = (2k-2)$.  Let $y \in \Gamma_{(2k-3)}$. Then, the following properties hold.
\begin{enumerate}[(i)]
  \item For any $s \in \cA(y,z(y))$, the function $v_{y,s}$ satisfies $\cS_{\mathrm{b}}(v_{y,s}) = \set{y,s,s'}$ with $s' \in \cA(y,z(y))$, possibly with $s=s'$.
  \item There exists  a unique $s_1 \in \cA(y,z(y))$ such that $v_{y,s_1}$ satisfies $\rho(v_{y,s_1},s_1) = 2$.
  \item For all $s \in \cA(s_1,z(y))$, $\cS_{\mathrm{b}}(v_{y,s}) = \set{y,s,s'}$ with $s' \in \cA(y,s_1)$. Furthermore when $s$ moves counter-clockwise in $\cA(s_1,z(y))$, $s'$ moves clockwise in $\cA(y,s_1)$.
  \item For all $s \in \cA(y,s_1)$, $\cS_{\mathrm{b}}(v_{y,s}) = \set{y,s,s'}$ with $s' \in \cA(s_1,z(y))$. Furthermore when $s$ moves counter-clockwise in $\cA(y,s_1)$, $s'$ moves clockwise in $\cA(s_1,z(y))$.
\end{enumerate}
Similar statements hold for the arc $\cA(z(y),y)$.
\end{lemma}%

The statements in Lemma~\ref{L-L35a} are illustrated in Figure~\ref{F-hmn3-L35a-1} for the arc $\cA(y,z(y))$.  The corresponding nodal patterns appear in Figure~\ref{F-hmn3-L35a-9a}.

\begin{figure}[!ht]
  \centering
  \includegraphics[scale=0.5]{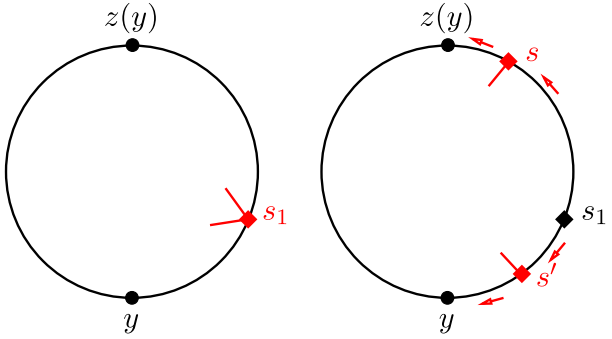}
  \caption{Lemma~\ref{L-L35a}: Assertions~(ii) and (iii)}\label{F-hmn3-L35a-1}
\end{figure}

\begin{proof}[Proof of Lemma~\ref{L-L35a}] Choose a generator $u_y$ of $U_y$ such that $\breve{u}_y$ is positive in the arc $\cA(y,z(y))$, and negative in $\cA(z(y),y)$.\smallskip

\emph{Assertion~(i).}~ Assume that the assertion is false: there exists some $s_0 \in \cA(y,z(y))$ such that $v_0 := v_{y,s_0}$ satisfies $\cS_{\mathrm{b}}(v_0) = \set{y,s_0,s'_0}$, with $s'_0 \in \cA(z(y),y)$. Since $s_0 \neq s'_0$, Lemma~\ref{L-L34b}\,(i) implies that $\breve{v}_0$ vanishes on $\Gamma$ only at the points $y$, $s_0$ and $s'_0$, does not change sign at $y$, and changes sign at $s_0$ and $s'_0$. Choose $v_0$ such that $\breve{v}_0 < 0$ in $\cA(s_0,s'_0)$. For $s \neq y,z(y)$, introduce the function
\begin{equation*}
\xi_s = a_0(s) u_y - b_0(s) v_0,
\end{equation*}
where
\begin{equation*}
\left\{
\begin{array}{l}
a_0(s) = \breve{v}_0(s) \big( \breve{v}_0^2(s) + \breve{u}_y^2(s)
\big)^{-\frac 12},\\[5pt]
b_0(s) = \breve{u}_y(s) \big( \breve{v}_0^2(s) + \breve{u}_y^2(s)
\big)^{-\frac 12}.
\end{array}
\right.
\end{equation*}
Since $b_0(s) \neq 0$, $\rho(\xi_s,y) = 2k-4$ and $\rho(\xi_s,s) \ge 1$, so that $\xi_s \in V_{y,s}$.\smallskip

\noi Choose $s \in \cA(s_0,z(y))$. Since $\breve{v}_0$ only changes sign at $s_0$ and $s'_0$, $\breve{\xi}_s(s'_0) > 0$. By Lemma~\ref{L-breve}, at $y$ along $\Gamma$, $\breve{u}_y$ vanishes at order $(2k-3)$, while $\breve{v}_0$ vanishes at order $(2k-4)$. In a pointed neighborhood $J\sm \set{y}$ of $y$ in $\Gamma$, we have that $\breve{\xi}_s \sim - b_0(s) \breve{v}_0 < 0$, and hence $\breve{\xi}_s$ vanishes at some $s' \in \cA(s'_0,y)$. It follows that $\xi_s \in V_{y,s}$ with $\cS_{\mathrm{b}}(\xi_s) = \set{y,s,s'}$, $\rho(\xi_s,s) = 1$, $\rho(\xi_s,s') = 1$.\smallskip

\noi Similarly, choosing $t \in \cA(y,s_0)$, we have $\breve{\xi}_t(s'_0) < 0$ and $\breve{\xi}_t(z(y)) > 0$, and $\breve{\xi}_t$ vanishes at some $t' \in \cA(z(y),s'_0)$. \\ Finally, we conclude as above that $\xi_t \in V_{y,t}$ with
\[
\cS_{\mathrm{b}}(\xi_t) = \set{y,t,t'}, \, \rho(\xi_t,t) = 1, \, \rho(\xi_t,t') = 1.
\]
These arguments can be visualized on Figure~\ref{F-hmn3-L35a-2}. \smallskip

\begin{figure}[!ht]
  \centering
  \includegraphics[scale=0.7]{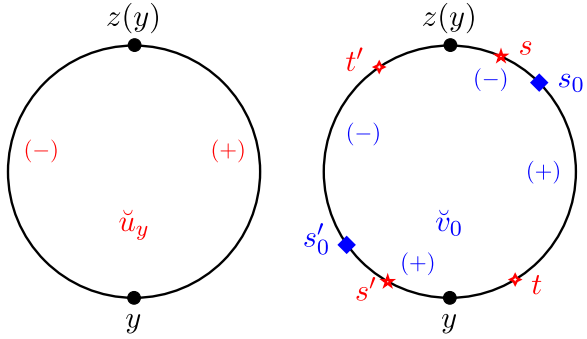}
  \caption{Proof of Lemma~\ref{L-L35a}(i)\\}\label{F-hmn3-L35a-2}
\end{figure}

From the assumed existence of $s_0$, we conclude that, for all $s \in \cA(y,z(y))$, $\xi_s \in V_{y,s}$ and $\cS_{\mathrm{b}}(\xi_s) = \set{y,s,s'}$, with $s' \in \cA(z(y),y)$, $\rho(\xi_s,s) = 1$, $\rho(\xi_s,s') = 1$. Because $s \in \cA(s_0,z(y))$ implies that $s' \in \cA(s'_0,y)$, with a parallel statement for $t$, the previous argument also shows that when the point $s$ moves counter-clockwise in $\cA(y,z(y))$, the point $s'$ moves counter-clockwise in $\cA(z(y),y)$.  According to Lemma~\ref{L-L34b}, Assertions~(ii) and (iii), $\xi_s$ tends to $u_y$ when $s$ tends to $y$ or to $z(y)$ in $\cA(y,z(y))$. Looking at the behavior of the nodal sets, we reach a contradiction as Figure~\ref{F-hmn3-L35a-4} shows. Assertion~(i) is proved. \quad \qedc \medskip

\begin{figure}[!ht]
  \centering
  \includegraphics[width=0.9\textwidth]{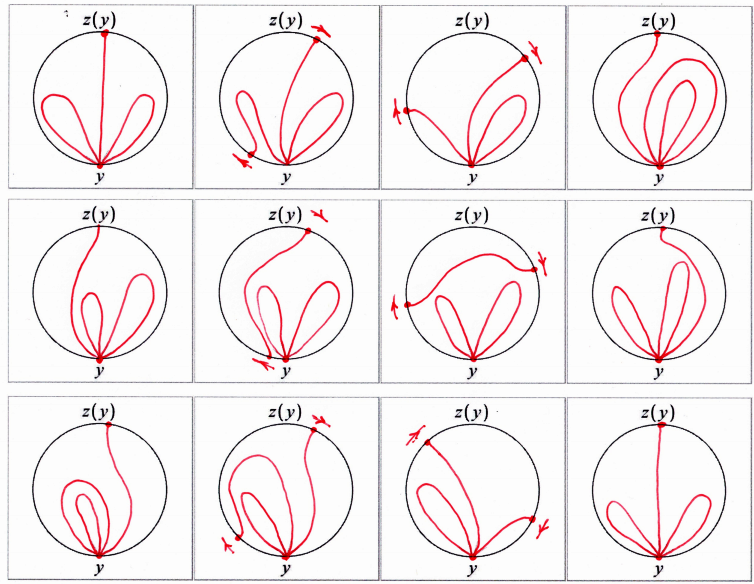}
  \caption{Proof of Lemma~\ref{L-L35a}(i)}\label{F-hmn3-L35a-4}
\end{figure}

\begin{remark} The previous arguments implicitly use the fact that the combinatorial type of $v_{y,s}$ does not change when $y$ is fixed and $s$ varies in $\Gamma  \sm \set{y}$ (the proof is similar to the proof of Lemma~\ref{L-L33c}\,(i)),
see Subsection~\ref{SSS-hmn32bas}. Note also that the reasoning in Figure~\ref{F-hmn3-L35a-4} is actually quite general, and only depends on the position of the arc from $y$ to $z(y)$ with respect to the loops.
\end{remark}

\begin{figure}[!ht]
  \centering
  \includegraphics[scale=0.7]{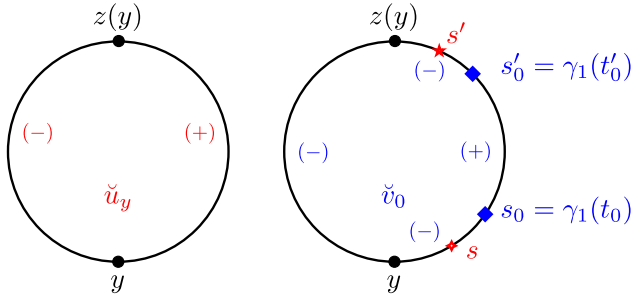}
  \caption{Proof of Lemma~\ref{L-L35a}(ii)\\}\label{F-hmn3-L35a-6}
\end{figure}

\emph{Assertion~(ii).} Let $\gamma_1 : [0,\ell_1] \to \Gamma $ be an arc-length parametrization of $\Gamma $, such that $\gamma_1(t)$ moves counter-clockwise, and $\gamma_1(0) = \gamma_1(\ell_1) = y$, $\gamma_1(\ell) = z(y)$. \smallskip

According to Lemma~\ref{L-L34b}\,(i), and to the above Assertion~(i), taking $s \in \cA(y,z(y))$ close to $y$, we have $\cS_{\mathrm{b}}(v_{y,s}) = \set{y,s,s'}$ with $s' \in \cA(y,z(y))$ close to $z(y)$. If $s=\gamma_1(t)$ and $s'=\gamma_1(t')$, for $t$ positive small enough, we have $t < t'$. Choose any such point $s_0 = \gamma_1(t_0)$ such that $v_0 = v_{y,s_0}$ satisfies $\cS_{\mathrm{b}}(v_0) = \set{y,s_0,s'_0}$, with $s'_0 = \gamma_1(t'_0)$ and $t_0 < t'_0$. By Lemma~\ref{L-L34b}\,(i), the function $\breve{v}_0$ vanishes precisely at the points $s_0$ and $s'_0$ and changes sign at these points. We choose it so that it is positive on $\cA(s_0,s'_0)$. Define $\xi_s$ as in the proof of Assertion~(i) (now with $s_0' \in \cA(y,z(y))$). Take $s \in \cA(y,s_0)$, $s=\gamma_1(t)$ with $0 < t < t_0$. With our previous choice of signs for $\breve{w}_y$, we find that $\breve{\xi}_s(z(y)) > 0$ and $\breve{\xi}_s(s'_0) < 0$, so that $\breve{\xi}_s$ vanishes at some $s'$ in $\cA(s'_0,z(y))$. Since $\xi_s \in V_{y,s}$, we have $\cS_{\mathrm{b}}(\xi_s) = \set{y,s,s'}$, with $s \in \cA(y,s_0)$, $s' \in \cA(s'_0,z(y))$, $\rho(\xi_s,s)=\rho(\xi_s,s')=1$. Then, $s' = \gamma_1(t')$ with $0 < t < t_0 < t'_0 < t' < \ell$, so that the map $(0,t_0) \ni t \mapsto t'$ is decreasing.\smallskip

Introduce the set
\begin{equation*}
J := \set{t \in (0,\ell) \mid \cS_{\mathrm{b}}(v_{y,\gamma_1(t)}) = \set{y,\gamma_1(t),\gamma_1(t')} \text{~with\quad} t < t'}.
\end{equation*}

If $t \in J$, $\rho(v_{y,\gamma_1(t)},\gamma_1(t)) = \rho(v_{y,\gamma_1(t)},\gamma_1(t')) = 1$. We have $(0,t_0) \subset J$ so that $J \neq \emptyset$, and hence $s_1 := \sup J$ exists. Since $\rho(v_{y,s},s)=1$ is an open condition, $s_1 \not \in J$. Take a subsequence $\set{t_n} \subset J$ tending to $s_1$, and choose corresponding functions $v_{y,t_n}$ in $\bS(U) \cap V_{y,t_n}$. A subsequence converges to some function $v_1$ in $\bS(U)$ which satisfies $\rho(v_1,y) \ge (2k-4)$ and $\rho(v_1,s_1) \ge 1$. By Lemma~\ref{RL-L34a},  this implies that  $v_1 \in V_{y,s_1}$ (use Remark~\ref{R-evp2-2}). Since $s_1 \not \in J$, we must have $\rho(v_1,s_1)= 2$, so that $v_1$ is a generator $v_{y,s_1}$ of $V_{y,s_1}$. \medskip

\emph{Assertions~(iii) and (iv).} Take $v_1 = v_{y,s_1}$ given by Assertion (ii). By Lemma~\ref{L-L34b}\,(i), we may choose the function $\breve{v}_1$ to be non-negative  and vanishing only at $y$ and $s_1$.  For $s\neq y, z(y)$, introduce the functions,
\begin{equation*}
\xi_s = a_1(s) u_y - b_1(s) v_1,
\end{equation*}
where
\begin{equation*}
\left\{
\begin{array}{l}
a_1(s) = \breve{v}_1(s) \big( \breve{v}_1^2(s) + \breve{u}_y^2(s)
\big)^{-\frac 12},\\[5pt]
b_1(s) = \breve{u}_y(s) \big( \breve{v}_1^2(s) + \breve{u}_y^2(s)
\big)^{-\frac 12}.
\end{array}
\right.
\end{equation*}

\begin{figure}[!ht]
  \centering
  \includegraphics[scale=0.7]{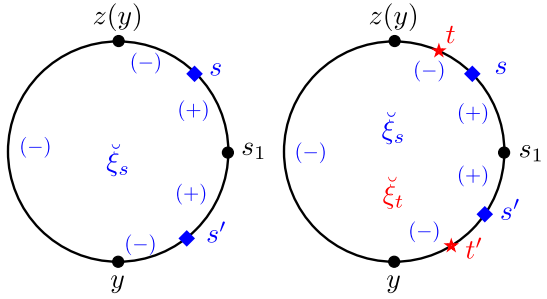}
  \caption{Proof of Lemma~\ref{L-L35a}, Assertion~(iii)}\label{F-hmn3-L35a-8}
\end{figure}

\noi When $s = s_1$, we have $\breve{\xi}_{s_1} = - b_1(s_1) \breve{v}_1 \le 0$, $\breve{\xi}_{s_1}$ only vanishes at $y$ and $s_1$, and $\xi_{s_1}$ belongs to $V_{y,s_1}$. When $s \in \cA(y,z(y))\sm \set{s_1}$, taking into account the properties of the functions involved, we find that
\begin{equation*}
\left\{
\begin{array}{l}
\breve{\xi}_s(z(y)) < 0,\\[5pt]
\breve{\xi}_s(s_1) > 0,\\[5pt]
\breve{\xi}_s < 0 \text{~in\quad} J_y\sm\set{y},\\[5pt]
\end{array}%
\right.
\end{equation*}
where $J_y$ is a small arc on $\Gamma_1$ centered at $y$.\smallskip

It follows that, for $s \in \cA(y,z(y))\sm \set{s_1}$, $\xi_s \in V_{y,s}$, $\cS_{\mathrm{b}}(\xi_s) = \set{y,s,s'}$, $\rho(\xi_s,s) = \rho(\xi_s,s') = 1$, with the point $s$ on one side of $s_1$ in $\cA(y,z(y))$, and the point $s'$ on the other side. For these points $s, s'$, we have $V_{y,s} = V_{y,s'}$, so that $\xi_s$ and $\xi_{s'}$, must be proportional (Lemma~\ref{RL-L34a}\,(i)),  and since the functions $\breve{\xi}_s$ and $\breve{\xi}_{s'}$ both take a positive value at $s_1$, $\xi_{s'} = a \, \xi_{s}$, with $a > 0$. We also have $\breve{\xi}_s(t) \breve{\xi}_t(s) \le 0$, since
\begin{equation*}
\breve{\xi}_s(t) \breve{\xi}_t(s) = - \big( \breve{v}_1^2(s) +
\breve{w}_y^2(s) \big)^{-\frac 12}\big( \breve{v}_1^2(t) +
\breve{w}_y^2(t) \big)^{-\frac 12}\big( \breve{v}_1(s)
\breve{w}_y(t)  - \breve{w}_y(s) \breve{v}_1(t) \big)^2.
\end{equation*}

Choose $s \in \cA(s_1,z(y))$, and $t \in \cA(s,z(y))$. Then $s',t' \in \cA(y,s_1)$.  The functions $\breve{\xi}_s$ and $\breve{\xi}_t$ are positive at $s_1$, and hence positive respectively on the arcs $\cA(s',s)$ and $\cA(t',t)$. We have $\breve{\xi}_s(t) < 0$, and hence, using the above properties, $\breve{\xi}_{s'}(t) > 0$, and $\breve{\xi}_t(s') < 0$. This implies that $t' \in \cA(y,s')$. We have proved that when $s$ moves counter-clockwise in $\cA(s_1,z(y))$, $s'$ move clockwise in $\cA(y,s_1)$. This is coherent with Lemma~\ref{L-L34b}(ii). The proof of Assertion~(iv) is similar.
\end{proof}

\begin{remarks} ~(i) Lemma~\ref{L-L35a} corresponds to the first part of the proof of Lemma~3.5 in \cite{HoMN1999} (from p.\,1181, line (-7), ``We consider the function'' to p.\,1182, line (+5), ``the following nodal domain'').\smallskip

(ii)  Figure~\ref{F-hmn3-L35a-9a} displays the possible nodal patterns of $u_y$, and the corresponding nodal patterns for the function $v_{y,s_1}$ with $s_1 \in \cA(y,z(y))$ (see  Lemma~\ref{L-L35a}\,(ii)), and for the function $v_{y,s}$  with $s \in \cA(s_1,z(y))$ (see Lemma~\ref{L-L35a}\,(iii)).
\end{remarks}%

\begin{figure}[!ht]
  \centering
  \includegraphics[width=0.8\linewidth]{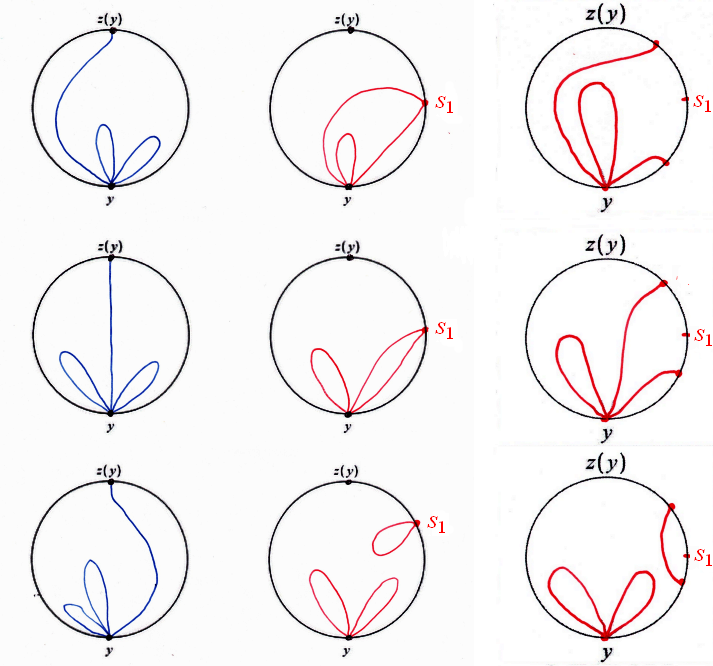}
  \caption{Lemma~\ref{L-L35a}\,(ii) and (iii): examples of nodal patterns for $u_y$ and $v_{y,s_1}$}\label{F-hmn3-L35a-9a}
\end{figure}
\FloatBarrier

\section{Conclusion}\label{S-comp}

The lemmas in Sections~\ref{S-hmn3}, \ref{S-gam0}, and \ref{S-hmn9}  explain, expand or provide proofs for the lemmas and certain statements from \cite{HoMN1999}, Section~3. Although we have mainly followed the ideas sketched in \cite{HoMN1999}, the way we organize the lemmas and some of our proofs are different. Correspondences are given in the table below.\smallskip

In \cite[Lemma~3.5]{HoMN1999}, the authors have implicitly assumed, without proof, that the set $\Gamma_{(2k-2)}$ contains at least two elements. We take care of this issue in Lemmas~\ref{L-L33c} and \ref{L-gam0-5}. In their Lemma~3.6, they mention, without proof, that if $y$ moves clockwise then $z(y)$ moves counter-clockwise. We prove this assertion in Lemma~\ref{L-L33b}.\smallskip

In the final steps of their proof, \cite[p.~1185, line (-5)]{HoMN1999}, the authors claim that the star at $x$ rotates by some positive quantity when passing on $\gamma_s$ over the point $\eta \in \Gamma_{(2k-2)}$. Our interpretation is rather that the star rotates when $x$ moves on $\gamma_s$, from above a point $\eta_1 \in \Gamma_{(2k-2)}$ to above the next point $\eta_2 \in \Gamma_{(2k-2)}$ (see Subsection~\ref{SS-hmn92}).

\begin{table}[!htb]
\centering
\begin{tabular}[c]{|c|c|}%
\hline
\cite{HoMN1999}, Section~3.9 & This Chapter~\ref{Ch-scpdsb} \\
\hline
Lemma p. 1178 & Lemmas~\ref{L-L32} and \ref{L-L33} \\
\hline
Lemma 3.3 & Lemmas~\ref{L-L32}, \ref{L-L33} and \ref{L-L33c} \\
\hline
Lemma 3.4 & Lemma~\ref{L-L34a} and Section~\ref{S-hmn32b} \\
\hline
Lemma 3.5 & Lemmas~\ref{L-L33b} and  \ref{L-L33c} \\
\hline
$\Gamma_{(2k-2)} \neq \emptyset$ & Lemma~\ref{L-gam0-5}\\
(overlooked in Lemma~3.5) &  \\
\hline
$\#(\Gamma_{(2k-2)}) \neq 1 $ & Lemma~\ref{L-L33c} \\
actually positive and even & Lemma~\ref{L-L33c} \\
(overlooked in Lemma~3.5) &\\
\hline
Lemma 3.6 & Lemma~\ref{L-L33b} and \ref{L-L33c}\\
\hline
Lemma 3.7 & Lemma~\ref{L-L37} \\
\hline
Lemma 3.8 & Lemma~\ref{L-L38} \\
\hline
Section 3.9 p.~1184 ff & Subsections~\ref{S-gam0} and \ref{S-hmn9} \\
\hline
\end{tabular}
\vspace{8pt}
\caption{ Correspondences between statements}\label{T-comp}
\end{table}

\section{Simpler proof of Lemma~\ref{L-L38}}\label{S-spL38}

In this section, we provide simpler statements and a simpler proof of Lemma~\ref{L-L38}.\medskip

Simpler statements:\\
(i) If $x \in \Omega$ tends to $y \in \Gamma$, then $[w_x]$ tends to $[u_y]$  in $\mathbb P(U)$.\\
(ii) $y \in \Gamma_{(2k-3)}$ if and only if there exists $z\neq y$ in $\Gamma$ such that for all $\beta$ small enough, and all $x$ close enough to $y$, $\cS_{\mathrm{b}}(w_x) \cap \cA(y;\beta) = \set{y(x)}$ and $\cS_{\mathrm{b}}(w_x) \cap \cA(z;\beta) = \set{z(x)}$.\\
(iii) $y \in \Gamma_{(2k-2)}$ if and only if  for all $\beta$ small enough, and all $x$ close enough to $y$, $\cS_{\mathrm{b}}(w_x) \subset \cA(y;\beta)$.\medskip

\begin{proof} Let $\set{x_n} \subset \Omega$ be a sequence tending to $y \in \Gamma$. We work locally in a neighborhood of $y$, and actually in $\bH$ via a conformal map.\smallskip

Let $w_n := w_{x_n} = \sum_{j=1}^{2k-2} a_{n,j} \phi_j \in \bS(U)\cap W_{x_n}$. Without loss of generality, we may assume that this sequence converges to some $w \in \bS(U)$, $C^m$-uniformly for any given $m \ge 1$. Since $\ord(w_n,x_n) = (k-1)$ for all $n$, $\ord(w,y) \ge (k-1) \ge 2$, so that $y$ is a singular point of $w$. Define $q := \rho(w,y) \ge 1$. Apply the local structure theorem to $w$ at the point $y$: for some $r$ smaller than the energy radius (Lemma~\ref{L-hmn3-L38-e}), the nodal set $\cZ(w)$ intersects the semi-circle $C_{+}(y,r)$ at $q$ points $A^{w}_{y,j}(r), 1 \le j \le q$. The proof of the local structure theorem implies that for $n$ large enough the function $w_n$ vanishes precisely once in the $\cG$-arcs $\cG^{w}_{y,j}(r) = \cA \big( A^{w}_{y,j}(r),\alpha \big) \subset C_{+}(y,r)$, and does not vanish elsewhere on $C_{+}(y,r)$. This implies that $\# \big( \cZ(w_n) \cap C_{+}(y,r)\big) = q$.\smallskip

According to Lemma~\ref{L-L37}, there are three possibilities
\begin{enumerate}[(A)]
  \item $\cS_{\mathrm{b}}(w_n) = \emptyset$.
  \item $\cS_{\mathrm{b}}(w_n) = \set{z_{n,1},z_{n,2}}$ with $z_{n,1} \neq z_{n,2}$ and $\rho(w_n,z_{n,i})=1$.
  \item $\cS_{\mathrm{b}}(w_n) = \set{z_n}$ with $\rho(w_n,z_n)=2$.
\end{enumerate}

\noi In the three cases, $\cZ(w_n)$ contains at least $(k-2)$ pairwise distinct loops at $x$ which do not intersect away from $x$. When $n$ is large enough, for energy reasons, each loop in $\cZ(w_n)$ intersects $C_{+}(y,r)$ at at least two distinct points.  It follows that this number satisfies $q = \# \big( \cZ(w_n) \cap C_{+}(y,r)\big) \ge (2k-4)$.\smallskip

\noi If there exists an infinite sub-sequence $\set{w_{s(n)}}$ satisfying (A), each $\cZ(w_{s(n)})$ contains $(k-1)$ loops and hence $q \ge (2k-2)$ and we must have $w \in U_y$ and $q = (2k-2)$.\smallskip

\noi Otherwise, for $n$ large enough, $\cS_{\mathrm{b}}(w_n) \neq \emptyset$ and there are two nodal intervals in $\cZ(w_n)$ from $x_n$ to the boundary. For energy reasons these intervals cannot both be contained in $D_{+}(y,r)$ and at least one of them must exit $D_{+}(y,r)$ so that $q \ge (2k-4) + 1 = (2k-3)$. This inequality implies that $w \in U_y$ and that $q \in \set{2k-3,2k-4}$.\smallskip

In summary, at this point, we have proved that when $x_n$ tends to $y$, any limit point $w$ of the sequence $\set{w_n}$ belong to $U_y$ and since $\dim U_y =1$, we conclude that $[W_n]$ tends to $[U_y]$.\medskip

We can now make a more precise analysis, looking at whether $y \in \Gamma_{(2k-3)}$ or $y \in \Gamma_{(2k-2)}$.\smallskip

\noi Assume that $y \in \Gamma_{(2k-3)}$. In that case, a limit point $w$ of $\set{w_n}$ belong to $U_y$ and satisfies $\cS(w) = \set{y,z}$ for some $z \neq y$, with $\rho(w,y) = (2k-3)$ and $\rho(w,z)=1$. Since $w_n$ tends to $w$ $C^1$-uniformly, $\breve{w}_n$ tends to $\breve{w}$. Since $\breve{w}$ changes sign at $y$ and $z$, $\breve{w_n}$ must change sign near $y$ and near $z$, so that $\cS_{\mathrm{b}}(w_n)$ contains precisely two points and belongs to Case (B). Fixing  some $\beta > 0$ small enough  we may choose $z_{n,1} \in \cA(y;\varepsilon)$ and  $z_{n,2} \in \cA(z,\beta)$.\medskip

\noi Assume that $y \in \Gamma_{(2k-2)}$. In that case, $\cS_{\mathrm{b}}(w) = \set{y}$. Assume that there exists an infinite sequence such that $\cS_{\mathrm{b}}(w_n) = \set{z_{n,1},z_{n,2}}$, possibly with $z_{n,1}=z_{n,2}$. We may assume that
$z_{n,1}$ tends to $z_1$ and $z_{n,2}$ tends to $z_2$. Since $\breve{w}_n$ tends uniformly to $\breve{w}$, we have $\breve{w}(z_1) = \breve{w}(z_2) = 0$ and since $\breve{w}$ only vanishes at $y$, we conclude that $z_1=z_2$.\smallskip

In summary, if $x_n$ tends to $y \in \Gamma_{(2k-2)}$, and if $w$ is a limit point of $\set{w_n}$ then for all $\beta$, there exists $N_{\beta}$ such that $\cS_n(w_n) \subset \cA(y;\beta)$, including the case $\cS_{\mathrm{b}}(w_n) = \emptyset$.
\end{proof}

\chapter{Further Results}\label{Ch-fr}

\section{Upper Bounds on the Multiplicities and the Nodal Line Conjecture}\label{S-nlc}

When $k=2$, Proposition~\ref{P-hmn-s1} gives $\mult(\lambda_2) \le 3$.  In view of Table~\ref{E-scr-2T}, Section~\ref{S-sketch}, a natural question is whether this bound is sharp (depending on the boundary condition). As we shall see, this question is related to the so-called  ``nodal line conjecture''.

\subsection{Nodal sets of second eigenfunctions}\label{SSS-hmn-0CA}

 We use the notation of Subsection~\ref{SS-hmn-0N}, and write the boundary of the domain $\Omega$ as $\Gamma = \bigcup_{j=1}^q \Gamma_j$, with $q \ge 1$. Let $u \in U(\lambda_2)$ be any second eigenfunction. By Courant's Theorem~\ref{T-RC},  $u$ has exactly two nodal domains. Euler's formula \eqref{E-euler-or2}  yields
\begin{equation}\label{E-hmn0-22}
\left\{
\begin{array}{ll}
0 = \kappa(u) - 2 = & \big[ b_0(\cZ(u) \cup \Gamma(u)) - 1 \big]  + \frac 12
\sum_{z \in \cS_{\mathrm{i}}(u)} (\nu(u,z)-2)\\[5pt]
& + \sum_{j \in J(u)}\frac 12 \, \big( \sum_{z\in \cS_{\mathrm{b}}(u) \cap \Gamma_j\,} \rho(u,z) - 2\big).
\end{array}%
\right.
\end{equation}

 Since all the terms is the right hand side are nonnegative (use Corollary~\ref{cor:nodinfo} and the definition of $J(u)$), we immediately deduce that
\begin{equation}\label{E-hmn0-24}
\left\{
\begin{array}{l}
\cZ(u) \cup \Gamma(u) \text{~is connected},\\[5pt]
\sum_{z \in \cS_{\mathrm{i}}(u)} (\nu(u,z)-2) = 0 \text{~~i.e,~} \cS_{\mathrm{i}}(u) = \emptyset,\\[5pt]
\sum_{z\in \cS_{\mathrm{b}}(u) \cap \Gamma_j\,} \rho(u,z) = 2 ~~\forall j \in J(u).
\end{array}%
\right.
\end{equation}

 The structure of $\cZ(u)$ depends on whether $J(u) = \emptyset$ or $J(u) \neq \emptyset$. We now consider the two simplest situations. The proofs of the following properties are clear.

\begin{property}\label{P-hmn-nl0}
Assume that $\Omega$ is \emph{simply connected}. Let $u$ be a second eigenfunction. Then, either $\cZ(u)$ does not hit $\Gamma$ and $\cS(u) = \emptyset$, or $\cZ(u)$ hits $\Gamma$, $\cS_{\mathrm{i}}(u) = \emptyset$, and $\sum_{z\in \cS_{\mathrm{b}}(u)} \rho(u,z) = 2$. More precisely, there are three distinct possibilities.
\begin{enumerate}[(a)]
  \item If $J(u) = \emptyset$, then $\cZ(u)$ is a nodal circle, i.e., a simple closed regular  connected curve contained in $\Omega$, not touching $\Gamma$. This case is characterized by the fact that the function $\breve{u}$  defined in \eqref{E-evp-dr} does not vanish on $\Gamma$.
  \item If $J(u) = \set{1}$ and $\cS_{\mathrm{b}}(u) = \set{y}$ for some $y \in  \Gamma$ with $\rho(u,y)=2$, then $\cZ(u)$ is a nodal loop at $y$, i.e., $\cZ(u) \sm \set{y}$ a simple regular connected curve contained in $\Omega$.  This case is characterized by the fact that the function $\breve{u}$ vanishes only at $y_1$ on $\Gamma$, and does not change sign.
  \item If $J(u) = \set{1}$ and $\cS_{\mathrm{b}}(u) = \set{y_1,y_2}$ for some $y_1 \neq y_2 \in  \Gamma$ with $\rho(u,y_1)=1$, $\rho(u,y_2)=1$, then $\cZ(u)$ is a nodal arc from $y_1$ to $y_2$, i.e., $\cZ(u) \sm \set{y_1,y_2}$ is a simple regular connected arc contained in $\Omega$.  This case is characterized by the fact that the function $\breve{u}$ vanishes precisely at $y_1$ and $y_2$ on $\Gamma$, and changes sign at these points.
\end{enumerate}
\end{property}%

\begin{property}\label{P-hmn-nl1}
Assume that $\Omega$ has \emph{one hole}, and write $\Gamma = \Gamma_1 \cup \Gamma_2$. Let $u$ be a second eigenfunction. Then, either $\cZ(u)$ does not hit $\Gamma$ and $\cS(u) = \emptyset$, or $\cZ(u)$ hits $\Gamma$, $\cS_{\mathrm{i}}(u) = \emptyset$, and $\sum_{z\in \cS_{\mathrm{b}}(u)\cap \Gamma_j} \rho(u,z) = 2$ for all $j \in J(u)$. More precisely, there are three distinct cases (up to relabeling the two components of $\Gamma$).
\begin{enumerate}[(1)]
  \item If $J(u) = \emptyset$, then $\cZ(u)$ is a nodal circle contained in $\Omega$, not touching $\Gamma$. This case is characterized by the fact that the function $\breve{u}$  defined in \eqref{E-evp-dr} does not vanish on $\Gamma$.
  \item If $J(u) = \set{1}$, then either $\cS_{\mathrm{b}}(u) = \set{y}$ for some $y \in  \Gamma_1$ with $\rho(u,y_1)=2$, or $\cS_{\mathrm{b}}(u) = \set{y_1,y_2}$ for some $y_1 \neq y_2 \in  \Gamma_1$ with $\rho(u,y_1)=\rho(u,y_2)=1$. We have either a nodal loop at $y$, or a nodal arc from $y_1$ to $y_2$.  This case is characterized by the fact that the function $\breve{u}$ vanishes only at $y$ (without changing sign along $\Gamma_1$), or vanishes at $y_1$ and $y_2$ (and changes sign along $\Gamma_1$). In both subcases, $\breve{u}$ does not vanish on $\Gamma_2$.
  \item If $J(u) = \set{1,2}$,  then $\cZ(u)$ hits both component $\Gamma_1$ and $\Gamma_2$ at one point with index $2$, or at two distinct points of index $1$. Furthermore, the components $\Gamma_1$ and $\Gamma_2$ are linked by two nodal arcs (possibly with one or two common boundary points).
\end{enumerate}
\end{property}%

\begin{remark}\label{R-hmn0-24}\phantom{}
It is not clear  a priori whether the possible nodal patterns described in Property~\ref{P-hmn-nl0} or \ref{P-hmn-nl1} are actually realized for some choice of domain $\Omega$ and potential $V$.  Applying Lemmas~\ref{L-zero1} or \ref{L-zero2}, one can at least prescribe one or two boundary singular points.
\begin{enumerate}[(i)]
\item Assume that $\dim U(\lambda_2(-\Delta+V)) \ge 3$. If $\Omega$ is simply connected, then there exists an eigenfunction whose nodal set satisfies (b), resp. (c), in Properties~\ref{P-hmn-nl0}. If $\Omega$ has one hole, then there exists an eigenfunction whose nodal set hits both $\Gamma_1$ and $\Gamma_2$.
  \item Assume that $\dim U(\lambda_2(-\Delta+V)) = 2$. Then, there exists an eigenfunction whose nodal domain hits $\Gamma$.
\end{enumerate}
\end{remark}%

Nodal sets and $\mult(\lambda_2)$ are known precisely in few circumstances only. For example, when $\Omega$ is convex and $V\equiv 0$, see \cite{Ales1994} and its list of references. The following figures display some particular cases. The second (Dirichlet or Neumann) eigenvalue of an equilateral triangle with rounded corners has multiplicity two, with one symmetric and one antisymmetric eigenfunction. The nodal domains of the symmetric eigenfunction appear in Figure~\ref{F-hmn1-nl-teeh} (left), see \cite{BeHe2021t}.

The second (Dirichlet or Neumann) eigenvalue of the interior of an ellipse is simple with nodal domains as in Figure~\ref{F-hmn1-nl-teeh} (center); this is a particular case of the domains described in \cite{Shen1988}, \cite{Putt1990} and \cite{Putt1991}. The nodal set of a second Dirichlet eigenfunction of $D\sm B$, where $D, B$ are convex symmetric domains, has been studied in \cite{Kiwa2018}, see Figure~\ref{F-hmn1-nl-teeh} (right).

\begin{figure}[!ht]
\centering
\includegraphics[width=0.9\linewidth]{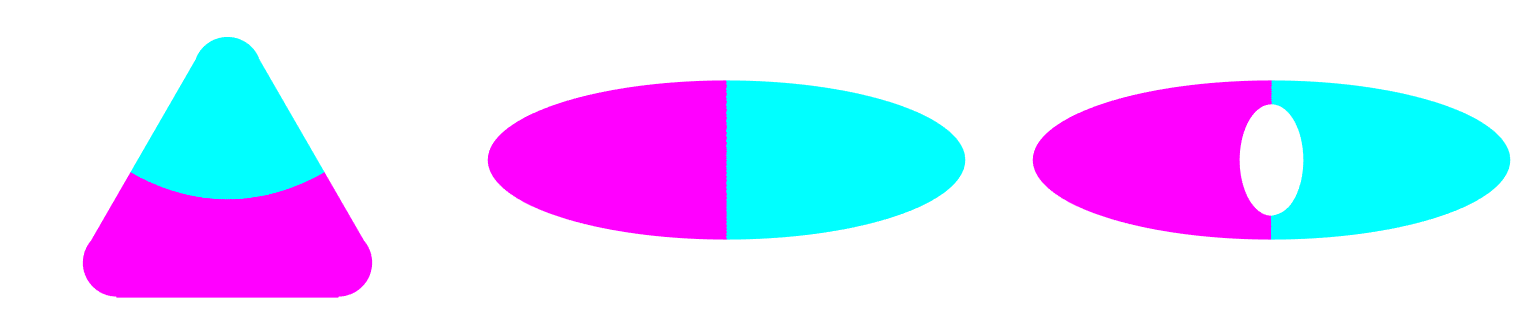}
\caption{Nodal patterns of second eigenfunctions}
\label{F-hmn1-nl-teeh}
\end{figure}

Numerical computations, playing with the position of the holes, give rise to
some other patterns, see Figure~\ref{F-hmn1-nl-eh}.

\begin{figure}[!ht]
\centering
\includegraphics[width=0.99\linewidth]{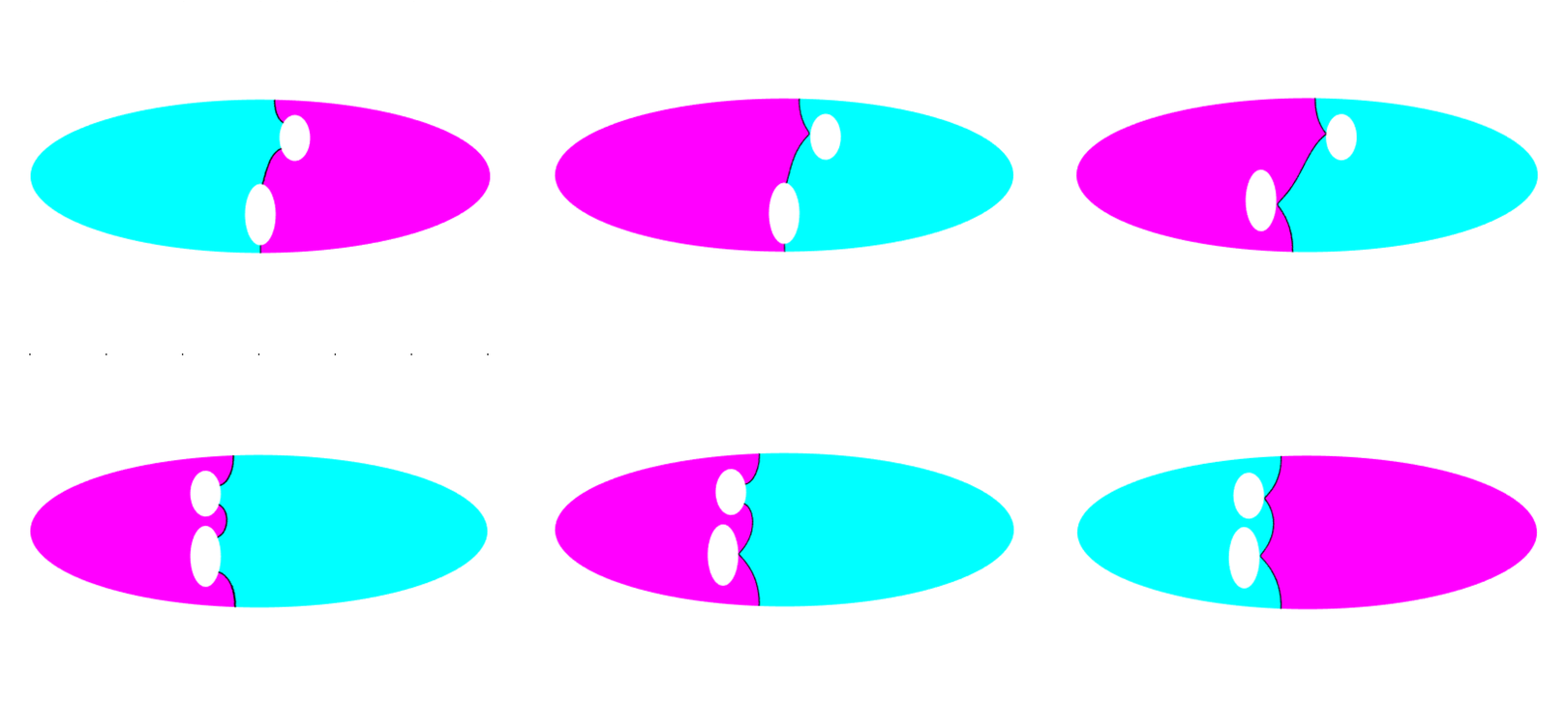}
\caption{Nodal patterns of second eigenfunctions}
\label{F-hmn1-nl-eh}
\end{figure}

In view of the above remarks, the following questions are natural.

\begin{description}
  \item[Question~1:] Are there examples in which the nodal set of some second eigenfunction of $-\Delta+V$ is a nodal circle ?
  \item[Question~2:] Are there examples in which the nodal set of some second eigenfunction of $-\Delta+V$ is a nodal loop at some $y \in \Gamma$ ?
\end{description}

\subsection{The nodal line conjecture}\label{SSS-hmn-0CB}

For a simply connected domain $\Omega$, Pleijel \cite[p.~546]{Plej1956} observed that a second Neumann eigenfunction of $-\Delta$  cannot have a closed nodal line. \index{Nodal line conjecture} The idea is to use the inequality \[\lambda_2(D,-\Delta,\mf{n}) \le \lambda_1(D,-\Delta,\mf{d})\] between the second Neumann eigenvalue and the first Dirichlet eigenvalue of $-\Delta$ in a domain $D \subset \R^2$, and the monotonicity property of the first Dirichlet eigenvalue. Indeed, if there exists a second Neumann eigenfunction $u_2$ one of whose two nodal domains (say $D_1$) has the property that $\partial D_1 \cap \Gamma$ consists of isolated points  (which occurs in particular when $\partial D_1 \subset \Omega$), then the restriction of $u_2$ to $D_1$ is the ground state of the Dirichlet problem in $D_1$, $\lambda_1(D_1,\mf{d}) = \lambda_2(\Omega,\mf{n})$. By the strict domain monotonicity of the Dirichlet eigenvalues, we get
\[
\lambda_1(\Omega,\mf{d}) <  \lambda_1(D_1,\mf{d})=\lambda_2(\Omega,\mf{n}),
\]
hence a contradiction. This argument may fail when $\Omega$ as a hole because both nodal domains of $u$ may touch the boundary.\smallskip

The inequality $\lambda_2(D,-\Delta,\mf{n}) \le \lambda_1(D,-\Delta,\mf{d})$ is due to Szeg\"{o} (1954) when $\Omega \subset \R^2$ is simply connected and smooth.  The strict inequality is proved in an earlier paper of P\'{o}lya (1952) which does apparently not use the assumption that the domain is simply connected. It was later generalized to higher dimensions, and smooth enough domains (not necessarily simply connected) by Weinberger (1956), see \cite{Payn1967}, Theorem~3, p.~463. It has been extended to domains with $C^1$ boundary, see \cite{Maz1991} in which Mazzeo revisits the earlier paper of L.~Friedlander \cite{Frie1991}.  When $\partial D_1$ meets $\Gamma$ we actually need the inequality to hold for domains with piecewise $C^1$ boundary   (see \cite{ArMa2007,ArMa2012} for the Lipschitz case). \smallskip

Concerning the inequalities between Dirichlet and Neumann eigenvalues, we point out the recent papers \cite{FrGoMi2025} and \cite{HuMuZh2024}, their introductions and lists of references. \smallskip

In view of Pleijel's observation, Payne conjectured that a second Dirichlet eigenfunction of $-\Delta$ cannot have a closed nodal line, see \cite{Payn1967}, \index{Payne conjecture} Conjecture~5, p.~467. One can make a similar conjecture for second Robin eigenfunctions,  and also consider domains in $\R^n$, $n\ge 3$,  see for example \cite{Dama2000}, \cite{Four2001}, \cite{FrKr2008}, \cite{Jeri1991}, \cite{Ken2013}, \cite{MuSa2022}, \cite{MuSa2025}, \cite{MaMuSa2025}, and their bibliographies.

\begin{remark}\label{R-hmn0-26}
As observed in \cite{Liq1995} (end of Section~2, p.~277), if a simply connected domain satisfies the \emph{nodal line conjecture}, then $\dim U(\lambda_2) \le 2$. This is an immediate consequence  of Remark~\ref{R-hmn0-24}(ii).
\end{remark}

In the next subsections, we consider the three boundary conditions, Dirichlet, Neumann and Robin separately.

\subsection{Dirichlet boundary condition}\label{SSS-hmn-0CD}

The following results are due to Lin and Ni. \index{Multiplicity}
\begin{enumerate}[$\diamond$]
\item \cite[Theorem~3.6]{LiNi1988}: For all $n \ge 2$, there exists a radius $R_n$ and a nonzero smooth radial potential $V_n$, such that
\begin{equation*}
  \mult(\lambda_2; B(R_n), -\Delta + V_n,\mf{d}) = 1,
\end{equation*}
and the corresponding eigenfunction is radial, with nodal set a sphere in $B(R_n)$. Here, $B(R_n)$ is the ball of radius $R_n$ in $\R^n$.
\item \cite[Theorem~3.8]{LiNi1988}: For all $n \ge 2$, there exists a radius $R_n$ and a nonzero smooth radial potential $V_n$, such that
\begin{equation*}
  \mult(\lambda_2; B(R_n), -\Delta + V_n,\mf{d}) = (n+1),
\end{equation*}
and there exists a radial second eigenfunction.
\end{enumerate}

In dimension $2$, the second assertion implies that the bound of the multiplicity $\mult(\lambda_2;\Omega,-\Delta+V,\mf{d}) \le 3$ is sharp.\medskip

The following results are due to M. and T. Hoffmann-Ostenhof and Nadirashvili \cite{HoHN1997, HoHN1998} (for corrections and complements on the second statement, see  \cite{HeHJ2020}).
\begin{enumerate}[$\diamond$]
\item \cite[Theorem~2.1]{HoHN1998}: There exists $N_0$ and domains $D_{N,\varepsilon} \subset \R^2$ such that for all $N \ge N_0$, and $\varepsilon$ small enough, $\lambda_2(D_{N,\varepsilon};-\Delta,\mf{d})$ is simple, with a closed nodal set contained in $D_{N,\varepsilon}$. The domain $D_{N,\varepsilon}$ is homeomorphic to a disk minus $N$ points.
  \item \cite[Theorem~2.2]{HoHN1998}: For all $N \ge 3$, and $\varepsilon$ small enough, the domains $D_{N,\varepsilon} \subset \R^2$  satisfy
\begin{equation*}
  \mult(\lambda_2; D_{N,\varepsilon}, -\Delta,\mf{d}) = 3.
\end{equation*}
\end{enumerate}

The second assertion implies that the upper bound $3$ for the second Dirichlet eigenvalue of $-\Delta$ is sharp for non simply connected domains.\medskip

In \cite{DaGH2021}, the authors give a counter-example to the nodal line conjecture for the Laplacian in a $2$-dimensional domain with six holes. In \cite{FrLe2025}, the authors construct  a counter-example in a \emph{doubly-connected} domain. The introduction of their paper provides a detailed up-to-date overview of the subject. \smallskip

The nodal line conjecture remains open for the Laplacian in simply-connected planar domains. This is in contrast with the results of \cite{LiNi1988} mentioned above and with \cite{Fre2002}. This is also in contrast with the higher dimensional case. Indeed, \cite{Ken2013} provides a simply-connected counter-example in dimension larger than or equal to $3$, based on \cite{Four2001} which generalizes \cite{HoHN1997} in dimension greater than or equal to $2$. \smallskip

The nodal line conjecture for $(\Omega,-\Delta)$ is known to be true when $\Omega$ is a bounded \emph{convex domain} in $\R^2$: \cite{Payn1973} and \cite{Lin1987} (under additional symmetry assumptions), \cite{Mela1992} (smooth convex domains) and \cite{Ales1994} (general convex domains).  This is also the case for domains which are convex in one direction only\footnote{See \cite[Corollary~2.7]{Liq1995}. Note however that the other results on the multiplicity presented in this paper are true under strong additional conditions only.  We thank the author for clarifying this point.}.\medskip

As a by-product of these results, we have the upper bound
\begin{equation*}
\mult(\lambda_2;\Omega,-\Delta,\mf{d}) \le 2 \text{~for any convex bounded domain~} \Omega.
\end{equation*}

This bound holds for domains which are convex in one direction and for domains which satisfy the nodal line conjecture (see Remark~\ref{R-hmn0-26}). This result supports the following conjecture.

\begin{conjecture}\label{C-aaa}
\begin{equation*}
\mult(\lambda_2;\Omega,-\Delta,\mf{d}) \le 2 \text{~for any simply connected bounded domain~} \Omega.
\end{equation*}
\end{conjecture}%

The properties that the nodal set of a second Dirichlet eigenfunction does or does not intersect the boundary is in some sense stable under certain small perturbations. We refer to \cite{MuSa2022}, \cite{MuSa2025} and \cite{MaMuSa2025} for more details and new examples.

\subsection{Neumann boundary condition}\label{SSS-hmn-0CN}
The discussion in Paragraph~\ref{SSS-hmn-0CB} shows that the nodal line conjecture holds for the Neumann Laplacian in any simply connected  bounded regular domain in $\R^2$. Nadirashvili proved that the multiplicity of the second eigenvalue of a simply connected domain with nonpositive curvature is at most $2$ and that this estimate is sharp, see \cite{Nadi1987}, Theorem~2 and Corollary~1. As a matter of fact, his proof also shows that the nodal line conjecture is true in such domains (for the Neumann condition).

\subsection{Robin boundary condition}\label{SSS-hmn-0CR}
As observed by James B.~Kennedy\footnote{We thank J.B. Kennedy for
useful discussions around this problem.} in \cite{Ken2011}, the proof of the nodal line conjecture for the $h$-Robin boundary condition works
in the same way as in the case of Neumann condition provided that the following
inequality holds,
\[
\lambda_2 (h,\Omega) \leq  \lambda_1^D(\Omega).
\]
Observing the monotonicity of the Robin problem with respect to $h$,
we obtain the existence of some $h_\Omega >0$ such that this inequality
holds for $h \leq  h_\Omega$.\smallskip

J.B. Kennedy also shows that, as in the Dirichlet case (which corresponds to $h=+\infty$), one can find examples of multiply connected domains for which counter examples to the nodal line conjecture can be constructed. One can also expect to construct examples for which the multiplicity is $3$
(as in  \cite{HoHN1998} and \cite{HeHJ2020})  but this is still open
at the moment.\smallskip

On the positive side, it is natural to ask if convexity is enough to ensure multiplicity at most $2$ for every Robin parameter $h>0$, following Lin's approach in \cite{Lin1987}. This is still open at the moment. What we do know is that a sufficient  condition on $\Omega$  is that nodal line conjecture holds (Remark~\ref{R-hmn0-26}).

\section{Upper Bounds for Multiplicities vs Courant-sharp Eigenvalues}\label{S-mcs}
\index{Courant nodal Theorem}
For simplicity, let us only consider Dirichlet eigenvalues in a $\Cty$ bounded domain $\Omega \subset \R^2$. The upper bounds on the eigenvalue multiplicities strongly rely on Courant's nodal domain theorem. As observed in \cite{HoMN1999}, they are actually a consequence of the following 3-step result, $n \in \set{1,2,3}$,   provided that the third step ($n=3$) is correct.

\begin{proposition}\label{P-propC1} Let $U(\lambda)$ be a Dirichlet eigenspace of $-\Delta + V$. Assume that \[\sup\set{\kappa(u) \mid u \in U(\lambda)} = \ell\]
for some $\ell \ge 3$. Then $\dim(U(\lambda)) \le (2\ell -n)$.
\end{proposition}

Indeed, since $u \in U(\lambda_k)$ implies that $\kappa(u) \le k$ (Courant's theorem):
\begin{enumerate}
  \item The first step, $n=1$, yields the upper bound $\mult(\lambda_k) \le (2k-1)$, see  Theorem~1 in \cite{Nadi1987}.
  \item The second step, $n=2$, yields the upper bound $\mult(\lambda_k) \le (2k-2)$, see Lemma~2.13 in \cite{HoMN1999} or Theorem~\ref{T-hmn-bh1}.
  \item The third step, $n=3$, yields the upper bound $\mult(\lambda_k) \le (2k-3)$, see Theorem~B in \cite{HoMN1999} or Theorem~\ref{T-hmn-bh2}.
\end{enumerate}\smallskip

According to Pleijel \cite{Plej1956}, \index{Pleijel Theorem} the equality in Courant's Theorem~\ref{T-RC} can only occur for finitely many eigenvalues, the so-called \emph{Courant-sharp} eigenvalues.  The eigenvalue $\lambda_k$ is called  \emph{Courant-sharp} whenever the associated eigenspace $U(\lambda_k)$ contains an eigenfunction with $k$ nodal domains, the maximum number allowed by Courant's nodal domain theorem. If $\lambda_k$ is not a Courant-sharp eigenvalue, in particular if $k$ is large enough (depending on the geometry of the domain,  see Remark~\ref{R-mcs-3} below), $u \in U(\lambda_k)$ implies that $\kappa(u) \le (k-1)$, and the above proposition implies that $\mult(\lambda_k) \le (2k-2-n)$. When $\lambda_k$ is not Courant-sharp, the upper bound for the multiplicity can be improved by $2$.\medskip

Proposition~\ref{P-propC1} can be restated as
\begin{equation}\label{E-mcs-2}
\mult(\lambda_k) \leq 2 \sup_{u \in U(\lambda_k)} \kappa(u)\, -1.
\end{equation}
 Since $\sup_{u \in U(\lambda_k)} \kappa(u) \le k$, we obtain
\begin{equation*}\label{E-mcs-2a}
\limsup_{k\rightarrow +\infty}\frac{\mult(\lambda_k)}{k} \le 2.
\end{equation*}

We can continue the discussion a little further by recalling Pleijel's proof.  \smallskip

 \emph{Sketch of Pleijel's proof.~} Let $\Omega\subset\mathbb R^n$ be an open set of finite measure. For $k\in\N$, let $\bar{\kappa}(\lambda_k)$ be the maximal number of nodal domains of an eigenfunction corresponding to the Dirichlet eigenvalue $\lambda_k(\Omega)$. Choose some Dirichlet eigenfunction $u$ associated with $\lambda_k$ and  such that $\kappa(u) =\bar{\kappa}(\lambda_k)$. Let $\{\omega_\alpha\}_\alpha$ be the nodal domains of $u$.  The first Dirichlet eigenvalue $\lambda_1(\omega_{\alpha})$ is equal to $\lambda_k(\Omega)$ and satisfies the Faber-Krahn inequality \index{Faber-Krahn inequality}  (see \cite{BeMe1982}),
\begin{equation*}
\lambda_1(\omega_{\alpha})|\omega_{\alpha}|^{\frac 2n} \ge \lambda_1(\B_1)|\B_1|^{\frac 2n} =: F_n,
\end{equation*}
where $\B_1$ denotes the unit ball in $\R^n$, $|\Omega|$ the volume of $\Omega$, and where $F_n$ is some universal constant (see \cite{BeMe1982}. Then,
\begin{equation*}
\frac{\bar{\kappa}(\lambda_k)}{k} = \frac{\lambda_k(\Omega)^\frac n2}{k} \sum_\alpha \lambda_1(\omega_\alpha)^{-\frac n2} \leq  F_n^{-\frac n2} \frac{\lambda_k(\Omega)^\frac n2}{k} \sum_\alpha |\omega_\alpha| = F_n^{-\frac n2} \frac{\lambda_k(\Omega)^\frac n2}{k} |\Omega| .
\end{equation*}

On the other hand,  Weyl's asymptotic formula \cite[Corollary~17.5.8]{Horm2007c}
\index{Weyl asymptotic formula}
\begin{equation*}
N(\lambda) := \# \set{j\mid \lambda_j < \lambda} \sim C_{W,n}|\Omega|\lambda^{\frac n2},
\end{equation*}
where $C_{W,n}$ is a universal constant (Weyl's constant), implies that
\begin{equation*}
\lambda_k(\Omega)^{\frac 2n} \sim C_{W,n}^{-1} \, \frac{k}{|\Omega|}.
\end{equation*}
It follows that
\begin{equation}
\limsup_{k\to\infty} \frac{\bar{\kappa}(\lambda_k)}{k} \leq F_n^{-\frac n2} C_{W,n}^{-1} =: \gamma_n,
\end{equation}
and it turns out that the constant $\gamma_n$ is (strictly) less than $1$.
\smallskip

In dimension $2$, $\gamma_2 = \frac{4}{j_{0,1}^2}$ where $j_{0,1}$ is the first positive zero of the Bessel function $J_0$, and hence we obtain Pleijel's estimate \cite{Plej1956}
\begin{equation}\label{E-mcs-W2}
\limsup_{k\to\infty} \frac{\bar{\kappa}(\lambda_k)}{k} \leq \frac{4}{j_{0,1}^2} < 1.
\end{equation}%

As a consequence of Weyl's asymptotic formula, Pleijel's method for Dirichlet eigenvalues, and Equation~\eqref{E-mcs-2}, we obtain the improved estimate
\begin{equation}\label{E-mcs-10}
\limsup_{k\rightarrow +\infty} \frac{\mult(\lambda_k)}{k} \le 2 \, \gamma < 2.
\end{equation}

For a \emph{regular} bounded domain $\Omega$ in $\R^2$, Weyl's asymptotic formula reads
\begin{equation}\label{E-mcs-W}
N(\lambda) = \frac{|\Omega|}{4\pi} \, \lambda + O(\sqrt{\lambda}).
\end{equation}

For Weyl's formula with a remainder term, we refer to Theorem~29.3.3 in H\"{o}rmander's book \cite{Horm2009d} and to Ivrii's papers \cite{Ivri1980r, Ivri1980}. These references actually give a much more precise formula which yields a two-term asymptotic formula under an assumption on the set of periodic billiard trajectories \cite[Corollary~29.3.4]{Horm2009d}.
\smallskip

For any $\delta$ small enough,
\begin{equation}\label{eq:E-6}  \mult(\lambda_k) = N(\lambda_k+\delta) - N(\lambda_k - \delta).
\end{equation}
According to \eqref{E-mcs-W}, we should expect that
\[
\mult(\lambda_k) = O(\sqrt{\lambda_k}),
\]
and hence
\[
\limsup_{k\rightarrow +\infty} \frac{\mult(\lambda_k)}{k} = 0,
\]
which shows that the inequality \eqref{E-mcs-10} is not really pertinent when the domain is regular.

 For extensions and improvements of Pleijel's theorem, we refer to  \cite{Peet1957}, \cite{BeMe1982}, \cite{Bour2015} and \cite{Stei2014}.  Pleijel's method does not readily apply to Neumann eigenvalues. This is because there exist nodal domains whose boundary contain a portion of the boundary of $\Omega$. For the extension of Pleijel's theorem to the Neumann or Robin boundary condition, we refer to \cite{Polt2009}, \cite{Lena2019},  \cite{HaSh2023} and \cite{BeCM2023}.

\begin{remark}\label{R-mcs-3}
For a bounded domain $\Omega \subset \R^n$ with $C^2$ boundary, and Dirichlet boundary condition, one can show that there exists a constant $C(\Omega)$, depending on $\Omega$ and invariant under dilations, such that for $k > C(\Omega)$, the $k$th Dirichlet eigenvalue $\lambda_k$ is not Courant-sharp. In dimension $2$, the constant $C(\Omega)$ can be estimated in terms of the area $|\Omega|$, the length $|\Gamma|$ of the boundary, the curvature and the cut-distance of $\Gamma$. We refer to \cite[Theorem~1.3]{BeHe2016} for more details, and to \cite[Theorem~1]{BeGi2016} for an extension to less regular domains. The proof of this result makes use of a lower bound on the remainder term $R(\lambda) := N(\lambda) - \frac{|\Omega|}{4\pi} \, \lambda$ in Weyl's asymptotic estimate, as given for example in \cite{BeLi2001}.\\
For the case of domains with Neumann or Robin boundary condition, we refer to \cite{GiLe2020} and \cite{GiHaLeSh2024}.
\end{remark}%

\begin{remark}\label{R-mcs-6}~  In order to estimate $\mult(\lambda_k)$ asymptotically, we could use the relation \eqref{eq:E-6}  together with a geometrical control of $N(\lambda)$ as in  Safarov  \cite{Safa2001} or  Van den
Berg--Lianantonakis \cite{BeLi2001} who give estimates of the form
\[
|N(\lambda) - C_n |\Omega| \lambda^{n/2}| \leq C_{geom} (\Omega) \lambda^{(n-1)/2} \ln \lambda,
\]
where $C_n$ is Weyl's constant.
\end{remark}

\begin{remark}\label{R-subL}
Some of the above topics are being investigated in the framework of sub-Laplacians, see for example \cite{EsLe2023}, \cite{FrHeLa2024}, \cite{FrHe2025}.
\end{remark}%

\backmatter


\begin{thebibliography}{100}
\providecommand{\url}[1]{\texttt{#1}}
\providecommand{\urlprefix}{URL }
\providecommand{\eprint}[2][]{\url{#2}}

\bibitem{Ales1994}
G.~Alessandrini, \href{https://doi.org/10.1007/BF02564478}{Nodal lines of
  eigenfunctions of the fixed membrane problem in general convex domains}.
  \emph{Comment. Math. Helv.} \textbf{69} (1994), no.~1, 142--154 \MR{1259610}

\bibitem{Ales1998}
G.~Alessandrini, \href{https://doi.org/10.1515/form.10.5.521}{On {C}ourant's
  nodal domain theorem}. \emph{Forum Math.} \textbf{10} (1998), no.~5, 521--532
  \MR{1644305}

\bibitem{ArMa2007}
W.~Arendt and R.~Mazzeo, Spectral properties of the {Dirichlet}-to-{Neumann}
  operator on {Lipschitz} domains. \emph{Ulmer Seminare} \textbf{12} (2007),
  28--38

\bibitem{ArMa2012}
W.~Arendt and R.~Mazzeo,
  \href{https://doi.org/10.3934/cpaa.2012.11.2201}{Friedlander's eigenvalue
  inequalities and the {D}irichlet-to-{N}eumann semigroup}. \emph{Commun. Pure
  Appl. Anal.} \textbf{11} (2012), no.~6, 2201--2212 \MR{2912743}

\bibitem{Aron1957}
N.~Aronszajn, A unique continuation theorem for solutions of elliptic partial
  differential equations or inequalities of second order. \emph{J. Math. Pures
  Appl. (9)} \textbf{36} (1957), 235--249 \MR{92067}

\bibitem{BaPP2024}
L.~Battaglia, A.~Pistoia, and L.~Provenzano, On the critical points of
  {S}teklov eigenfunctions. 2024,
  \urlprefix\url{http://arxiv.org/abs/2402.01190v3}

\bibitem{BeCM2023}
T.~Beck, Y.~Canzani, and J.~L. Marzuola, Uniform upper bounds on {C}ourant
  sharp {N}eumann eigenvalues of chain domains. 2023,
  \urlprefix\url{https://arxiv.org/abs/2305.16452}

\bibitem{BeKr1987}
S.~R. Bell and S.~G. Krantz,
  \href{https://doi.org/10.1216/RMJ-1987-17-1-23}{Smoothness to the boundary of
  conformal maps}. \emph{Rocky Mountain J. Math.} \textbf{17} (1987), no.~1,
  23--40 \MR{882882}

\bibitem{BeHe2015s}
P.~{B\'erard} and B.~{Helffer}, {Edited extracts from Antonie Stern's thesis}.
  In \emph{{Actes de S\'eminaire de Th\'eorie Spectrale et G\'eom\'etrie.
  Ann\'ee 2014--2015}}, pp. 39--72, St. Martin d'H\`eres: Universit\'e de
  Grenoble I, Institut Fourier, 2015

\bibitem{BeHe2015r}
P.~{B\'erard} and B.~{Helffer}, {Nodal sets of eigenfunctions, Antonie Stern's
  results revisited}. In \emph{{Actes de S\'eminaire de Th\'eorie Spectrale et
  G\'eom\'etrie. Ann\'ee 2014--2015}}, pp. 1--37, St. Martin d'H\`eres:
  Universit\'e de Grenoble I, Institut Fourier, 2015

\bibitem{BeHe2016}
P.~B\'{e}rard and B.~Helffer, \href{https://doi.org/10.4171/JST/138}{The weak
  {P}leijel theorem with geometric control}. \emph{J. Spectr. Theory}
  \textbf{6} (2016), no.~4, 717--733 \MR{3584180}

\bibitem{BeHe2021t}
P.~B\'{e}rard and B.~Helffer, \href{https://doi.org/10.5802/afst.1680}{Level
  sets of certain {N}eumann eigenfunctions under deformation of {L}ipschitz
  domains application to the extended {C}ourant property}. \emph{Ann. Fac. Sci.
  Toulouse Math. (6)} \textbf{30} (2021), no.~3, 429--462 \MR{4331305}

\bibitem{BeMe1982}
P.~B\'{e}rard and D.~Meyer, In\'{e}galit\'{e}s isop\'{e}rim\'{e}triques et
  applications. \emph{Ann. Sci. \'{E}cole Norm. Sup. (4)} \textbf{15} (1982),
  no.~3, 513--541 \MR{690651}

\bibitem{BeBo1982}
L.~B\'{e}rard-Bergery and J.-P. Bourguignon, Laplacians and {R}iemannian
  submersions with totally geodesic fibres. \emph{Illinois J. Math.}
  \textbf{26} (1982), no.~2, 181--200 \MR{650387}

\bibitem{Berd2018}
A.~Berdnikov, \href{https://doi.org/10.4171/JST/206}{Bounds on multiplicities
  of {L}aplace operator eigenvalues on surfaces}. \emph{J. Spectr. Theory}
  \textbf{8} (2018), no.~2, 541--554 \MR{3812808}

\bibitem{BeNP2016}
A.~Berdnikov, N.~Nadirashvili, and A.~Penskoi, Bounds on multiplicities of
  {Laplace-Beltrami} operator eigenvalues on the real projective plane. 2016,
  \urlprefix\url{https://https://arxiv.org/abs/1612.04805}

\bibitem{BeGM1971}
M.~Berger, P.~Gauduchon, and E.~Mazet, \emph{Le spectre d'une vari\'{e}t\'{e}
  riemannienne}. Lecture Notes in Mathematics, Vol. 194, Springer-Verlag,
  Berlin-New York, 1971 \MR{0282313}

\bibitem{Bers1955}
L.~Bers, \href{https://doi.org/10.1002/cpa.3160080404}{Local behavior of
  solutions of general linear elliptic equations}. \emph{Comm. Pure Appl.
  Math.} \textbf{8} (1955), 473--496 \MR{75416}

\bibitem{Bess1980}
G.~Besson, Sur la multiplicit\'{e} de la premi\`ere valeur propre des surfaces
  riemanniennes. \emph{Ann. Inst. Fourier (Grenoble)} \textbf{30} (1980),
  no.~1, x, 109--128 \MR{576075}

\bibitem{Bona2009}
F.~Bonahon, \emph{\href{https://doi.org/10.1090/stml/049}{Low-dimensional
  geometry}}. Student Mathematical Library 49, American Mathematical Society,
  Providence, RI; Institute for Advanced Study (IAS), Princeton, NJ, 2009
  \MR{2522946}

\bibitem{BoHe2017}
V.~Bonnaillie-No\"{e}l and B.~Helffer,
  \href{https://doi.org/10.1515/9783110550887-010}{Nodal and spectral minimal
  partitions---the state of the art in 2016}. In \emph{Shape optimization and
  spectral theory}, pp. 353--397, De Gruyter Open, Warsaw, 2017 \MR{3681154}

\bibitem{Bour2015}
J.~Bourgain, \href{https://doi.org/10.1093/imrn/rnt241}{On {P}leijel's nodal
  domain theorem}. \emph{Int. Math. Res. Not. IMRN}  (2015), no.~6, 1601--1612
  \MR{3340367}

\bibitem{BuCo1985}
M.~Burger and B.~Colbois, \`{A} propos de la multiplicit\'{e} de la premi\`ere
  valeur propre du {Laplacien} d'une surface de {R}iemann. \emph{C. R. Acad.
  Sci. Paris S\'{e}r. I Math.} \textbf{300} (1985), no.~8, 247--249 \MR{785061}

\bibitem{Chen1976}
S.~Y. Cheng, \href{https://doi.org/10.1007/BF02568142}{Eigenfunctions and nodal
  sets}. \emph{Comment. Math. Helv.} \textbf{51} (1976), no.~1, 43--55
  \MR{397805}

\bibitem{CiJLS2022}
D.~Cianci, C.~Judge, S.~Lin, and C.~Sutton, Spectral multiplicity and nodal
  sets for generic torus-invariant metrics. 2022,
  \urlprefix\url{https://arxiv.org/abs/2207.14405}

\bibitem{Colb1985}
B.~Colbois, Petites valeurs propres du {L}aplacien sur une surface de {R}iemann
  compacte et graphes. \emph{C. R. Acad. Sci. Paris S\'{e}r. I Math.}
  \textbf{301} (1985), no.~20, 927--930 \MR{829064}

\bibitem{CoCo1988}
B.~Colbois and Y.~Colin~de Verdi\`ere,
  \href{https://doi.org/10.1007/BF02566762}{Sur la multiplicit\'{e} de la
  premi\`ere valeur propre d'une surface de {R}iemann \`a courbure constante}.
  \emph{Comment. Math. Helv.} \textbf{63} (1988), no.~2, 194--208 \MR{948777}

\bibitem{CoGGS2024}
B.~Colbois, A.~Girouard, C.~Gordon, and D.~Sher,
  \href{https://doi.org/10.1007/s13163-023-00480-3}{Some recent developments on
  the {S}teklov eigenvalue problem}. \emph{Rev. Mat. Complut.} \textbf{37}
  (2024), no.~1, 1--161 \MR{4695859}

\bibitem{ColV1986}
Y.~Colin~de Verdi\`ere, \href{https://doi.org/10.1007/BF02621914}{Sur la
  multiplicit\'{e} de la premi\`ere valeur propre non nulle du {Laplacian}}.
  \emph{Comment. Math. Helv.} \textbf{61} (1986), no.~2, 254--270 \MR{856089}

\bibitem{ColV1987}
Y.~Colin~de Verdi\`ere, Construction de laplaciens dont une partie finie du
  spectre est donn\'{e}e. \emph{Ann. Sci. \'{E}cole Norm. Sup. (4)} \textbf{20}
  (1987), no.~4, 599--615 \MR{932800}

\bibitem{Cour1923}
R.~Courant, Ein allgemeiner {Satz} zur {Theorie} der {Eigenfunktionen}
  selbstadjungierter {Differentialausdr\"{u}cke}. \emph{Nachr. Ges. Wiss.
  G\"{o}ttingen, Math.-Phys. Kl.} \textbf{1923} (1923), 81--84

\bibitem{DaGH2021}
J.~Dahne, J.~G\'{o}mez-Serrano, and K.~Hou,
  \href{https://doi.org/10.1016/j.cnsns.2021.105957}{A counterexample to
  {P}ayne's nodal line conjecture with few holes}. \emph{Commun. Nonlinear Sci.
  Numer. Simul.} \textbf{103} (2021), Paper No. 105957, 13 \MR{4291475}

\bibitem{Dama2000}
L.~Damascelli, On the nodal set of the second eigenfunction of the {L}aplacian
  in symmetric domains in {$\Bbb R^N$}. \emph{Atti Accad. Naz. Lincei Cl. Sci.
  Fis. Mat. Natur. Rend. Lincei (9) Mat. Appl.} \textbf{11} (2000), no.~3,
  175--181 \MR{1841691}

\bibitem{Dies2017}
R.~Diestel, \emph{\href{https://doi.org/10.1007/978-3-662-53622-3}{Graph
  theory}}. Fifth edn., Graduate Texts in Mathematics 173, Springer, Berlin,
  2017 \MR{3644391}

\bibitem{DoFe1990a}
H.~Donnelly and C.~Fefferman, Nodal sets of eigenfunctions: {R}iemannian
  manifolds with boundary. In \emph{Analysis, et cetera}, pp. 251--262,
  Academic Press, Boston, MA, 1990 \MR{1039348}

\bibitem{EsLe2023}
S.~Eswarathasan and C.~Letrouit,
  \href{https://doi.org/10.1093/imrn/rnad219}{Nodal sets of eigenfunctions of
  sub-{L}aplacians}. \emph{Int. Math. Res. Not. IMRN}  (2023), no.~23,
  20670--20700 \MR{4675080}

\bibitem{FoGPP2023}
M.~Fortier~Bourque, \'{E}.~Gruda-Mediavilla, B.~Petri, and M.~Pineault, Two
  counterexamples to a conjecture of {C}olin de {V}erdi\`{e}re on multiplicity.
  2023, \urlprefix\url{https://arxiv.org/abs/2312.03504}

\bibitem{FoBP2023}
M.~{Fortier Bourque} and B.~Petri, Linear programming bounds for hyperbolic
  surfaces. 2023, \urlprefix\url{https://arxiv.org/abs/2302.02540}

\bibitem{FoBP2024}
M.~{Fortier Bourque} and B.~Petri,
  \href{https://doi.org/10.4310/jdg/1727712888}{The {K}lein quartic maximizes
  the multiplicity of the first positive eigenvalue of the {L}aplacian}.
  \emph{J. Differential Geom.} \textbf{128} (2024), no.~2, 521--556
  \MR{4801612}

\bibitem{Four2001}
S.~Fournais, \href{https://doi.org/10.1006/jdeq.2000.3868}{The nodal surface of
  the second eigenfunction of the {L}aplacian in {${\bf R}^D$} can be closed}.
  \emph{J. Differential Equations} \textbf{173} (2001), no.~1, 145--159
  \MR{1836248}

\bibitem{FrHe2025}
R.~L. Frank and B.~Helffer, \href{https://doi.org/10.5802/jep.307}{On {C}ourant
  and {P}leijel theorems for sub-{R}iemannian {L}aplacians}. \emph{J. \'Ec.
  polytech. Math.} \textbf{12} (2025), 1083--1160 \MR{4929455}

\bibitem{FrHeLa2024}
R.~L. Frank, B.~Helffer, and A.~Laptev, Inequalities between {D}irichlet and
  {N}eumann eignevalues on {C}arnot groups. 2024,
  \urlprefix\url{http://arxiv.org/abs/2411.11168}

\bibitem{FrSc2016}
A.~Fraser and R.~Schoen, \href{https://doi.org/10.1007/s00222-015-0604-x}{Sharp
  eigenvalue bounds and minimal surfaces in the ball}. \emph{Invent. Math.}
  \textbf{203} (2016), no.~3, 823--890 \MR{3461367}

\bibitem{Fre2002}
P.~Freitas, \href{https://doi.org/10.1512/iumj.2002.51.2208}{Closed nodal lines
  and interior hot spots of the second eigenfunction of the {L}aplacian on
  surfaces}. \emph{Indiana Univ. Math. J.} \textbf{51} (2002), no.~2, 305--316
  \MR{1909291}

\bibitem{FrKr2008}
P.~Freitas and D.~Krej{\v{c}}i{\v{r}}{\'i}k,
  \href{https://doi.org/10.1512/iumj.2008.57.3170}{Location of the nodal set
  for thin curved tubes}. \emph{Indiana Univ. Math. J.} \textbf{57} (2008),
  no.~1, 343--375 \MR{2400260}

\bibitem{FrLe2025}
P.~Freitas and R.~Leylekian, Payne's nodal line conjecture fails on
  doubly-connected planar domains. 2025,
  \urlprefix\url{https://arxiv.org/abs/2510.24436}

\bibitem{Frie1991}
L.~Friedlander, \href{https://doi.org/10.1007/BF00375590}{Some inequalities
  between {D}irichlet and {N}eumann eigenvalues}. \emph{Arch. Rational Mech.
  Anal.} \textbf{116} (1991), no.~2, 153--160 \MR{1143438}

\bibitem{FrGoMi2025}
M.~Fries, M.~Goffeng, and G.~Miranda,
  \href{https://doi.org/10.4171/jst/582}{Reproving {F}riedlander's inequality
  with the de {R}ham complex}. \emph{J. Spectr. Theory} \textbf{15} (2025),
  no.~4, 1593--1613 \MR{4980388}

\bibitem{GaHuLa2004}
S.~Gallot, D.~Hulin, and J.~Lafontaine,
  \emph{\href{https://doi.org/10.1007/978-3-642-18855-8}{Riemannian geometry}}.
  Third edn., Universitext, Springer-Verlag, Berlin, 2004 \MR{2088027}

\bibitem{Gibl2010}
P.~Giblin, \emph{\href{https://doi.org/10.1017/CBO9780511779534}{Graphs,
  surfaces and homology}}. Third edn., Cambridge University Press, Cambridge,
  2010 \MR{2722281}

\bibitem{GiTr1977}
D.~Gilbarg and N.~S. Trudinger, \emph{Elliptic partial differential equations
  of second order}. Springer-Verlag, Berlin-New York, 1977 \MR{0473443}

\bibitem{GiHaLeSh2024}
K.~Gittins, A.~Hassannezhad, C.~L\'{e}na, and D.~Sher, Nodal counts for the
  {R}obin problem on {L}ipschitz domains. 2024,
  \urlprefix\url{http://arxiv.org/abs/2411.11427}

\bibitem{GiHe2019}
K.~Gittins and B.~Helffer, \href{https://doi.org/10.4171/PM/2027}{Courant-sharp
  {R}obin eigenvalues for the square and other planar domains}. \emph{Port.
  Math.} \textbf{76} (2019), no.~1, 57--100 \MR{4016623}

\bibitem{GiLe2020}
K.~Gittins and C.~L\'{e}na, Upper bounds for {C}ourant-sharp {N}eumann and
  {R}obin eigenvalues. \emph{Bull. Soc. Math. {F}rance} \textbf{148} (2020),
  no.~1, 99--132

\bibitem{Gott1990}
D.~H. Gottlieb, \href{https://doi.org/10.1007/BF02925086}{Vector fields and
  classical theorems of topology}. \emph{Rend. Sem. Mat. Fis. Milano}
  \textbf{60} (1990), 193--203 \MR{1229491}

\bibitem{GoSa1995}
D.~H. Gottlieb and G.~Samaranayake, The index of discontinuous vector fields.
  \emph{New York J. Math.} \textbf{1} (1994/95), 130--148, electronic
  \MR{1341518}

\bibitem{HaWi1953}
P.~Hartman and A.~Wintner, \href{https://doi.org/10.2307/2372496}{On the local
  behavior of solutions of non-parabolic partial differential equations}.
  \emph{Amer. J. Math.} \textbf{75} (1953), 449--476 \MR{58082}

\bibitem{HaSh2023}
A.~Hassannezhad and D.~Sher, {On Pleijel's nodal domain theorem for the Robin
  problem}. 2023, {arXiv:2303.08094}

\bibitem{HeHN2002}
B.~Helffer, M.~Hoffmann-Ostenhof, T.~Hoffmann-Ostenhof, and N.~Nadirashvili,
  \href{https://doi.org/10.1007/PL00012652}{Spectral theory for the dihedral
  group}. \emph{Geom. Funct. Anal.} \textbf{12} (2002), no.~5, 989--1017
  \MR{1937833}

\bibitem{HeHO1999}
B.~Helffer, M.~Hoffmann-Ostenhof, T.~Hoffmann-Ostenhof, and M.~P. Owen,
  \href{https://doi.org/10.1007/s002200050599}{Nodal sets for groundstates of
  {S}chr\"{o}dinger operators with zero magnetic field in non-simply connected
  domains}. \emph{Comm. Math. Phys.} \textbf{202} (1999), no.~3, 629--649
  \MR{1690957}

\bibitem{HeHJ2020}
B.~Helffer, T.~Hoffmann-Ostenhof, F.~Jauberteau, and C.~L\'{e}na, On the
  multiplicity of the second eigenvalue of the {Laplacian} in non simply
  connected domains -- with some numerics --. \emph{Asymptotic Analysis}
  \textbf{121} (2021), no.~1, 35--57

\bibitem{HeHT2009}
B.~Helffer, T.~Hoffmann-Ostenhof, and S.~Terracini,
  \href{https://doi.org/10.1016/j.anihpc.2007.07.004}{Nodal domains and
  spectral minimal partitions}. \emph{Ann. Inst. H. Poincar\'{e} Anal. Non
  Lin\'{e}aire} \textbf{26} (2009), no.~1, 101--138 \MR{2483815}

\bibitem{HoHN1997}
M.~Hoffmann-Ostenhof, T.~Hoffmann-Ostenhof, and N.~Nadirashvili,
  \href{https://doi.org/10.1215/S0012-7094-97-09017-7}{The nodal line of the
  second eigenfunction of the {L}aplacian in {${\bf R}^2$} can be closed}.
  \emph{Duke Math. J.} \textbf{90} (1997), no.~3, 631--640 \MR{1480548}

\bibitem{HoHN1998}
M.~Hoffmann-Ostenhof, T.~Hoffmann-Ostenhof, and N.~Nadirashvili,
  \href{https://doi.org/10.1090/conm/217/02980}{On the nodal line conjecture}.
  In \emph{Advances in differential equations and mathematical physics
  ({A}tlanta, {GA}, 1997)}, pp. 33--48, Contemp. Math. 217, Amer. Math. Soc.,
  Providence, RI, 1998 \MR{1605269}

\bibitem{HoHN1999}
M.~Hoffmann-Ostenhof, T.~Hoffmann-Ostenhof, and N.~Nadirashvili,
  \href{https://doi.org/10.1023/A:1006595115793}{On the multiplicity of
  eigenvalues of the {L}aplacian on surfaces}. \emph{Ann. Global Anal. Geom.}
  \textbf{17} (1999), no.~1, 43--48 \MR{1674331}

\bibitem{HoMN1999}
T.~Hoffmann-Ostenhof, P.~W. Michor, and N.~Nadirashvili,
  \href{https://doi.org/10.1007/s000390050111}{Bounds on the multiplicity of
  eigenvalues for fixed membranes}. \emph{Geom. Funct. Anal.} \textbf{9}
  (1999), no.~6, 1169--1188 \MR{1736932}

\bibitem{Horm2007c}
L.~H\"{o}rmander, \emph{\href{https://doi.org/10.1007/978-3-540-49938-1}{The
  analysis of linear partial differential operators. {III}}}. Classics in
  Mathematics, Springer, Berlin, 2007 \MR{2304165}

\bibitem{Horm2009d}
L.~H\"{o}rmander, \emph{\href{https://doi.org/10.1007/978-3-642-00136-9}{The
  analysis of linear partial differential operators. {IV}}}. Classics in
  Mathematics, Springer-Verlag, Berlin, 2009 \MR{2512677}

\bibitem{HuMuZh2024}
B.~Hua, F.~M\"{u}nch, and H.~Zhang, Inequalities between {D}irichlet and
  {N}eumann eigenvalues on surfaces. 2024,
  \urlprefix\url{https://arxiv.org/abs/2412.19480}

\bibitem{Ivri1980}
V.~J. Ivri\u{\i}, \href{https://doi.org/10.1007/BF01086550}{Second term of the
  spectral asymptotic expansion of the {L}aplace-{B}eltrami operator on
  manifolds with boundary}. \emph{Funct Anal Its Appl.} \textbf{14} (1980),
  no.~2, 98--106

\bibitem{Ivri1980r}
V.~J. Ivri\u{\i}, The second term of the spectral asymptotics for a
  {L}aplace-{B}eltrami operator on manifolds with boundary. \emph{Funktsional.
  Anal. i Prilozhen.} \textbf{14} (1980), no.~2, 25--34 \MR{575202}

\bibitem{Jam2016}
P.~Jammes, \href{https://doi.org/10.2140/pjm.2016.282.145}{Multiplicit\'{e} du
  spectre de {S}teklov sur les surfaces et nombre chromatique}. \emph{Pacific
  J. Math.} \textbf{282} (2016), no.~1, 145--171 \MR{3463427}

\bibitem{Jeri1991}
D.~Jerison, \href{https://doi.org/10.1155/S1073792891000016}{The first nodal
  line of a convex planar domain}. \emph{Internat. Math. Res. Notices}  (1991),
  no.~1, 1--5 \MR{1104555}

\bibitem{JuZe2022}
J.~Jung and S.~Zelditch, $2$-nodal domain theorems for higher dimensional
  circle bundles. 2022, \urlprefix\url{https://arxiv.org/abs/2207.13498}

\bibitem{KaKP2014}
M.~Karpukhin, G.~Kokarev, and I.~Polterovich,
  \href{https://doi.org/10.5802/aif.2918}{Multiplicity bounds for {S}teklov
  eigenvalues on {R}iemannian surfaces}. \emph{Ann. Inst. Fourier (Grenoble)}
  \textbf{64} (2014), no.~6, 2481--2502 \MR{3331172}

\bibitem{Ken2011}
J.~B. Kennedy, \href{https://doi.org/10.1016/j.jde.2011.08.012}{The nodal line
  of the second eigenfunction of the {R}obin {L}aplacian in {$\Bbb R^2$} can be
  closed}. \emph{J. Differential Equations} \textbf{251} (2011), no.~12,
  3606--3624 \MR{2837697}

\bibitem{Ken2013}
J.~B. Kennedy, \href{https://doi.org/10.1512/iumj.2013.62.4975}{Closed nodal
  surfaces for simply connected domains in higher dimensions}. \emph{Indiana
  Univ. Math. J.} \textbf{62} (2013), no.~3, 785--798 \MR{3164844}

\bibitem{Kiwa2018}
R.~Kiwan, \href{https://doi.org/10.5802/afst.1585}{On the nodal set of a second
  {D}irichlet eigenfunction in a doubly connected domain}. \emph{Ann. Fac. Sci.
  Toulouse Math. (6)} \textbf{27} (2018), no.~4, 863--873 \MR{3884612}

\bibitem{Kuo1969}
T.~C. Kuo, \href{https://doi.org/10.1016/0040-9383(69)90007-X}{On
  {$C\sp{0}$}-sufficiency of jets of potential functions}. \emph{Topology}
  \textbf{8} (1969), 167--171 \MR{238338}

\bibitem{Lena2019}
C.~L\'{e}na, Pleijel's nodal domain theorem for {N}eumann and {R}obin
  eigenfunctions. \emph{Ann. Inst. Fourier (Grenoble)} \textbf{69} (2019),
  no.~1, 283--301 \MR{3973450}

\bibitem{LeMa2024}
C.~Letrouit and S.~Machado, Maximal multiplicity of {L}aplacian eigenvalues in
  negatively curves surfaces. 2024,
  \urlprefix\url{https://arxiv.org/abs/2307.06646v3}

\bibitem{Lewy1977}
H.~Lewy, \href{https://doi.org/10.1080/03605307708820059}{On the minimum number
  of domains in which the nodal lines of spherical harmonics divide the
  sphere}. \emph{Comm. Partial Differential Equations} \textbf{2} (1977),
  no.~12, 1233--1244 \MR{477199}

\bibitem{Lin1987}
C.~S. Lin, On the second eigenfunctions of the {L}aplacian in {${\bf R}^2$}.
  \emph{Comm. Math. Phys.} \textbf{111} (1987), no.~2, 161--166 \MR{899848}

\bibitem{LiNi1988}
C.~S. Lin and W.-M. Ni, \href{https://doi.org/10.2307/2045874}{A counterexample
  to the nodal domain conjecture and a related semilinear equation}.
  \emph{Proc. Amer. Math. Soc.} \textbf{102} (1988), no.~2, 271--277
  \MR{920985}

\bibitem{Liq1995}
Z.~Liqun, \href{https://doi.org/10.4310/CAG.1995.v3.n2.a3}{On the multiplicity
  of the second eigenvalue of {L}aplacian in {${\bf R}^2$}}. \emph{Comm. Anal.
  Geom.} \textbf{3} (1995), no. 1-2, 273--296 \MR{1362653}

\bibitem{MaMuSa2025}
S.~Maji, M.~Mukherjee, and S.~Saha, On the nodal domain count under metric
  perturbations, {C}ourant sharpness and boundary intersections. 2025,
  \urlprefix\url{https://arxiv.org/abs/2507.04928}

\bibitem{Maz1991}
R.~Mazzeo, \href{https://doi.org/10.1155/S1073792891000065}{Remarks on a paper
  of {L}. {F}riedlander concerning inequalities between {N}eumann and
  {D}irichlet eigenvalues}. \emph{Internat. Math. Res. Notices} \textbf{1991}
  (1991), no.~4, 41--48 \MR{1121164}

\bibitem{Mela1992}
A.~D. Melas, On the nodal line of the second eigenfunction of the {L}aplacian
  in {${\bf R}^2$}. \emph{J. Differential Geom.} \textbf{35} (1992), no.~1,
  255--263 \MR{1152231}

\bibitem{Mikh1978}
V.~P. Mikha\u{\i}lov, \emph{Partial differential equations}. ``Mir'', Moscow;
  distributed by Imported Publications, Chicago, Ill., 1978 \MR{601389}

\bibitem{Miln1997}
J.~W. Milnor, \emph{Topology from the differentiable viewpoint}. Princeton
  University Press, Princeton, NJ, 1997 \MR{1487640}

\bibitem{MuSa2022}
M.~Mukherjee and S.~Saha,
  \href{https://doi.org/10.1007/s00208-021-02144-3}{Nodal sets of {L}aplace
  eigenfunctions under small perturbations}. \emph{Math. Ann.} \textbf{383}
  (2022), no. 1-2, 475--491 \MR{4444128}

\bibitem{MuSa2025}
M.~Mukherjee and S.~Saha, \href{https://doi.org/10.4171/jst/570}{On the effects
  of small perturbation on low energy {L}aplace eigenfunctions}. \emph{J.
  Spectr. Theory} \textbf{15} (2025), no.~3, 1045--1087 \MR{4949372}

\bibitem{Nadi1987}
N.~S. Nadirashvili,
  \href{https://doi.org/10.1070/SM1988v061n01ABEH003204}{Multiple eigenvalues
  of the {L}aplace operator ({R}ussian)}. \emph{Mat. Sb. (N.S.)}
  \textbf{133(175)} (1987), no.~2, 223--237, 272 \MR{905007}

\bibitem{Paga1990}
K.~F. Pagani,
  \emph{\href{https://doi.org/10.3929/ethz-a-000592869}{Geometrische
  {Eigenschaften} der {L\"{o}sungen} von quasilinearen, elliptischen
  {Differentialgleichungen} zweiter {Ordnung} in der {Ebene}}}. ETH Z\"{u}rich,
  1990

\bibitem{Payn1967}
L.~E. Payne, \href{https://doi.org/10.1137/1009070}{Isoperimetric inequalities
  and their applications}. \emph{SIAM Rev.} \textbf{9} (1967), 453--488
  \MR{218975}

\bibitem{Payn1973}
L.~E. Payne, \href{https://doi.org/10.1007/BF01597076}{On two conjectures in
  the fixed membrane eigenvalue problem}. \emph{Z. Angew. Math. Phys.}
  \textbf{24} (1973), 721--729 \MR{333487}

\bibitem{Peet1957}
J.~Peetre, \href{https://doi.org/10.7146/math.scand.a-10484}{A generalization
  of {C}ourant's nodal domain theorem}. \emph{Math. Scand.} \textbf{5} (1957),
  15--20 \MR{92917}

\bibitem{PiVe2020}
S.~Pigola and G.~Veronelli, The smooth {R}iemannian extension problem.
  \emph{Ann. Sc. Norm. Super. Pisa Cl. Sci. (5)} \textbf{20} (2020), no.~4,
  1507--1551 \MR{4201188}

\bibitem{Plej1956}
{\AA}.~Pleijel, \href{https://doi.org/10.1002/cpa.3160090324}{Remarks on
  {C}ourant's nodal line theorem}. \emph{Comm. Pure Appl. Math.} \textbf{9}
  (1956), 543--550 \MR{80861}

\bibitem{Polt2009}
I.~Polterovich, \href{https://doi.org/10.1090/S0002-9939-08-09596-8}{Pleijel's
  nodal domain theorem for free membranes}. \emph{Proc. Amer. Math. Soc.}
  \textbf{137} (2009), no.~3, 1021--1024 \MR{2457442}

\bibitem{Putt1990}
R.~P\"{u}tter, \href{https://doi.org/10.1007/BF02566596}{On the nodal lines of
  second eigenfunctions of the fixed membrane problem}. \emph{Comment. Math.
  Helv.} \textbf{65} (1990), no.~1, 96--103 \MR{1036131}

\bibitem{Putt1991}
R.~P\"{u}tter, \href{https://doi.org/10.1080/00036819108840041}{On the nodal
  lines of second eigenfunctions of the free membrane problem}. \emph{Appl.
  Anal.} \textbf{42} (1991), no. 3-4, 199--207 \MR{1124959}

\bibitem{Safa2001}
Y.~Safarov, \href{https://doi.org/10.1006/jfan.2001.3764}{Fourier {T}auberian
  theorems and applications}. \emph{J. Funct. Anal.} \textbf{185} (2001),
  no.~1, 111--128 \MR{1853753}

\bibitem{Salo2014}
M.~Salo, Unique continuation for elliptic equations. 2014,
  \urlprefix\url{https://www.jyu.fi/science/en/maths/research/inverse-problems/past-events/unique_continuation_notes.pdf},
  notes edited by Mikko Salo. Department of {M}athematics and {Statistics}.
  University of {J}yv\"{a}skyl\"{a}. Finland.

\bibitem{Seve2002}
B.~S\'{e}vennec, \href{https://doi.org/10.1007/s00208-002-0337-1}{Multiplicity
  of the second {S}chr\"{o}dinger eigenvalue on closed surfaces}. \emph{Math.
  Ann.} \textbf{324} (2002), no.~1, 195--211 \MR{1931764}

\bibitem{Shen1988}
C.~L. Shen, \href{https://doi.org/10.1137/0519104}{On the nodal sets of the
  eigenfunctions of the string equation}. \emph{SIAM J. Math. Anal.}
  \textbf{19} (1988), no.~6, 1419--1424 \MR{965261}

\bibitem{SoZe2011}
C.~D. Sogge and S.~Zelditch,
  \href{https://doi.org/10.4310/MRL.2011.v18.n1.a3}{Lower bounds on the
  {H}ausdorff measure of nodal sets}. \emph{Math. Res. Lett.} \textbf{18}
  (2011), no.~1, 25--37 \MR{2770580}

\bibitem{Stei2014}
S.~Steinerberger, \href{https://doi.org/10.1007/s00023-013-0310-4}{A geometric
  uncertainty principle with an application to {P}leijel's estimate}.
  \emph{Ann. Henri Poincar\'{e}} \textbf{15} (2014), no.~12, 2299--2319
  \MR{3272823}

\bibitem{Ster1925}
A.~Stern, \emph{{Bemerkungen \"uber asymptotisches Verhalten von Eigenwerten
  und Eigenfunktionen}}. {Math-Naturwiss. Diss.}, Universit\"at G\"ottingen,
  1925

\bibitem{Vali1966}
G.~Valiron, \emph{Th\'{e}orie des {F}onctions (3\` {e}me \'{e}dition}. Masson
  et Cie, Paris, 1966

\bibitem{BeGi2016}
M.~van~den Berg and K.~Gittins, \href{https://doi.org/10.4171/JST/139}{On the
  number of {C}ourant-sharp {D}irichlet eigenvalues}. \emph{J. Spectr. Theory}
  \textbf{6} (2016), no.~4, 735--745 \MR{3584181}

\bibitem{BeLi2001}
M.~van~den Berg and M.~Lianantonakis,
  \href{https://doi.org/10.1512/iumj.2001.50.1913}{Asymptotics for the spectrum
  of the {D}irichlet {L}aplacian on horn-shaped regions}. \emph{Indiana Univ.
  Math. J.} \textbf{50} (2001), no.~1, 299--333 \MR{1857038}

\bibitem{YaZh2021}
Y.~Yang and J.~Zhou, \href{https://doi.org/10.1016/j.jmaa.2021.125440}{Blow-up
  analysis involving isothermal coordinates on the boundary of compact
  {R}iemann surface}. \emph{J. Math. Anal. Appl.} \textbf{504} (2021), no.~2,
  Paper No. 125440, 39 \MR{4280277}

\end{thebibliography}


\printindex

\end{document}